\documentclass[12pt]{amsbook}

\usepackage{amscd}
\usepackage{amsmath,amsfonts,amsthm,epsfig,latexsym,graphicx,amssymb}
\usepackage[english]{babel}
\usepackage{comment}
\usepackage{epsfig}
\usepackage{tikz, graphics} 
\usepackage{graphics,color}
\usepackage{psfrag}
\usepackage{hyperref}
\hypersetup{
	colorlinks=true,
	citecolor=green,
	linkcolor=blue,
	hypertexnames=false
}

\newtheorem{coro}{Corollary}[chapter]
\newtheorem{defi}{Definition}[chapter]
\newtheorem{prop}{Proposition}[chapter]

\newtheorem{theo}{Theorem}[chapter]
\newtheorem{lemm}{Lemma}[chapter]
\newtheorem{ques}{Question}[chapter]

\newtheorem{exem}{Example}[chapter]
\newtheorem{prob}{Problem}[chapter]
\newtheorem{rema}{Remark}[chapter]

\newtheorem{conv}{Convention}[chapter]
\newtheorem{rem}{Remarque}[chapter]

\newtheorem*{coro*}{Corollary}
\newtheorem*{defi*}{Definition}
\newtheorem*{prop*}{Proposition}
\newtheorem*{conj*}{Conjecture}
\newtheorem*{theo*}{Theorem}
\newtheorem*{lemm*}{Lemma}
\newtheorem*{ques*}{Question}
\newtheorem*{exer*}{Exercise}
\newtheorem*{exem*}{Example}
\newtheorem*{prob*}{Problem}
\newtheorem*{rema*}{Remark}
\newtheorem*{conv*}{Convention}
\newtheorem*{clai*}{Claim}
\newtheorem*{affi*}{Affirmation}
\newtheorem*{claim*}{Claim}
\newtheorem*{exes*}{Exercise}
\newtheorem*{pbm*}{Problem}

\def\R{I\kern -0.37 em R}
\def\N{I\kern -0.37 em N}
\def\Z{I\kern -0.37 em Z}

\def\AA{{\mathbb A}}  \def\CC{{\mathbb C}}

\def\EE{{\mathbb E}}

 \def\NN{{\mathbb N}} 

 \def\RR{{\mathbb R}}

 \def\ZZ{{\mathbb Z}}

  \def\cG{{\mathcal G}} \def\cM{{\mathcal M}} \def\cS{{\mathcal S}} 
\def\cB{{\mathcal B}}  \def\cH{{\mathcal H}}  \def\cT{{\mathcal T}} 
  \def\cI{{\mathcal I}} \def\cO{{\mathcal O}} \def\cU{{\mathcal U}}
\def\cD{{\mathcal D}}  \def\cJ{{\mathcal J}} \def\cP{{\mathcal P}} \def\cV{{\mathcal V}}
    \def\cW{{\mathcal W}}
\def\cF{{\mathcal F}}   \def\cR{{\mathcal R}}

\title{	An Algorithmic Classification of Generalized Pseudo-Anosov Homeomorphisms via Geometric Markov Partitions.}

\author{Inti Cruz Diaz}

\begin{document}

\begin{titlepage}
	
	\begin{center}
		\vspace*{1cm}
		
		\Large\textbf{\MakeUppercase{Thèse de doctorat de l'établissement}}\\
		\vspace{0.5cm}
		\Large\textbf{\MakeUppercase{Université Bourgogne Franche-Comté}}\\
		\vspace{1cm}
		\Large\textbf{\MakeUppercase{Préparée à l'université de Bourgogne}}\\
		\vspace{1cm}
		
		École doctorale Carnot-Pasteur\\
		Doctorat de mathématiques\\
		\vspace{1cm}
		
			Par \\
		Inti Cruz Diaz\\
		\vspace{1cm}
		
				Titre de la thèse:\\
		\textit{An Algorithmic Classification of Generalized Pseudo-Anosov Homeomorphisms via Geometric Markov Partitions.}\\

		\vspace{2cm}
		
		Directeur de thèse:\\
			\vspace{0.5cm}
		 Christian Bonatti - CNRS  $\&$ IMB
		
	\end{center}
\end{titlepage}

\tableofcontents

\newpage

\begin{center}
\textbf{Acknowledgements:}
I am grateful for the full funding provided by a CONACyT grant under the \emph{Becas al Extranjero Convenios GOBIERNO FRANCES 2019 - 1} program in order to pursue my PhD in Dijon. This thesis is the product of such investment. I appreciate the trust you have placed in my mathematics research project, particularly given the budget constraints that Mexico faces in the realm of pure science.
\end{center}

\newpage

\chapter*{Introduction}

\section*{Abstract}

This thesis explores the problem of classifying generalized pseudo-Anosov homeomorphisms under topological conjugacy. Since each of these homeomorphisms admits a Markov partition and preserves the orientation of the surface on which it acts, we are going to show how to associate each Markov partition a geometrization and then a finite combinatorial object, known as the geometric type of the Markov partition.

The idea is to establish a connection between the definition of generalized pseudo-Anosov homeomorphisms in terms of measured foliations and one given in by  geometric types. The thesis presents three essential results: the geometric type is a complete invariant of conjugation, there is a combinatorial characterization of geometric types that correspond to generalized pseudo-Anosov homeomorphisms, and finally, we give an algorithm  to determine when two realizable geometric types correspond to the same conjugacy class. In this manner, we can change one definition for another, with the hope that the combinatorial approach that we present permits us to use computers to understand these dynamical systems. We have not developed any program up to this moment, but this philosophical approach guides us.

In this introduction, we will attempt to place our problem and the obtained results into context. Throughout this process, it will be necessary to provide definitions that are not entirely formal, aiming to convey the intuition behind them. In any case, we will refer to the corresponding part of the text where you can find the complete definitions.

\section*{The classification problem for generalized pseudo-Anosov homeomorphisms. }

\pagenumbering{Roman}

It could be proven that any Anosov diffeomorphism $f$ on the $2$-torus $\mathbb{T}^2$ preserves two transverse and non-singular \emph{measured foliations}, $(\cF^s,\mu^s)$ and $(\cF^u,\mu^u)$, the \emph{stable} and \emph{unstable} foliations, respectively. Moreover, there is a number $\lambda>0$ such that $f$ contracts by a factor of $\lambda^{-1}$ the $\mu^u$-measure of every arc transverse to $\cF^u$, and expands by a factor of $\lambda$ the $\mu^s$-measure of any arc transverse to $\cF^s$

The only closed and orientable surface that supports non-singular foliations is $\mathbb{T}^2$. We will study a class of homeomorphisms defined on compact and oriented surfaces that generalize Anosov diffeomorphisms on the $2$-torus. These are called \emph{pseudo-Anosov homeomorphisms with spines} (also known as pseudo-Anosov homeomorphisms with marked points or generalized pseudo-Anosov homeomorphisms). Similar to Anosov diffeomorphisms, they preserve a pair of foliations. However, in this case, these foliations admit a finite number of singularities and can be defined on surfaces of any genus.

Let $S$ be a closed and oriented surface. A homeomorphism $f:S\rightarrow S$ is called \emph{pseudo-Anosov} if it preserves two transverse and \emph{singular measured foliations}: $(\cF^s,\mu^s)$ and $(\cF^s,\mu^u)$. Furthermore, there exists a \emph{dilatation factor} $\lambda>0$ such that $f$ induces a uniform expansion by a factor of $\lambda$ on the leaves of $\cF^u$ and a contraction by a factor of $\lambda^{-1}$ on the leaves of $\cF^s$. The difference with the Anosov case is that these foliations admit a finite number of saddle-type singularities with $k\geq 3$ separatrices (called $k$-prongs). In this thesis, we consider \emph{pseudo-Anosov homeomorphisms with  spines}, whose invariant foliations can exhibit singularities with only one separatrice, known as spines.

Let Hom$(S)$ be the group of homeomorphisms of $S$, and Hom$_+(S)$ be the group of orientation-preserving homeomorphisms of $S$. The mapping class group $\mathcal{M}(S)$ is obtained by considering equivalence classes in Hom$_+(S)$ under isotopy. The Nielsen–Thurston classification establishes a trichotomy for every homeomorphism in Hom$_+(S)$, where each homeomorphism is isotopic to one that is periodic, reducible, or pseudo-Anosov. It is in these works where the formal definition of a pseudo-Anosov diffeomorphism appears. However, we want to study Hom$_+(S)$ under another equivalence relation, topological conjugation. Two homeomorphisms $f$ and $g$ in Hom$(S)$ are \emph{topologically conjugate} if and only if there exists $h$ in Hom$(S)$ such that $f = h \circ g \circ h^{-1}$.  The equivalence classes formed by this relation are referred to as \emph{conjugacy classes}, and our interest lies in their classification. However, we will restrict ourselves to the case when conjugation is performed by a homeomorphism that preserves orientation, i.e. $h\in$Hom$_+(S)$.

We will study pseudo-Anosov homeomorphisms as a whole, regardless of whether they have spines. For practical writing purposes, we refer to them as \emph{generalized pseudo-Anosov homeomorphisms}. If necessary in the discussion or proofs, we will explicitly mention whether a particular homeomorphism has spine-like singularities or not. This thesis focuses on the conjugacy classes that contain a generalized pseudo-Anosov homeomorphism that in addition preserve the orientation of the surface.
 
Furthermore, the set of generalized  pseudo-Anosov homeomorphisms it is closed under topological conjugation. In other words, if $f\in$ Hom$_+(S)$ is a generalized pseudo-Anosov homeomorphism and $h\in$Hom$_+(S)$ is any homeomorphism, then $g := h \circ f \circ h^{-1}$ is also a generalized pseudo-Anosov homeomorphism.

Some authors use the term \emph{pseudo-Anosov homeomorphisms with marked points} to refer to pseudo-Anosov homeomorphisms with spines. This is relevant when studying isotopy classes relative to such marked points. However, in our approach, we consider conjugation by the whole group of homeomorphisms without explicitly considering the set of marked points. In terms of topological conjugation, we can inquire about homeomorphisms that are topologically conjugate and whose conjugation preserves a predefined set of marked points. The problem becomes more complicated as we seek to find a topological conjugation between two homeomorphisms that also preserves the entire marked set, which may include not only singularities. This is something we didn't consider in the thesis.

Our goal is to classify the conjugacy classes of pseudo-Anosov homeomorphisms and understand their behaviors. We face challenges in developing effective methods to distinguish between conjugacy classes by studying the dynamics, analyzing the structure of the invariant foliations, and considering the impact of spines. Let us clarify the notion of classification and the challenges we have imposed on ourselves in order to achieve this classification. We hope that our efforts will contribute to a broader understanding of these important dynamical systems.

\begin{prob}\label{Prob: Clasification}
	
The \emph{classification problem} for generalized pseudo-Anosov homeomorphisms can be divided into three fundamental components.

	\begin{enumerate}
		\item 	\textbf{Finite presentation:} This aspect aims to provide a concise and combinatorial description of the conjugacy classes that contain a generalized pseudo-Anosov homeomorphism. It can be further divided into the following sub-problems:
		\begin{itemize}
			\item[I)]\textbf{Combinatorial representation:} The objective of this problem is to establish a formal and precise association between every generalized pseudo-Anosov homeomorphism $f$ and a \emph{finite combinatorial object} denoted as $T:=T(f)$. This association is designed such that the same combinatorial object is linked to a pair of generalized pseudo-Anosov homeomorphisms if and only if they are topologically conjugated. It is important to note that the association may not be unique and $f$ can have multiple valid representations.
						
			\item[II)] \textbf{Combinatorial models:} Given a combinatorial information associated to a generalized pseudo-Anosov homeomorphism $f$, $T(f)$, we propose to reconstruct a pseudo-Anosov homeomorphism that is conjugate to $f$ using the information provided by $T(f)$. This problem represents a higher level of ambition compared to the previous point, since successfully achieving this reconstruction involves the combinatorial representation.

		\end{itemize} 
		
		\item \textbf{Realization:}The set of combinatorial objects forms a set denoted by $\mathcal{GT}$, and the subset of these objects associated with a pseudo-Anosov homeomorphism is referred to as the \emph{pseudo-Anosov class} and will be denoted by $\mathcal{GT}(p\AA)$. Generally, $\mathcal{GT}(p\AA)$ is a proper subset of $\mathcal{GT}$.
		
		The second problem is to provide a characterization and an algorithm for determining whether or not an abstract combinatorial information belongs to the pseudo-Anosov class. Ideally, we would like to set a time bound for the algorithm to provide an answer.
	
		\item \textbf{Algorithmic decidability:} A generalized pseudo-Anosov homeomorphism can have multiple combinatorial representations. The subsequent problem involves determining whether two combinatorial representations in the pseudo-Anosov class correspond to conjugate pseudo-Anosov homeomorphisms or not. Similar to the previous problem, our objective is to develop an algorithm that can efficiently determine, within finite and bounded time, whether two combinatorial representations represent the same conjugacy class.
		
	A more ambitious task is to identify a finite and distinguished family, denoted $\Theta[f]$, of the set of combinatorial objects associated with the pseudo-Anosov homeomorphism $f$. This family must satisfy three essential criteria: 
		\begin{itemize}
		\item $\Theta(f)$ must be computable from any other combinatorial object associated with $f$.
		\item $\Theta(f)$ must allow the computation of any other invariant $T(f)$ associated with $f$.
		\item $\Theta(f)$ must be, by contention,  the minimum set of combinatorial information that satisfies the previous items.
		\end{itemize}
	\end{enumerate}
\end{prob}

Let $S$ be a closed and oriented surface, and let $f: S \rightarrow S$ be a generalized pseudo-Anosov homeomorphism. The combinatorial object associated with $f$ will be the \emph{geometric type} of a \emph{geometric Markov partition} of $f$. In the course of presenting our results, we will provide an explanation and overview of these concepts.
 
 \section*{Presentation of results.}

 \subsection{Geometric Markov partitions and its geometric type.}

Let $S$ be a closed oriented surface, and let $f$ be a pseudo-Anosov homeomorphism with spines. We denote by $(\mathcal{F}^s,\mu^s)$ and $(\mathcal{F}^u,\mu^u)$ its invariant foliations equipped with their transverse measures. A \emph{rectangle} $R\subset S$ of $f$ is defined as a pair $(R,r)$, where $r$ is a (equivalence class of) continuous map $r: [-1,1]^2 \rightarrow S$ such that its image is $R$, i.e., $r([-1,1]^2)=R$. Moreover, $r$ is a homeomorphism on $(-1,1)^2$, and the pre-images of $\mathcal{F}^s$ and $\mathcal{F}^u$ under $r^{-1}$ correspond to the horizontal and vertical foliations, respectively, of the unit square. We identify two rectangles $r$ and $r'$ if there are mappings $\varphi_s,\varphi_u: [-1,1] \rightarrow [-1,1]$ such that $r'^{-1} \circ r = (\varphi_s,\varphi_u)$ and $\varphi_{s,u}$ are increasing homeomorphisms defined on the interval $[-1,1]$ (see Definition~\ref{Defi: rectangle}). The interior of $R$ is denoted as $\overset{o}{R}$ and corresponds to the restriction of $r$ to $(-1,1)^2$. Typically, when the parametrization is not necessary for the exposition, we refer to the image of $r$ simply as a rectangle.

A \emph{Markov partition}  $\cR=\{R_i\}_{i=1}^n$ of $f$ is a decomposition of the surface into a finite set of rectangles with disjoint interiors, satisfying the following property: if $\overset{o}{R_i}\cap f^{-1}(\overset{o}{R_j})\neq \emptyset$, then every connected component $C$ of $\overset{o}{R_i}\cap f^{-1}(\overset{o}{R_j})$ is the interior of a \emph{horizontal sub-rectangle} of $R_i$ (see Definition~\ref{Defi: vertical/horizontal sub rectangles}), and $f(C)$ is the interior of a vertical sub-rectangle of $R_j$.

The following result is a classical theorem for the case of pseudo-Anosov homeomorphisms without spines (see~\cite{fathi2021thurston}). In Chapter \ref{Section: Existence}, we provide a recipe for constructing Markov partitions for generalized pseudo-Anosov homeomorphisms¨. The advantage of our procedure is that it begins by selecting a point, and the rest of the construction follows an algorithmic approach.
  
\begin{coro*}[\ref{Coro: Existence adapted Markov partitions}]
Every generalized pseudo-Anosov homeomorphism admits Markov partitions.
\end{coro*}

\subsection{Finite presentations: The geometric type is a complete conjugacy invariant.}

If $R$ is a rectangle, the intersection $\mathcal{F}^s\cap R$ is referred to as the horizontal foliation of $R$, while $\mathcal{F}^u\cap R$ is its vertical foliation. The definition of a rectangle naturally provides an orientation for these foliations. By choosing an orientation for the vertical (unstable) foliation and assigning the unique orientation to the horizontal (stable) foliation, such that the parametrization $r$ restricted to the interior of the unit square, $r:(0,1)\times (0,1)\to \overset{o}{R}\subset S$, is an orientation-preserving homeomorphism. 
A \emph{geometric Markov partition} of $f$ is a Markov partition of the homeomorphism where a vertical orientation has been chosen for all the rectangles within it. In the book \cite{bonatti1998diffeomorphismes}, the authors associate to each geometric Markov partition a combinatorial information called \emph{geometric type}. Let us take a moment to review this concept. Throughout the rest of our presentation, we assume that all rectangles have been assigned orientations in their respective horizontal foliations, i.e, they are indeed, geometric Markov partitions.

Let $f$ be a generalized pseudo-Anosov homeomorphism and let $\cR=\{R_i\}_{i=1}^n$ be a geometric Markov partition of $f$. The set of horizontal sub-rectangles of $R_i$, which are the connected components of the intersection $\overset{o}{R_i}\cap f^{-1}(\overset{o}{R_k})$ for some $k\in \{1,\cdots,n\}$, is labeled from \emph{bottom to top} according to the vertical orientation in $R_i$. We denote these sub-rectangles as $\{H^i_j\}_{j=1}^{h_i}$, where $h_i$ is the number of horizontal sub-rectangles of $R_i$. These particular rectangles are known as the \emph{horizontal sub-rectangles of the Markov partition} $\cR$. Similarly, the vertical sub-rectangles of $R_i$, whose interiors are the connected components of $\overset{o}{R_i}\cap f(\overset{o}{R_k})$ for some $k\in \{1,\cdots,n\}$, are labeled from \emph{left to right} according to the horizontal order of $R_i$ as: $\{V^i_l\}_{l=1}^{v_i}$, where $v_i$ is the number of vertical sub-rectangles of $R_i$. These specific rectangles are referred to as the  \emph{vertical sub-rectangles of the Markov partition} $\cR$.

The numbers $\{(h_i,v_i)\}_{i=1}^n$ determine two formal sets:

\begin{eqnarray}\label{Equa: Horizontal vertical sub}
\cH=\{(i,j): i\in\{1,\dots,n\}, j\in\{1,\dots,h_i\},\\
\cV=\{(k,l): k\in\{1,\dots,n\}, l\in\{1,\dots,v_k\}.
\end{eqnarray}
which represent the vertical and horizontal rectangles of the Markov partition, indexed according to the order within each rectangle. The homeomorphism $f$ induces a bijection, $\phi: \mathcal{H} \to \mathcal{V}$, between them. It is defined as $\phi(i,j) = (k,l)$ if and only if $f(H^i_j) = V^k_l$. Here, a first numerical obstruction arises in the set $\{(h_i,v_i)\}_{i=1}^n$:

\begin{equation}\label{Equa: same sum v-i h-i}
\sum_{i=1}^{n}h_i=\sum_{i=1}^{n}v_i.
\end{equation}

Each oriented vertical leaf of a sub-rectangle $H^i_j$ is sent by the function $f$ to a vertical leaf of $V^k_l$. We can compare the orientations of these leaves and define $\epsilon(i,j) = 1$ if the vertical orientations of $f(H^i_j)$ and $V^k_l$ coincide, and $\epsilon(i,j) = -1$ otherwise. This procedure determines the function:
\begin{equation}\label{Equation: same horizontal vertical}
\epsilon\colon \cH \rightarrow \{1,-1\}.
\end{equation}

The \emph{geometric type} of a geometric Markov partition $\cR$ packages all this information into a quadruple:

\begin{equation}\label{Equation: Geometric type}
T(\cR):=(n,\{(h_i,v_i)\}_{i=1}^n,\Phi=(\phi,\epsilon)\colon \cH\to \cV\times\{-1,1\}).
\end{equation}

Later on, we will explain how to define an \emph{abstract geometric type} (Definition ~ \ref{Defi: abstract geometric type}). The set of abstract geometric types will be denoted by $\mathcal{GT}$. However, it is important to note that not every abstract geometric type corresponds to the geometric type of a geometric Markov partition of a pseudo-Anosov homeomorphism. The \emph{pseudo-Anosov class} of geometric types, denoted as $\mathcal{GT}(p\AA)\subset \mathcal{GT}$, consists of those geometric types $T$ for which there exists a generalized pseudo-Anosov homeomorphism with a geometric Markov partition $\mathcal{R}$ whose geometric type is $T$. This discrepancy is the reason for the realization problem in \ref{Prob: Clasification}.

In Chapter \ref{Chapter: TypeConjugacy}, the following theorem is proved, achieving a solution to \textbf{Item} $1.I)$ of the classification problem. It establishes that the geometric type serves as a complete invariant under topological conjugacy.

\begin{theo*}[\ref{Theo: conjugated iff  markov partition of same type}]
Let $f:S_f\rightarrow S_f$ and $g:S_g \rightarrow S_g$ be two generalized pseudo-Anosov homeomorphisms. Then, $f$ and $g$ have a geometric Markov partition of the same geometric type if and only if there exists an orientation preserving homeomorphism between the surfaces $h:S_f\rightarrow S_g$ that conjugates them, i.e., $g=h\circ f\circ h^{-1}$.
\end{theo*}

To prove Theorem~\ref{Theo: conjugated iff markov partition of same type}, we will construct a \emph{combinatorial model} of a pseudo-Anosov homeomorphism $f:S_f \to S_f$ with a Markov partition $\mathcal{R}$ of geometric type $T$. We will then demonstrate that $f$ is conjugated to this model. Essentially, the combinatorial model is obtained by taking the quotient of a sub-shift of finite type under an equivalence relation determined by the geometric type. Let us provide a brief overview of this construction.

 The geometric type $T$ determines the \emph{incidence matrix} $A:=A(T)$ of the Markov partition $\cR$. When the matrix $A$ has coefficients in $\{0,1\}$, it is called a \emph{binary matrix}, and we can define a symbolic dynamical system called, sub-shift of finite type $\sigma\colon \Sigma_A\to\Sigma_A$. Furthermore, it is possible to create a function $\pi_f:\Sigma_A \rightarrow S_f$ that depends on the homeomorphism $f$ and the Markov partition $\cR$ as follows: Let $(w_z)_{z\in \mathbb{Z}}=\underline{w}\in \Sigma_A$.

\begin{equation}\label{Equa: Projection}
\pi_f(\underline{w}):=\cap_{n\in \NN}(\overline{\cap_{z=-n}^n f^{-z}(\overset{o}{R_{w_z}})}).
\end{equation}
 
The map $\pi_f$ is the \emph{projection of the shift} with respect to the geometric Markov partition $(f, \mathcal{R})$. The projection associates each itinerary $\underline{w}=(w_z)_{z\in \mathbb{Z}}$ with a unique point $x$ on the surface $S$ such that $f^z(x)\in R_{w_z}$, i.e $x:= \pi_f(\underline{w}) $ follows the itinerary dictated by the code $\underline{w}$. The definition of $\pi_f$ is based on the following idea: the sets $\overset{o}{Q_n}=\cap_{z=-n}^n f^{-z}(\overset{o}{R_{w_z}})$ are open rectangles contained within the interior of $R_{w_0}$, and their diameters approach zero as $n$ tends to infinity. Consequently, the closure of $\overset{o}{Q_n}$ cannot simultaneously touch the two stable (or unstable) boundaries of $R_{w_0}$. As was said before, the intersection of the closures of all these rectangles, $\cap_{n\in \mathbb{N}} \overline{Q_n}$, corresponds to a unique point $\pi_f(\underline{w}):=x$ that follows the itinerary dictated by $\underline{w}$.

 It's important to mention that this is not the only way to define the projection. For instance, in \cite{fathi2021thurston}, the projection is given by $\pi_f(\underline{w}) = \cap_{z\in \ZZ} f^{-z}(R_{w_z})$. However, if we define $\pi_f$ in this manner, the projection may not always be a function as $\pi_f(\underline{w})$ is not necessarily a single point, except in cases where all the rectangles in $\cR$ are embedded.
 
The projection $\pi_f$ defined in Equation \ref{Equa: Projection} is a continuous, surjective, and finite-to-one map that acts as a semi-conjugation between the sub-shift of finite type $(\Sigma_A, \sigma)$ and the generalized pseudo-Anosov homeomorphism $f$.

If two generalized pseudo-Anosov homeomorphisms $f: S_f \to S_f$ and $g: S_g \to S_g$ have Markov partitions $\cR_f$ and $\cR_g$ of the same geometric type $T$, they will share the same incidence matrix and thus have associated identical sub-shifts of finite type. This implies the existence of projections $\pi_f: \Sigma_A \to S_f$ and $\pi_g: \Sigma_A \to S_g$ that semi-conjugate the shift with the respective homeomorphisms. However, it is not clear how we can establish a direct conjugation between $f$ and $g$ based solely on these projections.

The main challenge in constructing a conjugacy between $f$ and $g$ lies in determining the correspondence between codes in $\Sigma_A$ that represent the same points on the surfaces $S_f$ and $S_g$. This correspondence is particularly challenging because the only known information that relates $f$ to $g$ is the geometric type of their respective partitions. Therefore, our analysis will be based solely on this combinatorial information.

We adopt the approach of introducing an equivalence relation $\sim_T$ in $\Sigma_A$, which is purely determined by the geometric type $T$. This relation defines equivalence classes that correspond to the codes in $\Sigma_A$ that project to the same point under $\pi_f$. The advantage of using the equivalence classes determined by $\sim_T$ is that they do not depend on the specific homeomorphism $f$ or a Markov partition $\cR$. This allows us to compare different generalized pseudo-Anosov homeomorphisms that share the same geometric type for their Markov partitions. In Chapter \ref{Chapter: TypeConjugacy}, we construct this equivalence relation, and Proposition \ref{Prop: The relation determines projections} establishes its key properties.

\begin{prop*}[\ref{Prop: The relation determines projections}]
	
	Let $T$ be a geometric type with an incidence matrix $A:=A(T)$ having coefficients in $\{0,1\}$, and let $(\Sigma_A,\sigma)$ be the associated sub-shift of finite type. Consider a generalized pseudo-Anosov homeomorphism $f:S\rightarrow S$ with a geometric Markov partition $\cR$ of geometric type $T$. Let $\pi_f:\Sigma_A \rightarrow S$ be the projection induced by the pair $(f,\cR)$.

There exists an equivalence relation $\sim_T$ on $\Sigma_A$, algorithmically defined in terms of $T$, such that for any pair of codes $\underline{w},\underline{v}\in \Sigma_A$, they are $\sim_T$-related if and only if their projections coincide, $\pi_f(\underline{w}) = \pi_f(\underline{v})$.
\end{prop*}

Two codes $\underline{v}, \underline{v}\in \Sigma_A$ are $\sim_f$-related if $\pi_f(\underline{v}) = \pi_f(\underline{v})$ (recall that $f:S_f\to S_f$ is pseudo-Anosov). This relation determines the quotient space $\Sigma_f := \Sigma_A/\sim_f$ and is not difficult to see that $\Sigma_f$ is homeomorphic to $S_f$. A further consequence of Proposition \ref{Prop: The relation determines projections} is the following result, from which Theorem \ref{Theo: conjugated iff markov partition of same type} follows.

\begin{prop*}[\ref{Prop:  cociente T}]
Consider a geometric type $T$ in the pseudo-Anosov class, whose incidence matrix $A:=A(T)$ is binary. Let $f:S \rightarrow S$ be a generalized pseudo-Anosov homeomorphism with a geometric Markov partition $\cR$ of geometric type $T$.  Also, let $(\Sigma_A,\sigma)$ be the sub-shift of finite type associated with $A$ and denote by $\pi_f:\Sigma_A \rightarrow S$ the projection induced by the pair $(f,\cR)$. Then, the following holds:
	
	\begin{itemize}
\item The quotient space $\Sigma_T = \Sigma_A/\sim_T$ is equal to $\Sigma_f := \Sigma_A/\sim_f$. Thus, $\Sigma_T$ is a closed and orientable surface.
\item  The sub-shift of finite type $\sigma$ induces a homeomorphism $\sigma_T: \Sigma_T \rightarrow \Sigma_T$ through the equivalence relation $\sim_T$. This homeomorphism is a generalized pseudo-Anosov and topologically conjugate to $f: S \rightarrow S$ via the quotient homeomorphism $[\pi_f]: \Sigma_T = \Sigma_f \rightarrow S$
	\end{itemize}	
\end{prop*}

Let $g$ be any other generalized pseudo-Anosov homeomorphism with a geometric Markov partition of geometric type $T$. Then, $g$ is topologically conjugate to $\sigma_T$. Therefore, $f$ and $g$ are topologically conjugate.

\begin{defi*}[\ref{Defi: symbolic model}]
Let $T$ be a geometric type whose incidence matrix is binary. The dynamical system $(\Sigma_T,\sigma_T)$ is the \emph{combinatorial model of the geometric type} $T$.
\end{defi*}

The dynamical system $(\Sigma_T,\sigma_T)$ is constructed from the geometric type $T$, and it is topologically conjugate to $f$. Therefore, we have addressed \textbf{Item} $1.II$ of the Classification Problem \ref{Prob: Clasification} by demonstrating that the geometric type $T$ allows us to reconstruct a generalized pseudo-Anosov homeomorphism that is topologically conjugate to $f$. 


\subsection{The realization problem: The genus and the impasse}
 
An \emph{abstract geometric type} is defined as a quadruple $T=(n,\{(h_i,v_i)\}_{i=1}^n,\Phi:=(\phi,\epsilon))$, as described in Equation \ref{Equation: Geometric type}, where $n\in\NN{+}$, the numbers $h_i,v_i\in\NN_{+}$ satisfy the equality in Equation \ref{Equation: same horizontal vertical}, and $\phi$ is a bijection between the following sets.
 $$
 \cH(T):=\{(i,j)\in \NN: 1\leq i \leq n \text{ and } 1 \leq j \leq h_i\}
 $$
 and
 $$
 \cV(T):=\{(k,l)\in \NN: 1\leq k \leq n \text{ and } 1 \leq l \leq v_i\},
 $$
  additionally, $\epsilon:\cH(T) \rightarrow \{1,-1\}$ is an arbitrary function. These objects together form the set of \emph{abstract geometric types} denoted by $\mathcal{G}\mathcal{T}$.

As mentioned previously, an abstract geometric type $T$ belongs to the \emph{pseudo-Anosov class} of geometric types if there exists a generalized pseudo-Anosov homeomorphism  with a geometric Markov partition of geometric type $T$. The pseudo-Anosov class is denoted by $\mathcal{G}\mathcal{T}(p\AA)$. It is evident that $\mathcal{G}\mathcal{T}(p\AA)\subset \mathcal{G}\mathcal{T}$, and \textbf{Item} $II$ of Problem \ref{Prob: Clasification} seeks an algorithmic condition to determine whether an abstract geometric type belongs to the pseudo-Anosov class. Our approach to this problem is to translate it into the context of \emph{basic pieces} of Smale's surface diffeomorphisms.

The theory of geometric types was first introduced in the 1998 book \cite{bonatti1998diffeomorphismes} by Christian Bonatti, Rémi Langevin, and Emmanuelle Jeandenans. This work was followed by François Béguin's Ph.D. thesis \cite{beguin1999champs} and subsequent articles \cite{beguin2002classification} and \cite{beguin2004smale} published between 1999 and 2004. Together, these works provide a complete classification (in the sense of Problem 1) of the basic saddle-type pieces of Smale's diffeomorphisms on surfaces, as well as their invariant neighborhoods but excluding hyperbolic attractors and repellers. The definition of the geometric Markov partition and its geometric type for diffeomorphisms is very similar to the one given for pseudo-Anosov homeomorphisms, with the advantage that the Markov partition for the basic pieces consists of disjoint rectangles.  To apply the results obtained by these authors, our initial task is to establish a connection between Markov partitions for pseudo-Anosov homeomorphisms and those defined by saddle-type basic pieces on surfaces. Let's explore some key concepts in our discussion.

Consider a geometric type $T$, and let $A$ be its incidence matrix with coefficients in ${0,1}$. We say that $A$ is \emph{mixing} if there exists a positive integer $n$ such that each coefficient $a_{ij}^{(n)}$ of the matrix $A^n$ is positive. When the incidence matrix $A$ is mixing, we can define a \emph{concretization} of $T$ \footnote{In fact, a weaker condition is needed, namely, that the geometric type does not have double $s,u$-boundaries. However, this property is deduced from the fact that $A(T)$ is mixing.}.

 A concretization of $T$ is represented by a pair $(\cR,\phi)$, where $\cR$ is a collection of disjoint rectangles $\{R_i\}_{i=1}^n$ (we can take $R_i:=[0,1]^2 \times \{i\}$), and $\phi$ is a piecewise-defined function. The function $\phi$ maps a collection of disjoint horizontal sub-rectangles $H_j^i\subset R_i$ (for each $(i,j)\in \cH(T)$) to a family of disjoint vertical sub-rectangles $V_l^k\subset R_k$ (for each $(k,l)\in \cV(T)$). Within each rectangle, $\phi$ is a homeomorphism that preserves the orientation and trivial foliations of the rectangles, furthermore, $\phi$ follows the combinatorial pattern specified by $T$. Specifically, if $\phi(i,j)=(k,l)$, then, $f(H_j^i)=V_l^k$,  and the orientation of the vertical foliations is preserved by $\phi$ if and only if $\epsilon(i,j)=1$.

A concretization provides a visual representation of the geometric type $T$, but it alone cannot determine whether $T$ corresponds to the geometric type of a geometric Markov partition of a saddle-type basic piece. However, if $T$ does correspond to the geometric Markov partition of a saddle-type basic piece of a diffeomorphism on a compact surface $S$ (with finite genus), it is crucial that such a surface contains the rectangles of a realization of $T$  and their images under the iterations of such diffeomorphism. 

To establish that $T$ is a geometric type corresponding to a geometric Markov  partition of a saddle-type basic pieces in $S$, it is necessary for a specific surface with boundaries and corners, known as an $m$-realizer of $T$ (described in Definition \ref{Defi: m realizer }), to be embedded in the surface $S$. We will now describe the main properties and definitions related to the $m$-realizer that are necessary to state our results.

Consider a concretization $(\{R_i\}_{i=1}^n,\phi)$ of $T$.  The union of these rectangles will be denoted as $\cR:=\bigcup_{i=1}^n R_i$.
The $m$-\emph{realizer} of $T$ is defined as the surface (with vertices and boundaries):
$$
\cR_m=\cup_{i=0}^m \cR\times \{i\}/(x,i)\sim(\phi^{-1}(x),i+1)
$$

together with the function $\phi_m(x,i)=(\phi(x),i+1)$ whenever it makes sense.

The surface $\mathcal{R}_m$ has vertices and corners and supports $m-1$ iterations of the diffeomorphism $\phi_m$. These $m$-realizers determine a non-decreasing sequence $\{g_m\}_{m\in \mathbb{N}}$, where $g_m$ is the genus of the surface $\mathcal{R}_m$. The supremum of this sequence is defined as the \emph{genus of} $T$, denoted as $\text{gen}(T)$. The geometric type $T$ has \emph{finite genus} if $\text{gen}(T)<\infty$. Having finite genus is the first obstruction in realizing $T$ as the geometric type of a basic piece, and subsequently as the geometric type of a pseudo-Anosov homeomorphism.

The \emph{impasse property} of Proposition \ref{Prop: pseudo-Anosov iff basic piece non-impace} is quite technical (Definition \ref{Defi: Impasse geo}). For now, we won't explore it further to stay focused on the main exposition. The following Proposition is crucial for solving the realization problem. 

 \begin{prop*}[ \ref{Prop: pseudo-Anosov iff basic piece non-impace}]
	Let $T$ be an abstract geometric type. The following conditions are equivalent.
	\begin{itemize}
		\item[i)] The geometric type $T$ is realized as a mixing basic piece of a surface Smale diffeomorphism without impasse.
		\item[ii)] The geometric type $T$ is in the pseudo-Anosov class.
		\item[iii)] The geometric type $T$ satisfies the following properties:
		\begin{enumerate}
			\item  The incidence matrix $A(T)$ is mixing
			\item The genus of $T$ is finite
			\item $T$ does not have an impasse.
		\end{enumerate}
	\end{itemize}
\end{prop*}

To determine the realization of the geometric type in Problem \ref{Prob: Clasification}, it is necessary to demonstrate that each property listed in \textbf{Item} $iii)$ of Proposition \ref{Prop: pseudo-Anosov iff basic piece non-impace} can be determined algorithmically. This is achieved in the Preliminaries Chapter, specifically in Subsection \ref{Sec: Bonatti-Langevin theory}, where Proposition \ref{Prop: mixing+genus+impase is algorithm} establishes that the properties of mixing, finite genus, and impasse can be expressed in a purely combinatorial manner. This allows us to conclude the following theorem.

\begin{theo*}[\ref{Theo: caracterization is algoritmic}]
	There exists a finite algorithm that can determine whether a given geometric type $T$ belongs to the pseudo-Anosov class. Such algorithm requires calculating at most $6n$ iterations of the geometric type $T$, where $n$ is the first parameter of $T$.	
\end{theo*}

This resolves the problem of realization that we had set out to address.

\subsection{Algorithmic Decidability: The Béguin's Algorithm and Formal DA$(T)$}

To address \textbf{Item} $III$ of Problem \ref{Prob: Clasification} regarding algorithmic decidability, we utilize the \emph{Béguin's algorithm}. This algorithm, originally developed in \cite{beguin2004smale}, provides a solution for determining whether two geometric types represent conjugate basic pieces. Its essence can be summarized as follows:

\begin{theo*}[The Béguin's Algorithm]
Let $f$ be a Smale surface diffeomorphism and $K$ be a saddle-type basic piece with a geometric Markov partition $\mathcal{R}$ of geometric type $T=(n,\{(h_i,v_i)\}_{i=1}^n,\Phi_T)$. The steps of the algorithm are as follows:
\begin{enumerate}
\item Begin by defining the \emph{primitive geometric types of order} $n$ of $f$, denoted as $\cT(f,n)$, for all $n$ greater than a certain constant $n(f)\geq 0$.
\item Prove the existence of an upper bound $O(T)$ for $n(f)$ in terms of the number $n$ in $T$, i.e. $n(f)\leq O(T)$.
\item For every $n>n(f)$, there exist an algorithm to compute all the elements of $\cT(f,n)$ in terms of $T$.
\item Let $g$ be another Smale surface diffeomorphism and $K'$ be a saddle-type basic piece of $g$ with a geometric Markov partition of geometric type $T'$. Choose $n$ to be greater than or equal to the maximum of $O(T)$ and $O(T')$. After applying Step 3 of the algorithm, we obtain two finite lists of geometric types, $\cT(f,n)$ and $\cT(g,n)$. These lists will be equal if and only if $f$ is topologically conjugate to $g$ in some invariant neighborhoods of $K'$ and $K'$.
\end{enumerate}
\end{theo*}

As can be observed, the calculation of the upper bound $O(T)$ for $n(f)$ and the procedure in item 3 are the essential steps of the algorithm. We will revisit them at the end of this subsection. 

For now, let us explain how we address our problem by shifting the analysis to the formal \textbf{DA}$(T)$ (Definition \ref{Defi: Formal DA}). Consider a geometric type $T$ in the pseudo-Anosov class. According to Proposition \ref{Prop: pseudo-Anosov iff basic piece non-impace}, $T$ has finite genus. Interestingly, this condition is necessary and sufficient for $T$ to be realized as a basic piece of a Smale surface diffeomorphism. The \emph{formal derived from Anosov} of $T$ refers to this realization and is represented by the triplet:

\begin{equation}\label{Equa: DA of T}
 \textbf{DA}(T):=(\Delta(T),K(T),\Phi_T),
\end{equation}

The components are as follows: $\Delta(T)$ is a compact surface with boundary of finite genus, $\phi_T$ is a Smale diffeomorphism defined on $\Delta(T)$; and $K(T)$, the unique nontrivial saddle-type basic piece of $(\Delta(T),\Phi_T)$ with a Markov partition of geometric type $T$.  We call $\Delta(T)$ the \emph{domain} of the basic piece $K(T)$. It is a profound result of Bonatti-Langevin \cite[Propositions 3.2.2 and 3.2.5, Theorem 5.2.2]{bonatti1998diffeomorphismes} that the \emph{DA}$(T)$ is unique up to conjugacy and serves as a bridge between the realm of pseudo-Anosov homeomorphisms and the basic pieces of Smale surface diffeomorphisms.

Let $f$ a generalized pseudo-Anosov homeomorphism and $\cR$ a geometric Markov partition, a periodic point of $f$ that is in the stable (unstable) boundary of $\cR$ is a $s$-boundary point ($u$-boundary point). Let $f$ and $g$ be two generalized pseudo-Anosov homeomorphisms, and let $\cR_f$ and $\cR_g$ be geometric Markov partitions with geometric types $T_f$ and $T_g$, respectively.
 Let  $p(f)$ be the maximum period of  periodic points of $f$ on the boundary of $\cR_f$, and $p(g)$ be the maximum period of periodic points of $g$ on the boundary of $\cR_g$. Take $p = \max\{p(f), p(g)\}$.  We are going to construct  geometric Markov partitions $\cR_{\cR_f,\cR_g}$ of $f$ and $\cR_{\cR_g,\cR_f}$ of $g$ such that:

\begin{itemize}
\item The set of periodic boundary points of $\cR_{\cR_f,\cR_g}$ ( $\cR_{\cR_g,\cR_g}$ ) coincides with the set of periodic points of $f$ ( $g$ ) whose period is less or equal than $p$.
\item Every periodic boundary point in $\cR_{\cR_f,\cR_g}$ is a \emph{corner point} i.e. is $u$ and $s$-boundary.
\end{itemize}

 These refined partitions are referred to as the \emph{compatible refinements} of $\cR_f$ and $\cR_g$, and their respective geometric types are denoted $T_{(\cR_f,\cR_g)}$ and $T_{(\cR_g,\cR_f)}$.

Two geometric types, $T_1$ and $T_2$, which can be realized as basic pieces, are considered to be \emph{strongly equivalent} if there exists a Smale diffeomorphism $f$ of a compact surface and a nontrivial saddle-type basic piece $K$ of $f$, such that $K$ has a geometric Markov partition of geometric type $T_1$ as well as a geometric Markov partition of geometric type $T_2$.  We have proven the following corollary , which establishes a connection between the formal \textbf{DA} of the geometric types of the compatible refinements and the underlying pseudo-Anosov homeomorphisms.
 
\begin{coro*}\ref{Coro: equivalence pA and DA}
	Let $f$ and $g$ be generalized pseudo-Anosov homeomorphisms with geometric Markov partitions $\cR_f$ and $\cR_g$ of geometric types $T_f$ and $T_g$, respectively. Let $\cR_{f,g}$ be the joint refinement of $\cR_f$ with respect to $\cR_g$, and let $\cR_{g,f}$ be the joint refinement of $\cR_g$ with respect to $\cR_f$, whose geometric types are $T_{f,g}$ and $T_{g,f}$, respectively.  Under these hypotheses: $f$ and $g$ are topologically conjugated through an orientation preserving homeomorphism if and only if $T_{f,g}$ and $T_{g,f}$ are strongly equivalent.
\end{coro*}

  In the proof we use the formal \textbf{DA} in the next manner: If $T_{(\cR_f,\cR_g)}$ and $T_{(\cR_g,\cR_f)}$ are strongly equivalent, then \textbf{DA}$(T_{(\cR_f,\cR_g)})$ and \textbf{DA}$(T_{(\cR_g,\cR_f)})$ are conjugate through a homeomorphism $h$. This homeomorphism induces a conjugation between $f$ and $g$.  Proposition \ref{Prop: preimage is Markov Tfg type} establish  that if $f$ and $g$ are topologically conjugate, then \textbf{DA}$(T_{(\cR_f,\cR_g)})$ has a Markov partition of geometric type $T_{(\cR_g,\cR_f)}$. Hence they  are strongly equivalent.

 Finally we use the Béguin's algorithm to determine if $T_{(\cR_f,\cR_g)}$ and $T_{(\cR_g,\cR_f)}$ are strongly equivalent. This provides a solution for the fist part of Item $III$ in Problem \ref{Prob: Clasification} through the following theorem.
  
  \begin{theo*}[\ref{Theo: algorithm conjugacy class}]
Let $T_f$ and $T_g$ be two geometric types within the pseudo-Anosov class. Assume that $f: S \rightarrow S_f$ and $g: S_g \rightarrow S_g$ are two generalized pseudo-Anosov homeomorphisms with geometric Markov partitions $\cR_f$ and $\cR_g$, having geometric types $T_f$ and $T_g$ respectively.  We can compute the geometric types $T_{f,g}$ and $T_{g,f}$ of their joint refinements through the algorithmic process described in Chapter \ref{Chap: Computations}, and the homeomorphisms $f$ and $g$ are topologically conjugated by an orientation-preserving homeomorphism if and only if the algorithm developed by Béguin determines that $T_{f,g}$ and $T_{g,f}$ are strongly equivalent.
  \end{theo*}
 

The second part of \textbf{Item} $1.III$ in Problem \ref{Prob: Clasification}, requires the identification of a finite and well-defined family of geometric types, denoted as $\Theta[f]$, for a given pseudo-Anosov homeomorphism $f$. This family should be computable based on any geometric type associated with $f$, enabling the computation of any other geometric type realized by $f$. Furthermore, it is desirable for this family to be as small as possible while fulfilling these requirements.
In this direction, we have obtained some partial results, which are outlined below.

Let $p$ and $q$ be two periodic points of $f$. Here, $F^s(p)$ represents a stable separatrice of $p$, and $F^u(q)$ represents an unstable separatrice of $q$. Proposition \ref{Prop: Recipe for Markov partitions} presents a topological algorithm that, given any point $z\in F^s(p)\cap F^u(q)$, generates a corresponding Markov partition $\cR_z$. The significance of this algorithm lies in the fact that, under certain conditions, the geometric type of $\cR_z$ remains constant along the orbit of $z$. To obtain a canonical family of geometric types for $f$, it becomes necessary to identify a distinguished set of points contained in $F^s(p)\cap F^u(q)$.

We begin by considering $p$ and $q$ as singular points of $f$, which are the more distinguished points that we can take. A point $z\in F^s(p)\cap F^u(q)$ is called \emph{first intersection point} of $f$ if the stable segment $(p,z]^s\subset F^s(p)$ and the unstable segment $(q,z]^u\subset F^u(q)$ intersect at a single point, $(p,z]^s\cap (q,z]^u={z}$ (see Definition \ref{Defi: first intersection points}). There exist a finite number of orbits of first intersection points as proved in Proposition \ref{Prop: Finite number first intersection points}. Clearly, a conjugation between pseudo-Anosov homeomorphisms maps first intersection points to first intersection points (see Theorem \ref{Theo: Conjugates then primitive Markov partition}) and therefore are invariant under conjugation. This property allows us to define: a natural number that is constant for the entire conjugacy class of $f$, $n(f)\geq 0$ , such that for every integer $n\geq n(f)$ and every first intersection point $z$, there exists a Markov partition $\cR(n,z)$ of $f$ associated to these two parameters.  The Markov partitions constructed in this manner are referred to as \emph{primitive Markov partition} of order $n$ of $f$. The collection of all these partitions is denoted by $\cM(f,n)$. They have the following property.

\begin{coro*}[\ref{Coro: Finite orbits of primitive Markov partitions}]		
	Let $f$ be a generalized pseudo-Anosov homeomorphism, and let $n\geq n(f)$. Then, there exists a finite but nonempty set of orbits of primitive Markov partitions of order $n$.
\end{coro*}

Our Theorem \ref{Theo: Conjugated partitions same types} states that a Markov partition and its iterations have the same geometric type. As a consequence, we have the following Theorem

\begin{theo*}[\ref{Theo: finite geometric types}]
	Let $f$ a generalized pseudo-Anosov homeomorphism. The set 
	$$
	\cT(f,n) = \{T(\cR): \cR \in \cM(f,n)\}
	$$ 
	is finite for every $n \geq n(f)$. These geometric types are referred to as \emph{primitive types} of order $n$ for $f$.
\end{theo*}


Each of these families is a finite invariant under conjugation. If $f$ and $g$ are conjugate generalized pseudo-Anosov homeomorphisms, Corollary \ref{Coro: n(f) conjugacy invariant} establishes that $n(f)=n(g)$ and its set of \emph{canonical types} coincides too:
$$
\Theta[f]:=\cT(f,n(f))=\cT(g,n(g))=\Theta[g]
$$
Hence this  is a finite total invariant of the conjugation of $f$. They represent the finite presentation of \textbf{Item} $1.III$ in Problem \ref{Prob: Clasification}. Among the families of \emph{primitive types}, they are the minimum with  respect to the parameter $n$. 

At present, we are unable to compute the invariant $\Theta[f] := \cT(f,n(f))$ from any other geometric type or, in general, all geometric types associated with a conjugacy class as required by the \textbf{III} of Problem \ref{Prob: Clasification}. However, we can compute primitive geometric types of sufficiently large order. We believe it would be interesting to expand our theory and compute $\Theta([f])$, but this is still a work in progress.

We have concluded the presentation of the results in this thesis. Next, we will provide a brief survey of the results and perspectives surrounding the classification of pseudo-Anosov homeomorphisms. This is not an exhaustive bibliography, but rather an attempt to position our work within the context of the field.
\section*{Other classifications} 

The problem of classifying pseudo-Anosov homeomorphisms has been considered from different perspectives. In this section, we will discuss various results and approaches, in order to place our work within the current landscape of the mathematics. The conjugacy class and isotopy classes refer to the equivalence classes determined by topological conjugacy and isotopy, respectively.

\subsection{Lee Mosher: The classification of pseudo-Anosov}

In his 1983 thesis \cite{mosher1983pseudo} and the 1986 article \cite{mosher1986classification}, Lee Mosher classifies pseudo-Anosov homeomorphisms on a given closed and oriented surface $S$ equipped with a finite set of marked points $P$. To achieve this, Mosher considers the action of the mapping class group $\mathcal{M}(S,P)$ on a directed graph $\tilde{\mathcal{G}}(S,P)$, where each vertex represents a cellular decomposition of the surface with a single vertex and a distinguished sector around such vertex. The directed edges of $\tilde{\mathcal{G}}(S,P)$ correspond to elementary moves that transform one decomposition into another.

The mapping class group acts freely on the graph $\tilde{\mathcal{G}}(S,P)$, and the quotient yields a graph $\mathcal{G}(S,P)$ whose fundamental group is the mapping class group. Consequently, an element of the mapping class group can be represented by a cycle of directed edges in this graph. Each connected component of $\mathcal{G}(S,P)$ can be labeled using what Mosher refers to as \emph{polygonal types}, which essentially correspond to the list of polygons involved in the decomposition. It is worth noting that for a fixed genus of $S$ and a fixed number of marked points, there exists a finite number of polygonal types. Therefore, $\mathcal{G}(S,P)$ has a finite number of connected components.

Given a marked point-preserving pseudo-Anosov homeomorphism that fixes each of its  singularities and each of the separatrices of its singularities, Mosher define its \emph{singularity type}, which is determined by the number of interior singularities with $n$-prongs ($n \geq 3$) and marked singularities with $m$-prongs ($m \geq 1$). For each singularity type, he assigns a finite set of polygonal types of cellular decomposition to the pseudo-Anosov. His main result asserts that, for these polygonal types of cellular decomposition, the isotopy class of the pseudo-Anosov corresponds to a finite and explicit number of loops (cycles) on the graph $\mathcal{G}(S,P)$. Each of these loops has a finite combinatorial description, resulting in a finite number of canonical combinatorial descriptions for the pseudo-Anosov, they are called  \emph{pseudo-Anosov invariants.}

Thus, given two marked point-preserving pseudo-Anosov homeomorphisms that fix each of the separatrices of the singularities, they are conjugate if and only if their pseudo-Anosov invariants are the same. Mosher himself emphasizes that he does not know if his classification extends to a classification of all pseudo-Anosov homeomorphisms (without the assumption of fixing the separatrices).

As we have seen, Mosher's perspective and the one presented here are quite different. The distinction starts with the definition of a pseudo-Anosov homeomorphism. For Mosher, it is an isotopy class characterized by its action on a cell decomposition. In this thesis, we define a pseudo-Anosov homeomorphism through a Markov partition.  On which surface does it lay? What types of singularities does it possess? This information can be deduced from the geometric type of the partition, using a simple algorithm. Applying Mosher's result to a pseudo-Anosov homeomorphism defined by a geometric type of Markov partition seems to be a challenging task. However, a preliminary corollary, which we cite below, allows us to determine the singularity type of the pseudo-Anosov homeomorphism, which is a necessary piece of information for Mosher's result.

\begin{coro*}[\ref{Coro: algoritmo genero}]
	Let $T$ be a geometric type of the pseudo-Anosov class, and let $f:S\rightarrow S$ be a generalized pseudo-Anosov homeomorphism with a geometric Markov partition of type $T$. There exists a finite algorithm that, given $T$, determines the number of singularities of $f$ and the number of stable and unstable separatrices at each singularity and subsequently the genus of $S$.
\end{coro*}

 \subsection{The Bestvina-Handel Algorithm}

Let $S$ be a closed and oriented surface with marked points $P$. The objective of Bestvina and Handel in the article \cite{bestvina1995train} is to provide an algorithmic proof of the Handel-Thurston classification theorem for homeomorphisms in Hom$_+(S, P)$. In other words, given a homeomorphism $f$, their goal is to determine whether $f$ belongs to the isotopy class of a reducible, periodic, or pseudo-Anosov homeomorphism.

The article presents two main theorems which we will discuss below. These theorems are initially proven for surfaces with a single marked point and later extended to surfaces with a finite set of marked points. To facilitate the understanding of their results, we will provide an intuitive introduction to the key concepts involved.

A \emph{fibered surface} is a compact surface $F\subset S$ with  boundary that can be decomposed into arcs and polygons. These polygons are called \emph{junctions}, and the regions foliated by arcs are referred to as \emph{stripes}. We say that $F$ is \emph{carried} by a homeomorphism $f$ if $f$ maps $F$ into itself (up to isotopy) and sends the set of junctions to the set of junctions, as well as every stripe to the union of stripes.

If we collapse the fibered surface $F$, we obtain a graph $\cG$ where the vertices are the junctions and the edges are the stripes of $F$. If $F$ is carried by $f$, the homeomorphism induces a map on the graph $g:\cG\rightarrow \cG$ which sends vertices to vertices, and each 
edge to an edge-path. The map $g$ and the homeomorphism $f$ are irreducible if the incidence matrix of $g$ is irreducible and $\cG$ have no vertex of valence $1$ or $2$. Finally the growth rate of the map $g:\cG \to \cG$ is defined as the Perron-Frobenius eigenvalue of the incidence matrix of $g$, denoted by $\lambda(F,f)\geq 1$.

If every iteration of $g$ sends each edge of $\cG$ along a path without backtracking, and satisfy another technical condition of how the edge path cross every vertex, we say that $g$ is \emph{efficient} (see  \cite[Lemma 3.1.2]{bestvina1995train}).

The first theorem of the authors establishes a dichotomy for $f$: it is either periodic or it carries an efficient fibered surface. This dichotomy forms the basis of their subsequent steps in the classification algorithm.
 
 \begin{theo*}[B-H, Theorem 3.1.3 ]
Every homeomorphism on a once punctured surface is isotopic, relative to the puncture, to one that is either reducible or carried by an efficient fibered surface.
 \end{theo*}

The proof follows an algorithmic approach. A spine of $S$ is a CW-complex $\cG$ contained in $S$ such that $S$ deformation retracts to $\cG$. The authors suggest starting with a regular neighborhood $F$ of a spine $G$ that naturally has a fibered surface structure. After that they perform an isotopy of $f$ to ensure that $F$ is carried by the homeomorphism. Then Bestvina and Handel follow the philosophy: the growth rate of the map $g$ is a measure of how close we are to an efficient fibered surface. The smaller the growth rate, the more efficient the carrier.

The algorithm described in \cite[Section 3]{bestvina1995train} is used to prove Theorem 3.1.3. It involves modifying the homeomorphism $f$ and the fibered surface $F$. Each step of the algorithm either reduces the growth rate of the induced map or maintains it at the same level as the previous step. The algorithm includes isotopies of $f$ and collapsing certain sub-graphs of $\mathcal{G}$, resulting in natural modifications of $F$.

It is proven that the algorithm terminates in finite time if $f$ is reducible, or it continues until an efficient carried surface is obtained. The algorithm terminates because the set of growth rates of non-negative integer matrices of bounded size is a discrete subset of $[1, \infty)$. As we decrease the growth rate of $g$, either $g$ becomes efficient, or we find a reduction and the algorithm stops. Next, Bestvina and Handel present the following result:

\begin{theo*}[B-H, Theorem 3.1.4 ]
If a homeomorphism $f:S\to S$ of a once-punctured surface is carried by an efficient fibered surface, then $f$ is isotopic, relative to the puncture, to a homeomorphism that is either periodic, reducible, or pseudo-Anosov.
\end{theo*}

 To prove Theorem 3.1.4, Bestvina and Handel utilize the concept of \emph{train tracks}. Like the  fibered surface is efficient, the map $g$ has a incidence matrix with growth rate greater than or equal to 1. In Section 3.3 of \cite{bestvina1995train}, the proof is divided into two cases
  
  \begin{itemize}
  	\item If $\lambda(F,f)=1$, it is observed that the homeomorphism is isotopic to a periodic one.
  	\item If $\lambda(F,f)>1$, either they can construct an $f$-invariant train track in finite time, or they find a reduction of $f$ (in which case $f$ is reducible).
  \end{itemize}

The second point is addressed in Section 3.4. There, they assume the existence of the $f$-invariant train track and use it to construct a Markov partition for the homeomorphism. Subsequently, they obtain the invariant foliations of $f$, establishing that $f$ is isotopic to pseudo-Anosov.

In their examples, the authors \cite[Section 6]{bestvina1995train} illustrate the initial step of the algorithm, which involves decomposing $f$ into Dehn twists. This serves as the input information for their algorithm.

An interesting question arises: given a geometric type $T$ in the pseudo-Anosov class, how can we find a decomposition of the underlying pseudo-Anosov homeomorphism into Dehn twists? Conversely, if we have a decomposition of $f$ into Dehn twists, can we calculate a geometric type associate of $f$? We think  is possible that the machinery of train-tracks, utilized to construct the Markov partition of $f$, could provide a solution to the last question.

Lastly, the authors do not explicitly discuss the problem of conjugacy. It is not clear to us how their algorithm can be used to determine if two pseudo-Anosov homeomorphisms are topologically equivalent, although it is possible that with further investigation, this could be achieved.

 \subsection{Jerome Los's Approach: The $n$-Punctured Disk and the Braid Group }
 
 The intimate relation between the mapping class group of the $n$-punctured disk and the Braid groups is exploited by Jerome Los in \cite{los1993pseudo} to provide a finite algorithm for determining whether an isotopy class (of the punctured disk) contains a pseudo-Anosov element.
 
 \begin{theo}[Los, J. \cite{los1993pseudo}]
Given a braid $\beta\in B_n$ defining an isotopy class $[f_{\beta}]$ of orientation-preserving homeomorphisms of the punctured disc $D_n$, there exists a finite algorithm which enables one to decide whether $[f_{\beta}]$ contains a pseudo-Anosov element. In the pseudo-Anosov case, the algorithm gives an invariant train track and the dilatation factor.
 \end{theo}

The train tracks are graphs that were introduced to simplify the description of the measured foliations on a surface. The author considers an even simpler graphs called \emph{skeleton graph}. The spirit of their algorithm is to consider all the possible maps induced on these graphs by some elements of the isotopy class $[f_{\beta}]$.

J. Los is able to simplify the problem by reducing it to a finite set of skeletons and a finite set of induced maps. Each induced map on a specific graph is represented by a non-negative integer matrix known as the incidence matrix. The objective, similar to Bestvina and Handel, is to find a graph and an induced map whose incidence matrix has the minimum Perron-Frobenius eigenvalue among all possible graphs and induced maps.

Furthermore, when considering a homeomorphism $f$, J. Los identified three properties of an $f$-invariant train track that are crucial for $f$ to be isotopic to pseudo-Anosov (refer to \cite[Theorem 3.4]{los1993pseudo}). He then follows a finite number of steps, where he either finds an invariant train track for $f$ that satisfies the necessary conditions for the pseudo-Anosov class or determines that the class is not pseudo-Anosov.

The algorithm takes as input a braid $\beta \in B_n$, which is assumed to be given in \emph{minimal length form}. The author highlights the existence of an algorithm that, given a braid, can find a representative with minimal length using the Garside algorithm. However, it is noted that this method is not efficient due to its exponential complexity in relation to the initial length.

 By using a train track, it is possible to construct a Markov partition for $f$. Therefore, one consequence of Jerome's result is the existence of a coding for every pseudo-Anosov homeomorphism. However, as we will see, this coding is not sufficient to distinguish a conjugacy class from others. Nonetheless, it is a valuable corollary arising from his constructions.

  \subsection{Other Algorithms}
  
 There are certainly other classifications in the literature, but most of them focus on the classification modulo isotopy and its refinements.  However for the pseudo-Anosov case this is equivalent since two pseudo-Anosov homeomorphisms are isotopic if and only if they are topologically conjugated by a homeomorphism which is isotopic to the identity, the difference lies in the kind of arguments used for the classification. In the case of isotopy classification the authors use properties of foliations and homeomorphisms that are invariant under continuous deformations while we use arguments that highlight properties that are invariant under conjugation. 
 
 For instance, in the 1996 article by Hamidi et al. \cite{hamidi1996surface}, they present an algorithm that starts with a homeomorphism $f$ expressed as a combination of Dehn twists. The algorithm then determines whether $f$ is reducible or not. If it is reducible, it describes the isotopy classes fixed by $f$. If $f$ is pseudo-Anosov, the algorithm provides the contraction factor and invariant train tracks. Finally, if $f$ is periodic, it gives its period. These results offer additional insights and contribute to the overall understanding of surface homeomorphisms, up isotopy.
 
\section*{Some questions and  projects.}

\subsection{Repellers and attractors:} After the work of Bonatti, Langevin, Jeandenans and B\`eguin, a complete classification of saddle-type saturated sets for Smale diffeomorphisms on surfaces has been achieved, it remained to understand those hyperbolic sets containing a hyperbolic attractor or repeller. However, it has long been known that there is a deep relationship between hyperbolic attractors on surfaces and pseudo-Anosov homeomorphisms and we have used the framework they developed around geometric types to achieve a classification of pseudo-Anosov homeomorphisms . 
We believe that following this perspective we will be able to complement the developments made by this research group and thus complete the classification of all Smale diffeomorphisms on surfaces, considering this time all their saturated sets, including their attractors and repellers.

\subsection{Relations` in the pseudo-Anosov class:} As we can see, there are multiple geometric types that represent the same conjugacy class of pseudo-Anosov homeomorphisms. Another problem that appears naturally in our project is to provide an example of two geometric types in the pseudo-Anosov class with the same incidence matrix but do not represent topologically conjugate pseudo-Anosov homeomorphisms.  In a more general situation, we would like to describe an easy graphical or combinatorial operation that allows us to transition from one geometric type $T$ to any other that represents the same homeomorphism. 
We have the following construction. The permutation group in $n$-elements and the group $\{1, -1\}$ act on the set of geometric types with $n$-rectangles and $\alpha=\sum_{i=1}^{n}h_i$, denoted as $\mathcal{T}(n,\alpha)$. This action corresponds to changing the indexing of the rectangles and reversing the orientation of each of them when the geometric type is in the pseudo-Anosov class. Of course, the orbit of each geometric type $T$ under this action is associated with the same conjugacy class of $T$.
 However, are there other operations in the set $\mathcal{T}(n,\alpha)$ that relate geometric types in the pseudo-Anosov class, representing topologically conjugate pseudo-Anosov homeomorphisms? Can we use these quotients to estimate the number of different conjugacy classes for each $(n,\alpha)$?
 
 \subsection{Computational implementation.} Finally, we want to point out that we have written this thesis with the intention of its computational implementation. Therefore, properties such as mixing and finite-genus impasse are presented with formulas that only involve the geometric type. Additionally, we have attempted to provide some bounds for the algorithms in terms of iterations of a geometric type. However, we have not yet completed the programming of these algorithms. We believe it is necessary to do so in order to calculate geometric types associated with specific families of pseudo-Anosov homeomorphisms,  such as those obtained from branched covers or compositions of Dehn twists, for example.

\chapter{Preliminaries}\label{Chapter: Preliminares}

\pagenumbering{arabic}

\section*{Some general considerations.}

In this thesis, a surface $S$ will always refer to a connected, closed, and oriented two-dimensional topological manifold. In some cases, we may assume additional regularity or metric structure, such as $S$ being a $C^\infty$ Riemannian manifold. It is worth noting that the surface has finite genus and no boundary components.

Our intention is to study homeomorphisms under topological conjugation. The homeomorphisms we examine are orientation preserving and exhibit a nice property: any homeomorphism that is topologically conjugate to a generalized pseudo-Anosov homeomorphism  is also a generalized pseudo-Anosov homeomorphism. This closure property is intrinsic to topological conjugacy. This does not occur when considering the isotopy relation, where a canonical model of a pseudo-Anosov homeomorphism can be deformed into another homeomorphism that is no longer pseudo-Anosov . This will be advantageous when analyzing the structure of their conjugacy classes.

The concepts presented in this Chapter are derived from \cite{farb2011primer}, \cite{fathi2021thurston}, \cite{hiraide1987expansive}, and \cite{bonatti1998diffeomorphismes}. We have adapted them to our specific context. Pseudo-Anosov homeomorphisms, initially defined in terms of singular foliations and transverse measures, serve as the foundation for our discussion and are introduced first. We then define geometric Markov partitions and its corresponding geometric type. Additionally, we provide an overview of the classical hyperbolic theory of surface diffeomorphisms. Finally, we introduce the Bonatti-Langevin theory of geometric types, originally developed for Smale diffeomorphisms on surfaces. This theory will be employed to determine the realizability and equivalence of geometric types in our study.

\section{Generalized pseudo-Anosov homeomorphisms}

\subsection{Measured foliations with Spines}

A singular foliation in the classical sense is a foliation of $S$ except for a finite number of $k$-prong singularities, where $k\geq 3$. Our intention is to study a more general version of such foliations that allows for the existence of \emph{spine}-type singularities. The following discussion is based on \cite{hiraide1987expansive}.  For us, the natural numbers $\NN$ include zero, and $\NN_+$ denotes the positive integers.

The unitary rectangle in $\mathbb{C}$ is given by:
$$
\cD_2:=\{(x,y)\in \CC: \vert \text{Re}(x)\vert <1 \text{ and } \vert \text{ Im }(y) \vert <1\}.
$$
For all $z\in \mathbb{N}_{+}$, let $\pi_p:\mathbb{C}\rightarrow \mathbb{C}$ be given by $\pi_p(z)=z^p$. Then define:
$$
\cD_1:=\pi_2(\cD_2).
$$
For all $p\in \NN$ take:
$$
\cD_p:=\pi^{-1}_p(\cD_1).
$$

The function $\pi_p:\cD_p \rightarrow \cD_1$ is a $p$-fold  branched cover. The domain $\cD_2$ is equipped with the horizontal foliation $\cH_2$ and the vertical foliation $\cV_2$ of the unitary rectangle. The projection of $\cH_2$ and $\cV_2$ onto $\cD_1$ through $\pi_2:\cD_2\rightarrow \cD_1$ induces the decompositions $\cH_1$ and $\cV_1$ of $\cD_1$ respectively. Finally, for every $p\in \NN{+}$, we define the decompositions $\cH_p$ and $\cV_p$ of $\cD_p$ as the lifting of $\cH_1$ and $\cV_1$ respectively, using the map $\pi_p:\cD_p\rightarrow \cD_1$.

\begin{defi}\label{Defi: Singular foliation}
	Let $S$ be a closed and oriented surface. A decomposition $\cF$ of $S$ by curves and points is a  singular foliation if every $F\in \cF$ is path connected and for every $x\in S$ there exist $p(x)\in \mathbb{N}_{+}$ and a $C^{0}$ chart around $x$, $\phi_x:U_x\rightarrow \mathbb{C}$ such that:
	\begin{enumerate}
		\item $\phi_x(x)=0$
		\item $\phi_x(U_x)=\cD_{p(x)}$
		\item $\phi_x$  sends every non-empty connected component of $U_x\cap L$ onto some element of $\cH_{p(x)}$.
		\item 	There is a finite number of points $x$ in $S$ such that $p(x)\neq 2$. This set is the singular set of $\cF$ and is denoted by \textbf{Sing}$(\cF)$. 
	\end{enumerate}
A singular point $x$ with $p(x)=1$ is we called  \emph{Spine} (or $1$-prong) the set of spines of $\cF$ is \textbf{Sp}$(\cF)$ . In general for $p(x)\geq 2$,  $x$ is refereed as a $k$-prong.
\end{defi}

\begin{figure}[h]
	\centering
	\includegraphics[width=0.4\textwidth]{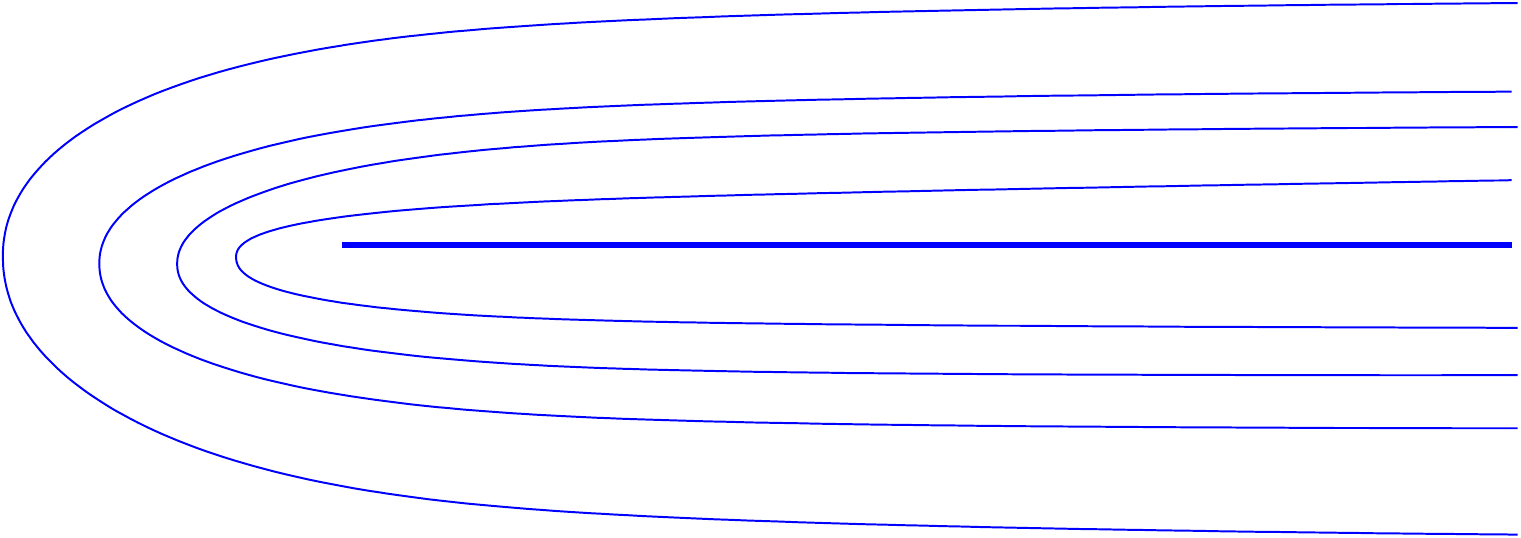}
	\caption{Spine or $1$-prong}
	\label{Fig: Spine}
\end{figure}

An element $F$ of a singular foliation $\cF$ is a \emph{leaf} of $\cF$. A separatrix of a point $x \in F$ is a connected component of $F \setminus \{x\}$. Now we will introduce a way to measure intervals across the leaves of a singular foliation.  A subset $J \subset M$ is an arc (resp. open arc) if there exists a $C^0$ embedding $h$ from a compact (resp. open) interval $I \subset \RR$ into $M$ such that $h(I) = J$, and the interior of $J$ is denoted as $\overset{o}{J} = \text{Int}(I)$.

\begin{defi}\label{Defi: Arch transversal arches}
Let $\cF$ be a singular foliation of $S$. An arc $J \subset S$ is transversal to $\cF$ if the interior of $J$ is contained in $S \setminus \textbf{Sing}(\cF)$, and for every $x \in J \setminus \textbf{Sing}(\cF)$ there exists a $C^0$-chart (as in Definition \ref{Defi: Singular foliation}) $\phi_x: U_x \rightarrow \cD_{2} \subset \CC$ such that, if $P_i$ is the projection onto the imaginary axes in $\CC$, then $P_i \circ \phi_x$ is injective on $U_x \cap \overset{o}{J}$. In other words, $\phi_x(\overset{o}{J})$ is topologically transversal to the leaves of $\cH_{2}$.
\end{defi}

Let $\cF$ be a singular foliation of $S$. A leaf-preserving isotopy between two arcs $\alpha$ and $\beta$ transversal to $\cF$ is an isotopy $H: [0,1] \times [0,1] \rightarrow S$ such that $H([0,1] \times {0}) = \alpha$ and $H([0,1] \times {1}) = \beta$, and for all $t \in [0,1]$, the following conditions hold:

\begin{itemize}
	\item $H([0,1] \times {t})$ is transverse to $\cF$,
	\item $H({0} \times [0,1])$ is contained in a single leaf of $\cF$.
	\item$H({1} \times [0,1])$ is contained in a single leaf of $\cF$.
\end{itemize}

In other words, we can deform $\alpha$ into $\beta$ without losing transversality to the foliation and without changing the leaves on which its endpoints move.

\begin{figure}[h]
	\centering
	\includegraphics[width=0.4\textwidth]{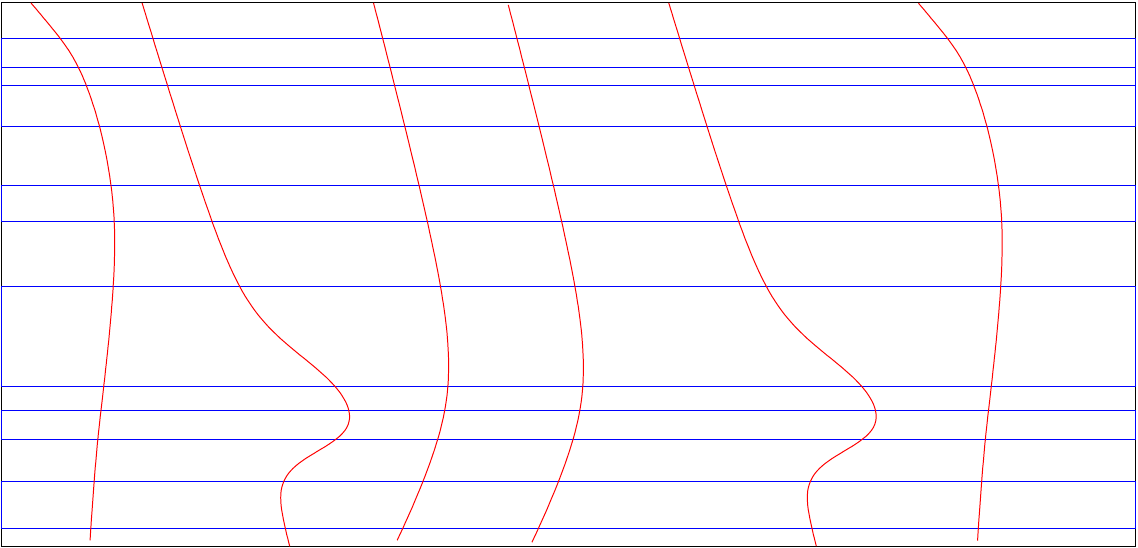}
	\caption{Leaf preserving isotopy}
	\label{Fig: Leaf preserving isotopy}
\end{figure}

\begin{defi}\label{Defi: Transverse measure}
Let $\cF$ be a singular foliation on $S$. An invariant transverse measure of $\cF$, denoted by $\mu$, is a Borel measure defined over any arc transversal to $\cF$, such that every non-trivial arc has positive (and finite) measure and every point has zero measure.

Even more, two leaf-isotopic arcs have the same measure. That is, if $\alpha$ and $\beta$ are transversal to $\cF$ and are leaf-isotopic, then
	$$
	\mu(\alpha)=\mu(\beta).
	$$
\end{defi}

\begin{defi}\label{Defi: measured foliation}
A pair $(\cF, \mu)$, where $\cF$ is a singular foliation of $S$ and $\mu$ is an invariant transverse measure of $\cF$, is a measurable foliation on the surface $S$.
\end{defi}

Let $f:S\rightarrow S$ be a homeomorphism of $S$ and $(\cF,\mu)$ be a measured foliation of $S$. The homeomorphism $f$ acts on the measured foliation in the following manner:

\begin{equation}\label{Equa: induced action on measured foliation}
f(\cF,\mu)=(f(\cF), f_*(\mu)).
\end{equation}

Here, $f(\cF)$ is the singular foliation of $S$ given by the image of leaves and singularities of $\cF$ under $f$. In other words, the elements of $f(\cF)$ are of the form $f(F)$, where $F\in \cF$.  For a transverse arc $\gamma$ to $f(\cF)$, we define the measure $f_*\mu$ as follows:

\begin{equation}\label{Equa: induced measure by f}
f_*(\mu)(\gamma)=\mu(f^{-1}(\gamma)).
\end{equation}

The equation $f_*(\mu)(\gamma)=\mu(f^{-1}(\gamma))$ only makes sense if $f^{-1}(\gamma)$ is transverse to $\cF$. As we have defined transversality at a topological level, this operation is well-defined. Moreover, the image of a measurable foliation under a homeomorphism is always a singular measured foliation, like we are going to see after.

\subsection{Pseudo-Anosov homeomorphism} In this thesis Hom$(S)$ is the group of homeomorphism over $S$ and Hom$_+(S)$ is the sub-group that preserve the orientation.

\begin{defi}\label{Defi: pseudo-Anosov}
Let $S$ be a closed and oriented surface. A homeomorphism $f: S \rightarrow S$ that preserve the orientation is a \emph{generalized pseudo-Anosov homeomorphism} if there are two measurable foliations $(\cF^s, \mu^s)$ and $(\cF^u, \mu^u)$, called the \emph{stable} and \emph{unstable} foliations of $f$, respectively, such that:
\begin{itemize}
	\item 
	A point $p$ is a $k$-prong singularity of $\cF^s$ if and only if $p$ is a $k$-prong singularity of $\cF^u$ for $k \geq 1$. The set of singularities of $f$ is given by:
	$$
	\textbf{Sing}(f):=\textbf{Sing}(\cF^s)=\textbf{Sing}(\cF^u).
	$$
	\item Every leaf of $\cF^s$ is transverse to every leaf of $\cF^u$ in the complement of their singularities.
	\item The foliations are invariant under $f$, i.e., $f(\cF^{s})=\cF^{s}$ and $f(\cF^{u})=\cF^{u}$.	
	\item 	There exists a number $\lambda > 1$, called the \emph{dilatation factor of} $f$, such that
	$$
	f_*(\mu^u)=\lambda\mu^u \text{ and } 	f_*(\mu^s)=\lambda^{-1}\mu^s 
	$$
\end{itemize}

\end{defi}

The set of spines of $f$ is denoted by \textbf{Sp}$(f)$. If $f$ has a non-empty set of spines, we call $f$ a \emph{pseudo-Anosov homeomorphism with spines}.  It is time to define the equivalence relation between homeomorphisms that we are going to study

\begin{defi}\label{Defi: Top conjugacy}
Let $S$ be a closed and orientable surface, and let $f:S\rightarrow S$ and $g:S\rightarrow S$ be two homeomorphisms. They are \emph{topologically conjugate} if there exists another homeomorphism $h:S\rightarrow S$ such that $g = h\circ f\circ h^{-1}$. In this case we said that   $h$ conjugates the homeomorphisms $f$ and $g$.
\end{defi}

The following proposition shows that the conjugacy class of a generalized  pseudo-Anosov homeomorphism is composed exclusively of generalized pseudo-Anosov homeomorphisms.

\begin{prop}\label{Prop:  conjugation produces pA}

Consider $f: S \rightarrow S$ as a generalized pseudo-Anosov homeomorphism, and let $h: S \rightarrow S$ be any other surface homeomorphism. Define $g = h \circ f \circ h^{-1}$. Then, $g: S \rightarrow S$ is also a generalized pseudo-Anosov homeomorphism of $S$ with the same dilation factor than $f$. Moreover, if $p \in S$ is a $k$-prong singularity of $f$ for $k \geq 1$, then $h(p) \in S$ is a $k$-prong singularity of $g$.
\end{prop}

\begin{proof}

Let $(\cF^{s,u},\mu^{u,s})$ be the invariant foliations of $f$. It is clear that $h(\cF^{s,u},\mu^{s,u})$ forms a singular foliation on $S$, as we can obtain $C^0$ charts of $h(\cF^{s,u})$ at a point $y \in S$ by composing $\phi_{h^{-1}(y)}$ with $h^{-1}(y)$. The conjugation between $f$ and $g$ implies that $h(\cF^{s,u})$ is invariant under the action of $g$. Using these charts, it is possible to see that the foliations of $g$ remain (topologically) transverse away from the singular points.

On the other hand, $h$ maps the singularities of $f$ to the singularities of $g$. Note that if $p$ is a $k$-prong singularity for $f$ with the chart $(U, \phi)$ around the point $p$ that maps the foliations of $f$ to $k$-prong curves in $\mathbb{C}$, then $(h(U), \phi \circ h^{-1})$ is a chart around $h(p)$ that maps the foliations of $g$ to $k$-prong curves in $\mathbb{R}^2$. Hence, $h(p)$ is a $k$-prong singularity for $g$.

If $J$ is an arc transversal to $h(\cF^{s,u})$ (in a topological sense), then $h^{-1}(J)$ is a transversal to $\cF^{s,u}$, and it is a non-trivial interval if and only if $J$ is a non-trivial interval. The pre-image of every Borel set contained in $J$ is a Borel set under $h^{-1}$ as $h$ is an homeomorphism. Hence, it is meaningful to calculate $\mu^{s,u}(h^{-1}(J))$ for every transversal to $h(\cF^{s,u})$, and we have $h_*(\mu^{s,u})(J) = \mu^{s,u}(h^{-1}(J))$. Now, let's observe the following computation, where $\lambda$ is the dilatation factor of $f$:

\begin{eqnarray*}
h_*(\mu^s(g(J)))=\mu^s(h^{-1}(h\circ f\circ h^{-1})(J))=\\
\mu^s(f(h^{-1}(J)))=\lambda^{-1}\mu^s(h^{-1}(J))=h_*\mu^s(J).
\end{eqnarray*}

This implies that $g$ has the same dilation factor as $f$.

Finally, if $f$ preserves the orientation of $S$, then any conjugation of $f$ also preserves the orientation \footnote{The homeomorphism $f$ preserves the orientation if the induced map in the second homology group of $S$ has degree one. Therefore, we can use the fact that deg$(h \circ f \circ h^{-1}) =$ deg$(h) \cdot$ deg$(f) \cdot$ deg$(h^{-1}) =$ deg$(f)$.}. This implies that $g$ preserves the orientation and is a generalized pseudo-Anosov homeomorphism of $S$.
\end{proof}

The theory of pseudo-Anosov homeomorphisms can indeed be extended to surfaces with boundaries, although the definition of pseudo-Anosov homeomorphisms is not canonical. It is important to note, as mentioned in  \cite[exposition 13]{fathi2021thurston}, that given a generalized pseudo-Anosov homeomorphism, we can perform a blow-up on the spine and obtain a pseudo-Anosov homeomorphism on a surface with boundaries. The theory of pseudo-Anosov homeomorphisms on surfaces with boundaries has been developed by B. Farb and D. Margalit in their book \cite{farb2011primer}. They have proved Proposition \ref{Prop: pseudo-Anosov properties.} in the case of surfaces with boundaries, and these results can be applied to the invariant foliations of a generalized pseudo-Anosov homeomorphism, which will be used in our future constructions.

\begin{prop}\label{Prop: pseudo-Anosov properties.}
Let $f: S \rightarrow S$ be a generalized pseudo-Anosov homeomorphism with invariant foliations $(\cF^s, \mu^s)$ and $(\cF^u, \mu^u)$. They have the following properties:
\begin{enumerate}
\item   None of the foliations contains a closed leaf, and none of the leaves contains two singularities (\cite[Lemma  14.11]{farb2011primer})	
\item They are \emph{minimal}, i.e., any leaf of the foliations is dense in $S$  \cite[Corollary 14.15]{farb2011primer}).
\end{enumerate}
\end{prop}

\subsection{The sectors of a point}

There is a natural local stratification of the surface $S$ given by the transverse foliation of a pseudo-Anosov homeomorphism $f$. We are going to exploit it. The following result corresponds to \cite[Lemme 8.1.4]{bonatti1998diffeomorphismes}

\begin{theo}\label{Theo: Regular neighborhood}
Let $f: S \rightarrow S$ be a generalized pseudo-Anosov homeomorphism, and let $p \in S$ be a singularity with $k \geq 1$ separatrices. Then there exists $\epsilon_0 > 0$ such that for all $0 < \epsilon < \epsilon_0$, there is a neighborhood $D$ of $p$ with the following properties:
\begin{itemize}
		\item The boundary of $D$ consists of $k$ segments of unstable leaves alternated with $k$ segments of stable leaves.
	\item For every (stable or unstable) separatrice $\delta$ of $p$, the connected component of $\delta \cap D$ which contains $p$ has a measure (stable or unstable) equal to $\epsilon$.
\end{itemize}

We say that $D$ is a \emph{regular neighborhood} of $p$ with side length $\epsilon$. If necessary, to make it clear, we write $D(p,\epsilon)$.
\end{theo}

\begin{figure}[h]
	\centering
	\includegraphics[width=0.4\textwidth]{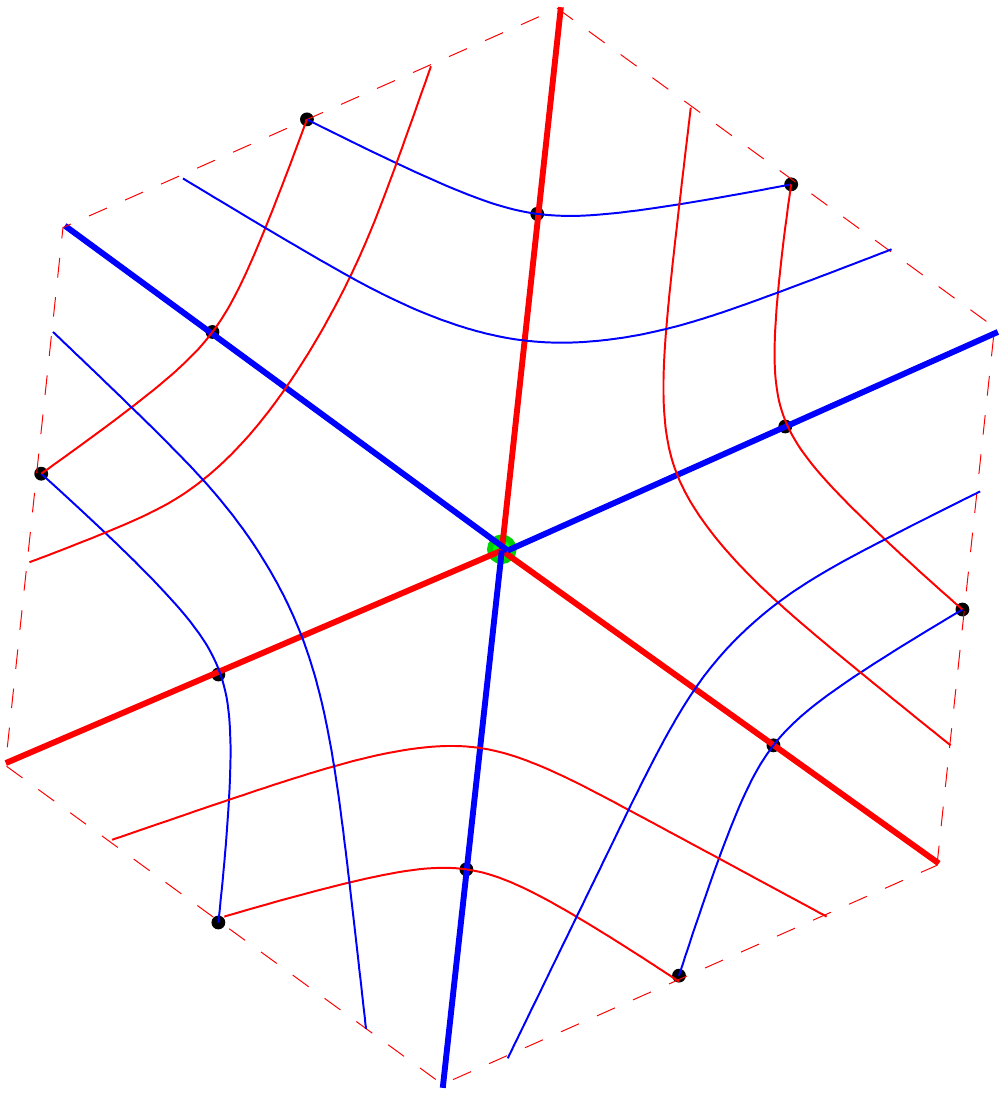}
	\caption{Regular neighborhood of a $3$-prong}
	\label{Fig: Regular neighborhood}
\end{figure}

Let $D(p,\epsilon)$ be a regular neighborhood of $p$ with side length $0<\epsilon\leq\epsilon_0$. We assume that $p$ has $k$ separatrices. Using the orientation of $S$, we label the stable and unstable separatrices of $p$ cyclically counterclockwise as $\{\delta_i^s\}_{i=1}^k(\epsilon)$ and $\{\delta_i^u\}_{i=1}^k(\epsilon)$.

 We define $\delta_i^s(\epsilon)$ as the connected component of $\delta_i^s\cap D(p,\epsilon)$ containing $p$, and $\delta_i^u(\epsilon)$ as the connected component of $\delta_i^u\cap D(p,\epsilon)$ containing $p$. We assume that $\delta_i^u(\epsilon)$ is located between $\delta_i^s(\epsilon)$ and $\delta_{i+1}^s(\epsilon)$, where $i$ is taken modulo $k$.
  
The connected components of $\text{Int}(D(p,\epsilon_0)) \setminus (\cup_{i=i}^k \delta_{i=1}^s(\epsilon_0) \cup \delta_{i=1}^u(\epsilon_0))$ are labeled with a cyclic order and denoted as $\{E(\epsilon_0)_j(p)\}_{j=1}^{2k}$, where the boundary of $E(\epsilon_0)_1(p)$ consists of $\delta^s_1(\epsilon_0)$ and $\delta^u_1(\epsilon_0)$. These conventions lead to the following definitions.

\begin{defi}\label{Defi:converge in a sector }
Let $\{x_n\}$ be a sequence convergent to $p$. We say $\{x_n\}$ \emph{converges to $p$ in the sector} $j$ if and only if there exists $N \in \mathbb{N}$ such that for every $n > N$, $x_n \in E(\epsilon_0)_j(p)$. The set of sequences convergent to $p$ in the sector $j$ is denoted by $E(p)_j$.
\end{defi}

The set of sequences that converge to $p$ in a sector is $\cup_{j=1}^{2k} E(p)_j$. We are going to define an equivalence relation on this set.

\begin{defi}\label{Defi: Sector equiv}
	Let $\{x_n\}$ and $\{y_n\}$ be sequences that converge to $p$ in a sector. They are in the same sector of $p$, and we denote ${x_n} \sim_q {y_n}$ if and only if $\{x_n\}$ and $\{y_n\}$ belong to the same set $E(p)_j$.
\end{defi}

\begin{rema}\label{Rema: caracterisation sim-p}
The previous definition is equivalent to the existence of $j \in \{1, \ldots, 2k\}$ and $N \in \mathbb{N}$ such that for all $n \geq N$, $x_n, y_n \in E(\epsilon_0)_j(p)$. We are using this characterization to prove the following lemma.
\end{rema}

\begin{lemm}\label{Lemm: Equiv relation}
In the set of sequences that converge to $p$ in a sector, $\sim_q$ is an equivalence relation. Moreover, every equivalence class consists of the set $E(p)_j$ for certain $j \in \{1, \ldots, 2k\}$.
\end{lemm}

\begin{proof}
	The less evident property is transitivity. If $\{x_n\}\sim_q \{x_n\}$ and $\{y_n\}\sim_q \{z_n\}$, suppose $x_n,y_n \in  E(\epsilon_0)_j(p)$ for $n>N_1$ and $y_n,z_n \in  E(\epsilon_0)_{j'}(p)$ for $n>N_1$. For every $n>N:=\max{N_1,N_2}$, $y_n\in  E(\epsilon_0)_j(p)\cap  E(\epsilon_0)_{j'}(p)$This is only possible if and only if $j = j'$.
\end{proof}

\begin{defi}\label{Defi: Sector}
	The equivalent class of  successions convergent to $p$ in the sector $j$  is the \emph{sector}  $e(p)_j$ of $p$.
\end{defi}

This notion is important in view of the next proposition, which establishes that generalized pseudo-Anosov homeomorphisms have a well-defined action on the sectors of a point.

\begin{prop}\label{Prop: image secto is a sector}
Let $f$ be a generalized pseudo-Anosov homeomorphism and $p$ any point in the underlying surface with $k$ different separatrices. If  $\{x_n\}\in e(p)_j$, there is a unique $i\in \{1,\cdots, 2k\}$ such that $\{f(x_n)\}\in e(f(p))_i$. This can be summarized by saying that the image of the sector  $e(p)_j$ is the sector  $e(f(p))_{i}$.
\end{prop}

\begin{proof}
	Note that $p$ is a $k$-prong a regular points or a spine if and only if $f(p)$ is a $k$-prong, a regular point or a spine, therefore $p$ and $f(p)$ have same number of sectors. Let $0<\epsilon<\epsilon_0$ be such that $f(D(p,\epsilon)) \subset D(f(p),\epsilon_0)$, such $\epsilon$ exist because $f$ is continuous.

	Let $\delta^s(p)_j$ and $\delta^u(p)_j$ be the separatrices of $p$ which bound the set $E(\epsilon)_{j}(p) \subset S$, then $f(\delta^s(p)_j)$ and $f(\delta^u(p)_j)$ are in two contiguous separatrices of $f(p)$, which determine a unique set $E(\epsilon_0)_{i}(f(p))$ for some $i\in \{1,\cdots,2k\}$.

	Let $N\in \NN$ such that, for all $n>N$, $x_n\in  E(\epsilon)_j(p)$ This implies that, for every $n>N$, $f(x_n)\in E(\epsilon_0)_{i}(f(p)$ and the succession $\{f(x_n)\}$ is in the sector $e(f(p))_{i}$.
\end{proof}

\section{Geometric Markov partitions}

\subsection{Rectangles} A Markov partition of pseudo-Anosov homeomorphism with marked points $f:S\rightarrow S$ is a decomposition of the surface  in pieces well adapted to its foliations, $(\cF^s,\mu^s)$ and $(\cF^u,\mu^u)$, and well behaved with respect iterations of $f$. Such pieces are called \emph{rectangles}. This sets admits a parametrization like follows

\begin{defi}\label{Defi: rectangle}
Let $f$ be a generalized pseudo-Anosov homeomorphism, and let $R\subset S$ be a compact subset. $R$ is a rectangle if there exists a continuous function $\rho:[0,1]\times [0,1] \rightarrow S$ satisfying the following conditions:
	\begin{itemize}
		\item Its image is $R$: $h([0,1]\times [0,1])=R$.
		\item Restricted to $(0,1)\times (0,1)$, $\rho$ is a homeomorphic embedding.
		\item For every $t\in [0,1]$, $I_t:=\rho([0,1]\times \{t\})$ is contained in a unique leaf of $\mathcal{F}^s$, and $\rho$ restricted to $[0,1]\times \{t\})$ is a homeomorphism onto its image. 
		\item For every $t\in [0,1]$, $J_t:=\rho(\{t\} \times [0,1])$ is contained in a unique leaf of $\mathcal{F}^u$, and $\rho$ restricted to $\{t\}\times [0,1]$  is a homeomorphism onto its image.
	\end{itemize}
The function $\rho$ is called a parametrization of the rectangle $R$.
\end{defi}

A rectangle $R$ admits more than one parametrization. In fact, there are two distinguished classes of parametrizations that will allow us to formulate the notion of oriented rectangles.
 
 \begin{defi}\label{Defi: equivalent parametrizations}
Let $\rho_1$ and $\rho_2$ be two parametrizations of the rectangle $R$. They are equivalent if
 $$
\rho_2^{-1} \circ \rho_1:= (\varphi_1,\varphi_2):(0,1)\times (0,1)\rightarrow (0,1)\times (0,1),
$$ 
satisfies the condition that $\varphi_i: (0,1) \to (0,1)$ is an increasing homeomorphism for $i = 1,2$. 

Is to said  that the change of parametrization preserve the orientation in the stable and unstable leaves at the same time.
 \end{defi}

Imagine that $\rho_1$ and $\rho_2$ are equivalent parametrizations of $R$, and let $I_t:=\rho_1([0,1]\times \{t\})$ and $I_s:=\rho_2([0,1]\times \{s\})$  be two intervals intersecting  within their interiors. In such a case,
 $$
 \rho_2^{-1} \circ \rho_1\circ ((0,1)\times \{t\}):= (\varphi_1(0,1),\varphi_2(t))=(\varphi_1(0,1),s)
 $$
 Like $\varphi_1(0,1)=(0,1)$ it follows that
 $$
 \rho_1((0,1)\times \{t\})=\rho_2((0,1),s)
 $$

So $I_s$ and $I_t$ coincide in their interior, the parameterizations restricted to a horizontal interval of $[0,1]\times[0,1]$ are homeomorphisms onto their respective images. Therefore, $I_t=I_s$. This observation allows us to give the following definition independently of the parametrization.

\begin{defi}\label{Defi: Vertical Horizontal foliations}
The \emph{horizontal or stable foliation} of $R$ is the decomposition of $R$ into its stable segments, denoted by $\{I_t\}_{t\in [0,1]}$, where $I_t$ is the \emph{horizontal leaf} of $R$ at $t$, i.e. for some parametrization $\rho$ of $R$, $I_t:=\rho([0,1]\times \{t\})$.

 The \emph{vertical or unstable foliation} of $R$ is the decomposition of $R$ into its unstable segments, denoted by $\{J_t\}_{t\in [0,1]}$, where $J_t$ is is¨ the \emph{vertical leaf} of $R$ at $t$. i.e. for some parametrization  $\rho$ of $R$,  $J_t:=\rho(\{t\}\times [0,1])$.
\end{defi}
 
 \begin{conv*}
Usually, we are not going to rely on a specific parametrization of a rectangle $R$. If the parametrization is not necessary for our discussion, we refer to the image of any equivalent parametrization of $R$ as a rectangle, along with its vertical and horizontal foliations.
 \end{conv*} 

There are other subsets within a rectangle to which we will often refer. They are defined below.
 
\begin{defi}\label{Defi: Interiors rec and segments}
With the same notation as Definition \ref{Defi: rectangle}, we have the following distinguished sets:
	\begin{itemize}
	\item The \emph{interior of the rectangle}: $\overset{o}{R}=h((0,1)\times (0,1))$.
	\item The \emph{interior of a horizontal leaf} $I_t$: $\overset{o}{I_t}=h(\{t\} \times (0,1))$.
	\item The \emph{interior of a vertical leaf} $J_t$: $\overset{o}{J_t}=h((0,1)\times \{t\})$
	\item The \emph{horizontal or stable boundary of the rectangle} $R$: $\partial^s R=I_0 \cup I_1$.
	\item The\emph{ vertical or unstable boundary of the rectangle} $R$: $\partial^u R=J_0\cup J_1$
	\item The \emph{boundary} of $R$: $\partial R=\partial^s R \cup \partial^uR$.
	\item The \emph{corners} of $R$ that are define by pairs: if  $t,s\in \{0,1\}$ the respective corner is $C_{s,t}=\rho(s,t)$.
	\end{itemize}
\end{defi}

\begin{defi}\label{Defi: vertical/horizontal sub rectangles}

Let $R$ be a rectangle. A subset $H\subset R$ is a \emph{horizontal sub-rectangle} of $R$ if it is a rectangle itself, and for every $x\in \overset{o}{H}$, the horizontal leaf of the rectangle $H$ through $x$ coincides with the horizontal leaf of the rectangle $R$ through $x$.

Similarly, a subset $V\subset R$ is a \emph{vertical sub-rectangle} of the rectangle $R$ if it is a rectangle itself, and for every $x\in \overset{o}{V}$, the vertical leaf through $x$ of the rectangle $V$ coincides with the vertical leaf through $x$ of the rectangle $R$.
\end{defi}

\begin{figure}[h]
	\centering
	\includegraphics[width=0.3\textwidth]{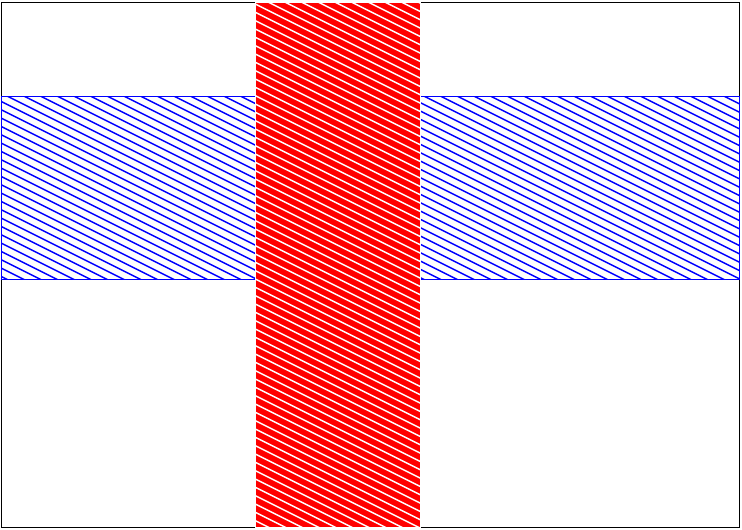}
	\caption{Sub-rectangles}
	\label{fig:Sub-rec}
\end{figure}

If $p\in S$ is a singularity of $\cF^{s,u}$, $p$  is in a unique leaf even if it have $p$ is a $k$-prong with $k$-separatrices. The interior of a rectangle is trivially bi-foliated and  there is non-singularity in the interior of a rectangle, therefore, if a rectangle contains a singularity such singularity is in the boundary of the rectangle. Even more, if $h([0,1]\times \{0\})$ have a singularity, $h([0,1]\times \{0\})$ is contained in at most $2$ stable separatrices of $p$.There are some other configurations of rectangles that are excluded from our definitions.

\begin{itemize}
\item[i)] As the foliation doesn't have closed leaves, a rectangle is never an annulus, where the horizontal boundaries are closed curves and the the vertical boundaries correspond to the same unstable interval. Look the figure \ref{Fig: No annulus}:

\begin{figure}[h]
	\centering
	\includegraphics[width=0.2\textwidth]{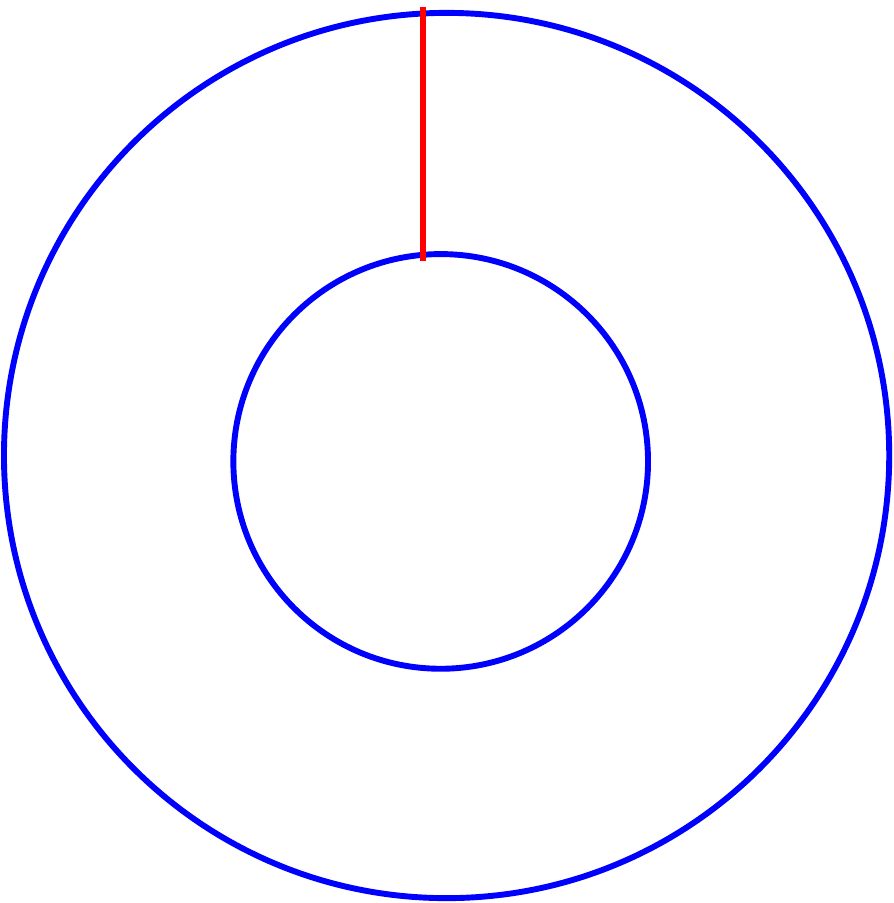}
	\caption{Annulus}
	\label{Fig: No annulus}
\end{figure}

 \item[ii)]   If $p$ is a singularity on the unstable boundary of $R$, there are at most two  stable separatrices of side small enough  of $p$ such that they intersect $R$. This is because in a small neighborhood of the unstable boundary of $R$, all the vertical  leaves of the rectangle $R$ are parallel to the boundary. Any additional separatrice of $p$ intersecting the rectangle must be transverse to the boundary and this behavior is not  coherent with the trivial bi-foliation in the interior of $R$. See for example the Figure \ref{Fig: Non-rectangles}:

 \begin{figure}[h]
 	\centering
 	\includegraphics[width=0.4\textwidth]{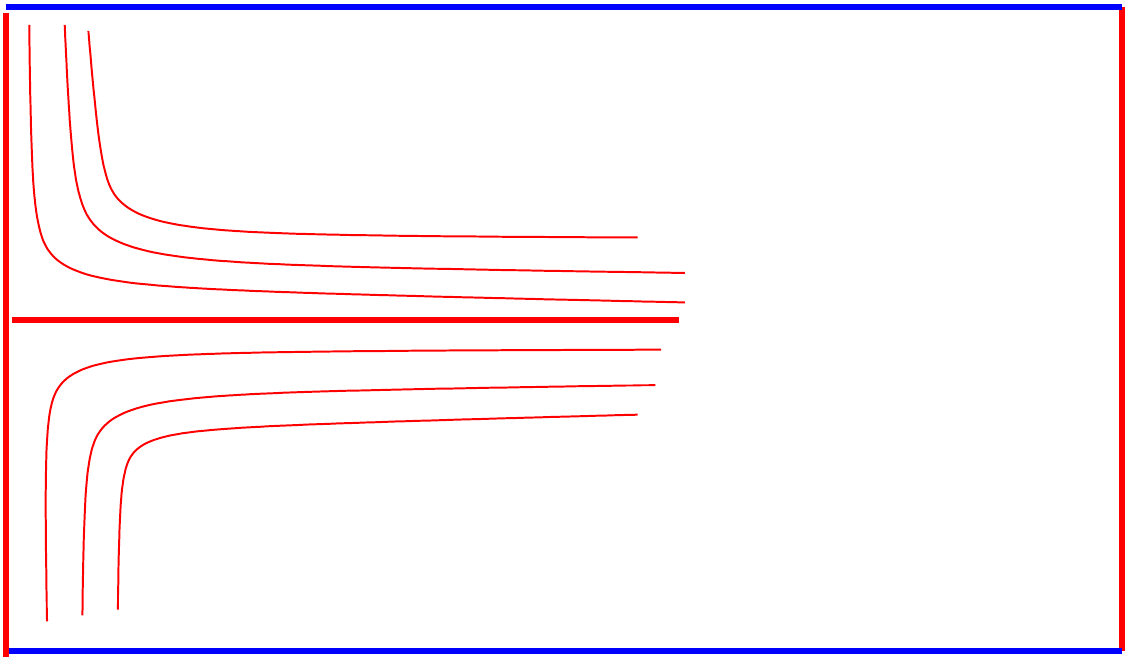}
 	\caption{$3$-prong with a separatrice in the interior of $R$}
 	\label{Fig: Non-rectangles}
 \end{figure}

\item[iii)] There are other configurations that are not rectangles in the sense of definition \ref{Defi: rectangle} and we describe a couple of them in Figure \ref{Fig: more Non-rectangles} 

\begin{figure}[h]
	\centering
	\includegraphics[width=0.7\textwidth]{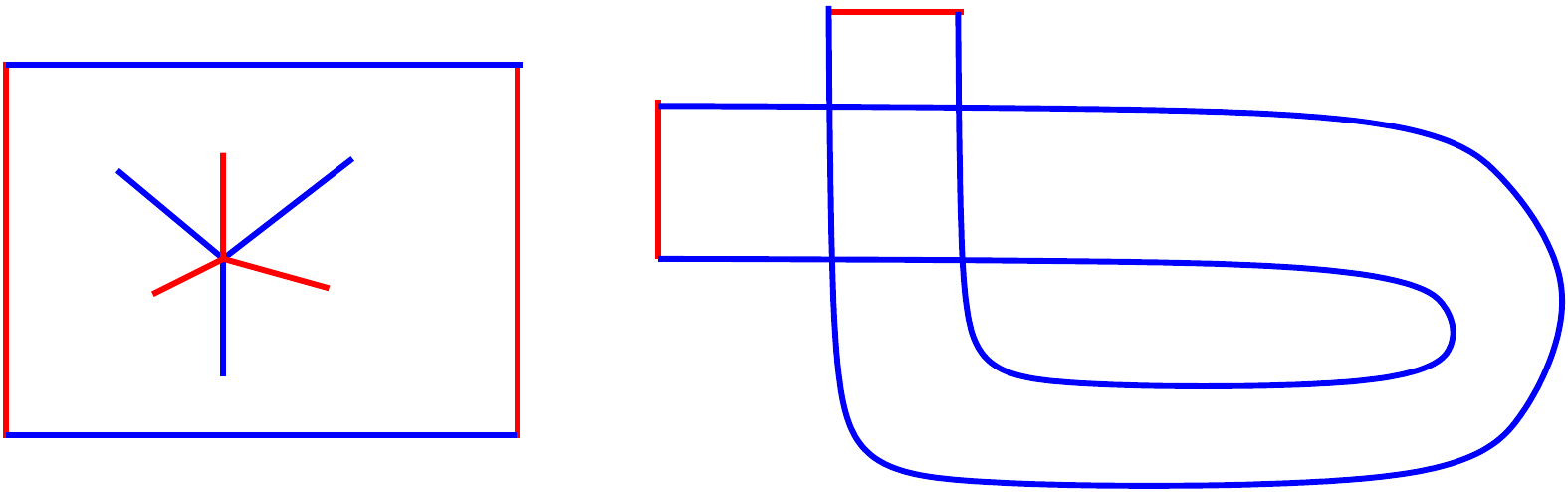}
	\caption{More non-rectangles}
	\label{Fig: more Non-rectangles}
\end{figure}

\end{itemize}

Since the rectangles we are considering can have self-intersections at their boundaries, there is no inherent notion of an 'upper', 'lower', 'left', or 'right' side of the rectangle. However, we have defined the notion of equivalent parametrizations in Definition\ref{Defi: equivalent parametrizations}, which gives us two distinct equivalence classes of parametrizations for the same rectangle $R$. Each equivalence class induces a unique orientation for $R$ that preserves the orientation of the parametrization. Therefore, once we fix a parametrization, we choose an orientation for $R$, and this parametrization induces a unique orientation in the interior of every vertical and horizontal leaf of $R$. We will utilize this construction to exploit the orientations.

\begin{defi}\label{Defi: Oriented rectangle}
An oriented rectangle $R$ is the selection of an equivalence class of parameterizations for the rectangle $R$.
\end{defi}

It is clear that this definition is equivalent to the following one, which will be the one we will actually use more frequently in our technical arguments.

\begin{defi}\label{Defi: vertical and horizontal orientation}
Let $R$ be an oriented rectangle. Let $\rho$ be a parametrization of a rectangle $R$ chosen from the equivalence class of parametrizations. This parametrization induces a unique orientation in the vertical foliation of $R$, which we will refer to as the  \emph{vertical orientation}  of $R$. Similarly, the \emph{horizontal orientation} of $R$ is the unique orientation in the horizontal foliation of $R$ induced by $\rho$.
\end{defi}

The vertical and horizontal orientation of $R$ is local and as we have seen it supports two different ones. Using the vertical and horizontal orientation of $R$ induced by a parametrization $\rho$  we make the following conventions:

\begin{itemize}
	\item The \textbf{ left side} of $R$ is define as $\partial^u_{-1}R:=\rho(\{0\}\times [0,1]))$ and the \textbf{right side} is $\partial^u_{1}R=\rho(\{1\}\times [0,1]))$. 
	\item The \textbf{ lower side} of $R$ is $\partial^s_{-1}R=\rho([0,1] \times \{0\} ))$ and the \textbf{ upper side} is $\partial^s_{1}R=\rho([0,1] \times \{1\} ))$.
\end{itemize}

\subsection{Geometric Markov partition} Now we are ready to introduce the main tool of this section, the \emph{Markov partitions}.

\begin{defi}[Markov partition]\label{Defi: Markov partition}
	Let $f:S\rightarrow S$ be pseudo-Anosov homeomorphism. A Markov partition of $f$  is a family of rectangles $\cR=\{R_i\}_{i=1}^n$ with the following properties:
	\begin{itemize}
		\item The surface is the union of the rectangles: $S=\cup_{i=1}^n R_i$.
		\item If $i\neq j$, then $\overset{o}{R_i}\cap \overset{o}{R_j}=\emptyset$.
		\item For every $i,j\in\{1,\cdots,n\}$, each  connected component of $\overset{o}{R_i} \cap f^{-1}(\overset{0}{R_j})$is either empty or the interior of a horizontal sub-rectangle of  $R_i$.
		\item For every $i,j\in\{1,\cdots,n\}$, each  connected component of $f(\overset{o}{R_i}) \cap \overset{0}{R_j}$ is either empty or the interior of a vertical sub-rectangle of $R_j$.
	\end{itemize}
	\end{defi}	

The closure of any connected component of $\overset{o}{R_i} \cap f^{-1}(\overset{o}{R_j})$ is a distinguished horizontal sub-rectangle of $R_i$, as they correspond to the vertical sub-rectangles of $R_j$ obtained as the closure of the connected components of $f(\overset{o}{R_i}) \cap \overset{o}{R_j}$. Now we are going to establish a convention through the following definition.

\begin{defi}\label{Defi: horizontal-vertical of the partition}
We define a \emph{horizontal sub-rectangle of the Markov partition} $\cR$ as any sub-rectangle obtained as the closure of a connected component of $\overset{o}{R_i} \cap f^{-1}(\overset{0}{R_j})$ for $i,j\in \{1,\cdots,n\}$. Similarly, a \emph{vertical sub-rectangle of the Markov partition} $\cR$ is any sub-rectangle obtained as the closure of a connected component of $f(\overset{o}{R_i}) \cap \overset{0}{R_j}$ for $i,j\in \{1,\cdots,n\}$.
\end{defi}

We obtain the following corollary.

\begin{coro}\label{Coro: image of sub horzontal-vertical }
The image of any horizontal sub-rectangle of the Markov partition $\mathcal{R}$ under $f$ is a vertical sub-rectangle of the Markov partition $\mathcal{R}$, and vice versa, the image of any vertical sub-rectangle of the Markov partition $\mathcal{R}$ under $f^{-1}$ is a horizontal sub-rectangle of the Markov partition $\mathcal{R}$.

\end{coro}

We could say that $f$ induces a bijection between the horizontal and vertical sub-rectangles of $\mathcal{R}$. We will now proceed to distinguish certain sets and types of points within a Markov partition.

\begin{rema}\label{Rema: dependece partition homeo}
	They can be a Markov partition for $f$, some power of $f$, or even another homeomorphism. To indicate this dependence, we write $(f, R)$ whenever necessary.
\end{rema}

\begin{defi}\label{Defi: boundary points}
Let $\mathcal{R} = \{R_i\}_{i=1}^n$ be a Markov partition of $f$. We have the following distinguished sets:
	\begin{enumerate}
		\item The \emph{stable boundary} of $\cR$ is the union of the stable boundaries of the rectangles in $\cR$:  $\partial^s \cR:=\cup_{i=1}^n \partial^s R_i$.
		\item The \emph{unstable boundary} of $\cR$ is the union of the unstable boundaries of the rectangles in $\cR$: $\partial^u \cR=\cup_{i=1}^n \partial^u R_i$.
		\item The \emph{boundary} of $\cR$ is $\partial \cR:= \partial^s \cR \cap \partial^u \cR$.
		\item The \emph{interior} of $\cR$  is the union of the interiors of all the rectangles in  $\cR$, $\overset{o}{\cR}=\cup_{i=1}^n \overset{o}{R_i}$.
	\end{enumerate}
Let $p$ be a periodic point of $f$. Then:
	\begin{enumerate}
\item  $p$ is an $s$-\emph{boundary periodic point} of $\mathcal{R}$ if $p \in \partial^s \mathcal{R}$, an $u$-\emph{boundary periodic point} if $p \in \partial^u \mathcal{R}$, and a \emph{boundary periodic point} if $p \in \partial \mathcal{R}$. These sets are denoted as Per$^{s,u,b}(f,\mathcal{R})$ accordingly.
\item $p$ is an \emph{interior periodic point} if $p \in \overset{o}{\mathcal{R}}$, and this set is denoted as Per$^I(f,\mathcal{R})$.
\item $p$ is a \emph{corner periodic point} if there exists $i \in \{1, \ldots, n\}$ such that $p$ is a corner point of $R_i$. This set is denoted as Per$^C(f,\mathcal{R})$.
	\end{enumerate}    
	\end{defi}

In the following, we introduce two families of Markov partitions based on the type of periodic points on their boundary. Please note that the periodic points are determined by the specific homeomorphism associated with $R$ as a Markov partition. 

\begin{defi}\label{Defi: well suited partition}
	A Markov partition $\cR$ of $f$ is considered \emph{well suited} if all of its periodic boundary points are corner periodic boundary points.
\end{defi}

\begin{defi}\label{Defi: adapted partition}
Let $\cR$ be a Markov partition of $f$. The partition is considered \emph{adapted} to $f$ if it is well-suited and all the boundary periodic points are singularities of $f$.
\end{defi}

We have introduced the concept of oriented rectangle, which leads us to the following definition. Although it may seem redundant, the indexing of the rectangles plays a crucial role in defining the geometric type of a Markov partition, so we include it as part of the definition.

\begin{defi}\label{Defi: geometric Markov partition}
	 Let $f:S\rightarrow S$ be a generalized pseudo-Anosov homeomorphism, and let $\cR$ be a Markov partition of $f$. We refer to $\cR$ as a \emph{geometric Markov partition} once we have chosen an indexing and an orientation in the vertical direction (ergo in the horizontal direction) for each the rectangle in $\cR$, i.e. $\cR:=\{R_i\}_{i=1}^n$.
\end{defi}

The following lemma will be used many times in our future explanations.

\begin{lemm}\label{Lemm: Boundary of Markov partition is periodic}
Both the upper and lower boundaries of each rectangle in the Markov partition $\cR$ lie on the stable leaf of some periodic point of $f$. Similarly, the left and right boundaries lie on the unstable leaf of some periodic point.
\end{lemm}

\begin{proof}
Let $x$ be a point on the stable boundary of $R_i$. For all $n\geq 0$, $f^{n}(x)$ remains on the stable boundary of some rectangle because the image of the stable boundary of $R_i$ coincides with the stable boundary of some vertical rectangle in the Markov partition, and such vertical rectangle have stable boundary in $\partial^s \cR$. However, there are only a finite number of stable boundaries in $\cR$. Hence, there exist $n_1, n_2\in \mathbb{N}$ such that $f^{n_1}(x)$ and $f^{n_2}(x)$ are on the same stable boundary component. This implies that this leaf is periodic and corresponds to the stable leaf of some periodic point. Therefore, $x$ is one of these leaves and is located on the stable leaf of a periodic point.
A similar reasoning applies to the case of vertical boundaries.
\end{proof}

To finish this exposition about Markov partition with a result that relate the partition with the sector of a point.

\begin{lemm}\label{Lemm: sector contined unique rectangle}
Let $f: S \rightarrow S$ be a generalized pseudo-Anosov homeomorphism with a Markov partition $\cR$. Let $x \in S$ and let $e$ be a sector of $x$. Then, there exists a unique rectangle in the Markov partition that contains the sector $e$.
\end{lemm}

\begin{proof}
Let $\{x_n\}$ be a sequence that converges to $x$ within the sector $e$. Consider a canonical neighborhood $U$ of size $\epsilon > 0$ around $x$, and let $E$ be the unique connected component of $U$ minus the local stable and unstable manifolds of $x$ that contains the sequence $\{x_n\}$.

By choosing $\epsilon$ small enough, we can assume that the local stable separatrix $I$ of $x$ that bound $E$ is contained in at most two rectangles of the Markov partition, and similarly, the local unstable separatrix $J$ of $x$ is contained in at most two rectangles. By take the right side of the local separatrices, we can take:  rectangles $R$ and $R'$ in the Markov partition, a horizontal sub-rectangle $H$ of $R$ that contains $I$ in its upper or lower boundary, and a vertical sub-rectangle $V$ of $R'$ whose upper or left boundary contains $J$. These rectangles can be chosen small enough such that the intersection of their interiors is a rectangle contained within $E$, denoted as $\overset{o}{Q} := \overset{o}{H} \cap \overset{o}{V} \subset E$. This implies that $R = R'$ since the intersection of the interiors of rectangles in the Markov partition is empty.

Furthermore, by considering a sub-sequence of $\{x_n\}$, we don't change its equivalent sector class. Therefore, $\{x_n\} \subset \overset{o}{Q} \subset R$. This completes our proof.
\end{proof}

\section{The geometric type}

In this section, we will show how to associate combinatorial data to any geometric Markov partition. This information is the geometric type of a geometric Markov partition. This data takes into account how $f$ shuffles the vertical and horizontal rectangles of the Markov partition and how their orientation changes. Later on, we will define the geometric type in a more general manner.

\subsection{The geometric type of a geometric Markov partition}Let $f: S \rightarrow S$ be a generalized pseudo-Anosov homeomorphism, and let $\cR=\{R_i\}_{i=1}^n$ be a geometric Markov partition of $f$. We will label the horizontal sub-rectangles of the Markov partition contained in $R_i$ using the vertical orientation of $R_i$, from bottom to top, as follows:
\begin{equation}
\{H^i_j\}_{j=1}^{h_i} \, \text{ where } h_i\geq1.
\end{equation}
Here, $h_i$ is the number of horizontal sub-rectangles of the Markov partition contained in $R_i$. The sub-rectangle $H^i_1$ shares its lower boundary with the lower boundary of $R_i$, and the upper boundary of $H^i_{h_i}$ has the same upper boundary as $R_i$.

We label the vertical sub-rectangles of the Markov partition contained in $R_k$, from left to right, using the horizontal orientation of $R_k$ as follows:
\begin{equation}
\{V_l^k\}_{k=1}^{v_k} \, \text{ where } v_k\geq 1 .
\end{equation}	

The number of vertical sub-rectangles of the Markov partition contained in $R_k$ is denoted by $v_k$. The sub-rectangle $V^k_1$ has its left boundary coinciding with the left boundary of $R_k$, and the right boundary of $V^k_{v_k}$ is equal to the right boundary of $R_k$.

\begin{figure}[h]
	\centering
	\includegraphics[width=0.5\textwidth]{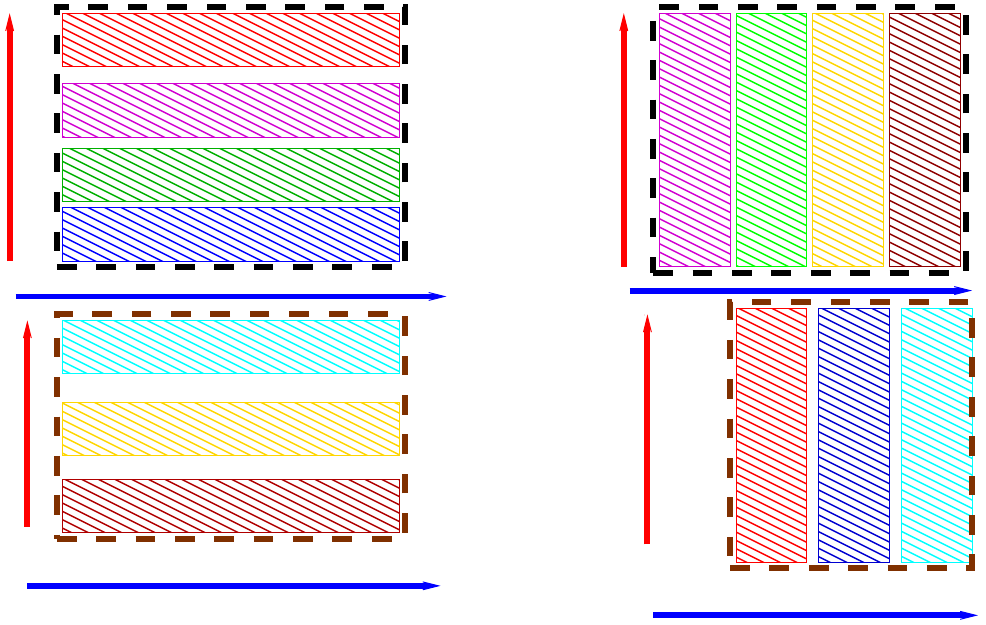}
	\caption{Horizontal and vertical sub-rectangles}
	\label{fig:Horizontal vertical}
\end{figure}

There are two sets determined by such indexations of rectangles
\begin{equation}
\cH= \{(i,j): i\in \{1, \cdots,n\} \text{ and } j\in \{1,\cdots,h_i\}\}.
\end{equation}
and
\begin{equation}
\cV= \{(k,l): l\in \{1, \cdots,n\} \text{ and } k\in \{1,\cdots,v_i\}\}.
\end{equation}

We have seen in Corollary \ref{Coro: image of sub horzontal-vertical } that $f$ establishes a bijection between the horizontal and vertical sub-rectangles of the Markov partition. In particular, the number of horizontal and vertical sub-rectangles in $\cR$ is the same and we can deduce the following equality:
\begin{equation}\label{Equa: Same number horizontal and verticals}
\sum_{i=1}^{n}h_i=\sum_{i=1}^n v_i.
\end{equation}

This allows us to provide the following definition, where the function encodes the order in which the horizontal sub-rectangles are shuffled by $f$.

\begin{defi}\label{Defi: Permutation of type}
	The function $\rho:\cH \rightarrow \cV$ is defined as $\rho(i,j)=(k,l)$ if and only if $f(H^i_j)=V^k_l$.
\end{defi}

On the other hand, the change in orientation of $f$ in every horizontal sub-rectangle is measured by the following function.

\begin{defi}\label{Defi: Permutation type}
Suppose $f(H_j^i)=V_l^k$. The rectangle $H^i_j$ has a vertical orientation with respect to $R_i$, and $V_l^k$ has a vertical orientation with respect to $R_k$. We define:
\begin{equation}
\epsilon:\cH \rightarrow \{ 1,-1\} \text{ as }
\end{equation}
by $\epsilon(i,j)=1$ if $f$ preserves such orientations, and $-1$ otherwise.
\end{defi}

\begin{figure}[h]
	\centering
	\includegraphics[width=0.5\textwidth]{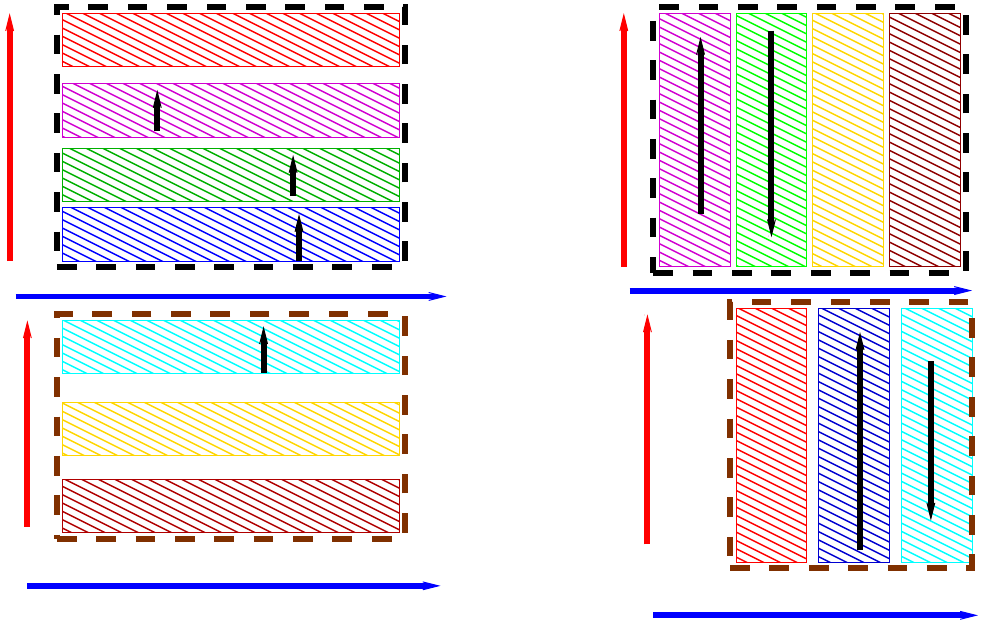}
	\caption{$\epsilon$ captures the change of orientation}
	\label{fig: orietation}
\end{figure}

\begin{defi}\label{Defi: geometric type of a Markov partition}
	Let $\cR$ be a Markov partition for $f$. The geometric type of $(f,\cR)$, is denoted by $T:=T(f, \cR)$ and defined as the quadruple:
\begin{equation}
	T(f,\cR)=(n,\{(h_i,v_i)\}_{i=1}^n, \Phi_T).
\end{equation}
	where $\Phi_T:\cH \rightarrow \cV \times \{-1,1\}$ is $\Phi_T:=(\rho,\epsilon)$ .
\end{defi}

\begin{conv*}
	Many times we will work with a fixed homeomorphism $f$, so if it is not necessary to specify the pair $(f,\cR)$, we simply write $T(\cR)$ to indicate the geometric type of the partition $\cR$. When there are multiple partitions and homeomorphisms, we will use the notation from the definition.
\end{conv*}

\subsection{Abstract geometric types} An abstract geometric type is a generalization of the geometric type of a geometric Markov partition. As we will see, it is determined as the skew product of a permutation and a reflection.
 
 \begin{defi}\label{Defi: abstract geometric type}
An abstract geometric type is an ordered quadruple:
 $$
 T=(n, \{h_i,v_i\}_{i=1}^n,\Phi_T=(\rho_T,\epsilon_T))
 $$ 
which contains the following information:
\begin{itemize}
\item $n,h_i,v_i\in \NN_+$ are positive integers.
\item $\sum_{i=1}^{n}h_i=\sum_{i=1}^n v_i$
\item $\rho_T:\cH(T)\rightarrow \cV(T)$ is a bijection between the sets
$$
\cH(T)=\{(i,j): 1\leq i \leq n \text{ and } 1\leq j\leq h_i\}
$$
 and 
 $$
 \cV(T)=\{(k,l): 1\leq k \leq n \text{ and } 1\leq l\leq v_k\}
 $$ 
 \item $\epsilon_T:\cH(T) \rightarrow \{-1,1\}$ is any function.
\end{itemize}
The set of abstract geometric types is denoted by $\mathcal{G}\mathcal{T}$.
 \end{defi}

Given that $\rho_T$ is a bijection, we can define the formal inverse of a geometric type.

\begin{defi}\label{Defi: inverse of Type}
The inverse of a geometric type $T=(n,\{h_i,v_i\}_{i=1}^n,\Phi_T=(\rho_T,\epsilon_T))$ could be define as the geometric type:
$$
T^{-1}:=(N,\{H_i,V_i\}_{i=1}^N, \Phi_{T^{-1}}:=(\rho_{T^{-1}}, \epsilon_{T^{-1}})  )
$$
where:
\begin{itemize}
\item The numbers are: $N = n$, $H_i = v_i$, and $V_i = h_i$.
\item In this manner,  $\cH(T^{-1})=\cV(T)$ and  $\cV(T^{-1})=\cH(T)$
\item The functions are given by: $\rho_{T^{-1}}=\rho_T^{-1}$ and $\epsilon_{T^{-1}}(k,l)=\epsilon(\rho_T^{-1}(k,l))$.
\end{itemize}
\end{defi}

\begin{rema}\label{Rema: Inverse type and inverse pAnosov}
If $f$ is a generalized pseudo-Anosov homeomorphism and $T$ is the geometric type of a geometric Markov partition $\mathcal{R}$ of $f$, $\mathcal{R}$ is a Markov partition for $f^{-1}$ too. However, the horizontal rectangles of $\mathcal{R}$ as a partition of $f$ now become vertical rectangles of $\mathcal{R}$ as a partition of $f^{-1}$. It is not difficult to convince ourselves that $T^{-1}$ is the geometric type of the Markov partition $\mathcal{R}$ for $f^{-1}$.
$$
T(f^{-1},\cR)=T^{-1}(f,\cR).
$$
\end{rema}

\subsection{The pseudo-Anosov class} Not every abstract geometric type is the geometric type of the Markov partition of a pseudo-Anosov homeomorphism. Those that are determine a subset of $\mathcal{G}\mathcal{T}$, which we will formalize in the following definition.

\begin{defi}\label{Defi: pseudo-Anosov class}
A geometric type $T \in \mathcal{G}\mathcal{T}$ is in the \emph{pseudo-Anosov class} if there exists a surface $S$ and a generalized pseudo-Anosov homeomorphism $f:S\rightarrow S$ possessing a Markov partition $\mathcal{R}$ of geometric type $T$, i.e., $T := T(f,\mathcal{R})$. The pseudo-Anosov class of geometric types is denoted by $\mathcal{G}\mathcal{T}(p\AA)$.
\end{defi}

\begin{defi}\label{Defi: T is realized by}
Let $T$ be a geometric type in the pseudo-Anosov class. We say that $T$ is realized by a pair of homeomorphism and partition if there exists a generalized pseudo-Anosov homeomorphism $f$ with a geometric Markov partition $\mathcal{R}$ of geometric type $T$, i.e., $T := T(f,\mathcal{R})$.
\end{defi}

In Chapter  \ref{Chapter: Realization} , we will see the necessary and sufficient conditions for a geometric type to be in the pseudo-Anosov class.

\subsection{The powers of a geometric type}Finally, we must define powers of geometric types. To do this, let us note that if $f$ is a pseudo-Anosov homeomorphism with a geometric Markov partition $\cR=\{R_i\}_{i=1}^n$, then for all $m\in \mathbb{N}$:

\begin{itemize}
\item The map $f^m$ is a generalized pseudo-Anosov homeomorphism with the same stable and unstable foliations.
\item Every set $R_i$ is a rectangle for $f^m$.
\item The family of rectangles $\cR=\{R_i\}_{i=1}^n$ is a geometric Markov partition for $f^m$, with the same labels and the same orientation in every rectangle.
\end{itemize}

If we consider the partition $\mathcal{R}$ as a Markov partition for $f^m$, we denote them as $(f^m,\mathcal{R})$.

\begin{defi}\label{Defi: powers of the type}
Let $f:S\rightarrow S$ be a generalized pseudo-Anosov homeomorphism with a geometric Markov partition $\mathcal{R}$ of geometric type $T$, and let $m\in \mathbb{N}$. The geometric type of the geometric Markov partition $(f^m,\mathcal{R})$ is the power $m$ of $T$ and is denoted by $T^m$.
\end{defi}

In fact, it is possible to define the power of an abstract geometric type in a way that matches our definition when we restrict ourselves to the pseudo-Anosov class or for types that are realized as basic pieces of Smale diffeomorphisms of surfaces, which we will study in the next section. The following result was made explicit by François Béguin in \cite[Appendix A]{beguin2004smale}  and forms the core of the algorithms we will describe in the future.

\begin{prop}\label{Prop: algoritm iterations type}
Let $T$ be an abstract geometric type. There exists an algorithm to compute $T^m$ for all $m\in \mathbb{N}$.
\end{prop}

\subsection{Equivalent definition of Markov partition} We end this section by introducing an equivalent definition of a Markov partition that is easier to verify. We do this up to this moment because the proof utilizes the language of geometric types.

\begin{prop}\label{Prop: Markov criterion boundary}
	Let $f:S \rightarrow S$ be a pseudo-Anosov homeomorphism and $\cR=\{R_i\}_{i=1}^n$ be a family of rectangles satisfying the following conditions:
	\begin{itemize}
		\item $S=\cup_{i=1}^n R_i$
		\item For all $i\neq j$,  $\overset{o}{R_i}\cap \overset{o}{R_j}=\emptyset$
	\end{itemize}
Then, the family $\mathcal{R}$ is a Markov partition for $f$ if and only if the following conditions hold:
	\begin{itemize}
		\item $\partial^s\cR=\cup_{i=1}^m\partial^sR_i$ is $f$-invariant, and 
		\item $\partial^u\cR=\cup_{i=1}^m\partial^uR_i$ is $f^{-1}$-invariant.
	\end{itemize}
\end{prop}

\begin{figure}[h]
	\centering
	\includegraphics[width=0.7\textwidth]{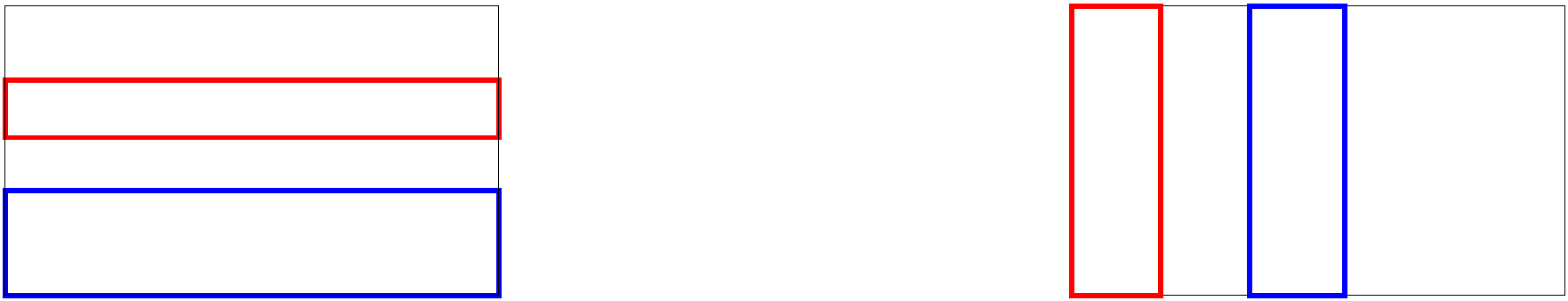}
	\caption{Markov partition implies $f,f^{-1}$-invariant boundaries}
	\label{Fig: criterio Markov 1}
\end{figure}

\begin{figure}[h]
	\centering
	\includegraphics[width=0.7\textwidth]{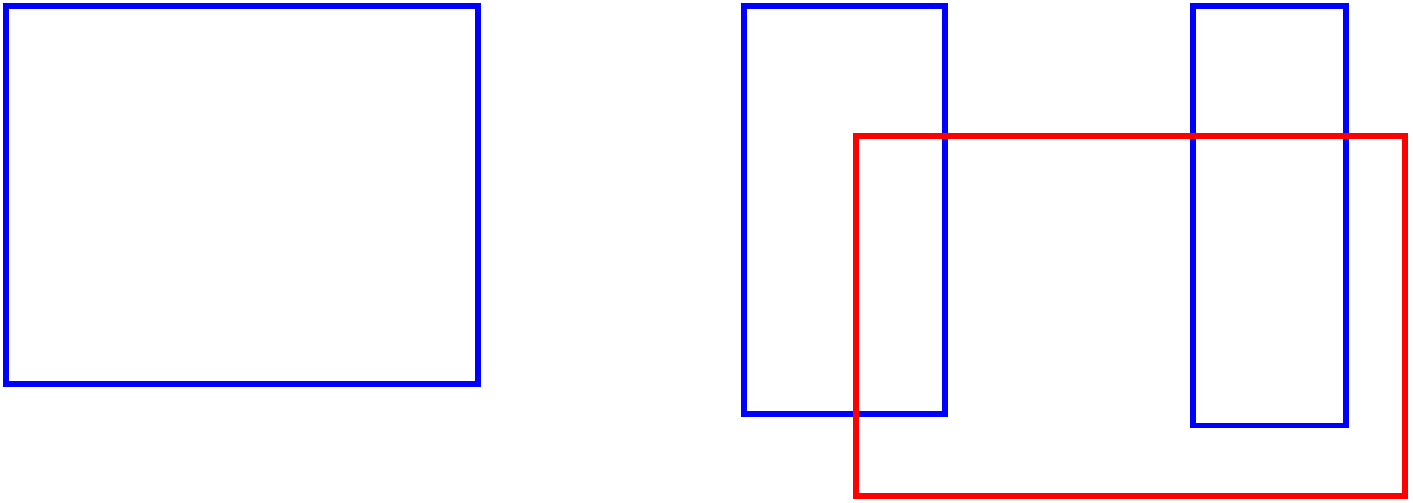}
	\caption{Invariant implies Markov partition}
	\label{Fig: criterio Markov 2}
\end{figure}

\begin{proof}
	
Suppose $\mathcal{R}$ is a Markov partition of $f$ and is endowed with a geometrization. If $I$ is the lower ( upper) boundary of $R_i$, then $I$ corresponds to the lower ( upper) boundary of $H^i_1$ ( $H^i_{h_i}$) as we have $f(H^i_1)=V^k_l$, for some $(k,l)\in \mathcal{V}$, and $V^k_l$ is a vertical sub-rectangle of $R_k$. The interval $f(I)$ is necessarily on the horizontal boundary of $R_k$. This argument proves that $\partial^s \mathcal{R}$ is $f$-invariant.
	
If $J$ is the left side of $R_k$, then $J$ is the left side of the vertical sub-rectangle $V^k_1$. Now, $f^{-1}(V^k_l)=H^i_j$ is a horizontal sub-rectangle of $R_i$, and $f^{-1}(J)$ is on the unstable boundary of $H^i_j$ (for some $(i,j)\in \mathcal{H}$), which is contained in the unstable boundary of $R_i$. This shows that $\partial^u(\mathcal{R})$ is $f^{-1}$-invariant.	

Let $C$ be a nonempty connected component of $f^{-1}(\overset{o}{R_k})\cap \overset{o}{R_i}$. We need to show that $C$ is the interior of a horizontal sub-rectangle of $R_i$. Let $x\in C$ be any point, $\overset{o}{I_x}$ be the interior of the stable segment of $R_i$ passing through $x$, and $\overset{o}{I_x'}$ be the connected component of $\mathcal{F}^s\cap C$ containing $x$. Clearly, $\overset{o}{I_x'}\subset \overset{o}{I_x}$. If $I_x\neq I_x'$, then at least one endpoint $z$ of $I_x'$ lies in $\overset{o}{I_x}$ and then in the interior of $R_i$. However, $z$ must be equal to $f^{-1}(z')$ with $z'\in \partial^u R_k$, but since $\partial^u \mathcal{R}$ is $f^{-1}$-invariant, $z \in \partial^u \cR$ this leads to a contradiction, as $z\in \partial^u\mathcal{R} \cap \overset{o}{R_i}=\emptyset$. Similarly, we can show that $f(C)$ is the interior of a vertical sub-rectangle of $R_k$.
\end{proof}

\section{Hyperbolic diffeomorphisms.}
Let $S$ be a compact and oriented Riemann surface. The Riemannian metric is denoted by $|\cdot|$ and the metric induced on $S$ by $d$. The tangent space of $S$ at the point $x\in S$ will be denoted by $T_xS$, and the tangent bundle of $S$ is denoted by $TS$. Suppose $f:S\rightarrow S$ is a diffeomorphism. The \emph{tangent map} of $f$ is the map $Df:TS\rightarrow TS$, and the \emph{differential} at $x\in S$ is denoted by $D_xf:T_xS\rightarrow T_{f(x)}S$. In this thesis, we will always assume that the diffeomorphisms we study are orientation preserving.

\subsection{Hyperbolic sets} Hyperbolic sets are fundamental in dynamical systems, defined by the contraction and expansion properties of the differentials in their tangent spaces. This simple geometric definition has profound dynamical implications, which we will explore further. We will now present a formal definition and introduce some key results that we will utilize later. For the sake of simplicity, we will focus on the case of surfaces, although the theory can be extended to any dimension.

\begin{defi}\label{Defi: Hyperbolic sets}
Let $S$ be a compact and oriented surface, and let $f: S \rightarrow S$ be a diffeomorphism. A compact set $K \subset S$ is hyperbolic if the following properties are satisfied:
\begin{enumerate}
	\item  $K$ is $f$-invariant, that is:  for all $z \in \ZZ$, $f^z(K)=K$.
	\item There exists a continuous decomposition of the tangent bundle of $S$ restricted to $K$ as a direct sum of two sub-bundles:
	 $$
	 T_K S=\EE^s \oplus \EE^u,
	 $$
	where $\EE^s$ is called the \emph{stable bundle} and $\EE^u$ is called the \emph{unstable bundle}. We require this decomposition to be $Df$-invariant:
	$$
	Df(\EE^s)=\EE^s \text{ and }  Df(\EE^u)=\EE^u.
	$$.
	\item There are constants $1 < \lambda$ and $C > 0$ such that for all $x\in K$,  $u_x\in \EE^u_x$ and all $v_x\in \EE^s_x$, and  $n\in \NN$, the following inequalities are satisfied for:
	\begin{eqnarray*}
	\Vert Df^n_x(v_x)\Vert_{f(x)} \leq \lambda^{-n}C\Vert v_x \Vert_x\\
	\Vert Df^{n}_x(u_x)\Vert_{f(x)} \geq \lambda^n C \Vert u_x \Vert_x
	\end{eqnarray*}
\end{enumerate}
\end{defi}

Regarding the last item of the definition, it is usually said that the differential is a contraction restricted to the stable bundle and an expansion in the unstable bundle by a constant factor of $\lambda^{-1}$ and $\lambda$, respectively. It is a classic result that there exists a Riemannian metric on $S$ called \emph{adapted metric}, such that:
\begin{itemize}
\item The constant $C$ is equal to $1$
\item The contraction and expansion is detected in the fist iteration:
	\begin{eqnarray*}
	\Vert Df_x(v_x)\Vert_{f(x)} \leq \lambda^{-1}\Vert v_x \Vert_x\\
	\Vert Df_x(u_x)\Vert_{f(x)} \geq \lambda \Vert u_x \Vert_x
\end{eqnarray*}
\end{itemize}
 We always assume that our Riemannian metric is adapted to the diffeomorphism.

\subsection{Invariant foliations} Similar to the case of pseudo-Anosov homeomorphisms, hyperbolic sets of $f$ define two transverse \emph{laminations} invariant under $f$. The following is a brief presentation of this theory, which begins with the notion of stable and unstable manifolds.

\begin{defi}\label{Defi: Stable/unstable manifolds}
Let $S$ be a compact oriented Riemann surface, $f:S\rightarrow S$ be a diffeomorphism, $K$ be a hyperbolic set of $f$, and $x \in K$. The \emph{stable manifold} of $x$, denoted as $W^s(x,f)$, is defined as follows:
$$
W^s(x,f):=\{y\in S:\lim_{n\rightarrow \infty} d(f^n(x),f^n(y))=0\}.
$$
Let $\epsilon > 0$. The \emph{local stable manifold} of $x$ with size $\epsilon$ is defined as the set:
\begin{eqnarray*}
W^s_{\epsilon}(x,f):=\{y\in S : \lim_{n\rightarrow \infty} d(f^n(x),f^n(y))=0 \text{ and }\\  
\forall n\in \NN, \,\, d(f^n(x),f^n(y))\leq \epsilon\}.
\end{eqnarray*}
The \emph{unstable manifold} and \emph{local unstable manifold} of $x$ with size $\epsilon$ are similarly defined as:
$$
W^u(x,f):=\{y\in S:\lim_{n\rightarrow \infty} d(f^{-n}(x),f^{-n}(y))=0\}.
$$
and 
\begin{eqnarray*}
W^u_{\epsilon}(x,f):=\{y\in S : \lim_{n\rightarrow \infty} d(f^{-n}(x),f^{-n}(y))=0  \text{ and }\\
 \forall n\in \NN, \,\, d(f^{-n}(x),f^{-n}(y))\leq \epsilon  \}.
\end{eqnarray*}
\end{defi}

It is possible to express the stable and unstable manifolds in terms of the local stable and unstable manifolds as follows:
\begin{eqnarray*}
W^s(x,f)=\cup_{n\geq 0}f^{-n}(W^s_{\epsilon}(f^n(x),f)).\\
W^u(x,f)=\cup_{n\geq 0}f^{n}(W^s_{\epsilon}(f^{-n}(x),f)).
\end{eqnarray*}

It is a classical result that the stable and unstable manifolds are, in fact, injective immersed sub-manifolds of $S$ that are as regular as the diffeomorphism $f$ itself. They are tangent to their respective stable or unstable bundles at every point. The following version of the stable manifold theorem corresponds to \cite[Theorem 6.2]{shub2013global}, and here we recall the main properties that we are going to use.

\begin{theo}\label{Theo: Stable manifold theorem}
Let $K$ be a hyperbolic set of $f$. There exists an $\epsilon > 0$ such that for all $x \in K$:
\begin{enumerate}
	\item $W^s_{\epsilon}(x,f)$ is an embedded disk of the same dimension that $\EE^s_x$, $T_xW^s_{\epsilon}(x,f)=\EE^s_x$ and similarly for the unstable case. 
	\item for all $y\in W^s_{\epsilon}(x,f)$ and every $n\geq 0$
	$$
	d(f^n(x),f^n(y))\leq \lambda^n d(x,y)
	$$
	and for all  $y\in W^u_{\epsilon}(x,f)$ and every $n\geq 0$
	$$
	d(f^{-n}(x),f^{-n}(y))\leq \lambda^n d(x,y);
	$$
	\item The embedding of $W^{s,u}_{\epsilon}(x,f)$ varies continuously with respect to  $x$.
	\item  It is possible to write
	$$
	W^s_{\epsilon}(x,f)=\{y\in S :	\forall n\in \NN \,\, d(f^n(x),f^n(y))\leq \epsilon\}.
	$$
	and
	$$
	W^u_{\epsilon}(x,f)=\{y\in S :	\forall n\in \NN \,\, d(f^{-n}(x),f^{-n}(y))\leq \epsilon  \}.
	$$
	\item The manifold $W^{s,u}_{\epsilon}(x,f)$ is as smooth as $f$.
\end{enumerate}
\end{theo}

\begin{conv*}
If there is no possibility for confusion about the diffeomorphism in question, we denote the (local) stable and unstable manifolds of $f$ at the point $x$ as $W^{s}(x)$ and $W^u(x)$ ($W^{s,u}_{\epsilon}(x)$), respectively.
\end{conv*}

A surface diffeomorphism $f: S \rightarrow S$ satisfies the \emph{Axiom A} if the set of \emph{non-wandering points} of $f$, denoted $\Omega(f) \subset S$, is hyperbolic and coincides with the closure of the periodic points of $f$, i.e., $\Omega(f) = \overline{\text{Per}(f)}$.

The diffeomorphism $f$ satisfies the \emph{strong transversality condition} if for every $x, y \in \Omega(f)$, the stable manifold of $x$, $W^s(x)$, intersects transversely (at any point) with the unstable manifold of $y$, $W^u(y)$.

\begin{defi}[Smale diffeomorphism]\label{Defi: Smale difeomorphism}
A surface diffeomorphism $f: S \rightarrow S$ is a \emph{Smale diffeomorphism} if it satisfies the Axiom A and has the strong transversality condition.
\end{defi}

One of the fundamental results in the hyperbolic theory of dynamical systems, established by Robbin \cite{robbin1971structural}, Robinson \cite{robinson1976structural}, and Ma~n'e \cite{mane1987proof}, states that in the $C^1$-topology, being a Smale diffeomorphism is equivalent to being structurally stable. We have introduced the previous definition as the transversality of the invariant manifolds is a concept commonly utilized in our discussion.  Additionally, Smale demonstrated  that the non-wandering set of diffeomorphisms satisfying the Axiom A can be decomposed into a finite collection \emph{basic pieces}.

\begin{defi}\label{Defi: Basic piece}
Let $f$ be a surface diffeomorphism that satisfies the Axiom A. A \emph{basic piece} of $f$ is any compact invariant set $\Lambda \subset \Omega(f)$ that is transitive and maximal among the sets satisfying these properties.
\end{defi}

\subsection{The structure of Invariant Manifolds for Saturated Sets} A remarkable property of the basic pieces is that they are saturated sets.

\begin{defi}\label{Defi: Saturated set}
Let $f$ be a diffeomorphism of a compact surface. A hyperbolic set $K$ of $f$ is called saturated if $W^s(K)\cap W^u(K)\subset K$.
\end{defi}

Understanding the topology of invariant curves in basic pieces and saturated sets is crucial for our purposes. We will introduce various types of points and segments within a saturated set, along with their invariant manifolds, which will be of great importance in our subsequent analysis.

\begin{defi}\label{Defi: Boundary points}
Let $f:S \rightarrow S$ be an orientation-preserving Smale surface diffeomorphism, and let $K$ be a saturated set of $f$ with $x\in K$. We introduce the following definitions:
\begin{enumerate}
	\item The point $x$ is $s$-\emph{boundary} if there exists an open interval (relative to the unstable manifold of $x$ ) contained in the unstable manifold of $x$, disjoint from $K$ but having $x$ in its closure relative to $W^u(x)$.	
	
	\item The point $x$ is $u$-\emph{boundary} if there exists and open interval (relative to the stable manifold of $x$ ) contained in the stable manifold of $x$,  disjoint from $K$ but having $x$ in its closure relative to $W^s(x)$.	
	
	\item Any connected component of $W^s(x)\setminus{x}$ is called a \emph{stable separatrice} of $x$. Similarly, a connected component of $W^u(x)\setminus{x}$ is called an \emph{unstable separatrice} of $x$.
	
	\item A separatrix of $x$ is \emph{free} if it is disjoint from $K$.
	
	\item Let $x$ be an $s$-boundary point.	Suppose $W^u_0(x)$ is an unstable separatrice of $x$, and there exists an open interval of $W^u_0(x)$, disjoint from $K$ but having $x$ in its closure, in this case we say that $x$ is isolated from the hyperbolic set $K$ on the side of the separatrice $W^u_0(x)$. If there is no such interval we say that $x$ is saturated by $K$ on the side of $W^u_0(x)$.
	
	\item Let $x$ be an $u$-boundary point, assume that $W^s_0(x)$ is a stable separatrice of $x$ and there exists a stable interval of $W^s_0(x)$, disjoint from $K$ but having $x$ in its closure. Then we say that $x$ is isolated from $K$ on the side of the separatrice $W^s_0(x)$. If there is no such interval we say that $x$ is saturated by $K$ on the side of $W^s_0(x)$.
	
	\item If $x$ is isolated from $K$ on both stable separatrices, we say $x$ is a \emph{double $u$-boundary point}. If $x$ is isolated from $K$ on  its two unstable separatrices, we say $x$ is a \emph{double $s$-boundary point}.
	
\end{enumerate}
\end{defi}

We are going to enumerate some properties of saturated basic sets. The results correspond to \cite[Proposition 2.1.1]{bonatti1998diffeomorphismes}.

\begin{prop}\label{Prop: sub-boundary points 2.1.1}
For a Smale surface diffeomorphism $f:S \rightarrow S$ that preserves the orientation of $S$ and a saturated hyperbolic set $K$, the following propositions hold true for $s$-boundary points (and their equivalent formulations for $u$-boundary points).
\begin{enumerate}
\item If $x$ is an $s$-boundary point, then every point in the $f$-orbit of $x$ is also an $s$-boundary point, and every point in the stable manifold of $x$, $W^s(x)$, is an $s$-boundary point.

\item If $x$ is a periodic $s$-boundary point, then $x$ has at least one free unstable separatrix.

\item The set of periodic $s$-boundary points is finite.

\item If $x$ has a free separatrix, then $x$ is periodic.

\item If $x$ is an $s$-boundary point, then $x$ lies in the stable manifold of a periodic $s$-boundary point.

\item If $K$ has no $s$ or $u$-boundary points, then $K$ is an \emph{Anosov diffeomorphism}.  If $K$ has $u$-boundary points but no $s$-boundary points, then $K$ is a \emph{hyperbolic attractor}.  If $K$ has $s$-boundary points but no $u$-boundary points, then $K$ is a \emph{hyperbolic repeller}.

\item A point $x\in W^s(K)$ is a double $s$-boundary point if and only if it corresponds to the stable manifold of a periodic point whose orbit is a basic piece of $f$ and is minimal under the Smale order.

\end{enumerate}
\end{prop}

We should not emphasize the Smale order because we are not going to use it. In fact, all the theory we need will be applied to a special type of saturated sets: those that contain neither attractors nor repellers and are not Anosov.

\begin{defi}\label{Defi: saddle-type basic piece}
Let $f$ be a surface Smale diffeomorphism. A saturated hyperbolic set $K$ of $f$ is \emph{saddle-type} if $K$ has both $s$-boundary and $u$-boundary points.
\end{defi}

A saddle-type hyperbolic set could consist only of an isolated hyperbolic periodic point. In this case, such a periodic point is  a double $s$-boundary point and a double $u$-boundary point. Here comes a convention.

\begin{conv}\label{Conv: non double boundary points}
In this thesis, every saddle-type hyperbolic set $K$ does not contain double $s$ or $u$ boundary points and is not reduced to a trivial hyperbolic point. Furthermore, our focus is on studying mixing saddle-type basic pieces. However, we present the theory in a slightly more general setting to highlight certain difficulties and phenomena that are better understood for non-trivial saddle-type basic sets without double boundaries.
\end{conv}

\subsection{Markov partitions for saddle-type saturated sets}
A saddle-type saturated set without double boundaries has the property of being totally disconnected, and the stable and unstable laminations are locally modeled by the product of a Cantor set and an interval. This property allows us to define Markov partitions for the hyperbolic set using disjoint embedded rectangles, which constitutes the main difference compared to the pseudo-Anosov case. For this definition, it is necessary to exclude saturated hyperbolic sets that contain double boundary points. In such cases, we would need to consider degenerated rectangles within our Markov partition, i.e., rectangles that are reduced to stable or unstable intervals. This observation justifies the Convention \ref{Conv: non double boundary points}. In the context of saturated hyperbolic sets without double boundaries, the definitions of rectangles, boundaries, and the interior of rectangles are similar to those discussed in the previous section and are quite intuitive.

\begin{defi}\label{Defi: Markov partition for Smale difeos}
Let $K$ be a saddle-type saturated set of $f:S \rightarrow S$. A \emph{Markov partition} of $K$ is a finite family of disjoint embedded rectangles $\cR={R_i}_{i=1}^n$ satisfying the following conditions:
\begin{itemize}
	\item $K\subset \cup_{i=1}^n R_i$
	\item For all $i,j\in \{1,\cdots,n\}$, , every nonempty connected component of $R_i \cap f^{-1}(R_j)$ is a horizontal sub-rectangle of $R_i$ and a vertical sub-rectangle of $f^{-1}(R_j)$.
	\end{itemize}
\end{defi}

A Markov partition for a saturated set admits an equivalent formulation as previously exposed for pseudo-Anosov homeomorphisms in Proposition \ref{Prop: Markov criterion boundary}. The proofs of the following result are the same as in the other case.

\begin{prop}\label{Prop: Markov criterion boundary diffeomorphism}
	Let $f:S \rightarrow S$ Smale surface diffeomorphism, be $K$ a saddle-type saturated set and be $\cR=\{R_i\}_{i=1}^n$ a family of  rectangles such that:
	\begin{itemize}
		\item $K\subset \cup_{i=1}^n R_i$ and
		\item for all $i\neq j$,  $R_i\cap R_j=\emptyset$
	\end{itemize}
	Then the family $\cR$ is a Markov partition of $K$ if and only if:
	$\partial^s\cR=\cup_{i=1}^m\partial^sR_i$ is $f$-invariant and $\partial^u\cR=\cup_{i=1}^m\partial^uR_i$ is $f^{-1}$-invariant.
\end{prop}

A geometric Markov partition for a saddle-type saturated set is a Markov partition in which we have chosen an orientation for each rectangle in the partition. Similar to the pseudo-Anosov situation, these orientations determine an ordered family of disjoint horizontal and vertical sub-rectangles of the Markov partition. The horizontal sub-rectangles of $R_i$ are denoted by $\{H^i_j\}_{i=1}^{h_1}$, and $\{V^k_l\}_{l=1}^{v_k}$ denotes the vertical sub-rectangles of $R_k$.

\begin{defi}\label{Defi: geometric type Smale Markov partition}
	Let $f: S \rightarrow S$ be a Smale surface diffeomorphism, $K$ a saddle-type saturated set, and $\cR=\{R_i\}_{i=1}^n$ a geometric Markov partition of $K$. The geometric type of $\mathcal{R}$ can be described by a quadruple
	$$
	T:=T(f,\cR)=(n,\{h_i,v_i\}_{i=1}^n,\Phi_T:=(\rho_T, \epsilon_T))
	$$ 
	where the functions are  determine by the formulas:
	$$
	\rho_T(i,j)=(k,l) \text{ if and only if } f(H^i_j)=V^k_l,
	$$
	and $\epsilon_T(i,j)=1$ if $f\vert_{H^i_j}$ preserves the vertical direction, and $-1$ otherwise.
\end{defi}

If $\mathcal{R}$ is a Markov partition for a basic piece $K$ of a Smale surface diffeomorphism, for each $m \in \mathbb{N}$, the function $f^m$ is still a Smale surface diffeomorphism, but $K$ is not necessarily a basic piece; it is a saturated hyperbolic set. Nonetheless, $\mathcal{R}$ can still serve as a family of rectangles and a Markov partition for $K$ as a saturated set of $f^m$. This Markov partition has a geometric type denoted by $T^m$, which is the power $m$ of $T$. Similarly to the pseudo-Anosov situation, to indicate that $\mathcal{R}$ is viewed as a Markov partition of $f^m$, we write $(f^m,\mathcal{R})$.

 \section{Survey on Bonatti-Langevin theory for Smale diffeomorphism}\label{Sec: Bonatti-Langevin theory}

Geometric types were introduced by Christian Bonatti and Rémi Langevin to study the dynamics of surface Smale diffeomorphisms in the neighborhood of a saddle-type saturated set without double boundaries. In a series of papers \cite{beguin2002classification}, \cite{beguin2004smale}, and his Thesis \cite{beguin1999champs}, François Béguin gives a \emph{Characterization of realizable geometric types} and develops an algorithm to determine when \emph{two geometric types represent the same saturated set}.

At some point, we will transfer our problems of \emph{realization} and \emph{decidability} from homeomorphisms to the context of basic saddle-type pieces for Smale's surface diffeomorphisms. Therefore, we will provide a review of the concepts and results that we will use, emphasizing the simplified version of the Bonatti-Langevin theory that applies to  basic pieces that don't have double boundaries. This obey that in the future we are going to deal with iteration of the diffeomorphism restricted to mixing basic piece.

The \emph{Smale Spectral Decomposition Theorem} asserts that the non-wandering set of an Axiom A diffeomorphism $f$ can be decomposed into a finite number of basic pieces. Each of these basic sets $K$ can further be decomposed into a finite disjoint union of invariant sets $\{K_i\}$, that under the iteration of $f^n$ for some power $n$, are mixing. As a consequence, these $K_i$ become basic sets of $f^n$.

This phenomenon introduces complexity, as now a basic set $K$ might only be saturated for $f^n$, and not for the original $f$. However, this complication doesn't arise if $K$ is mixing, allowing us to work within the same theoretical framework as before. We can still consider powers of $f$, their Markov partitions, and the geometric types of these partitions in a coherent manner.

\subsection{ The domain of a basic piece}\label{Sub-sec: Domain basic piece}

Once again, let $f: S \rightarrow S$ be a Smale surface diffeomorphism, and let $K$ be a saddle-type basic piece of $f$. The objective of this subsection is to define an invariant neighborhood of $K$ that has $K$ as its maximal invariant set. Such a surface has finite genus, is minimal in a certain sense, and is uniquely determined up to topological conjugacy. This distinguished neighborhood is called the \emph{domain of $K$} and is denoted by $\Delta(K)$. Later on, we will define the \emph{formal derived from pseudo-Anosov} of a geometric type in terms of this domain. More details about this construction can be found in \cite[Chapter 3]{bonatti1998diffeomorphismes}.

If our domain is intended to be an invariant neighborhood of $K$, it needs to contain the stable and unstable manifolds of $K$. Moreover, for it to have finite topology, it needs to contain all polygons bounded by arcs of stable and unstable manifolds. This is the motivation for the following definition.

\begin{defi}\label{Defi: restricted domain}
The \emph{restricted domain} of $K$ is the set $\delta(K)$, consisting of the union of the stable and unstable manifolds of $K$ together with all disks bounded by a polygon formed by arcs of the stable and unstable manifolds of $K$.
\end{defi}
 
 A basic piece $\Lambda$ is always an isolated set, which means that there exists an open neighborhood $U$ of $\Lambda$ such that $\Lambda = \cap_{n=-\infty}^{\infty} f^n(U)$. This $U$ is an invariant neighborhood of $K$. This property can be deduced from the fact that every Smale diffeomorphism admits a \emph{filtration adapted} to $f$ (see \cite[Chapter 2]{shub2013global}). Such a filtration produces an invariant neighborhood $U$ of $f$ with the following properties:
  
\begin{itemize}
\item $U$ is $f$ invariant.
\item $K \subset U$, and for all compact subsets of $U$, their maximal invariant sets are contained in $K$. In other words, $K$ is the maximal invariant set of $U$.
\item $U$ admits a compactification that is a surface of finite type $S'$. Such a surface carries a Smale diffeomorphism, and the complement of $U$ in $S'$ consists of a finite number of periodic attractor or repeller points.
\item The invariant manifolds $W^s(K)$ as $W^u(K)$ are closed sub-sets of $U$.
\item We could suppose every connected component of $U$ intersect $K$.
\item Finally, $\delta(K)\subset U$.
\end{itemize}

Such $U$ is not canonical in the sense that it depends on the filtration. To obtain the domain $\Delta(K)$, Bonatti-Langevin describe the complement of $\delta(K)$ in $U$. The domain of $K$ results from extracting a neighborhood of $\delta(K)$ in $U$. This procedure is discussed in \cite[Proposition 3.1.2]{bonatti1998diffeomorphismes}.

It turns out that $U \setminus \delta(K)$ has a finite number of connected components, all of which are homeomorphic to $\mathbb{R}^2$ and periodic under the action of $f$. The specific construction of $\Delta(K)$ is described in \cite[Subsection 3.2]{bonatti1998diffeomorphismes}, but it essentially involves removing some "strips" from the complement of $U \setminus \delta(K)$. This results in a compact surface, $\Delta(K)$, from which we have removed a finite number of periodic points. We will consider the class of surfaces homeomorphic to $\Delta(K)$ with some additional properties restricted to $K$ as the domain of $K$.

\begin{defi}\label{Defi: Domain of K}
We call a neighborhood $O$ of $K$ a \emph{domain} of $K$ if it is homeomorphic to $\Delta(K)$ under a homeomorphism that coincides with the identity over $K$ and conjugates the restrictions of $f$ to $O$ and $\Delta(K)$.
\end{defi}

The uniqueness, or the canonical nature, of $\Delta(K)$ is summarized in the following proposition, corresponding to \cite[Proposition 3.2.2]{bonatti1998diffeomorphismes}.

\begin{prop}[uniqueness of domain]\label{Prop: unicity of domain}
For every neighborhood $D$ of $K$ that is $f$-invariant and contains $\delta(K)$, there exists an embedding from $\Delta(K)$ to $D$ that is the identity over the restricted domain and commutes with $f$ restricted to $\delta(K)$.
\end{prop}

\begin{defi}\label{Defi: versatile neighboorhood}
A versatile neighborhood of $K$ refers to any $D$-neighborhood of $K$ that is $f$-invariant, has finite topology (i.e., a compact surface minus a finite number of points), and finite genus. Furthermore, it satisfies the minimality property described in Proposition \ref{Prop: unicity of domain}, which states that for every $f$-invariant $D'$-neighborhood of $K$ with finite topology, there exists an embedding from $D$ to $D'$ that commutes with $f$ and is the identity over $\delta(K)$.
\end{defi}

The importance of the domain, apart from its minimality, lies in the fact that $\Delta(K)$ is unique in a specific sense. The main challenge is that the domain is neither closed nor open. In fact, it is through the use of the boundary components of the domains that we are able to connect different domains of basic pieces together and reconstruct the dynamics of the entire diffeomorphism.  Bonatti-Langevin introduce the concept of a \emph{closed module} for $f$ (refer to \cite[Definition 3.2.4]{bonatti1998diffeomorphismes}). While we won't delve into the details of this definition, the significance of it is captured in the following corollary (\cite[Corollary 3.27]{bonatti1998diffeomorphismes}).

\begin{coro}[Uniqueness of the Domain]\label{Coro: Uniqueness of domain}
Any $f$-invariant neighborhood of $K$ with finite topology that is versatile, closed modulo $f$, and such that every connected component of it intersects $K$, is conjugate to $\Delta(K)$ by a homeomorphism that coincides with the identity on $\delta(K)$.
\end{coro}

This corollary allows us to provide a representation of the Smale diffeomorphism on a basic piece in terms of its geometric type. This result is presented in  \cite[Theorem 5.2.2]{bonatti1998diffeomorphismes}, and we reproduce it below:

\begin{theo}\label{Theo: Presentation in a domain}
Let $f$ and $g$ be two Smale diffeomorphisms on compact surfaces, and let $K$ and $L$ be two basic pieces for $f$ and $g$ respectively. We denote their domains as $(\Delta(K), f)$ and $(\Delta(L), g)$. If $K$ and $L$ have Markov partitions with the same geometric types, then there exists a homeomorphism from $\Delta(K)$ to $\Delta(L)$ that conjugates $f$ and $g$.
\end{theo}

It is clear that if $f$ and $g$ are conjugate restricted to $\Delta(K)$ and $\Delta(L)$, and if $K$ admits a Markov partition of geometric type $T$, then its image under conjugation is a Markov partition of $L$ with geometric type $T$. Therefore, Theorem  \ref{Theo: Presentation in a domain} provides a necessary and sufficient condition for the conjugacy between $f$ and $g$ in terms of their geometric types and domains.

Two basic pieces $K$ and $L$ have Markov partitions of the same geometric type if and only if they have the same domain (up to conjugation), as stated in Theorem \ref{Theo: Presentation in a domain}. This allows us to provide the definition of the formal derived-from Anosov of a geometric type in the pseudo-Anosov case (see Definition \ref{Defi: Formal DA}).
\subsection{ The genus of a geometric type}\label{Sub-sec: The genus}

The pseudo-Anosov class of geometric type refers to the geometric type of Markov partitions of a generalized pseudo-Anosov homeomorphism. There exists a completely analogous notion for the geometric type of saddle-type basic pieces, which we introduce below.

\begin{defi}\label{Defi: realizable as Basic piece}
Let $T$ be a geometric type. $T$ is \emph{realizable as a (surface) basic piece} if: There exists a Smale surface diffeomorphism $f$ with a \emph{non-trivial} saddle-type basic piece $K$ and a Markov partition $\mathcal{R}$ of $K$ with geometric type $T$. In this case, we say that $K$ realizes the geometric type $T$.
\end{defi}

If we start with a geometric type $T$, we would like to know whether or not it is realizable as a basic piece of surfaces. Christian Bonatti, Rémi Langevin, and Emmanuelle Jeandenans have formulated in \cite[Chapter 7]{bonatti1998diffeomorphismes} a necessary condition for $T$ to be realizable as a basic piece of surface: the genus of $T$ must be finite. Subsequently, François Béguin has shown in \cite{beguin2004smale} that such a condition is sufficient. In this subsection, we define the genus of a geometric type and formulate a criterion that allows us to decide whether the genus is finite.

Let $T=(n,\{h_i,v_i\}_{i=1}^,\Phi)$ be a geometric type. The strategy for constructing a diffeomorphism on a compact surface (of finite type) with a basis piece realizing $T$ begins by considering a family of disjoint, non-degenerate rectangles $\cR=\{R_i\}_{i=1}^n$ together with a function $\phi$ defined piecewise on a family of horizontal sub-rectangles of $\mathcal{R}$ whose image is a family of vertical sub-rectangles of $\mathcal{R}$ satisfying the combinatorics of $\Phi$.

Any saddle-type basis piece has no double $s,u$-boundary and does not admit degenerate rectangles, i.e., rectangles reduced to a point or an interval. If we want to have hope that such a family of rectangles has, as its maximal invariant set, a non-trivial saddle-type basic piece, there is a combinatorial phenomenon dictated by $\Phi$ that we have to avoid.

Suppose there exists a subfamily of rectangles $\{R_{i(k)}\}_{k=1}^m$ of $\cR$, such that they all have a single horizontal sub-rectangle, and the function $\phi$ defined on the horizontal sub-rectangles of the family satisfies the following conditions: $\phi(R_{i(k)}\subset R_{i(k+1)}$ for $k \in \{1,\cdots, m-1\}$ and $\phi(R_{i(m)})\subset R_{i(1)}$. This condition implies that $\phi(R_{i(1)})\subset R_{i(1)}$ is a horizontal sub-rectangle of $R_{i(1)}$. Therefore, in order to have a uniform expansion (as we desire) in the unstable direction, it is necessary that $R_{i(k)}$ reduces to a stable interval. However, this is incompatible with the requirement that the desired basic piece should not have double $s$-boundaries or double $u$-boundaries. This property can be read from the geometric type $T$, which we describe in the following definition

	\begin{defi}\label{Defi: double boundary}
	A geometric type $T$ has a \emph{double $s$-boundary} if there exists a cycle of indices between $1$ and $n$:
	$$
	i_i\rightarrow i_2 \rightarrow\cdots \rightarrow i_k \rightarrow i_{k+1}=i_i.
	$$
	Such that, for all $t\in \{2,\cdots,k\}$, $h_{i_t}=1$ and $\phi_T(i_{t},1)=(i_{t+1},l)$.\\
	Using the inverse of the geometric type is possible to define \emph{double-$u$-boundary} as a double $s$-boundary of $T^{-1}$.
\end{defi}

  It was shown in \cite[Proposition 7.2.2]{bonatti1998diffeomorphismes} that a double boundary of geometric type $T$ corresponds to double boundaries on a hyperbolic set that has a Markov partition of geometric type $T$, so the analogy we present is compatible with the combinatorial formulation. Don't have double boundaries is the fist obstruction of $T$ to be realized as a saddle-type basic piece.

Once we have introduced our first obstruction for $T$, in order to construct a surface with boundaries and corners on which there exists an (affine) diffeomorphism realizing $T$, the next step is to formalize what we mean by a diffeomorphism with a combinatorics induced by $T$.
	
	\begin{defi}[Concretization]\label{Defi: Concretization}
			Let $T=(n,\{h_i,v_i\}{i=1}^n, \Phi_T)$ be a geometric type without double boundary. A \emph{concretization} of $T$ is a pair $(\cR=\{R_i\}_{i=1}^n,\phi)$ formed by a family of $n$ oriented, non-degenerate, and disjoint rectangles $\cR$  and a function $\phi$ defined over a subset of $\cup \cR:=\cup_{i=1}^n R_i$ with the following characteristics:
		
		\begin{enumerate}
			\item Every rectangle is endowed with the trivial vertical and horizontal foliations. We give an orientation to the vertical leaves of $R_i$ and let the horizontal leaves be oriented in such a way that a pair of horizontal and vertical leaves gives the orientation chosen for $R_i$.
			
			\item For every $1\leq i \leq n$ and $1\leq j \leq h_i$, there is a horizontal sub-rectangle $H^i_j$ of $R_i$. The sub-rectangles are disjoint and their order is compatible with the vertical orientation in $R_i$. We demand that the lower (upper) boundary of $H^i_1$ ($H^i_{h_i}$) coincides with the lower (upper) boundary of $R_i$.
					
			\item For every $1\leq k \leq n$ and $1\leq l \leq v_k$, we have a vertical sub-rectangle $V^k_l$ of $R_k$. The sub-rectangles are disjoint and the order is compatible with the horizontal orientation in $R_k$. We demand that the left (right) boundary of $V^k_1$ ($V^k_{v_k}$) coincides with the left (right) boundary of $R_k$.
			
			\item The function $\phi:\cup_{i,j}H^i_j \rightarrow \cup_{k,l}V^k_l$ is  orientation-preserving diffeomorphism restricted to every $H^i_j$ that preserves the horizontal and vertical foliations.
			
			\item If $\Phi_T(i,j)=(k,l,\epsilon)$, then $\phi(H^i_j)=V^k_l$.

			\item If $\epsilon(i,j)=1$, $f$ preserves the orientation of the vertical direction restricted to $H^i_j$, and $f$ reverses it in the case $\epsilon(i,j)=-1$.
		\end{enumerate}
	\end{defi}
	
	We have obtained a graphical representation of $T$, but another problem arises: the function $\phi$ is not necessarily hyperbolic. To ensure hyperbolicity, we require the existence of an integer $m\in\mathbb{N}$ such that, for all points in its maximal invariant set, the derivative of $\phi^m$ expands vectors in the vertical direction and contracts horizontal vectors uniformly. This leads us to the following definition.
	
	\begin{defi}[realization]\label{Defi: realization}
Let $T$ be a geometric type without double boundaries. A concretization of $T$ is considered a realization if the maximal invariant set of $\phi$ is hyperbolic.
	\end{defi}

An easy way to define $\phi$ is as an affine transformation over every horizontal sub-rectangle. Such concretizations are called \emph{affine concretizations}. It turns out that every affine concretization is a realization, and every geometric type without a double boundary admits an affine concretization. This implies the result of \cite[Proposition 7.2.8]{bonatti1998diffeomorphismes}, which we reproduce below.
	
	\begin{prop}\label{Prop: Types transitive have a realization}
	Every geometric type without a double boundary admits a realization.
	\end{prop}

Starting from a realization $(\cR=\{R_i\}_{i=1}^n,\phi)$ of $T$, and for each positive natural number $m \in \mathbb{N}+$, we will construct the "minimal" surface with boundaries and corners that supports $m-1$ iterations of $\phi$. This construction involves gluing horizontal and vertical sub-rectangles of the realization according to the rules determined by the diffeomorphism $\phi$. Let us introduce a general definition that will be helpful in our discussions.

\begin{defi}\label{Defi: Vetical-horizontal stripes}
For every $i\in \{1,\cdots,h_i-1\}$, the \emph{horizontal stripe} $\tilde{H^i_j}$ is the vertical sub-rectangle of $R_i$ that is bounded by $H^i_j$ and $H^i_{j+1}$. Similarly, we define the \emph{vertical stripe} $\tilde{V^k_l}$ with $l\in \{1,\cdots,v_k-1\}$.
\end{defi}

\begin{figure}[h]
	\centering
	\includegraphics[width=0.5\textwidth]{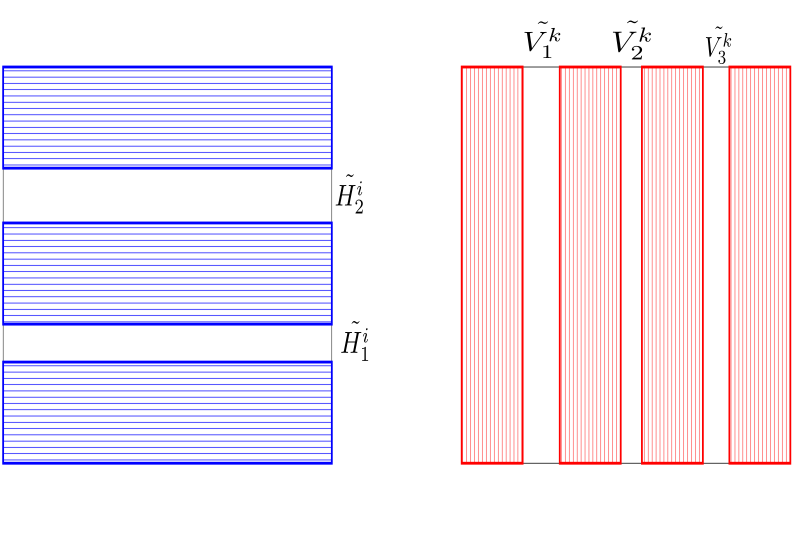}
	\caption{Horizontal and vertical stripes}
	\label{Fig: Stripes}
\end{figure}

We refer to the compact surface $S$ with boundary and vertices as an \emph{HV-surface} if it has an atlas whose charts are open sets contained in the upper half-plane, the first quadrant of the origin in $\RR^2$, or the union of the three quadrants of the origin in the real plane. Furthermore, it is required that the changes of charts are elements of Hom$(\RR)\times$Hom$(\RR)$. This is the type of surface where we will define and iterate $\phi$. It is not immediate that we can extend the dynamics to a finite-type compact surface, but this is the next problem we are going to address.

	\begin{defi}\label{Defi: m realizer }
	Let $T$ be a geometric type without double boundaries, and let $(\{R_i\}_{i=1}^n,\phi)$ be a realization of $T$. We denote the union of the rectangles as $\cR=\cup_{i=1}^n R_i$, and for each $m \in \NN$, we consider the disjoint union of $m$ copies of $\cR$ given by:
		$$
	\cup_{i=0}^m \cR \times \{i\}.
		$$
We define $\cR_m$ as the quotient space of this set by the equivalence relation: for all  $ 1\leq i \leq m-1$
		$$
		 (x,i)\sim(\phi^{-1}(x),i+1),
		$$
for each point where $\phi^{-1}$ is defined, and we leave $(x,i)$ unidentified otherwise.		
 The space $\cR_m$ is the $m$-th \emph{realizer} of $T$ relative to the realization $(\{R_i\}_{i=1}^n,\phi)$. 
	\end{defi}

The space $\cR_m$ is an oriented HV-surface and is endowed with two transverse foliations: the vertical and horizontal foliations (refer to \cite[Section 7.1]{bonatti1998diffeomorphismes} and \cite[Lemma 7.3.3]{bonatti1998diffeomorphismes} for more details). An HV-homeomorphism is a homeomorphism between HV-surfaces that conjugates the vertical and horizontal foliations.
	
For all $m\in \NN$  and every $i\leq m$, there exist an embedding:  
$$
\Psi :\cR\times \{i\} \rightarrow \cR_{m'},
$$

defined as $\Psi(x,i)=[x,i]_{m'}$, which turns out to be an HV-homeomorphism. We denote $R_i$ as $R_i:=\Psi(R_i\times \{0\})$. Let's consider two natural numbers $m' > m$ and $\Psi: \cR \times {i} \rightarrow \cR_{m'}$ as described above. Let's take
 $$
 \phi_m: \cup_{i=0}^{m-1} \Psi(\cR\times \{i\}) \subset \cR_{m'} \rightarrow \cup_{i=1}^m \Psi(\cR\times \{i\}) \subset \cR_{m'},
 $$
given by $\phi_m(\pi(x,i))=\Psi(x,i+1)$.  Provided $\phi(x)$ is defined, $\phi_m(\pi(x_i))=[(x,i+1)]_{m'}=[\phi(x),i]_m'=\Psi(\phi(x,i))$. In this way, the space $\cR_m$ is considered immersed in $\cR_{m'}$ and we have endowed $\cR_m$ with some 'dynamics' that is possible to iterate at most $m$ times.

	\begin{rema*}
 When it makes sense, $\phi_m([x,i])=[\phi(x),i]$, and for all $0 \leq k \leq m$, the image by $\Psi$ of $R_i \times {k}$ is equal to $\phi^k_m(R_i\times {0})$.
	\end{rema*}
	
The following proposition essentially states that there is a unique $m$-realizer up to HV-conjugation. It is proved in \cite[Proposition 7.3.5]{bonatti1998diffeomorphismes}.
 
	\begin{prop}\label{Prop: unique m-realizer}
		Let $T$ be a geometric type,$(\{R_i\},\phi)$ and $(\{R_i'\},\phi')$ be two realizations of $T$, and $\cR_m$ and $\cR_m'$ be their respective $m$-realizers. Then there exists an HV-homeomorphism $\theta:\cR_m \rightarrow \cR_m'$ which conjugates $\phi_m$ to $\phi_m'$. Moreover, the image of $R_i$ under $\theta$ is $R_i'$, and $\theta$ preserves the orientation of the two foliations on the rectangles.	
	\end{prop}
	
Suppose there exists a basic piece $K$ of a Smale surface diffeomorphism $f$ with a geometric Markov partition $\cR$ of geometric type $T$. We would like to compare the surface generated by the iterations $f^n(\cR)$ ($n\geq 0$) of the Markov partition and the $m$-realizer of the geometric type $T$. The \cite[Proposition 7.3.9]{bonatti1998diffeomorphismes} that we cite below provides insight into this problem.

	\begin{prop}\label{Prop: realizer and Markov partition}
	Let $T$ be the geometric type of a geometric Markov partition $\cR:=\{R_i\}$ of a basic piece $K$ of a Smale surface diffeomorphism $f$. If $\cR_m$ is the $m$-realizer of $T$ relative to any realization of $T$, then $\cup_{j=0}^m(\cup_{i}f^j(R_i))$  is an HV-surface, HV-homeomorphic to $\cR_m$.
	\end{prop}
	
The dynamics of $f$ restricted to the Markov partition $\cR$ induces a realization of $T$. Proposition \ref{Prop: unique m-realizer} implies that $f^m$ restricted to $\cup_{j=0}^m(\cup_{i}f^j(R_i))$ is conjugate to $\phi_m$. Each realizer $\cR_m$ has finite genus, and we define $g_m := \text{gen}(\cR_m)$ as the genus of $\cR_m$. Since $\cR_m$ is embedded in $\cR_{m+1}$, the sequence $\{g_m\}_{m\in \NN}$ is non-decreasing. Therefore, it is either stationary or tends to infinity. This observation allows us to define the genus of $T$.
	
	\begin{defi}\label{Defi: Genus of T}

Let $T$ be a geometric type without double boundaries. The genus of $T$ is the upper bound of the sequence $\{g_m\}_{m\in \NN}$ and is denoted by $\text{gen}(T)$.
	\end{defi}

	If the geometric type $T$ is realized as a basic piece, then there exists a surface of finite genus containing the realizer $m$ of $T$ for all $m in \NN$. It follows that the genus of $T$ is finite, as shown in \cite[Corollary 7.3.10]{bonatti1998diffeomorphismes}. Later, François Béguin proved in \cite[Corollary 4.4]{beguin1999champs} that this is a sufficient condition. Thus, we obtain the following theorem:
	
	\begin{theo}\label{Theo: finite genus iff realizable}

Let $T$ be a geometric type. $T$ is realizable as a saddle type basic piece of a surface Smale diffeomorphism if and only if $T$ has no double boundaries and gen$(T)$ is finite.
	\end{theo}
	
This is our main tool for determining whether or not a geometric type is realizable. But how can we check if $T$ does not have double boundaries and has finite genus? The next few sections are devoted to formulating a mechanism to verify this.

\subsection{ Topological formulation of finite genus and impasse}\label{Sub-sec: Top impasse genus}

In this subsection, we assume that $T$ is a geometric type without double boundaries, and $\cR_m$ is the $m$-realizer of $T$ with respect to a realization $(\{R_i \}_{i=1}^n,\phi)$. Our goal is to describe three topological obstructions for the geometric type $T$ to have finite genus. These obstructions are formulated in terms of the $m$-realizer and certain subsets called \emph{ribbons}.

\begin{defi}\label{Defi: k Ribbon}
 For each $k$ in $\{1,\ldots,m\}$, we define a \emph{ribbon} of $k$-generation of $\cR_m$ as the closure of any connected component of $\cR_k \setminus \cR_{k-1}$. Thus, a ribbon of $\cR_m$ can be any ribbon from the $k$-generation for $0 < k \leq m$.
\end{defi}

The following proposition summarizes some of the properties of the ribbons that are relevant to our discussion. A detailed exposition of these properties can be found in \cite[Chapter 7.4]{bonatti1998diffeomorphismes}.

\begin{prop}\label{Prop: properties ribbons}
Let $T$ be a geometric type without double boundaries, and let $(\{R_i\}_{i=1}^n, \phi)$ be a realization of $T$. Consider $0 < k \leq m$. The following propositions hold: 
\begin{itemize}
\item Every ribbon of $k$-generation is a rectangle whose boundary is contained in the horizontal boundary of $\cR$: $\partial^h\cR_0 = \cup_{i=1}^n \partial^h(\cup \pi(R_i \times {0}))$. Moreover, if $B$ is the complement of the interior of the vertical stripes and their positive iterations by $\pi$ contained in the horizontal boundary of  $\partial^h\cR_0$, then for all $m > 0$, the stripes of $\cR_m$ have horizontal boundaries in $B$.

\item A $k$-generation ribbon is the image of the horizontal stripes in $\cR_0$ under $\phi^k_m$, or equivalently, the image of the horizontal stripes $\tilde{H^i_j}\times \{k\}$ under $\pi$.

\item For all $m > 0$, the closure of each connected component of $\cR_m \setminus \cR_0$ is a $k$-generation ribbon for $0 < k \leq m$.
\end{itemize}
\end{prop}

From this result, we can deduce that there is a correspondence between $k$-ribbons and consecutive horizontal sub-rectangles of $\phi^k$. This correspondence is established as follows. Let $0 < k \leq m$ and $r$ be a $k$-generation ribbon of $\cR_m$. Suppose $r = \phi^k_m(\tilde{H^i_j} \times {0})$. Let $H^{i,(k)}{j}$ and $H^{i,(k)}{j+1}$ be two consecutive horizontal sub-rectangles of the rectangle $R_i = \pi(R_i \times {0})$ determined by the map $\phi^k$, such that $\tilde{H^i_j}$ is between them. Since $\phi^k$ is well-defined for such sub-rectangles, we have:

$$
\phi^k_m(H^{i,(k)}_{j}\times \{0\} )=\phi^k(H^{i,(k)}_{j})\times \{0\}=V^{z,(k)}_{l} \times  \{0\}
$$
and 
$$
\phi^k_m(H^{i,(k)}_{j=1}\times \{0\} )=\phi^k(H^{i,(k)}_{j})\times \{0\}=V^{z',(k)}_{l'}\times \{0\}.
$$
Where $V^{z,(k)}_{l}$ and $V^{z',(k)}_{l'}$ are horizontal sub-rectangles of the Markov partition $\cR_0=\{R_i=\pi(R_i\times \{0\})\}$  for the map $\phi^k$. i.e $(\cR_0,\phi^k)$.

Now we are ready to provide the topological criteria, in terms of the realization $\cR_m$, for $T$ to have finite genus. In the following definitions, as usual, we consider a geometric type  without double boundaries:
$$
T=(n,\{(h_i,v_i)\}_{i=1}^n, \Phi:=(\rho,\epsilon)),
$$
 and $(\{R_i\}_{i=1}^n, \phi)$  is a realization of $T$.
 
We reiterate that when we refer to the Markov partition on $\cR_m$, we mean the family of rectangles $\cR_0=\{\pi(R_i\times \{0\})\}$ under the action of $\phi_m$. However, formally speaking, it is not a Markov partition in the strict sense of the definition, as there is currently no complete diffeomorphism for which it satisfies all the properties of a Markov partition according to our definition.

\begin{defi}\label{Defi: Type 1 obstruction top}
Let $R$ be a rectangle in the Markov partition $\cR_0$, $A$ be a horizontal boundary component of $R$, and let $r$ be a ribbon with both horizontal boundaries in $A$.
The Markov partition in $\cR_m$ has a topological obstruction of type-$(1)$ if there exists another ribbon $r'$, distinct from $r$, with a single horizontal boundary contained in $A$ and lying between the horizontal boundaries of $r$. (See Figure \ref{Fig: obstruction 1}) for a visual representation.)

\begin{figure}[hh]
	\centering
	\includegraphics[width=0.3\textwidth]{Preliminares/Imagenes/Type1obs.png}
	\caption{ $(1)$-type obstruction.}
	\label{Fig: obstruction 1}
\end{figure}
\end{defi}

\begin{defi}\label{Defi: Type 2 obtruction top}
Let $A$ and $A'$ be two distinct horizontal boundary components of the Markov partition $\cR_0$. Consider a ribbon $r$ with one horizontal boundary contained in $A$ and the other boundary in $A'$.

The ribbon $r$ has a horizontal orientation when viewed as a horizontal sub-rectangle of a certain rectangle in the Markov partition, $\pi(R_i\times \{k\})$ (with $k\leq m$). This orientation may not correspond to the orientations of $A$ or $A'$ directly. However, we can assign orientations to $A$ and $A'$ that are compatible with the orientation of the horizontal boundary of $r$.
 
The Markov partition has a topological obstruction of type-$(2)$ in $\cR_m$ if there exists another ribbon $r'$ with one horizontal boundary $\alpha$ contained in $A$ and the other horizontal boundary $\alpha'$ contained in $A'$. With the orientations previously fixed on $A$ and $A'$, the order of $\alpha$ and $r\cap A$ in $A$ is the reverse of the order of $\alpha'$ and $r\cap A'$ in $A'$. (See Figure \ref{Fig: obstruction 2} for a graphical representation.)

\begin{figure}[hh]
	\centering
	\includegraphics[width=0.4\textwidth]{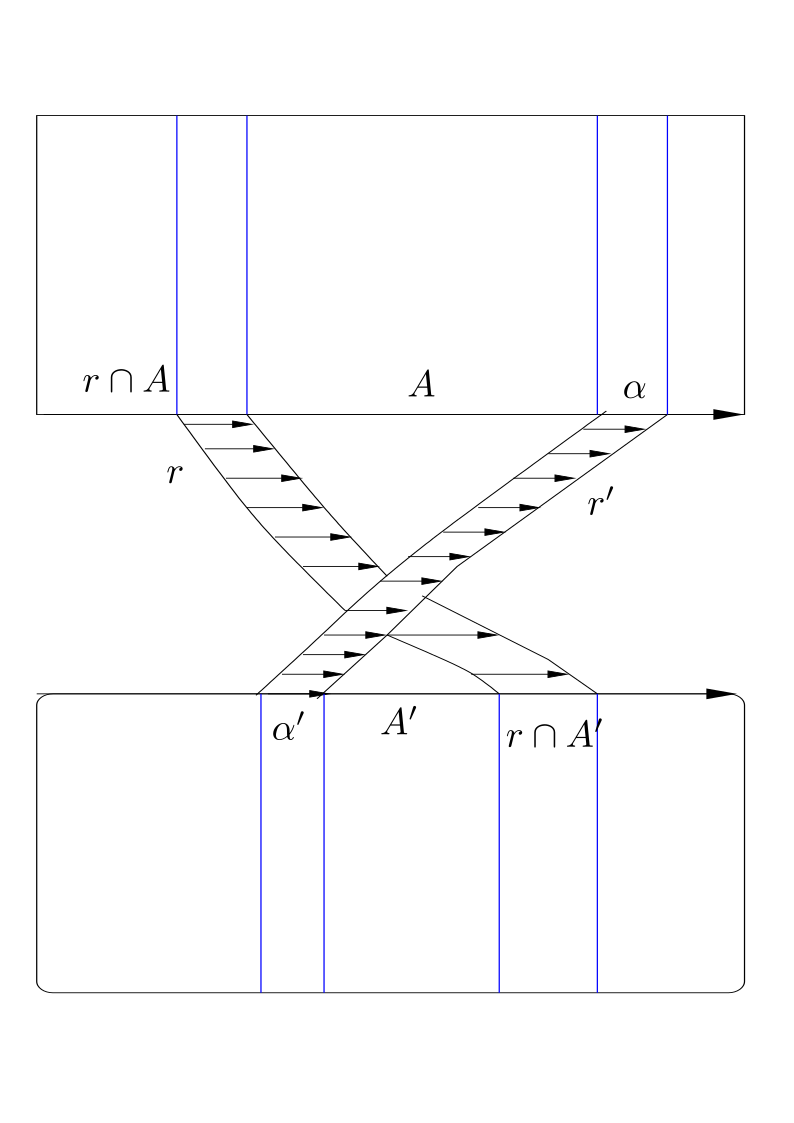}
	\caption{ $(2)$-type obstruction.}
	\label{Fig: obstruction 2}
\end{figure}
\end{defi}

\begin{defi}\label{Defi: p-peridic embrinary}
A horizontal boundary $A$ of a rectangle $R$ in the Markov partition is said to be $p$-\emph{periodic} if $\phi^p(A)\subset A$. If $A$ is $p$-periodic, we refer to any connected component of $A\setminus \pi^{kp}(A)$, with $k\in \NN_{>0}$, as an \emph{embryonic separatrix} of $A$.
\end{defi}

\begin{defi}\label{Defi:  Type 3 obtruction geo}
Let $S_1$, $S_2$, and $S_3$ be three different embryonic separatrices (i.e., they are on periodic sides but not in the image of the iterations of that side), and let $r$ be a ribbon with one horizontal boundary on $S_1$ and the other on $S_2$. We say that the Markov partition has a type-$(3)$ obstruction in $\cR_m$ if there exists another ribbon $r'$ with one side in $S_1$ and the other in $S_3$ (see Figure \ref{Fig: obstruction 3}).

\begin{figure}[hh]
	\centering
	\includegraphics[width=0.5\textwidth]{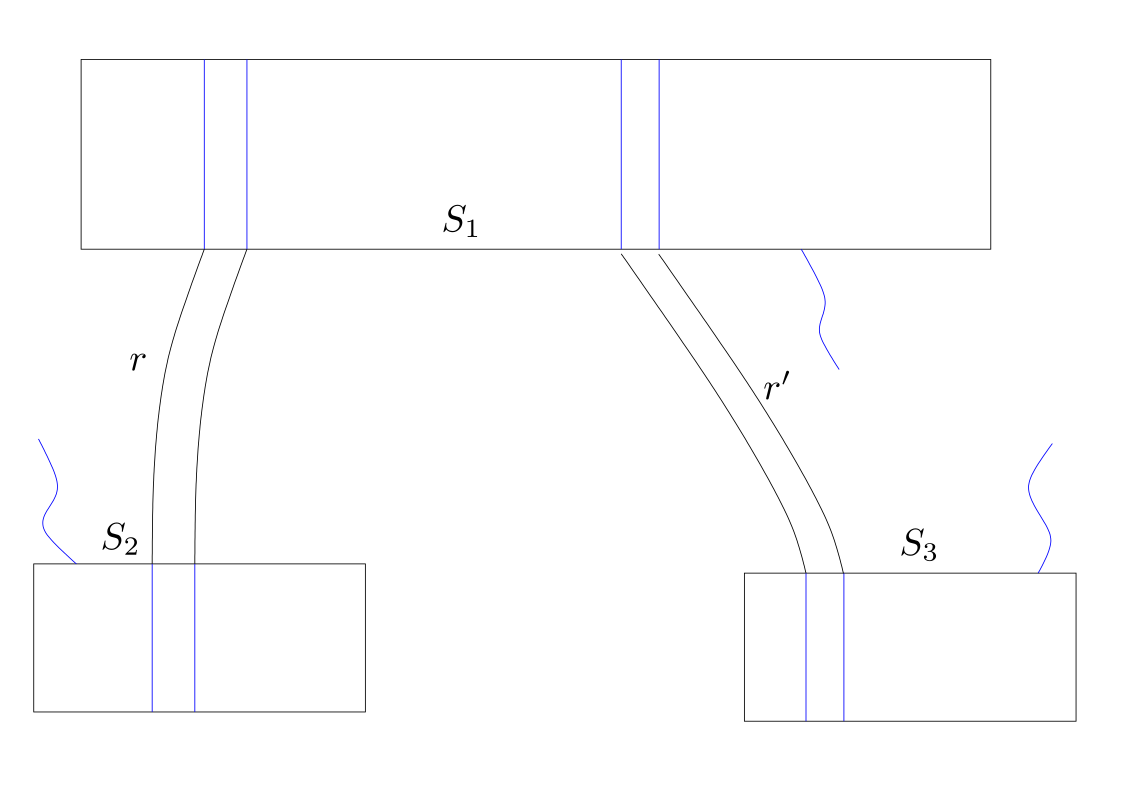}
	\caption{ $(3)$-type obstruction.}
	\label{Fig: obstruction 3}
\end{figure}

\end{defi}

 The following theorem is key to our future analysis, as it allows us to determine whether the genus of $T$ is finite by verifying the obstructions on a bounded number of realizers. This result corresponds to \cite[Theorem 7.4.8]{bonatti1998diffeomorphismes}. Here, we reformulate it in the context of basic pieces, as we will be using it.

\begin{theo}\label{Theo: finite type iff non-obtruction}
Let $T$ be a geometric type without double boundaries, let $(\{R_i\}_{i=1}^n,\phi)$  be a realization of $T$, and let $\cR_{6n}$ be its $6n$-realizer, where $n$ is the number of rectangles in the realization of $T$. The following statements are equivalent:

\begin{itemize}
	\item[i)] The Markov partition does not have obstructions of types $(1)$, $(2)$, and $(3)$ in $\cR_{6n}$.
	\item[ii)] The genus of $T$ is finite, i.e., $\text{gen}(T) < \infty$.
\end{itemize}
\end{theo}

Let's formulate another topological condition on the $6n$-realizer of geometric type $T$. This condition will be important in the problem of realization in the context of pseudo-Anosov homeomorphisms. We begin by assuming that $K$ is a basic piece of a Smale diffeomorphism of surfaces and that this piece has a Markov partition $\cR$ of geometric type $T$.

\begin{defi}\label{Defi: arc}
An $s,u$-arc is an interval contained in the unstable or stable foliation (respectively) of $K$ whose interior does not intersect $K$ but whose ends do.
\end{defi}

\begin{defi}\label{Defi: Impasse geo}
Let $f:S\rightarrow S$ be a surface Smale diffeomorphism, and let $K$ be a saddle-type basic piece of $f$. We call  \emph{topological impasse}, any open disk $\overset{o}{D}\subset S$ that is disjoint from $K$ and whose boundary consists of the union of a $u$-arc with an $s$-arc.
\end{defi}

\begin{figure}[h]
	\centering
	\includegraphics[width=0.6\textwidth]{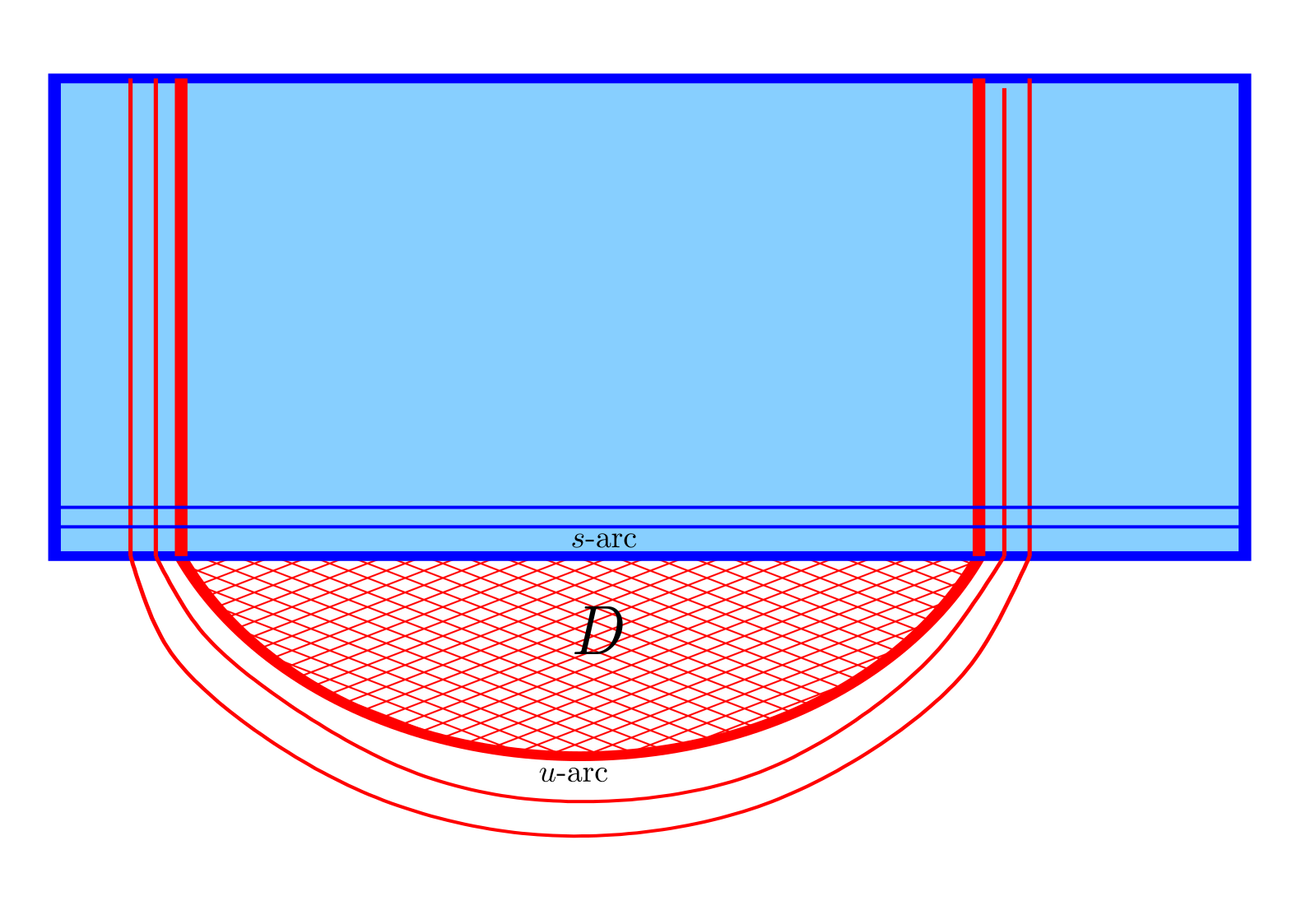}
	\caption{ Topological Impasse}
	\label{Fig: Impasse}
\end{figure}

Two $u$-arcs $\beta$ and $\beta'$ are equivalent if there is a rectangle that has $\beta$ and $\beta'$ as unstable boundaries. This relation defines an equivalence relation between $u$-arcs, and the equivalence classes are rectangles. An $u$-arc on the boundary of an equivalence class is called an extremal $u$-arc. Similarly, an equivalence relation is defined between $s$-arcs, and the concept of $s$-extremal arc is introduced (see \cite[Definition 2.2.6]{bonatti1998diffeomorphismes} for a precise definition). In \cite[Corollary 2.4.8]{bonatti1998diffeomorphismes}, it is proven that the boundary arcs of an impasse are extremal arcs.

\begin{rema}\label{Rema: extremal arc}

The important property of the arcs bounding an impasse is that the $u$-arc is saturated by leaves of the unstable lamination on the side opposite to the impasse, just as the $s$-arc is saturated by leaves of the stable lamination on the opposite side of the impasse.

\end{rema}

\subsection{ Combinatorial formulation of finite genus and impasse}\label{Sub-sec: Comb Impasse genus}

The obstructions described above provide a valuable tool for analyzing our Markov partition from a visual perspective. However, in order to obtain an algorithmic characterization of the pseudo-Anosov class of geometric types, we need to translate these three topological obstructions and the impasse condition into a combinatorial formulation that allows us to determine in finite time whether or not these conditions are satisfied.

The key insight lies in Proposition \ref{Prop: properties ribbons}, where we observed that every ribbon of $k$-generation corresponds to the image of a horizontal strip between two consecutive horizontal sub-rectangles of the Markov partition for the map $\phi^k$ under the projection $\pi$. In this section, we will consider the geometric type.

$$
T=(n,\{(h_i,v_i)\}_{i=1}^n,\Phi_T:=(\rho_T,\epsilon_T )).
$$

The trick in our upcoming definitions is that whenever we want to talk about a "ribbon" with a certain property, we need to refer to two consecutive horizontal rectangles and analyze their images under $\phi^k$. This way, a "ribbon" is formed, and we can express the obstructions and the impasse property in terms of the images of these consecutive horizontal rectangles.

\subsubsection{Combinatorial obstructions for finite genus.}
\begin{defi}\label{Defi: Type 1 combinatoric}
	A geometric type $T$ satisfies the \emph{combinatorial condition} of type-$(1)$ if and only if there are:
	
	\begin{itemize}
		\item Three (not necessarily distinct) numbers, $k,i_0,i_1\in \{1,\cdots,n\}$.
		\item Three different  numbers, $l_0,l_1,l_2\in \{1,\cdots,v_{k}\}$ with $l_0<l_1<l_2$.
		\item A pair  $(k',l')\in\cV(T)$, such that $(k',l')\notin \{(k,l):l_0\leq l \leq l_2\}$.
		\item And four pairs, $(i_0,j_0),(i_0,j_0+1),(i_1,j_1)(i_1,j_1+1)\in \cH(T)$
		\end{itemize}
	Such that they satisfy one of the following relations:
\begin{enumerate}
\item If  $\rho_T(i_0,j_0)=(k,l_0)$ and $\rho_T(i_0,j_0+1)=(k,l_2)$, we have two available possibilities:
\begin{itemize}
	\item[i)] $\epsilon_T(i_0,j_0)=1$ and $\epsilon_T(i_0,j_0+1)=-1$.  In such a case:
	 $$
	 \rho_T(i_1,j_1)=(k,l_1), \, \epsilon_T(i_1,j_1)=1 \text{ and } \rho_T(i_1,j_1+1)=(k',l'),
	 $$
	  or well,
	  $$
	  \rho_T(i_1,j_1+1)=(k,l_1), \,  \epsilon_T(i_1,j_1+1)=-1,  \text{ and } \rho_T(i_1,j_1)=(k',l').
	  $$
	
	\item[ii)] $\epsilon_T(i_0,j_0)=-1$ and $\epsilon_T(i_0,j_0+1)=1$. In such a case  
	$$
	\rho_T(i_1,j_1)=(k,l_1), \,  \epsilon_T(i_1,j_1)=-1 \text{ and  }\rho_T(i_1,j_1+1)=(k',l')
	$$
	or well
	$$
	\rho_T(i_1,j_1+1)=(k,l_1), \,  \epsilon_T(i_1,j_1)=1 \text{ and  } \rho_T(i_1,j_1)=(k',l').
	$$.

\end{itemize}
\item  In the symmetric case when, $\rho_T(i_0,j_0)=(k,l_2)$ and $\rho_T(i_0,j_0+1)=(k,l_0)$ there are two options:
\begin{itemize}
	\item[i)] $\epsilon_T(i_0,j_0)=1$ and $\epsilon_T(i_0,j_0+1)=-1$. In such a case: 
	$$
	\rho_T(i_1,j_1)=(k,l_1), \, \epsilon_T(i_1,j_1)=1  \text{ and }  \rho_T(i_1,j_1+1)=(k',l')
	$$
	or well
	$$
	\rho_T(i_1,j_1+1)=(k,l_1), \, \epsilon_T(i_1,j_1+1)=-1  \text{ and }  \rho_T(i_1,j_1)=(k',l')
	$$

	\item[ii)]  $\epsilon_T(i_0,j_0)=-1$ and $\epsilon_T(i_0,j_0+1)=1$. In such a case:  
	$$
	\rho_T(i_1,j_1)=(k,l_1),\, \epsilon_T(i_1,j_1)=-1 \text{ and  } \rho_T(i_1,j_1+1)=(k',l')
	$$
	or well
	$$
\rho_T(i_1,j_1+1)=(k,l_1),\, \epsilon_T(i_1,j_1)=1 \text{ and  } 	\rho_T(i_1,j_1)=(k',l').
	$$
\end{itemize}
\end{enumerate}
	A geometric type has the \emph{combinatorial obstruction} of type-$(1)$ if there exists $m \in \mathbb{N}_{>0}$ such that $T^m$ satisfies the combinatorial condition of type-$(1)$.
\end{defi}

The way we have defined the combinatorial condition of type $(1)$ is by identifying the $1$-generation ribbons of the Markov partition $\cR_0$ in $\cR_1$ that give rise to the topological obstruction of type-$(1)$ in $\cR_1$ as the image of a horizontal stripe contained between two consecutive sub-rectangles. This understanding of ribbons corresponds to two consecutive elements $(i,j)$ and $(i,j+1)$ in $\cH$ and their images under $\rho_T$ in $\cV$, which represent the ribbons. In the following lemma we are going to follow the Figure \ref{Fig: Type one proof} to explain or proof.

\begin{figure}[h]
	\centering
	\includegraphics[width=0.83\textwidth]{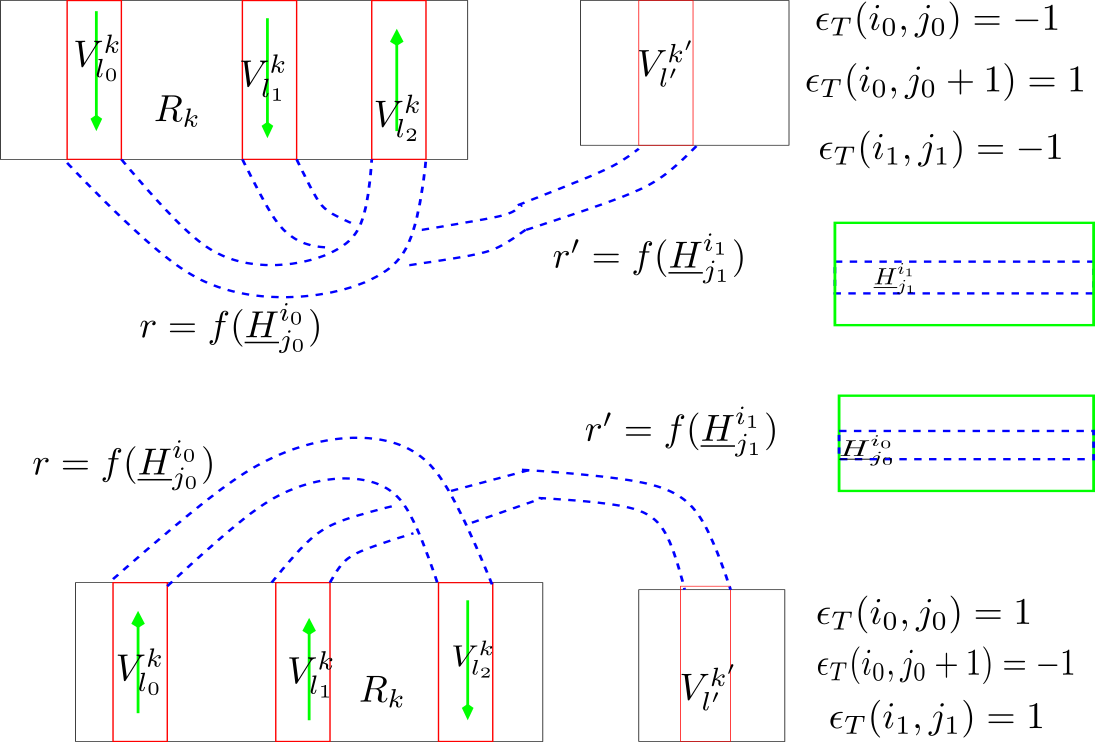}
	\caption{Type $(1)$ combinatorial condition}
	\label{Fig: Type one proof}
\end{figure}

\begin{lemm}\label{Lemm: Equiv com and top type 1}
A geometric type $T$ has the combinatorial condition of type-$(1)$ if and only if the Markov partition in $\cR_1$ exhibits the topological obstruction of type-$(1)$.
\end{lemm}

\begin{proof}
The terms $(k,l_0)$, $(k,l_1)$, and $(k,l_2)$ represent three different vertical sub-rectangles of $R_k$ ordered by the horizontal orientation of $R_k$. The term $(k',l')$ represents a horizontal sub-rectangle of $\cR$ that is not located between the rectangles $V^k_{l_0}$ and $V^k_{l_2}$. Let's illustrate the meaning of the condition in the first case, as the rest are symmetric.

The conditions $\rho(i_0,j_0)=(k,l_0)$ and $\rho(i_0,j_0+1)=(k,l_2)$ are equivalent to the fact that the image under $f$ of the horizontal strip $\underline{H}_{j_0}$ bounded by $H^{i_0}_{j_0}$ and $H^{i_0}_{j_0+1}$ is a ribbon $r=f(\underline{H}_{j_0})$ of the first generation that connects the horizontal boundaries of $V^k_{l_0}$ and $V^{k}_{l_2}$. In this case, in Item $(i)$, the conditions $\epsilon_T(i_0,j_0)=1$ and $\epsilon_T(i_0,j_0+1)=-1$ means that the ribbon $r$ connects the upper boundary of $V^k_{l_0}$ with the upper boundary of $V^{k}_{l_2}$.

The conditions $\rho(i_1,j_1)=(k,l_1)$ and $\rho(i_1,j_1+1)=(k',l')$ indicate that the image of the strip $\underline{H}_{j_1}$ bounded by $H^{i_1}_{j_1}$ and  $H^{i_1}_{j_1+1}$ is a ribbon $r'$ connecting a horizontal boundary of $V^k_{l_1}$ with a horizontal boundary of $V^{k'}_{l'}$. Moreover, since $\epsilon_T(i_1,j_1)=1$, it means that $r'$ connects the upper boundary of $V^k_{l_1}$ with the stable boundary $\alpha$ of $V^{k'}_{l'}$. But the condition $(k',l')\notin \{(k,l): l_0\leq l\leq l_2\}$ ensures that the stable boundary $\alpha$ cannot be situated between the upper boundaries of $V^k_{l_0}$ and $V^k_{l_2}$, as we have specified from the beginning. Therefore the Markov partition $\cR_1$ have the topological obstruction of type $(1)$.

Is not difficult to see that if $\cR_1$ have the topological obstruction of type $(1)$, then the ribbons $r$ and $r'$ given by the definition \ref{Defi: Type 1 obstruction top}, are determined by be the image of a stripe comprised between consecutive sub-rectangles $H^{i_0}_{j_0}$ and $H^{i_0}_{j_0}$ for $r$ and $H^{i_1}_{j_1}$ and $H^{i_1}_{j_1+1}$ for $r'$. The rest of the conditions are totally determined.
\end{proof}

This coding method of representing a ribbon as two consecutive indices $(i,j)$ and $(i,j+1)$ in $\cH$ will be utilized frequently throughout this subsection.

\begin{defi}\label{Defi: Type 2 combinatoric}
	The geometric type $T$ has the type-$(2)$ \emph{combinatorial condition} if and only if there are:
	\begin{itemize}
\item Four different pairs $(k_1,l^1),(k_1,l^2),(k_2,l_1),(k_2,l_2)\in \cV(T)$ with $l_1<l_2$ and $l^1<l^2$.
\item Pairs $(i_1,j_1),(i_1,j_1+1),(i_2,j_2),(i_2,j_2+1)\in \cH(T)$ 
	\end{itemize}
Such that:
	$$
	\rho_T(i_1,j_1)=(k_1,l^1) \text{ and } \rho_T(i_1,j_1+1)=(k_2,l_2).
	$$
	
Additionally, depending on the signs of $\epsilon_T(i_1,j_1)$ and $\epsilon_T(i_1,j_1+1)$ we have the next obstructions:
	
	\begin{enumerate}
\item If $\epsilon_T(i_1,j_1)=\epsilon_T(i_1,j_1+1)$, then we have two options:
\begin{itemize}
\item[i)] If $\rho_T(i_2,j_2)=(k_1,l^2)$ and $\rho_T(i_2,j_2+1)=(k_2,l_1)$ then: $$\epsilon_T(i_1,j_1)=\epsilon_T(i_2,j_1+1)=\epsilon_T(i_2,j_2)=\epsilon_T(i_2,j_2+1)$$.

\item[ii)] If  $\rho_T(i_2,j_2)= (k_2,l_1)$ and $\rho_T(i_2,j_2+1)=(k_1,l^2)$. Then
$$-\epsilon_T(i_1,j_1)=-\epsilon_T(i_2,j_1+1)=\epsilon_T(i_2,j_2)=\epsilon_T(i_2,j_2+1)$$.
\end{itemize}

\item If $\epsilon_T(i_1,j_1)=-\epsilon_T(i_1,j_1+1)$, then we have two options:

\begin{itemize}
\item[i)] If $\rho_T(i_2,j_2)=(k_1,l^2)$ and $\rho_T(i_2,j_2+1)=(k_2,l_1)$ then:
$$
\epsilon_T(i_2,j_2)=\epsilon_T(i_1,j_1) \text{ and  } \epsilon_T(i_2,j_2+1)=\epsilon_T(i_1,j_1+1).
$$
\item[ii)] If $\rho_T(i_2,j_2)=(k_2,l_1)$ and $\rho_T(i_2,j_2+1)=(k_1,l^2)$ then:
$$
\epsilon_T(i_2,j_2)=-\epsilon_T(i_1,j_1+1) \text{ and  } \epsilon_T(i_2,j_2+1)=-\epsilon_T(i_1,j_1).
$$

\end{itemize}

\end{enumerate}

A geometric type $T$ has the \emph{combinatorial obstruction} of type-$(2)$ if there exists $m\in \NN$ such that $T^m$ satisfies the combinatorial condition of type-$(2)$.
\end{defi}

\begin{figure}[h]
	\centering
	\includegraphics[width=0.9\textwidth]{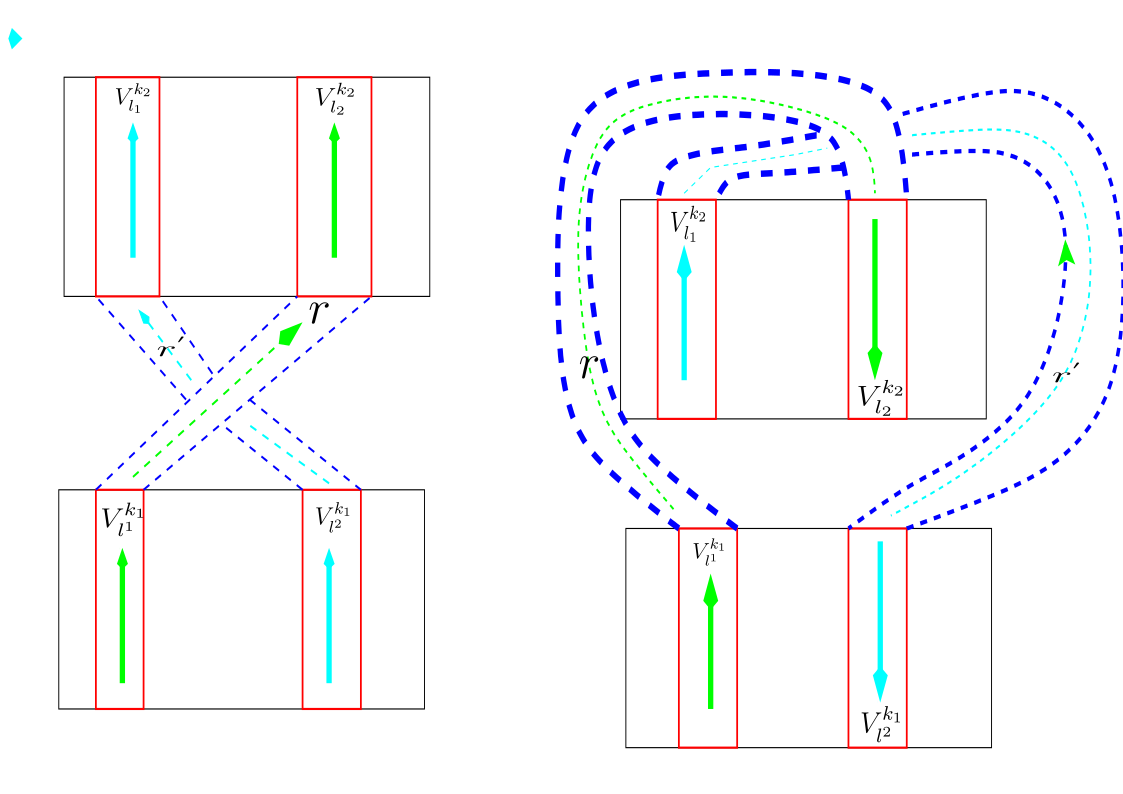}
	\caption{Type $(2)$ combinatorial condition}
	\label{Fig: Type two proof}
\end{figure}

 The Figure \ref{Fig: Type two proof} give some help to understand the proof.
 
\begin{lemm}\label{Lemm: Equiv com and top type 2} 
A geometric type $T$ satisfies the combinatorial condition of type-$(2)$ if and only if the Markov partition $\cR_0$ has the topological obstruction of type-$(2)$ on $\cR_1$.
\end{lemm}

\begin{proof}
The pairs $(k_1,l^1),(k_1,l^2),(k_2,l_1),(k_2,l_2)$ represent the respective vertical sub-rectangles of $\cR$ and the condition $\rho_T(i_1,j_1)=(k_1,l^1) \text{ and } \rho_T(i_1,j^1+1)=(k_2,l_2)$ implies the existence of a rubber $r=f(\underline{H}^{i_1}_{j_1})$ joining a pair of vertical boundaries of the rectangles $V^{k_2}_{l_2}$ and $V^{k_1}_{l^1}$ that correspond to the image  of the upper boundary of $H^{i_1}_{j_1}$ inside $V^{k_1}_{l^1}$ and the image of the inferior boundary of $H^{i_1}_{j_1+1}$ inside $V^{k_2}_{l^2}$.

In the situation when $\epsilon_T(i_1,j_1)=\epsilon_T(i_2,j_1+1)$  we can assume that both are positive (or negative, but the reasoning is the same). 

The condition: $\rho_T(i_2,j_2)=(k_1,l^2)$ and $\rho_T(i_2,j_2+1)=(k_2,l_1)$ implies that the rubber $r'=f(\underline{H}^{i_2}_{j_2})$ joins the stable boundary of $V^{k_2}_{l_2}$ that is the image by $f$ of the inferior boundary of $H^{i_2}_{j_2+1}$ with the stable boundary $V^{k_1}_{l^1}$ that is the image  by $f$ of the upper boundary of $H^{i_2}_{j_2}$. 

Like $\epsilon_T(i_1,j_1)=\epsilon_T(i_2,j_1+1)$ this implies that $r$ joins the upper boundary of $V^{k_1}_{l^1}$ with the inferior boundary of $V^{k_2}_{l_2}$, even more like $$\epsilon_T(i_1,j_1)=\epsilon_T(i_2,j_1+1)=\epsilon_T(i_2,j_2)=\epsilon_T(i_2,j_2+1)=1$$
the $r'$ joins the upper boundary of $V^{k_1}_{l^2}$ (that is the image of the upper boundary of $H^{i_2}_{j_2}$) with the inferior boundary of $V^{k_2}_{l_1}$.
Like $l_1<l_2$ and $l^1<l^2$ the fist combinatorial condition implies that $\cR$ have a topological obstruction of type $(2)$.

Now if  $\rho_T(i_2,j_2)= (k_2,l_1)$ and $\rho_T(i_2,j_2+1)=(k_1,l^2)$, and 
$$\epsilon_T(i_2,j_2)=\epsilon_T(i_2,j_2+1)=-1$$,
The rubber rubber $r'$ have stable boundary in the upper boundary of $V^{k_1}_{l_2}$ as this correspond to the image by $f$ of the inferior boundary of $H^{i_2}_{j_2+1}$ with the  inferior boundary of $V^{k_2}_{l_1}$ that correspond to the image by $f$ of the upper boundary of $H^{i_2}_{j_2}$. Therefore we have a type $(2)$ obstruction.

Consider now that $\epsilon_T(i_1,j_1)=-\epsilon_T(i_1,j_1+1)$. Too fix ideas  $\epsilon_T(i_1,j_1)=1$ and $\epsilon_T(i_1,j_1+1)=-1$. This means that $r$ joins the upper boundary of $V^{k_1}_{l^1}$ (that is the image by $f$ of the upper boundary of $H^{i_0}_{j_0}$)
with the upper boundary of $V^{k_2}_{l_2}$ (that correspond to the image by $f$ of the inferior boundary of $H^{i_1}_{j_2+1}$ ).

 The condition  $\rho_T(i_2,j_2)=(k_1,l^2)$ and $\rho_T(i_2,j_2+1)=(k_2,l_1)$ of the point $(2)$ in the definition together with:
 $$
\epsilon_T(i_2,j_2)=1 \text{ and  } \epsilon_T(i_2,j_2+1)=-1.
$$
implies that $r'$ joints the superior boundary of $V^{k_1}_{l^2}$ that is the image by $f$ of the upper boundary of $H^{i_1}_{j_1}$, with the upper boundary of $V^{k_2}_{l_2}$ that is the image by $f$ of the inferior boundary of $H^{i_1}_{j_1+1}$. Similarly we have the obstruction of type $(2)$ in the Markov partition.

Finally if  $\epsilon_T(i_1,j_1)=1$ and $\epsilon_T(i_1,j_1+1)=-1$ but now $\rho_T(i_2,j_2)=(k_2,l_1), \rho_T(i_2,j_2+1)=(k_1,l^2)$ and 
$$\epsilon_T(i_2,j_2)=1 \text{ and  } \epsilon_T(i_2,j_2+1)=-1.$$

we can deduce that $r'$ joints the upper boundary of $V^{k_1}_{l^2}$ that is the image by $f$ of the inferior boundary of $H^{i_2}_{j_2+1}$ and the upper boundary of $V^{k_2}_{l_1}$ that is the image by $f$ of the upper boundary of $H^{i_2}_{j_2}$. Once again we get a type two obstruction.

If the Markov partition $\cR$ have the type $(2)$ topological obstruction in $\cR_1$the ribbons $r$ and $r'$ determine the horizontal sub-rectangles determined by $(k_1,l^1),(k_1,l^2),(k_2,l_1),(k_2,l_2)$  and we can fix the condition 	$\rho_T(i_1,j_1)=(k_1,l^1)$ and $\rho_T(i_1,j^1+1)=(k_2,l_2)$ to indicate that $r$ joints such rectangles. The condition $(1)$ reflex the case when the two rectangles $R_{k_1}$ and $R_{k_2}$ have cohered orientation and the other case when they have inverse orientation. In any case, we need to remember that $r$ joint the image of the upper boundary of $H^{i_0}_{j_0}$ with the image of the inferior boundary of $H^{i_0}_{j_0}$. The four conditions enunciated are all the possible case when the ribbon $r'$ have stable boundaries in the same stables boundaries components of $\cR$ than $r$.

\end{proof}

\begin{figure}[h]
	\centering
	\includegraphics[width=0.83\textwidth]{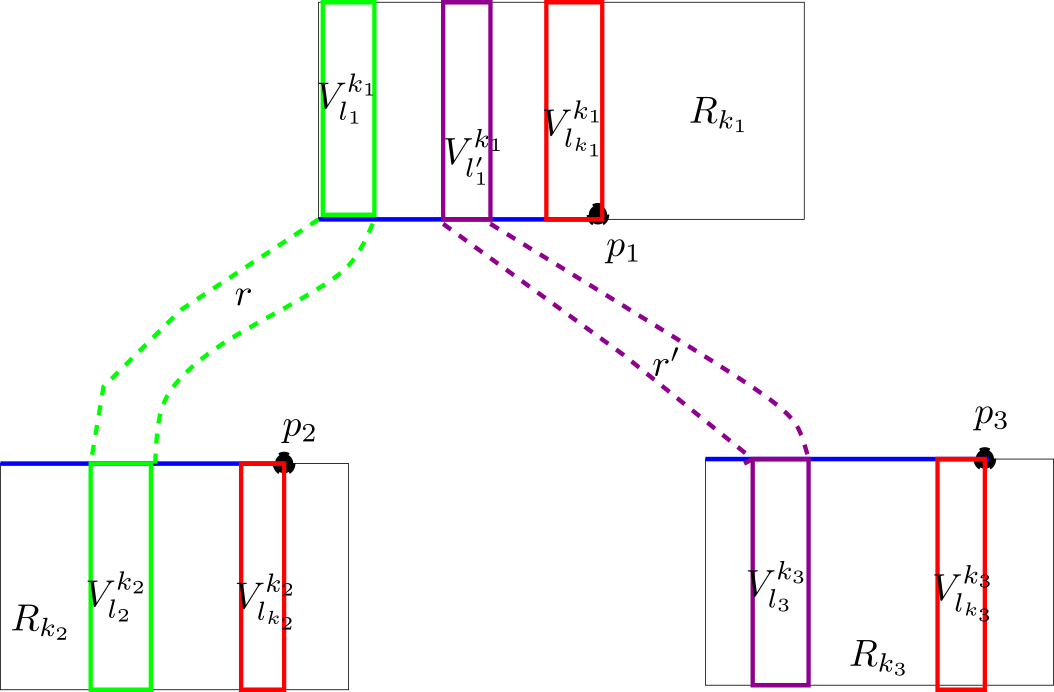}
	\caption{Type $(3)$ combinatorial condition}
	\label{Fig: Type three proof}
\end{figure}

\begin{defi}\label{Defi: Type 3 combinatoric}
A geometric type $T$ satisfies the \emph{combinatorial condition} of type-$(3)$ if there exist indices  described as follows:
\begin{itemize}
\item[i)]  There are indexes $(k_1,j_1)\neq (k_2,j_2)$, and $(k_3,j_3) \in \cH(T)$  were $j_{\sigma}=1$ or $j_{\sigma}=h_{\sigma}$ for $\sigma=k_1,k_2,k_3$, that satisfy:
\begin{eqnarray}\label{Equ: periodic boundary}
\Phi(\sigma,j_{\sigma})=(\sigma,l_{\sigma},\epsilon(\sigma,j_{\sigma}) =1),
\end{eqnarray}

where $l_{\sigma}\in \{1,\cdots,v_{\sigma}\}$. We don't exclude the possibility that $(k_3,j_3)$ is equal to $(k_1,j_1)$ or $(k_2,j_2)$.

\item[ii)] A couple of numbers $l_{1},l_{1}'\in \{1,\cdots, v_{k_1}\}$  such that, either 
$$
l_1<l_1'<l_{k_1} \text{ or } l_{k_1}<l_1<l_1'.
$$

\item[iii)] For  the pair $(k_2,j_2)$, there is a number $l_{2}\in \{1,\cdots, v_{k_2}\}$  with $l_{2}\neq l_{k_2}$.

\item[iv)] There are three possible situations for the pair $(k_3,j_3)$ depending if it is  equal or not to some other pair in item $i)$:

\begin{itemize}
\item[1)] If  $(k_1,j_1)\neq(k_3,j_3) \neq (k_2,j_2) $, there is a number $l_{2} \in \{1,\cdots, v_{k_2}\}$ different form $l_{k_2}$ and a number $l_3\in \{1,\cdots, v_{k_3}\}$ different from $l_{k_3}$.

\item[2)] In case that $(k_2,j_2) = (k_3,j_3)$. There is a numbers $l_3\in \{1,\cdots, v_{k_2}=v_{k_3}\}$, such that: $l_2<l_{k_2}=l_{k_3}<l_3$ or $l_3<l_{k_3}=l_{k_2}<l_2$. 

\item[3)] In case that $(k_1,j_1) = (k_3,j_3)$. There is a number $l_3\in \{1,\cdots, v_{k_1}=v_{k_3}\}$, such that: $l_1<l_1'<l_{k_1}=l_{k_3}<l_3$ or $l_3<l_{k_3}=l_{k_1}<l_1<l_1'$ depending on the situation of item $ii)$. 
\end{itemize}
 
\item[v)] There is a pair of pairs $(i_r,j_r),(i_{r'},j_{r'})\in \cH(T)$ such that $(i_r,j_r+1),(i_{r'},j_{r'}+1)\in \cH(T)$
\end{itemize} 
All of them need to satisfy the next equations:

\begin{eqnarray}
\rho(i_r,j_r)=(k_1,l_1) \text{ and } \rho(i_r,j_r+1)=(k_2,l_2), \text{ or }\\
\rho(i_2,j_2+1)=(k_1,l_1) \text{ and } \rho(i_r,j_r)=(k_2,l_2).
\end{eqnarray}
and at the same time:
\begin{eqnarray}
\rho(i_{r'},j_{r'})=(k_1,l_1') \text{ and } \rho(i_{r'},j_{r'}+1)=(k_3,l_3), \text{ or well }\\
\rho(i_{r'},j_{r'}+1)=(k_1,l_1') \text{ and } \rho(i_{r'},j_{r'})=(k_3,l_3).
\end{eqnarray}
A geometric type $T$ has the \emph{combinatorial obstruction} of type-$(3)$ if there exists an $m\in \NN$ such that $T^m$ satisfies the combinatorial condition of type-$(3)$.
\end{defi}

The distinction between the last combinatorial and topological conditions and obstructions lies in the fact that, for the type $3$ combinatorial condition, we need to consider some power of the geometric type to determine the periodic stable boundaries and their embrionary separatrices. This is why we formulated Lemma \ref{Lemm: Equiv com and top type 3} in terms of combinatorial and topological obstructions, rather than the combinatorial condition of type $3$ itself. In the future (see Lemma \ref{Lemm: T pA class then no condition 3}), we will prove that no geometric type in the pseudo-Anosov class has the combinatorial condition of type $3$. Therefore, its iterations won't have such a condition either. This will be sufficient to establish that any iteration of such a geometric type will don't have the combinatorial condition of type $3$, and the corresponding combinatorial and topological obstructions will not hold for $T$.

\begin{lemm}\label{Lemm: Equiv com and top type 3}
A geometric type have $T$ has the combinatorial obstruction of type $(3)$  if and only if the Markov partition $\cR$ has the topological obstruction of type-$(3)$.
\end{lemm}

\begin{proof}

Assume that $T'$ has the combinatorial obstruction of type $(3)$, meaning there exists an $m \geq 1$ such that for $T := T'^{m}$, and the combinatorial condition of type $(3)$ is satisfied for $T$. We will proceed with our analysis using $T$.

The condition in Item $(i)$ regarding the pairs of indices $(k_1, j_1) \neq (k_2, j_2)$ and  $(k_3, j_3) \in \cH(T)$ implies that the rectangle $R_{\sigma}$ has a fixed stable boundary and a periodic point on it. Let's denote the stable boundaries as $A_{k_1}$, $A_{k_2}$, and $A_{k_3}$. Item $(i)$ indicates that $A_{k_1} \neq A_{k_2}$, but it is possible that $A_{k_3}$ could be equal to one of the other two.

To clarify, let's consider a specific case to fix the ideas. Let's assume that $A_{k_1}$ is the lower boundary of $R_{k_1}$ and $A_{k_2}$ is the upper boundary of $R_{k_2}$. If $A_{k_3}$ is different from the other two stable boundaries, we'll assume that $A_{k_3}$ is the upper boundary of $R_{k_3}$. Note that the other cases are either symmetric or can be entirely determined based on these conventions.

The periodic point $p_1$ that lies on the stable boundary $A_{k_1}$ is contained in the lower boundary of the vertical sub-rectangle $V^{k_1}_{l_{k_1}}$. Then, the conditions $l_1<l_1'<l_{k_1}$ or $l_{k_1}<l_1<l_1'$ imply that the lower boundary of the rectangles $V^{k_1}_{l_1}$ and $V^{k_2}_{l_1'}$ lies on the same embryonic separatrice of $A_{k_1}$. Let's denote this embrionary separatrice as $S_1$. Similarly, the upper boundary of $V^{k_2}_{l_2}$ lies on an embrionary separatrice $S_2$ within $A_{k_2}$. Depending on the situation described in item $iv)$, we can have the following cases:

\begin{itemize}
\item[i)] The periodic point $p_3$ is different from $p_1$ and $p_2$, and the horizontal sub-rectangle $V^{k_3}_{l_3}$ has its upper boundary in an embrionary separatrice $S_{3}$.

\item[ii)] In this case, $p_3=p_2$, but the condition $l_2<l_{k_2}=l_{k_3}<l_3$ or $l_3<l_{k_3}=l_{k_2}<l_2$ implies that the upper boundary of $V^{k_3}_{l_3}$ is in an embrionary  separatrice $S_3$, distinct from $S_2$.

\item[iii)]  In this case, $p_1=p_2$, and the condition $l_1<l_1'<l_{k_1}=l_{k_3}<l_3$ or $l_3<l_{k_3}=l_{k_1}<l_1<l_1'$ implies that the \emph{inferior boundary} of $V^{k_3}_{l_3}$ is in an embryonic separatrice $S_3$ different from $S_1$.
\end{itemize}

The conclusion is that $S_1\neq S_2\neq S_3$. Finally, item $(4)$ along with the conditions about the indexes $(i_r,j_r)$ and $(i_{r'},j_{r'})$ implies that there is a ribbon $r$ from $S_1$ to $S_2$ and another ribbon $r'$ from $S_1$ to $S_3$. Therefore, we have the type $(3)$ topological obstruction in $\cR_0$.

In the converse direction. Imagine that $\cR$ has the topological obstruction of type $(3)$. This means there are three or two different periodic stable boundaries of the Markov partition $A_{k_1}\neq A_{k_2}$ and $A_{k_3}$ with $A_{\sigma}\subset R_{\sigma}$ that contain different embrionary separatrices $S_1,S_2$ and $S_3$. Additionally, there is a ribbon $\underline{r}$ of generation $k$ that joins $S_1$ with $S_2$, and another ribbon $\underline{r}'$ of generation $k'$ that joins $S_1$ with $S_2$.

By taking a certain power of $T' = T^m$ with $m$ a multiple of the period of every stable boundary and greater than $k'$ and $k$, we consider the realization $(\cR,\phi^m)$ of $T$. We can assume that $A_{\sigma}\subset R_{\sigma}$ is a fixed stable boundary for $\phi^m$.
 
We are now in the setting of the combinatorial condition of type $(3)$ for $T$. The fact that the stable boundaries are fixed implies the existence of the indexes $(k_1,j_1)\neq (k_2,j_2)$, and $(k_3,j_3) \in \cH(T')$  in item $(1)$, and clearly, the fixed points are contained in $V^{\sigma}_{l_{\sigma}}$.

The ribbons $\underline{r}$ and $\underline{r}'$, contain ribbons of generation $m$, denoted as $r'$ and $r_2$, respectively. These ribbons joins certain vertical sub-rectangles of the realization of $T$,  $(\cR,\phi^m)$, $V^{k_1}_{l_1}\subset S_1$ with $V^{k_2}_{l_2}\subset S_2$, and $V^{k_1}_{l_1'}\subset S_1$ with $V^{k_3}_{l_3}\subset S_3$. These vertical sub-rectangles produce the indices in items $ii)$, $iii)$, and $iv)$.  The items $2)$ and $3)$ represent the situations where $S_1$ and $S_3$ (or $S_3$ and $S_2$) are in the same stable boundary of a rectangle.

Finally, the ribbons $r$ and $r'$ are determined by two consecutive horizontal sub-rectangles of $(\cR,\phi^m)$. This corresponds to the situation described in item $v)$, and they satisfy the rest of the properties as outlined in the argument.

\end{proof}

\subsubsection{The impasse property.}

Finally we proceed to formulate the impasse condition in terms of the geometric type.

\begin{defi}\label{Defi: Impasse combinatoric}
 Let $T$ be an abstract geometric type of finite genus. Then $T$ has the \emph{impasse property} if there exist $(i,j),(i,j+1)\in \cH(T)$ such that one of the following conditions holds:
 
\begin{eqnarray}
(\rho,\epsilon)(i,j)=(k,l,\epsilon(i,j)) \text{ and } (\rho,\epsilon)(i,j+1)=(k,l+1,-\epsilon(i,j)), \text{ or }\\
(\rho,\epsilon)(i,j)=(k,l+1,\epsilon(i,j)) \text{ and } (\rho,\epsilon)(i,j+1)=(k,l,-\epsilon(i,j)).
\end{eqnarray}
A geometric type $T$ has a \emph{combinatorial impasse} if there exists $m \in \mathbb{N}$ such that $T^m$ has the impasse property.
\end{defi}

Unlike the first three obstructions, where a ribbon is determined by two horizontal sub-rectangles and their respective indices $(i,j),(i,j+1)\in \mathcal{H}$, a topological impasse is defined in terms of a disjoint disk of $K$ and two arcs. Proving the equivalence between the topological and combinatorial formulations just given is a more subtle task, and we address it in the final part of this subsection.

\begin{theo}\label{Theo: Geometric and combinatoric are equivalent}
Let $T=(n,\{(h_i,v_i)\}_{i=1}^n,\Phi)$  be a geometric type of finite genus. Let $f: S \rightarrow S$ be a surface Smale diffeomorphism, and let $K$ be a saddle-type basic piece of $f$ that has a Markov partition of geometric type $T$. The following conditions are equivalent:

\begin{itemize}
\item[i)] The basic piece $K$ has a topological impasse.
\item[ii)] The geometric type $T^{2n+1}$ has the impasse property.
\item[iii)] The geometric type $T$ has a combinatorial impasse.
\end{itemize}
\end{theo}

\begin{proof}
 \textbf{ i) implies ii):} Let  $\cR=\{R_i\}_{i=1}^n$ be a Markov partition of $K$ of geometric type $T$, and let $\overset{o}{D}$ be a topological impasse with $\alpha$ as the $s$-arc and $\beta$ as the $u$-arc, whose union forms the boundary of $D$.

 \begin{lemm}\label{Lemm: impasse disjoin int Markov partition}
The impasse is disjoint from the interior of the Markov partition, i.e.
$$
\overset{o}{D}\cap \cup_{i=1}^n \overset{o}{R_i}=\emptyset.
$$
Moreover, the intersection $\{k_1,k_2\}=\alpha\cap \beta$ is not contained in the interior of the Markov partition.
 \end{lemm}
 
\begin{proof}
Suppose $\overset{o}{D}\cap \overset{o}{R_i}\neq \emptyset$. Let $\{k_1,k_2\}:=\alpha\cap \beta$ be the endpoints of the $s,u$-arcs  on the boundary of $D$. Let us take a point $x$ in $\overset{o}{D}\cap \overset{o}{R_i}$, since $x$ is not a hyperbolic point and have a open neighborhood $U\subset \overset{o}{R_i}$ disjoint of $K$, there must exist a sub-rectangle  $Q\subset \overset{o}{R_i}$  containing $x$ in its interior, whose interior is disjoint from $K$, whose horizontal boundary consists of two disjoint intervals $I,I'$ of $W^s(K)$ and whose vertical boundary consists of two disjoint intervals $J, J'$´ of $W^u(K)$ (See Figure \ref{Fig: Rec Q} ). 
\begin{figure}[h]
	\centering
	\includegraphics[width=0.6\textwidth]{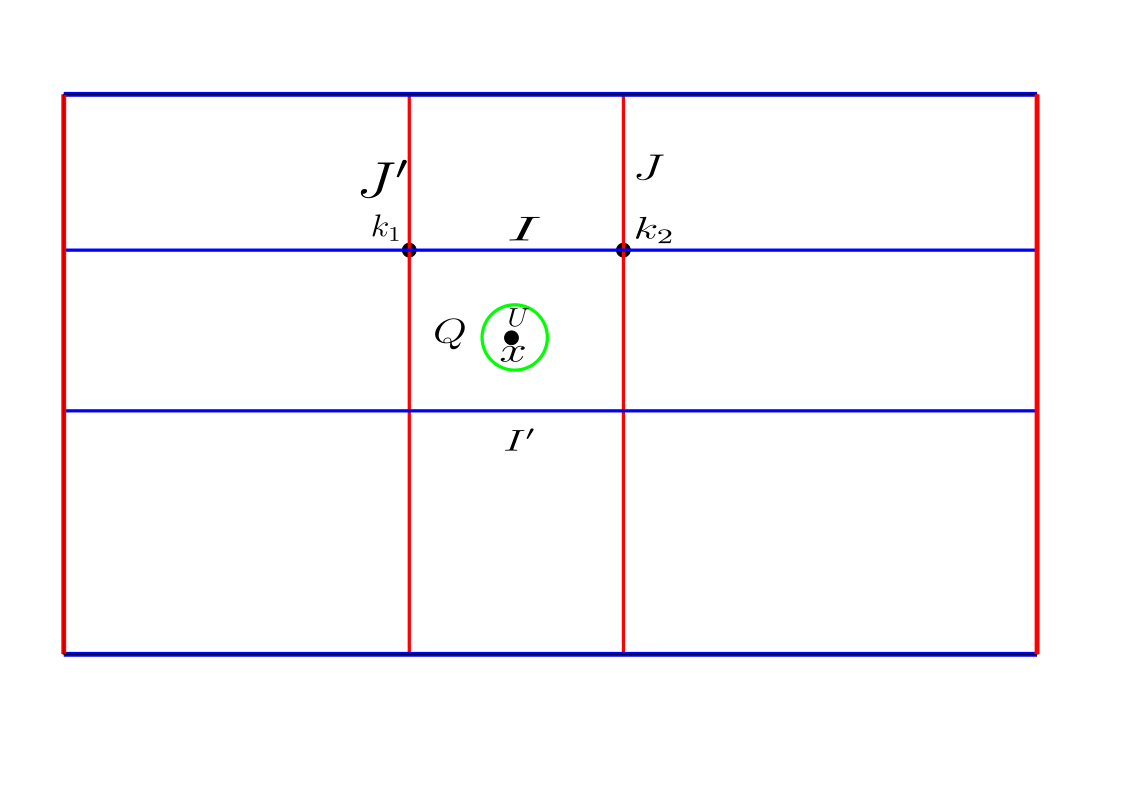}
	\caption{The rectangle $Q$}
	\label{Fig: Rec Q}
\end{figure}

Even more, as $\overset{o}{D}$ is disjoint from $K$ and  $k_1\in K$  is in the closure of $\overset{o}{D}$, the point $k_1$ is a corner point of the rectangle $Q$ and hence $k_1$ is in $\overset{o}{R_i}$.

By similar arguments, we can conclude that $k_2$ is in $\overset{o}{R_i}$. However, this leads to a contradiction, since there would be a horizontal segment $I$ of $R_i$ containing $\alpha$ and a vertical segment $J$ of $R_i$ containing $\beta$ which intersect at two points in the interior of $R_i$ (namely $k_1$ and $k_2$). This is not possible because $R_i$ is an embedded rectangle. Therefore, $\overset{o}{D}$ is disjoint from the interior of $R_i$.

It is evident that if $k_1$ or $k_2$ is in the interior of $R_i$, then $\overset{o}{D}\cap \overset{o}{R_i}\neq \emptyset$, and we appeal to the previous argument to reach a contradiction.

\end{proof}

If $\overset{o}{D}$ is an impasse, then $f^z(\overset{o}{D})$ is also an impasse for all integers $z \in \ZZ$. In the following lemma, we will construct a "well-suited" impasse that will make it easier for us to prove the first implication.

\begin{lemm}\label{Lemm: well suit impasse}
If $K$ has an impasse, then there exists another impasse $D'$ of $K$ such that the $s$-arc $\alpha'$ and the $u$-arc $\beta'$ on the boundary of $D'$ satisfy the following conditions:

\begin{itemize}
\item[i)]  $\beta'$ is in the interior of the unstable boundary of the Markov partition,  $\beta' \subset \overset{o}{\partial^u  \cR}$,
\item[ii)] $f(\beta')$ is not a subset of the unstable boundary of the Markov partition,  $\partial^u\overset{o}{\cR}= \overset{o}{\partial^u  \cR}$.
\item[iii)] $f^{2n+1}(\alpha)\subset \overset{o}{\partial^s \cR}$, where $n$ is the number of rectangles in the Markov partition.
\end{itemize}
\end{lemm}
 
 \begin{proof}
Let $D$ be an impasse of $K$ with $\alpha$ and $\beta$ as its boundary arcs. Let $\{k_1,k_2\}:=\alpha \cap \beta \subset K$ be the extreme points of these arcs. If $k_1$ is a periodic point, then there exists a positive iteration of $k_2$ that intersects $\alpha$ in its interior. However, this is not possible since the interior of $\alpha$ is disjoint from $K$. Therefore, we deduce that $k_1$ and $k_2$ are on the same stable separatrice of a periodic $s$-boundary point, and they are not periodic points themselves. Similarly, $k_1$ and $k_2$ are on the same unstable separatrice of a $u$-boundary point, but they are not periodic.

Since the negative orbit of $k_1$ converges to a periodic point $p_1$ on the stable boundary of the Markov partition, the negative orbit of $k_1$ approaches $p_1$ through points contained in the boundary of the Markov partition. Therefore, there exists $m \in \NN$ such that $f^{-m}(k_1)\in \overset{o}{\partial^u \cR}$.  We claim that if $f^{-m}(k_1)\in \overset{o}{\partial^u_{\epsilon}R_i}$, then $f^{-m}(k_2)$ is in the interior of the same unstable boundary component as $f^{-m}(k_1)$, i.e., $f^{-m}(k_2)\in \overset{o}{\partial^u_{\epsilon}R_i}$.

Indeed, since $f^{-m}(k_1)$ and $f^{-m}(k_2)$ are the boundary points of the $u$-arc $f^{-m}(\beta)$, there are no points of $K$ between them. The point $f^{-m}(k_1)$ lies in the interior of $\partial^u_{\epsilon} R_i$, which means that $f^{-m}(\beta)$ intersects the interior of this stable boundary component. This implies that $f^{-m}(k_2)$ is either an extreme point of $\partial^u_{\epsilon} R_i$ or lies in its interior, i.e., $f^{-m}(k_2) \in \overset{o}{\partial^u_{\epsilon}R_i}$.

If $f^{-m}(k_2)$ were an extreme point of $\partial^u_{\epsilon} R_i$, such an extreme point would not be surrounded by elements of $K$ on either side. This is because on one side we have the $u$-arc $f^{-m}(\beta)$, and on the other side it is on the unstable boundary of a rectangle in the Markov partition. However, this is not possible since $K$ has no double $s$-boundary points. Therefore, we conclude that $f^{-m}(k_2) \in \overset{o}{\partial^u_{\epsilon}R_i}$, as we claimed.

Now let us consider the impasse $D' = f^{-m}(D)$ instead of $D$. We will still use $\alpha$ and $\beta$ to denote the $s$-arc and $u$-arc on the boundary of $D$, but we assume that $\beta \subset \overset{o}{\partial^u \cR}$. Due to the uniform expansion along the unstable leaves, we know that there exists $m \in \NN$ such that $f^m(\beta) \subset \overset{o}{R_i}$, but $f^{m+1}(\beta)$  is no longer contain in  the unstable boundary of the Markov partition. Let us define $D' = f^{m}(D)$ with boundary arcs $\alpha'$ and $\beta'$. This impasse satisfies items $i)$ and $ii)$ of our lemma, since $\beta'$ is contained in the interior of $\overset{o}{\partial^u \cR}$ and $f(\beta')$ is not a subset of $\partial^u \overset{o}{\cR} = \overset{o}{\partial^u \cR}$.

In view of Lemma \ref{Lemm: impasse disjoin int Markov partition}, the points $f(k_1)$ and $f(k_2)$ are not in the interior of the Markov partition. We conclude that $f(k_1)$ and $f(k_2)$ belong to $\partial^s \cR$ because both are on the boundary of $\cR$ but not inside the unstable boundary. By the invariance of the stable boundary of $\cR$ under positive iterations of $f$ and the Pigeonhole principle applied to the $2n$ components of stable boundaries of the Markov partition, we know that $f^{2n}(f(k_1))$ and $f^{2n}(f(k_2))$ lie on periodic stable boundaries of the Markov partition. Moreover, $f^{2n}(f(k_1))$ and $f^{2n}(f(k_2))$ lie on the same stable separatrice, and on each separatrice, there is a single stable boundary component of $\cR$ that is periodic (even though they are all pre-periodic). This implies that $f^{2n}(f(k_1))$ and $f^{2n}(f(k_2))$ are in $\overset{o}{\partial^s_{\epsilon}R_j}$ for a single stable boundary component of the rectangle $R_j$. This proves item $iii)$ of our lemma.

\end{proof}

With the simplification of  Lemma \ref{Lemm: well suit impasse} we deduce that $\beta\subset \partial^u R_i$ and $f^{2n+1}(\alpha)\subset \partial^s R_k$. 

Let us consider the partition $\cR$ viewed as a Markov partition of $f^{2n+1}$. We claim that $\beta$ is an arc joining two consecutive sub-rectangles of $(\cR,f^{2n+1})$, denoted as $H^i_j$ and $H^{i}_{j+1}$. In effect, $\beta$ is a $u$-arc joining two consecutive rectangles or is properly contained in a single sub-rectangle $H$ of $(\cR,f^{2n+1})$. This is because the stable boundaries of $H$ are not isolated from $K$ within $H$, as would be the case if the ends of $\beta$ were on the stable boundary of $H$. Therefore, the only possibility is that $\beta$ is properly contained in a rectangle $H$. In this case, $f^{2n+1}(H)$ is a vertical sub-rectangle of $R_i$ containing $f^{2n+1}(\beta)$ as a proper interval. However, the endpoints of $f^{2n+1}(\alpha)$, which coincide with the endpoints of $f^{2n+1}(\beta)$, are not in the interior of the stable boundary of $R_k$, which contradicts the hypothesis. Thus, we conclude that $\beta$ joins two consecutive sub-rectangles.

Let $H^i_{j}$ and $H^i_{j+1}$ be the consecutive rectangles joined by $\beta$. Suppose that $f^{2n+1}(H^i_j)=V^k_{l}$ with the change of vertical orientation encoded by $\epsilon_{T^{2n+1}}(i,j)$ and $f^{2n+1}(H^i_{j+1})=V^k_{l'}$ with the change in the vertical orientation encoded by $\epsilon_{T^{2n+1}}(i,j+1)$. Between $V^k_{l}$ and $V^k_{l'}$ there are no points of $K$ because the $s$-arc $f^{2n+1}(\alpha)$ joins them in the horizontal direction, this implies that $l'\in \{l+1,l-1\}$. The vertical orientation of $\beta$ is the same as $H^i_j$ and $H^i_{j+1}$, furthermore $f^{2n+1}(\beta)$ joins the horizontal sides of the rectangles $V^k_{l}$ and $V^k_{l'}$ which are on the same stable boundary component of $R_k$,  this implies that the vertical orientation of $f(H^i_j)$ is the inverse of the orientation on $f(H^i_{j+1})$ within $R_k$ (to visualize this it suffices to follow the segment $f^{2n+1}(\beta)$ with a fixed orientation), thus $\epsilon_{T^{2n+1}}(i,j)=-\epsilon_{2n+1}(i,j')$.

The geometric type of the Markov partition $\cR$ for the map $f^{2n+1}$ is $T^{2n+1}$. Our construction implies that $T^{2n+1}$ has the impasse property. With this, we conclude the first implication of the theorem..

 \textbf{ ii) implies iii):} The geometric type $T^{2n+1}$ has the impasse property, which implies that $T$ has a combinatorial impasse.

\textbf{ iii) implies i):} Suppose $T^m$ has the impasse property for $m=2n+1$. Let $H^i_j$ and $H^i_{j+1}$ be consecutive horizontal sub-rectangles in the Markov partition $\cR$ for the map $f^{m}$, as given by the impasse property. This condition implies that if $f^{m}(H^i_j)=V^k_{l}$ and $f^{m}(H^i_{j+1})=V^k_{l'}$, then $l$ and $l'$ are consecutive indexes and have inverse vertical orientations.

The stable segments of  the stable boundary of $R_k$ between $V^k_{l}$ and $V^k_{l'}$ do not contain elements of the maximal invariant set of $f^m\vert_{\cR }$, since they are consecutive sub-rectangles of the partition $\cR$ seen as a Markov partition of $f^m$. In fact they do not contain elements of $K$ since $K=f^m(K)$ and therefore those segments are $s$-arcs of $f$. Furthermore, the image by $f^{m}$ of the two arcs $u$ on the boundary  of $R_i$ that join $H^i_j$ and $H^i_{j+1}$ are $u$-arcs of $f$. Between the two pairs of arcs there is an $s$-arc $\alpha$ and a $u$-arc $\beta$ such that they intersect each other only at their ends. We claim that $\alpha$ and $\beta$ bound a disk $D$ whose  interior is disjoint from $K$.

Suppose that $\gamma:=\alpha \cup \beta$ does not bound a disk. This means that $\gamma$ is not homotopically trivial, and neither are its iterations $\{f^m(\gamma)\}_{m=1}^{\infty}$. In fact, since the stable segments of such curves lie in the stable manifold of a periodic point, there exists a $k\in \NN$ such that $\{f^{mk}(\gamma)\}_{m=1}^{\infty}$ is a set of disjoint curves. We will argue that two of these curves are not homotopic, leading to a contradiction with the fact that the surface $S$ where the basic piece is contained has finite genus.

\begin{lemm}\label{Lemma: gamma not homotopic imagen}
The curve $\gamma$ is not homotopic to any other curve $f^{2mk}(\gamma)$ for $k\in \mathbb{N}$.
\end{lemm}

\begin{figure}[h]
	\centering
	\includegraphics[width=0.6\textwidth]{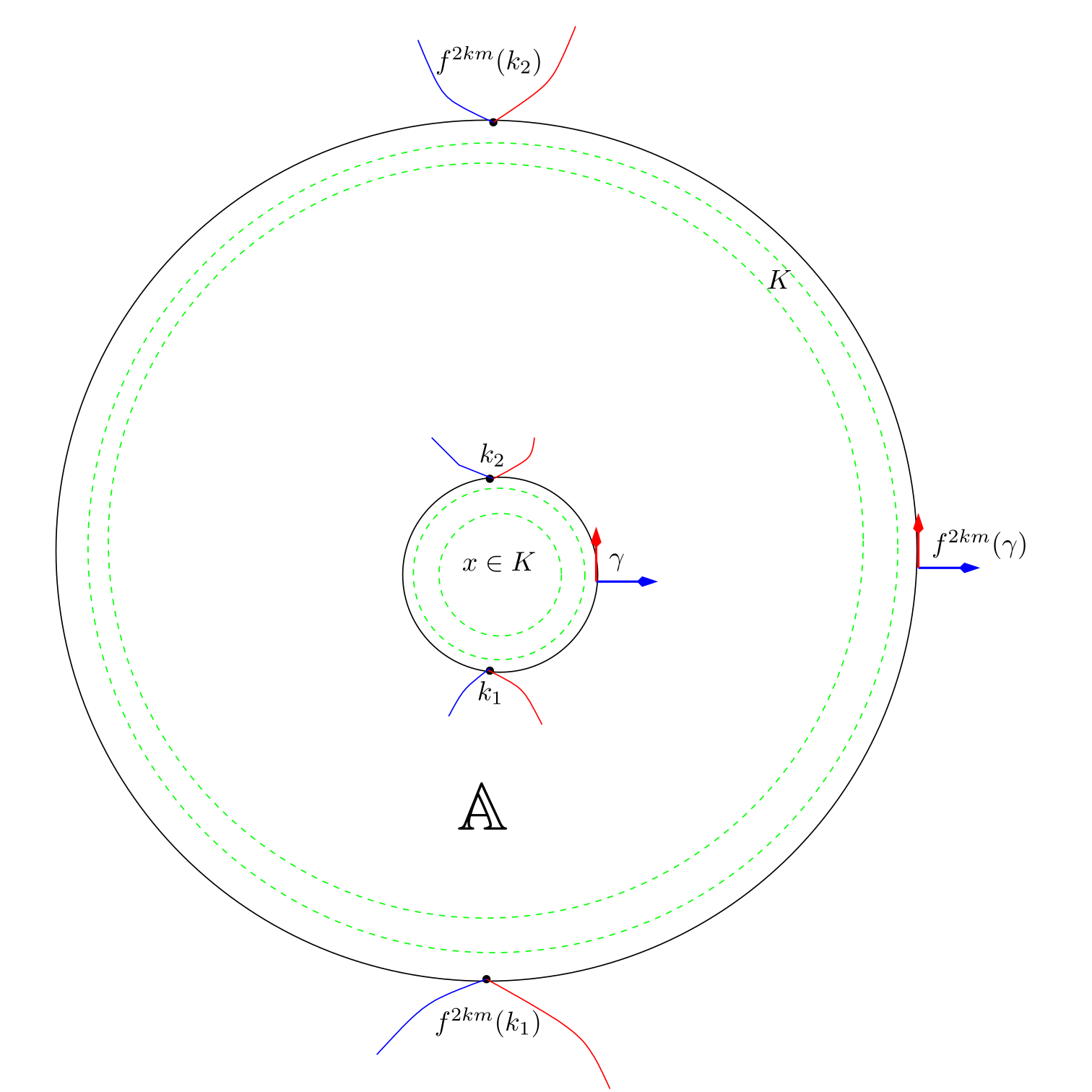}
	\caption{ The Annulus}
	\label{Fig: The annulus}
\end{figure}

\begin{proof}
Suppose $\gamma$ and $f^{m}(\gamma)$ are homotopic, which means they bound a topological annulus $\AA$ in $S$. 

Then we can take small enough sub-annulus $\AA'$ of $\AA$, which has $\gamma$ as one of its boundary components and  is send to another annulus $f^m(\AA')$ that have $f^m(\gamma)$  as a boundary component. 

 We have the following situations:
\begin{itemize}
\item[i)]  $f^m(\overset{o}{\AA'}) \cap \overset{o}{R}\neq \emptyset$ or
\item[ii)] $f^m(\overset{o}{\AA'}) \cap \overset{o}{R} = \emptyset$
\end{itemize}

Lets to understand the fist situation. Take $x\in \gamma$ and a small interval $J$ that have $x$ as and end point and pointing toward the interior of $\AA$, take a interval $I\subset \gamma$ that have $x$ as and end point and give to $I$ an orientation such that $J$ and $I$ generate a frame that is coherent with the orientation of $S$. While $f^m$ preserve the orientation of $S$ it could change the orientation of $I$ and $J$ at the same time, in such a manner that $f^m(J)$ is and interval pointing towards the interior of $\AA$, this is the mechanism behind the situation in item $i)$. Anyways by taking the annulus $\AA$ bounded by $\gamma$ and $f^{2m}(\gamma)$, we can take $\AA'\subset \AA$ small enough such that:
$$f^{2m}(\overset{o}{\AA'}) \cap \overset{o}{R} = \emptyset$$
Therefore, we can assume that $\gamma$ and $f^{2km}(\gamma)$ satisfy the property enunciated in item $ii)$.

Suppose that $\gamma$ is saturated by the hyperbolic set $K$ in the interior direction of $\AA$, and therefore, it is isolated from $K$ in the exterior direction of $\AA$. Then $f^{2km}(\gamma)$ is isolated from $K$ in the interior direction of $\AA$ and saturated by $K$ in the exterior direction of $\AA$. In particular, there exists $x\in K \cap (S \setminus \AA)$.
 
Since ${k_1, k_2}$ and ${f^m(k_1), f^{2km}(k_2)}$ are not periodic points, they do not have free separatrices. Moreover, $k_1$ and $k_2$ lie on the same separatrice. Without loss of generality, let's assume that $k_1$ has a stable separatrice $I$ that does not connect it to a periodic point. The separatrice $I$ cannot remain entirely inside $\AA$ because there exists a point $x \in K$ outside of $\AA$. Non-free separatrices of points $y \in K$ that do not connect $y$ to a periodic point are dense in $K$, so $I$ needs to approach $x$ arbitrarily closely. The boundary of $\AA$ is the union of four arcs, which means that the separatrice $I$ cannot intersect them in their interiors. Therefore, the possible scenarios are as follows:

 \begin{itemize}
\item The intersection point is $k_2$, which is impossible since it would imply that $I$ is inside a closed leaf of $W^s(K)$,
\item The other point is $f^{2km}(k_1)$ or $f^m(k_2)$. This is not possible because in that case, $I$ would need to be contained in the separatrice of $f^{2km}(k_1)$ (or $f^{2km}(k_2)$), which is not the same as the $s$-arc in $f^{2km}(\gamma)$. But this separatrice points towards the exterior of $\AA$
 \end{itemize}
 
So, $I$ couldn't intersect either $f^{2k¨m}(\gamma)$ or $\gamma$. We conclude that $I$ is a subset of $\AA$, which is a contradiction. This ends our proof.
 
\end{proof}

Then, if $\gamma$ is not null-homotopic, we could find an infinite amount of disjoint and non-isotopic curves (See \cite[Lemma 3.2]{juvan1996systems}), which is a contradiction with the finite genus of the surface $S$. It follows that $\gamma$ bounds a disk.

\begin{coro}\label{key}
	
	The disk bounded by $\gamma$ doesn't intersect $K$.
\end{coro}

\begin{proof}
If there exists $x \in K \cap \overset{o}{D}$, then $x$ has at least one non-free stable separatrix $I$. Such a separatrix is not confined within $D$ because there exist points in $K$ that are not in $D$ (they are those on the other side of the $s$-arc $\alpha$ on the boundary of $D)$. Additionally, $I$ is dense in $K$. Therefore, $I$ must intersect the $u$-arc $\beta$ on the boundary of $D$. However, the separatrice $I$ does not intersect the extreme points $\{k_1,k_2\}=\alpha \cap \beta$ because no local separatrices of them intersect the interior of $D$. This leads to a contradiction. Hence, $K \cap \overset{o}{D} = \emptyset$.
\end{proof}

This implies that $D$ is an impasse for $f$.

\end{proof}

\subsection{Algorithmic Determination of Finite Genus and Impasse }\label{Sub-sec: algoritm finite genus}

The objective is to prove the following proposition that will serve as the key element in determining, algorithmically and in finite time, whether a geometric type belongs to the pseudo-Anosov class or not. In the future, we will derive the property of double boundaries from the fact that the incidence matrix of a geometric type in the pseudo-Anosov class is mixing. Therefore, we formulate this property as the first one to be verified.

\begin{prop}\label{Prop: mixing+genus+impase is algorithm}
	Given any abstract geometric type $T=(n,\{(h_i,v_i)\}_{i=1}^n,\Phi)$, there exists a finite algorithm that can determine whether $T$ satisfies the following properties:
	\begin{enumerate}
		\item The incidence matrix of $T$, $A(T)$ is mixing. 
		\item The genus of $T$ is finite.
		\item  $T$ exhibits an impasse.
	\end{enumerate}
	Furthermore, the number of iterations of $T$ required by the algorithm to determine these properties is upper bounded by $6n$.
\end{prop}

\begin{proof}
	
	We have formulated specific conditions on the combinatorics of $T$ to determine whether $T^n$ has an impasse or one of the three obstructions to have finite genus. Furthermore, according to Theorem \ref{Theo: finite type iff non-obtruction}, determining whether or not $T$ has finite genus requires computing at most $T^{6n}$. To determine whether $T$ has an impasse, Theorem \ref{Theo: Geometric and combinatoric are equivalent} states that it is necessary to compute at most $T^{2n+1}$. Since $n\geq1$, it is clear that $2n+1<6n$, which provides the upper bound we claim in the proposition. Proposition \ref{Prop: algoritm iterations type} asserts that the computation of $T^{6n}$ is algorithmic, thus confirming the existence of an algorithm to determine the finite genus and the presence of an impasse for $T$.

	We have left the mixing criterion for last. The following lemma is the basis for the upper bound on the number of iterations to verify the mixing of the incidence matrix.
	
	\begin{prop}\label{Prop: bound positive incidence matriz }
	A matrix $A$  of size $n \times n$ with non-negative coefficients  is  mixing if and only if  every coefficient in $A^n$ is positive .
	\end{prop}
	
	\begin{proof}	
		If $A^n$ has only positive coefficients, since $A$ don't have negative coefficients, the matrix  $A^{n+1}$ has only positive coefficients and $A$ is mixing.
			
		Consider the directed graph $\mathcal{G}$ with $n$ vertices corresponding to the matrix $A$. In this graph, there is an edge pointing from vertex $v_i$ to vertex $v_j$ if and only if $a_{ij} > 0$ in the matrix $A$. Since $A$ is mixing, it means that there is always a path from any vertex $v_i$ to any other vertex $v_j$ in the graph $\mathcal{G}$.
		
		To prove our Lemma, it is sufficient to demonstrate that there exists a path of length less than or equal to $n$ between any two vertices in the graph.

		Consider a path $\gamma = [v_{i(1)}, \cdots, v_{i(m)}]$ from $v_i = v_{i(1)}$ to $v_j = v_{i(m)}$ in the graph $\mathcal{G}$. Let $\gamma$ have $m$ vertices and be of minimum length among all paths connecting these vertices.
		
		If any two vertices of the path $\gamma$ are equal, it implies the existence of a cycle within $\gamma$. We can eliminate this cycle and obtain a new path $\gamma'$ from $v_i$ to $v_j$ with a shorter length than $\gamma$. This contradicts the assumption that $\gamma$ has the minimum length. Hence, all vertices in $\gamma$ must be distinct.
		
		Therefore, the length of $\gamma$ is necessarily less than or equal to $n$, since there are only $n$ distinct vertices in the graph.
	\end{proof}

	If $T$ is in the pseudo-Anosov class, we have $A^n(T) = A(T^n)$. According to the previous Lemma, if $A(T)$ is a non-negative matrix, then $A(T^n)$ is positive definite. This implies that $A(T^n)$ is mixing and does not have double boundaries. Therefore, for a given geometric type $T$ with $A(T)$ being a non-negative matrix, the positivity of $A(T^n)$ is equivalent to its mixing property, which in turn implies the absence of double boundaries.
	
	\begin{coro}\label{Coro: algoritmic mixing}
		Given a geometric type $T$, if $A(T^n)$ is not positive definite, then $T$ is not in the pseudo-Anosov class.
	\end{coro}
	
	Clearly, we only need to iterate $T$ at most $n$ times, and the upper bound given in the proposition remains valid. 
	This concludes our proof.
\end{proof}

\chapter{Primitive Markov partition and canonical types.}\label{Chapter: Markov partitions}

\section{The construction of Markov partitions.}\label{Section: Existence}

It is a classical result that pseudo-Anosov homeomorphisms admit Markov partitions (see \cite[Proposition 10.17]{fathi2021thurston}). However, our goal is to provide a recipe for constructing Markov partitions of generalized pseudo-Anosov homeomorphisms, which subsequently allows us to generate a distinguished family of adapted Markov partitions. From this particular family, we can extract a canonical and finite set of geometric types that provide a partial solution to Item $III)$ of Problem \ref{Prob: Clasification}. Throughout our exposition, we will use  the characterization of Markov partitions described in Proposition \ref{Prop: Markov criterion boundary}.

Let $S$ be a closed and oriented surface of genus $g$. We fix a generalized pseudo-Anosov homeomorphism $f: S \rightarrow S$ whose measured foliations are $(\mathcal{F}^s, \mu^s)$ and $(\mathcal{F}^u, \mu)$, and \textbf{Sing}$(f)$ is the set of singularities.

\subsection{$s,u$-adapted graphs.} Our construction of the Markov partition begins by creating an $f$-invariant graph of non-trivial compact intervals contained in stable leaves. We formalize the object to be constructed in the following definition.

\begin{defi}\label{Defi:Adapted graph}
  Let $\delta^s=\{I_i\}_{i=1}^n$ be a family of compact stable intervals not reduced to a point. The family $\delta^s$ is an $s$-\emph{adapted graph} to $f$ if:
 
	\begin{itemize}
		\item[i)] $\cup \delta^s:=\cup_{i=1}^n I_i$ is $f$-invariant.
		\item[ii)] For all $i\in \{1,\cdots,n\}$, $I_i \subset \cF^s(\textbf{Sing}(f))$.
		\item[iii)] For all $x\in \textbf{Sing}(f)$ and every stable separatrice of $\cF^s(x)$, there exists an interval $I_i$ contained in that separatrice with an endpoint equal to  $x$.
		\item[iv)] The endpoints of each interval  $I_i$ belong to $\cF^u(\textbf{Sing}(f))$.
	\end{itemize}
	The definition of a $u$-\emph{adapted graph} to $f$ is completely symmetric.
\end{defi}

This definition is not surprising since the boundary of an adapted Markov partition (Definition \ref{Defi: adapted partition}) consists of $s$ and $u$-adapted graphs, as discussed in the following example.

\begin{exem}\label{Exem: Markov boundary is adapted}
Let $f: S \rightarrow S$ be a generalized pseudo-Anosov homeomorphism and $\mathcal{R}$ be an adapted Markov partition of $f$. Then the family of (unstable) stable boundary components of $\mathcal{R}$, denoted by $\partial^s \cR = \cup_{i=1}^n \partial^s_{\pm} R_i$ ($\partial^u \cR=  \cup_{i=1}^n \partial^u_{\pm} R_i$), is a ($u$) $s$-adapted graph to $f$.

Indeed, $\partial^s \mathcal{R}$ is $f$-invariant. The only periodic points on the boundary are singularities of $f$. Furthermore, the singular points of $f$ are at the corners of a rectangle, and at each separatrice of these corners, we have a stable boundary component of a rectangle from $\mathcal{R}$. Finally, the extreme points of the stable boundary are on the unstable leaf of a singularity because, as we know, singularities are the only periodic points on the unstable boundary of an adapted Markov partition.
	
\end{exem}

 Now we will show how to construct adapted graphs starting with a family of $f^{-1}$-invariant unstable segments.

\begin{lemm}\label{Lemm: the generated graph is adapted}
	 Let $\mathcal{J} = \{J_j\}_{j=1}^n$ be a family of nontrivial compact segments, each contained in $\mathcal{F}^u(\textbf{Sing}(f))$, such that $\cup \mathcal{J} = \cup_{j=1}^n J_i$ is $f^{-1}$-invariant. Let $\delta^s(\mathcal{J}) = \{I_i\}_{i=1}^m$ be the family of all stable segments with one endpoint in \textbf{Sing}$(f)$, the other endpoint in $\cup \mathcal{J}$, and with interior disjoint from $\cup \mathcal{J}$.
	 
The family $\delta^s(\mathcal{J})$ is an $s$-adapted graph to $f$. The set $\delta^s(\mathcal{J})$ is the $s$-\emph{adapted graph generated by} $\mathcal{J}$.
\end{lemm}

\begin{proof}

By construction, each interval in the graph $\delta^s(\cJ)$ is contained in $\mathcal{F}^s(\textbf{Sing}(f))$, has one endpoint in $\textbf{Sing}(f)$ and the other in $\mathcal{J} \subset \mathcal{F}^u(\textbf{Sing}(f))$, and for every stable separatrice of any point in $\textbf{Sing}(f)$, there exists a segment in $\delta^s(\mathcal{J})$ that belongs to that separatrice. It remains to show that $\cup \delta^s(\mathcal{J})$ is $f$-invariant.

Let $I:=I_i$ be an interval in $\delta^s(\mathcal{J})$. $I$ is contained in a stable separatrice of a point $p \in \textbf{Sing}(f)$, therefore $f(I)$ is contained in a stable separatrice of $f(p)$, and within that separatrice, there exists an interval $I'$ of $\delta^s(\mathcal{J})$. If $f(I)$ is not a subset of $I'$, then $I' \subset f(I)$, which implies that $I'$ has an endpoint $x'$ in $\cup \mathcal{J}$ that belongs to the interior of $f(I)$. Since $\cup \mathcal{J}$ is $f^{-1}$-invariant, then $f^{-1}(x') \in \cup \mathcal{J}$ is in the interior of $I$, which contradicts our definition of $I$ because the interior of $I$ cannot intersect $\cup \mathcal{J}$. Consequently, $f(I) \subset I' \subset \cup \mathcal{J}$, and thus $\cup \delta^s(\mathcal{J})$ is $f$-invariant.

\end{proof}

If $\mathcal{I}$ is a family of stable segments contained in $\mathcal{F}^s(\textbf{Sing}(f))$ and $f$-invariant, we define the $u$-\emph{adapted graph generated by} $\mathcal{I}$, denoted by $\delta^u(\mathcal{I})$, in a symmetric manner.

\subsection{Rails and rectangles.}
In an $s$-adapted graph, we can classify the points into two distinct classes based on their relative position within the graph. This classification enables us to define an equivalence relation on the surface without the graph, such that the closure of the equivalence classes corresponds to rectangles.

\begin{defi}\label{Defi: regular points}
Let  $\delta^s=\{I_i\}_{i=1}^n$ be an $s$-adapted graph of $f$. The regular part of $\delta^s$ is defined as:
	
	$$
	\overset{o}{\delta^s}=\delta^s \setminus \{\text{ endpoints of } I_i , \, i=1,\dots n \}.
	$$
	Similarly, we define the regular part of a $u$-adapted graph by removing its endpoints.
\end{defi}

Now we define two types of unstable segments depending on the location of their endpoints.

\begin{defi}\label{Defi: Rail}
		Let $\delta^s$ be an $s$-adapted graph. A \emph{$u$-regular rail} for $\delta^s$ is any unstable segment $J$ whose interior is disjoint from $\delta^s$ but has endpoints in the regular part of the $s$-adapted graph, $\overset{o}{\delta^s}$.
				
		A \emph{$u$-extremal rail} is any unstable segment $J$ that is not a regular rail for $\delta^s$, for which there exists an \emph{embedded rectangle} $R$ with stable boundary $\partial^s R$ contained in $\delta^s$, with $J$ as its left or right boundary, and such that any other vertical segment of $R \setminus J$ is a $u$-regular rail for $\delta^s$. The set of extremal rails of $\delta^s$ is denoted by Ex$^u(\delta^s)$.
		
\end{defi}

\begin{figure}[h]
	\centering
	\includegraphics[width=0.5\textwidth]{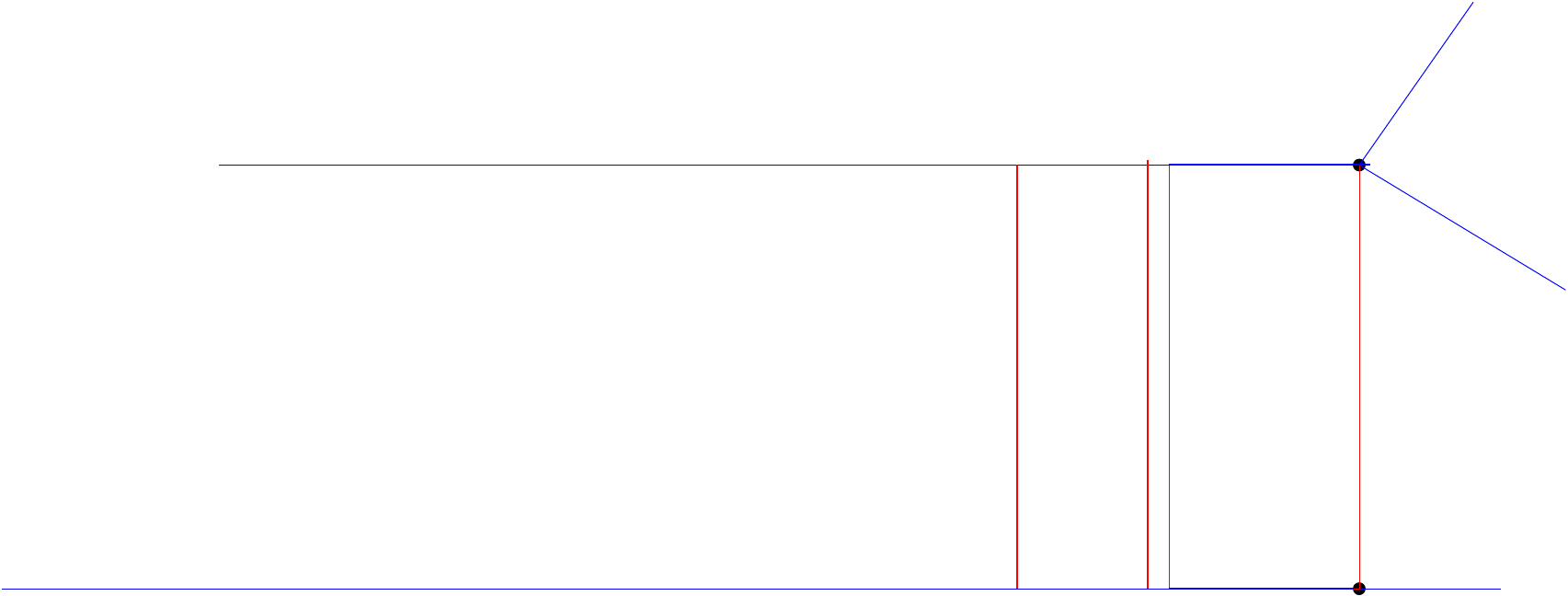}
	\caption{Unstable rails}
	\label{Fig: u Rails}
\end{figure}

\begin{rema}\label{Rema: finite rails intersect or equals}

The set of extremal rails is finite because there are only a finite number of extreme points in $\delta^s$, and a finite number of unstable separatrices pass through each of them.

If two regular rails intersect each other, they either coincide (i.e., they are equal) or have a single common endpoint along $\delta^s$.
\end{rema}

Let $\delta^s$ be an $s$-adapted graph of $f$. An equivalence relation will be defined on the surface $S$ minus the graph and its extremal rails, denoted as 
$$
\overset{o}{S}=S\setminus (\delta^s \cup \text{Ex}^u(\delta^s)).
$$
to this set. 
Since the graph $\delta^s$ and its extreme rails consist of a finite number of compact intervals, their complement generates a finite number of connected components. The following lemma is then established.

\begin{lemm}\label{Lemm: finite c.c. }
	The surface $\overset{o}{S}$ has a finite number of connected components. Denote the family of these connected components by $\{r_i\}_{i=1}^n$.
\end{lemm}

There is another perspective from which we can consider the connected components of $\overset{o}{S}$.

     \begin{defi}[Equivalent rails]\label{defi: equivalent rails}
	Let $\delta^s$ be an $s$-adapted graph for $f$. Two points $x$ and $y$ in $\overset{o}{S}$ are considered equivalent with respect to $\delta^s$ if there exists a finite family of $u$-regular rails $\{J_0, J_1, \cdots, J_k\}$ and embedded rectangles $\{R_1, R_2, \cdots, R_k\}$ satisfying the following conditions:
		\begin{itemize}
		\item Either $x,y\in J_0$, or $x\in J_0$ and $y\in J_k$.
		\item For each $i\in \{1,\cdots,k\}$, $J_{i-1}$ and $J_i$ are unstable boundaries of $R_i$.
		\item For each $i\in \{1,\cdots,k\}$, the intersection of $R_i$ with $\delta^s$ is equal to the stable boundary $\partial^s R_i$, i.e.  $R_i\cap \delta^s=\partial^s R_i$.
	\end{itemize}.
We denote $x \sim_{\delta^s} y$ when this defined relation holds between $x$ and $y$.
\end{defi}

\begin{lemm}\label{lemm: equiv classe = c.c}
The relation $\sim_{\delta^s}$ is an equivalence relation on $\overset{o}{S}$.
\end{lemm}

 \begin{proof}
The relationship is reflexive since we allow for the case when $x$ and $y$ belong to the same $u$-regular rail $J_0$, and it is symmetric due to a modification in the labeling of the rectangles. Let's prove the transitivity.

Suppose $x \sim_{\delta^s} y$ through rails $\{J_0, \cdots, J_k\}$ and rectangles $\{R_1, \cdots, R_k\}$, and also $y \sim_{\delta^s} z$ through rails $\{J'_k, \cdots, J_n\}$ and rectangles $\{R'_{k+1}, \cdots, R_l\}$. The segments $J_k$ and $J'_k$ intersect each other in their interior because $y \in \overset{o}{S}$ is not an extreme point of them. By Remark  \ref{Rema: finite rails intersect or equals}, they are equal, and we can concatenate the family of regular rails and rectangles to obtain the equivalence between $x$ and $z$.

\end{proof}

\begin{prop}\label{lemm: equivalent class rectangles}
Each equivalence class $r$ of the equivalence relation $\sim_{\delta^s}$  is contained in a unique connected component of $\overset{o}{S}$ and each connected component of $\overset{o}{S}$ is contained in a single equivalent class of the equivalence relation $\sim_{\delta^s}$ .
\end{prop}

\begin{proof}
	
By definition two points of the same equivalent class $r$ are contained in the union of a finite family of rectangles intersecting at their boundaries and, therefore, the equivalent class $r$ is path-connected, therefore $r$ is connected. This implies that $r$ is within a single connected component of $\overset{o}{S}$.

Now let's take $x\in r$  and the $u$-regular $J$-rail passing through $x$, this $u$-rail has endpoints $x_1$ and $x_2$ in the regular part of $\delta^s$. Since $x_1$ and $x_2$ are in the interior of $\delta^s$ we can consider a small rectangle $H$ whose unstable boundary is given by two $u$-regular rails $J_1$ and $J_2$, such that: $J_1$ have extreme points $y_1,y_2$ and $J_2$ have  extreme points $z_1,z_2$, but in such a way that, $x_1\in (y_1,z_1)^s \subset \overset{o}{\delta^s}$ and $x_2 \in (y_2,z_2)\subset \overset{o}{\delta^s}$. Let's take now $U=\overset{o}{H}$, with this construction we can deduce that for all $y\in U$,  $y\sim_{\delta^s} x$ and then $U\subset r$. Therefore $r$ is an open set in $\overset{o}{S}$.

Let $C$ be the connected component of $\overset{o}{S}$ that contains $r$. If there was another equivalent class contained in $C$, we could create a disconnection of $C$ using open sets, which is impossible since $C$ is a connected component. Therefore, $C$ contains only the equivalent class $r$, and it follows that $C = r$.
\end{proof}

  \begin{lemm}\label{lemm: Equivalent class rectangle}
The closure of every equivalent class $r$ in $S$ is a rectangle.
\end{lemm}

\begin{proof}
	Let $x\in r$ be any point, and let $J_x$ be the unique $u$-regular rail passing through $x$. If $x\sim_{\delta^s}y$, the rectangles that establish the equivalence between them share a compatible vertical orientation. This allows us to define the upper end of $J_x$, denoted by $a(x)$, and the lower end, denoted by $b(x)$. Since the foliations are $C^0$ and $x$ and its extreme points are far from the singular points of $f$ (which some of the end points of $\delta^s$), there exists an open set $U_x\subset r$ such that for all $y\in U_x$, $x\sim_{\delta^s} y$ (see the argument used in  the proof of Proposition \ref{lemm: equivalent class rectangles}). Therefore, we can define two continuous functions on $r$:
	$$
	\pi_{\pm}: r \rightarrow \delta^s,
	$$
where $\pi_+(x)=a(x)$ and $\pi_-(x)=b(x)$. Consider $I_+:= \pi_+(r)$ and $I_-:= \pi_-(r)$. Since the functions $\pi_{\pm}$ are continuous and $r$ is connected, $I_+$ and $I_-$ are intervals. In fact, $I_+$ and $I_-$ correspond to the interiors of closed intervals $[a_-,a_+]^s$ and $[b_-,b_+]^s$, respectively, both contained in $\delta^s$. Here, we have assumed that $a_-$ and $b_-$ correspond to the left-hand sides of our intervals, which makes sense since all rectangles that form the equivalence class $r$ share a congruent horizontal orientation.

\begin{lemm}\label{Lemm: Same unstable leaf}
The points $a_-$ and $b_-$ belong to the same unstable leaf, and the interval $[b_-,a_-]^u$ is an extreme rail of $\delta^s$. Moreover, the rectangle defined by the extremal rail $[b_-,a_-]^u$ (as given by the definition) intersects the equivalent class $r$. The same statement holds for the points $a_+$ and $b_+$.

\end{lemm}
	
\begin{figure}[h]
	\centering
	\includegraphics[width=0.7\textwidth]{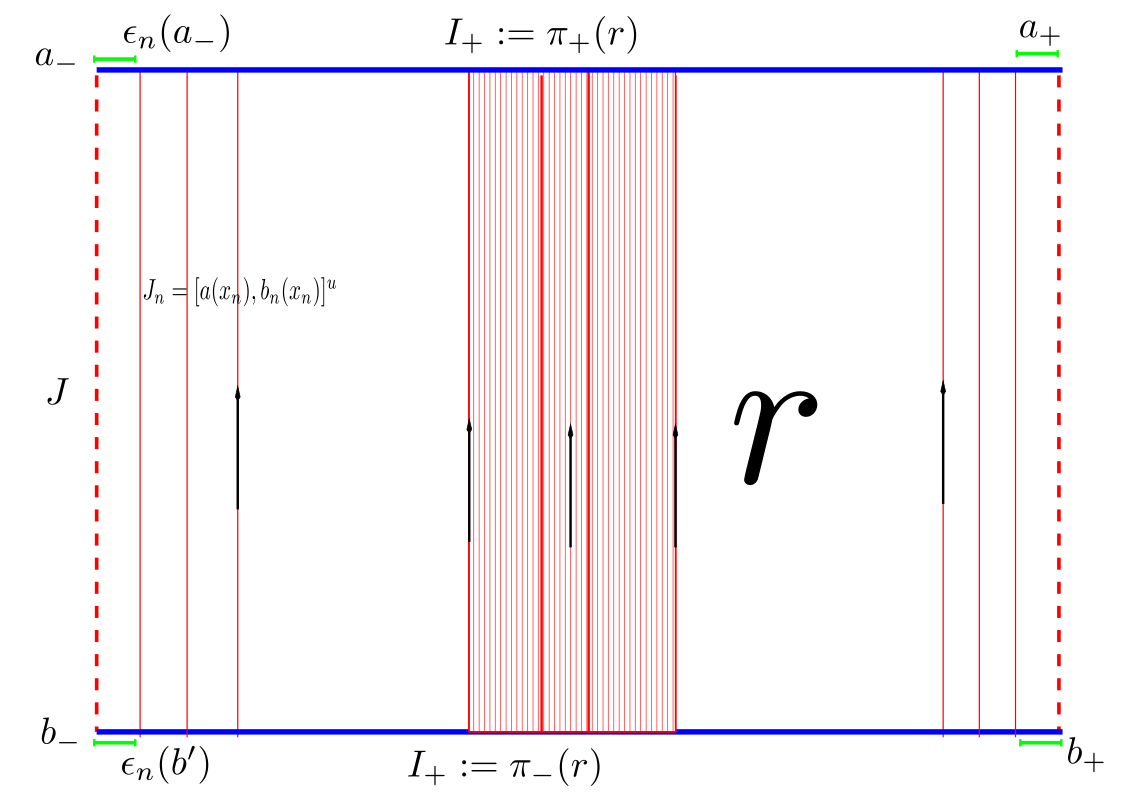}
	\caption{Intervals $[a_-,a_+]^s$ and $[b_-,b_+]^s$}
	\label{Fig: Intervals same leaf}
\end{figure}

	\begin{proof}
		
Look at Figure \ref{Fig: Intervals same leaf} for a pictorial representation of our arguments.
 
  Let $\{a_n\}_{n=1}^{\infty} \subset I_+ $ and $\{b_n\}_{n=1}^{\infty} \subset I_-$ be two convergent successions, $a_n\rightarrow a_-$ and $b_n \rightarrow b_-$, such that $a_n$ and $b_n$ are the extreme points of a regular $u$-rail $J_n\subset r$.

  Suppose that $a_-$ is not in the same unstable leaf than $b_-$. Therefore, either through $a_-$ passes an unstable interval $J$ intersecting $(b_-,b_+)^s$ only at one end point or through $b_-$ passes an unstable interval intersecting $(a_-,a_+)^s$ only at one end point.  Any way there is a situation that must exclude that is that $J$ contains a singularity in its interior and $p\in \overset{o}{J}$ and $p$ have a separatrice $W^s(p)$ entering to the equivalent class $r$ that we are considered (See the Figure \ref{Fig: No conf P}).
  
  \begin{figure}[h]
  	\centering
  	\includegraphics[width=0.7\textwidth]{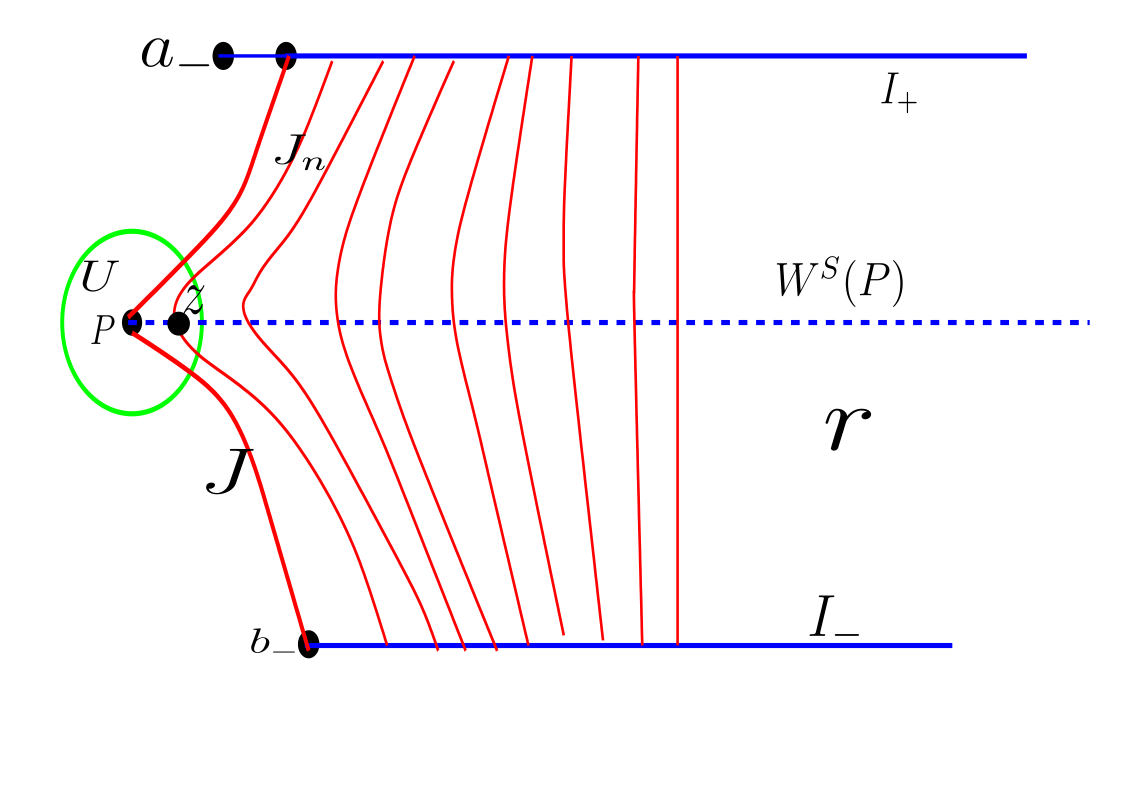}
  	\caption{Impossible configuration of the stable interval $J$}
  	\label{Fig: No conf P}
  \end{figure}
  
   Assume this is the situation, like $p\in \delta^s$ is a extreme point of the stable graph, every open neighbor $U$ of $p$ intersect the rectangular part of $\delta^s$, by the $C^0$ continuity of the foliations,  then there exist a regular rail $J_n$ (with a end point close enough to $a^-$) whose interior intersect $U\cap \overset{o}{\delta^s}$ in a point $z$ that is not contain in the stable interval $\overset{o}{I_-}$ or in $\overset{o}{I_+}$, but this is not possible as every regular $u$-rail only intersect $\overset{o}{\delta^s}$ in its end points, in this manner we carried a contradiction.  Without loss of generality we can assume that  $J$ intersecting $(b_-,b_+)^s$ only at one of its end point, and in case of intersect a singularity in its interior, such singularity don't have a separatrice entering in $r$.

  Let $\epsilon_n(a_-):=\mu^u([a_-,a_n]^s)$ and $\epsilon_
 n(b_-):\mu^u([b',b_n]$, like the measure $\mu^u$ is invariant by leaf isotopy, $\epsilon_n(a_-)=\epsilon_n(b')$. We deduce that $b_n\rightarrow b'$ and finally that $b'=b_-$, but if $a_-$ is not in the same leaf ot $b_-$, the equality $\epsilon_n(a_-)=\epsilon_n(b')$ couldn't hold. We arrive to a contradiction, hence  $a_-$ is in the same unstable leaf than $b_-$. 
 
Let $J=[a_-,b_-]^u$ be the interval determined by these points. It is not possible for $a_-$ and $b_-$ to be regular points, as we could increase the equivalent class $r$ (by moving to the left of $a_-$ and $b_-$). Therefore, $J$ is not a regular rail.
    
  By taking $J_n$ close enough to $J$, we can construct a rectangle $R$ with an unstable boundary equal to the disjoint union of $J$ and $J_n$. Furthermore, the intervals $[a_-,a_n]^s$ and $[b_-,b_n]^s$ are disjoint, and therefore, $R$ is embedded. The stable boundary of $R$ is contained in $\delta^s$, and every unstable segment $J'$ of $R$ other than $J$ has its ends in $\overset{o}{\delta^s}$, so $J'$ is a regular rail equivalent to $J_n$. This shows that $J$ is an extremal rail, and $R$ intersects $r$ by construction.
  
  	\end{proof}

We deduce the existence of two embedded rectangles $R_-$ and $R_+$ with $J_-: =[a_-,b_-]^u$ on the unstable boundary of $R_-$ and $J_+=[a_+,b_+]^u$ on the unstable boundary of $R_+$. Both $R_-$ and $R_+$ intersect $r$ (i.e., $R_- \cap r \neq \emptyset$ and $R_+ \cap r \neq \emptyset$), and they can be made arbitrarily thin. Furthermore, we can ensure that $R_- \cap R_+ \cap r = \emptyset$, meaning they are disjoint within the equivalent class $r$. Moreover, $R_-$ and $R_+$ are the rectangles defined by the extreme rails $J_-$ and $J_+$, respectively, so any other $u$-rail contained in them is regular.

If $J'_{-}$ and $J'_{+}$ are the other vertical sides of $R_{-}$ and $R_+$ (respectively), there exists a finite family of rectangles $\{R_i\}_{i=1}^n$ that establish an equivalence between $J'_{-}$ and $J'_{+}$. Finally, we have:

	$$
	r=[r\cap R_-] \cup [ \cup_{i=1}^n  r \cap R_i] \cup [r \cap R_+].
	$$

	Two consecutive rectangles $r \cap R_i$ and $r \cap R_{i+1}$ share a single rail. Moreover, if $r \cap R_i \cap (r \cap R_{i+2})$ (when it makes sense) is not empty, then they intersect in the interior of a regular rail. Therefore, $R_i \cup R_{i+1} \cup R_{i+2}$ forms an annulus, and $\alpha^s$ contains a closed curve, which is not possible.
	
	Based on this argument, we can deduce that for any $x \in r$, there exists a unique closed stable interval $I_x$ and a unique closed unstable interval $J_x$ such that:

	\begin{itemize}
\item $\overset{o}{I_x},\overset{o}{J_x}\subset r$
\item $\overset{o}{I_x}\cap \overset{o}{J_x}=\{x\}$
\item The extreme points of $I_x$ are denoted as $a(x)$ and $b(x)$, and they satisfy that $a(x) \in (a_-,a_+)^s$ and $b(x) \in (b_-,b_+)^s$. Furthermore, it is true that $a(x) \neq b(x)$.
\item The extreme points of $J_x$ are denoted as $c(x)$ and $d(x)$, and they satisfy that $c(x) \in J_-$ and $d(x) \in J_+$. Additionally, we have $c(x) \neq d(x)$.
	\end{itemize}
A further conclusion is the existence of a homeomorphism:

	$$
	h:(0,1)\times (0,1) \rightarrow r,
	$$
such that it sends the horizontal lines of $(0,1)\times (0,1)$ to the stable leaves of $r \cap \mathcal{F}^s$ and the vertical lines of $(0,1)\times (0,1)$ to the unstable leaves of $r \cap \mathcal{F}^u$. However, it might be impossible to extend $h$ to a homeomorphism on $[0,1]\times [0,1]$ because the rectangles $\{R_i\}_{i=1}^n$ can intersect at their stable boundaries, and $R_-$ could intersect with $R_{+}$ at their stable boundaries or at their extreme rails $J_-$ and $J_+$. With this in mind, we can define:
	$$
	R:=R_- \cup [ \cup_{i=1}^n   R_i] \cup  R_+,
	$$
	where it is clear that $\overline{r}=R$, and the homeomorphism $h$ extends to a continuous map:
	$$
	h:[0,1]\times [0,1]\rightarrow R.
	$$
	in the following manner:
	\begin{itemize}
		\item 	Let $t_0\in (0,1)$. If $h(t_0,(0,1))=\overset{o}{I_x}$, then $h(t_0,0)=b(x)$ and $h(t_0,1)=a(x)$.
		
		\item  Let $s_0\in (0,1)$. If $h((0,1),s_0)=\overset{o}{J_x}$, then $h(0,s_0)=c(x)\in J_-$ and $h(1,s_0)=d(x)\in J_+$.
		
		\item The corner points are determined by: $h(0,0)=b_-$, $h(1,0)=b_+$, $h(0,1)=a_-$ and $h(1,1)=a_+$.
	\end{itemize}
	
	This continuous function, when restricted to a vertical or horizontal line in $[0,1]\times [0,1]$, is a homeomorphism whose image is an interval contained in $\mathcal{F}^u\cap R$ or $\mathcal{F}^s\cap R$, respectively. This proves that $R$ is an (immersed) rectangle.
\end{proof}

 Lemma \ref{Lemm: finite c.c. }  establishes that there exists a finite number of connected components of the surface $\overset{o}{S}$, denoted as $\{r_i\}_{i=1}^n$. Therefore, according to the recent Lemma \ref{lemm: Equivalent class rectangle}, there exists a finite family of rectangles associated with the $s$-adapted graph $\delta^s$. From the proof of Lemma \ref{lemm: Equivalent class rectangle} , it is possible to determine that the stable boundaries of these rectangles are contained in $\delta^s$, and the unstable boundary of each rectangle consists of two extreme rails. Clearly, the boundary of $\cR(\delta^s)$ is equal to the union of the graph $\delta^s$ and its extreme rails Ex$(\delta^s)$, which were eliminated from $S$ to obtain $\overset{o}{S}$. We can formalize this discussion in the following Definition and Corollary.

 \begin{defi}\label{Defi: s,u-rectangles}
Let $\delta^s$ be an $s$-adapted graph of $f$. Let $\{r_i\}{i=1}^n$ be the set of equivalent classes of the relation $\sim_{\delta^s}$. We define the rectangle

$$
R_i(\delta^s):=\overline{r_i}.
$$
and the family of \emph{induced rectangles} by $\delta^s$ as  $\cR(\delta^s)=\{R_i(\delta^s)\}_{i=1}^n$.
\end{defi}

\begin{coro}\label{Coro: Boundaries induced rectangles}
Any rectangle $R_i = R_i(\delta^s)$ in the family of rectangles induced by $\delta^s$ satisfies the following properties:

\begin{itemize}
\item Its stable boundary $\partial^s R_i$ is contained in $\delta^s$.
\item  Its unstable boundary $\partial^u R_i$ consists of two extreme rails of $\delta^s$.
\end{itemize}

Furthermore, the stable boundary of $\cR(\delta^s)$ is equal to $\delta^s$, i.e., $\partial^s \cR(\delta^s) = \cup_{i=1}^n \partial^s R_i = \delta^s$, and its unstable boundary is the set of extreme rails of $\delta^s$, i.e., $\partial^u \cR(\delta^s) = \cup_{i=1}^n \partial^u R_i = \text{Ex}(\delta^s)$.

\end{coro}

\subsection{Compatible graphs and Markov partitions.}

In view of the previous Corollary \ref{Coro: Boundaries induced rectangles}, the stable boundary of $\cR(\delta^s)$ is $f$-invariant. However, this alone is not sufficient to create a Markov partition. In order to achieve that, we need the vertical boundaries of the rectangles to be $f^{-1}$-invariant. This leads us to introduce the notion of \emph{compatible graphs}.

  \begin{defi}\label{defi:compatible graphs}
	Let $\delta^s$ and $\delta^u$ be $s$-adapted and $u$-adapted graphs for $f$ respectively. We call them \emph{compatible} if:
	\begin{itemize}
		\item The extreme points of $\delta^s$ belong to $\cup\delta^u$, and the extreme points of $\delta^u$ belong to $\cup \delta^s$.
		\item  The union of the intervals in the graph $\delta^u$ contains the extreme $u$-rails of $\delta^s$, and the union of the intervals in $\delta^s$ contains the extreme $s$-rails of $\delta^u$.
	\end{itemize}
\end{defi}

The following proposition is key to obtaining a Markov partition for $f$.

 \begin{prop}\label{Prop:compatibles implies Markov partition}
	If $\delta^u$ and $\delta^s$ are compatible graphs, we denote $\cup \delta^{s,u}$ to the union of the intervals in $\delta^{s,u}$ and  $\delta^s\cup\delta^s:=\cup \delta^u \, \cup \, \cup\delta^s$ is the boundary of a Markov partition adapted to $f$. This Markov partition consists of the rectangles obtained by intersecting the family $\{R_i(\delta^s)\}_{i=1}^n$ with the family $\{R_j(\delta^u)\}_{j=1}^m$  of rectangles induced by $\delta^s$ and $\delta^u$.
\end{prop}

\begin{proof}
By considering the equivalence relations $\sim_{\delta^s}$ and $\sim_{\delta^u}$, we obtain two families of rectangles: the family induced by $\delta^s$, denoted by $\{R_i^s:=R_i(\delta^s)\}_{i=1}^n$, and the family induced by $\delta^u$, denoted by  $\{R^u_j=R_j(\delta^u)\}_{j=1}^m$ where we have simplified the notation. This allows us to define another family of sets:

\begin{eqnarray*}
\cR(\delta^s,\delta^u):=\{ \overline{C}: C \text{ is a c.c. of } \overset{o}{R^s_i}\cap \overset{o}{R_j^u} \\ \text{ with }  1\leq i \leq n, \text{ and } 1\leq j\leq m  \}.
\end{eqnarray*}

\begin{figure}[h]
	\centering
	\includegraphics[width=0.7\textwidth]{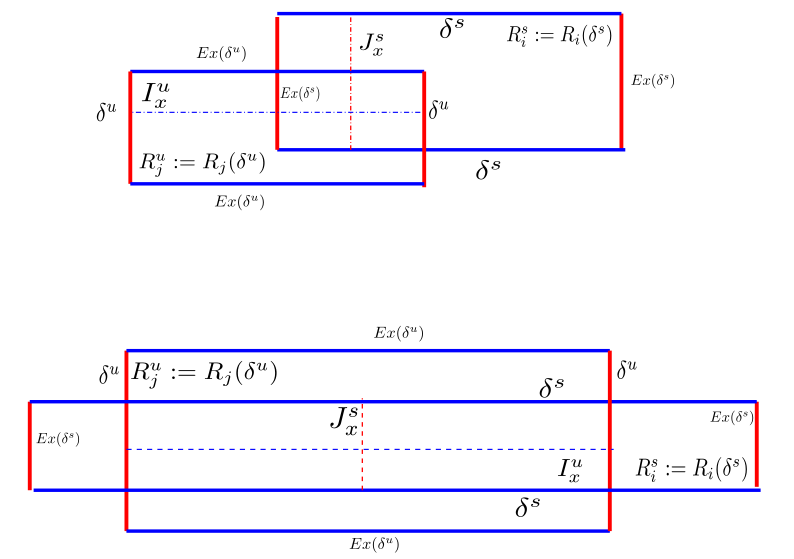}
	\caption{Impossible and possible configurations for compatible graphs}
	\label{Fig: Intersection adapted}
\end{figure}

It is evident that this family covers $S$, and the sets have disjoint interiors. We will now show that they are, in fact, rectangles. Please refer to Figure \ref{Fig: Intersection adapted} to follow our notation.

Let $C$ be a nonempty connected component of $\overset{o}{R^s_i}\cap \overset{o}{R_j^u} $, and let $x\in C$ be any point. Let $J^s_x\subset R^s_i$ denote the vertical segment of $R^s_i$ passing through $x$. We claim that the interior of $J^s_x$ is contained in $C$. If this is not the case, $J^s_x$ would intersect $\partial^s R_j^u$ at an interior point (refer to the top image in Figure \ref{Fig: Intersection adapted}). However, $\partial^s R_j^u$ is formed by two extremal $s$-rails of $\delta^u$. By hypothesis, such extremal rails are contained in $\delta^s$. Therefore, we deduce that $J^s_x$ intersects $\delta^s$ in its interior, which is not possible since the only points where a $u$-rail of $\delta^s$ intersects the graph $\delta^s$ are endpoints. Hence, $C$ must be the interior of a vertical sub-rectangle of $R^s_i$.

Let $I^u_x$ be the horizontal segment of $R_j^u$ passing through $x$. We claim that the interior of $I^u_x$ is contained in $C$. If this is not the case (as shown in the top image of Figure \ref{Fig: Intersection adapted}), the interior of $I^u_x$ would intersect $\partial^u R_i^s$, which is an extreme $u$-rail of $\delta^s$. However, by hypothesis, such an extremal $u$-rail is contained in $\delta^u$. Thus, $I^u_x$ would intersect $\delta^u$ in its interior, which is not possible since any $s$-rail of $\delta^u$ intersects the graph $\delta^u$ only at its endpoints. Consequently, we deduce that $C$ must be the interior of a horizontal sub-rectangle of $R_j^u$.

Then the configuration of the intersection between $\overset{o}{R_i^s}$ and $\overset{o}{R_j^u}$ is as shown in the bottom image of Figure  \ref{Fig: Intersection adapted}. In particular, $\overline{C}$ is a rectangle with horizontal boundary contained in $\delta^s$ and  vertical boundary contained in $\delta^u$. Therefore, the union of the horizontal boundaries of the rectangles in $\cR(\delta^s,\delta^u)$ is contained in the union of the intervals in $\delta^s$, denoted as $\cup\delta^s$, and the union of their vertical boundaries is contained in $\cup \delta^u$.

In fact, every point $z\in \cup \delta^s$ lies on the stable boundary of some $\overline{C}\in \cR(\delta^s,\delta^u)$ because $z$ is on the horizontal boundary of some $R_i^s$ and within the closure of some $R_j^u$ such that $\overset{o}{R_i^s}\cap \overset{o}{R_j^u}\neq \emptyset$. Therefore, $z$ lies on the stable boundary of the closure of a connected component of such intersection. This implies that $\cup \delta^s$ is contained in the horizontal boundary of $\cR(\delta^s,\delta^u)$, and hence they are equal, i.e., $\partial^s\cR(\delta^s,\delta^u)=\cup \delta^s$. A symmetric argument shows that $\partial^u\cR(\delta^s,\delta^u)=\cup \delta^u$.

Like $\cup \delta^s$ is $f$-invariant, the horizontal boundary of $\cR(\delta^s,\delta^u)$ is also $f$-invariant. Likewise, we deduce that the vertical boundary of $\cR(\delta^s,\delta^u)$ is $f^{-1}$-invariant. In conclusion, $\cR(\delta^s,\delta^u)$ is a Markov partition.

Finally, since $\delta^s$ and $\delta^u$ are graphs adapted to $f$, the Markov partition itself is adapted.
\end{proof}

If we start with two graphs adapted to $f$, the following result provides a formula to obtain compatible graphs by considering the minimal property of the measured foliations in $S$, i.e., their leaves being dense in $S$ (Proposition \ref{Prop: pseudo-Anosov properties.}), and taking into account the uniform expansion and contraction along the leaves.

 \begin{lemm}\label{Lemm: Iteration to be adapted}
Let $\delta^s$ and $\delta^u$ be graphs adapted to $f$. Then there exists an $n := n(\delta^s, \delta^u) \in \mathbb{N}$ such that $\delta^s$ and $h^n(\delta^u)$ are compatible whenever $m \geq n$, and $n$ is the minimum natural number with this property.
\end{lemm}

\begin{proof}

First, we prove that there exists $N_1 \in \mathbb{N}$ such that for all $n > N_1$, $\delta^s$ contains the extreme points of $h^n(\delta^u)$ and $h^n(\delta^u)$ contains the extreme points of $\delta^s$. Next, we establish an $N_2 \in \mathbb{N}$ such that for all $n > N_2$, the extreme $u$-rails of $\delta^s$ are in $h^n(\delta^u)$ and the extreme $s$-rails of $h^n(\delta^u)$ are contained in $\delta^s$. These conditions imply that:

$$
\{n\in \NN : \delta^s \text{ is compatible with }f^n(\delta^u)\}\neq \emptyset,
$$

and then, we can define $n := n(\delta^s, \delta^u)$ as the minimum of this set.

Let's start by assuming that  $\delta^u=\{J_j\}_{j=1}^h$ and $\delta^s=\{I_i\}_{i=1}^l$. We define:
$$
L^u:=\min\{\mu^s(J_j): 1\leq j\leq h\}.
$$
It is clear that for every $n\in \NN$ and every interval $J_j$
$$
\mu^s(f^n(J_i))=\lambda^n\mu^s(J_i)\leq \lambda^n L^u.
$$

Let $z\in \delta^s$ be an extreme point of $\delta^s$. Since $\delta^s$ is an adapted graph, the point $z$ lies on the unstable leaf of a singularity of $f$. Let $[p_z,z]^u$ denote the unique compact interval of the unstable leaf passing through $z$ and connecting a singularity $p_z$ of $f$ with $z$. Now, let's define:
$$
F^u:=\max\{\mu^s([p_z,z]^u): z \text{ is extreme point of  } \delta^s \}.
$$

By the uniform expansion in the unstable leaves of $\mathcal{F}^u$, there exists $n_1\in \mathbb{N}$ such that for every $n\geq n_1$, we have $\lambda^nL^u > F^u$. Moreover, if $n > n_1$ and $z$ is an extreme point of $\delta^s$, there exists a unique $J_j$ in $\delta^u$ such that $z$ lies on the same separatrice as $f^n(J_j)$. However,
$$
\mu^u(f^n(J_j))\geq \lambda^nL^u > F^u\geq \mu^s([p_z,z]^u),
$$
implying that $z$ belongs to $f^n(J_j)$, or equivalently $z$ belongs to $f^n(\delta^u)$ for all $n\geq n_1$.

In the same way, but using $f^{-1}$ and the measure $\mu^u$, we can deduce the existence of $n_2\in \NN$ such that for all $n\geq n_2$ and every extreme point $z$ of $\delta^u$, $z$ is also an extreme point of $f^{-n}(\delta^s)$. In other words, for all $n\geq n_2$ and every extreme point $z$ of $f^n(\delta^u)$, $z$ is contained in $\delta^s$. Let's take $N_1=\max\{n_1,n_2\}$.

Now, let's consider the set of $u$-extreme rails of $\delta^s$. This set is finite, and each $u$-extreme rail is a closed interval.

For each $p\in \textbf{Sing}(f)$ and each unstable separatrice $F^u_i(p)$ of $p$ (if $p$ is a $k$-prong, we consider $i=1,\cdots,k$), we define $J(p, i)\subset F^u_i(p)$ as the minimal compact interval containing the intersection of all the $u$-extreme rails of $\delta^s$ with the separatrice $F^u_i(p)$, in case that some $u$-extremal rail $J$ of of $\delta^s$ contains a singularity in its interior we consider just the sub-interval of $J$ contained in the separatrice $F^u_i(p)$. Now, we consider the following quantity, which is finite:

$$
M^u:=\max\{ \mu^s(J(p,i)): p\in \textbf{Sing}(f) \text{ is a $k$- prong and } i=1,\cdots k  \}.
$$

Like $\delta^u=\{J_j\}_{j=1}^h$ the next quantity is finite too,
$$
G^u=\min\{\mu^s(J_j): j=1,\cdots,h\}.
$$

By the uniform expansion in unstable leaves of $\mathcal{F}^u$, there exists $n_1\in \mathbb{N}$ such that for all $n\geq n_1$, $\lambda^n G^u \geq M^u$. Furthermore, for any $k$-prong $p\in \textbf{Sing}(f)$ and every $i=1, \cdots, k$, there exists an interval $J_j$ in $\delta^u$ such that $f^n(J_j)$ is contained in the same unstable separatrice as $J(p,i)$, indeed $J_j$ is the interval of $\delta^u$ contained in $F^u_i(p)$ with and end point in $p$ that is given by item $iii)$ in Definition \ref{Defi:Adapted graph} (in the $u$ case). Moreover, we have the following computation:

 $$
 \mu^s(f^n(J_j))=\lambda^n\mu^s(J_i)\geq \lambda^n G^u \geq M^u.
 $$

 Since $\mu^s(J(p,i)) \leq M^u$, this calculation implies that $J(p,i)\subset f^n(J_j)$. By the construction of $J(p,i)$, we deduce that for each $n\geq n_1$, any $u$-extreme rail of $\delta^s$ is contained in $f^n(\delta^u)$.

A similar proof using $f^{-1}$ and the measure $\mu^u$ gives another natural number $n_2\in \NN$ such that for all $n\geq n_2$, any extreme $s$-rail of $\delta^u$ is contained in $f^{-n}(\delta^s)$, or equivalently, all extreme $s$-rails of $f^{n}(\delta^u)$ are contained in $\delta^s$. Let $N_2=\max(n_1,n_2)$. The conclusion is that the following quantity exists:
 $$
n(\delta^s,\delta^u):=\min \{N\in \NN: \forall n\geq N \, \delta^s \text{ is compatible with } f^n(\delta^u) \},
 $$ 

Which is the number we were looking for
\end{proof}

\subsection{A recipe for Markov partitions.}
We can now summarize all these steps to construct a Markov partition:

\begin{prop}\label{Prop: Recipe for Markov partitions}
Let $p$ and $q$ be singular points of $f$, with separatrices $F^s(p)$ and $F^u(q)$ respectively, and let $z$ be a point in their intersection $F^s(p)\cap F^u(q)$. Consider the following construction:
\begin{enumerate}
\item Define $J^u(z)$ as the unstable interval $[p,z]^u$ in $F^u(p)$, which is referred to as the \emph{primitive segment}.

\item Define  $\cJ^u(z)=\cup_{i\in \NN}f^{-i}(J^u(z))$. Due to contraction in the unstable foliation, $\mathcal{J}^u(z)$ is a finite union of closed intervals and is $f^{-1}$-invariant.

\item  The graph $\mathcal{J}^u(z)$ satisfies the conditions of the Lemma {Lemm: the generated graph is adapted}, then the graph generated by $\mathcal{J}^u(z)$ is an $s$-adapted graph to $f$, denoted as $\delta^s(z)$.

\item The graph $\delta^s(z)$ is $f$-invariant. By applying  Lemma \ref{Lemm: the generated graph is adapted} to $\delta^s(z)$, we obtain the $u$-adapted graph $\delta^u(z)$ generated by $\delta^s(z)$.

\item Using Lemma \ref{Lemm: Iteration to be adapted}, we can determine a number $n(z) = n(\delta^s(z), \delta^u(z))$ such that for all $n > n(z)$, $\delta^s(z)$ and $f^n(\delta^u(z))$ are compatible.

\item Finally, for all $n\geq n(z)$, Proposition  \ref{Prop:compatibles implies Markov partition} implies the existence of a adapted Markov partition $\mathcal{R}(z,n)$ with  stable boundary equal to $\cup \delta^s(z)$ and  unstable boundary equal to $\cup f^n(\delta^u(z))$.

\end{enumerate}
 We refer to $n(z)$ as the \emph{compatibility coefficient} of $z$.
\end{prop}

Therefore we recover a classic result.

\begin{coro}\label{Coro: Existence adapted Markov partitions}
Every generalized pseudo-Anosov homeomorphism admits adapted Markov partitions.
\end{coro}

\section{Canonical Markov partitions and canonical geometric types}

We want to apply Proposition  \ref{Prop: Recipe for Markov partitions} to a particular family of points known as \emph{first intersection points}. These points play a crucial role in determining a distinguished and infinite family of Markov partitions. In fact, this family can be regarded as the orbit of a finite number of Markov partitions. We will show that the geometric type of a Markov partition and its iterates are equal. This allows us to define a finite family of geometric types associated with each pseudo-Anosov homeomorphism, based on this distinguished family of Markov partitions. In this way, we obtain a countable family of finite invariants for each conjugacy class. These invariants have the property of being calculable to some extent in terms of any other geometric type. In the final chapter, we will utilize this family of invariants to address point $III$ in Problem \ref{Prob: Clasification} . However, before proceeding, we need to construct these invariants.

\subsection{First intersection points} We start by defining these special points.

\begin{defi}\label{Defi: first intersection points}	
		
	Let $f$ be a generalized pseudo-Anosov homeomorphism, $p, q \in \textbf{Sing}(f)$, $F^s(p)$ be a stable separatrice of $p$, and $F^u(q)$ be an unstable separatrice of $q$.
	
	A point $x \in [\cF^s(p) \cap \cF^u(q)] \setminus \textbf{Sing}(f)$ is a \emph{first intersection point} of $f$ if the stable interval $[p,x]^s \subset F^s(p)$ and the unstable segment $[q,x]^u \subset F^u(q)$ have disjoint interiors. In other words:
	
	 $$
	 (p,x]^s\cap (q,x]^u=\{x\}.
	 $$. 
\end{defi}

\begin{lemm}
There exists at least one first intersection point for $f$
\end{lemm}

\begin{proof}
Let $I$ be a compact interval contained in $F^s(p)$ with one endpoint equal to $p$. Take a singularity $q$ (that could be equal to $p$), like any unstable separatrice $F^u(q)$ of $q$ is minimal there exist a closed interval $J\subset F^u(q)$ such that $\overset{o}{J}\cap \overset{o}{I}=\emptyset$.
Like $J$ and $I$ are compact sets, their intersection $J\cap I$ consist in a finite number of points, $\{z_0,\cdots,z_n\}$. We can assume that $n>1$ and we orient the interval $J$ pointing towards $q$ and whit this orientation $z_i<z_{i+1}$. In this manner:
\begin{itemize}
\item If $p=q$, then $z_0=p=q$ but $p\neq z_1$ and we take $z=z_1$.
\item If $p\neq q$, then $p\neq z_0 \neq q$ and we take $z=z_0$.
\end{itemize}

 Clearly, $(p,z]^s\cap (q,z]^u=\{z\}$, and $z$ is a first intersection point.
\end{proof}

\begin{figure}[h]
	\centering
	\includegraphics[width=0.5\textwidth]{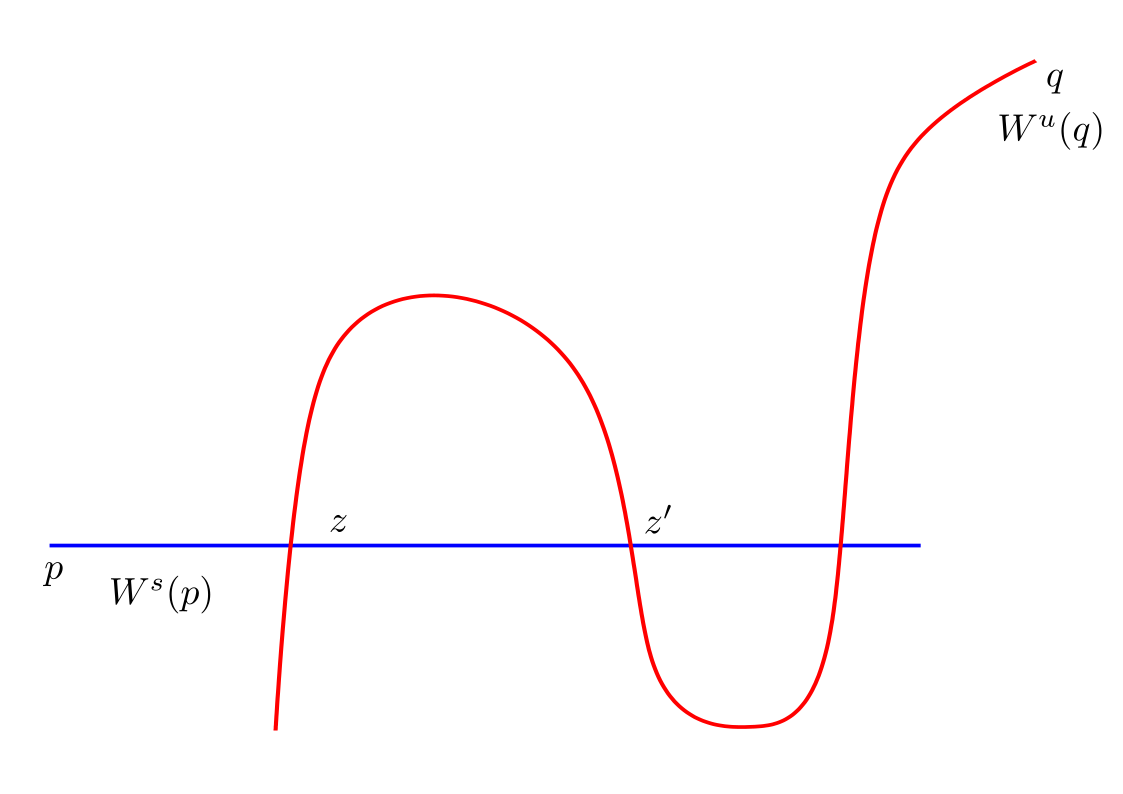}
	\caption{$z$ first intersection point and $z'$ non first intersection point}
	\label{Fig: First intersection point}
\end{figure}

Now we will proceed to prove some properties about these points.

\begin{lemm} \label{Lemm: Image of first intersection is first intersection}
If $z$ is a first intersection point of $f$, then $f(z)$ is also a first intersection point of $f$.
\end{lemm}

\begin{proof}
If $z$ is a first intersection point, there exist singular points of $f$, $p$ and $q$, such that:
$$
\{z\}=(p,z]^s\cap (q,z]^u.
$$
Therefore,
$$
\{f(z)\}=(f(p),f(z)]^s\cap (f(q),f(z)]^u
$$
is a first intersection point of $f$.
\end{proof}

The following result establishes the finiteness of the orbits of these first intersection points:

\begin{prop}\label{Prop: Finite number first intersection points}
Let $f$ be a generalized pseudo-Anosov homeomorphism. Then there exists a finite number of orbits of first intersection points under $f$.
\end{prop}

\begin{proof}

There is only a finite number of singularities and separatrices of $f$, so Proposition \ref{Prop: Finite number first intersection points} is a direct consequence of the following lemma.

\begin{lemm}\label{lemm: finite intersect in fundamental domain}
		Let $p$ and $q$ be singularities of $f$, and let $F^s(p)$ and $F^u(q)$ be stable and unstable separatrices of these points. There exists a finite number of orbits of first intersection points contained in $F^s(p)\cap F^u(q)$
	\end{lemm}
 
 \begin{proof}

Let $n$ be the smallest natural number such that $g:=f^n$ fixes the separatrices $F^s(p)$ and $F^u(q)$. Take a point $z_0 \in F^s(p)\cap F^u(q)$, which may or may not be a first intersection point of $f$. The interval $(g(z_0),z_0]^s \subset F^s(p)$ serves as a fundamental domain for $g$ and contains a fundamental domain for $f$. To prove the lemma, we need to show that within this interval there exists a finite number of first intersection points.

Consider a first intersection point $z$ in $F^s(p) \cap F^u(q)$. There exists an integer $k \in \mathbb{Z}$ such that $g^k(z)\in (g(z_0),z_0]^s \cap F^u(q)$. Therefore, there are two possible configurations:

 	\begin{enumerate}
\item Either $g^k(z)\in (g(z_0),z_0]^s \cap (q,g(z_0)]^u$, or
\item $g^k(z)\in (g(z_0),z_0]^s \cap (g(z_0),\infty))^u$.
 	\end{enumerate}

  We claim that option $(2)$ is not possible. If such a configuration were to exist, it would imply that $g(z_0) \in (p,g^k(z)]^s \cap [q,g^k(z)]^u$, and thus $g^k(z)$ would not be a first intersection point. However, Lemma \ref{Lemm: Image of first intersection is first intersection} implies that $f^{kn}(z)$ is a first intersection point, leading to a contradiction.
  
Hence, we deduce that $g^k(z) \in (g(z_0),z_0]^s \cap [q,g(z_0)]^u \subset [g(z_0),z_0]^s \cap [q,g(z_0)]^u$, which is the intersection of two compact intervals. Therefore, this intersection is a finite set.
 	
This implies that every first intersection point $z\in F^s(p) \cap F^u(q)$ has some iteration lying in a finite set. Consequently, there are only a finite number of orbits of first intersection points in $F^s(p)\cap F^u(q)$.

 \end{proof}
 
\end{proof}

\subsection{Primitive Markov partitions.} Applying the recipe for Markov partitions described in Proposition \ref{Prop: Recipe for Markov partitions} to the first intersection points of $f$, we obtain a distinguished family of Markov partitions.
 
 \begin{defi}\label{Defi: Primitive Markov partitions}
Let $z$ be a first intersection point of $f$. For all $n \geq n(z)$, we refer to the Markov partition $\mathcal{R}(z,n)$ constructed using the method described in Proposition 
\ref{Prop: Recipe for Markov partitions} as the  \emph{primitive Markov partition} generated by $z$ of order $n$.
 \end{defi}
 
Next, we establish some properties of the first intersection points of $f$ and the primitive Markov partitions that are preserved under topological conjugation. It is important to emphasize that, as shown in Proposition \ref{Prop:  conjugation produces pA}, conjugating a pseudo-Anosov homeomorphism by any homeomorphism results in another pseudo-Anosov homeomorphism.

\begin{theo}\label{Theo: Conjugates then primitive Markov partition}
Let $f: S_f \rightarrow S_f$ and $g: S_g \rightarrow S_g$ be two  pseudo-Anosov homeomorphisms that are conjugated by a homeomorphism $h: S_f \rightarrow S_g$, i.e., $g = h \circ f \circ h^{-1}$. Let $z \in S_f$ be a first intersection point of $f$. The following propositions are true:
	
	\begin{itemize}
	\item[i)] $h(z)$ is a first intersection point of $z$
	\item[ii)] $h(\cJ^u(z))=\cJ^u(h(z))$
	\item[iii)] $h(\delta^s(z))=\delta^s(h(z))$ and $h(\delta^u(z))=\delta^u(h(z))$
	\item[iv)] The graph $\delta^s(z)$ is compatible with $f^n(\delta^u(z))$ if and only if the graph $\delta^s(h(z))$ is compatible with $g^n(\delta^u(z))$.
	\item[v)] The compatibility coefficient of $z$ and $h(z)$ are equal, i.e. $n(z)=n(h(z))$.
	\item[vi)] $r$ is a equivalent class of $u$-rails for $\delta^s(z)$ if and only if $h(r)$ is a equivalent class of $u$-rails for $\delta^s(h(z))$. Similarly for equivalent classes of $s$-rail of $f^n(\delta^u(z))$ and $g^n(\delta^u(h(z)))$.
	\item[vii)] For all $n\geq n(z)=n(h(z))$,  $h(\cR(p,z))=\cR(p,h(z))$ is a primitive Markov partition of $g$ generated by $h(z)$ of order $n$.
	\end{itemize}
\end{theo}

\begin{proof}
	First, $h$ send transverse foliations of $f$ into transverse foliations of $g$ and singularities of $f$ into singularities of $g$, that is how we use the hypothesis that $f$ and $g$ are pseudo-Anosov homeomorphisms. But even more, $h$ sends disjoint segments of stable and unstable leaves of $f$ to disjoint segments of stable and unstable leaves of $g$, therefore, if $z$ is first intersection point of $f$, $h(z)$ is first intersection point of $g$. This shows the accuracy of Item $i)$.
	
 A direct computation proof Item $ii)$:
	$$
	h(\cJ^u(z))=h(\cup f^{-n}(J^u(z)))=\cup g^{-n}(J^u(h(z)))= \cJ^u(h(z)).
	$$

	A compact interval $[p,x]^s \in \delta^s(z)$ is characterized by having interior disjoint of $\cJ^u(z)$ and endpoint $x\in \cJ^u(z)$. Then $h([p,x]^s)$ is a segment with disjoint interior of $h(\cJ^u(z))=\cJ^u(h(z))$ but endpoint $h(x)\in \cJ^u(h(z))$, this implies $h([p,x]^s)$ is contained in $\delta^s(h(z))$, hence $h(\delta^s(z))\subset \delta^s(h(z))$. The symmetric argument starting from the segment  $[p',x']^s\in \delta^s(h(z))$ and using $h^{-1}$ gives the equality. Analogously it is possible to show that $\delta^u(h(z))=h(\delta^u(z))$. The arguments in this paragraph proved the Item $iii)$.	
	
	Clearly the point $e$ is an extreme point of $\delta^s(z)$ if and only if $h(e)$ is an extreme point of $h(\delta^s(z))=\delta^s(h(z))$, moreover $e\in f^n(\delta^n(z))$ if and only if $h(e)\in g^{n}(\delta^u(h(z)))$. Therefore an extreme point $e$ of $\delta^s(z)$ is contained in $f^n(\delta^n(z))$ if and only if the extreme point $h(e)$ of $\delta^s(h(z))$ is contained in $g^{n}(\delta^u(h(z))$.	Another consequence is that the regular part of $\delta(z)$ is sent by $h$ to the regular part of $\delta(z)$ and an $u$-regular $J$-rail of $\delta(z)$ is such that $h(J)$ is a regular rail of $\delta^s(h(z))$.

	We claim that $J$ is an extremal rail of $\delta^s(z)$ if and only if $h(J)$ is an extremal lane of $\delta^u(h(z))$. Indeed if $R$ is the embedded rectangle of $S_f$ given by the definition of extreme rail having $J$ as a vertical boundary  component, therefore the embedded rectangle $h(R)$ satisfies the definition for $h(J)$ to be an extreme rail of $\delta^s(h(z))$, since  except for $h(J)$ all its other vertical segments are regular rails of $\delta^s(h(z))$ and it stable boundary is contained in $h(\delta^s(z))=\delta^s(h(z))$. Moreover, $J \subset f^n(\delta^u(z))$ if and only if $h(J)\subset g^n(\delta^u(h(z)))$.
		
		This proves that $\delta^s(z)$ is compatible with  $f^n(\delta^u(z))$ if and only if $\delta^s(h(z))$ is compatible with $g^n(\delta^u(h(z)))$. In this way the point $iv)$ is corroborated, but at the same time we deduce that $n(z)=n(h(z))$ because they are the minimum of the same set of natural numbers and thus the point $v)$ is obtained.
	
	Let $r$ be an equivalent class of $u$-rails for $\delta^s(z)$, as stated before $I$ is a regular rail for $\delta^s(z)$ if and only if $h(I)$ is a regular rail for $\delta^s(h(z))$. 	We denote $I\equiv_{\delta^s(z)}I'$ to indicate that any point in $I$ is $\delta^s(z)$ equivalent to any point in $I'$. In this manner $h(I)\equiv_{\delta^s(h(z))}h(I')$, because the image by $h$ of rectangles and regular rails realizing the equivalence between points in $I$ and points in $I'$ are rectangles and regular rails realizing the equivalence between points in $h(I)$ and points in ´$h(I')$. This implies that $h(r)$ is an equivalent class of $u$-rails for $\delta^s(h(z))$. Analogously the image by $h$ of an equivalent class $r'$ of $s$-rails for $f^n(\delta^u(z)$ is an equivalent class of $s$-rails for $g^n(\delta^u(h(z)))$.

	As the interior of a rectangle $R$ of $R(z,p)$ is a connected component of the intersection a equivalent class $r$ of $u$-rails of $\delta^s(z)$ with  a class $r'$ of $s$-rails for $f^n(\delta^u(z))$, then $h(R)$ is a connected component of the intersection of $h(r)\cap h(r')$ and $h(r)$ and $h(r')$ are $s,u$-rail classes for $\delta^s(h(z))$ and $g^n(\delta^u(h(z)))$, by definition $h(\overset{o}{R})$ correspond to the interior of a rectangle of $\cR(h(z),n)$ and then $h(\cR(z,n))=\cR(h(z),n)$. In this manner we obtain Item $vi)$

	Since the interior of a rectangle  $R$ of $R(z,p)$ is a connected component of the intersection an $r$ equivalent class of $u$-rails of $\delta^s(z)$ with an $r'$ class of $s$-rails for $f^n(\delta^u(z))$, then $h(R)$ is a connected component of the intersection of $h(r)\cap h(r')$, anyway $h(r)$ and $h(r')$ are $s,u$-rail classes for $\delta^s(h(z))$ and $g^n(\delta^u(h(z)))$, therefore $h(\overset{o}{R})$ correspond to the interior of a rectangle of $\cR(h(z),n)$ and then $h(\cR(z,p))=\cR(h(z),p)$. This probe the Item $vii)$ and finish our proof.
\end{proof}

In particular, if $f = g = h$, we have $f = f \circ f \circ f^{-1}$. We deduce that $n(z) = n(f(z))$ for every first intersection point $z$ of $f$. Therefore, the quantity $n(f^m(z))$ is constant over the entire orbit of $z$. In view of Proposition \ref{Prop: Finite number first intersection points}, there are only a finite number of orbits of first intersection points, and the following corollary follows immediately.

\begin{coro}\label{Coro: n(f) number}
The number
$$
n(f)=\max\{n(z): z \text{ is a first intersection point of } f \}
$$
is finite. We refer to $n(f)$ as the \emph{compatibility order} of $f$.
\end{coro}

\begin{coro}\label{Coro: n(f) conjugacy invariant}
Let $f:S_f\rightarrow S_f$ and $g:S_g\rightarrow S_g$ be two pseudo-Anosov homeomorphisms, conjugated by a homeomorphism $h:S_f\rightarrow S_g$, i.e. $g:=h\circ f \circ h^{-1}$, then the compatibility order of $f$ and $g$ is the same, i.e. $n(f)=n(g)$.
\end{coro}

\begin{proof}
We are in the setting of \ref{Theo: Conjugates then primitive Markov partition}, Items $i)$ and $v)$ of that theorem imply the following set equalities:

\begin{eqnarray*}
\{n(z): z \text{ is a first intersection point of } f \}=\\
\{n(h(z)): z \text{ is a first intersection point of } f \}=\\
\{n(z'): z' \text{ is a first intersection point of } g \}.
\end{eqnarray*}
Therefore, it follows that its maximum is the same and finally that  $n(f)=n(g)$.
\end{proof}

\begin{defi}\label{Defi: primitive Markov parition}
		For every $n>n(f)$	the set all of primitive Markov partition of $f$ generated by first intersection points  of order $n$ is denoted $\cM(f,n)$. 
	$$
	\cM(f,n):=\{\cR(z,n): z \text{ first intersection point of } f \}.
	$$
\end{defi}

Another application of Theorem \ref{Theo: Conjugates then primitive Markov partition} in the case $f=g=h$ is that if $z$ is a first intersection point of $f$ and $n\geq n(f)$, then $f(\cR(z,n))=\cR(f(z),n)$. This yields the following corollary:

\begin{coro}\label{Coro: Finite orbits of primitive Markov partitions}		
Let $n\geq n(f)$. Then, there exists a finite but nonempty set of orbits of primitive Markov partitions of order $n$.
\end{coro}

\begin{proof}
As $n\geq n(f)$, there exists at least one Markov partition $\cR(z,n)$ in $\cM(f,n)$, and therefore the orbit of $\cR(z,n)$, given by $\{\cR(f^m(z),n)\}_{m\in \ZZ}$, is contained in $\cM(f,n)$.

Let $\{z_1, \cdots, z_k\}$ be a set of first intersection points of $f$ such that any other first intersection point can be written as $f^m(z_i)$ for a unique $i \in \{1,\cdots, k\}$, and no two different points $z_i$ and $z_j$ are in the same orbit. . Let $\cR(z,n)$ be a primitive Markov partition of order $n$. If $z$ is any first intersection point, $z=f^m(z_i)$ and we have:
$$
f^m(\cR(z_i,n))=\cR(f^m(z_i),n)=\cR(z,n). 
$$
Therefore, in $\cM(f,n)$, there are at most $k$ different orbits of primitive Markov partitions of order $n$.
\end{proof}

\subsection{The geometric type of the orbit of a primitive Markov partition.}

We are interested in studying the set of geometric types produced by the orbit of a primitive Markov partition. Therefore, it is important to understand the behavior of a Markov partition under the action of a homeomorphism that preserve the orientation.
 
 Let  $\cR=\{R_i\}_{i=1}^n$ be a geometric Markov partition of a generalized pseudo-Anosov homeomorphism $f:S\rightarrow S$, and let $g:S\rightarrow S$ be another generalized pseudo-Anosov homeomorphism that is conjugate to $f$ by a homeomorphism $h:S\rightarrow S$ preserving the orientation. The function $h$ maps the foliations and singularities of $f$ to the foliations and singularities of $g$ while preserving the orientation of the transversal foliations.
 
 It is clear that if $r:[0,1]\times [0,1] \rightarrow R \subset S$ is an oriented rectangle for $f$, then $h\circ r: [0,1]\times [0,1] \rightarrow h(R) \subset S$ is an oriented rectangle for $g$. This is because $h$ determines a unique orientation for the transverse foliations of $h(R)$. Furthermore,  $h(\cR)=\{h(R_i)\}_{i=1}^n$ has a $g$-invariant horizontal boundary and a $g^{-1}$-invariant vertical boundary. These properties are deduced from the conjugation provided by $h$. Based on these observations, we can now give the following definition.

  \begin{defi}\label{Defi: induced geometric Markov partition}
Let $f:S\rightarrow S$ be a generalized pseudo-Anosov homeomorphism, and let $h:S\rightarrow S$ be an orientation-preserving homeomorphism, and $g:=h\circ f \circ h^{-1}$. If $\cR$ is a geometric Markov partition of $f$, the \emph{induced geometric Markov partition} of $g$ by $h$ is the Markov partition $h(\mathcal{R})=\{h(R_i)\}_{i=1}^n$. For each $h(R_i)$, we choose the unique orientation in the vertical and horizontal foliations such that $h$ preserves both orientations at the same time.
 \end{defi}

The following lemma clarifies the correspondence between the horizontal and vertical rectangles of the partition $\mathcal{R}$ and those of $h(\mathcal{R})$.
 
\begin{lemm}\label{Lemm: Conjugated partitions same type.}
	Let $f$ and $g$ be two generalized pseudo-Anosov homeomorphisms conjugated through a homeomorphism $h$ that preserves the orientation. Let $\cR=\{R_i\}_{i=1}^n$  be a geometric Markov partition for $f$, and let $h(\cR)=\{h(R_i)\}_{i=1}^n$ be the geometric Markov partition of $g$ induced by $h$. In this situation, $H$ is a horizontal sub-rectangle of $(f,\mathcal{R})$ if and only if $h(H)$ is a horizontal sub-rectangle of $(g,h(\mathcal{R}))$. Similarly, $V$ is a vertical sub-rectangle of $(f,\mathcal{R})$ if and only if $h(V)$ is a vertical sub-rectangle of $(g,h(\mathcal{R}))$.
\end{lemm}

\begin{proof}

Observe that $h(\overset{o}{R_i})=\overset{o}{h(R_i)}$. Therefore, $C$ is a connected component of $\overset{o}{R_i}\cap f^{\pm}(\overset{o}{R_j})$ if and only if $h(C)$ is a connected component of $\overset{o}{h(R_i)}\cap g^{\pm}(\overset{o}{h(R_j)})$.

\end{proof}

Let $T(f,\cR)=(n,\{h_i,v_i\},\Phi:=(\rho,\epsilon))$ be the geometric type of the geometric Markov partition $\cR$, and let $T(g,h(\cR))=(n',\{h'_i,v'_i\},\Phi'_T:=(\rho',\epsilon'))$be the geometric type of the induced Markov partition.  A direct consequence of Lemma  \ref{Lemm: Conjugated partitions same type.} is that: $n=n'$, $h_i=h'_i$ and $v_i=v'_i$.

Let $\{\overline{H^i_j}\}_{j=1}^{h_i}$ be the set of horizontal sub-rectangles of $h(R_i)$, labeled with respect to the induced vertical orientation in $h(R_i)$. Similarly, we define $\{\overline{V^k_l}\}_{l=1}^{v_k}$ as the set of vertical sub-rectangles of $h(R_k)$, labeled with respect to the horizontal orientation induced by $h$ in $h(R_k)$.

 By the we choose the orientations in $h(R_i)$ and in $h(R_k)$ is clear that:
 $$
 h(H^i_j) =\overline{H^i_j} \text{ and } h(V^k_l) =\overline{V^k_l}.
 $$
Even more, using the conjugacy we have that: if $f(H^i_j)=V^k_l$ then
$$
g(\overline{H^i_j})= g(h(H^i_j))=h(f(H^i_j))=h(V^k_l)=\overline{V^k_l}.
$$

This implies that in the geometric types, $\rho = \rho'$.

Suppose that $\overline{V^k_l} = g(\overline{H^i_j})$, so the latter set is equal to $h \circ f \circ h^{-1}(\overline{H^i_j})$. The homeomorphism $h$ preserves the vertical orientations between $R_i$ and $h(R_i)$, as well as between $R_k$ and $h(R_k)$. Therefore, $f$ sends the positive vertical orientation of $H^i_j$ with respect to $R_i$ to the positive vertical orientation of $V^k_l$ with respect to $R_k$ if and only if $g$ sends the positive vertical orientation of $\overline{H^i_j}$ with respect to $h(R_i)$ to the positive vertical orientation of $\overline{V^k_l}$ with respect to $h(R_k)$. It follows then that $\epsilon(i,j) = \epsilon'(i,j)$. We summarize this discussion in the following Theorem

\begin{theo}\label{Theo: Conjugated partitions same types}
Let $f$ and $g$ be generalized pseudo-Anosov homeomorphisms conjugated through a homeomorphism $h$ that preserves the orientation, i.e., $g = h \circ f \circ h^{-1}$. Let $\cR=\{R_i\}_{i=1}^n$ be a geometric Markov partition for $f$, and let $h(\cR)=\{h(R_i)\}_{i=1}^n$ be the geometric Markov partition of $g$ induced by $h$. In this situation, the geometric types of $(g,h(\cR))$ and $(f,\cR)$ are the same.
\end{theo}

Since $f$ preserves the orientation, we can consider the case where $f = g = h$, and apply the previous theorem. If $n \geq n(f)$ and $z$ is a first intersection point of $f$, then for all $m \in \mathbb{Z}$, the geometric type of $f^m(\cR(z,n)) = \cR(f^m(z),n)$ is the same as the geometric type of $\cR(z,n)$. This leads to the following corollary.

\begin{coro}\label{Coro: constant type in orbit}
For every primitive Markov partition $\cR(z,n)$ with $n \geq n(f)$ and $z$ as a first intersection point of $f$, we have the following property for all $m \in \mathbb{Z}$:
$$
T(\cR(z,n))= T(\cR(f^m(z),n)).
$$
\end{coro}

In view of Corollary \ref{Coro: Finite orbits of primitive Markov partitions}, for all $n \geq n(f)$ there exists a finite number of primitive Markov partition orbits of $f$ of order $n$. Moreover, the geometric type is constant within each orbit of such a Markov partition. Therefore, for all $n \geq n(f)$, there exists only a finite number of distinct geometric types corresponding to the geometric types of partitions in the set $\cM(f,n)$. We summarize this in the following theorem.

\begin{theo}\label{Theo: finite geometric types}
For every $n > n(f)$, the set $\cT(f,n) := \{T(\cR) : \cR \in \cM(f,n)\}$ is finite. These geometric types are referred to as the \emph{primitive geometric types} of $f$ of order $n$.
\end{theo}

Of course, within all the primitive geometric types, there is a distinguished family that, in some sense, has the least complexity with respect to its compatibility coefficient.

\begin{defi}\label{defi: Canonical types}

The canonical types of $f$ are the set $\cT(f,n(f))$.
\end{defi}

In fact, in Corollary \ref{Coro: n(f) conjugacy invariant}, we have proved that if $f$ and $g$ are topologically conjugate, then $n(f)=n(g)$. By Theorem  \ref{Theo: Conjugated partitions same types}, we deduce that for every $n\geq n(f)=n(g)$, the set of primitive types of $f$ of order $n$ coincides with the set of primitive geometric types of $g$ of order $n$. In particular, we have the following corollary.

\begin{coro}\label{Coro: conjugated implies same cannonical types}
If $f$ and $g$ are conjugate generalized pseudo-Anosov homeomorphisms, for all $m\geq n(f)=n(g)$, $\cT(f,m)=\cT(g,m)$ and their canonical types are the same, i.e., $\cT(f,n(f))=\cT(g,n(g))$.
\end{coro}

In the next chapter, we will demonstrate that if $f$ and $g$ are pseudo-Anosov homeomorphisms with geometric Markov partitions of the same geometric type, then they are topologically conjugate trough an orientation preserving homeomorphism. Therefore, having the same canonical type is a sufficient condition for topological conjugacy. We are going provided an answer to the finite presentation problem, but, currently, we do not have a finite algorithm to compute $n(f)$ or $\cT(f,n(f))$ starting from an arbitrary Markov partition of $f$ with geometric type $T$. What is possible, and will be shown in the last chapter, is the existence of a finite algorithm that, given a geometric type $T$, computes a number $O(T)\geq n(f)$ and the set $\cT(f,n)$ for $n\geq O(T)$. This allows us to obtain a finite presentation of $f$ although it may not be canonical.

\chapter{Combinatorial representatives of conjugacy classes.}\label{Chapter: TypeConjugacy}

\section*{A complete conjugation invariant.}

The objective of this chapter is to prove the following Theorem \ref{Theo: conjugated iff  markov partition of same type} that addresses Item 1.I in Problem \ref{Prob: Clasification}, regarding the combinatorial representation of the conjugacy class of a generalized pseudo-Anosov homeomorphism.

\begin{theo}\label{Theo: conjugated iff  markov partition of same type}
Let $f:S_f\rightarrow S_f$ and $g:S_g \rightarrow S_g$ be two generalized pseudo-Anosov homeomorphisms. Then, $f$ and $g$ have a geometric Markov partition of the same geometric type if and only if there exists an orientation preserving homeomorphism between the surfaces $h:S_f\rightarrow S_g$ that conjugates them, i.e., $g=h\circ f\circ h^{-1}$.
\end{theo}

We will prove that every generalized pseudo-Anosov homeomorphism $f$ with a geometric Markov partition of geometric type $T$ is conjugate to a unique \emph{combinatorial model} determined by $T$, denoted by $(\Sigma_T, \Sigma_T)$. Before proceeding, let's outline our strategy.

Suppose $f: S \rightarrow S$ is a generalized pseudo-Anosov homeomorphism with a geometric Markov partition of geometric type $T$. In order to prove Theorem \ref{Theo: conjugated iff  markov partition of same type}, our strategy is to construct a \emph{combinatorial model} of $f$ using the sub-shift of finite type associated with the geometric type $T$ and quotient it by an equivalence relation $\sim_T$ determined by $T$. Then, Proposition \ref{Prop: cociente T} establishes that every generalized pseudo-Anosov homeomorphism with a Markov partition of geometric type $T$ is topologically conjugate to such a combinatorial model. This implies that any two pseudo-Anosov homeomorphisms with the same geometric type are topologically conjugate. After that we need to prove that the homeomorphism that conjugated the pseudo-Anosov is orientation preserving.

If we let $A := A(T)$ be the incidence matrix of $T$, the construction of the combinatorial model involves taking the quotient of the sub-shift of finite type $(\Sigma_A, \sigma)$ by an equivalence relation $\sim_T$. This results in a space denoted by $\Sigma_T := \Sigma_A/\sim_T$, where the shift map $\sigma$ induces a generalized pseudo-Anosov homeomorphism $\sigma_T$. The equivalence relation $\sim_T$ is directly related to every pair $(f,\mathcal{R})$ that realizes $T$ through the concept of \emph{projection} $\pi_f: \Sigma_A \rightarrow S$. The projection $\pi_f$ is a continuous function that semi-conjugates $\sigma$ with $f$. The explicit relationship between the equivalence relation $\sim_T$ and the projections is established in Proposition \ref{Prop: The relation determines projections}, which serves as the main technical result in this chapter.

\begin{prop}\label{Prop: The relation determines projections}
Let $T$ be a geometric type in the pseudo-Anosov class with a binary incidence matrix $A = A(T)$. Consider the sub-shift of finite type $(\Sigma_A, \sigma)$ associated with $T$. We can define an equivalence relation $\sim_T$ on $\Sigma_A$ algorithmically in terms of $T$.

The equivalence relation $\sim_T$ has the following property: If $(f, \mathcal{R})$ is a pair that realizes the geometric type $T$, and $\pi_f: \Sigma_A \rightarrow S$ is the projection induced by $(f, \mathcal{R})$, then for any pair of codes $\underline{w}, \underline{v} \in \Sigma_A$, they are $\sim_T$-related if and only if their projections coincide, i.e., $\pi_f(\underline{w}) = \pi_f(\underline{v})$.
	
\end{prop}

First, we will introduce the concepts of symbolic dynamical systems that form the basis of our construction. Then, we will proceed to define the equivalence relation step by step. The main result of this chapter is that the geometric type serves as a complete conjugacy invariant for pseudo-Anosov homeomorphisms. As a consequence, we obtain Corollary\ref{Coro: finite canonical types}, which provides a collection of finite invariants for the conjugacy class.

\begin{coro}\label{Coro: finite canonical types}
	Two generalized pseudo-Anosov homeomorphisms $f$ and $g$ are topologically conjugate trough and orientation preserving homeomorphism if and only if there exists an integer $n\geq \max\{n(f),n(g)\}$ such that $\cT(f,n) = \cT(g,n)$. 
\end{coro}

 In particular, they are topologically conjugate if and only if their canonical geometric types are the same, i.e., $\cT(f,n(f)) = \cT(g,n(g))$.

\section{Symbolic dynamical systems induced by a geometric type.}\label{Sub: symbolic dynamics of a Type}

In this exposition, we introduce the properties and definitions of \emph{symbolic dynamical systems} from which we will develop our combinatorial models. We formulate and prove some classical results in terms of geometric type to provide a unified approach and highlight the challenge of associating a symbolic dynamical system to a homeomorphism.
 
 \subsection{The incidence matrix of a geometric Type}\label{Subsec: Incidence matrix}

 	Let $f$ be a generalized pseudo-Anosov homeomorphism and $\mathcal{R}$ be a geometric Markov partition of $f$ with geometric type $T$. The incidence matrix of $(f, \mathcal{R})$ is the matrix $A := A(f,\mathcal{R}) = (a_{i,j})$ with coefficients given by:
 	
 $$
 a_{i,j}=\# \{c.c. \text{ of } \overset{o}{R_i} \cap f^{-1}(\overset{o}{R_j})\}.
 $$
 Since our original information is a geometric type, we explain how to obtain the incidence matrix of a geometric Markov partition in terms of the geometric type of the partition. The coefficient $a_{i,k}$ can be obtained by counting how many horizontal sub-rectangles of $R_i$ are mapped to vertical sub-rectangles of $R_k$, i.e., by counting all $j \in \{1, \ldots, h_i\}$ such that $\Phi(i,j) = (k,l,\pm 1)$.  This gives rise to the formula:
  
$$
a_{i,k}=\#\{j \in \{1,\cdots,h_i\}: \Phi(i,j)=(k,l,\epsilon(i,j))\}.
$$

The incidence matrix contains less information than the geometric type, and there can be multiple geometric types with the same incidence matrix. However, this formula suggests a general definition.

\begin{defi}\label{Defi: Incidence matriw of a type}

Let $T=\{n,\{(h_i,v_i)\}_{i=1}^n, \Phi_T\}$ be an abstract geometric type. The  \emph{incidence matrix of the geometric type } $T$ is denoted by $A(T)$ and defined as the matrix $A(T) = (a_{ik})$ with coefficients given by:

$$
a_{i,k}=\#\{j \in \{1,\cdots,h_i\}: \Phi(i,j)=(k,l,\epsilon(i,j))\}.
$$
\end{defi}

If $T$ belongs to the pseudo-Anosov class and the pair $(\mathcal{R}, f)$ realizes $T$, i.e., $T = T(f, \mathcal{R})$, we can conveniently write $A(T) = A(T(f, \mathcal{R})) = A(\mathcal{R})$ to refer to the same object, depending on the level of precision required. Now, let us pose an interesting question.

\begin{ques}\label{Ques: Same incidence non cojugated}
Give an example of two non-conjugate generalized pseudo-Anosov homeomorphisms $f$ and $g$ with geometric Markov partitions whose geometric types $T_f$ and $T_g$ have the same incidence matrices, i.e., $A(T_f) = A(T_g)$. How are the genera of $S_f$ and $S_g$ related? Are they equal?
\end{ques}

\begin{defi}\label{Defi: Mixing and binary}
A matrix $A$ with non-negative coefficients is called \emph{mixing} if there exists an $n$ such that $A^n=(a_{i,j}^{(n)})$ is positive definite, i.e., for all $i$ and $j$, $a^{(n)}_{i,j}>0$.

A matrix $A$ whose coefficients are either $0$ or $1$ is called a \emph{binary} matrix.
\end{defi}

The coefficients of the incidence matrix of a Markov partition $\cR$ can take any value in $\mathbb{N}$. For example, if the image of the rectangle $R_i$ intersects the interior of $R_j$ twice, then $a_{i,j}=2$. This introduces complications in our arguments. However, there is a way to obtain a Markov partition $H(\cR)$ such that the image of each rectangle in $H(\cR)$ intersects another rectangle in $H(\cR)$ at most once. This means that the incidence matrix of $H(\cR)$ has coefficients of either $0$ or $1$, i.e., it is binary. Subsection  \ref{Subsec: The horizontal refinament} addresses this refinement, where we discuss the difficulties that will arise in Subsection \ref{Subsec: finite type shift Markov partition} when we consider the sub-shift of finite type associated with a Markov partition and its projections.

 The following result is proved in \cite[Lemma 10.21]{fathi2021thurston}

\begin{lemm}\label{Lemm: incidence matriz pA is mixing}

If $A$ is the incidence matrix of a Markov partition of a pseudo-Anosov homeomorphism, then $A$ is mixing.
\end{lemm}

\begin{rem}\label{Rema: mixing after positive stay positive}
	The term $a^n_{i,j}$ counts how many times $f^n(\overset{o}{R_i})$ intersects $\overset{o}{R_j}$. The fact that $A(\mathcal{R})$ is mixing is equivalent to the existence of an $n \in \mathbb{N}$ such that for all $i, j$, $f^n(\overset{o}{R_i}) \cap \overset{o}{R_j} \neq \emptyset$. In fact, if for some $n\in \NN_{+}$ and every coefficient in the matrix $A^n$ is positive, i.e $a^n_{i,j} > 0$, then for all $m > n$, $a^m_{i,j} > 0$ as well.	
\end{rem}

\subsection{The horizontal refinement of a geometric type.}\label{Subsec: The horizontal refinament}

  The partition $H(f,\mathcal{R})$ that we define next contains all the horizontal sub-rectangles of the original Markov partition $(f,\mathcal{R})$. We explain how to make $H(f,\mathcal{R})$ a geometric partition of $f$ by using the orientation induced by $\mathcal{R}$ and labeling the sub-rectangles according to the \emph{lexicographic order}. We then prove that $H(f,\mathcal{R})$ is a geometric Markov partition of $f$ (Lemma \ref{Lemm: Algoritmic refinament}) called the \emph{horizontal refinement} of $(f,\cR)$, whose geometric type is uniquely determined by $T$ (Proposition  \ref{Prop: Unique horizontal type}).

   \begin{defi}\label{Defi: Horizontal refinament}
      Let $T:=(n,\{(h_i,v_i)\}_{i=1}^n,\Phi_T)$ be a geometric type in the pseudo-Anosov class. Let $(f, \mathcal{R})$ be a pair that realizes $T$, where $\cR=\{R_i\}_{i=1}^n$. Consider the following construction:
     
  \begin{itemize}
\item Let $H(f,\cR):=\{H^i_j\}_{(i,j)\in \cH(T)}$ be the family of horizontal sub-rectangles of $(f,\cR)$.

\item Endow each rectangle $H^i_j$ with the orientation induced by $R_i$.

\item Endow the set $\cH(T):=\{(i,j): i \in\{1,\cdots,n\}  \text{ and } j\in\{1,\cdots,h_i \} \}$ with the \emph{lexicographic order}  using the function  $r:\cH(T)\rightarrow \{1,\cdots, \sum_{i=1}^{n}h_i=\alpha(T)\}$ defined by the formula:
$$
r(i_0,j_0)=\sum_{i<i_0} h_i + j_0
$$
\item  The rectangles of $H(f,\cR)$ are indexed by the formula $H_{r(i,j)}:=H^i_j$,which we can synthesize as  $H(\cR)=\{H_r\}_{r=1}^{\alpha(T)}$.
  \end{itemize}
The family of oriented rectangles constructed in this way, $H(f,\cR)$, is called the \emph{horizontal refinement} of $(f,\mathcal{R})$.
   \end{defi}
     
   The horizontal refinement $H(f,\cR)$ is shown to be a Markov partition of $f$ in Lemma \ref{Lemm: Algoritmic refinament}. To gain some intuition about the proof, we suggest the reader to refer to Figure \ref{Fig: Horizontal refinament}.

\begin{lemm}\label{Lemm: Algoritmic refinament}
  Let $T$ be a geometric type in the pseudo-Anosov class, and $(f, \cR)$ be a pair realizing $T$. The horizontal refinement $H(f,\cR)$ is a geometric Markov partition of the generalized pseudo-Anosov homeomorphism $f$. 
\end{lemm}

\begin{proof}

Suppose $f:S\rightarrow S$ and the geometric Markov  partition is $\cR={R_i}_{i=1}^n$. The geometric type of $(f,\cR)$ is denoted by $T$ and is given by:
$$
T(f,\cR):=T=(n,\{(h_i,v_i)\}_{i=1}^n, \Phi_T=(\rho_T,\epsilon_T)).
$$

The family of horizontal sub-rectangles of $(f,\cR)$, $\{ H^i_j: (i,j)\in \cH \}$, always covers the surface $S$ and has disjoint interiors. Thus, $H(f,\cR)$ is a partition of $S$ into rectangles. Let's analyze its boundary.

The unstable boundary of $H(f,\cR)$, denoted by  $\partial^u H(f,\cR)=\cup_{(i,j)\in \cH} \partial^u H^i_j$, coincides with the unstable boundary of $\cR$, making it $f^{-1}$-invariant.

The stable boundary of $H(f,\cR)$ consists of the stable boundaries of the rectangles $H^i_j$, denoted by $\partial^s H(f,\cR)=\cup_{(i,j)\in \cH} \partial^s H^i_j$. Since $f(H^i_j)=V^k_l$ and $V^k_l$ is a vertical sub-rectangle of $R_k\in \cR$, the horizontal boundary of $V^k_l$ is contained in $\partial^s R_k$. This implies that $f(\partial^s H^i_j) \subset \partial^s H^k_{1}\cup \partial^sH^k_{h_k}$. Thus, we deduce that the stable boundary of $H(f,\cR)$ is $f$-invariant.
  
 Since $H(f,\cR)$ satisfies the properties of Proposition  \ref{Prop: Markov criterion boundary}, it is a Markov partition of $f$. It is clear that the conventions regarding the order and orientations of the rectangles given in Definition \ref{Defi: Horizontal refinament} of the horizontal refinement make $H(f,\cR)$ a geometric Markov partition of $f$. 
  
\end{proof}

\begin{figure}[h]
	\centering
	\includegraphics[width=0.6\textwidth]{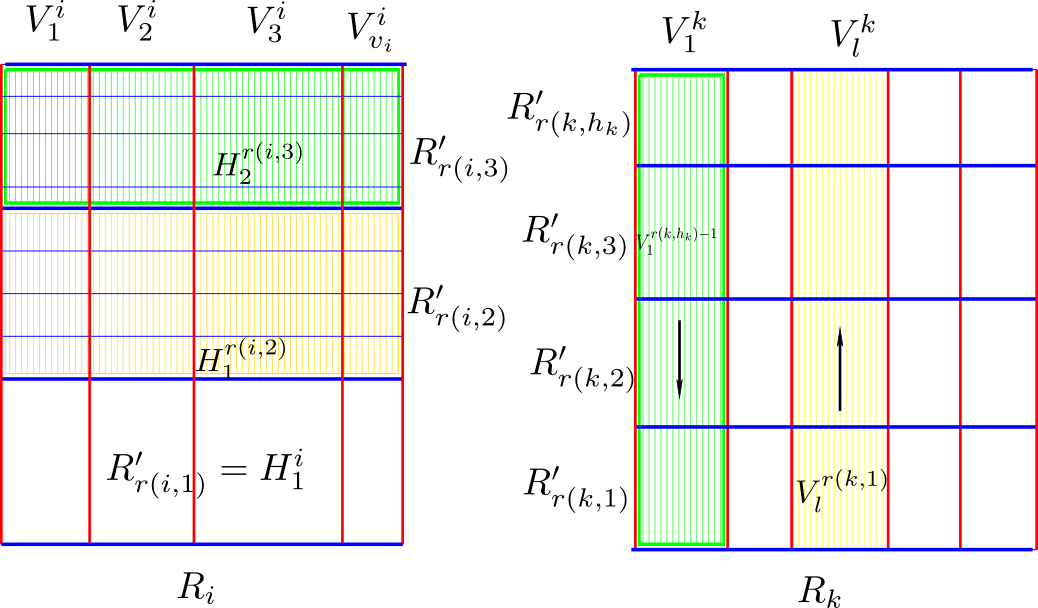}
	\caption{Horizontal refinement}
	\label{Fig: Horizontal refinament}
\end{figure}

 \begin{defi}\label{Defi: geometric type of a refinement}
Let $T$ be a geometric type in the pseudo-Anosov class, and let $(f,\cR)$ be any pair that realizes $T$. The geometric type of the horizontal refinement $H(f,\cR)$ is denoted by $T_H(f,\cR):=T(H(f,\cR))$.
 \end{defi}

The horizontal refinement was defined using a pair, $(f,\cR)$, that realizes $T$. We would like to clarify that if $(g,\cR_g)$ is another realization of $T$, its respective horizontal refinement has the same geometric type, $T_H(f,\cR)=T_H(g,\cR_g)$, and thus they define a unique geometric type $H(T)$ that is canonically associated with $T$.
	
Proposition \ref{Prop: Unique horizontal type} establishes that $T_H(f,\cR)=T_H(g,\cR_g)$, and indeed $T_H(f,\cR)$ is expressed in terms of $T$. This is the definition we use for the horizontal refinement of a geometric type $H(T)$. It also states that the geometric type of the horizontal refinement has the desired property of having an incidence matrix with coefficients in $\{0,1\}$.

\begin{prop}\label{Prop: Unique horizontal type}
	Let $T$ be a geometric type in the pseudo-Anosov class, and let $(f,\cR_f)$ and $(g,\cR_g)$ be two pairs  that realize $T$. We can deduce the following:
	
	\begin{itemize}
\item[i)] The  incidence matrix of $H(f,\cR)$  is binary.
\item[ii)] The geometric type of $H(f,\cR_f)$ is uniquely determined by an algorithm and formulas in terms of $T$.
\item[iii)] The  horizontal refinements of $H(f,\cR_f)$ and  $H(g,\cR_g)$  have the same geometric type, denoted as $H(T):=T(H(f,\cR_f ))=T(H(g,\cR_g))$.
	\end{itemize}
We refer to $H(T)$ as the \emph{horizontal type} of $T$.
\end{prop}

\begin{proof}

For every $(i,j),(k,j')\in \cH(T)$ the number of connected components in

For every $(i,j), (k,j') \in \cH(T)$, the number of connected components in the intersection $f(H_{r(i,j)}) \cap H_{r(k,j')}$ determines the coefficient $a_{r(i,j)r(k,j')}$ of the incidence matrix $A := A(H(f,R))$. Please consider the following equality between sets:

\begin{equation}\label{Equ: incidence matrice}
f(H_{r(i,j)}) \cap H_{r(k,j')}= f(\overset{o}{ H^i_j})\cap \overset{o}{H^k_{j'}}  =
\overset{o}{ V^{k}_{l} }\cap \overset{o}{H^k_{j'}}
\end{equation}

By definition and the left side of Equation  \ref{Equ: incidence matrice}, the coefficient $a_{r(i,j)r(k,j')}$ in $A$ is equal to the number of connected components that result from the intersection of the interiors of a horizontal sub-rectangle and a vertical sub-rectangle of $R_k$.

With this interpretation in mind, if $k \neq k'$, the interiors of the rectangles $R_k$ and $R_{k'}$ are disjoint, and thus $a_{r(i,j),r(k,j')} = 0$. Since the interior of $R_k$ is an embedded (open) rectangle, if $k = k'$, the intersection of the interior of a horizontal sub-rectangle and a vertical sub-rectangle of $R_k$ has only one connected component. Therefore, $a_{r(i,j)r(k,j')} = 1$. This proves \textbf{Item} $i)$.

Let's deduce the geometric type of $H(f,\cR)$ using only the information from $T$, the chosen orientation, and the lexicographic order. Let's denote the resulting geometric type as 
$$
T_H(f,\cR)=\{n',\{(h_i',v_i')\}_{i'=1}^{n'}, \Phi_{T'}=(\rho',\epsilon')\}.
$$
We need to determine all the parameters. The number of elements in $\cH(T)$ determines $n'$, which can be calculated using the formula:

\begin{equation}\label{Equ: number rectangles in horizontal ref}
n'=\sum_{i=1}^{n}h_i=\sum_{i=1}^n v_i=\alpha(T).
\end{equation}

Just remember that the number of elements in $\cH(T)$ corresponds to the number of horizontal sub-rectangles in $(f,\cR)$.
 
 Let $(i,j)\in \cH(T)$ and suppose $\Phi_T(i,j)=((k,l), \epsilon_T(i,j))$ for the rest of the proof. The pair $(f,H(f,\cR))$ denotes the horizontal refinement $H(f,\cR)$ viewed as a Markov partition of $f$.
  
 Based on our previous analysis, the vertical sub-rectangles of the rectangles $H^{r(i,j)}$ in $(f,H(f,\cR))$ correspond to the connected components of the intersection  of the interiors of all the vertical sub-rectangles $V^i_l$ of $R_i$ with $H^i_j$. The incidence matrix determines that all vertical sub-rectangles of $R_i$ intersect every horizontal sub-rectangle of $R_i$ exactly once. Therefore, the number of vertical sub-rectangles of $H_{r(i,j)}$ is constant with respect to $i$. We can infer that the number of vertical sub-rectangles of $H_{r(i,j)}$ is equal to $v_i$, i.e.
 
 \begin{equation}\label{Equ: number vertical sub in the horizontal ref}
 v'_{r(i,j)}=v_{i}
 \end{equation}

Moreover, these vertical sub-rectangles are ordered from left to right in a coherent way with the horizontal orientation of $R_i$ by 

 \begin{equation}\label{Equ: Vertical order of horizontal ype}
 \{V^{r(i,j)}_{l}\}_{l=1}^{v_i}.
 \end{equation}

The number of horizontal sub-rectangles of $H_{r(i,j)}$, denoted as $h'_{r(i,j)}$, is equal to $h_{k}$ because  $f(H_{r(i, j)})=f(H^{i}_{j})=V^{k}_{l}$  intersects exactly $h_{k}$ distinct horizontal sub-rectangles of $R_k$ and no other horizontal sub-rectangles of $(f,\cR)$. These horizontal sub-rectangles are ordered in increasing order according to the vertical orientation of $H_{r(i,j)}$, which is inherited from $R_i$, in the following manner:

\begin{equation}\label{Equa: Horizontal rec of the horizontal ref}
\{H^{r(i,j)}_{j'}\}_{j'=1}^{h_k}.
\end{equation}

We proceed to determine $\rho'$ and $\epsilon'$. Recall that we assumed $\Phi_T(i,j)=((k,l), \epsilon_T(i,j))$. To establish $\rho'$, we need to consider the change of orientation in $f(H^k_j)$ given by $\epsilon_T(i,j)$. We will split the calculation based on the two cases.

First, assume $\epsilon_T(i,j)=1$ and $j_0\in \{1,\cdots, h_{k}\}$, representing the horizontal sub-rectangle $H^{r(i,j)}_{j_0}$ of $H_{r(i,j)}$. Since $f(H_{r(i,j)})=V^k_l$ preserves the vertical orientation, the horizontal sub-rectangle of $R_k$ that intersects $f(H_{r(i,j)})=V^k_l$ at position $j_0$ with respect to the vertical orientation of $R_k$ is labeled by $r(k,j_0)$. Also, $f(H^{r(i,j)}_{j_0})$ corresponds to the vertical sub-rectangle of $R_k$ that is at position $l$, so $f(H^{r(i,j)}_{j_0})$ is the vertical sub-rectangle $V^{r(k,j_0)}_l$. We can express this construction in terms of the geometric type using the formula:

\begin{equation}\label{Equa: horizontal type for preseving orientation }
\Phi_{T'}(r(i,j),j_0):=(\rho'(r(i,j),j_0),\epsilon'(r(i,j),j_0)=(r(k,j_0),l,\epsilon_T(i,j)).
\end{equation}

In the case of $\epsilon_T(i,j)=-1$, $f$ changes the vertical orientation of $H^i_j$ with respect to $V^k_l=f(H^i_j)$. This implies that the horizontal sub-rectangle of $R_k$ containing the image of $H^{r(i,j)}_{j_0}$ is located at position $j_0$ with the inverse vertical orientation of $R_k$, which corresponds to position $(h_k-(j_0-1))$ with the prescribed orientation in $R_k$. Therefore, the horizontal sub-rectangle of $R_k$ corresponding to $f(H^{r(i,j)}_{j_0})$ is given by:

$$
H^k_{h_k-(j_0-1)}=H_{r(k,h_k-(j_0-1))}\in H(f,\cR)
$$

The vertical sub-rectangle of $(f,H_{r(k,h_k-(j_0-1))})$ that contains $f(H^{r(i,j)}_{j_0})$ is a subset of $V^k_l$, so it is located at position $l$ with respect to the horizontal orientation of $R_k$. We can conclude that:

$$
f(H^{r(i,j)}_{j_0})=V^{r(k,h{}-(j_0-1))}_l.
$$
 
The change in the vertical direction of the sub-rectangle $H^{r(i,j)}{j_0}$ is quantified by $\epsilon'(r(i,j),j_0)$. However, the change in the vertical direction of $H^{r(i,j)}{j_0}$ is determined by the change in the orientation of $f$ restricted to $H^i_j$. This information is governed by $\epsilon_T(i,j)$ in the geometric type $T$, and we have $\epsilon'(r(i,j),j_0)=\epsilon_T(i,j)$. In terms of $T$, we have the formula

\begin{equation}\label{Equa: horizontal type for change orientation }
\Phi_{T'}(r(i,j),j_0)=(r(k,h_{k}) -(j_0-1),l,\epsilon_T(i^,j_0)).
\end{equation}

	The equations \ref{Equa: horizontal type for preseving orientation } and  \ref{Equa: horizontal type for change orientation } are determined by $T$, and their computation is algorithmic, proving \textbf{Item} $ii)$.
	 
 The horizontal refinements of $H(f,\cR_f)$ and $H(g,\cR_g)$ yield the same equations:  (\ref{Equ: number rectangles in horizontal ref}), 	(\ref{Equ: number vertical sub in the horizontal ref}), (\ref{Equ: Vertical order of horizontal ype}), (\ref{Equa: Horizontal rec of the horizontal ref}), (\ref{Equa: horizontal type for preseving orientation }) and (\ref{Equa: horizontal type for change orientation }). Therefore, they result in the same geometric type, $H(T):=T(H(f,\cR_f))=T(H(g,\cR_g))$. This proves \textbf{Item} $iii)$.

\end{proof}

The following corollary is a direct consequence of Proposition \ref{Prop: Unique horizontal type}. However, its formulation will be useful for future arguments related to the general case of the Theorem \ref{Theo: conjugated iff  markov partition of same type}.

\begin{coro}\label{Coro: horizontal type is pseudo-Anosov}
	Let $T$ be a geometric type in the pseudo-Anosov class, the horizontal type of $T$, $H(T)$, is in the pseudo-Anosov class.
\end{coro}

\subsection{The sub-shift of fine type of a Geometric type.}\label{Subsec: finite type shift Markov partition}

Let $\Sigma = \prod_{\mathbb{Z}} \{1, \cdots, n\}$ be the set of bi-infinite sequences in $n$-digits. The elements of $\Sigma$ are called \emph{codes}, denoted by $\underline{w} = (w_z)_{z \in \mathbb{Z}}$, and $\Sigma$ is endowed with the topology induced by the metric:

\begin{equation}\label{Equa: Metric shift}
d_{\Sigma}(\underline{x},\underline{y})= \sum_{z\in\ZZ} \frac{\delta(x_z,y_z)}{2^{\vert z \vert}},
\end{equation}

where $\delta(x_z, y_z) = 0$ if and only if $x_z = y_z$ and $\delta(x_z, y_z) = 1$ if and only if $x_z \neq y_z$. In this manner, if there is a big interval of integers around zero where two codes coincide, their distance is small. Like $\Sigma$ is the product of compact spaces, Tychonoff's theorem implies  that $\Sigma$ is a compact space. The \emph{shift} over $\Sigma$ is the homeomorphism $\sigma: \Sigma \rightarrow \Sigma$ given by the relation:

\begin{equation}
(\sigma(\underline{w}))_{z}= w_{z+1}.
\end{equation}
for all $z\in \ZZ$ and all codes $\underline{w}\in \Sigma$. 

Take a binary matrix $A\in \cM_{n\times n}(\{0,1\})$. Then $A$ permits the definition of a subset of $\Sigma$ given by:

\begin{equation}\label{Equa: Sigma A set}
\Sigma_A:=\{\underline{w}=(w_z)_{z\in\ZZ}\in \Sigma:\, \forall z \in \ZZ, \, (a_{w_z,w_{z+1}})=1 \}.
\end{equation}

Lemma \ref{Lemm: Sigma A is invariant subset} is a well-known result in the theory of symbolic dynamical systems, whose proof  only use that $\Sigma_A$ is a closed sub-set of the compact space $\Sigma$.

\begin{lemm}\label{Lemm: Sigma A is invariant subset}
The set $\Sigma_A$ is a compact subset of $\Sigma$ and is invariant under the shift map $\sigma$.
\end{lemm}

\begin{defi}\label{Defi: Finite type sub-shift}
 The \emph{sub-shift of finite type} associated with the matrix $A\in \cM_{n\times n}(\{0,1\})$ is the dynamical system $(\Sigma_A, \sigma)$, where $\Sigma_A$ is the set defined earlier and $\sigma$ is the shift map.
\end{defi}

\begin{defi} \label{Defi: Subshift of type for a Markov partition}
	Let $T$ be a geometric type in the pseudo-Anosov class, and suppose its incidence matrix $A$ is binary. The sub-shift of finite type associated with $T$ is $(\Sigma_A, \sigma)$, 
\end{defi}

Any geometric type $T$ in the pseudo-Anosov class has an associated sub-shift of finite type through the incidence matrix of its horizontal refinement, $H(T)$. This sub-shift of finite type can be seen as an example of a symbolic dynamical system induced by a geometric type.

\subsection{The projection $\pi_f$}\label{Subsec: projection is funtion}

Following the intuition, the \emph{projection} $\pi_f:\Sigma_A \rightarrow S$ that we are going to define is a natural way to relate $(\Sigma_A,\sigma)$ with $(f,\cR)$ as dynamical systems. It recovers the idea of associating every code $\underline{w}=(w_z)$ with a point $x$ in $S$ such that $f^z(x)$ visits the rectangle $R_{w_z}$ at any time $z\in\ZZ$. Then, the $f$-orbit of $x$ is ruled by $\underline{w}$. The definition requires that the incidence matrix has coefficients in ${0,1}$. Therefore, its definition depends strongly on the Markov partition, and this is the first obstruction for $\pi_f$ to be a function.

\begin{defi}\label{Defi: projection pi}

	Let $T$ be a geometric type in the pseudo-Anosov class, with an incidence matrix $A := A(T)$ that is binary, meaning it has coefficients in ${0,1}$. Let $(f,\cR)$ be a realization of $T$. The \emph{projection} with respect to $(f,\cR)$, denoted as $\pi_f: \Sigma_A \rightarrow S$, is defined by the relation:
		\begin{equation}
		\pi_f(\underline{w})=\bigcap_{n\in \NN} \overline{\cap_{z=-n}^{n} f^{-z}(\overset{o}{R_{w_z}})}.
		\end{equation}
	for every $\underline{w}=(w_z)_{z\in \ZZ} \in \Sigma_A$.	
\end{defi}
 
 \begin{rema}\label{Rema: other projection}
 There exists another way to define the projection $\pi_f$, which is more classical and is given by the relation:

$$
\phi_f(\underline{w})=\cap_{z\in \ZZ} f^{-z}(R_{w_z}).
$$

The problem with this definition is that,  if the rectangles of $\cR$ are not embedded, the projection $\phi_f$ defined in this manner is not necessarily a single valued map since the set $\cap_{z\in\ZZ} f^{-z}(R_{w_z})$ could contain more than one point.
 \end{rema}

It could useful to observe that if $\underline{w}\in\Sigma_{A}$ then, for all $z_1,z_2\in \ZZ$  the intersection of the rectangles, $f^{-z_1}(R_{w_{Z_1}})\cap f^{-z_2}(R_{w_{z_2}})$ is a rectangle and the intersection $f^{-z_1}(\overset{o}{R_{w_{Z_1}}})\cap f^{-z_2}(\overset{o}{R_{w_{z_2}}})$ the interior of a rectangle. We are going to use this observation with regularity in the following text.

Proposition \ref{Prop:proyecion semiconjugacion} is a classic result in the theory. It was proven in \cite[Lemma 10.16]{fathi2021thurston}, where $\pi_f$ is defined as $\phi_f(\underline{w}) = \cap_{z\in \ZZ} f^{-z}(R_{w_z})$. Here, we include the proof for the sake of completeness.

\begin{prop}\label{Prop:proyecion semiconjugacion}
If $\cR$ is the Markov partition of a pseudo-Anosov homeomorphism $f:S\rightarrow S$, with a binary incidence matrix $A:=A(\cR)$, then the projection $\pi_f:\Sigma_A \rightarrow S$ is a continuous, surjective, and finite-to-one map\footnote{ What is implicit in this statement is that $\pi_f(\underline{w})$ is a single point in the surface.}. Furthermore, $\pi_f$ semi-conjugates $f$ with $\sigma$, i.e., $f\circ\pi_f=\pi_f \circ \sigma_A$.
\end{prop}

The proof of Proposition \ref{Prop:proyecion semiconjugacion} will be divided into several Lemmas. Let's fix some notation.

Let $R$ be a rectangle. The \emph{unstable length} of $R$, denoted as $L^u(R)$, is defined as the measure $\mu^s(J)$ for any unstable segment $J$ of $R$. Similarly, the \emph{stable length} of $R$, denoted as $L^s(R)$, is defined as the measure $\mu^u(I)$ for any stable segment $I$ of $R$. Finally, the \emph{width} of $R$ is given by:

$$
W(R)=\max\{L^s(R),L^u(R)\}.
$$
 
 If $\cR=\{R_i\}_{i=1}^n$  is a Markov partition, then the \emph{width} of $\cR$ is given by:
 $$
 W(\cR)=\max\{W(R_i): i\in \{1,\cdots,n\}\}.
 $$

\begin{lemm}\label{Lemm: pif is well defined}
If $\cR$ is a Markov partition of $f$ with a binary incidence matrix, then the projection $\pi_f$ is a well-defined map. Is to said that $\pi_f(\underline{w})$ it is a unique point in the surface.
\end{lemm}

\begin{proof}
Let $\underline{w}$ be in $\Sigma_A$ we must prove that $\pi_f(\underline{w})$ has only one point.
 The unstable length of $R_i$ was previously define. Let's  take the maximum of such lengths:
 $$
 L^u(\cR):=\max\{L^u(R_i): i=1,\cdots, n\}=L^u.
 $$
For every $n\geq 0$, let  $H_n:=\cap_{i\leq n} f^{-i}(R_{w_i})\subset f^{-n}(R_{w_{n}})$. Since the incidence matrix of $\cR$ has only $0$ or $1$ as coefficients the sets $H_n$ have only one connected component, and they are  horizontal sub-rectangles of $R_{w_0}$, hence
$$
L^u(H_n)= \mu^s(\cap_{i\leq n} f^{-i}(R_{w_i})) \leq \mu^s( f^{-n}(R_{w_{n}})))= \lambda^{-n}\mu^s(R_{w_n})\leq \lambda^{-n} L^u
$$
Therefore their  vertical lengths converges to zero, i.e. $\lim_{n\rightarrow \infty} \mu^s(H_n)=0$, so $\cap_{z\geq 0} f^{-z}(R_{w_z})$ is an horizontal segment $I_{\underline{w}}$ of $R_{w_0}$ (the measures $\mu^s$ take positive values in any non-trivial interval).

 Similar arguments using the $\mu^u$-measure show that $V_n=\cap_{0 \geq z \geq n }f^{-z}(R_{w_z})$ is  horizontal sub-rectangles of $R_{w_0}$ whose horizontal length converges to zero, in view of the following computation: for all $n\in \NN$, 
 $$
 L^s(V_n)= \mu^u(\cap_{i\leq n} f^{i}(R_{w_{-i}})) \leq \mu^u( f^{n}(R_{w_{-n}})))= \lambda^{-n}\mu^u(R_{w_n})\leq \lambda^{-n} L^s(\cR)
  $$
Therefore $\cap_{z\leq 0} f^{-z}(R_{w_z})$ is a vertical segment $J_{\underline{w}}$ of $R_{w_0}$. The intervals $J_{\underline{w}}$ and $I_{\underline{w}}$ intersects in a finite number of points as are they are compact.
$$
I_{\underline{w}}\cap J_{\underline{w}}=\{x_i\}_{i=1}^m
$$

Then there exist pairwise disjoint open sets $\{U_i\}_{i=1}^m$ such that $x_i\in U_i$ and they are disjoints. Now lets come back to our projection.

Clearly 
$$
\pi_f(\underline{w})\subset \cap_{z\in \ZZ} f^{-z}(R_{w_z}) =  I_{\underline{w}}\cap  J_{\underline{w}}
$$

 Let $\overset{o}{Q_n}=\cap_{z=-n}^nf^{-z}(\overset{o}{R_{w_z}})$ an open rectangle and its closure $Q_n=\overline{\cap_{z=-n}^nf^{-z}(\overset{o}{R_{w_z}})}$ have the finite intersection property, therefore its intersection is not empty. So $\pi_f(\underline{w})$ is finite collection of points. But as the width of $\overset{o}{Q_n}$ converge to zero and is a connected set, there exist $N\in \NN$  and a unique $i\in \{1,\cdots,m\}$ such that for all $n\geq N$, $\overset{o}{Q_n}\subset U_i$. This implies that $\pi_f(\underline{w})\subset U_i$ and the only possibility is $\pi_f(\underline{w})=x_i$. This end the proof.
\end{proof}

With the metric that we defined for $\Sigma$ and $\Sigma_A$, two codes $\underline{w}$ and $\underline{v}$ are 'close' if there exists a certain $N \in \NN$ such that for all $-N \leq z \leq N$, $w_z = v_z$. This means that there is a central block around $0$ where the codes coincide. In particular, for all $n \geq N$, the rectangles $\overline{Q_n}(\underline{w})$ induced by $\underline{w}$ and $\overline{Q_n}(\underline{v})$ induced by  $\underline{v}$ are the same. This observation leads to the next corollary.

\begin{coro}\label{Coro: pif is continuous funtion}
With the hypothesis of Proposition \ref{Prop:proyecion semiconjugacion}, the projection $\pi_f$ is continuous.
\end{coro}

\begin{lemm}\label{Lemm: Semiconjugacy pif}
Under the assumptions of Proposition \ref{Prop:proyecion semiconjugacion}, the projection $\pi_f$ is a semi-conjugation between  $\sigma$ and $f$, i.e. $f\circ \pi_f = \pi_f \circ \sigma_A$.
\end{lemm}

\begin{proof}
Let $\underline{w}\in \Sigma_A$  need to show that
$$
f\circ \pi_f (\underline{w})=\pi_f \circ \sigma(\underline{w}).
$$
A direct computation gives that:
\begin{eqnarray*}
f\circ \pi_f (\underline{w})=f \left( \bigcap_{n\in \NN} \overline{\cap_{z=-n}^{n} f^{z}\left( \overset{o}{R_{w_{-z}}} \right)} \right)= \\
\bigcap_{n\in \NN} f\left(\overline{\cap_{z=-n}^{n} f^{z}\left(\overset{o}{R_{w_{-z}}}\right)}\right)= \bigcap_{n\in \NN} \overline{\cap_{z=-n}^{n} f(f^{z}\left(\overset{o}{R_{w_{-z}}}\right))}=\\
\bigcap_{n\in \NN} \overline{\cap_{z=-n}^{n} (f^{z+1}\left(\overset{o}{R_{w_{-z}}}\right))}=\bigcap_{n\in \NN} \overline{\cap_{k=-n+1}^{n+1} (f^{k}\left(\overset{o}{R_{w_{-k+1}}}\right))}\\
\subset \bigcap_{n\in \NN} \overline{\cap_{k=-(n-1)}^{n-1} (f^{k}\left(\overset{o}{R_{w_{-k+1}}}\right))} = \bigcap_{n\in \NN} \overline{\cap_{k=-n}^{n} (f^{k}\left(\overset{o}{R_{w_{-k+1}}}\right))} \\
=\pi_f \circ \sigma(\underline{w}).
\end{eqnarray*}

Since the Lemma \ref{Lemm: pif is well defined}  established that $f\circ \pi_f (\underline{w})$ and $\pi_f \circ \sigma(\underline{w})$ have only one element, the previous contention is indeed and equality and this argument complete our proof.

\end{proof}

\begin{lemm}\label{Lemm: projection is surjective}
The projection $\pi_f:\rightarrow S$ is surjective.
\end{lemm}

For all $x\in S$, we will construct an element of $\Sigma_A$ that projects to $x$. The \emph{sector codes} we define below will do the job.  It was shown in Lemma \ref{Lemm: sector contined unique rectangle} that each sector is contained in a unique rectangle of the Markov partition, and Proposition  \ref{Prop: image secto is a sector} shows that the image of a sector is a sector. This allows for the following definition.

\begin{defi}\label{Defi: Sector codes}
Let $\cR=\{R_i\}_{i=1}^n$ be a Markov partition whose incidence matrix is binary. Let $x \in S$ be a point with sectors $\{e_1(x),\cdots, e_{2k}(x)\}$ (where $k$ is the number of stable or unstable separatrices in $x$).

The \emph{sector code} of $e_j(x)$ is the sequence $\underline{e_j(x)}=(e(x,j)_z)_{z\in \ZZ} \in \Sigma$, given by the rule: $e(x,j)_z:=i$, where $i\in \{1,\dots,n\}$ is the index of the unique rectangle in $\cR$ such that the sector $f^z(e_j(x))$ is contained in the rectangle $R_i$.

\end{defi}

Any $\underline{w}\in\Sigma$ that is equal to the sector code of $e_j(x)$ (for some $j$ ) is called a sector code of $x$. The space $\Sigma$ of bi-infinite sequences is larger than $\Sigma_A$. We need to show that every sector code is, in fact, an \emph{admissible code}, i.e., that $\underline{e_j(x)}\in \Sigma_A$.

\begin{lemm}\label{Lemm: sector code is admisible}
For every $x \in S$, every sector code $\underline{e}:=\underline{e_j(x)}$ is an element of $\Sigma_A$.
\end{lemm}

\begin{proof}
Let $A=(a_{ij})$ the incidence matrix.	The code $\underline{e}=(e_z)$ is in $\Sigma_A$ if and only if for all $z\in \ZZ$, $a_{e_z e_{z+1}}=1$. By definition, this happens if and only if $f(\overset{o}{R_{e_z}})\cap \overset{o}{R_{e_{z+1}}}\neq \emptyset$.

Let $\{x_n\}_{n\in \NN}$ be a sequence converging to $f^z(x)$ and contained in the sector $f^z(e)$. By Lemma \ref{Lemm: sector contined unique rectangle} the sector $f^z(e)$ is contained in a unique rectangle $R_{e_z}$, and we can assume $\{x_n\}\subset R_{e_z}$. Moreover, there exists $N\in \NN$ such that $x_n \in \overset{o}{R_{e_z}}$ for all $n\geq N$. Remember that the sector $f^z(e)$ is bounded by two consecutive local stable and unstable separatrices of $f^z(x)$: $F^s(f^z(x))$ and $F^u(f^z(x))$. If for every $n\in \NN$, $x_n$ is contained in the boundary of $R_{e_z}$, this boundary component is a local separatrice of $x$ between $F^s(f^z(x))$ and $F^u(f^z(x))$, which is not possible.

Since the image of a sector is a sector, the sequence $\{f(x_n)\}$ converges to $f^{z+1}(x)$ and is contained in the sector $f^{z+1}(e)$. The argument in the last paragraphs also applies to this sequence, and $f(x_n)\in \overset{o}{R_{e_{z+1}}}$ for $n$ big enough. This proves that $f(\overset{o}{R_{e_z}})\cap \overset{o}{R_{e_{z+1}}}$ is not empty.
	
\end{proof}

The sector codes of a point $x$ are not only admissible; as the following Lemma shows, they are, in fact, the only codes in $\Sigma_A$ that project to $x$.  

\begin{lemm}\label{Lemm: every code is sector code }
If $\underline{w}=(w_z)\in \Sigma_A$ projects under $\pi_f$ to $x$, then $\underline{w}$ is equal to a sector code of $x$.
\end{lemm}

\begin{proof}
		
		For each $n \in \N$, we take the rectangle $F_n=\cap_{j=-n}^n f^{-j}(\overset{o}{R_{w_j}})$, which is non-empty because $\underline{w}$ is in $\Sigma_A$. The following properties hold:

	\begin{itemize}
		\item[i)] $\pi_f(\underline{w})=x \in \overline{F_n}$ for every $n\in \ZZ$.
		\item[ii)] For all $n\in \NN$, $F_{n+1}\subset F_n$.
		\item[iii)]	For every $n\in \NN$, there exists at least one sector $e$ of $x$ contained in $F_n$. If this does not occur, there is $\epsilon>0$ such that the regular neighborhood of size $\epsilon$ around $x$, given by Theorem \ref{Theo: Regular neighborhood}, is disjoint from $F_n$. However, $x\in F_n$ by item i).
	
		\item[iv)] If the sector $e \subset H_n$, then for every $m\in  \ZZ$ such that $\vert m\vert \leq n$:
		$$
		f^m(e) \subset f^m(H_n)= \cap_{j=m-n}^{m+n}f^{-j}(\overset{o}{R_{w_{m-j}}}) \subset \overset{o}{R_m}.
		$$ 
		which implies that $e_m=w_m$ for all $m\in \{-n,\cdots,n\}$.
	\end{itemize}
	
	By item $ii)$, if a sector $e$ is not in $F_n$, then $e$ is not in $F_{n+1}$. Together with the fact that for all $n$ there is always a sector in $F_n$ (item $iii)$), we deduce that there is at least one sector $e$ of $x$ that is in $F_n$ for all $n$. Then, we apply point $iv)$ to deduce $e_z=w_z$ for all $z$.
\end{proof}
Let $x$ be a point with $k$ stable and $k$ unstable separatrices. Then $x$ has at most $2k$ sector codes projecting to $x$.
\begin{coro}\label{Coro: Caracterisation fibers}
For all $x\in S$, if $x$ has $k$ separatrices, then $\pi_f^{-1}(x)=\{\underline{e_j(x)}\}_{j=1}^{2k}$. In particular, $\pi_f$ is finite-to-one.
\end{coro}

This ends the proof of Proposition \ref{Prop:proyecion semiconjugacion}.

\subsection{The quotient space  is a surface}\label{subsec: quotien surface}

There is a very natural equivalence relation in $\Sigma_A$ in terms of the projection $\pi_f$. Two codes $\underline{w}$ and $\underline{v}$ in $\Sigma_A$ are $f$-related, written as $\underline{w} \sim_f \underline{v}$, if and only if $\pi_f(\underline{w}) = \pi_f(\underline{v})$. The quotient space is denoted as $\Sigma_f = \Sigma_A/\sim_f$, and $[\underline{w}]_f$ represents the equivalence class of $\underline{w}$. If $\underline{w} \sim_f \underline{v}$, by definition:
$$
[\pi_f]([\underline{w}]_f)=\pi_f(\underline{w})=\pi_f(\underline{v})=[\pi_f]([\underline{v}]_f).
$$

Even more, since $\pi_f: \Sigma_A \rightarrow S$ is a continuous function, $\Sigma_A$ is compact, and $S$ is a Hausdorff topological space, the \emph{closed map lemma} implies that $\pi_f$ is a closed function. Furthermore, since $\pi_f$ is surjective, it follows that $\pi_f$ is a quotient map. Therefore, the projection $\pi_f$ induces a homeomorphism $[\pi_f]: \Sigma_f \rightarrow S$ in the quotient space. The shift behaves well under this quotient since: $$[\sigma_f([\underline{w}]_f)]_f:=[\sigma([\underline{w}])]_f
$$
and $\underline{w} \sim_f \underline{v}$. The semi-conjugation of $f$ and $\sigma$ through $\pi_f$ implies that:
$$
[\sigma(\underline{w})]\sim_f [\sigma(\underline{v})].
$$
then the map $[\sigma]:\Sigma_f \rightarrow \Sigma_f$ is well defined and is, in fact, a homeomorphism as well, and

\begin{eqnarray*}
[\pi_f]\circ[\sigma]([\underline{w}]_f)=
[\pi_f][\sigma_A(\underline{w})]_f =\pi_f(\sigma(\underline{w})))= \\
f \circ (\pi_f(\underline{w}))= f\circ[\pi_f]([\underline{w}]_f).
\end{eqnarray*}

Therefore, $[\pi_f]$ determines a topological conjugation between $f$ and $[\sigma_A]$. This implies that $\Sigma_f$ is a surface homeomorphic to $S$, and $[\sigma]$ is a pseudo-Anosov homeomorphism. It is almost the ideal situation. Let us summarize this discussion in the following proposition.

\begin{prop}\label{Prop: quotien by f}
	The quotient space $\Sigma_f:=\Sigma_A/\sim_f$ is homeomorphic to $S_f$, and the quotient shift $\sigma_f:\Sigma_f\rightarrow \Sigma_f$ is a generalized pseudo-Anosov homeomorphism, topologically conjugated to $f:S_f \rightarrow S_f$ through the quotient projection:
	$$
	[\pi_f]:\Sigma_f\rightarrow S_f.
	$$
\end{prop}

If we have two pseudo-Anosov maps $f:S_f\rightarrow S_f$ and $g:S_g\rightarrow S_g$ with Markov partitions of the same geometric type, after a horizontal refinement, they have the same incidence matrix $A$ with coefficients in $\{0,1\}$ and are associated with the same sub-shift of finite type $(\Sigma_A,\sigma)$. However, the projections $\pi_f$ and $\pi_g$ are not necessarily the same. In particular, while $\Sigma_A/\sim_f$ is homeomorphic to $S_f$ and $\Sigma_A/\sim_g$ is homeomorphic to $S_g$, we cannot affirm that $S_f$ is homeomorphic to $S_g$.

The problem arises from the fact that given $x\in S_f$ and $y\in S_g$, we don't know if $\pi_f^{-1}(x)\cap\pi_g^{-1}(x)\neq \emptyset$ implies $\pi_f^{-1}(x)=\pi_g^{-1}(x)$. In the proof of the finite-to-one property Lemma, it was shown that every code in $\pi_f^{-1}(x)$ is a sector code of $x$. Therefore, if $\pi_f^{-1}(x)\cap \pi_g^{-1}(y)\neq \emptyset$, there is a common sector code of $x$ and $y$. However, this does not imply a unique (or continuous) correspondence between the sets of sectors of $x$ and $y$. For example, it is possible for $x$ to have a different number of prongs than $y$.

This ambiguity cannot be resolved by the incidence matrix alone because two sector codes are considered different if, after iterations of $\sigma$, they are in different rectangles. If a point $x$ is in the corner of $4$ rectangles and $y$ is in the corner of $3$ rectangles, then $x$ has $4$ different sector codes, while $y$ has only $3$. This discrepancy in the number of incident rectangles at a point is overcome by the implementation of the geometric type.

To address this, we will construct another quotient space of $\Sigma_A$ called $\Sigma_T$ and define an equivalence relation $\sim_T$ in terms of the geometric type $T$.

\section{The equivalent relation $\sim_T$.}\label{Sub-sec:Sub-conjugation}

 In this section, we will define a the equivalence relation that satisfy the properties enumerated  in Proposition \ref{Prop: The relation determines projections}. Given a Markov partition $\cR=\{R_i\}_{i=1}^n$, we have introduced the following notation: the interior of the partition is denoted by $\overset{o}{\mathcal{R}}:=\cup_{i=1}^n \overset{o}{R_i}$, and the stable or unstable boundary of $\mathcal{R}$ is denoted by $\partial^{s,u}\mathcal{R}:=\cup_{i=1}^n \partial^{s,u}R_i$.
 
 The general idea is to start with a geometric type $T$ within the pseudo-Anosov class and define a decomposition of $\Sigma_{A(T)}$ in terms of $T$ into three subsets: $\Sigma_{I(T)}$, $\Sigma_{S(T)}$, and $\Sigma_{U(T)}$. We will then introduce three relations, $\sim_{I,s,u}$, over these sets based on the properties of $T$. These relations will be extended to an equivalence relation $\sim_T$ in $\Sigma_A$. Finally, we will prove that for any pair $(f,\mathcal{R})$ that realizes $T$, $\pi_f(\underline{w})=\pi_f(\underline{v})$ if and only if $\underline{w}\sim_T \underline{v}$. The subsets and the equivalence relation $\sim_T$ will be determined through several steps, which we will outline below:

\begin{itemize} 
	\item 	We define the set of periodic points in $\Sigma_A$ denoted by $\Sigma_P$. It is proven that $\pi_f(\Sigma_P)$ corresponds to the periodic points of $f$, and for every periodic point $p$ of $f$, $\pi_f^{-1}(p)\subset \Sigma_P$.

	\item 	We construct a finite family of positive codes $\mathcal{S}(T)^+$, called $s$-\emph{boundary label codes}. It is proven that every code $\underline{w}\in \Sigma_A$ whose positive part is in $\mathcal{S}(T)^+$ is projected by $\pi_f$ to the stable boundary of a Markov partition $(f,\mathcal{R})$ whose geometric type is $T$. This determines a subset $\underline{\mathcal{S}(T)}\subset \Sigma_{A(T)}$ of $s$-\emph{boundary codes}.

	\item 	A relation $\sim_s$ is defined in $\underline{\mathcal{S}(T)}$, which satisfies the property that two codes in $\underline{\mathcal{S}(T)}$ are $\sim_s$-related, i.e., $\underline{w}\sim_s \underline{v}$ if and only if they project to the same point $\pi_f(\underline{w})=\pi_f(\underline{v})$ in the stable boundary of two adjacent rectangles of the Markov partition.

	\item  	Similarly, we construct a finite family of negative codes $\mathcal{U}(T)^-$, called $u$-\emph{boundary label codes}, and the subset $\underline{\mathcal{U}(T)}\subset \Sigma_{A(T)}$ of codes whose negative part is in $\mathcal{U}(T)^-$. It is proven that $\underline{\mathcal{U}(T)}$ projects by $\pi_f$ to the unstable boundary of rectangles in the Markov partition. A relation $\sim_u$ is defined in $\underline{\mathcal{U}(T)}$ such that if two codes are $\sim_u$-related, they are projected to the same point in the unstable boundary of two adjacent rectangles of the Markov partition.

	\item 	Using the relations $\sim_s$ and $\sim_u$, we define a set of periodic codes $\Sigma_{\mathcal{P}(T)}\subset \Sigma_A$, which are projected by $\pi_f$ to the boundary periodic points of $\mathcal{R}$.
	
	\item 	We define the $s$-\emph{boundary leaf codes} $\Sigma_{\mathcal{S}(T)}\subset\Sigma_A$ as the set of codes $\underline{w}$ for which there exists $k\in \mathbb{Z}$ such that $\sigma^k(\underline{w})\in \underline{\mathcal{S}(T)}$. Similarly, the $u$-\emph{boundary leaf codes} $\Sigma_{\mathcal{U}(T)}\subset\Sigma_A$ is the set of codes $\underline{w}$ for which there exists $k\in \mathbb{Z}$ such that $\sigma^k(\underline{w})\in \underline{\mathcal{U}(T)}$. The relations $\sim_s$ and $\sim_u$ are extended to the sets of $s$-\emph{boundary leaf codes} and $u$-\emph{boundary leaf codes}, respectively, by declaring that $\underline{w},\underline{v}\in \Sigma_{\mathcal{S}(T)}$ are $S$-related, i.e., $\underline{w}\sim_S \underline{v}$, if and only if there is $k\in\mathbb{Z}$ such that $\sigma^k(\underline{w}),\sigma^k(\underline{v})\in \Sigma_{\mathcal{S}(T)}$ and $\sigma^k(\underline{w})\sim_s\sigma^k(\underline{v})$. Similarly, we define the relation $\sim_U$ over $\Sigma_{\mathcal{U}(T)}\subset \Sigma_A$.

	\item We prove that $\pi_f(\Sigma_{\mathcal{S}(T)})$ and $\pi_f(\Sigma_{\mathcal{U}(T)})$ corresponds to the stable and unstable leaves of the periodic boundary points of $(f,\mathcal{R})$, and they are the only codes in $\Sigma_A$ with this property.
	
	\item 	It is proved that if $\underline{w}, \underline{v} \in \Sigma_{\mathcal{S}(T)}, \Sigma_{\mathcal{U}(T)}$ are $S, U$-related, then $\pi_f(\underline{w}) = \pi_f(\underline{v})$.
	
	\item 	We extend the relations $\sim_S$ and $\sim_U$ to a relation $\sim_{S,U}$ in the union $\Sigma_{\mathcal{S}(T)} \cup \Sigma_{\mathcal{U}(T)}$. It is proved that $\underline{w}, \underline{v} \in \Sigma_{\mathcal{S}(T)} \cup \Sigma_{\mathcal{U}(T)}$ are $\sim_{S,U}$-related if and only if they are projected by $\pi_f$ to the same point.

	\item 	The subset $\Sigma_{\mathcal{I}(T)}$ of $\Sigma_A$ is the complement of $\Sigma_{\mathcal{S}(T)} \cup \Sigma_{\mathcal{U}(T)}$ in $\Sigma_A$. $\pi_f(\Sigma_{I})$ is contained in the complement of the stable and unstable foliation of the periodic boundary points of $\mathcal{R}$ and are called \emph{totally interior codes}. We define $\sim_I$ as the identity relation in $\Sigma_{\mathcal{I}(T)}$. It is proved that $\underline{w} \in \Sigma_{I}$ if and only if all the sector codes of $\pi_f(\underline{w})$ are the same.

	\item 	We extend the relations $\sim_{S,U,I}$ to a relation $\sim_T$ in $\Sigma_A$. We prove that it is an equivalence relation in $\Sigma_A$ and satisfies Proposition \ref{Prop: The relation determines projections}.
	
\end{itemize}

\subsection{Periodic codes and periodic points}
Let $T$ be a geometric type whose incidence matrix is binary. Let $f: S \rightarrow S$ be a generalized pseudo-Anosov homeomorphism with a Markov partition of geometric type $T$, denoted as $(f, \mathcal{R})$. Consider $\pi_f: \Sigma_A \rightarrow S$ as the projection induced by $(f, \mathcal{R})$ into $S$. In this subsection, $A := A(T)$ represents the incidence matrix of $T$. The periodic points in the dynamical system $(\Sigma_A, \sigma)$ are denoted as $\text{Per}(\sigma) \subset \Sigma_A$, and the periodic points of $f$ are denoted as $\text{Per}(f) \subset S$. Note that $(f, \mathcal{R})$ implies that $\mathcal{R}$ is seen as a Markov partition of $f$. Hence, to avoid overloading the notation, we may omit the explicit mention of the homeomorphism if it is clear from the context.

The following lemma is classical in the literature when the authors work with the more classical definition of the projection, anyway, this lemma  represents the type of result we aim to prove. On one hand, we have periodic shift codes determined by $T$, and on the other hand, a valid property holds for every pseudo-Anosov homeomorphism with a Markov partition of the prescribed geometric type.

\begin{lemm}\label{Lemm: Periodic to peridic}
Let $T$ be a geometric type in the pseudo-Anosov class. Assume that its incidence matrix is binary, and $(f, \mathcal{R})$ is a pair realizing $T$. If $\underline{w} \in \text{Per}(\sigma)$ is a periodic code, then $\pi_f(\underline{w})$ is a periodic point of $f$. Moreover, if $p$ is a periodic point for $f$, then $\pi^{-1}_f(p) \subset \text{Per}(\sigma)$, meaning that all codes projecting onto $p$ are periodic.
\end{lemm}

\begin{proof}
If $\underline{w}\in \Sigma_{A}$ is periodic of period $k$, the semi-conjugation given by $\pi_f$ implies that:

$$
f^k(\pi_f(\underline{w}))=\pi_f(\sigma^k(\underline{w}))=\pi_f(\underline{w}).
$$

If $p$ is a periodic point of period $P$, and let $\underline{w} \in \Sigma_{A}$ be a code such that $\pi_f(\underline{w}) = p$, then from our proof that $\pi_f$ is finite-to-one, we established that every code projecting to $p$ is a sector code. Thus, we have $\underline{w} = \underline{e}$, corresponding to some sector $e$ of $p$.

If $p$ is a saddle point of index $k$ (with $k \geq 1$), there are at most $2k$ rectangles in $\cR$ that can contain $p$. Let's consider the rectangle $R_{w_0}$. We are aware that under $f$, the sectors incident to $p$ can be permuted, with each sector contained within a single rectangle. The action of $f^{kP}$ on the set of sectors of $p$ is a permutation. Hence, there exists a multiple of the period of $p$, denoted as $Q = km$, such that $f^Q(e) = e$. However, once we observe that a sector $e$ returns to itself after the action of $f^Q$ for all $m \in \{1, 2, \ldots, Q\}$, it  becomes necessary that the sector $f^{Q+m}(e)$ and $f^{m}(e)$ must coincide. Consequently, the sector code $\underline{e} = \underline{w}$ is periodic, with a period $Q$.

\end{proof}

 Suppose that the pair $(f, \mathcal{R})$ realizes the geometric type $T$, i.e., $T(f, \mathcal{R}) = T$. For a fixed Markov partition $\mathcal{R}$, we give in Definition \ref{Defi: Interiors rec and segments} the notions of:  interior of $\mathcal{R}$ denoted by $\overset{o}{\mathcal{R}}$, and  stable and unstable boundaries of $\mathcal{R}$ denoted by $\partial^{s,u}\mathcal{R}$. Similarly, in Definition \ref{Defi: boundary points}, we introduced the following notation: the stable and unstable periodic boundary points are denoted by Per$^{s,u}(f, \mathcal{R}) =$ Per$^{s,u}(\mathcal{R})$, the interior periodic points by Per$^{I}(f, \mathcal{R}) =$ Per$^{I}(\mathcal{R})$, and the corner periodic points by Per$^{c}(f, \mathcal{R}) =$ Per$^{c}(f, \mathcal{R})$.
 
 The following Lemma characterizes the periodic boundary points of $(f,\cR)$ in terms of their iterations under $f$. It is a technical result that we will use in future arguments.

\begin{lemm} \label{Lemm: no periodic boundary points}
	Let $\underline{w} \in \Sigma_A$ be a code such that for every $k \in \mathbb{Z}$, $f^k(\pi_f(\underline{w})) \in \partial^s\mathcal{R}$, then $\pi_f(\underline{w})$ is a periodic point of $f$. Similarly, if $f^k(\pi_f(\underline{w})) \in \partial^u\mathcal{R}$ for all integers $k$, then $\underline{w}$ is a periodic code.
\end{lemm}

\begin{proof}

Suppose $\underline{w}$ is non-periodic. Lemma \ref{Lemm: Periodic to peridic}  implies that $x = \pi_f(\underline{w})$ is non-periodic and lies on the stable boundary of $\mathcal{R}$, which is a compact set with each component having finite $\mu^u$ length. Take $[x,p]^s$ to be the stable segment joining $x$ with the periodic point on its leaf (which exists by Lemma \ref{Lemm: Boundary of Markov partition is periodic}). For all $m \in \mathbb{N}$,
$$
\mu^u([f^{-m}(x), f^{-m}(p)]^s)=\mu^u(f^{-m}[x,p]^s)=\lambda^m\mu^u([x,p]^s),
$$
 which is unbounded. Therefore, there exists $m \in \mathbb{N}$ such that $f^{-m}(x)$ is no longer in $\partial^s\mathcal{R}$, contradicting the hypothesis of our lemma. A similar reasoning applies to the second proposition.
 \end{proof}

This will allow us to provide a characterization of the inner periodic points of $\mathcal{R}$. The proof of Corollary \ref{Coro: interior periodic points unique code} follows from the fact that all sector codes of a point satisfying the hypotheses are equal.

\begin{coro}\label{Coro: interior periodic points unique code}

Let $\underline{w}\in \Sigma_A$ be a periodic code. If $\pi_f(\underline{w})\in \overset{o}{\cR}$, then for all $z\in \ZZ$, $f^z(\pi_f(\underline{w}))\in \overset{o}{\cR}$, which means it stays in the interior of the Markov partition. Moreover, $\pi^{-1}_f(\pi_f(\underline{w}))=\{\underline{w}\}$ have a unique point.

\end{coro}

\begin{proof}
 If, for some $z\in \ZZ$, $f^z(x)$ is in the stable (or unstable) boundary of the Markov partition, as $x$ is periodic and the stable boundary of a Markov partition is $f$-invariant, the entire orbit of $x$ lies in the stable boundary of $\cR$. This leads to a contradiction because at least one point in the orbit of $x$ is in the interior of $\cR$.
 
 From this perspective, all the sectors of $f^z(x)$ are contained within the interior of the same rectangle, and therefore all its sector codes are the same.
\end{proof}

\subsection{Decomposition of the surface: totally interior points and boundary laminations.}

The property that the entire orbit of $x$ is always in the interior of the partition is very important because it distinguishes the points that are in the complements of the stable and unstable laminations of periodic boundary points. We are going to give them a name.

\begin{defi}\label{Defi: totally interior points}
Let $(f,\cR)$ be a realization of $T$. The \emph{totally interior points} of $(f,\cR)$ are points $x \in \cup\cR=S$ such that for all $z \in \ZZ$, $f^z(x) \in \overset{o}{\cR}$. They are denoted Int$(f,\cR)\subset S$.
\end{defi}

 Proposition \ref{Prop: Carterization injectivity of pif} characterizes the points of $S$ where the projection $\pi_f$ is invertible as the \emph{totally interior points} of $\cR$.

\begin{prop}\label{Prop: Carterization injectivity of pif}
Let $x$ be any point in  $S$. Then $\vert \pi_f^{-1}(x)\vert =1$ if and only if $x$ is a totally interior point of $\cR$.
\end{prop}

\begin{proof}
If $\vert \pi_f^{-1}(x)\vert = 1$, for all $z\in \ZZ$ the sector codes of $f^z(x)$ are all equal. Therefore, for all $z\in \ZZ$, the sectors of $f^z(x)$ are all contained in the interior of the same rectangle. Furthermore, the union of all sectors determines an open neighborhood of $f^z(x)$, which must be contained in the interior of a rectangle. Therefore, $f^z(x)\in \overset{o}{\cR}$.

On the other hand, if $f^z(x)\in \overset{o}{\cR}$ for all $z\in \ZZ$, then all sectors of $f^z(x)$ are contained in the same rectangle. This implies that the sector codes of $x$ are all equal. In view of Lemma  \ref{Lemm: every code is sector code }, this implies $\vert \pi_f^{-1}(x)\vert = 1$.

\end{proof}

Lemma \ref{Lemm: Caraterization unique codes}  characterizes the stable and unstable manifolds of periodic boundary points of $\cR$ as the complement of the totally interior points. In this manner, we have a decomposition of $S$ as the union of three sets: the totally interior points, the stable lamination of periodic $s$-boundary points $\cF^s(\text{Per}^{s}(f,\cR))$, and the unstable lamination of periodic $u$-boundary points $\cF^u(\text{Per}^{u}(f,\cR))$.

\begin{lemm}\label{Lemm: Caraterization unique codes}
A point $x\in S$ is a totally interior point of $\cR$ if and only if it is not in the stable or unstable leaf of a periodic boundary point of $\cR$.
\end{lemm}

\begin{proof}

Suppose that for all $z\in \ZZ$, $f^z(x)\in \overset{o}{\cR}$ and also $x$ is in the stable (unstable) leaf of some periodic boundary point $p\in \partial^s R_i$ ($p\in \partial^u R_i$). The contraction (expansion) on the stable (unstable) leaves of $p$ implies the existence of a $z\in \ZZ$ such that $f^z(x)\in \partial^s R_i$ ($f^z(x)\in \partial^u R_i$), which leads to a contradiction.

On the contrary, Lemma  \ref{Lemm: Boundary of Markov partition is periodic} implies that the only stable or unstable leaves intersecting the boundary $\partial^{s,u} R_i$ are the stable and unstable leaves of the $s$-boundary and $u$-boundary periodic points. If $x$ is not in these laminations, neither are its iterations $f^z(x)$, and we have $f^z(x)\in \overset{o}{\cR}$ for all $z\in \ZZ$.
\end{proof}

It will be easier to provide a combinatorial definition of codes that project to stable and unstable laminations of periodic $s,u$-boundary points than to define codes that project to a totally interior point based on $T$. However, we can define the latter as the complement of the union of the former.

In the next section, we will demonstrate how to construct codes that project to stable and unstable $\cR$-boundaries and then expand this set to include codes that project to theirs stable and unstable leaves.

\subsection{Codes that project to the stable or unstable leaf.}

 Given $\underline{w} \in \Sigma_A$, the task is to determine, in terms of $(\Sigma_A, \sigma)$ and $T$, the subset $\underline{F}^s(\underline{w})$ of codes that are projected by $\pi_f$ to the stable leaf of $\pi_f(\underline{w})$. The following definition is completely determined by $T$ and its sub-shift, this what we use to call a \emph{combinatorial definition}.

\begin{defi}\label{Defi: s,u-leafs}
Let $\underline{w} \in \Sigma_A$. The set of \emph{stable leaf codes} of $\underline{w}$ is defined as follows:
	$$
	\underline{F}^s(\underline{w}):=\{\underline{v}\in \Sigma_A: \exists Z\in \ZZ \text{ such that } \forall z\geq Z, \, v_z=w_z\}.
	$$
Similarly, the \emph{unstable leaf codes} of $\underline{w}$ is the set:
	$$
	\underline{F}^u(\underline{w}):=\{\underline{v}\in \Sigma_A: \exists Z\in \ZZ \text{ such that } \forall z\leq Z, \, v_{z}=w_{z}\}.
	$$
\end{defi}

We will now prove that these sets project to the stable leaf at the point $\pi_f(\underline{w})$, and that they are the only codes that do so. This provides a combinatorial definition of stable and unstable leaves.

Let's introduce some notation. For $\underline{w}\in \Sigma_A$, we define its positive part as $\underline{w}_+ := (w_n)_{n\in \NN}$ (with $0\in \NN$), and its negative part as $\underline{w}_- := (w_{-n})_{n\in \NN}$. We denote the set of positive codes in $\Sigma_A$ as $\Sigma_A^+$, which corresponds to the positive parts of all codes in $\Sigma_A$. Similarly, we define the set of negative codes in $\Sigma_A$ as $\Sigma^-_A$. Moreover, we use $F^s(x)$ to denote the stable leaf $\cF^s(f)$ passing through $x$, and $F^u(x)$ for the corresponding unstable leaf $\cF^u(f)$.

\begin{prop}\label{Prop: Projection foliations}
Let $\underline{w}\in \Sigma_A$. Then we have $\pi_f(\underline{F}^{s}(\underline{w}))\subset F^{s}(\pi_f(\underline{w}))$. Furthermore, assume that $\pi_f(\underline{w})=x$.
\begin{itemize}
\item If $x$ is not a $u$-boundary point, then for all $y\in F^u(x)$, there exists a code $\underline{v}\in \underline{F}^s(\underline{w})$ such that $\pi_f(\underline{v})=y$.

\item  If $x$ is a $u$-boundary point and $\underline{w}_0=w_0$, then for all $y$ in the stable separatrix of $x$ that enters the rectangle $R_{w_0}$, denoted $F^s_0(x)$, there is a code $\underline{v}\in \underline{F}^s(x)$ that projects to $y$, i.e., $\pi_f(\underline{v})=x$.

\end{itemize}

A similar statement applies to the unstable manifold of $\underline{w}$ and its respective projection onto the unstable manifold of $\pi_f(\underline{w})$.
	\end{prop}

\begin{proof}

Let $\underline{v}\in \underline{F}^s(\underline{w})$ be a stable leaf code of $\underline{w}$. Based on the definition of stable leaf codes, we can deduce that $w_z=v_z$ for all $z\geq k$ for certain $k\in \NN$. Consequently, $\pi_f(\sigma^k(\underline{w})),\pi_f(\sigma^k(\underline{v}))\in R_{w_k}$. Since the positive codes of $\underline{w}$ and $\underline{v}$ are identical starting from $k$, they define the same horizontal sub-rectangles of $R_{w_k}$ in which the codes $\sigma^k(\underline{w})$ and $\sigma^k(\underline{v})$ are projected. For $n\in \mathbb{N}$, let $H_n$ be the rectangle determined by:

$$
H_n=\cap_{z=0}^n f^{-z}(R_{w_{z+k}})=\cap_{z=0}^n f^{-z}(R_{v_{z+k}}).
$$
The intersection of all the $H_n$ forms a stable segment of $R_{w_k}$. Furthermore, each $H_n$ contains the rectangles:
$$
\overset{o}{Q_n}=\cap_{z=-n}^{n} f^{-z}(\overset{o}{R_{w_{z+k}}})= \cap_{z=-n}^{n} f^{-z}(\overset{o}{R_{v_{z+k}}})
$$
Therefore, the projections $\pi_f(\underline{v})$ and $\pi_f(\underline{w})$ are in the same stable leaf. 
	
	\begin{figure}[h]
		\centering
		\includegraphics[width=1\textwidth]{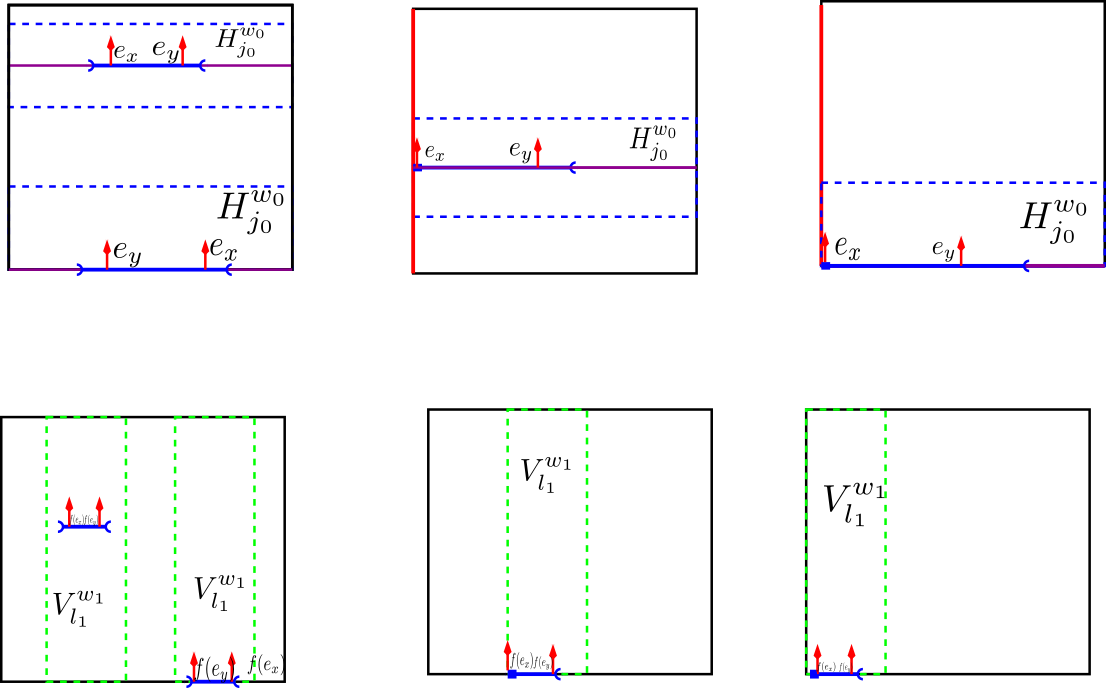}
		\caption{Projections of the stable leaf codes}
		\label{Fig: Proyection fol}
	\end{figure}

For the other part of the argument, let's refer to the image in Figure \ref{Fig: Proyection fol} to gain some visual intuition. Let's address the situation when $x=\pi_f(\underline{w})$ is not a $u$-boundary point. We recall that is implicit that the incidence matrix $A$ is binary.

Let $y\in F^s(\pi_f(\underline{w}))$ such that $y\neq x$. Suppose $y$ is not a periodic point (if it is, we can replace $y$ with $x$). Since  $x$ is not a $u$-boundary point, there is a small interval $I$ properly contained  in the stable segment of  $\cR_{w_0}$ that pass trough $x$. This allows us to apply  the next argument on both stable separatrices of $x$.  By definition of stable leaf, there exist $k\in \mathbb{N}$ such that $f^k(y)\in I$. If we prove there is  $\underline{v}\in \underline{F}^s(\underline{w})$ such that $\pi_f(\underline{v})=f^k(y)$, we obtain that $\sigma^{-k}(\underline{w})\in \underline{F}^s(\underline{w})$  is a code that projects to $y$, i.e $\pi_f(\sigma^{-k}(\underline{v}))=y$. Therefore we can assume  that $y\in I$. This situation is given by the left side pictures in Figure \ref{Fig: Proyection fol}.

There are two situations for $f^z(y)$: either it lies in the stable boundary of $R_{w_0}$ or  not. In any case $x$, the code $\underline{w}$ is equal to  a sector code $\underline{e}_x$ of $x$. The sector $e_x$ is contained in  a unique horizontal sub-rectangle of $R_{w_0}$, $H^{w_0}_{j_0}$, therefore the sector $f(e_x)$ is contained in  $f(H^{w_0}_{j_0})=V^{w_1}_{l_1}$, implying $w_1=\underline{e_x}_1$.

Take the sector   $e_y$ of $y$ in such that in the stable direction points toward $x$ and in the unstable direction point in the same direction than $e_y$. In this manner $e_y$ is contained in the rectangle $H^{w_0}_{j_0}$, therefore $f(e_y)$ is contained in $V^{w_1}_{l_1}$ and $\underline{e_y}_1=w_1$.

 In fact, it turns out that $f^{n}(e_y)$ and $f^n(e_x)$ are contained in the same $R_{w_n}$ for all $n\in \NN$. We can deduce that the positive part of $\underline{e_y}$ coincides with the positive part of $\underline{w}$. Therefore $\underline{e_y}\in \underline{F}^s(\underline{w})$.


In the case where $x\in \partial^u\cR$, there is a slight variation. Let's assume that $\underline{w}=\underline{e_x}$, where $e_x$ is a sector of $x$. This sector code is contained within a unique horizontal sub-rectangle of $R_{w_0}$, $H^{w_0}_{j_0}$ such that $f(H^{w_0}_{j_0})=V^{w_1}_{l_1}$. This horizontal sub-rectangle contain a unique stable interval $I$ that contains $x$ in one of its end points. Consider $y\in I$.

Similar to the previous cases, we can define a sector  $e_y$ of $x$ that is contained in horizontal sub-rectangle $H^{w_0}_{j_0}$, i.e. $\underline{e_y}=w_0$. This implies that $f(e_y)$ is contained in $V^{w_1}_{l_1}$. Then the first term of $\underline{e_y}$ is $w_1$. By applied this process inductively  for all $n\in \NN$, we get that $\underline{e_y}_n=w_n$ and therefore $\underline{e_y}\in \underline{F}^s(\underline{w})$ and projects to $y$.
\end{proof}

The next task is to define, in terms of $T$, those codes that project to the boundary of the Markov partition. We will then use Proposition  \ref{Prop: Projection foliations} to determine the codes that project to the stable and unstable laminations of $s,u$-boundary periodic points.

\subsection{Boundary codes and $s,u$-generating functions.}

Let's proceed with the construction of the codes that are projected onto the boundary of the Markov partition, $\partial^{s,u}\cR$. We assume that $\cR=\{R_i\}_{i=1}^n$ is a geometric Markov partition of $f$ with geometric type $T$. For each $i\in \{1,\cdots,n\}$, we label the boundary components of $R_i$ as follows:

\begin{itemize}
\item $\partial^s_{+1}R_i$ denotes the upper stable boundary of the rectangle $R_i$.
\item $\partial^s_{-1}R_i$ denotes the lower stable boundary of $R_i$.
\item $\partial^u_{-1}R_i$ denotes the left unstable boundary of $R_i$.
\item $\partial^u_{+1}R_i$ denotes the right unstable boundary of $R_i$.
\end{itemize}

By using these labeling conventions, we can uniquely identify the boundary components of $\cR$ based on the geometric type $T$.

\begin{defi}\label{Defi; s,u boundary labels of T}
	Let $T=\{n,\{(h_i,v_i)\}_{i=1}^n,\Phi_T\}$  be an abstract geometric type. The $s$-\emph{boundary labels} of $T$ are defined as the formal set:
		$$
	\cS(T):=\{(i,\epsilon): i\in \{1,\cdots,n\} \text{ and } \epsilon\in \{1,-1\} \}, 
	$$
Similarly, the $u$-\emph{boundary labels} of $T$ are defined as the formal set:
	$$
	\cU(T):=\{(k,\epsilon): k\in \{1,\cdots,n\} \text{ and } \epsilon\in \{1,-1\} \}
	$$
\end{defi}

In a while, we will justify such names. It's important to note that this definition was made using only the value of $n$ given by the geometric type $T$, so it does not depend on the specific realization. With these labels, we can formulate in terms of the geometric type the codes that $\pi_f$ projects to $\partial^{u,s}_{\pm 1 }R_i$. The first step is to introduce a \emph{generating function}

We begin by relabeling the stable boundary component of $\mathcal{R}$ using the next  function:
$$
\theta_T:\{1,\cdots,n\}\times \{1,-1\}\rightarrow \{1,\cdots,n\} \times\cup_{i=1}^n \{1,h_i\}_{i=1}^n \subset \cH(T)
$$ 
that is defined as:

\begin{equation}\label{Equa: theta T relabel}
\theta_T(i,-1)=1 \text{ and } \theta(i,1)=h_i.
\end{equation}

The effect of $\theta_T$ is to choose the sub-rectangle of $R_i$ that contains the stable boundary component in question. Specifically, the boundary $\partial^s_{-1}R_1$ is contained in $\partial^s_{-1}H^i_1$, and the boundary $\partial^s_{+1}R_1$ is contained in $\partial^s_{+1}H^i_{h_i}$. This allows us to track the image of a stable boundary component, as the image of $\partial^s_{+1}R_i$ under $f$ should be contained in the stable boundary of $f(H^i_{h_i})$. But remember, the image of $\partial^s_{-1}R_i$ is a boundary component of $f(H^i_1) = V^k_l$, and the pair $(k,l)$ is uniquely determined by $\Phi_T$. By considering the value of $\epsilon_T(i,h_i)$, we can trace the image of the upper boundary component of $H^i_{h_i}$ in a more precise manner. If $\epsilon_T(i,h_i) = 1$, it indicates that the map $f$ does not alter the vertical orientations. As a result, the image of the upper boundary component of $H^i_{h_i}$ will remain on the upper boundary component of $f(H^i_{h_i})$. i.e. in this example:
$$
f(\partial^s_{+1}R_i)\subset f(\partial^s_{+1}H^i_{h_i}) \subset \partial^s_{+1}R_k.
$$ 
On the contrary, if $\epsilon_T(i,h_i)=-1$, it means that the map $f$ changes the vertical orientation. This has the following implication:
 $$
 f(\partial^s_{+1}R_i)\subset f(\partial^s_{+1}H^i_{h_i})\subset \partial_{-1}R_k.
 $$

 We don't really care about the index $l$ in $(k,l) \in \mathcal{V}(T)$, so it's convenient to decompose $\rho_T$ into two parts: $\rho_T:=(\xi_T,\nu_T)$. In this decomposition
$$
\xi_T:\cH(T)\rightarrow \{1,\cdots, n\},
$$
is defined as $\xi_T(i,j)=k$ if and only if $\rho_T(i,j)=(k,l)$.

 Let's continue with our example. In order to determine where $f$ sends the upper boundary component $\partial^s_{+1}R_i$, we need to identify the rectangle in $\mathcal{R}$ to which $f$ maps $H^i_j$. This can be determined using the following relation:
 $$
 \xi_T(i,\theta_T(1))=\xi_T(i,h_i)=k.
 $$ 
We now incorporate the change of orientation to determine the boundary component of $R_k$ that contains $f(\partial^s_{+1}H^i_{h_i})$. If we assume that $f(\partial^s_{+1}R_i) \subset \partial^s_{\epsilon}R_k$, then the value of $\epsilon$ can be determined using the following formula:
  $$
  \epsilon=+1 \cdot\epsilon_T(i,h_i) = +1 \cdot \epsilon_T(i,\theta_T(i,1)).
  $$   
This procedure is illustrated in Figure  \ref{Fig: theta T}.
 	\begin{figure}[h]
 	\centering
 	\includegraphics[width=0.6\textwidth]{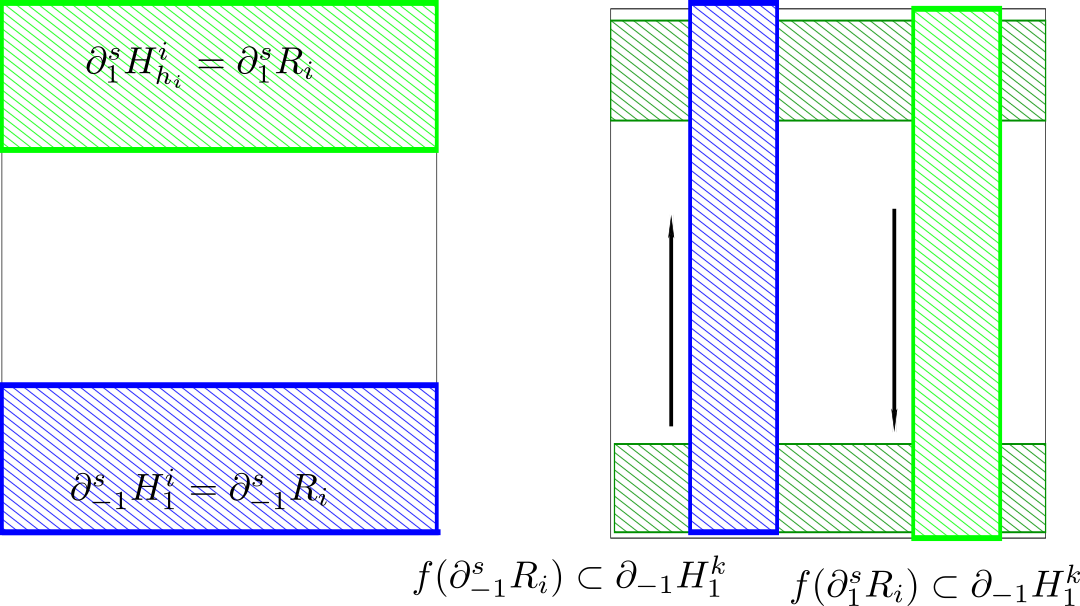}
 	\caption{The effect of $\theta_T$}
 	\label{Fig: theta T}
 \end{figure}
 The $S$-generating function utilizes $\theta_T$ and $\xi_T$ to generalize this idea.

\begin{defi}\label{Defi: s-boundary generating funtion}
 The $s$-generating function of $T$ is the function 
 	$$
 \Gamma(T):\cS(T) \rightarrow \cS(T),
 $$
defined for every $s$-boundary label $(i_0,\epsilon_0) \in \mathcal{S}(T)$ by the following formula:

\begin{equation}\label{Equ: Gamma generation funtion}
	\Gamma(T)(i_0,\epsilon_0)=(\xi_T(i_0,\theta_T(i_0,\epsilon_0)),
\epsilon_0 \cdot \epsilon_T(i_0,\theta_T(i_0,\epsilon_0)) ).
\end{equation}
	
\end{defi}

The $u$-boundary generating function can be defined as either the $s$-boundary generating function of $T^{-1}$ (associated with $f^{-1}$), or it can be defined directly as follows.

\begin{defi}\label{Defi: u-boundary generating funtion}
The $u$-\emph{generating function} of $T$ is
$$
\Upsilon(T):\cU(T)\rightarrow  \cU(T).
$$
defined as the $s$-\emph{generating function} of $T^{-1}$, $\Gamma(T^{-1})$.
\end{defi}

 The orbit of a label $(i_0,\epsilon_0)\in \cS(T)$ under $\Gamma(T)$ is defined as:
$$
\{(i_m,\epsilon_m)\}_{m\in \NN}:=\{\Gamma(T)^m(i_0,\epsilon_0) : m\in\NN\}. 
$$
We define $\varsigma_T:\cS(T)\rightarrow {1,\dots,n}$ as the projection onto the first component of the funtion $\Gamma_T$, i.e $\varsigma_T(i,\epsilon)=k$ if and only if $\Gamma_T(i,j)=(k,\epsilon)$. We will focus on the $s$-generating function. The ideas for the unstable case naturally extend, but sometimes we are going to recall them too.

The generating functions allow us to create a positive (or negative) code in $\Sigma^+$ for every $s$-boundary label of $T$ by composing $\Gamma(T)$ ($\Upsilon(T)$) with $\varsigma_T$ to obtain:
 
\begin{eqnarray}\label{Equa: $s$-boundary code +}
\underline{I}^{+}(i_0,\epsilon_0):=\{ \varsigma_T\circ\Gamma(T)^m(i_0,\epsilon_0)\}_{m\in\NN}=\{i_m\}_{m\in \NN}
\end{eqnarray}

Similarly, for $(k_0,\epsilon_0)\in \cU(T)$, we can construct the negative code in $\Sigma^{-}$ given by:

\begin{eqnarray}\label{Equa: $u$-boundary code -}
\underline{J}^{-}(k_0,\epsilon_0):=\{ \varsigma_T\circ\Upsilon(T)^m(k_0,\epsilon_0)\}_{m\in\NN}=\{k_{-m}\}_{m\in \NN}
\end{eqnarray}

These positive and negative codes are of great importance and deserve a name.

\begin{defi}\label{Defi: positive negative boundary codes }
	
	 Let $(i_0,\epsilon_0)\in \cS(T)$. The $s$-\emph{boundary positive code} of $(i_0,\epsilon_0)$ is denoted as $\underline{I}^+(i_0,\epsilon_0)\in \Sigma^+$, and it is determined by Equation \ref{Equa: $s$-boundary code +}. The set of $s$-boundary positive codes of $T$ is denoted as:
	 $$
	\underline{\cS}^{+}(T)=\{\underline{I}^{+}(i,\epsilon): (i,\epsilon)\in \cS(T)\} \subset \Sigma^+.
	$$
	Let $(k_0,\epsilon_0)\in \cU(T)$. The $u$-\emph{boundary negative code }of $(k_0,\epsilon_0)$ is denoted as $\underline{J}^{-}(k_0,\epsilon_0)=\{k_{-m}\}_{m=0}^{\infty}$ and is determined by Equation \ref{Equa: $u$-boundary code -}. The set of $u$-boundary negative codes of $T$ is denoted as
		
	$$
	\underline{\cJ}^{-}(T)=\{\underline{J}^{-}(k,\epsilon): (k,\epsilon)\in \cU(T)\}\subset \Sigma^{-}
	$$
\end{defi}
 
 Lemma \ref{Lemm: positive code admisible} below asserts that every $s$-boundary positive code of $T$ corresponds to the positive part of at least one element of $\Sigma_A$. The analogous result for $u$-boundary negative codes is true using a fully symmetric approach with the $u$-generating function. Therefore, we will only provide the proof for the stable case.

\begin{lemm}\label{Lemm: positive code admisible}
	Every $s$-boundary positive code $\underline{I}^{+}(i_0,\epsilon_0)$ belongs to $\Sigma_A^+$. Similarly, every $u$-boundary negative code $\underline{J}^{-}(k_0,\epsilon_0)$ belongs to $\Sigma_A^-$.
\end{lemm}

\begin{proof}
The incidence matrix $A$ of $T$ has a relation to any realization $(f,\cR)$ of $T$, which we can exploit to visualize the situation, anyway observe that our arguments only use the properties of $A$ or $T$. 
	In the construction of the positive code $\underline{I}^{+}(i_0,\epsilon_0)$, the term $i_1$ corresponds to the index of the unique rectangle in $\cR$ such that $f(H^{i_0}_{\theta_T(1,\epsilon_0)})=V^{i_1}_{l_1}$ (where $\theta_T(i_0,\epsilon)=1$ or $h_i$). In either case, we have:
	
	$$
	f^{-1}(\overset{o}{R{i_1}})\cap \overset{o}{R_{i_0}}=\overset{o}{H^{i_0}_{\theta_T(1,\epsilon_0)}´}\neq \emptyset,
	$$
	
	 which implies $a_{i_0,i_1}=1$. By induction, we can indeed extend this result to the entire positive code $\underline{I}^{+}(i_0,\epsilon_0)$.
	 
	 Similarly, for the case of $u$-boundary negative codes, we can use the symmetric approach by considering $T^{-1}$ and the inverse of the incidence matrix $A^{-1}$. Additionally, we can consider the realization $(f^{-1},\cR)$ to obtain a visualization. 
\end{proof}

Proposition \ref{Prop: s code Inyective} is the first step in proving that each $s$-boundary label uniquely determines a set of codes that projects to a single boundary component of the partition.

\begin{prop}\label{Prop: s code Inyective}	
	The map $I:\cS(T)\rightarrow \Sigma_A^+$ defined by $I(i,\epsilon):=\underline{I}^+(i,\epsilon)$ is  injective map. Similarly, the map $J:\cU(T)\rightarrow \Sigma_A^-$ defined by $J(i,\epsilon):=\underline{J}^-(i,\epsilon)$ is also injective.
\end{prop}

\begin{proof}
	
	Clearly $I$ is a  well defined map as $\Gamma^n(T)$ is itself a well defined map with the same domain. 	If $i_0 \neq i'_0$, then the sequences $\underline{I}^+(i_0,\epsilon_0)$ and $\underline{I}^+(i'_0,\epsilon'_0)$ will differ in their first term, making them distinct. The remaining case involves considering $(i_0,1)$ and $(i_0,-1)$.
	
	Let's denote $\underline{I}^+(i_0,1) = {i_m}$ and $\underline{I}^+(i_0,-1) = {i'_m}$ as their positive codes. We will analyze the sequences ${\Gamma(T)^m(i_0,1)}$ and ${\Gamma(T)^m(i_0,-1)}$ and show that there exists an $m \in \NN$ such that $i_m \neq i'_m$. Our approach starts with the following technical lemma:
	
	\begin{lemm}\label{Lemm: positive h-i}
If $T$ is in the pseudo-Anosov class and $A$ is a binary matrix, then there exists $M \in \mathbb{N}$ such that if  $\underline{I}^+(i_0,1) = \{i_m\}_{m\in \NN}$, then  $h_{i_M} > 1$.
	\end{lemm}
	
	\begin{proof}
		Since $T$ is in the pseudo-Anosov class, there exists a realization $(f,\mathcal{R})$. The infinite sequence ${i_m}$ takes a finite number of values (at most $n$), so there exist natural numbers $m_1 < m_2$ such that $R_{i_{m_1}} = R_{i_{m_2}}$. If such an $M$ does not exist, then for all $m_1 \leq m \leq m_2$, $h_{i_m} = 1$. This implies that $f^{m_2-m_1}(R_{i_{m_1}}) \subset R_{i_{m_2}} = R_{i_{m_1}}$, which is a vertical sub-rectangle of $R_{i_{m_1}}$.
	
	By uniform expansion in the vertical direction, the rectangle $R_{i_{m_1}}$ reduces to a stable interval, and this is not permitted in a Markov partition, leading us to a contradiction. 	
	\end{proof}
	
In view of Lemma $M:=\min\{m\in \NN: h_{i_m}>1\}$ exists. If there is $0\leq m\leq M$ such that $i_m\neq i'_m$, our proof is complete. If not, the following Lemma addresses the remaining situation.
	
	\begin{lemm}\label{Lemm: diferen positive code}
	Suppose that for all $0\leq m\leq M$, $i_m=i_m'$, then $i_{M+1}\neq i'_{M+1}$.
	\end{lemm}
	
	\begin{proof}
		
	Observe that for all $0\leq m\leq M$, if $\Gamma(T)^m(i_0,+1)=(i_m,\epsilon_m)$ then $$\Gamma(T)^m(i_0,-1)=(i_m,-\epsilon_m)$$.
	
	 They have the same index $i_m=i'_m$ but still have inverse $\epsilon$ part, i.e., $\epsilon_m=-\epsilon'_m$. In fact, the first term $\epsilon_0=-\epsilon'_0$, without lost of generality $\epsilon_0=1$  and $\epsilon'_0=-1$,   by hypothesis and considering that  $1=h_{i_0}=h_{i'_0}$ we infer that:
		$$
		\theta_T(i_0,\epsilon_0)=(i_0,h_{i_m})=(i'_0,1)=\theta_T(i'_0,\epsilon_0'),
		$$
Therefore:	
		$$
		\epsilon_1=\epsilon_0\cdot \epsilon_T(i_0,h_{i_0})=-\epsilon'_0\epsilon_T(i'_0,1)=-\epsilon'_1.
		$$
		then continue the argument by induction. In particular: $\Gamma(T)^M(i_0,1)=(i_M,\epsilon_M)$ and  $\Gamma(T)^M(i'_0,-1)=(i_M,-\epsilon_M)$.
		
		 The incidence matrix of $T$ have $\{0,1\}$ has coefficients ${0,1}$, and since that $1\neq h_{i_M}$,  if $\rho_T(i_M,1)=(k,l)$ then $\rho_T(i_M,h_{i_M})=(k',l')$ where $k\neq k'$.
		 
	Consider the  case when, $\theta_T(i_M,\epsilon_M)=(i_M,1)$ and  $\theta_T(i_M,\epsilon'_M)=(i_M, h_{i_M})$. Lets apply the formula of $\Gamma(T)$:
	$$
	\Gamma(T)^{M+1}(i_0,1)=\Gamma(T)(i_M,\epsilon_M)=(\xi_T(i_M,1),\epsilon_M\cdot \epsilon_T(i_M,1))=(k,\epsilon_{M+1})
	$$
	and 				
		$$
	\Gamma(T)^{M+1}(i_0,-1)=\Gamma(T)(i_M,-\epsilon_M)=(\xi_T(i_M,h_{i_M}),-\epsilon_M\cdot \epsilon_T(i_M,h_{i_M}))=(k',\epsilon'_{M+1})
	$$	
	therefore $i_{M+1}=k\neq k'=i'_{M+1}$.
	
		 The situation $\theta_T(i_M,\epsilon_M)=(i_M,h_{i_M})$ and  $\theta_T(i_M,\epsilon'_M)=(i_M, 1)$ is treated similarly. The has been lemma  proved.

	\end{proof}

The proposition follows from the previous lemma. The result for negative codes associated with $u$-boundary labels is proven using a fully symmetric approach with the $u$-generating function.
\end{proof}

If $\Gamma(T)(i,\epsilon)=(i_1,\epsilon_1)$, applying the shift to this code gives another $s$-boundary positive code. That is clearly given by 
$$
\sigma(\underline{I}^+(i,\epsilon))=\underline{I}^+(i_1,\epsilon_1),
$$
where $\Gamma(T)(i_0,\epsilon_0)=(i_1,\epsilon_1)$. Since there are $2n$ different $s$-boundary positive codes, there exist natural numbers $k_1\neq k_2$ with $k_1,k_2\leq 2n$ such that $\sigma^{k_1}(\underline{I}^+(i,\epsilon))=\sigma^{k_2}(\underline{I}^+(i,\epsilon))$, and the code $\underline{I}^+(i,\epsilon)$ is pre-periodic. This implies the following corollary.

\begin{coro}\label{Coro: preperiodic finite s,u boundary codes}
There are exactly $2n$ different $s$-boundary positive codes and $2n$ different $u$-boundary negative codes. Furthermore, every $s$-boundary positive code and every $u$-boundary negative code is pre-periodic under the action of the shift $\sigma$. 

Moreover, for every $s$-boundary positive code $\underline{I}^+(i,\epsilon)$, there exists $k\leq 2n$ such that $\sigma^k(\underline{I}^+(i,\epsilon))$ is periodic. Similarly, for every $u$-boundary negative code $\underline{J}^-(i,\epsilon)$, there exists $k\leq 2n$ such that $\sigma^{-k}(\underline{J}^-(i,\epsilon))$ is periodic.
\end{coro}

Now we are ready to define a family of admissible codes that project onto the stable and unstable leaves of periodic boundary points of $\cR$.

\begin{defi}\label{Defi: s,u-boundary codes}
	The set of $s$-\emph{boundary codes} of $T$ is:
	\begin{equation}
	\underline{\cS}(T):=\{\underline{w}\in \Sigma_A: \underline{w}_+\in \underline{\cS}^{+}(T)\}.
	\end{equation}
The set of  $u$-\emph{boundary codes} of $T$ is 
\begin{equation}
\underline{\cU}(T):=\{\underline{w}\in \Sigma_A: \underline{w}_-\in \underline{\cU(T)}^{-}(T)\}.
\end{equation}
\end{defi}

Next, Proposition \ref{Prop: positive codes are boundary}  states that the projection of $s$-boundary codes of $T$ through $\pi^f$ is always contained in the stable boundary of the Markov partition $(f,\cR)$. Similarly, Proposition \ref{Prop: boundary points have boundary codes} ensures that these are the only codes in $\Sigma_A$ that project to the stable boundary. This provides us with a symbolic characterization of the boundary of the Markov partition. As we have done so far, we will provide a detailed proof for the stable case, noting that the unstable version follows by symmetry.

\begin{prop}\label{Prop: positive codes are boundary}
Let $(f,\cR)$ be a pair consisting of a homeomorphism and a partition that realizes $T$. Suppose $(i,\epsilon)\in \cS(T)$ is an $s$-boundary label of $T$, and let $\underline{w}\in \underline{\cS}(T)$ be a code such that $\underline{I}^+(i,\epsilon)=\underline{w}_+$. Then, $\pi_f(\underline{w})\in \partial^s_{\epsilon} R_{i}$.

Similarly, if $(i,\epsilon)\in \cU(T)$ is a $u$-boundary label of $T$, and $\underline{w}\in \underline{\cU}(T)$ is a code such that $\underline{J}^-(i,\epsilon)=\underline{w}_-$, then $\pi_f(\underline{w})\in \partial^u_{\epsilon}R_i$.

\end{prop}

\begin{proof}
	Make $(i,\epsilon)=(i_0,\epsilon_0)$ to make coherent the following notation and $\underline{I}^+(i,\epsilon)=\{i_m\}_{m\in \NN}$. 	For every $s\in \NN$ define the rectangles $\overset{o}{H_s}=\cap_{m=0}^s f^{-1}(\overset{o}{R_{i_m}})$. The limit of the closures of such rectangles when $s$ converge to infinity is a unique stable segment of $R_{i_0}$. The proof is achieved  if $\partial^s_{\epsilon_0}R_{i_0}\subset H_s:=\overline{ \overset{o}{H_s} }$ for all $s\in \NN$. In this order of ideas is enough to argument that for all $s\in \NN$:
	$$
	f^s(\partial^s_{\epsilon_0}R_{i_0}) \subset R_{i_{s}}.
	$$
	We are going to probe  by induction over $s$ something more specific:
	$$
	f^s(\partial_{\epsilon_0} R_{i_0})\subset \partial_{\epsilon_s}R_{i_s}.
	$$  
	
\emph{Base of induction}: For $s=0$. This is the case as $f^0(\partial^s_{\epsilon_0}R_{i_0})\subset \partial^s_{\epsilon_0}R_{i_0}$.
	
\emph{Hypothesis of induction }: Assume  that $f^s(\partial_{\epsilon_0} R_{i_0})\subset \partial_{\epsilon_s}R_{i_s}$ and  $f^s(\partial^s_{\epsilon_0}R_{i_0})\subset R_{i_s}$.

\emph{Induction step}: We are going to prove that
$$
f^{s+1}(\partial_{\epsilon_0} R_{i_0})\subset \partial_{\epsilon_{s+1}}R_{i_{s+1}}
$$ 

 and then  that $f^{s+1}(\partial^s_{\epsilon_0}R_{i_0})\subset R_{i_{s+1}}$. For that reason consider the two following cases:
	
	\begin{itemize}
		\item $\epsilon_{s}=1$. In this situation $f^s(\partial^s_{\epsilon_0}R_{i_0})\subset H^{i_s}_{h_{i_s}}$.  Hence  $f^{s+1}(\partial^s_{i_0}R_{i_0})\subset f(H^{i_s}_{h_{i_s}}) \subset R_{i'_{s+1}}$. Where $R_{i'_{s+1}}$ is the only rectangle such that $\xi_T(i_s,h_{i_s})=(i'_{s+1})$.		
		Even more $f^{s+1}(\partial^s_{\epsilon_0} R_{i_0})\subset \partial^s_{\epsilon'_{s+1}} R_{i'_{s+1}}$, where $\epsilon'_{s+1}$ obey to the formula $\epsilon'_{s+1}= \epsilon_T(i_s,h_{i_s})=\epsilon_s \cdot \epsilon_T(i_s,h_{i_s})$.
		
		\item  $\epsilon_{s}=-1$. In this situation  $f^s(\partial^s_{\epsilon_0}R_{i_0})\subset H^{i_s}_1$. Hence $f^{s+1}(\partial^s_{i_0}R_{i_0})\subset f(H^{i_s}_{1}) \subset R_{i'_{s+1}}$, where $R_{i_{s+1}}$ is the only rectangle such that $\xi_T(i_s,1)=(i'_{s+1}$. 
		Even more, $f^{s+1}(\partial^s_{\epsilon_0} R_{i_0})\subset \partial^s_{\epsilon'_{s+1}} R_{i'_{s+1}}$, where  $\epsilon'_{s+1}$ obey to the formula $\epsilon'_{s+1}= -\epsilon_T(i_s,1)=\epsilon_s \cdot \epsilon_T(i_s,1)$.
	\end{itemize}
	In bot situations:
	$$
	f^{s+1}(\partial^s_{\epsilon_0})R_{i_0})\subset \partial^s_{\epsilon'_{s+1}}R_{i'_{s+1}}
	$$
and they follows the rule:
	$$
	(i'_{s+1},\epsilon'_{s+1})=(\xi_T(i_s,\theta_T(i_s,\epsilon_s)),\epsilon_s\cdot \epsilon_T(i_s,\theta_T(\epsilon_s)))=\Gamma(T)^{s+1}(i_0,\epsilon_0)=(i_{s+1},\epsilon_{s+1}).
	$$
Therefore $f^{s+1}(\partial_{\epsilon_0} R_{i_0})\subset \partial_{\epsilon_{s+1}}R_{i_{s+1}}$, as we claimed and the result is proved.
The unstable case is totally symmetric.

\end{proof}

Proposition \ref{Prop: positive codes are boundary} provides justification for naming the $s$-boundary and $u$-boundary labels of $T$ as such, as they generate codes that are projected to the boundary of the Markov partition. The next proposition asserts that these codes are the only ones that have such a property.

\begin{prop}\label{Prop: boundary points have boundary codes}
	If $\underline{w}\in \Sigma_A$ projects to the stable boundary of the Markov partition $(f,\cR)$ under $\pi_f$, i.e., $\pi_f(\underline{w})\in \partial^s \cR$, then it follows that $\underline{w}\in \underline{\cS}(T)$. Similarly, if $\pi_f(\underline{w})\in \partial^u\cR$, then $\underline{w}\in \underline{\cU}(T)$.
	
\end{prop}

\begin{proof}

Like $A(T)$ has coefficients ${0,1}$, the sequence $\underline{w}_+$ determines, for all $m\in \mathbb{N}$, a pair $(w_m,j_m)\in \mathcal{H}(T)$ such that $\xi_T(w_m, j_m)=w_{m+1}$ and a number $\epsilon_T(w_m,j_m)=\epsilon_{m+1} \in {1,-1}$. It is important to note that, like  $f^m(x)\in \partial^s\cR$, in fact,  there are only two  cases (unless $h_{w_m}=1$ where they are the same):
$$
w_{m+1}=\xi_T(w_m,1) \text{ or well } w_{m+1}=\xi_T(w_m,h_{w_m}).
$$

This permit to define $\epsilon_m\in \{-1,+1\}$ as the only number such that:
\begin{equation}\label{Equa: determine epsilon m}
w_{m+1}=\xi_T(w_m,\theta_T(w_m,\epsilon_m)).
\end{equation}
Even more $\epsilon_m$ determine $\epsilon_{m+1}$ by the formula:
$$
\epsilon_{m+1}=\epsilon_m\cdot \epsilon_T(w_m,\theta_T(w_m,\epsilon_m)).
$$
In resume:
\begin{equation}\label{Equa: w determine by gamma}
\Gamma(T)(w_m,\epsilon_m)=(w_{m+1},\epsilon_{m+1})
\end{equation}
follow the rule dictated by the $s$-generating function. So if we know $\epsilon_M$ for certain $M\in \NN$ we can determine  $\sigma^{M}(\underline{w})_+=\underline{I}^+(w_M,\epsilon_M)$, so it rest to determine at least one $\epsilon_M$. 

Lemma \ref{Lemm: positive h-i} can be adapted to this context to prove the existence of a minimal $M\in \mathbb{N}$ such that $h_{w_M}>1$. Then, equation \ref{Equa: determine epsilon m} determines $\epsilon_M$. Now, we need to recover $\epsilon_0$, but we proceed backwards. Since $h_{w_m}=1$ for all $m<1$, we have:

$$
\epsilon_M=\epsilon_{M-1}\cdot \epsilon_T(w_{M-1},1)
$$
and then $\epsilon_{M-1}=\epsilon_M \cdot \epsilon_T(w_{M-1},1)$. By applied this procedure we can determine $\epsilon_0$ and then using \ref{Equa: w determine by gamma} to get that
$$
\Gamma(T)^m(w_0,\epsilon_0)=(w_m,\epsilon_m).
$$
The conclusion is that $\underline{w}_+=\underline{I}^{+}(w_0,\epsilon_0)$. 

The unstable boundary situation is analogous.
\end{proof}

Therefore, $\underline{\cS}(T)$ and $\underline{\cU}(T)$ are the only admissible codes that project to the boundary of a Markov partition. With this in mind, we can distinguish the periodic boundary codes from the non-periodic ones. It is important to note that the cardinality of each of these sets is less than or equal to $2n$.

\begin{defi}\label{Defi: s,u-boundary periodic codes}
	The set of $s$-\emph{boundary periodic codes} of $T$ is:
	\begin{equation}
\text{ Per }(\underline{\cS(T)}):=\{\underline{w}\in \underline{\cS}(T): \underline{w} \text{ is periodic }\}.
	\end{equation}
	The set of  $u$-\emph{boundary periodic codes} of $T$ is 
	\begin{equation}
\text{ Per }(\underline{\cU(T)}):=\{\underline{w}\in \underline{\cU}(T): \underline{w} \text{ is periodic }\}.
	\end{equation}
\end{defi}

\subsection{Decomposition of $\Sigma_{A}$.}
 Now we can describe the stable and unstable leaves of $f$ corresponding to the periodic points of the boundary.

\begin{defi}\label{Defi: stratification Sigma A}
	We define the $s$-\emph{boundary leaves codes} of $T$:
	\begin{equation}
	\Sigma_{\cS(T)}=\{\underline{w}\in \Sigma_A:   \exists k\in\NN \text{ such that } \sigma^k(\underline{w})\in \underline{\cS(T)}\}.
	\end{equation}
	We define the $u$-\emph{boundary leaves codes} of $T$
	\begin{equation}
	\Sigma_{\cU(T)}=\{\underline{w}\in \Sigma_A:  \exists k\in\NN \text{ such that } \sigma^{-k}(\underline{w})\in \underline{\cU(T)}\}.
	\end{equation}
	Finally, the \emph{totally interior codes} of $T$ are
	$$
	\Sigma_{\cI nt(T)}=\Sigma_A\setminus(\Sigma_{\cS(T)} \cup \Sigma_{\cU(T)}).
	$$
\end{defi}

The importance of such a division of $\Sigma_A$ is that its projections are well determined, as indicated by the following lemma.

\begin{lemm}\label{Lemm: Projection Sigma S,U,I}

	A code $\underline{w}\in \Sigma_A$ belongs to $\Sigma_{\cS(T)}$ if and only if its projection $\pi_f(\underline{w})$ is within the stable leaf of a boundary periodic point of $\cR$.
	
A code $\underline{w}\in \Sigma_A$ belongs to $\Sigma_{U(T)}$ if and only if its projection $\pi_f(\underline{w})$ lies within the unstable leaf of a boundary periodic point of $\cR$.

	A code $\underline{w}\in \Sigma_A$ belongs to $\Sigma_{\cI nt(T)}$ if and only if its projection $\pi_f(\underline{w}$ is contained within Int$(f,\cR)$>
\end{lemm}

\begin{proof}
The $s$-boundary leaf codes satisfy Definition  \ref{Defi: s,u-leafs}, and according to Proposition \ref{Prop: Projection foliations}, if $\underline{w}\in \Sigma_{\cS(T)}$, their projection lies on the same stable manifold as an $s$-boundary component of $\cR$. Stable boundary components of a Markov partition represent the stable manifold of an $s$-boundary periodic point, and for each $\underline{w}\in \Sigma_{\cS(T)}$, there exists $k=k(\underline{w})\in \mathbb{N}$ such that $\sigma^k(\underline{w})_+$ is a periodic positive code (Corollary \ref{Coro: preperiodic finite s,u boundary codes}). This positive code corresponds to a periodic point on the boundary of $\cR$ within whose stable manifold $\pi_f(\underline{w})$ is contained. This proves  one direction in the first assertion  of the lemma.

If $\underline{v}$ is a periodic boundary code ,$\pi_f(\underline{v})\in \partial^s\cR$, by definition $\underline{v}\in \underline{\cS(T)}$.  Suppose that $\pi_f(\underline{w})$ is on the stable leaf of $\underline{v}$, then there exist $k \in \NN$ such that $\pi_f(\sigma^k(\underline{w}))$ is on the same stable boundary of $\cR$ as $\pi_f(\underline{v})$.  The proposition\ref{Prop: boundary points have boundary codes} implies that $\sigma^k(\underline{w})\in \underline{\cS(T)}$, hence  $\underline{w}\in \Sigma_{\cS(T)}$ by definition. This completes the proof of the first assertion of the lemma. A similar argument proves the unstable case.	
	
 The conclusion of these items is that  $\underline{w}\in \Sigma_{\cS(T)}\cap \Sigma_{\cU(T)}$ if and only if $\pi_f(\underline{w})\in \cF^{s,u}(\text{ Per }^{s,u}(\cR))$. 
 
   As proved in the Lemma \ref{Lemm: Caraterization unique codes} totally interior points are disjoint from the stable and unstable lamination generated by boundary periodic points. If $\underline{w}\in  \Sigma_{\cI nt(T)}$ and $\pi_f(\underline{w})$ is on the stable or unstable leaf of a $s,u$-boundary periodic point, we have seen that $\underline{w}\in \Sigma_{\cS(T),\cU(T)}$ which is not possible, therefore $\pi_f(\underline{w})\in $Int$(f,\cR)$. In the conversely direction, if  $\pi_f(\underline{w})$ is not in the stable or unstable lamination of $s,u$-boundary points, $\underline{w}\notin (\Sigma_{\cS(T)})\cup \Sigma_{\cU(T)}$, so $\underline{w}\in\Sigma_{\cI nt(T)}$. This ends the proof.
 \end{proof}
We have obtained the decomposition 
$$
\Sigma_{A(T)}=\Sigma_{\cI nt(T)} \cup \Sigma_{\cS(T)} \cup \Sigma_{\cU(T)}
$$ 
 and have characterized the image under $\pi_f$ of each of these sets. In the next subsection we use this decomposition to define relations on each of these parts and then extend them to an equivalence relation in $\Sigma_A$.

\subsection{Relations in $\Sigma_{\cS(T)}$ and $\Sigma_{\cU(T)}$}

Let $\underline{w}\in \Sigma_{\cS(T)}\setminus$Per$(\sigma_A)$ be non-periodic $s$-boundary leaf code of $T$. Since it is not a periodic boundary points of $T$ there exists a number $k:=k(\underline{w})\in \ZZ$ with the following properties: 
\begin{itemize}
\item $\sigma^k(\underline{w})\notin \underline{\cS(T)}$, but
\item $\sigma^{k+1}(\underline{w})\in \underline{\cS(T)}$
\end{itemize}
 i.e. $\pi_f(\sigma^k(\underline{w}))\notin\partial^s\cR$ but $\pi_f(\sigma^{k+1}(\underline{w}))\in\partial^s\cR$. This easy observation lead to the next Lemma.

\begin{lemm}\label{Lemm: minumun k}
The number $k:=k(\underline{w})$ is the unique integer such that $f^{k}(\pi_f(\underline{w})) \in \cR\setminus\partial^s\cR$ and for all $k'>k(\underline{w})$,  $f^{k'}(\pi_f(\underline{w}))\in \partial^s\cR \setminus$Per$(f)$.
\end{lemm}

Let $x=\pi_f(\underline{w})$, as consequence of Lemma \ref{Lemm: minumun k} there are indices $i\in \{1,\cdots,n\}$ and $j\in \{1,\cdots,h_{i}-1\}$ such that:
\begin{itemize}
\item $f^k(x)$ is the rectangle $R_i$ but not in its stable boundary, i.e. $f^k(x)\in R_{i}\setminus \partial^s R_i$ 
\item $f^k(x)$ is in two adjacent horizontal sub-rectangles of $R_i$, i.e. $f^k(x)\in \partial^s_{+1} H^i_j$ and $x\in \partial_{-1}^s H^i_{j+1}$
\end{itemize}

As $f^{k+1}(x)\in \partial^s \cR$, the rectangles  $H^i_j$ and $H^i_{j+1}$ determine the following conditions:

\begin{itemize}
\item As $f^k(x)\in \partial^s_{+1}H^i_j$,  we have that $f^{k+1}(x)\in \partial^s_{\epsilon_0} R_{i_0}$ where:
\begin{equation}
i_0=\xi_T(i,j) \text{ and  } \epsilon_0=\epsilon_T(i,j)
\end{equation}
\item  As $f^k(x)\in \partial^s_{-1}H^i_{j+1}$ we have that $f^{k+1}(x)\in \partial^s_{\epsilon'_0} R_{i'_0}$ where
\begin{equation}
i_0=\xi_T(i,j+1)  \text{ and  } \epsilon'_0=-\epsilon_T(i,j+1)
\end{equation}
\end{itemize}

This mechanism of identification is illustrated in Figure  \ref{Fig: stable identification}
\begin{figure}[h]
	\centering
	\includegraphics[width=0.6\textwidth]{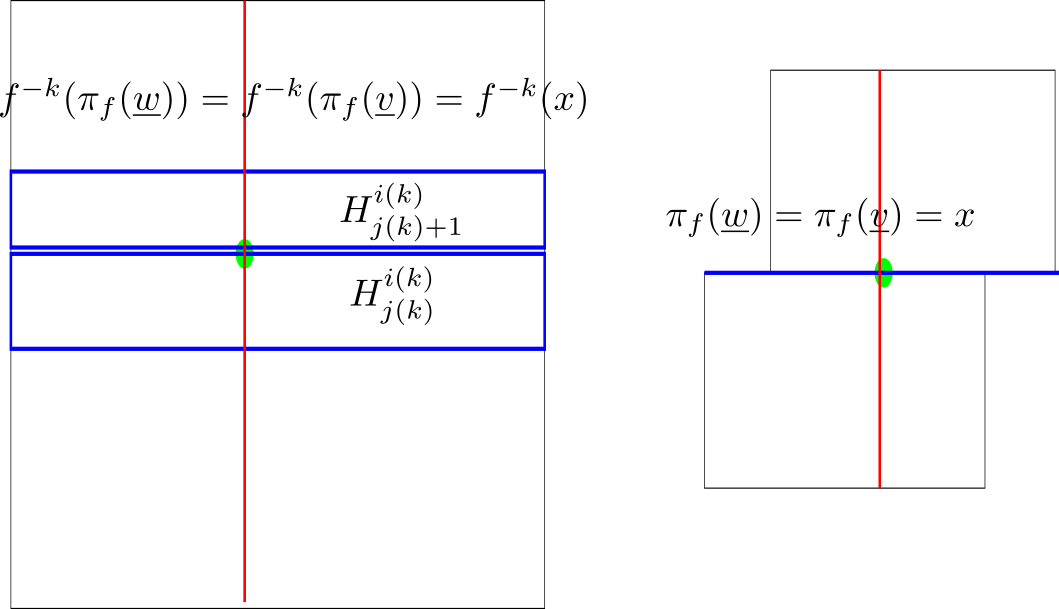}
	\caption{The stable identification mechanism}
	\label{Fig: stable identification}
\end{figure}

This analysis suggest the following definition

\begin{defi}\label{Defi: s realtion in Sigma S no per}
Let $\underline{w},\underline{v}\in \Sigma_{\cS(T)}\setminus$Per$(\sigma_A)$, they are $s$-related and its write $\underline{w}\sim_{s} \underline{v}$ if they are equals or: There exist  $k\in \ZZ$ such that
\begin{itemize}
\item[i)]  $\sigma^k(\underline{w}),\sigma^k(\underline{w})\notin \underline{\cS(T)}$ but $\sigma^{k+1}{\underline{w}},\sigma^{k+1}(\underline{v})\in \underline{\cS(T)}$ 

\item[ii)] The zero terms of $\sigma^k(\underline{w})$ and  $\sigma^k(\underline{v})$ are equals.i.e  $v_k=w_k$ and $h_{w_k}=h_{v_k}>1$. 

\item[iii)] Lets take $i=v_k=w_k$. There exist $j\in \{1, \cdots,h_i -1\}$ such that only one of the next options occurs:
\begin{equation}\label{Equa: Sim s obtion 1}
\xi_T(i,j)=w_{k+1} \text{ and then }  \xi_T(i,j+1)=v_{k+1},
\end{equation}
or
\begin{equation}\label{Equa: Sim s obtion 2}
\xi_T(i,j)=v_{k+1} \text{ and then }  \xi_T(i,j+1)=w_{k+1}.
\end{equation}

\item[iv)] Suppose the positive code of $\sigma^{k+1}(\underline{w})$  is equal to the $s$-boundary code $\underline{I}^{+}(w_{k+1},\epsilon_w)$ and the positive code of $\sigma^{k+1}(\underline{v})$  is equal to the $s$-boundary code $\underline{I}^{+}(v_{k+1},\epsilon_v)$, then:

If equation \ref{Equa: Sim s obtion 1} happen:
\begin{equation}\label{Equa: sim s epsilon 1}
\epsilon_w=\epsilon_T(w_{k},j) \text{ and } \epsilon_v=-\epsilon_T(v_{k},j+1)
\end{equation}
If equation \ref{Equa: Sim s obtion 2} happen:
\begin{equation}\label{Equa: sim s epsilon 2}
\epsilon_v=\epsilon_T(v_{k},j) \text{ and } \epsilon_w=-\epsilon_T(w_{k},j+1)
\end{equation}

\item[v)] The negative codes  $\sigma^{k}(\underline{w})_{-}$ and $\sigma^{k}(\underline{v})_{-}$ are equal.
\end{itemize}
\end{defi}

Our challenge is to show that $\sim_s$ is an equivalence relation in $\Sigma_{\cS(T)}$ and that two $\sim_s$-related  codes project by $\pi_f$ to the same point. The basis of the argument will be that when $\underline{w},\underline{v}\in \underline{\cS(T)}$ and are $\sim_s$-related, they project to adjacent stable boundaries of the Markov partition.  We subsequently can generalize the result to the whole set $\Sigma_{\cS(T)}\setminus$Per$(\sigma_A)$. The situation with the periodic points will be discussed later.

\begin{prop}\label{Prop: sim s equiv in Sigma non per}
The relation $\sim_{s}$ is an equivalence relation in $\Sigma_{\cS(T)}\setminus$Per$(\sigma_A)$.
\end{prop}

\begin{proof}
Reflexivity and symmetry are fairly obvious from the definition and we will concentrate on transitivity. Lets to assume $\underline{w}\sim_{s}\underline{v}$ and $\underline{v}\sim_s\underline{u}$. 

The number $k\in \ZZ$ of item $i)$  is unique because it is characterized as:
$$
k(\underline{w})=\min\{z\in \ZZ: \sigma^k(\underline{w})\in \underline{w}\}-1,
$$
 therefore $k:=k(\underline{w})=k(\underline{v})=\underline{u}$ is the same for the three codes.

The $k$-terms of such codes are equal too, as $w_k=v_k$ and $v_k=u_k$, lets take $i=v_k=w_k=u_k$. Without lost of generality we could assume that $j\in \{1,\cdots,h_{i}-1\}$ is such that:
$$
\xi_T(i,j+1)=w_{k+1} \text{ and } \xi_T(i,j)=v_{k+1},
$$ 

Suppose too that $\sigma^{k+1}(\underline{w})_+=\underline{I}^+(w_{k+1},\epsilon_w)$  and $\sigma^{k+1}(\underline{v})_+=\underline{I}^+(v_{k+1},\epsilon_v)$. They satisfy the relations:
$$
\epsilon_w=-\epsilon_T(i,j+1) \text{ and } \epsilon_v=\epsilon_T(i,j).
$$

Like  $\underline{v}\sim_s \underline{u}$ there is a unique $j'\in \in \{1,\cdots,h_{i}-1\}$ such that:
$$
\xi_T(i,j')=u_{k+1}
$$
but the relation, $\underline{v}\sim_s \underline{u}$ implies that, $j'=j+1$ or $j-1$. If we prove that $j'=j+1$ necessarily $\underline{u}=\underline{w}$ and we have finished.   Lets to analyses the situation of $j'=j-1$.

 Suppose $\sigma^{k+1}(\underline{u})_+=\underline{I}^+(u_{k+1},\epsilon_u)$, as $j=(j-1)+1$ we have in the situation of Equation \ref{Equa: Sim s obtion 1} and then we applied Equation \ref{Equa: sim s epsilon 1} to obtain that:
$$
\epsilon_v= -\epsilon_T(i,j) \text{ and } \epsilon_u=\epsilon_T(i,j-1).
$$
Therefore, $\epsilon_T(i,j)=-\epsilon_T(i,j)$ and this is a direct contradiction.  Then $j'=j+1$ and the positive part of $\sigma^k(\underline{w})$ coincides with the positive part of $\sigma^k(\underline{u})$.
 
  Item $v)$ implies, that the negative part of $\sigma^{k}(\underline{w})$ is equal to the negative part of $\sigma^{k}(\underline{v})$, and  the negative part of   $\sigma^{k}(\underline{v})$ is equal to the negative part of $\sigma^{k}(\underline{u})$. Hence, the negative part of   $\sigma^{k}(\underline{w})$ is equal to the negative part  of $\sigma^{k}(\underline{u})$, which implies that $\underline{w}\sim_s \underline{u}$. 
\end{proof}

\begin{rema}\label{Rema: two codes in sim s}
We deduce from the proof that the equivalence class of every code $\underline{w}\in \Sigma_{\cS(T)}\setminus$Per$(\sigma_A)$ have at most $2$ elements.  In fact it has exactly $2$-elements. The number $k(\underline{w})$ always exists  and as $\sigma^k(\underline{w})\notin \underline{\cS(T)}$, its projection is in the complement of $\partial^s\cR$, this implies $\pi_f(\sigma^k{\underline{w}})\notin \partial^s\cR$ but  $\pi_f(\sigma^{k+1}{\underline{w}})\in \partial^s\cR$ therefore is in the boundary of two consecutive horizontal sub-rectangles of $R_{w_k}$, this gives rise to two  different codes $\sim_s$-related.
\end{rema}

It remains to extend the relation $\sim_s$ to the $s$-boundary periodic codes that were defined in \ref{Defi: s,u-boundary periodic codes} by Per$(\underline{\cS(T)})$. Observe first that, Per$(\sigma)\cap\Sigma_{\underline{\cS(T)}}=$Per$(\underline{\cS(T)})$ because if $\underline{w}\in$Per$(\sigma)\cap\Sigma_{\cS(T)}$ there is $K\in \ZZ$ such that for all $k\geq K$, $\sigma^k(\underline{w})\in \underline{\cS(T)}$ but there exist $r>k$ such that $\sigma^r(\underline{w})=\underline{w}$, so $\underline{w}\in \underline{\cS(T)}$ since the beginning. The relation $\sigma_s$ could be extended to $\Sigma_S$ as follows.

\begin{defi}\label{Defi: sim s in per}
Let $\underline{\alpha},\underline{\beta}\in$Per$\underline{\cS(t)}$ be  $s$-boundary periodic codes. They are $s$-related, $\underline{\alpha}\sim_s \underline{\beta}$,  if and only if they are equal or there are $\underline{w},\underline{v} \in \Sigma_{\cS(T)}\setminus$Per$(\sigma)$ such that:
\begin{itemize}
	\item $\underline{w}\sim_{s}\underline{v}$,
	\item There exists $k\in \ZZ$ such that,  $\sigma^k(\underline{w})_+=\underline{\alpha}_+$ and   $\sigma^k(\underline{v})_+=\underline{\beta}_+$.
\end{itemize}
\end{defi}

\begin{prop}\label{Prop: sim s equiv in Sigma S }
The relation $\sim_s$ in $\Sigma_{\cS(T)}$ is an equivalence relation.
\end{prop}

\begin{proof}
Two codes in $\Sigma_{\cS(T)}$ are $s$-related if and only if both are periodic or non-periodic.We have already determined the non-periodic situation, it remains to analyze what happens in the periodic setting. But reflexivity is for free, while symmetry and transitivity are inherited from the relation in $\Sigma_{\cS(T)}\setminus$Per$(\sigma)$ to the periodic codes in view of the definition we have given.
\end{proof}

Now we are ready to see that $\sim_s$ related codes  are projected to the same point.

\begin{prop}\label{Prop: s-relaten implies same projection}
Let $\underline{w},\underline{v}\in \Sigma_{\cS(T)}$ be to $s$-boundary leaves codes. If  $\underline{w}\sim_{s}\underline{v}$ then $\pi_f(\underline{w})=\pi_f(\underline{v})$.
\end{prop}

\begin{proof}

Assume that $\underline{w},\underline{v}$ are $s$-related. For simplicity assume that the $k$ in the definition is equal to zero. Like the negative part of such codes is equal $\pi_f(\underline{w})$ and $\pi_f(\underline{v})$ are in the same unstable segment of $R_{w_0}$. Take $w_0=v_0=i$.
		
	We shall prove that $x_w:=\pi_f(\underline{w})$ and $x_v:=\pi_f(\underline{v})$ are in the same stable segment of $R_i$. We can assume that $x_w\in H^{i}_{j+1}$ and $x_v \in H^{i}_{j}$, condition imposed by Item $iii)$ of Definition \ref{Defi: s realtion in Sigma S no per}.
	
	The point $x_w$ is on the stable boundary $\delta_w=+1,-1$ of $H^{i}_{j+1}$ because its image is on the stable boundary of $\cR$. If $\delta_w=+1$ then, following the dynamics of $f$, the point $f(x_w)$ is the boundary $\epsilon_w=\epsilon_T(i,j+1)$, which is not possible since $\epsilon_w=-\epsilon_T(i,j+1)$ by definition of the relation $\sim_s$. Then $\delta_w=-1$ and $x_w$ is on the lower boundary of $H^i_{j+1}$.
	
	Analogously $x_v$ is the boundary component $\delta_v=+1,-1$ of $H^{i}_j$. 	If $\delta_v=-1$,  following the dynamics of $f$, $f(x_v)$ is the boundary component of $R_{v_1}$ corresponding to $\epsilon_v=-\epsilon_T(i,j)$ which is a contradiction with the condition $\epsilon_v=\epsilon_T(i,j)$ imposed by $\sim_s$. Therefore $\delta_v=+1$ and is on the upper boundary of $H^i_j$.
	
	Then $\pi_f(\underline{w})$ and $=\pi_f(\underline{v})$ are on the lower boundary of $H^{i}_{j+1}$ and the upper boundary of $H^{i}_j$ respectively, which are the same stable segment of $R_i$, so $\pi_f(\underline{w})=\pi_f(\underline{v})$.

\end{proof}

In the same spirit of $\sim_s$ there is a equivalence relation $\sim_u$ for the elements in $\Sigma_{\cU(T)}$, there is a easy ways to determine such a relation that appeals to the idea of the inverse of the geometric type $T$ as introduced in \ref{Defi: inverse of Type}. As the Observation \ref{Defi: inverse of Type} says if $(f,\cR)$ represent $T$, then $(f^{-1},\cR)$ represent $T^{-1}$. In this setting $f^{-1}(V^k_l)=H^i_j$ if and only if $f(H^i_j)=V^k_l$ and $f^{-1}$ preserve the horizontal direction restricted to $V^k_l$ if and only if $f$ preserve the horizontal direction restricted to $H^i_j$. We summarize this discussion as follows:

\begin{lemm}\label{Lemm: geometric type of inverse}
If $\cR$ is a geometric Markov partition for the pseudo-Anosov homeomorphism $f$ with geometric type $T$, then:
\begin{itemize}
\item $\cR$ is a geometric Markov partition for $f^{-1}$ in which the horizontal and vertical directions of $(f,\cR)$ become the vertical and horizontal directions of $(f^{-1},\cR)$ respectively.
\item The geometric type of $\cR$ view as Markov partition for $f^{-1}$ is $T^{-1}$.
\end{itemize}
\end{lemm}

With this notion we could define the relation $\sim_u$ in $\Sigma_{\cU(T)}$ as the relation $\sim_s$ in $\Sigma_{\cS(T^{-1})}$, this formulation is sometimes useful. However, we prefer to give the following definition which is fully symmetric to  \ref{Defi: s realtion in Sigma S no per} and illuminates the mechanism by which codes are identified in $\Sigma_{\cU(T)}$, this insight will be use later when extending the $\sim_s$ and $\sim_u$ relations to $\Sigma_A$.

\begin{defi}\label{Defi: u realtion in Sigma U no per}
	Let $\underline{w},\underline{v}\in \Sigma_{\cU(T)}\setminus$Per$(\sigma)$, they are $u$-related and write $\underline{w}\sim_{u} \underline{v}$ if and only if they are equal or:
	\begin{itemize}
		\item[i)] There exist $z\in \ZZ$ such that $\sigma^z(\underline{w}),\sigma^z(\underline{v})\notin \underline{\cU(T)}$ but $\sigma^{z-1}(\underline{w}),\sigma^{z-1}(\underline{v})\in \underline{\cU(T)}$.
		
		\item[ii)] The $0$-terms, $\sigma^z(\underline{w})_0=w_z$ and  $\sigma^z(\underline{v})_0=v_z$ are equals, i.e. $w_z=v_z:=k\in\{1,\cdots,n\}$  and the number $v_k>1$ (given by $T$).
		
		\item[iii)] There is $l\in \{1, \cdots,v_{k}-1\}$ such that only one of the two possibilities happen:
	\begin{equation}\label{Equa: Sim u 1 option}
	\xi_{T^{-1}}(k,l)=(w_{z-1}) \text{ and  } \xi_{T^{-1}}(k,l+1)=(v_{z-1})
	\end{equation}
		or
	\begin{equation}\label{Equa: Sim u 2 option}
	\xi_{T^{-1}}(k,l)=(v_{z-1}) \text{ and  } \xi_{T^{-1}}(k,l+1)=(w_{z-1})
	\end{equation}
		
		\item[iv)] Suppose the negative code of $\sigma^{z-1}{\underline{w}}$  is equal to the $u$-boundary code $\underline{J}^-(w_{z-1},\epsilon_w)$ and  the negative code of  $\sigma^{z-1}(\underline{v})$  is equal to the $u$-boundary code $\underline{J}^-(v_{z-1},\epsilon_v)$, then:
		
		If Equation \ref{Equa: Sim u 1 option} holds:
	\begin{equation}\label{Equa: Sim u 1 epsilon}
\epsilon_w=\epsilon_{T^{-1}}(k,l) \text{ and } \epsilon_v=-\epsilon_{T^{-1}}(k,l+1)
	\end{equation}
		In the case Equation \ref{Equa: Sim u 2 option} happen:
	\begin{equation}\label{Equa: Sim u 2 epsilon}
		\epsilon_v=-\epsilon_{T^{-1}}(k,l) \text{ and } \epsilon_w=-\epsilon_{T^{-1}}(k,j+1)
	\end{equation}
		\item[v)] The positive codes of $\sigma^z(\underline{w})$ is equal to the positive code of $\sigma^z(\underline{v})$.
	\end{itemize}
\end{defi}

Similarly to Proposition \ref{Prop: sim s equiv in Sigma non per} we could prove $\sim_{u}$ is a equivalent relation in $\Sigma_{U}\setminus$Per$(\sigma_A)$ and then extend the relation to the periodic boundary points as in  \ref{Defi: sim s in per}.

\begin{defi}\label{Defi: sim u in per}
	Let $\underline{\alpha},\underline{\beta}\in$Per$(\Sigma_{\cU(T)})$ be periodic $u$-boundary codes. They are $u$-related if and only if there exist $\underline{w},\underline{v}\in \Sigma_{\cU(T)}\setminus$Per$(\sigma$ such that:
	\begin{itemize}
		\item $\underline{w}\sim_{u}\underline{v}$,
		\item There exist $p\in \ZZ$ such that the negative  codes  $\sigma^p(\underline{w})_-$ is equal to the negative code of $\underline{\alpha}$ and the negative code  $\sigma^p(\underline{v})_-$ coincide with the negative code of $\underline{\beta}$.
	\end{itemize}
\end{defi}

With the same techniques used in \ref{Prop: sim s equiv in Sigma S } we could prove 

\begin{prop}\label{Prop: sim u equiv in Sigma U }
	The relation $\sim_u$ in $\Sigma_{\cU(T)}$ is of equivalence.
\end{prop}

And similarly 

\begin{prop}\label{Prop: u-relaten implies same projection}
	Let $\underline{w},\underline{v}\in \Sigma_{\cU(T)}$ be two $s$-boundary leaves codes, if $\underline{w}\sim_{u}\underline{v}$ then $\pi_f(\underline{w})=\pi_f(\underline{v})$.
\end{prop}

We have defined $\sim_{s,u}$ in terms of the combinatorial information of the geometric type $T$ so that they do not depend on the homeomorphism-partition representation $(f,\cR)$ of $T$. Note that the Propositions \ref{Prop: s-relaten implies same projection} and \ref{Prop: u-relaten implies same projection} give a sufficient condition for two codes to project to the same point but this is not a necessary condition, since there could be more codes projecting to the same point, for this reason we need to complete our relations to $\sim_T$ in $\Sigma_A$.  We end this subsection by defining an equivalence relation $\sim_I$ on the totally interior codes $\Sigma_{\cI nt(T)}$. As proved in Lemma \ref{Lemm: Projection Sigma S,U,I} totally interior codes are the only codes that project to totally interior points of any representation $(f,\cR)$ of $T$, moreover any totally interior point of $(f,\cR)$ has associated to it a unique totally interior code that projects to it, this inspires the following definition.

\begin{defi}\label{Defi: I sim relation}
Let $\underline{w},\underline{v}\in \Sigma_{\cI nt(T)}$be two totally interior codes, they are $I$-related and we write $\underline{w}\sim_I\underline{v}$ if and only if $\underline{w}=\underline{v}$.
\end{defi} 

The following result follows from  Proposition \ref{Prop: Carterization injectivity of pif} where we have characterized totally interior points as having a single code projecting to them. The equivalent relation part is totally trivial because the relation $\sim_I$ is the equality between codes.

\begin{prop}\label{Prop: totally interior points projection sim I}
The relation $\sim_I$ is an equivalence relation in $\Sigma_{\cI nt(T)}$ and two codes $\underline{w},\underline{v}\in \Sigma_{\cI nt(T)}$ are $\sim_I$ related if and only if $\pi_f(\underline{w})=\underline{v}$, i.e they project to the same point.
\end{prop}

\subsection{ The equivalent relation $\sim_T$ on $\Sigma_{A(T)}$}

Finally we are ready to define the relation $\sim_T$ in $\Sigma_A$ that Proposition \ref{Prop: The relation determines projections} claim to exist, it consist essentially in the relation generated by $\sim_s,\sim_u$ and $\sim_I$.

\begin{defi}\label{Defi: Sim-T equivalent relation}
Let $\underline{w},\underline{v}\in \Sigma_{A}$ they are $T$-related and write $\underline{w}\sim_T\underline{v}$ if and only if any of the following disjoint situations occurs:
\begin{itemize}
	\item[i)] $\underline{w},\underline{v}\in \Sigma_{\cI nt(T)}$ and $\underline{w}\sim_I\underline{v}$, i.e. they are equals.
	\item[ii)] $\underline{w},\underline{v}\in \Sigma_{\cS(T)} \cup \Sigma_{\cU(T)}$ and there exist a finite number of codes $\{\underline{x_i}\}_{i=1}^m \subset  \Sigma_{\cS(T)} \cup \Sigma_{\cU(T)}$ such that:
\begin{equation}
	\underline{w}\sim_{s}\underline{x_i}\sim_{u} \underline{x_2}\sim_{s} \cdots \sim_{s} \underline{x_m}\sim_{u} \underline{v},
\end{equation}
or
\begin{equation}
\underline{w}\sim_{u}\underline{x_i}\sim_{s} \underline{x_2}\sim_{u} \cdots \sim_{u} \underline{x_m}\sim_{u} \underline{v}.
\end{equation}
\end{itemize}
\end{defi}

\begin{prop}
The relation $\sim_T$ is an equivalence relation in $\Sigma_A$.
\end{prop}

\begin{proof}
For  $\underline{w}\in \Sigma_{\cI nt(T)}$ the relation is reflexive, symmetric and transitive by virtue of equality. 
In the situation $\underline{w},\underline{v} \Sigma_{\cS(T)} \cup \Sigma_{\cU(T)}$, the reflexivity and symmetry are free, the transitivity comes from the concatenation of the codes and the fact that $\sim_{s}$ and $\sim_{u}$ are transitive as well.

\end{proof}

 If $\underline{w}\sim_{s,u}\underline{v}$  then $\sigma(\underline{w})\sim_{s,u}\sigma(\underline{v})$ because the numbers $k,z\in \ZZ$ at Item $i)$ of the definitions \ref{Defi: s realtion in Sigma S no per} and  \ref{Defi: u realtion in Sigma U no per} are changed to $k-1$ and $z-1$ respectively and the rest of the conditions hold for $\sigma(\underline{w})$ and $\sigma(\underline{v})$. The same property holds for the relation $\sim_I$. This remark implies the following lemma which will simplify some arguments.

\begin{lemm}\label{lemma: simT related iterations related}
If two codes $\underline{w},\underline{v}\in \Sigma_A$ are $\sim_T$ related then for all $k\in \ZZ$, $\sigma^k(\underline{w})\sim_T\sigma^k(\underline{v})$.
\end{lemm}

The only property that remains to be corroborated in the relation $\sim_T$ to obtain Proposition\ref{Prop: The relation determines projections}   is the one that relates it to the projection. That is the content of the following result.

\begin{prop}\label{Prop: T related iff same projection}
Let $\underline{w},\underline{v}\in \Sigma_A$ be any admissible codes. Then $\underline{w} \sim_T\underline{v}$ if and only if $\pi_f(\underline{w})=\pi_f(\underline{v})$.
\end{prop}

\begin{proof}

If  $\underline{w}\sim_T\underline{v}$ then we have two options:

\begin{itemize}
\item $\underline{w},\underline{v}\in \Sigma_{\cI nt(T)}$ and they are $\sim_I$-related. By Proposition \ref{Prop: totally interior points projection sim I}  this happen if and only if $\pi_f(\underline{w})=\pi_f(\underline{w})$. This situation is over.
\item  $\underline{w},\underline{v}\in \Sigma_{\cS(T)}\cup\Sigma_{\cU(T)}$. Using alternately Propositions \ref{Prop: s-relaten implies same projection} and \ref{Prop: u-relaten implies same projection} we deduce
$$
\pi_f(\underline{w})=\pi_f(\underline{x_1})=\cdots=\pi_f(\underline{x_m}) =\pi_f(\underline{v}).
$$
and complete a direction of the proposition.
\end{itemize}

Now suppose that  $x=\pi_f(\underline{w})=\pi_f(\underline{v})$, we need to prove that they are $\sim_T$ related. Since the only codes that project to the same point are sector codes of the point (Lemma \ref{Lemm: every code is sector code }), $\underline{w},\underline{v}$ are sector codes of $x$ and the following Lemma implies our Proposition.

\begin{lemm}\label{Lemm: sector codes}
Let $\{\underline{e_i}\}_{i=1}^{2k}$ be the sector codes of the point $x=\pi_f(\underline{e_i})$. Then $\underline{e_i}\sim_T \underline{e_j}$ for all $i,j\in \{1,\cdots,2k\}$.
\end{lemm}

\begin{proof}
If $x$ is a totally interior point the Corollary\ref{Coro: interior periodic points unique code} implies that all sector codes of $x$ are equal to a then $\underline{e_j}\sim_T\underline{e_j}$. In the Figure \ref{Fig: Sim T items i y ii} this correspond to the situation when $f^k(x)$ have all its quadrants, like in item $b)$.

The remaining situation is when  $x\in \cF^s(\text{ Per }^s(\cR))\cup \cF^u(\text{ Per }^u(\cR))$, i.e. $x$ is in stable or unstable lamination generated by $s,u$-boundary periodic points, we will concentrate on this case.  Numbering the sectors of $x$ in cyclic order, we consider three situations depending on where $x$ is located:
\begin{itemize}
\item[i)] $x\in \cF^s(\text{ Per }^s(\cR))$ but $x\notin \cF^u(\text{ Per }^u(\cR))$ (Items $c)$ and $f)$ in Figure \ref{Fig: Sim T items i y ii})
\item[ii)] $x\in \cF^u(\text{ Per }^u(\cR))$ but $x\notin \cF^s(\text{ Per }^s(\cR))$ (Items $a)$ and $d)$ in Figure \ref{Fig: Sim T items i y ii}).
\item[iii)]  $x\in \cF^s(\text{ Per }^s(\cR))\cap\cF^u(\text{ Per }^u(\cR))$ ((Items $b)$ and $e)$ in Figure \ref{Fig: Sim T items i y ii})).
\end{itemize}

In either case we first consider that $x$ is not periodic, which means that no sector code $\underline{e_j}$  is periodic and the point $x$ have $4$ sectors because  is not a periodic point and then it is not a singularity.

\textbf{Item} $i)$ (Look at (Items $c)$ and $f)$ in Figure \ref{Fig: Sim T items i y ii})).
 There exist $k\in \ZZ$  such that $f^k(x)\notin\partial^s\cR$ and like $x\notin  \cF^u(\text{ Per }^u(\cR))$, $f^z(x)\notin \partial^u\cR$ for all $z\in \ZZ$. In particular $f^z(x)\in \overset{o}{\cR}$ for all $z\leq k$, we conclude that : $f^k(x)$ have only four quadrants like $x$ and according with  the definition of the $\sim_s$ relation ( discussion accompanied by the  Figure \ref{Fig: stable identification}):
\begin{itemize}
\item For all $z< k$, all the quadrants of $f^{z}(x)$ are in the same rectangle: $\underline{e_1}_{z}=\underline{e_2}_{z}=\underline{e_3}_{z}=\underline{e_4}_{z}$.
\item $f^{k+n}(x)\in \partial^s\cR$ and therefore its quadrants are like in   (Items $c)$ and $f)$ in Figure \ref{Fig: Sim T items i y ii}) then: $\underline{e_1}_{k+n}= \underline{e_2}_{k+n}$ and $\underline{e_3}_{k+n}= \underline{e_4}_{k+n}$

\item Therefore:  $\underline{e_1}\sim_{s}\underline{e_4}$ and $\underline{e_2}\sim_{s}\underline{e_3}$. 

\item Like $f^z(x)\notin \partial^u\cR$ for all $z\in \ZZ$ the sector codes of $x$ satisfy that: $\underline{e_2}=\underline{e_1}$ and $\underline{e_3}=\underline{e_4}$. 
\end{itemize}
In conclusion $\underline{e_2}\sim_u\underline{e_1}$ and $\underline{e_3}\sim_u\underline{e_4}$ and then: 
 $$
 \underline{e_1}\sim_s\underline{e_4}\sim_u\underline{e_3}\sim_s \underline{e_2}\sim_u\underline{e_1}.
 $$ 
 so $\underline{e_i}\sim_T\underline{e_j}$ for $i,j=1,2,3,4$.
 
\textbf{Item} $ii)$ is argued in the same way. Look at Figure \ref{Fig: Sim T items i y ii} to get the intuition of \textbf{Item} $i)$ and realize that the same idea applies for the second item. In this case up some iteration $f^k(x)$ is in are a rectangle like in items $a)$ or $d)$ in the image and the negative iterations of $f^k(x)$ keep this configurations. Therefore the  terms of the sector codes $\underline{e_1}_{k-n}=\underline{e_4}_{k-n}$ and $\underline{e_2}_{k-n}=\underline{e_3}_{k-n}$ and like $f^{k+n}(x)\in \overset{o}{\cR}$ the therms $\underline{e_{\sigma}}_{k+n}$ is the same for $\sigma=1,2,3,4$: this implies that
\begin{itemize}
\item $\underline{e_1} \sim_u \underline{e_4}$ and $\underline{e_2}\sim_u \underline{e_3}$.
\item  $\underline{e_1}\sim_s \underline{e_2}$ and $\underline{e_3}\sim_s \underline{e_4}$.
\end{itemize}

Finally: $\underline{e_i}\sim_T\underline{e_j}$ for $i,j=1,2,3,4$.
\begin{figure}[h]
	\centering
	\includegraphics[width=0.9\textwidth]{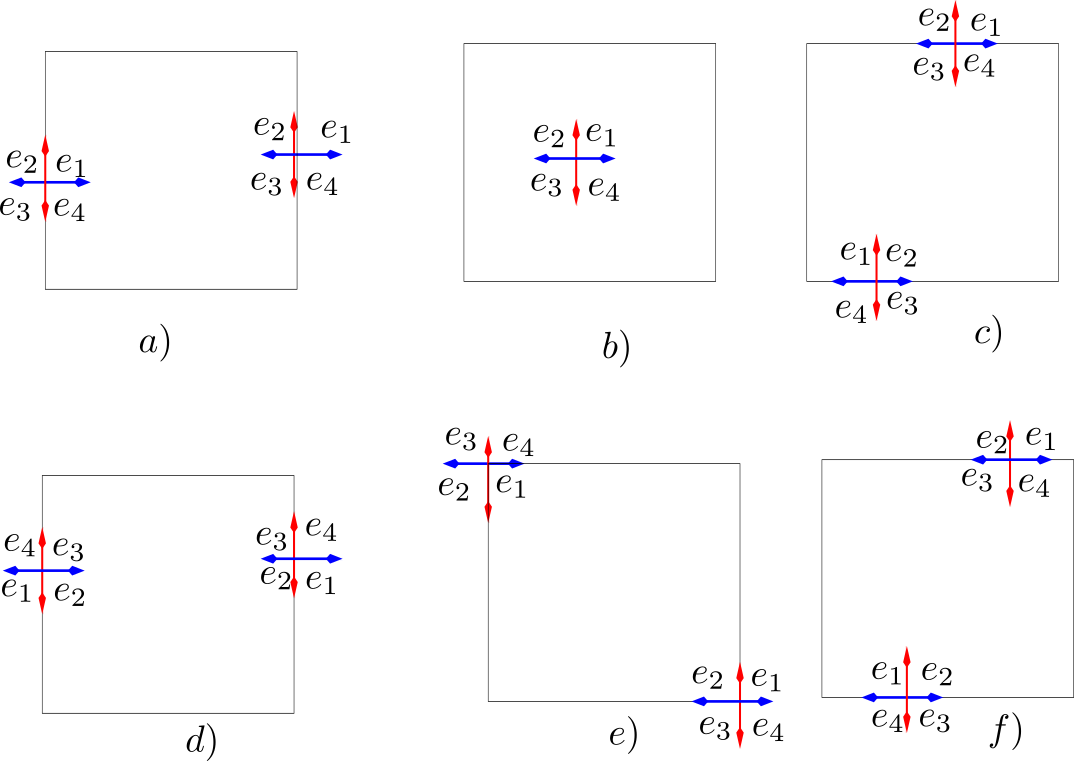}
	\caption{The sector codes that are identify.}
	\label{Fig: Sim T items i y ii}
\end{figure}

\textbf{Item} $iii)$ is the most technical situation. There are unique numbers $k(s),k(u)\in \ZZ$ defined as:
$$
k(s):=\max\{k\in \ZZ:f^k(x)\notin \partial^s \cR \text{ but } f^{k+1}(x)\in \partial^s \cR\},
$$
and 
$$
k(u)=\min\{k\in \ZZ:f^k(x)\notin \partial^u \cR \text{ but } f^{k-1}(x)\in \partial^u \cR\},
$$

The $f$-invariance of $\cR^s$ implies that for all $k>k(s)$, $f^k(x)\in \partial^s \cR$, but for all $k\leq k(s)$,  $f^k(x)\notin \partial^s \cR$. Similarly the $f^{-1}$ invariance of $\partial^u \cR$ implies that for all $k<k(u)$, $f^k(x)\in \partial^u\cR$ but for all $k\geq k(u)$ $f^k(x)\notin \partial^u \cR$.

They are three possibilities to detail: $k(u)<k(s)$, $k(u)=k(s)$ and $k(u)>k(s)$. Lets to divide the proof in this cases.

\textbf{First case} $k(u)<k(s)$. This inequality implies that $f^{k(s)}(x)\in \overset{o}{\cR}$ (and $f^{k(u)}(x)\in \overset{o}{\cR}$), this conditions  implies:

\begin{itemize}
\item The sector codes  take the same value for all $k(u)\leq  k\leq k(s)$.
	
\item  For all $k\geq k(s)$ the configuration of the sectors of $f^k(x)$ is like in Items $c)$ or $f)$ in in figure \ref{Fig: Sim T items i y ii}.

\item For all $k< k(u)$ the configurations of the sector of  $f^{k}(x)$ is like in Items  Item $a)$ or $d)$ in in figure \ref{Fig: Sim T items i y ii}.

\end{itemize}

 In view of Lemma \ref{lemma: simT related iterations related} we can deduce that: 
$$
\underline{e_1}\sim_s\underline{e_4}\sim_u\underline{e_3}\sim_s \underline{e_2}\sim_u\underline{e_1}.
$$
Hence $\underline{e_i}\sim_T\underline{e_j}$ for all $i=1,\cdots,4$.

\textbf{Second case} $k(u)=k(s)$. Analogously, $f^{k(u)}(x)=f^{k(s)}(x)\in \overset{o}{\cR}$ and we repeat the analysis of the previous case to deduce that $x$ have $4$ sector codes and all of them are $\sim_T$ related.  Figure \ref{Fig: k(u) less than k(s)} illustrates the ideas behind our arguments.

\begin{figure}[h]
	\centering
	\includegraphics[width=0.7\textwidth]{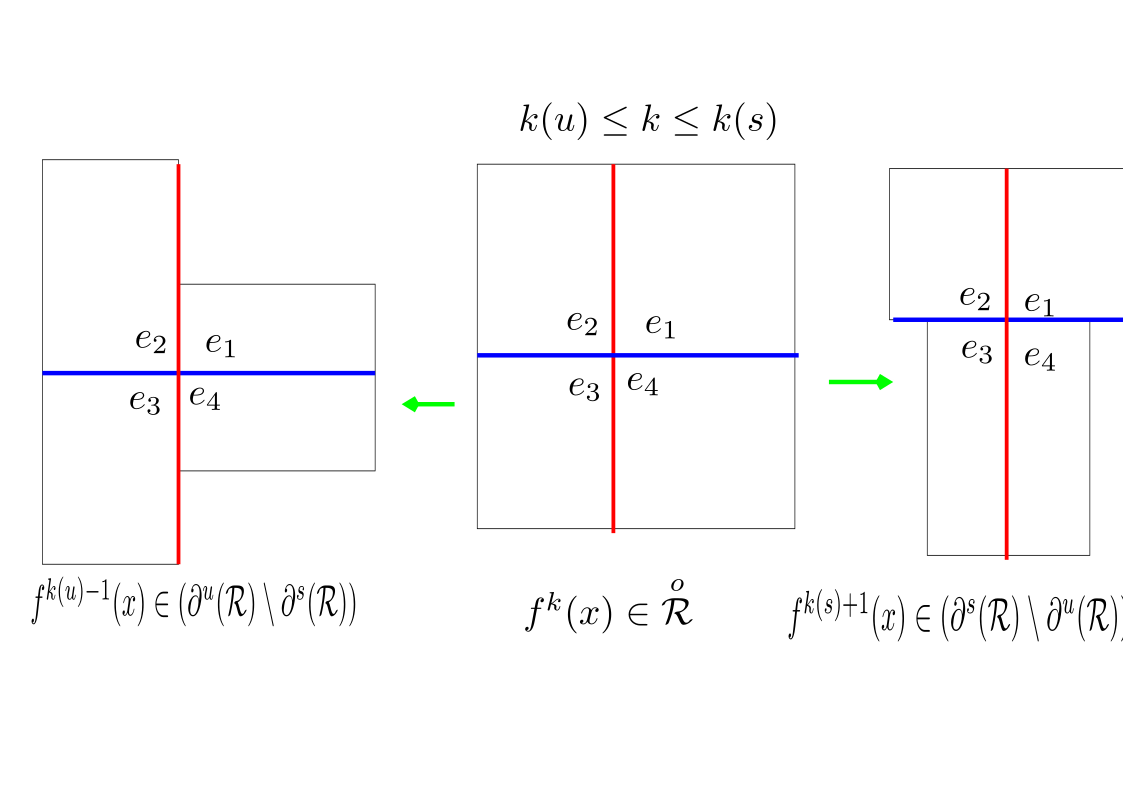}
	\caption{Situation $k(u)\leq k(s)$}
	\label{Fig: k(u) less than k(s)}
\end{figure}

\textbf{Third case} $k(s)<k(u)$.   In this situation $f^{k(s)}(x)\notin \partial^s \cR$ but $f^{k(s)}\in \partial^u\cR$ and for all $k$,  between $k(s)$ and $k(u)$, $f^k(x)$ is a corner point, and this have the following consequences:

\begin{itemize}
\item For all $n\in \NN$ $f^{k(s)-n}(e_1)$ and $ f^{k(s)-n}(e_4)$  are in the same rectangle of $\cR$ and similarly  $f^{k(s)}(e_2)$ and $f^{k(s)}(e_3)$ are in the same rectangle of $\cR$.
\item For all $n\in \NN$, $f^{k(u)+n}(e_1)$ and $f^{k(u)+n}(e_4)$ are in the same rectangle, and  similarly $f^{k(u)+n}(e_2)$ and $f^{k(u)+n}(e_3)$, are in the same rectangle.
\end{itemize}

Now for all $n\in \NN$, the negative part of the codes $\sigma^{k(s)-n}(e_1)$ and  $\sigma^{k(s)-n}(e_1)$ are equal. But for $k\in\{k(s),\cdots,k(u)\}$, $f^{k}(x)$ is in the stable boundary of $\cR$ and the configuration of the sector is like in Items $e)$, for $k\geq k(s)+1$ the configuration of the sector is like in $c)$ or $f)$. This lets us to conclude that $\underline{e_1}\sim_s \underline{e_4}$. The same argument applied for $\underline{e_2}$ and $\underline{e_3}$.

Using Lemma \ref{lemma: simT related iterations related} we deduce that:
$$
\underline{e_1}\sim_s \underline{e_4} \text{ and  } \underline{e_2}\sim_s \underline{e_3}.
$$

 Finally  for all $n\in \NN$, the sector codes of $f^{k(u)+s}(x)$ are like in configurations $c)$ and $f)$ so thee positive therms of $\sigma^{k(u)+1}(e_1)$ and  $\sigma^{k(u)+1}(e_2)$ are equal and $f^{k(u)}(x)$ is a corner point like in items $e)$, in particular the sector codes $\sigma^{k(u)}(\underline{e_1})$ and $\sigma^{k(u)}(\underline{e_1})$ share a unstable boundary of the Markov partition and then $\underline{e_1}\sim_u \underline{e_2}$. This process applied for the rest of sectors. Once again we can use Lemma \ref{lemma: simT related iterations related} to deduce that:
$$
\underline{e_1}\sim_u \underline{e_2} \text{ and  } \underline{e_3}\sim_u \underline{e_4}.
$$  
Then we put all together to obtain that:
$$
\underline{e_1}\sim_s\underline{e_4}\sim_u\underline{e_3}\sim_s \underline{e_2}\sim_u\underline{e_1}.
$$
proving that $\underline{e_i}\sim_T \underline{e_j}$ for $i,j=1,2,3,4$. This mechanism is illustrated in Figure \ref{Fig: k(s) less than k(s)}.

\begin{figure}[h]
	\centering
	\includegraphics[width=0.6\textwidth]{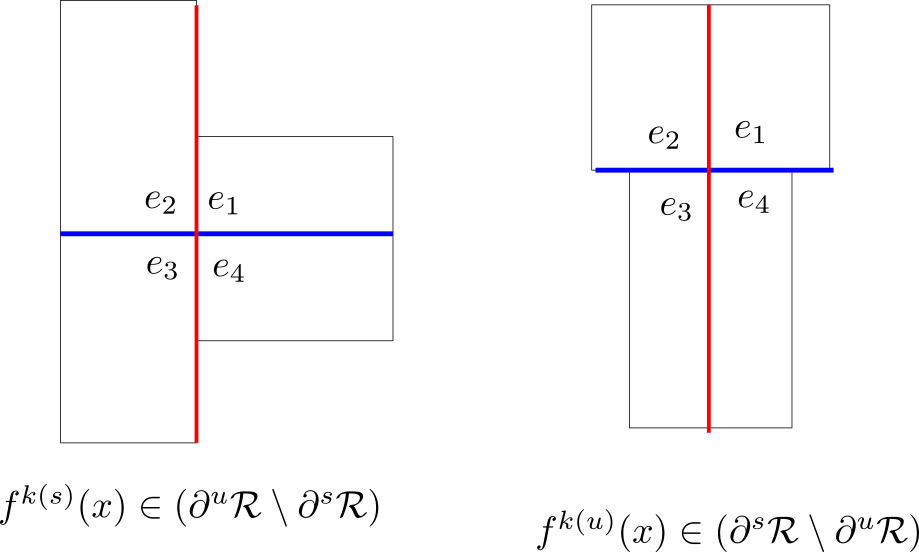}
	\caption{Situation $k(s)\leq k(u)$}
	\label{Fig: k(s) less than k(s)}
\end{figure}

It remains to address the periodic coding case. Let $x$ be a periodic point with $2k$-sectors, we label them with the cyclic order with respect to the surface orientation.  Two adjacent $e_i,e_{i+1}$-sectors are $s$-related if they share a stable local separatrix, as indeed, the stable separatrix lies in unique pair of rectangles and the boundary code of such sides are $s$-related. The reason is the mechanism of stable boundary identification that appears in Figure  \ref{Fig: stable identification} and that motivated the definition of the $\sim_s$-relation. If $ e_i,e_{i+1}$ share the same unstable separatrice they are $u$-related for the same reason. In this way one can get from one sector of $x$ to another by means of a finite number of intermediate sectors, which are alternatively $\sim_s$ and $\sim_u$-related. Then $\underline{e_i}\sim_T\underline{e_j}$ for all $i,j=1,\cdots, 2k$. This argument ends the proof.
\end{proof}

\end{proof}

\section{The geometric type is a total conjugacy invariant.}

All the work of this Chapter is condensed in the following Proposition.

\begin{prop}\label{Prop:  cociente T}
	Let $T$ be a symbolically modelable geometric type and $(f,\cR)$ a realization of $T$, where $f:S\rightarrow S$ is a pseudo-Anosov homeomorphism (with or without Spines). Let $A:=A(T)$ be the incidence matrix of $T$, assume that $A$ is a binary matrix,  $(\Sigma_A,\sigma)$ the  sub-shift of finite-type associated to $T$ and $\pi_f:\Sigma_A\rightarrow S$ is the projection.
	
	The quotient space $\Sigma_T=\Sigma_A/\sim_T$ is equal to $\Sigma_f:=\Sigma_A/\sim_{f}$. Therefore $\Sigma_T$ is a closed surface. The sub-shift of finite type $\sigma$ pass to the quotient by $\sim_T$ to a homeomorphism $\sigma_T:\Sigma_T \rightarrow \Sigma_T$ which is a  generalized pseudo-Anosov homeomorphism topologically conjugate to $f:S\rightarrow S$ via the quotient homeomorphism:
	$$
	[\pi_f]: \Sigma_T=\Sigma_f \rightarrow S.
	$$ 
\end{prop}

\begin{proof}
As the proposition \ref{Prop: The relation determines projections} indicates, the relation $\sim_f$ given by $\underline{w}\sim_f\underline{v}$ if and only if $\pi_f(\underline{w})=\pi_f(\underline{w})$ is equal to $\sim_T$, regardless of the pair $(f,\cR)$ that represent $T$. Therefore the quotient spaces $\Sigma_T$ and $\Sigma_f$ are the same.

The Proposition \ref{Prop: quotien by f} implies that $S$ and $\Sigma_f$ are homeomorphic, hence $\Sigma_T$ is homeomorphic to $S$ which is a closed surface. Moreover the shift map $\sigma$ passes to the quotient under $\sim_f$ to a homeomorphism $[\sigma]$ which is topologically conjugate to $f$ via the quotient map $[\pi_f]$. On the other hand the shift $\sigma$ passes to the quotient under $\sim_T$ too, so we have the homeomorphism $\sigma_T:\Sigma_T\rightarrow \Sigma_T$, which is identical to $[\sigma]$ and  we conclude  as in Proposition \ref{Prop: quotien by f} that:
	$$
[\pi_f]^{-1}\circ f\circ [\pi_f]= \sigma_T.
	$$
Finally, the Proposition \ref{Prop:  conjugation produces pA} implies that $\sigma_T:\Sigma_T \rightarrow \Sigma_T$ is  pseudo-Anosov because it is topologically conjugate to a pseudo-Anosov homeomorphism. 
\end{proof}

Observe that if $(g,\cR_g)$ and $(f,\cR_f)$ are two pair that represent $T$. It is possible to consider the product homeomorphism $[\pi_f]^{-1}\circ [\pi_g]$, because if $\underline{w}\in \pi_f^{-1}(x)\cap \pi^{-1}_g(y)\neq \emptyset$, any other codes $\underline{X}\in \pi^{-1}(x)$ and $\underline{Y}\in \pi^{-1}_g(y)$ are $\sim_T$ related, i.e. $\underline{X}\sim_T\underline{w}\sim_T\underline{Y}$, therefore  $ \pi_f^{-1}(x)= \pi^{-1}_g(y)$. It is time to prove our main Theorem \ref{Theo: conjugated iff  markov partition of same type}

\begin{theo*}
	Let $f:S_f\rightarrow S_f$ and $g:S_g \rightarrow S_g$ two pseudo-Anosov homeomorphism maybe with Spines. If $f$ and $g$ had  Markov partitions  of same geometric type then there exist an orientation preserving homeomorphism between the surfaces $h:S_f\rightarrow S_g$ that conjugate them $g=h\circ f\circ h^{-1}$.
\end{theo*}

\begin{proof}
We already know that if $f$ and $g$ are conjugate they have Markov partitions of the same geometric type, so let's focus on the other direction of the proof.

Let $T := T(f, \cR_f) = T(g, \cR_g)$ be the geometric type of the Markov partitions of $f$ and $g$. If $T$ does not have a binary incidence matrix, we can consider a horizontal refinement of $(f, \cR)$ and $(g, \cR)$ in such a manner that the incidence matrix of such refinement is binary and the number of horizontal and vertical sub-rectangles in the respective Markov partitions is bigger than $2$.  The geometric types of these refinements remain the same. Therefore, we can assume from the beginning that $T$ has an incidence matrix $A$ with coefficients in $\{0, 1\}$ and $v_i,h_i>2$

 The quotient spaces of $\Sigma_A$ by $\sim_f$ and $\sim_g$ are equal to $\Sigma_T$ as proved in Proposition \ref{Prop:  cociente T}, i.e.  $\Sigma_g=\Sigma_{T}=\Sigma_f$. Moreover, the quotient shift  $\sigma_{T}$ is topologically conjugate to $f$ via $[\pi_f]:\Sigma_T \rightarrow S_f$ and to $g$ via $[\pi_g]:\Sigma_T \rightarrow S_g$. Therefore $f$ is topologically conjugate to $g$ by $h:=[\pi_g]\circ[\pi_f]^{-1}:S_f \rightarrow S_g$  which is a well-defined homeomorphism. It rest to prove that $h$ preserve the orientation. 
 
 \begin{lemm}
Let $H^i_j$ and $\underline{H}^i_j$ the respective horizontal sub-rectangles of $\cR_f$ and $\cR_g$, then $h(H^i_j)=\underline{H}^i_j$.
 \end{lemm}
 
 \begin{proof}
The set  $R(i,j)=\{\underline{w}: w_0=i \text{ and } \rho(i,j)=(w_1,l_0) \}$ is such that: $\pi_f(R_i,j)=H^i_j$ and  $\pi_g(\underline{H}^i_j)$. Consider the set $R(i,j)_T$ like the equivalent classes of $\Sigma_T$ that contains and element of $R(i,j)$. Clearly $[\pi_g]\circ[\pi_f]^{-1}(H^i_j)=[\pi_g](R(i,j)_T)=\underline{H}^i_j$
 \end{proof}

This have the following consequence.

\begin{coro}
The homeomorphism $h$ restricted to every rectangle $R_i$ preserve the traverse orientation of its vertical and horizontal foliations. In particular $h$ preserve the orientation restricted to $R_i$.
\end{coro}

\begin{proof}
Let $R_i$ be a rectangle in $\cR_f$ such that $x$ is its bottom-left corner point. This implies that the stable separatrix $I$ that determines the lower boundary of $R_i$ points in the opposite direction from $x$. Similarly, for the left stable boundary $J$ of $R_i$, it also points away from $x$. It's worth noting that within $R_{i}$, there are two sub-rectangles $H^i_1$ and $V^i_1$, whose the intersection defines a sector $e_x$ associated with $x$.

Now, consider $h(x)$, which is a corner point of $\underline{R}_i$. The rectangles $h(H^i_1)=\underline{H}^i_1$ and $h(V^i_1)=\underline{V}^i_1$ that are incident in $h(x)$ define a sector $h(e_x)$. It's necessary that the rectangle $\underline{H}^i_2$, the image of $H^i_2$, be adjacent to $\underline{H}^i_1$. This adjacency ensures the coherent ordering of horizontal sub-rectangles within $\underline{R}_i$, in accordance with the horizontal orientation of $\underline{R}_i$. Simultaneously, $h(V^i_2)$ becomes the vertical sub-rectangle of $\underline{R}_i$ next to $\underline{V}^i_1$. The sequence of rectangles $\underline{V}^i_l$ retains its coherence with the horizontal orientation of $\underline{R}_i$.

As both $h_i$ and $v_i$ are greater than $2$, $h$ simultaneously maps positive-oriented vertical and horizontal sub-rectangles of $R_i$ to positive-oriented vertical and horizontal sub-rectangles of $\underline{R}_i$. This also implies that $h$ preserves the transversal orientation of the foliation, and consequently, it maintains the orientation restricted to $R_i$.

\end{proof}

Finally, if the rectangles $R_i$ and $R_j$ intersect at a stable boundary point $x$, we can assign vertical orientations to $R_i$ and $R_j$ such that $x$ becomes the lower boundary point of $R_i$ and the upper boundary point of $R_j$. Necessary adjustments can be made to the stable directions of these rectangles to ensure that their orientations remain unchanged. By doing so, we can define $\underline{R_i}$ and $\underline{R_j}$ with vertical orientations induced by $h$. Since $h$ preserves the transversal orientation of the foliations in both $R_i$ and $R_j$, it also preserves the horizontal orientations of these rectangles. Consequently, $h$ preserve the orientation in the union of $R_i$ and $R_j$.

\end{proof}

The space $\Sigma_T$ and the homeomorphism $\sigma_T$ are constructed in terms of $T$ and as we have seen represent the conjugacy class of pseudo-Anosov homeomorphisms realizing $T$. In this sense we think that the symbolic dynamical system $(\Sigma_T,\sigma_T)$ deserves a name.

\begin{defi}\label{Defi: symbolic model}
Let $T$ be a geometric type in the pseudo-Anosov class, with an incidence matrix $A$ that is binary. Let $(f, \cR)$ be a pair representing $T$. The pair $(\Sigma_T, \sigma_T)$ represents the symbolic model of the geometric type $T$.
\end{defi}

Finally, we have obtained a combinatorial representation of a geometric type $T$ in the pseudo-Anosov class by solving Topic 1.III in \ref{Prob: Clasification}.

\chapter{The pseudo-Anosov class of geometric types.}\label{Chapter: Realization}

\section*{The realization problem.}

The objective of this section is to prove the following Theorem

\begin{theo}[Algorithmic characterization]\label{Theo: caracterization is algoritmic}
	There exists a finite algorithm that can determine whether a given geometric type $T$ belongs to the pseudo-Anosov class. Such algorithm requires calculating at most $6n$ iterations of the geometric type $T$, where $n$ is the first parameter of $T$.	
\end{theo}

It will be deduce from the following proposition which is the main technical result in this section, we dedicate two subsections to probe it.

 \begin{prop}\label{Prop: pseudo-Anosov iff basic piece non-impace}
	Let $T$ be an abstract geometric type. The following conditions are equivalent.
\begin{itemize}
	\item[i)] The geometric type $T$ is realized as a mixing basic piece of a surface Smale diffeomorphism without impasse.
\item[ii)] The geometric type $T$ is in the pseudo-Anosov class.
\item[iii)] The geometric type $T$ satisfies the following properties:
	\begin{enumerate}
		\item  The incidence matrix $A(T)$ is mixing
		\item The genus of $T$ is finite
		\item $T$ does not have an impasse.
	\end{enumerate}
\end{itemize}
\end{prop}
 
In  Proposition \ref{Prop: mixing+genus+impase is algorithm} we demonstrate that  there exist a finite algorithm that start whit any geometric type and determines if the three properties of last item are or not satisfy with the proper bounds in the numbers of iterations. This is discussed in the first section when we translate de definitions of finite genus and impasse that where given in therms of a realization in the Chapter of preliminaries, to some combinatorial and algebraic formulations in therms of the geometric type, after that we can determine in a precise time when $T$ is realizable by check such formulas.

The proof of \ref{Prop: pseudo-Anosov iff basic piece non-impace} is presented  along  three sections: Section \ref{Sec: type basic piece then pA} determine that Item $i)$ implies Item $ii)$, Section \ref{Sec: Tipes PA finite genus but not impasse} that  Item $ii)$ implies Item $iii)$ and Section \ref{Sec: finite genus no impas implies basic piece} close the cycle by proof that Item $iii)$ implies Item $i)$.

\section{The geometric type of a mixing basic piece without impasse}\label{Sec: type basic piece then pA}

In this section, our goal is to prove that if a geometric type $T$ is realized as a mixing, saddle-type basic piece of a surface Smale diffeomorphism without impasse, then $T$ is in the pseudo-Anosov class, in this manner we obtain that Item $(1)$ implies Item $(2)$ in the Proposition \ref{Prop: pseudo-Anosov iff basic piece non-impace}. Our starting point is the following theorem by Christian Bonatti and Emmanuelle Jeandenans.(see \cite[Theorem 8.3.1]{bonatti1998diffeomorphismes}):

\begin{theo}\label{Theo: Basic piece projects to pseudo-Anosov}

	Let $f$ be a Smale diffeomorphism of a compact and oriented surface $S$, $K$ be a non-trivial saddle-type basic piece of $f$, $\Delta(K)$ be its domain, and suppose that $K$ does not have an impasse. Then, there exists a compact surface $S'$, a generalized pseudo-Anosov homeomorphism $\phi$ on $S'$, and a continuous and surjective function $\pi: \Delta(K) \rightarrow S'$ such that:
	
	$$
	\pi\circ f\vert_{\Delta(K)}=\phi\circ \pi.
	$$

Furthermore, the semi-conjugation $\pi$ is injective over the periodic orbits of $K$, except for the finite number of periodic $s,u$-boundary points..
\end{theo}

The surface $S'$ has an induced orientation through $\pi$, we are going to clarify this point in Lemma \ref{Lemm: S is oriented} . Using the function $\pi$ and the orientation of $S'$, we will establish a more detailed version of Theorem \ref{Theo: Basic piece projects to pseudo-Anosov} because we aim to compare the geometric types of the basic piece and the pseudo-Anosov homeomorphism. We have defined the induced partition Markov partition by an orientation-preserving homeomorphism in the context of pseudo-Anosov homeomorphisms. 
There is a similar construction using $\pi$ that produces the induced Markov partition, which we will refer to in the following Proposition. We will provide more information about this concepts before starting the proof, but let's present the statement first.

\begin{prop}\label{Prop: type of basic piece is type of pseudo-Anosov}
	 Let $f:S\rightarrow S$ be a Smale surface diffeomorphism. Let $K$ be a mixing saddle-type basic piece of $f$ without impasse, and $\mathcal{R}=\{R_i\}_{i=1}^n$ be a geometric Markov partition of $K$ of geometric type $T$. Let $\pi:\Delta(K) \rightarrow S'$ be the projection, and let $\phi$ be the generalized pseudo-Anosov homeomorphism described in Theorem \ref{Theo: Basic piece projects to pseudo-Anosov}.
	 Then, the induced partition $\pi(\mathcal{R})=\{\pi(R_i)\}_{i=1}^n$ is a geometric Markov partition of $\phi$ of geometric type $T$.
\end{prop}

The proof requires two main results. First, Proposition \ref{Prop: pi cR is Markov partition} demonstrates that $\pi(\cR)$ is a Markov partition for $\phi$. Under certain natural assumptions regarding the orientation of the rectangles in $\pi(\cR)$,  $\pi(\cR)$ possesses a well-defined geometric type. Subsequently, Lemma \ref{Lemm: pi send sub-rec in sub rec} establishes the relationships between the vertical and horizontal sub-rectangles of $\cR$ and $\pi(\cR)$. Finally in the sub subsection \ref{Subsec: Proof Proposition} we provide a concise proof of Proposition \ref{Prop: type of basic piece is type of pseudo-Anosov}.

The Bonatti-Jeandenans Theorem \ref{Theo: Basic piece projects to pseudo-Anosov} contains all the ideas that we are going to use achieve Proposition \ref{Prop: type of basic piece is type of pseudo-Anosov}. This imposes the task of explaining how the function $\pi$ collapses the invariant laminations of $f$. In a simplified manner, the strategy in the proof of Bonatti-Jeandenans can be divided into the following parts:

\begin{itemize}
\item An equivalence relation $\sim_{R}$ is defined on the domain of $K$, $\Delta(K)$.

\item It is proven that the quotient space of $\Delta(K)$ by $\sim_{R}$ is a compact surface $S'$, and the diffeomorphism $f$ induces a generalized pseudo-Anosov homeomorphism $\phi$ on $S'$.

\item The function $\pi:S\rightarrow S'$  associates every point in $S$ with its equivalence class, i.e., it is the canonical projection.

\item  It is proved that $\pi$ is a semi-conjugation with the desired properties.
\end{itemize}

Understand the equivalence relation $\sim_{R}$ is to understand the function $\pi$. For this reason, we will explain how they work. The facts about $K$ and its invariant manifolds that were proven in \cite[Proposition 2.1.1]{bonatti1998diffeomorphismes} and recalled in Proposition \ref{Prop: sub-boundary points 2.1.1} will be especially useful at this point.
 
 \subsection{The equivalent relation $\sim_{R}$ described by parts.}
 
Assume that $f:S\rightarrow S$ is a Smale surface diffeomorphism and $K$ is a non-trivial, mixing, saddle-type basic piece of $f$ without double boundary points nor impasse. These are the only general assumptions in this section. 
 
 \subsubsection{Neighbor points and adjacent separatrices}
 
 In Figure \ref{Fig: Neigboor}, there are four periodic points of a saddle type basic piece $K$ with non-double boundaries, let us explain the notations:
 
 \begin{itemize}
\item[i)] The point $p_1$ is a $s$-boundary point but not a $u$-boundary point, while the point $p_3$ is a $u$-boundary point but not a $s$-boundary point.

\item The points $p_2$ and $p_3$ are $s$ and $u$ boundary points, respectively; we refer to them as \emph{corner points}.

\item The non-free separatrices of these periodic points are labeled as $W^s_{1,2}(p_1)$, $W^{s,u}_1(p_2)$, $W^u_{1,2}(p_3)$, and $W^{s,u}_1(p_4)$.

\item The unstable intervals $J_{p_1,p_2}(1,1)$ and $J'_{p_1,p_2}(1,1)$ are $u$-arches that have one endpoint in $W^s_1(p_1)$ and the other in $W^s_1(p_2)$. Similarly, $I_{p_3,p_4}(2,1)$ and $I'_{p_3,p_4}(2,1)$ are $s$-arcs that connect $W^u_2(p_3)$ with $W^u_1(p_4)$.

\item The rest of the arcs follows the same login in their notations.

 \end{itemize}

We are going  to discuss the notion of \emph{neighborhood point}. Take a $s$-boundary periodic point of $K$, in our example we take $p_1$ and let $x\in W^s_1(p_1) \cap K$. Like $x$ is not a $u$-boundary point, there exist a $u$-arc, like $J_{p_1,p_2}(1,1)$, with one end point in $x$ and the other in a point $y\in W^s_1(p_2)$, where $p_2$ is a $s$-boundary point and $W^s_1(p_2)$ is non-free separatrice. In this setting,  $p_2$ is what we call a \emph{neighbor point} of $p_1$ and $W^s_1(p_2)$ is an \emph{adjacent separatrice}  of $W^s_1(p_1)$.

This procedure can be applied to any $u$ or $s$ periodic boundary point. For example $p_2$ is a $u$-boundary point. Therefore if $x\in W^u_1(p_2)\cap K$ there is a $s$-arc, $I_{p_2,p_3}(1,1)$, that have and end point in $x$ and the other in a unstable separatrice of a $u$-boundary point, in the case of the picture such point is $p_3$, we said that $p_3$ is a neighbor point of $p_2$ and $W^u_1(p_2)$ is an adjacent separatrice of $W^u_1(p_2)$.

In general, let $p$ a $s$ (or $u$)-boundary point and $W^s_1(p)$ ( resp. $W^u_1(p)$) a non-free stable (unstable) separatrice. There are two properties that we what to establish:
\begin{enumerate}
\item The adjacent separatrice of $W^s_1(p)$ (resp. $W^u_1(p)$) is unique, and 
\item The adjacent separatrice of $W^s_1(p)$ (resp. $W^u_1(p)$) is different than $W^s_1(p)$ (resp. $W^u_1(p)$).
\end{enumerate}

Since the basic piece $K$ does not have an impasse, there cannot be a $u$-arc joining two points in the same stable separatrice. As a result, the adjacent separatrice of $W^s_1(p)$ cannot be $W^s_1(p)$ itself. However, it is possible that the adjacent separatrice of $W^s_1(p)$ is the other stable separatrice of $p$. We will soon see that this configuration leads to the formation of a spine.  The adjacent separatrices will be well determined  if, for every pair of points $x_1$ and $x_2$ on the same stable separatrice $W^u_1(p)$, the $u$-arcs (see $J_1$ and $J_2$ in Figure \ref{Fig: Neigboor}) that pass through $x_1$ and $x_2$ have their other endpoints in the same stable separatrice.

\begin{figure}[h]
	\centering
	\includegraphics[width=0.8\textwidth]{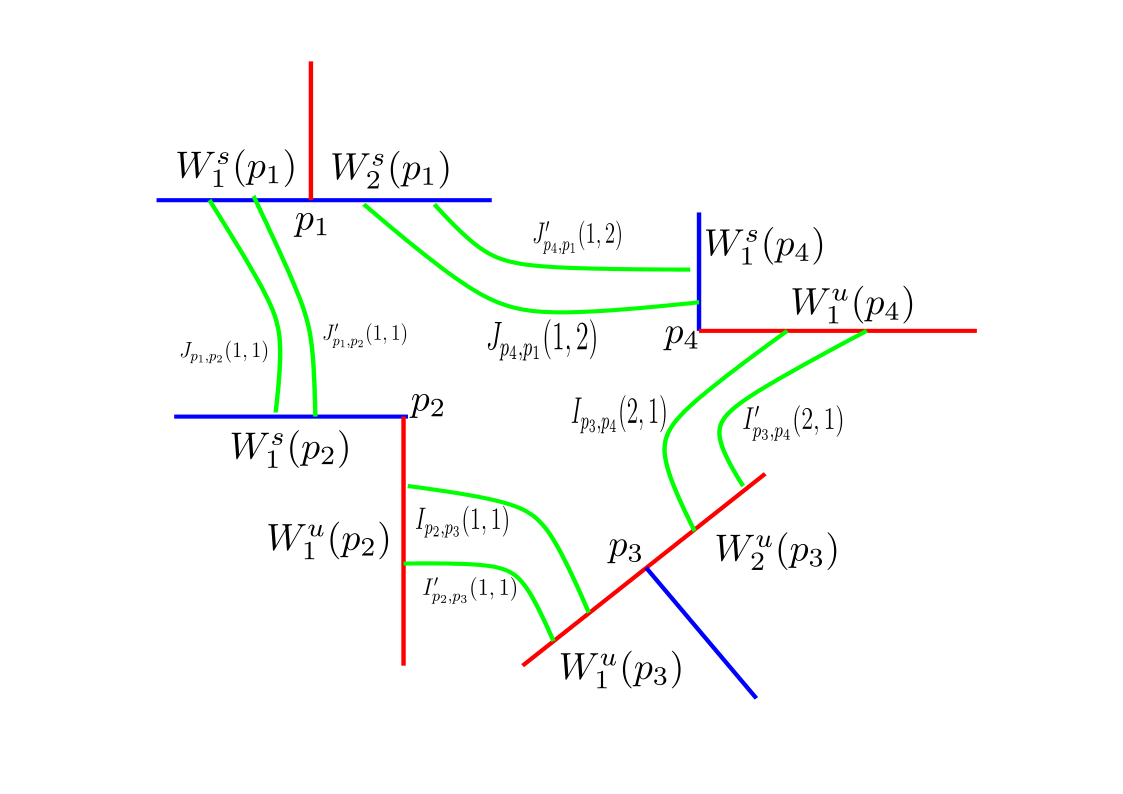}
	\caption{ Neighbors points.}
	\label{Fig: Neigboor}
\end{figure}

These assertions were proven in \cite[Proposition 2.4.9]{bonatti1998diffeomorphismes}. The hypothesis that $K$ does not have an impasse is crucial (see \cite[Definition 2.3.2]{bonatti1998diffeomorphismes} for understanding the statement). We summarize this discussion in the following lemma.

\begin{lemm}\label{Lemm: well define neighboor }
Let $K$ be a saddle-type basic piece without an impasse. If a $u$-arc $J_1$ joins two $s$-boundary stable separatrices $W^s_1(p_1)$ and $W^s_2(p_2)$, then:
\begin{itemize}
\item[i)] The separatrices are different, i.e., $W^s_1(p_1) \neq W^s_2(p_2)$.
\item[ii)]  If $J_2$ is any other $u$-arc with one endpoint in $W^s_1(p_1)$, then the other endpoint of $J_2$ is in $W^s_2(p_2)$.
\end{itemize} 
\end{lemm}

\begin{defi}
Let $p$ and $q$ be periodic boundary points of the same nature, i.e. $p$ is $s$-boundary (resp. $u$-boundary) if and only if $q$ is $s$-boundary (resp. $u$-boundary) . They are \emph{neighbor points} if one of the following conditions is satisfied:
\begin{enumerate}
\item They are $s$-boundary points and have non-free stable separatrices $W^s_1(p)$ and $W^s_2(q)$, with $W^s_1(p) \neq W^s_2(q)$, for which there exists a $u$-arc with one endpoint in $W^s_1(p)$ and the other in $W^s_2(q)$. In this case, we say that $W^s_1(p)$ is the \emph{adjacent separatrice} of $W^s_2(q)$.

\item They are $u$-boundary points and have unstable separatrices $W^u_1(p)$ and $W^u_2(q)$, with $W^u_1(p) \neq W^u_2(q)$, for which there exists an $s$-arc with one endpoint in $W^u_1(p)$ and the other in $W^u_2(q)$. In this case, we say that $W^u_1(p)$ is the \emph{adjacent separatrice} of $W^u_2(q)$.
\end{enumerate}

\end{defi}

\subsubsection{The cycles}

Next, we define an equivalence relation on the set of periodic $s$-boundary points and $u$-boundary points.

\begin{defi}\label{Defi: relation cycle}
Let $p$ and $q$ two boundary points of $K$ we have the next relations:
 \begin{itemize}
\item  If $p$ and $q$ are $s$-boundary points, then $p$ and $q$ are $\sim_s $-related if they are neighbors points.
\item  If $p$ and $q$ are $u$-boundary points, then $p$ and $q$ are $\sim_u $-related if they are neighbors points.
\item The points $p$ and $q$ are $\sim_c$-related if there is a sequence of boundary periodic points $\{p = p_0, \ldots, p_{k-1} = q\}$ such that $p_i \sim_s p_{i+1}$ or $p_i \sim_u p_{i+1}$ for $i \in \mathbb{Z}/k$ (the integers modulo $k$).
 \end{itemize}
\end{defi}

Is not difficult to see that the relation $\sim_c$ is reflexive, symmetric and transitive, therefore the next lemma is immediate.

\begin{lemm}\label{Lem: equiv cycle}
	The relation $\sim_c$ is a equivalence relation in the finite set of boundary periodic points of $K$, and every equivalent $\mathcal{C}$ class is called a  \emph{cycle}. We denote $\mathcal{C}=\{p_1, \ldots, p_k\}$ to indicate that $p_1\sim_{s,u}p_{i+1}$, in this manner we give an order to the cycle.
\end{lemm}

It might not be immediately apparent that if we follow the cycle, we return to the original point, i.e.,  $\mathcal{C}=\{p_i: i\in \ZZ/k \}$, and in particular $p_1=p_k$. However, this is proven in \cite[Lemma 8.3.4]{bonatti1998diffeomorphismes}. Another important feature is that when we follow the adjacent separatrices that determine the cycle, we always return to the same stable or unstable separatrice. In other words, the cycle is closed, this property is ensured by \cite[Lemma 8.3.5]{bonatti1998diffeomorphismes}, which proves that the number of corner points in a cycle is even. Thus, if we start with a stable separatrice, we will end with a stable separatrice of the same point.

In \cite[Proposition 3.1.2]{bonatti1998diffeomorphismes}, it is proven that every connected component of $\Delta(K)\setminus \delta(K)$ (See definitions \ref{Defi: Domain of K} and \ref{Defi: restricted domain} resp.) is homeomorphic to $\mathbb{R}^2$ and is periodic under the action of the diffeomorphism $f$. The boundary of each connected component has two possibilities: it is either the union of two free separatrices of a corner point or it consists of an infinite chain of arcs and their closures. The possibility of having an infinite chain of arcs is excluded by \cite[Proposition 2.6.4, Lemma 2.6.7]{bonatti1998diffeomorphismes}, in conjunction with the fact that there are no impasses in $K$. Consequently, every curve formed by a corner point and its free separatrices corresponds to the boundary of a connected component of $\Delta(K)\setminus \delta(K)$. Let $p$ be a corner point. Based on the previous observations, we can define a set $C(p) \subset \Delta(K)$ as the connected component of $\Delta(K)\setminus\delta(K)$ that is bounded by the (stable or unstable) free separatrices of $p$.

\begin{defi}\label{Defi: P(C) region of cycle}
Let $\mathcal{C} = \{p_1, \ldots, p_k\}$ be a cycle. We denote $\mathcal{P}(\mathcal{C})$ as the union of points in the cycle together with their free separatrices, and the sets $C(p)$ for all the corner points in the cycle.
\end{defi}

\subsubsection{Four different equivalent classes}

Now we can define certain families in $\Delta(K)$ that will serve as models for the equivalence classes for the relation $\sim_R$. They are:

\begin{defi}\label{Defi: equivalen clases sim-r}
		Let $\Delta(K)$ be the domain of the basic piece $K$, consider the four distinguished types of subsets of the domain of $K$.
\begin{itemize}
	\item[i)] The \emph{singletons} $\{x\in K: \, x \text{ is not a boundary point}\}$.
	
	\item[ii)] The \emph{minimal rectangles}, i.e., rectangles whose boundary consists of two unstable arcs and two stable arcs.
	
	\item[ii)] The arcs included in a non-boundary invariant manifold.
	
	\item[iii)] The sets $\cP(C)$, where $C$ is a cycle.
\end{itemize}
\end{defi}

By considering all sets of the form described above, we obtain a partition of $\Delta(K)$ that is invariant under $f$ as its proved in \cite[Lemma 8.3.7]{bonatti1998diffeomorphismes}. Moreover, this partition enables us to define an equivalence relation ( See \cite[Definition 8.3.8]{bonatti1998diffeomorphismes}).

\begin{defi}\label{Defi: sim R equiv relation}
We define $\sim_{R}$ as the equivalence relation where the classes are the components of $\Delta(K)$ that correspond to one of the four families previously described.
\end{defi}

The surface $S'$ in the Bonatti-Jeandenans Theorem \ref{Theo: Basic piece projects to pseudo-Anosov} is the quotient of $S$ by this equivalence relation,  i.e $S'=S/\sim_R$, and the map $\pi: S \rightarrow S'$ is the projection that maps each point to its equivalence class.

\subsection{ The projection of a Markov partition is a Markov partition.}

In this subsection we are going to prove that the projection of the Markov partition $\cR$ by $\pi$ is a Markov partition for $\phi$. For that reason, there are two important aspects to consider: the behavior of rectangles in a Markov partition $(\cR,f)$ under the action of $\pi$, and the behavior of the complement of this Markov partition in $\Delta(K)$. The first situation is summarized in \cite[Lemma 8.4.2]{bonatti1998diffeomorphismes}, which we reproduce below as Lemma \ref{Lemm: projection of rectangle by pi}. Essentially, the lemma states that the interior of a rectangle in the Markov partition projects to the interior of a rectangle under the pseudo-Anosov homeomorphism $\phi$ (see Theorem \ref{Theo: Basic piece projects to pseudo-Anosov}).

\begin{lemm}\label{Lemm: projection of rectangle by pi}
	Let $\cO$ the interior of a rectangle where every boundary component is not isolated in $\cO$ (i.e it boundary is accumulated by $K$ from the inside). The quotient of $\cO$ by the equivalence relation is homeomorphic to  a rectangle $(a_1,a_2)\times (b_1,b_2)\subset \RR^2$. 
	
	Moreover, the stable lamination of $\cO$ given by the segment of stable manifold of $K$, $W^s(K)\cap \cO$ has image the foliation by horizontal lines of the rectangle $(a_1,a_2)\times (b_1,b_2)$ and the unstable lamination of $\cO$ induced by $W^u(K)\cap \cO$ has  image the foliation by vertical lines of such rectangle.
	
	Finally, if the coordinates of the rectangle $(a_1, a_2) \times (b_1, b_2)$ are given by $(s, t)$, then the \emph{Margulis measures} $\upsilon^s$ and $\upsilon^u$ of $f$ pass to induce the measures $dt$ and $ds$, respectively.
	
\end{lemm}

The proof relies on an alternative formulation the projection $\pi$ inside these types of rectangles. It begins by considering a parametrization of the rectangle $\overline{\cO}$, $r:[-1,1]^2 \rightarrow \overline{\cO}$ with coordinates $(s,t)$.  Then, a function is defined $\psi:\cO \to (a_1,a_2)\times(b_1,b_2)\subset \RR^2$ as follows:
$$
\psi=(\psi_1,\phi_2)(m)=(\epsilon\int_{0}^{s}d\upsilon^u,\nu \int_{0}^{t} d\upsilon^s).
$$
Where $m \in \cO$ has coordinates $(s,t)$, $\epsilon = +1$ if $s \geq 0$ and $-1$ otherwise, and $\nu = +1$ if $t \geq 0$ and $-1$ otherwise. 

 The pre-image of any point in the rectangle $(a_1,a_2)\times(b_1,b_2)$ under this map coincides with the equivalence classes determined by $\sim_R$ inside $\cO$. Therefore,  $(\psi_1,\phi_2)$  is a covering map. By taking the quotient we obtain a homeomorphism, that intrinsically depends on $r$:
$$
\Psi^{-1}=(\Psi_1,\Psi_2):  \cO/\sim_{R} \rightarrow (a_1,a_2)\times (b_1,b_2)
$$
The homeomorphism has an inverse $\Phi: (a_1,a_2)\times (b_1,b_2) \to \cO/\sim_R$: ,  which serves as a parametrization of $\pi(\cO)$.
The next lemma is implicit in the Bonatti-Jeandenans proof of Lemma \ref{Lemm: projection of rectangle by pi}.

\begin{lemm}\label{Lemm: psi bonatti }
	The function $\psi$ is continuous and respect the parametrization $r$, $\phi_{1,2}$ is  non-decreasing along stable and unstable laminations in the rectangle $\cO$. Even more
	$$
	\pi(\cO)=\Psi( (a_1,a_2)\times (b_1,b_2) ).
	$$
\end{lemm}

\subsubsection{The induced orientation}

Once we have a the parametrization $\Psi^{-1}$ of  $\pi(\cO)$, whose inverse coincides with $\pi$. We can used $\Psi$ to endow $\pi(\cO)$ with a vertical and horizontal orientations, that is coherent with the vertical and horizontal orientation of $\cO$ that are induced by the parametrization $r$.

\begin{lemm}\label{Lemm: pi preserve transversals}
The homeomorphism $\Phi: (a_1,a_2)\times (b_1,b_2) \times \cO$ preserve the transversal orientations of the vertical and horizontal foliations.
\end{lemm}

\begin{proof}
The parametrization  $r$ of $\cO$ determines the vertical and horizontal positive orientation, furthermore with this orientations, $\psi$ is non-decreasing along the stable and unstable leaves of $\cO$, and when we pass to the quotient $\Psi^{-1}$ is increasing along the stable and unstable leafs of $\pi(\cO)$. This property ensures that $\Phi$ preserves the transversal orientations.
\end{proof}

\begin{defi}\label{defi: orientations induced by pi}
The orientation given to $\pi(\cO)$ by the homeomorphism $\Psi$ is what we define like the \emph{orientation induced by $\pi$} in $\pi(\cO)$. The vertical and horizontal orientations of $\cO$ are such given by $\Phi$ in the therms of Definition  \ref{Defi: vertical and horizontal orientation}. With all this conventions we refer to $\pi(\cO)$ as the \emph{oriented rectangle induced by $\pi$}.
\end{defi}

\subsubsection{The image of rectangle is a rectangle.} The interiors of the rectangles in a Markov partition  satisfy the properties of Lemma \ref{Lemm: projection of rectangle by pi}, but our desire is the projection  by $\pi$ of every rectangle be a rectangle. Therefore, we need to extend the homeomorphism $\Psi$ to a parametrization of a rectangle.

\begin{lemm}\label{Lemm: pi(R) is a rectangle}
	For all $i\in \{1,\cdots,n\}$, $\pi(R_i)$ is a rectangle 
\end{lemm}

\begin{proof}

Take a rectangle $R$ in the Markov partition. This rectangle has non-isolated boundaries from its interior. In view of Lemma \ref{Lemm: projection of rectangle by pi}, $\pi(\overset{o}{R_i})$ is homeomorphic to a certain afine rectangle $\overset{o}{H}:=(a_1,a_2)\times (b_1,b_2)$ through the parametrization:

$$
\Phi:\overset{o}{H}\rightarrow \pi(\overset{o}{R}).
$$
Such parametrization sends (open) oriented vertical and horizontal intervals of $(a_1,a_2)\times (b_1,b_2)$ to oriented unstable and stable intervals of $\pi(\overset{o}{R})$, respectively.

It remains to show that this parametrization can be continuously extended to the closed rectangle $H=[a_1,a_2]\times [b_1,b_2]$ in such a way that its image is $\pi(R)$. Moreover, it is necessary for $H$ to be a homeomorphism restricted to any vertical or horizontal interval of $H$ whose image is contained in a single stable or unstable leaf of $\pi(R_i)$, as expected. In particular, the image of every boundary of the affine rectangle is contained in a single leaf.

Take a vertical segment $J\subset (a_1,a_2)\times [b_1,b_2]$. The closure of $\overset{o}{J}$ relative to the unstable leaf in which it is contained consists of adding the two boundary points to the image to obtain a closed interval: $\overline{\Psi(\overset{o}{J})}=\Psi(\overset{o}{J})\cup \{B_1,B_2\}$. Since $\pi(R)$ has the induced geometrization by $\pi$, $\Psi(\overset{o}{J})$ is positively oriented. Let $B_1$ be the lower point and $B_2$ be the upper point of $\overline{\Psi(\overset{o}{J})}$. The natural way to extend $\Psi$ to $J$ is given by the formula:

$$
\Psi(s,b_1)=B_1 \text{ and } \Psi(s,b_2)=B_2.
$$

We can do this for every unstable leaf and every stable leaf. The corners are the only points that remain to be determined. They will be naturally determined if we can prove that the image of the interior of the horizontal boundary of $[a_1,a_2]\times [b_1,b_2]$ lies in a single stable leaf, and the same holds for the image of the vertical boundary. Let us explain the idea.

Assume for a moment that the parametrization $\Psi$ sends the left side $\overset{o}{J}$ of $[a_1,a_2]\times (b_1,b_2)$ into an open interval contained in a unique unstable leaf $L^u$. Furthermore, assume that the lower boundary $I$ of $H$ is such that $\Psi(\overset{o}{I})$ lies in a unique stable leaf $F^s$. Then, $\Psi(\overset{o}{J})$ has a unique lower boundary point $A$, and $\Psi(\overset{o}{I})$ has a well-defined left endpoint $B$. These points are forced to be the same.

To visualize that they coincide, consider a sequence $J_n$ of vertical intervals of $[a_1,a_2]\times [b_1,b_2]$ that converges to the left side of $J$, and similarly, a decreasing sequence of intervals $I_n$ that converges to the lower side of the rectangle. The sequence of points $\{x_n\}=\Psi(\overset{o}{I_n})\cap \Psi(\overset{o}{J_n})$ (which lies in the interior of the rectangle) converges simultaneously to the lower point $A$ of $\Psi(I)$ and the left point $B$. Therefore, $\Psi$ is well-defined at the lower left corner point. The same argument can be applied to the other corner configurations.

So lets prove that the boundary is included in a  single leaf. Let be $s,s'\in(a_1,a_2)$ different numbers,  we need to observe that $\Psi(s,b_1)$  and $\Psi(s',b_1)$ are in the same stable leaf.

\begin{figure}[h]
	\centering
	\includegraphics[width=0.6\textwidth]{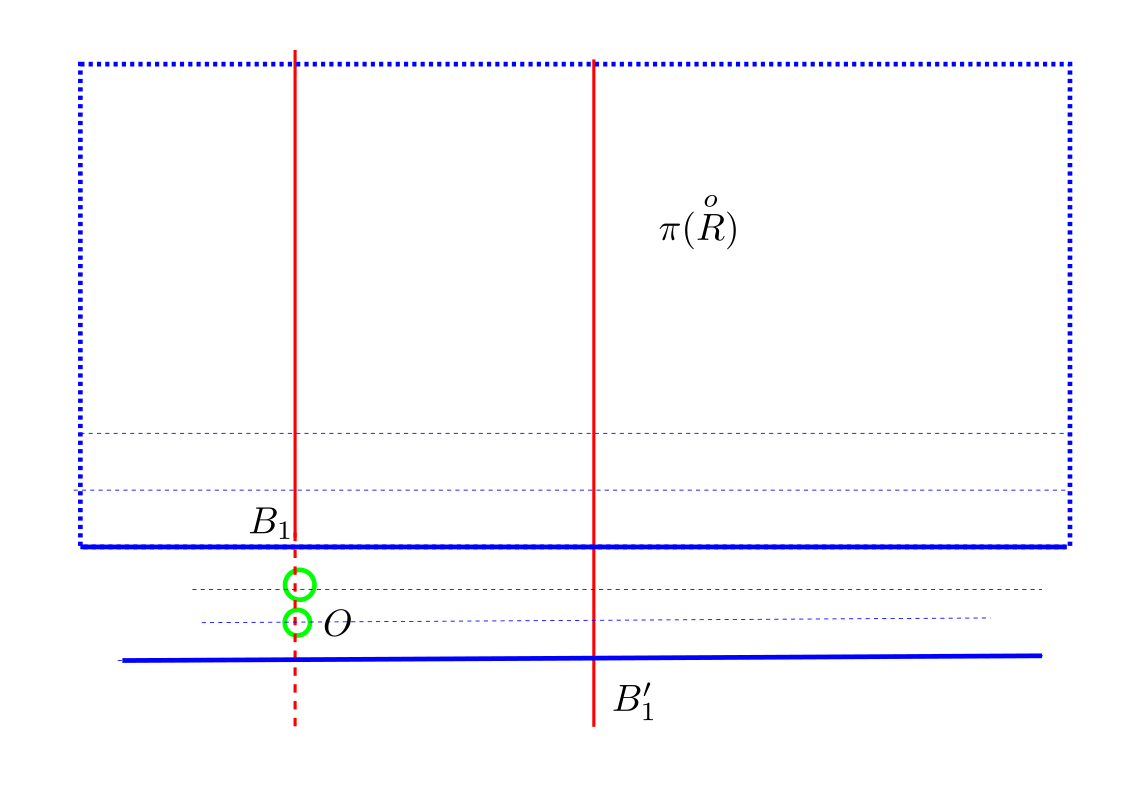}
	\caption{ Same unstable boundary.}
	\label{Fig: Same boundary}
\end{figure}

If $\Psi(s_1,b_1)=B_1$ and $\Psi(s'b_1)=B'_1$ are not in the same leaf, we would have a configuration like Figure \ref{Fig: Same boundary}. Using the horizontal orientation, we can assume that $B_1<B'_1$. However, this would create a problem because,  the intersection of a stable leaf near but above $B'_1$ with the vertical leaf passing trough $B_1$ would induce a point $O$ in the interior of the rectangle $\pi(\overset{o}{R})$. According to the trivial product structure in the interior, this configuration implies that $\Psi(s',b_1)=B'_1$ is in the interior of the corresponding unstable interval $\pi(\overset{o}{R})$, not in its boundary as intended. Therefore, $B_1$ and $B'_1$ must be in the same leaf.

Let $x$ be a stable boundary point of $h(R)$ and let $U$ be an open set containing $x$. The transverse measures of the pseudo-Anosov homeomorphism $\phi$ are Borel measures, and we can observe the existence of an open rectangle with a small enough diameter that can be sent inside $U$ under $\Psi$. This allows us to prove the continuity of $\Psi$ at the boundary points.

Take $J$ to be the interval such that $\Psi(J)$ has its lower point equal to $x$. Let $\epsilon > 0$. There exist vertical segments $\overset{o}{J_1}, \overset{o}{J_2} \subset H$ such that, in the horizontal order of rectangle $H$, $J_1 \leq J < J_2$, with equality only if $J$ is the left side of $H$, in which case it is not difficult to make the necessary adjustments to our arguments.

We have $\Psi(\overset{o}{J_1}) < \Psi(\overset{o}{J}) < \Psi(\overset{o}{J_2})$. We know that the length of every stable interval between $\Psi(\overset{o}{J_1})$ and $\Psi(\overset{o}{J_2})$ is a constant $C(1, 2) > 0$ (due to the invariance of the transverse measures under leaf isotopy). We can choose $J_1$ and $J_2$ to be sufficiently close to each other so that $C(1, 2) < \epsilon/4$.

Now, let $\overset{o}{I}$ be a horizontal leaf in $H$ such that the unstable segments contained into $\Psi(\overset{o}{J_1})$ and $\Psi(\overset{o}{J_2})$, with one endpoint in $\overset{o}{I}$ and the other in the stable leaf trough $x$, have $\mu^s$ length less than $\epsilon/4$.

Continuing from the previous point, we can choose the horizontal leaf $\overset{o}{I}$ to be sufficiently close to $x$ so that the unstable segments contained between $\Psi(\overset{o}{J_1})$ and $\Psi(\overset{o}{J_2})$, with one endpoint in $\overset{o}{I}$ and the other in the stable leaf of $x$, have length less than $\epsilon/4$.

In this way, by considering the rectangle whose boundary is given by the intersections of $I', J_1, J_2$, and the lower boundary of $H$, we can ensure that its image is contained in a rectangle with diameter less than $\epsilon$. For $\epsilon$ small enough, we can provide an open set $H$ whose image is completely contained in $U$. This completes the argument for continuity.
Finally we are going to see that:
$$
\Psi(H)=\pi(R).
$$

Let $x\in \partial^s R$. Without loss of generality, we can assume that $x$ is on the lower boundary of $R$. In order to show that $\pi(R)\subset \Psi(H)$, we need to prove that $\pi(x)$ is in $\Psi(H)$. The equivalence class of $x$ under the relation $\sim_R$ can either be a minimal rectangle or an unstable arc. In either case, there is a point in $K$ that belongs to the same equivalence class as $x$. Since all points in the equivalence class collapse to the same point, we can assume that $x\in K$. Let $x_n\in K\cap \overset{o}{R_n}$ be a sequence converging to $x$. By the continuity of $\pi$, we have $\pi(x_n)\rightarrow \pi(x)$.

After taking a refinement, we can assume that $\pi(x_{n+1})< \pi(x_{n})$ with respect to the vertical and horizontal orientation of the foliations induced by $\pi(\overset{o}{R})$. For every $\pi(x_n)$, there exists a unique point $y_n\in \overset{o}{H}$ such that $\Psi(y_n)=\pi(x_n)$. Moreover, we can assume that the sequence $(y_n)$ is decreasing with respect to both the vertical and horizontal orientation of $H$.

For each point $y_n$, it passes through a unique pair of stable and unstable intervals of $H$, denoted as $I_n$ and $J_n$ respectively. Since the intervals are decreasing and bounded from below with respect to the vertical order, there is a horizontal interval $I$ to which $I_n$ converges. Similarly, $J_n$ converges to an interval $J$, which may or may not be in the boundary of $H$. In any case, there exists a point $y\in I\cap J$ that is accumulated by the sequence $(y_n)$. By continuity, we have $\lim \pi(x)=\pi(x_n)=\lim \Psi(y_n)=\Psi(y)$.

To show the other inclusion $\Psi(H)\subset \pi(R)$, we consider a point $x\in H$ and aim to find a point $y\in R$ such that $\pi(y)=\Psi(x)$. We will focus on the case when $x$ lies on the inferior boundary of $H$, and the argument can be adapted for other boundary cases.

Let $J\subset H$ be the vertical interval passing through $x$. The set $\Psi(J)$ is a vertical interval in $\pi(R)$. Consider $\pi^{-1}(\Psi(\overset{o}{J}))$, which is a horizontal sub-rectangle of $R$ that may be reduced to an open interval. In any case, the stable boundary of this sub-rectangle belongs to a unique equivalence class under $\sim_R$. This equivalence class contains a point $k\in K\cap \partial R$. Let $J'$ be the unstable interval of $R$ that contains $k$ as one of its boundary points.
By continuity and the orientation-preserving properties of $\Psi$ and $\pi$, we have $\Psi(J)=\pi(J')$. Moreover, $\Psi(x)$ corresponds to the inferior boundary of $\pi(J')$ under this identification. Thus, we have found a point $y\in R$ such that $\pi(y)=\Psi(x)$, completing the proof.

\end{proof}

 \subsubsection{The projection of a Markov partition is a Markov partition}

\begin{prop}\label{Prop: pi cR is Markov partition}
	Let $\cR=\{R_i\}_{i=1}^n$ be a Markov partition of $f$. The family of rectangles $\pi(\cR)=\{\overline{\pi(\overset{o}{R_i})}\}_{i=1}^n$ is a  Markov partition of $\phi$.
\end{prop}

\begin{lemm}\label{Lemm: disjoint interiors}
	If $i\neq j$, then  $\pi(\overset{o}{R_i})\cap \pi(\overset{o}{R_j})= \emptyset$
\end{lemm}

\begin{proof}
	
	Please observe that $\pi\left(\overset{o}{R_i}\right)=\overset{o}{\left(\pi(R_i)\right)}$.  If $x\in \pi(\overset{o}{R_i})$, then $\pi^{-1}(x)$ is an equivalence class determined by the relation $\sim_R$, and we need to argue that $\pi^{-1}(x)\subset \overset{o}{R_i}$. This implies that if $\pi(\overset{o}{R_i})\cap \pi(\overset{o}{R_j}) \neq \emptyset$, then $\overset{o}{R_i}=\overset{o}{R_j}$.

	The points in $\pi^{-1}(x)$ do not belong to a cycle since cycles do not project to the interior of the rectangles. If $\pi^{-1}(x)$ is a singleton in the basic set, then $\pi^{-1}(x)\subset \overset{o}{R_i}$. If $\pi^{-1}(x)$ is a minimal rectangle, its stable (or unstable) boundary is different from the stable boundary of $R_i$ because $\pi^{-1}(x)$ is a minimal rectangle whose interior is disjoint from the basic piece $K$. Since the stable boundary of $R_i$ is isolated from one side, these two leaves are different, and the minimal rectangle lies in the interior of $R_i$. In the case where $\pi^{-1}(x)$ is an arc of a non-isolated leaf, we have $\pi^{-1}(x)\subset \overset{o}{R_i}$ because its endpoints do not belong to the stable boundary of $R_i$.
\end{proof}

\begin{lemm}\label{Lemm: the rectangles pi R cover S}
	The function $\pi$ restricted to $K$ is surjective, and $\cup_{i}^n \pi(R_i)=S'$.
\end{lemm}

\begin{proof}
	It is enough to observe that in any equivalence class of $\sim_{R}$ there is an element of $K$. Since $K$ is contained in $\cup_{i=1}^n R_i$, we have our result.
\end{proof}

\begin{lemm}\label{Lemm: Boundaries are f-invariant}
	The stable boundary of $\pi(\cR)$, $\partial^s \pi(\cR) = \cup_{i=1}^n \partial^s \pi(R_i)$, is $\phi$-invariant. Similarly, the unstable boundary of $\pi(\cR)$, $\partial^u \pi(\cR) = \cup_{i=1}^n \partial^u \pi(R_i)$, is $\phi^{-1}$-invariant.
	
\end{lemm}

 Proposition \ref{Prop: pi cR is Markov partition} follows from  Lemmas \ref{Lemm: pi(R) is a rectangle}, \ref{Lemm: disjoint interiors}, \ref{Lemm: the rectangles pi R cover S}, and \ref{Lemm: Boundaries are f-invariant}.

\begin{proof}
We need to show that $\partial^s \pi(\cR)=\cup_{i=1}^n\partial^s\pi(R_i)$ is $f$-invariant and $\partial^u \pi(\cR)=\cup_{i=1}^n\partial^u\pi(R_i)$ is $f^{-1}$-invariant. 
Given $i\in\{1,\cdots,n\}$, it is clear that $\partial^s\pi(R_i)=\pi(\partial^sR_i)$,This implies the following contentions and equalities between sets:
\begin{eqnarray*}
	\phi(\partial^s\pi(R_i))=\phi(\pi(\partial^sR_i))=\\
	\pi(f(\partial^sR_i))\subset \pi(\partial^s \cR)=\partial^s\pi(\cR).
\end{eqnarray*}
This proves $\partial^s \pi(\cR)$ is $\phi$-invariant. A similar argument proves the unstable boundary of $\pi(\cR)$ is  $\phi^{-1}$-invariant.
\end{proof}

\subsection{The surface $S$ is oriented}

Before to prove that a geometric Markov partition of $f$ projects to a geometric Markov partition we shall to determine  that the surface $S'$ is orientable, because this is the setting where a geometric Markov partition was defined (See Definition \ref{Defi: geometric Markov partition}).  Once we determine that $S'$ is orientable,  we can endow it with the orientation induced by the quotient homeomorphism $[\pi]: S/\phi \rightarrow S'$.  This orientation assigns to each rectangle in $\pi(\cR)$ the orientation induced by $\pi$.

The core of the proof in the next lemma that the orientation induced by $\pi$ in the rectangles of the Markov partition $\pi(\cR)$ is consistent along the boundary of the partition, which means that the orientations of adjacent squares in $\pi(\cR)$ match along their shared boundaries.

\begin{lemm}\label{Lemm: S is oriented}
	The surface $S'$ is orientable. 
\end{lemm}

\begin{proof}
Like the union of all $\pi(R_i)$ covers the whole surface $S'$, in order to obtain a orientation on it, is necessary to check that the orientations of the rectangles $\{\pi(R_i)\}_{i=1}^n$ are coherent in their boundaries. Let us study the situation of two adjacent stable boundaries
Let $x_i\in \partial^s R_i$ and $x_j\in \partial^s R_j$ be two points such that $\pi(x_i)=\pi(x_j)$. We will consider the case where they are not periodic $s$ or $u$-boundary points (i.e., they do not belong to a cycle). Without loss of generality, let's assume that $x_i$ and $x_j$ belong to the same unstable leaf $F^u$ and are both contained in $K$. 

Possibly after changing the vertical orientation of $R_i$ and/or $R_j$, it can be assumed that the horizontal direction of the stable boundaries where the points $x_i$ and $x_j$ of these rectangles meet, point to the respective periodic $s$-boundary points. This change the geometrization of the Markov partition, but  for this new geometrization, there is a unique orientation of the unstable leaf $F^u$ containing $J$ such that the vertical and horizontal orientations of $R_i$ and $R_j$, with respect to this geometrization, match the orientation induced by $\pi$. When we collapse the $u$-arc $J$ with end point $x_i$ and $x_j$, the stable and unstable directions of the leaves passing trough $\pi(x_i)=\pi(x_j)$ preserve their geometrization in every rectangle, and the resulting orientation at $\pi(x_j)$ is coherent for the two adjacent rectangles $\pi(R_i)$ and $\pi(R_j)$. (see \ref{Fig: Gluing orientation}).

\begin{figure}[h]
	\centering
	\includegraphics[width=0.4\textwidth]{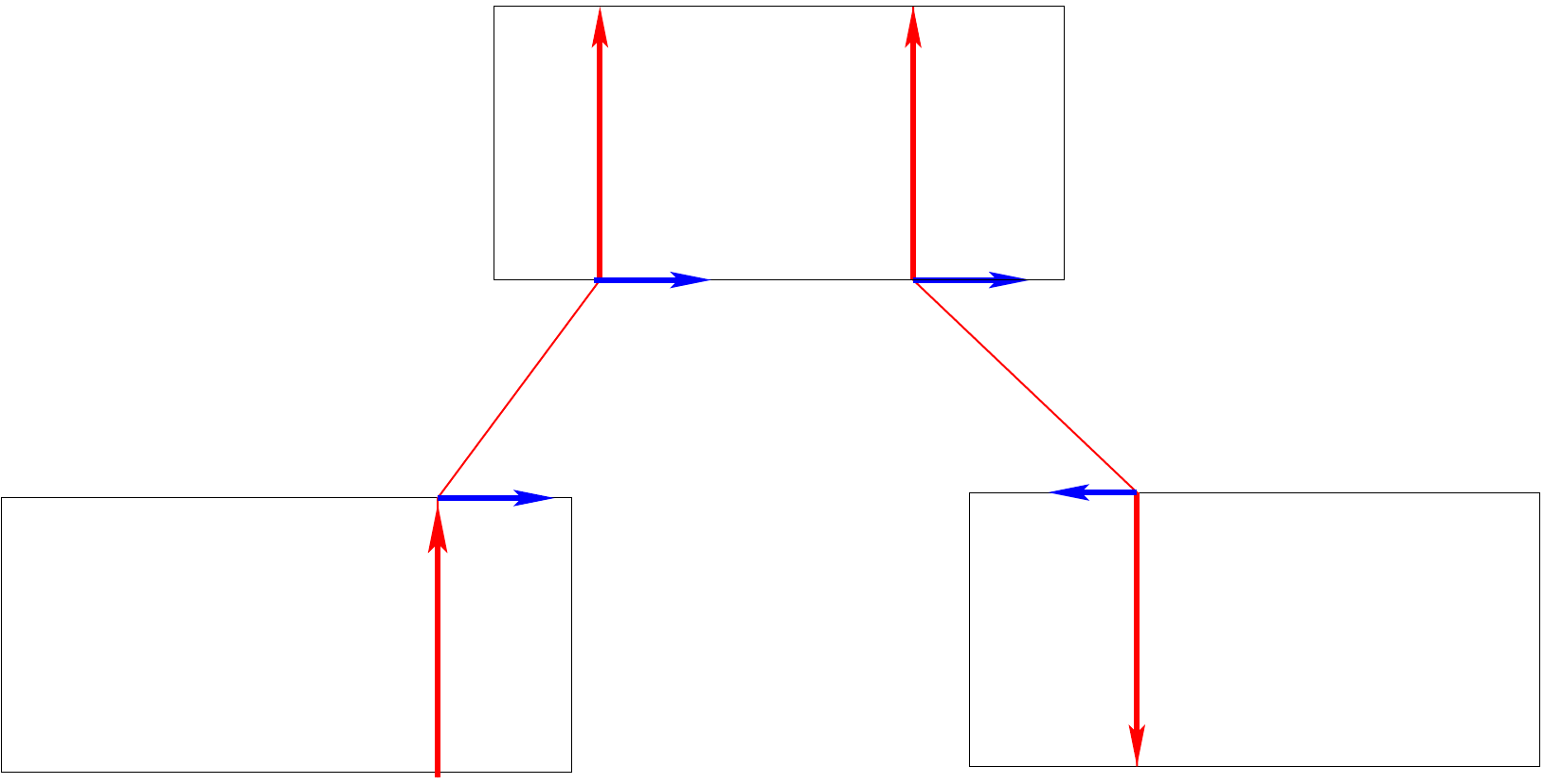}
	\caption{ The projection by $\pi$ match in the boundary}
	\label{Fig: Gluing orientation}
\end{figure}

The situation for a cycle is similar. Consider a cycle $[p_1,\cdots, p_k]$ and assign orientations to the non-free stable separatrices of each point in the cycle, pointing towards the periodic boundary point. Let's take a pair of rectangles $R_i$ and $R_j$ that contain two neighboring points $p_i$ and $p_j$ in the cycle. The stable boundaries of these rectangles intersect the two adjacent stable separatrices determined by $p_i$ and $p_j$.

We can orient $R_i$ and $R_j$ by completing their vertical orientations with the previously chosen horizontal orientation. After collapsing these two rectangles, $\pi(R_i)$ and $\pi(R_j)$ become two rectangles that share a stable interval in their boundary. The horizontal orientation of $\pi(R_i)$ and $\pi(R_j)$ is induced by the orientation of the two adjacent stable separatrices determined by $p_i$ and $p_j$. This forces a coherent orientation when we glue $\pi(R_i)$ with $\pi(R_j)$. This orientation consistency applies to all rectangles in the cycle, and the cyclic order around $\pi(p_i)$ must also be preserved.

\end{proof}

\begin{defi}\label{Defi: Orientation S'}
We endow the surface $S'$ with the orientation induced by the quotient homeomorphism $[\pi]: S/\phi \rightarrow S'$. 
\end{defi}

\subsection{The geometric type of the projection}
  The orientation of $S'$ given in Definition \ref{Defi: Orientation S'}, assigns to each rectangle in $\pi(\cR)$ the orientation induced by $\pi$ and this orientation is consistent along the boundary of the partition.

\begin{defi}\label{Defi: geometric partition induced by}
Let $\cR$ a geometric Markov partition of $f$ whose geometric type is $T$. Let $\pi(\cR)$ the geometric Markov partition of $\phi$ that consist of the  rectangles  whose orientation is induced by $\pi$,  its geometric type is denoted by $\pi(T)$ and we call $\pi(\cR)$ the geometric Markov partition induced by $\pi$.
\end{defi}

 Proposition \ref{Prop: type of basic piece is type of pseudo-Anosov} is consequence of this other proposition
 
 \begin{prop}\label{Prop: same type projection }
The geometric type $T$ is equal to $\pi(T)$
 \end{prop}
The proof will arrive from another Lemma.

\subsubsection{Same sub-rectangles:} For each $i \in \{1, \ldots, n\}$, the rectangle $R_i$ of the Markov partition $\cR = \{R_i\}_{i=1}^n$ for the Smale diffeomorphism $(f, S)$ has horizontal sub-rectangles $\{H^i_j\}_{j=1}^n$, which are enumerated in increasing order with respect to the vertical order in $R_i$. Similarly, the horizontal sub-rectangles of the rectangle $\pi(R_i)$ of the Markov partition $\pi(\cR)$ for the pseudo-Anosov homeomorphism $(\phi, S')$ are enumerated as $\{\underline{H^i_j}\}_{j=1}^{\underline{h_i}}$ in increasing order with respect to the vertical orientation of $\pi(R_i)$ induced by $\pi$.

\begin{lemm}\label{Lemm: pi send sub-rec in sub rec}
With the previous notation, we observe that $h_i = \underline{h_i}$, meaning that the number of horizontal sub-rectangles in $R_i$ is equal to the number of horizontal sub-rectangles in $\pi(R_i)$. Furthermore, for each $j \in \{1, \ldots, n\}$, we have $\pi(H^i_j) = \underline{H^i_j}$.

Similarly, if $\{V^k_l\}_{l=1}^{v_k}$ are the vertical sub-rectangles of $R_k$ and $\{\underline{V^k_l}\}_{l=1}^{\underline{v_k}}$ are the vertical sub-rectangles of $\pi(R_k)$, we have $v_k = \underline{v_k}$, indicating that the number of vertical sub-rectangles in $R_k$ is equal to the number of vertical sub-rectangles in $\pi(R_k)$. Moreover, we have $\pi(V^k_l) = \underline{V^k_l}$ for each $l \in \{1, \ldots, v_k\}$.

\end{lemm}

\begin{proof}
		 
Let $H:=H^i_j$ be a connected component of $R_i \cap f^{-1}(R_k)$, which is an horizontal sub-rectangle of the Markov partition $\cR$. Therefore, $\overset{o}{H}$ is a connected component of $\overset{o}{R_i} \cap f^{-1}(\overset{o}{R_k})$ for certain indexes of rectangles.
	 
 Now, let $\pi(\overset{o}{H})$ is a connected component of $\pi(\overset{o}{R_i}) \cap \pi(f^{-1}(\overset{o}{R_k}))= \overset{o}{\pi(R_i)} \cap \phi^{-1}(\pi(\overset{o}{R_k}))$. Therefore $\pi(\overset{o}{H})$ is an horizontal sub-rectangle of $\pi(\cR)$. If $H_1,H_2$ are two different  horizontal sub-rectangles of $R_i$. Their interior don't intersect and a adaptation of the argument in \ref{Lemm: disjoint interiors} implied $\pi(\overset{o}{H_1})\cap \pi(\overset{o}{H_2})=\emptyset$. Therefore there is a bijection between the horizontal sub-rectangles of $R_i$ and those of $\pi(R_i)$.
 
The projection $\pi$ preserve the vertical order of the rectangles, then the rectangles $H^i_{j}$ and $H^i_{j+1}$ are projected in adjacent rectangles preserving the order, hence $\pi(H^i_{j})=\underline{H^i_{j}}$.

The proof of the assertions concerning the vertical sub-rectangles is completely similar.
\end{proof}

\begin{coro}\label{Coro: Same sub-rectangle}
For all $i\in \{1,\cdots, n\}$ and $j\in \{1,\cdots,h_i\}$, $\phi(\underline{H^i_j})=\underline{V^k_l}$ if and only if $f(H^i_j)=V^k_l$.
\end{coro}

\begin{proof}

If $f(H^i_j)=V^k_l$, then $\pi(f(H^i_j))=\pi(V^k_l)$. Using the semi-conjugation, we obtain that $\phi(\pi(H^i_j))=\phi(\underline{H^i_j})=\underline{V^k_l}$.

Conversely, if $\phi(\underline{H^i_j})=\underline{V^k_l}$, then $\underline{H^i_j}=\pi(H^i_j)$ and $\underline{V^k_l}=\pi(V^k_l)$. In the case that $f(H^i_j)=V^{k'}{l'}$, it is clear that $\underline{V^k_l}=\pi(f(H^i_j))=\pi(V^{k'}{l'})=\underline{V^{k'}_{l'}}$, which implies $k'=k$ and $l'=l$.

\end{proof}

\begin{rema}\label{Rema: pi preserve the orientation}
As the orientation in $\pi(R_i)$ is induced by $\pi$, it is clear that $\pi$ preserves the vertical orientation relative to $R_i$ and $\pi(R_i)$. Similarly, $\pi$ preserves the horizontal orientation relative to $R_k$ and $\pi(R_k)$.

This implies that $f$ restricted to $H^i_j$ preserves the vertical orientation if and only if $\phi$ restricted to $\underline{H^i_j}$ preserves the vertical orientation.
\end{rema}

\subsubsection{Proof of proposition \ref{Prop: type of basic piece is type of pseudo-Anosov}. } \label{Subsec: Proof Proposition}
 After \ref{Coro: Same sub-rectangle} and Remark \ref{Rema: pi preserve the orientation} we can conclude the proof of Proposition \ref{Prop: type of basic piece is type of pseudo-Anosov}.
 
\begin{proof}
	Let  $T(\cR)=\{n,\{(h_i,v_i)\}_{i=1}^n, (\rho,\epsilon)\}$ be the geometric type of $\cR$, and let $\pi(T)=\{n',\{(h'_i,v'_i)\}_{i=1}^{n'}, (\rho',\epsilon')\}$  be the geometric type of $\pi(\cR)$.
	
	The number of rectangles in $\cR$ and $\pi(\cR)$ is equal to $n$, and for every $i\in \{1,\cdots, n\}$, Lemma \ref{Lemm: pi send sub-rec in sub rec} implies that $h_i=h_i'$ and $v_i=v_i'$.
	
	The function $\rho(i,j)=(k,l)$ if and only if $f(H^i_j)=V^k_l$. Corollary \ref{Coro: Same sub-rectangle} states that this occurs if and only if $\phi(\underline{H^i_j})=\underline{V^k_l}$. In terms of the geometric type, this is equivalent to the fact that $\rho'(i,j)=(k,l)$.
	
	Finally, Remark \ref{Rema: pi preserve the orientation} tells us that $\epsilon(i,j)$ is $1$ if and only if $\epsilon'(i,j)$ is $1$. This completes our proof.

\end{proof}
 
\section{All geometric types of the pseudo-Anosov class have finite genus and no impasse.}\label{Sec: Tipes PA finite genus but not impasse}

In this section, we assume that $T$ is a geometric type in the pseudo-Anosov class. We are going to prove that its incidence matrix is mixing, it genus is finite and does not display any impasse. The arguments of this section use the equivalence between the combinatorial and topological conditions of finite genus and impasse.

Nowadays the following Proposition is a  classic result, its proof could be found for example in \cite{farb2011primer} and \cite{fathi2021thurston} and directly implies  the first Item of Proposition \ref{Prop: pseudo-Anosov iff basic piece non-impace}.

\begin{prop}\label{Prop: Incidence matrix mixing}
	Let $\cR$ be the Markov partition of a generalized pseudo-Anosov homeomorphism  then its incidence matrix is \emph{mixing}, i.e. there exists a number $n\in \NN$ such that for all $i,j\in {1,\cdots,n}$, $a_{i,j}^{(m)}>0$.
\end{prop}

According to Proposition \ref{Prop: Types transitive have a realization} any geometric type without double boundaries admits a realization. The following Lemma can be applied to a geometric type $T$ in the pseudo-Anosov class and implies that $T$ does not have double boundaries.

\begin{lemm}\label{Lemm: T mixing implies non double boundaries}
If the incidence matrix of $T$ is mixing, then $T$ does not have double boundaries.
\end{lemm}

\begin{proof}

	Let $\{e_i\}_{i=1}^n$  be the canonical basis of $\mathbb{R}^n$, and let $A$ be the incidence matrix of $T$. We recall that the entry $a_{ij}$ represents the number of horizontal sub-rectangles of $R_i$ that are mapped to $R_j$. In particular, if $h_i=1$, there exists a unique $j_0$ such that $f(R_i)\cap R_{j_0}$ is a horizontal sub-rectangle of $R_{j_0}$. This implies that $a_{ij_0}=1$, and for all $j\neq j_0$, $a_{ij}=0$.
	
	The condition of double $s$-boundaries implies the existence of $\{i_s\}_{s=1}^S\subset\{1,\cdots,n\}$  such that$A(e_{i_s})=e_{i_{s+1}}$ for $1\leq s<S$ and $A(e_{i_S})=e_{i_1}$. This forces $A^m$ to not be positive definite for any $m\in \mathbb{N}$, which contradicts the mixing property of the incidence matrix.
	
	Therefore, if $T$ has a mixing incidence matrix, it does not have double $s$-boundaries.
	
\end{proof}

This lemma have the following corollaries.

\begin{coro}\label{coro: pseudo Anosov non-doble boundaries}
	If $T$ is in the pseudo-Anosov class, then $T$ does not have double boundaries.
\end{coro}

\begin{coro}\label{Coro: T pA then realization}
If $T$  is a geometric type in the pseudo-Anosov class, $T$ it admits a realization.
\end{coro}

\subsection{Some convention in notation and orientations}

\subsubsection{Corresponding rectangles and curves.}
 Let $T$ be a geometric type in the pseudo-Anosov class that we denoted by:
$$
T:=\{n,\{h_i,v_i\}_{i=1}^n, \Phi:=(\rho,\epsilon)\}.
$$
In future discussions, we will compare two different objects associated to a to the geometric type $T$:

\begin{itemize}
\item[i] A \emph{realization} by \emph{disjoint rectangles} of the geometric type $T$. We shall denote such realization by a pair \emph{partition/diffeomorphism}: $\{\cR=\{R_i\}_{i=1}^n, \phi\}$. 

\item[ii)] A \emph{geometric Markov partition} by non-disjoint rectangles, $\tilde{\cR}=\{\tilde{ R_{i}}\}_{i=1}^n,$ of a pseudo-Anosov homeomorphism $f:S\rightarrow S$ with geometric type $T$.
\end{itemize}

We shall to fix some notations:

\begin{itemize}
\item[i)] For all $i\in \{1,\cdots,n\}$, $\tilde{R_i}$ is a rectangles in the Markov partition and $R_i$ is a rectangle in the realization.

\item[ii)]  The vertical and horizontal sub-rectangles of the realization are denote by:
$$
\{V^k_l: (k,l)\in \cV(T)\} \text{ and } \{H^i_j: (i,j)\in \cH(T)\}
$$.
\item[iii)] The vertical and horizontal sub-rectangles of the Markov partition are denoted by:
$$
\{\tilde{V^k_l}:(k,l)\in \cV(T)\} \text{ and  } \{ \tilde{H^i_j}: (i,j)\in \cH(T) \}.
$$

\item[iv)] If $\alpha$ is the lower boundary of a vertical sub-rectangle $V^k_l$ in the realization, then $\tilde{\alpha}$ is the lower boundary of the corresponding vertical sub-rectangle $\tilde{V^k_l}$ in the Markov partition. Similarly, if $\alpha$ is the left boundary of $H^i_j$, then $\tilde{\alpha}$ is the left boundary of $\tilde{H^i_j}$.

\item[v)] Take two vertical sub-rectangles $V^k_{l_1}$ and $V^k_{l_2}$, with $l_1<l_2$, and let $\alpha$ and $\beta$ be the horizontal boundaries of these sub-rectangles. Suppose both boundaries are contained in either the upper or lower boundary of $R_k$. Without loss of generality, let's assume they are in the lower boundary of $R_k$. Using the fact that $l_1<l_2$, we can write $\alpha=[a_1,a_2]^s<\beta=[b_1,b_2]^s$, and we denote 
$$
\gamma=[a_2,b_1]^s
$$
 as the curve contained in the lower boundary of $R_k$ and lying between $\alpha$ and $\beta$.

\item[vi)] The curve $\tilde{\gamma}$ is the curve contained in the lower stable boundary of $\tilde{R_k}$ that corresponds to the curve obtained by taking the union of the lower boundaries of all vertical sub-rectangles in the Markov partition $\tilde{\cR}$ that lie between $\tilde{V^k_{l_1}}$ and $\tilde{V^k_{l_2}}$. This is
$$
\tilde{\gamma}:=\cup_{l_1<l<l_2}\partial^s_{-}\tilde{V^k_l}.
$$

\end{itemize}

With respect to Item $v)$ and $vi)$: if $l_2=l_1 +1$ the curve $\tilde{\gamma}$ consist in a point but $\gamma$ is a non-trivial interval.

\begin{rema}\label{Rema: geometric type rules the images}
Suppose that: $\Phi_T(i,j)=(k,l,\epsilon)$. This implies that: $\phi(\tilde{H^i_j}) = \tilde{ V^k_l}$ and  $f(H^i_j)=V^k_l$.
\end{rema}

\subsubsection{Orientation} Here we should define three different orientations of a curve contained in the stable boundary of the Markov partition and the realization. This will complete the notation that we require to deduce that the realization doesn't have the topological obstructions by looking at what kind of obstruction it induces in the Markov partition.

Let $H^i_j$ and $V^k_l$ be a horizontal and a vertical sub-rectangle of $R_i$ in the realization. They have orientations induced by the rectangle $R_i$ in which they are contained. The function $\epsilon_T$ in the geometric type $T$ measures the change in relative vertical orientation induced by the action of $f$.

\begin{defi}\label{Defi: induced and glin for realization}
	Let $\alpha$ be a horizontal boundary component of $V^k_l. $ Since $R_k$ is a geometrized rectangle, there are two types of orientations for this curve:
\begin{itemize}
	\item The \emph{induced orientation} of $\alpha$ is the one that is inherited from the horizontal orientation of $R_k$. This means that $\alpha$ is oriented in the same direction as the horizontal sides of $R_k$.
	
	\item  	The \emph{gluing orientation} of $\alpha$, defined as follows: If $\epsilon_T(i,j) = 1$, then $\alpha$ has the induced orientation. However, if $\epsilon(i,j) = -1$, then $\alpha$ has the inverse orientation of the induced orientation. In other words, if the change in relative orientation between $H^i_j$ and $V^k_l$ is negative, then the gluing orientation of $\alpha$ is the opposite of the induced orientation.
	
\end{itemize}
\end{defi}

Similarly, for the Markov partition, we can define the induced orientation and gluing orientation of the curve  $\tilde{\alpha}$.

\begin{defi}\label{Defi: Induced/gluin in Markov}
	Let  $\tilde{\alpha}$ be the corresponding curve in the Markov partition:
\begin{itemize}
	\item The \emph{induced orientation} of $\tilde{\alpha}$ is the one that is inherited from the horizontal orientation of $\tilde{R}_k$.
	
	\item  The \emph{gluing orientation} of $\tilde{\alpha}$ is defined in a similar manner as for the concretization. If $\epsilon(i,j) = 1$, then $\tilde{\alpha}$ has the induced orientation. However, if $\epsilon(i,j) = -1$, then $\tilde{\alpha}$ has the inverse orientation of the induced orientation.
\end{itemize}
\end{defi}

 If $\tilde{\alpha}$ is contained in the stable boundary of two rectangle $R_1$ and $R_2$, the induced orientation of $\tilde{\alpha}$ coincides with its glue orientation if and the horizontal orientation of $R_1$ and $R_2$ is coherent along the stable boundaries that contains $\tilde{\alpha}$.

The boundary of a geometric Markov partition is contained in the stable leaves of periodic points and we can define another orientation in such stable boundaries.

\begin{defi}\label{Def: dynmaic orientation}
	Let $\tilde{O}$ be a periodic point of $f$ and $\tilde{L}$ be a stable separatrix of $\tilde{O}$. The  \emph{dynamic orientation} of the separatrix  $\underline{L}$ is defined by declaring that: for all $x \in \tilde{L}$, $f^{2Pm}(x) < x$, where $P$ is the period of $f$. 
\end{defi}

This orientation points towards the periodic point. Suppose that $\tilde{L}$ intersects the stable boundary of $\tilde{R_i}$ in an interval $\tilde{I}$, which may or may not contain a periodic point $\tilde{P}$. The \emph{dynamic orientation} of $\tilde{I}$ is the orientation induced by the dynamic orientation of the separatrix  $\tilde{L}$ within the interval. Let's address the case of the realization.

\begin{defi} \label{Defi: Dynac orinted for realization}
Suppose $O$ is a periodic point located on the horizontal boundary $L$ of $R_i$. Let $I$ be a connected component of $L \setminus \{O\}$. It can be observed that $I$ is mapped into  a smaller interval under the action of $\phi$, and after a certain number of iterations, it becomes contained within the interior of $I$.  This contraction property in  us to define the \emph{dynamical orientation} of $I$ in such a way that it points towards the periodic point $O$.
\end{defi}

Suppose that in the realization, there is a ribbon $r$ joining the boundary $\alpha$ of $V^k_l$ with the boundary $\beta$ of $V^{k'}_{l'}$. This configuration implies that in the Markov partition of $f$, the boundary $\tilde{\alpha}$ of $\tilde{V^k_l}$ is identified with the boundary $\tilde{\beta}$ of $\tilde{V^{k'}_{l'}}$. This identification preserves the \emph{gluing orientation} of $\tilde{\alpha}$ and $\tilde{\beta}$. However, it is possible for the curves to have different induced orientations. The induced orientation may vary along the curves, but the gluing orientation remains the same, ensuring coherence in the identification process.

\subsection{Finite genus and not impasse} After establishing the formalism, we proceed to prove that $T$ has no impasse and none of the obstructions for finite genus. We begin by assuming that $T$ satisfies the topological conditions for have fine genus and not impasse in the first realizer $\cR_1$, which implies the existence of certain ribbons connecting the boundaries $\alpha$ and $\beta$ of vertical sub-rectangles in the realization. This implies that $\tilde{\alpha}$ and $\tilde{\beta}$ are identified in the Markov partition, and this identification must be consistent with the induced, gluing, and dynamic orientations previously defined. By carefully considering these orientations, we will be able to derive certain contradiction, thereby establishing that $T$ has finite genus and not impasse.

In the next lemma we are going to prove that if there is a stripe that joints two stable intervals inside a rectangle $R_k$ of the realization implies that such rectangle contain a spine in one of its horizontal boundaries, like is indicated in Figure \ref{Fig: Colapse point}. 

\begin{figure}[h]
	\centering
	\includegraphics[width=0.4\textwidth]{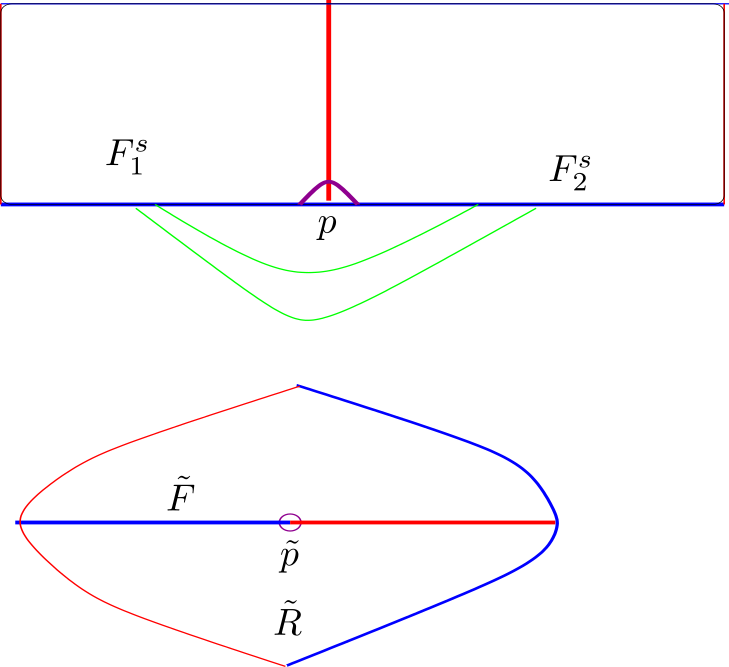}
	\caption{ Two stable separatrices collapse in a Spine }
	\label{Fig: Colapse point}
\end{figure}

\begin{lemm}\label{Lemm: gamma periodic points}
Consider a rectangle $R_k$, with the horizontal orientation of $R_k$ suppose there exists a ribbon joining the boundaries $\alpha=[a_1,a_2]^s\subset V^k_{l_1}$ and $\beta:=[b_1,b_2]^s \subset V^k_{l_2}$, assume that $l_1<l_2$ and  let $\gamma=[a_2,b_1]^s$ the curve between $\alpha$ and $\beta$. Let  $\tilde{\alpha}:=[a_1',a_2']^s$, $\tilde{\beta}:=[b_1',b_2']^s$ and $\tilde{\gamma}:=[a_1',b_2']$ be the previously defined curves with the induced orientation in $\tilde{R_k}$. In this circumstances:
\begin{itemize}
\item[i)] The curve $\tilde{\gamma}$ is a closed interval with one end point equal to a Spine type periodic point of $f$.
\item[ii)] The curve $\gamma$ contains a periodic point.
\end{itemize}
\end{lemm}

\begin{proof}

The curve $\tilde{\gamma}$ is a closed manifold of dimension one, it is either an interval or a circle. Since $\tilde{\gamma}$ is contained in the stable leaf $I$ of a periodic point $\tilde{ O}$ of the pseudo-Anosov homeomorphism $f$, it couldn't be a circle therefore is closed interval $I$. We claim that such interval must contains an end  point of a stable separatrix  of a periodic point  $\tilde{O}$ of $f$, and then such end point is $\tilde{ O}$.

 Look at figure \ref{Fig: tildeRk} to follow our next arguments, the curve $\tilde{\gamma}$ is the union of all the lower  (or upper) boundaries of $\tilde{ V^k_l}$ with $l_1<l<l_2$ an it is contained in inferior boundary of $\tilde{R_k}$. It could be a single point if $l_2=l_1$ or the not trivial interval, $\tilde{\gamma}:=[a_1',b_2']$, lets  to consider the last situation, let $\tilde{o}$ the extreme point of $I$ that is not in $\tilde{\alpha}$ nor $\tilde{\beta}$.
 
 \begin{figure}[h]
 	\centering
 	\includegraphics[width=0.4\textwidth]{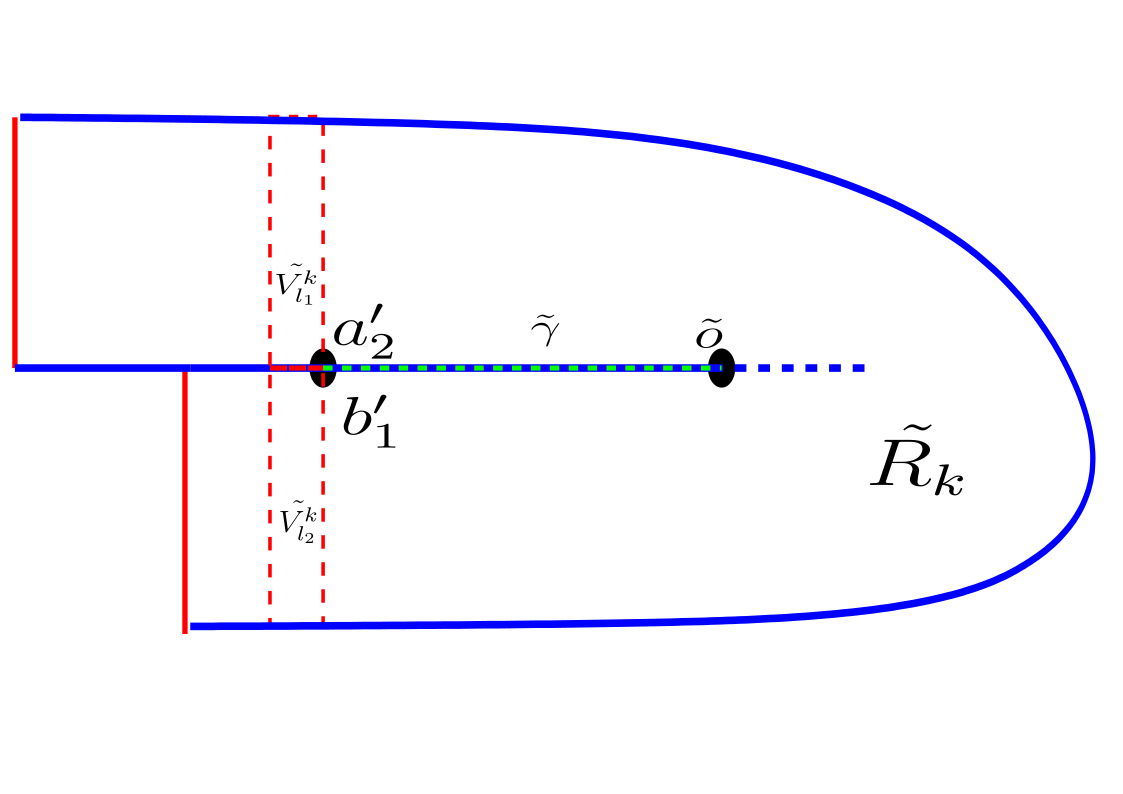}
 	\caption{Rectangle $\tilde{R_k}$ }
 	\label{Fig: tildeRk}
 \end{figure}

Let $\Psi: H := [0,1]\times[0,1] \rightarrow \tilde{R_k}$ the parametrization of the rectangle $\tilde{R_k}$. Following the induced horizontal orientation of $\tilde{R_k}$ the point $\tilde{o}$ lies in the interior of the lower boundary of $\tilde{R_k}$ and we can assume that $\Psi(t_0,0) = \tilde{o}$ for some $t_0 \in (0,1)$. Since $\tilde{R_k}$ is a rectangle and there exist two closed intervals $A_1$ and $A_2$ in $H$ such that by the action of $\Phi$ these intervals are identified, i.e. $\Phi(A_1) =\tilde{ \alpha}= \tilde{\beta}=\Psi(A_2)$, then there exists closed  semicircle $U'\Subset \overset{o}{H}\cup \partial^s H$ with center at $(t_0,0)$ that contains its  that $\Psi(U')=U'\Subset \overset{o}{R_k}$ (see Figure) . Consequently, $\tilde{o}$ has an open neighborhood $\overset{o}{U}$ that is contained in the interior of $\tilde{R_k}$, that have center in $\tilde{o}$. Moreover, it is clear that the diameter of $U'$ was collapsed by the action of $\Phi$ in a closed interval with a extreme in $\tilde{o}$ and therefore $\tilde{o}$ has only one separatrice. This allows us to deduce that $\tilde{o}$ is a Spine, and so it is equal to the periodic point $\tilde{O}$.

 \begin{figure}[h]
	\centering
	\includegraphics[width=0.5\textwidth]{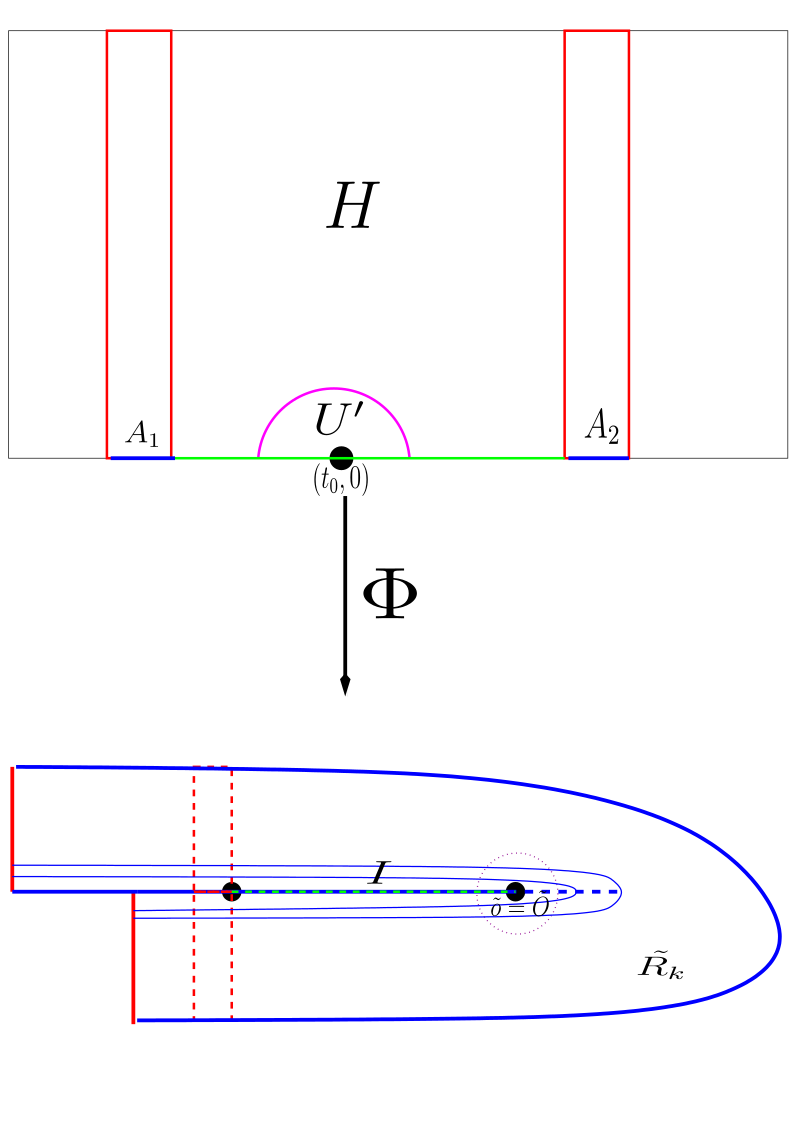}
	\caption{The projection by $\Phi$ of the rectangle $H$}
	\label{Fig: Rectangle H}
\end{figure}

The situation when $I$ is a point is similarly proved.

\end{proof}

\begin{coro}\label{Coro: T pA clas no com impasse}
A geometric type $T$ in the pseudo-Anosov class does not have combinatorial impasses.
\end{coro}

\begin{proof}
First, observe that the geometric type $T$ does not have the impasse condition. If there were an impasse, it would be in a situation as described in the previous lemma, where $\tilde{\alpha}$ and $\tilde{\beta}$ are the horizontal boundaries that are identified. The curve $\tilde{\gamma}$ would be reduced to a periodic point, and $\gamma$ would have a periodic point. However, $\gamma$ is the $s$-arc given by the impasse, so it cannot have periodic points. This contradiction implies that $T$ does not have the impasse condition.

Now, if $T$ represents (via a Markov partition) the pseudo-Anosov homeomorphism $f$, then $f^n$ has a Markov partition of geometric type $T^n$. Since $f^n$ is pseudo-Anosov, $T^n$ is also in the pseudo-Anosov class and does not have the impasse condition. This implies that $T$ does not have an impasse.
\end{proof}

\begin{lemm}\label{Lemm: T pA class then no condition 1}
A geometric type $T$ in the pseudo-Anosov class does not satisfy the combinatorial condition of type $(1)$.
\end{lemm}

\begin{proof}

Suppose $T$ satisfies the combinatorial condition of type $1$. This implies that the $1$-realization of $T$ has the topological obstruction of type $1$, according to the Lemma \ref{Lemm: Equiv com and top type 1}. Therefore, there exist curves $\alpha$ and $\alpha'$ in the boundary of a rectangle $R_k$ in the realization that are joined by a ribbon $r$. Let $\gamma$ the curve that is comprised between $\alpha$ and $\alpha'$ with respect the vertical orientation of $R_k$.  Then  there exists another ribbon $r'$ with a horizontal boundary $\beta$ that is contained within the interior of $\gamma$, while the other boundary of $r'$ is located outside of $\gamma$.

The curve  $\tilde{\gamma}$ contains  $\tilde{\beta}$  in particular $\tilde{\gamma}$ is a non-trivial interval and due to the identification induced by the ribbon $r'$ it contains $\tilde{\beta'}$. However, like it was established in Lemma \ref{Lemm: gamma periodic points}, the stable boundary of $R_k$ is collapsed into a Spine, the curve $\gamma$ have a periodic point $p$ in their interior and moreover bout separatrices of $p$ are identified. 

Is not difficult to see that the  gluing orientation and the dynamical orientation of $\gamma$ must be the same, for that reason look at Figure \ref{Fig: Type 1 obstruction} where:

\begin{itemize}
\item The horizontal h orientation of the stable boundary of $R_k$ is indicated with a blue solid arrow.
\item The induced orientation in $\tilde{\gamma}$ is the indicated with two dotted blue lines.
\item The gluing orientation in $\tilde{\gamma}$ is the indicated with two dotted black lines,and clearly they point towards the spine $\tilde{ O}$.
\end{itemize}

  Hence, any interval $I\subset \gamma$ that is  one separatrice of $p$ must be identify with another interval $I'$ in the other stable separatrice of $p$, and they are the only intervals that are identify. In particular $\beta$ must me identify with another interval inside $\gamma$, like this interval is $\beta'$ we deduce that $\beta'$ is inside $\gamma$ and this is a contradiction.

\begin{figure}[h]
	\centering
	\includegraphics[width=0.7\textwidth]{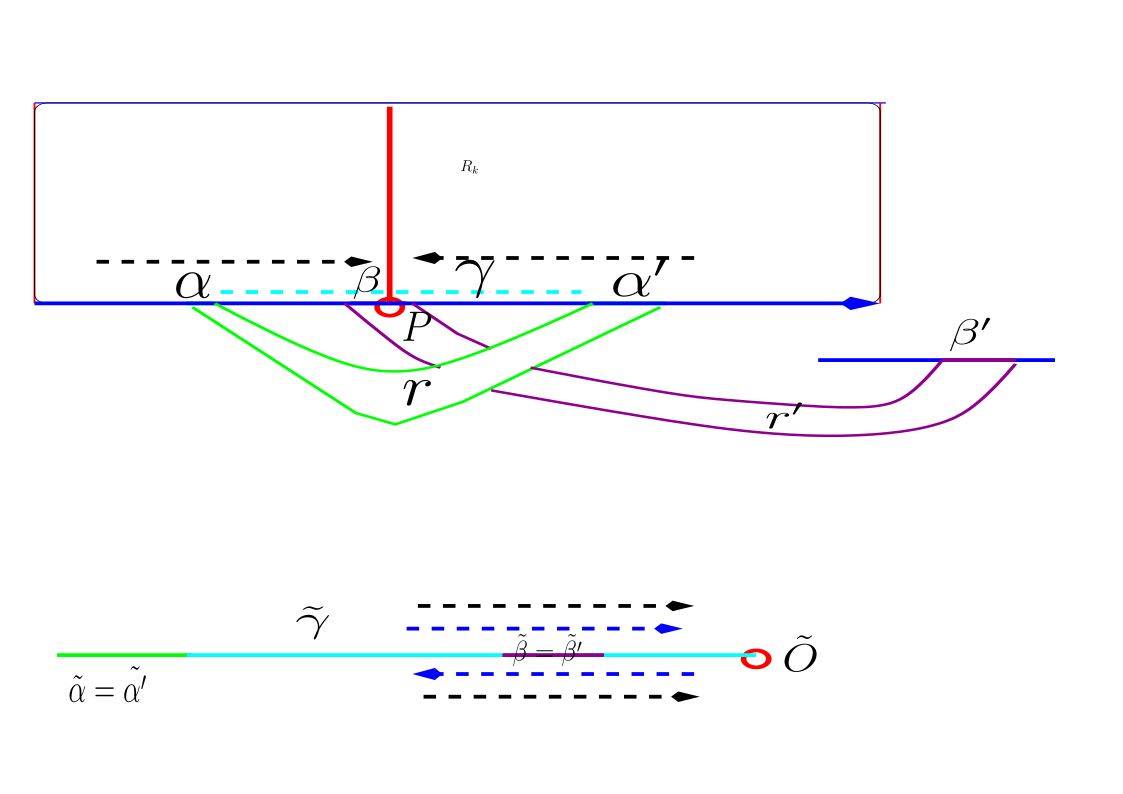}
	\caption{ Type $1$ obstruction}
	\label{Fig: Type 1 obstruction}
\end{figure}
\end{proof}

If $T$ is a geometric type in the pseudo-Anosov class and $f$ is a generalized pseudo-Anosov homeomorphism that realizes $T$, it follows that $T^n$ is the geometric type of a Markov partition of $f^n$. Since $f^n$ is also pseudo-Anosov, we can conclude that $T^n$ does not have the type $(1)$ property or an impasse. In particular, for any positive integer $n$, $T^{6n}$ does not have the type $(1)$ property, and $T^{2n+1}$ does not have an impasse. This discussion leads to the following corollary.

\begin{coro}\label{Lemm: T pA class then no obstruction 1}
A geometric type $T$ in the pseudo-Anosov class don't have the type-$(1)$ obstruction or impasse.
\end{coro}

\begin{figure}[h]
	\centering
	\includegraphics[width=0.7\textwidth]{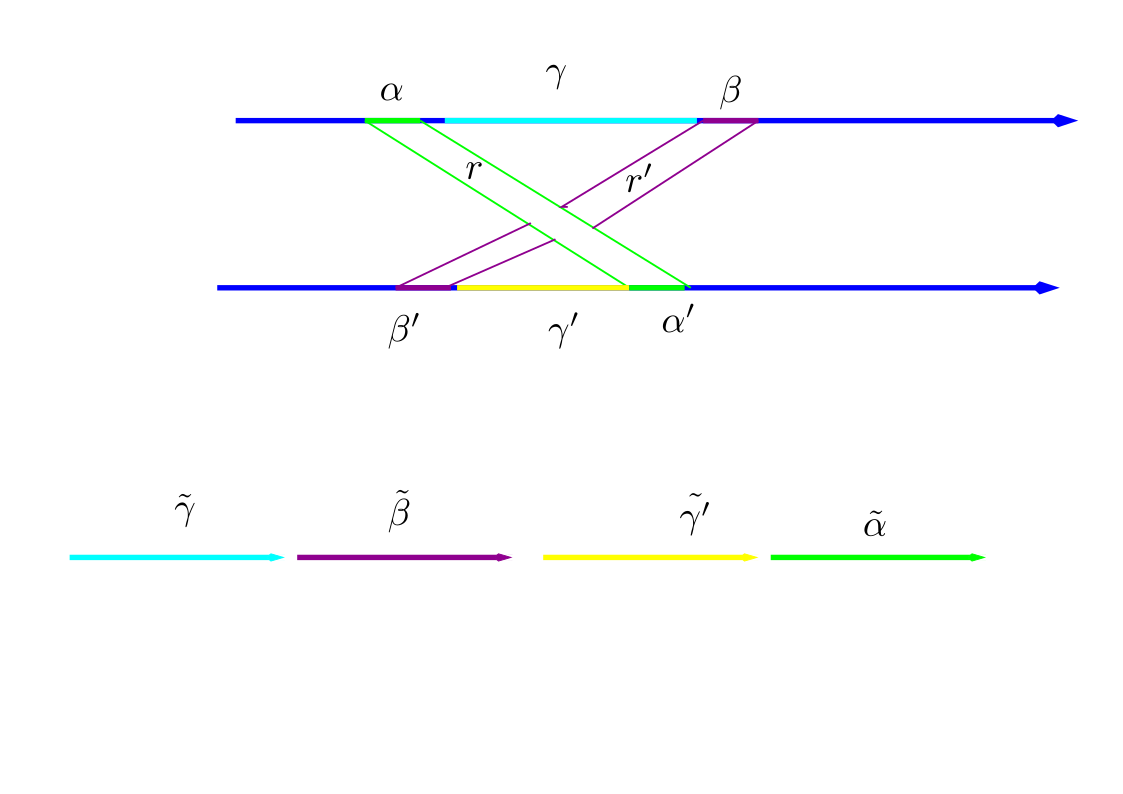}
	\caption{ Type $2$ obstruction}
	\label{Fig: Type 2 obstruction}
\end{figure}

\begin{lemm}\label{Lemm: T pA class then no condition 2}
A geometric type $T$ in the pseudo-Anosov class does not have the combinatorial condition of type-$(2)$. 
\end{lemm}

\begin{proof}

Assume this is not true. In view of Lemma \ref{Lemm: Equiv com and top type 2}, $T$ has the topological condition of type-$(2)$ in its $1$-realization. This implies the existence of a ribbon $r$ joining two curves $\alpha$ and $\alpha'$ that lie on different sides of the stable boundary of the Markov partition. Similarly, there is another ribbon $r'$ joining curves $\beta$ and $\beta'$, where $\beta$ is on the same stable side as $\alpha$ and $\beta'$ is on the same side as $\alpha'$. These curves correspond to the boundaries of vertical sub-rectangles in the realization. The type-$(2)$ condition imposes certain conditions about the orientation on these curves, namely:

\begin{itemize}
\item  The gluing orientation of $\alpha$ is equal to its induced orientation, and the gluing orientation of $\alpha'$ is equal to its respective induced orientation.

\item  The gluing orientation of $\beta$ coincides with its induced orientation if and only if the r gluing and the induced orientations of $\beta'$ coincides.
\end{itemize}

As usual, $\gamma$ is the curve between $\alpha$ and $\beta$ with respect to the induced orientation: $\alpha \leq \gamma \leq \beta$. Similarly, $\gamma'$ is the curve between $\beta'$ and $\alpha'$ with respect to the induced orientation: $\beta' \leq \gamma' \leq \alpha'$. The same orientation properties hold for the curves $\tilde{\alpha}$, $\tilde{\alpha'}$, $\tilde{\beta}$, $\tilde{\beta'}$, $\tilde{\gamma}$, and $\tilde{\gamma'}$ in the Markov partition of $f$.

Furthermore, the dynamic orientation of $\tilde{\alpha}$ matches its induced orientation if and only if the dynamic orientation of $\tilde{\alpha'}$ matches the induced orientation. We adopt the convention that the induced orientation, gluing orientation, and dynamic orientation of $\tilde{\alpha}$ and $\tilde{\alpha'}$ are the same and fixed from now on. Even more, we assume that the dynamical orientation and gluing orientation of $\tilde{\beta}$ and $\tilde{\beta'}$ are the same, as they both point towards the same periodic point. We shall analyze two disjoint situations: either the dynamical orientation of $\tilde{\beta}$ is the same as that of $\tilde{\alpha}$, or it is different.

If the dynamical orientation of $\tilde{\alpha}$ and $\tilde{\beta}$ are the same, we define the curve $\tilde{L}$ as follows:
$$
\tilde{L}:=\tilde{\alpha'}\cdot \tilde{\gamma'}\cdot \tilde{\beta}\cdot \tilde{\gamma}.
$$

This curve $\tilde{L}$ is a simple closed curve because its entire interior is oriented with respect to the dynamical orientation in a stable leaf. However, this is not possible because the stable foliation of $f$ does not have closed leaves.

If the dynamical orientations of $\tilde{\alpha}$ and $\tilde{\beta}$ are different, this implies that $\tilde{\gamma}$ contains a periodic point $\tilde{O}$, and we have the following inequalities with respect to the induced orientation: 
$$
\tilde{\alpha} < \tilde{O} < \tilde{\beta}.
$$ 
Like the induced and dynamical orientations of $\tilde{\alpha}$ and $\tilde{\alpha'}$ coincide: $\tilde{\alpha'}<\tilde{O}$. Similarly, the dynamical orientation of $\tilde{\beta}$ matches that of $\tilde{\beta'}$ and they are the inverse of the induced orientation, so $\tilde{O}<\tilde{\beta'}$. This implies that, with respect to the induced orientation:
$$
\tilde{ O}<\tilde{\beta'}< \tilde{\alpha'}<\tilde{O}
$$
so the only way for this to occur is to have a closed leaf, which contradicts the fact that the stable foliation of $f$ does not have closed leaves.
This ends our proof by contradiction.
\end{proof}

\begin{coro}\label{Lemm: T pA class then no obstruction 2}
A geometric type $T$ in the pseudo-Anosov class does not have the type-$(2)$ obstruction.
\end{coro}

\begin{proof}

It is sufficient to prove that $T^m$ does not have the type-$(2)$ condition. However, $T^m$ is a geometric type in the pseudo-Anosov class associated with $f^m$. By Lemma \ref{Lemm: T pA class then no condition 2}, we know that $T^m$ does not have the type-$(2)$ condition. Therefore, it follows that $T$ does not have the type-$(2)$ obstruction.
\end{proof}


\begin{lemm}\label{Lemm: T pA class then no condition 3}
A geometric type $T$ in the pseudo-Anosov class does not have the combinatorial condition of type $(3)$.
\end{lemm}

\begin{figure}[h]
	\centering
	\includegraphics[width=0.7\textwidth]{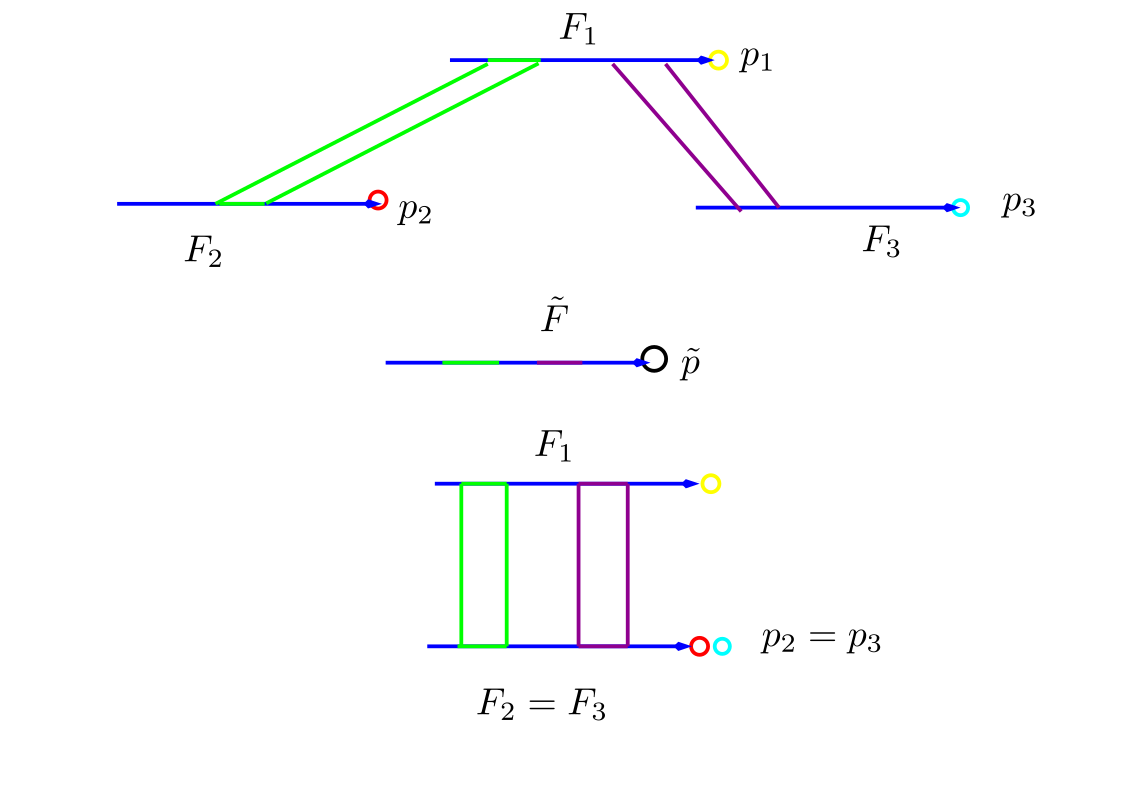}
	\caption{ Type $3$ obstruction}
	\label{Fig: Type 3 obstruction}
\end{figure}

\begin{proof}

Suppose $T$ has the combinatorial condition of type $(3)$. Then, according to Lemma \ref{Lemm: Equiv com and top type 3}, the first realization of $T$ has the topological obstruction of type $(3)$. Let $\alpha$ and $\alpha'$ be the curves identified by the ribbon $r$, and let $r'$ be the other ribbon given by the combinatorial condition that connects $\beta$ and $\beta'$. Suppose that $\alpha$ and $\beta$ are contained in the separatrix  $F^s_1(p_0)$, while $\beta'$ is contained in $F^s_1(p_1)$ and $\alpha'$ is contained in $F^s_1(p_2)$, where $p_0$, $p_1$, and $p_2$ are periodic points. By hypothesis  $F^s_1(p_0)$, $F^s_1(p_1)$, and $F^s_1(p_2)$ are all distinct and the curves are all embrionary . Now, let us analyze what happens in the Markov partition.

By the identifications, the curves $\tilde{\alpha}$, $\tilde{\alpha'}$, $\tilde{\beta}$, and $\tilde{\beta'}$ are all contained in the same stable separatrix  $\tilde{F}$ associated with a periodic point $\tilde{O}$ and they don't contain this point (remember they comes from embrionary separatrices).

Now, the separatrix  $\tilde{F}$ has only two sides. On one side, we have the rectangle $\tilde{R_k}$ which contains $\tilde{\alpha}$ and $\tilde{\beta}$ in their periodic boundary. On the other side, we have the rectangle $\tilde{R_a}$ which contains $\tilde{\alpha'}$, and $\tilde{R_b}$ which contains $\tilde{\beta'}$.

If $\tilde{R_a}$ is not the same as $\tilde{R_b}$, then only one of these rectangles can have a periodic stable boundary. However, our hypothesis states that both $\alpha'$ and $\beta'$ are in embrionary separatrices, which means that both $\tilde{R_a}$ and $\tilde{R_b}$ should have periodic boundaries containing these curves. This leads to a contradiction.

Therefore, the rectangle $R_a$ of the Markov partition contains, in one of its periodic boundaries, the curves $\alpha'$ and $\beta'$. However, by the coherence of the dynamical orientation with the induced orientations, $\alpha'$ and $\beta'$ should be in the same stable separatrix. This contradicts our assumption.

\end{proof}

\begin{coro}\label{Lemm: T pA class then no obstruction 3}
A geometric type $T$ in the pseudo-Anosov class doesn't have the type  $(3)$ combinatorial obstruction.
\end{coro}

\begin{proof}

Similarly to the other obstruction, for all $m\in \mathbb{N}$, $T^m$ is the geometric type of the pseudo-Anosov homeomorphism $f^m$. This implies that $T^m$ doesn't have the type $(3)$ combinatorial condition.
\end{proof}

\section{Finite genus and no-impasse implies basic piece without impasse.}\label{Sec: finite genus no impas implies basic piece}

In this part we are going to prove the next implication of Proposition \ref{Prop: pseudo-Anosov iff basic piece non-impace}. It correspond tho the following proposition.

\begin{prop}\label{Prop: Charaterization pseudo Anosov class}
A geometric type $T$  which satisfy the following properties:
	\begin{enumerate}
	\item The incidence matrix in transitive.
	\item The genus of $T$ is finite.
	\item $T$ don't have impasse.
	\end{enumerate}
is realized as a \emph{mixing} basic piece of a surface Smale diffeomorphism without topological impasse
\end{prop}

\begin{proof}
 The Theorem \ref{Theo: finite genus iff realizable}  due to Bonatti-Langevin-Beguin implies that $T$ is realized as a saturated basic set $K$ of a surface Smale diffeomorphism . Like the incidence matrix is mixing, we deduce $K$ is mixing and then it is basic piece (Look at \cite[Proposition 18.7.7]{hasselblatt2002handbook}). The Theorem \ref{Theo: Geometric and combinatoric are equivalent} implies that $K$ does not have a topological impasse.
\end{proof}

\chapter{Some computations}\label{Chap: Computations}

\section{Partitioning a Markov Partition Using a Periodic Code.}\label{Sec: Cut in a code}

Let $T$ be a geometric type in the pseudo-Anosov class with binary incidence matrix $A := A(T)$. Consider $f: S \rightarrow S$ as a generalized pseudo-Anosov homeomorphism, which has a geometric Markov partition $\cR$ of geometric type $T$. In this section, $\pi_f: \Sigma_A \rightarrow S$ is the projection from the Shift space to the surface ( Definition \ref{Defi: projection pi} ). Let $\underline{w} \in \Sigma_A$ be a periodic code of period $P \geq 1$, and let $p := \pi_f(\underline{w})$ be the corresponding periodic point of $f$ under this projection. In this section, we will construct a geometric Markov partition for $f$ that includes $p$ as an $s$-boundary point and calculate its geometric type in terms of $\underline{w}$ and $T$. Our main technical assumption is that the periodic code $\underline{w}$ is not $s$-boundary, at the end of our exposition we are going to remark why this is not a limitation for our future applications.

\subsection{The $s$-boundary refinement with respect to a periodic code.}

We will construct the $s$-\emph{boundary refinement} of the Markov partition $\cR$ by dividing it along the orbit of a stable interval $I$ containing $p=\pi_f(\underline{w})$. Let's begin by characterizing the orbit of this interval.

\begin{defi} \label{Def: Stable intervals of codes}
	
Let $T$ be a geometric type in the pseudo-Anosov class with a binary incidence matrix $A := A(T)$. Consider $\underline{w} \in \Sigma_A$ as a periodic code. Define $\underline{F}^s(\underline{w})$ as the stable leaf codes of $\underline{w}$ ( Definition \ref{Defi: s,u-leafs}). Let $t\in\NN$, the stable manifold codes for $\sigma^t(\underline{w})$, where the non-negative code matches that of $\sigma^t(\underline{w})$, are denoted by:
	
	\begin{equation}\label{Equa: Projection stable intervals}
	\underline{I}_{t, \underline{w}} := \{\underline{v}\in \underline{F}^s(\sigma^t(\underline{w})): v_{n}=w_{t+n} \text{ for all } n\in \NN\}.
	\end{equation}

\end{defi}

If $\pi_f^{-1}(\pi_f(\underline{w}))$ has more than one element, it's possible that there exist $t_1$ and $t_2$ (both not multiples of the period of $p$, to be denoted  $P$) such that $\sigma^{t_1}(\underline{w}) \neq \sigma^{t_2}(\underline{w})$ but their projections coincide. The following lemma clarifies this situation.

\begin{lemm}\label{Lemm: who the intervals intersect }
Let  $I_{t,\underline{w}} := \pi_f(\underline{I}_{t,\underline{w}})$,  then: 

\begin{itemize}
\item[i)] $I_{t,\underline{w}}$  is a unique stable interval of $R_{w_t}$ that contains $f^t(p)$.

\item[ii)] If $\underline{w}$ is a $s$ boundary code, then $I_{t,\underline{w}}$ is a stable boundary component of $R_{w_t}$.

\item[iii)] If $\underline{w}$ is not a $s$-boundary code and there exist $t_1$ and $t_2$ in $\NN$ such that $\sigma^{t_1}(\underline{w}) \neq \sigma^{t_2}(\underline{w})$ but their projections coincide, i.e. $f^t(p):=\pi_f(\sigma^{t_1}(\underline{w}))=\pi_f(\sigma^{t_2}(\underline{w}))$, then the stable intervals $I_{t_1, \underline{w}}$ and $I_{t_2, \underline{w}}$ have disjoint interior and only intersect at the point $f^t(p)$.

\end{itemize}
 
 Let us point that in item $iii)$, $t_1-t_2$ cannot be a multiple of the period of $P$. Because if this were the case,  $t_1=t_2+mP$ and therefore:
 
 $$
 \sigma^{t_1}(\underline{w})=\sigma^{t_2}(\sigma^{mP}(\underline{w}))=\sigma^{t_2}.
 $$

\end{lemm}

\begin{proof}
\textbf{Item} $i)$. According to Proposition \ref{Prop: Projection foliations}, the set $\underline{I}_{t, \underline{w}}$ is contained in the stable manifold of $\pi_f(\underline{w})$. If $\underline{w}$ is a $u$-boundary code, then $\underline{I}_{t, \underline{w}}$ is in a unique stable separatrice. In this case, $\underline{I}_{t, \underline{w}}$ is contained in this separatrice and cover a stable interval contained in $R_{w_t}$. If $p$ is an interior point, $\underline{I}_{t, \underline{w}}$ is contained in both separatrices of $f^t(p)$ and covers a whole stable interval of $R_{w_t}$. Just remains to prove that $I_{t, \underline{w}}$ is indeed a unique  stable interval.

If the positive code of $\sigma^t(\underline{v})$ coincides with $\sigma^t(\underline{v})$, then, for all $n\in \NN$, $v_{t+n}=w_{t+n}$ and $v_{t+n+1}=w_{t+n+1}$, and since the incidence matrix is binary, we can deduce that their projections are in the same horizontal sub-rectangle of $\cR$,  i.e.  $\pi_f(\sigma^{t+n}(\underline{w})),\pi_f(\sigma^{t+n}(\underline{v}))\in H^{w_{t+n}}_{j_{t+n}}$. Then:

\begin{eqnarray}
\pi_f(\sigma^t{\underline{w}}) \in \pi_f(\underline{I}_{t,\underline{w}}) \subset \bigcap_{n\in \NN} H^{w_{t+n}}_{j_{t+n}}=I \text{ and } \\
\pi_f(\sigma^t{\underline{v}})  \in \pi_f(\underline{I}_{t,\underline{w}}) \subset \in \bigcap_{n\in \NN} H^{v_{t+n}}_{j_{t+n}}=I.
\end{eqnarray}

Therefore, $I_{t,\underline{w}}$ is contained in $I$, and $I$ is a unique horizontal interval of $R_{w_t}$, this probes the first item.

\textbf{Item} $ii)$. The projection of a $s$-boundary code it is always a $s$-boundary point, therefore $I_{t,\underline{w}}$ is a stable interval of $R_{w_t}$ contained in the stable boundary of such rectangle, i.e. it is a stable boundary component of $R_{w_t}$.

\textbf{Item} $iii)$. Considering that $\underline{w}$ is not an $s$-boundary code and multiple codes project to the same point, $f^t(p)$ becomes a $u$-boundary point but not an $s$-boundary point. Consequently, only two distinct codes project to $f^t(p),$ specifically $\sigma^{t_1}(\underline{w})$ and $\sigma^{t_2}(\underline{w})$. Consequently, $\underline{F}^s(\sigma^{t_1}(\underline{w}))$ and $\underline{F}^s(\sigma^{t_2}(\underline{w}))$ are mapped to different stable separatrices of $f^t(p)$. This implies that $I_{t_1, \underline{w}}$ and $I_{t_2, \underline{w}}$ belong to separate separatrices and have non-overlapping interiors. Since $f$ has no closed leaves, their combination does not form a closed curve, only intersecting at $f^t(p)$.

\end{proof}

\begin{defi}\label{Defi: s interval of R-i induced by  t }
	We call to $I_{t,\underline{w}} := \pi_f(\underline{I}_{t,\underline{w}})$ the stable interval of $R_i$ induced by the iteration $t$ of $\underline{w}$.
\end{defi}

\begin{rema}\label{Rema; finite intervals}
If $\underline{w}$ has period $P$ then $I_{t,\underline{w}} =I_{t+P,\underline{w}}$. 
\end{rema}

\begin{defi}\label{Defi: Stable intervals inside R-i}
For every rectangle index $i \in \{1, \cdots, n\}$, the set of stable intervals of the rectangle $R_i \in \cR$, determined by the periodic code $\underline{w}$, is given by:
\begin{equation}
\cI(i,\underline{w}):=\{I_{t,\underline{w}}: w_t=i \text{ and } t\in \NN \}.
\end{equation}
	Let $\cI(\underline{w})=\cup_{i=1}^n \cI(i,\underline{w})$ set that contain to all of such intervals.
\end{defi}

By Remark \ref{Defi: s interval of R-i induced by  t }, $\cI(i,\underline{w})$ and $\cI(\underline{w})$ are finite sets.

\begin{conv*}
To avoid overwhelming notation, we use $\cup \cI(i,\underline{w})$ and $\cup \cI(\underline{w})$ to represent the union of all stable intervals contained within these finite sets.
\end{conv*}

The next lemma will be useful for establishing the $f$-invariance of the family $\cI(\underline{w})$.

\begin{lemm}\label{Lemm: image of stable intervals}
For every interval $I_{i,\underline{w}}\in \cI_(i,\underline{w})$, its image under $f$ satisfies that: 
$$
f(I_{t,\underline{w}})\subset I_{t+1,\underline{w}}.
$$
\end{lemm}

\begin{proof}
	The proof follows from the next computation:
	
	\begin{eqnarray*}
		f(\pi_f(\underline{I}_{t,\underline{w}}))\subset \pi_f(\sigma(\underline{I}_{t,\underline{w}}))=\\ 
		\pi_f(\underline{I}_{t+1,\underline{w}}) = I_{t+1,\underline{w}}.
	\end{eqnarray*}
	
\end{proof}

We are going to divide the rectangles of $\cR$ trough the stable intervals in the union of all $\cI(i,\underline{w})$.

\begin{lemm}\label{Lemm: rectangles in R S(p)}
Let $\tilde{R}$ be the closure of a connected component of $\overset{o}{R_i}\setminus \cI(\underline{w})$. Then $\tilde{R}$ is a horizontal sub-rectangle of $R_i$
\end{lemm}

\begin{proof}

If $i\neq j$, then $\overset{o}{R_i}\cap \cup \cI(j,\underline{x})=\emptyset$, because for any point $x\in R_i$ and every code $\underline{v}\in \pi^{-1}_f(x)$, $v_0=k$. In contrast, if $y\in \cup\cI(j,\underline{w})$, at least one code $\underline{w}'=\sigma^t(\underline{w})$ in $\pi^{-1}(x)$ has its zero term equal to $j$, i.e., $w'_0=j$. Therefore, $x\neq y$.

We only need to consider the case when $I_{t,\underline{w}}\subset R_i$. In this situation, $I_{t,\underline{w}}$ is a horizontal interval of $R_i$. Consequently, each connected component of $\overset{o}{R_i}\setminus \cup \cI(i,\underline{w})$ is a horizontal sub-rectangle $H$ of $R_i$ minus its stable boundaries, and its closure coincides with $H$. This argument proves our lemma.
\end{proof}

These will be the rectangles in the Markov partition that we are looking for.

\begin{defi}\label{Lemm: cR-S(p) is markov part}
Let 

\begin{equation}
\cR_{S(\underline{w})}=\{\tilde{R}_r\}_{r=1}^N
\end{equation}

be the family of rectangles that were described in Lemma \ref{Lemm: rectangles in R S(p)}.
\end{defi}

\begin{lemm}\label{Lemm: Markov partition }
The family of rectangles $\cR_{S(\underline{w})}$  is a Markov partition of $f: S\rightarrow S$.
\end{lemm}

\begin{proof}
Every rectangle in $\cR_{S(\underline{w})}$ is a horizontal sub-rectangle of a rectangle $R_i$ in $\cR$ and they don't overlap. Therefore, any pair of rectangles in $\cR_{S(\underline{w})}$ has disjoint interiors, and their union is the whole surface $S$.

The unstable boundary of $\cR_{S(\underline{w})}$ is identical to the unstable boundary of $\cR$ and is thus $f^{-1}$-invariant. To complete our analysis, we need to confirm that the stable boundary is also $f$-invariant.

The stable boundary of $\cR_{S(\underline{w})}$ consists of two kinds of elements: stable intervals from $\cR$ that are part of the stable boundary of $\cR$, and stable segments of the form $I_{t,\underline{w}}$. As $\cR$ itself is a Markov partition of $f$, the first type of intervals have an image within the stable boundary of $\cR$, which is included in $\partial^s \cR_{S(\underline{w})}$.

Lemma \ref{Lemm: image of stable intervals} implies that the image of any interval $I_{t,\underline{w}}$ under $f$ is contained within $I_{t+1,\underline{w}}$, which is contained in the stable boundary of $\cR_{S(\underline{w})}$. This establishes that the stable boundary of $\cR_{S(\underline{w})}$ is $f$-invariant. Then, $\cR_{S(\underline{w})}$ is a Markov partition for $f$.

\end{proof}

\subsection{The geometrization of $\cR_{S(\underline{w})}$} Establishing a geometrization for the Markov partition $\cR_{S(\underline{w})}$, requires to assign labels and orientations to the rectangles in $\cR_{S(\underline{w})}$.

\begin{defi}\label{Defi: Orientation in R-S(w)}
If the rectangle $\tilde{R} \in \cR_{S(\underline{w})}$ is contained within the rectangle $R_i \in \cR$, we assign to the rectangle $\tilde{R}$ the vertical and horizontal orientations that are consistent with the respective vertical and horizontal orientations of the leaves in $R_i$.
\end{defi}
		
\begin{defi}\label{Defi: Label of R-S(w)}
We label the rectangles in $\mathcal{R}_{S(\underline{w})}$ using the \emph{lexicographic order}. We start by numbering the horizontal sub-rectangles of $R_i$ that belong to $\mathcal{R}_{S(\underline{w})}$ from the bottom to the top with respect the vertical orientation of $R_i$ as: $\{\tilde{R}_{(i,s)}\}_{s=1}^{N_i}$ , where $N_i$ is a positive integer. In this manner, the rectangle $\tilde{R}_{(i,s)} \in \cR_{S(\underline{w})}$ is contained within $R_i$ at vertical position $s$. Then we fix the label: $\tilde{R}_{(i,s)}:=\tilde{R}_{r}$, where the number $r$ is determined by the formula:
$$
r=\tilde{r}(i,s):=\sum_{i'<i}^{i} N_{i'} + s.
$$
\end{defi}

\begin{defi}\label{Defi: Geometric s-boundary refinament}
	The  Markov partition $\cR_{S(\underline{w})}$ endowed with the geometrization given by Definitions \ref{Defi: Orientation in R-S(w)} and \ref{Defi: Label of R-S(w)} is called the $s$-\emph{boundary refinement of $\cR$ with respect to the code $\underline{w}$} and its geometric type is denoted by:

\begin{equation}\label{Equa: the s boundary type}
T_{S(\underline{w})}=\{N,\{H_r,V_r\}_{r=1}^N,\Phi_{S(\underline{w})}:=(\rho_{S(\underline{w})},\epsilon_{S(\underline{w})})\}.
\end{equation}
And we recall the notation to such the rectangles in the partition:
\begin{equation}
\cR_{S(\underline{w})}:=\{\tilde{R}_r\}_{r=1}^N
\end{equation}

\end{defi}

\begin{rema}\label{Rema: Why not s boundary code?}
If the periodic code $\underline{w}$ is an $s$-boundary code, Item $ii)$ of Lemma \ref{Lemm: who the intervals intersect } tells that  every interval in the set $\cI_(i,\underline{w})$ is contained within the stable boundary of $R_i$. Consequently, the set of rectangles in $\cR_{S(\underline{w})}$ coincides with $\cR$. The geometrization of $\cR_{S(\underline{w})}$ is the same as that of $\cR$, and the geometric type $T_{S(\underline{w})}$ is identical to $T$.

 This observation justifies considering $\underline{w}$ as a periodic code that is not $s$-boundary because, in this case, cutting the Markov partition $\cR$ along $\cI_(\underline{w})$ has no effect on $\cR$ and will not change the geometric type $T$. Therefore from now on we are should assume that $\underline{w}$ is not a $s$-boundary code.
\end{rema}

\begin{figure}[h]
	\centering
	\includegraphics[width=0.7\textwidth]{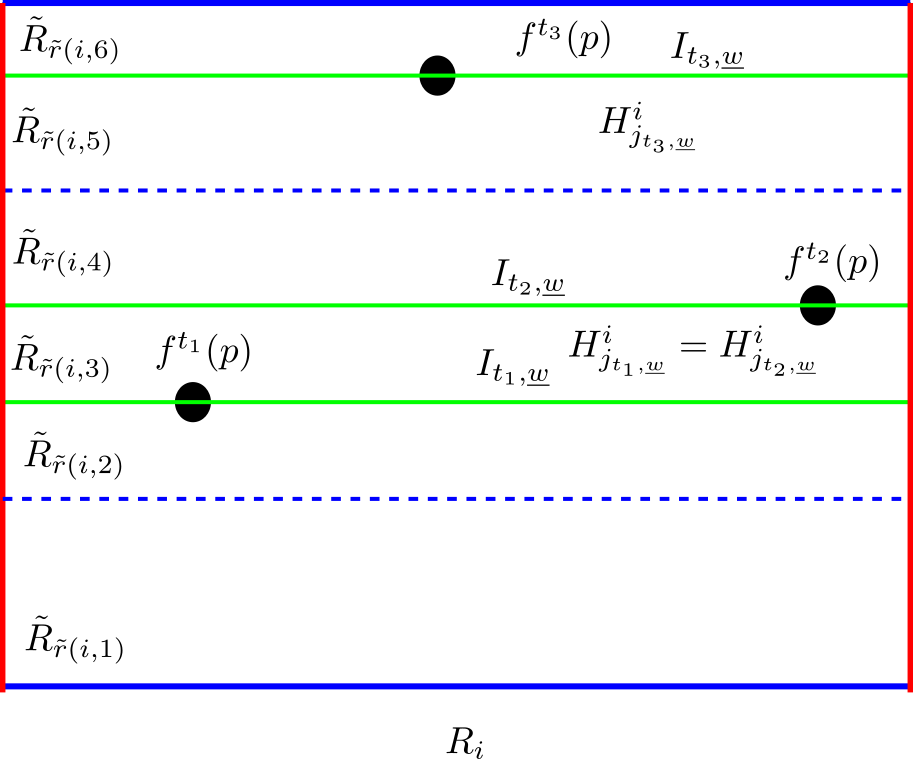}
	\caption{The set $\cI(i,\underline{w})$ and the rectangles $H^i_{ j_{t,\underline{w}} }$ }
	\label{Fig: Inva boundary}
\end{figure}

We will begin by determining the parameters in the geometric type $T_{S(\underline{w})}$ in the following order:

\begin{enumerate}
\item Determine the number $N$ of rectangles in $\cR_{S(\underline{w})}$.
\item For each rectangle $\tilde{R}_r$, find the pair $(i,s)$ such that $r=\tilde{r}(i,s)$.
\item Then, determine the number of vertical sub-rectangles in $\tilde{R}_r$.
\item Next, we identify the such  intervals in $\cI(\underline{w})$ that are the stable boundary of $\tilde{R}_r$.
\item Afterward, calculate the number of horizontal sub-rectangles in $\tilde{R}_r$ and determine the permutation $\rho_{S(\underline{w})}$.
\item Finally, establish the function $\epsilon_{S(\underline{w})}$.
\end{enumerate}

\subsection{The number of rectangles in $\cR_{S(\underline{w})}$}

The number of rectangles in $\cR_{S(\underline{w})}$ is equal to the sum over all integers $i\in\{1,\cdots,n\}$ of the total number of sub-rectangles in $R_i$ obtained by cutting it along $\cI(i,\underline{w})$.  In fact, since the number  of horizontal intervals in $\cI(\underline{w})$ is equal to the period of $\underline{w}$:

$$
N=n + \text{Per}(\underline{w}).
$$

Anyway for our future arguments we are going to develop another approach.

\begin{defi}\label{Defi: cO(i) and O(i)}
Let $\underline{w}\in \Sigma_A$ be a non-$s$-boundary periodic code with a period $P$. Then, for every $i \in \{1, \ldots, P-1\}$, we take:
\begin{equation}\label{Equa: cO(i,w)}
\cO(i,\underline{w})=\{(t,\underline{w}): t \in \{0, \ldots, P-1\} \text{ and } w_t=i \}.
\end{equation}
The cardinality of this set is denoted by $O(i,\underline{w})$.

\end{defi}

The next lemma is nothing  but a direct observation.

\begin{lemm}\label{Lemm: Number  of rectangles in a rec of the s refinaent}
For all $i \in \{1, \ldots, n\}$, the number of rectangles in the Markov partition $\cR_{S(\underline{w})}$ that are contained in $R_i$ is equal to $O(i,\underline{w})  + 1$.
\end{lemm}

We have an immediate corollary:

\begin{coro}\label{Coro: Number N in the type T-S(w)}
The number $N$ in the geometric type $T_{S(\underline{w})}$ that is equal to the number of rectangles in the $s$-boundary refinement, it is determined by the following formula:
\begin{equation}\label{Equa: Number N }
N=\sum_{i=1}^{n} O(i,\underline{w})+1.
\end{equation}
\end{coro}

In the following sub-sections, we are going to consider a rectangle $R_r$ and we are going to assume that we know the terms $(i, s)$ such that $r=\tilde{r}(i,s)$. The next lemma indicates how to determine such parameters and justifies our future hypotheses regarding them.

\begin{lemm}\label{lemm: determinating r=(i,s)}
Let $\tilde{R}_r$ a rectangle in $\cR_{S(\underline{w})}$ then we can determine two unique numbers $(i_r,s_r)$ such that $r=\tilde{r}(i_r,s_r)$ in following manner:
\begin{itemize}
	\item[i)] There exist unique $i_r\in \{1,\cdots,n\}$ such that:
	$$
	\sum_{i<i_r} O(i,\underline{w}) +1 < r \leq \sum_{i\leq i_r} O(i,\underline{w}) +1 
	$$
	\item[ii)] Take $s_r:= r - \sum_{i<i_r} O(i,\underline{w}) +1$.
\end{itemize}
\end{lemm}

\begin{proof}
The only point that we need to argue is the first item, as the second one is totally determined by such a number $i_r$. But since $r \in \{1, \cdots, \sum_{i=1}^n O(i, \underline{w}) + 1\}$, in fact, there exists a maximum $i_0\in\{1,\cdots, n\}$ such that $\sum_{i<i_0} O(i, \underline{w}) + 1 < r$. Clearly $i_r = i + 1$ satisfies the second equality because if this is not the case:  $\sum_{i<i_0+1} O(i, \underline{w}) + 1 < r$ contradicting that $i_0$ is the maximum with this property.
\end{proof}

\subsection{The vertical sub-rectangles in $T_{S(\underline{w})}$.}

Now, we are going to determine the number of vertical sub-rectangles in every rectangle $\tilde{R}_r\in \cR{S(\underline{w})}$, but first, we need to characterize them.

\begin{defi}\label{Defi: Vertical sub R-S(w)}
Let $\tilde{R}_r$ be a rectangle in the geometric Markov partition $\cR_{S(\underline{w})}$, and suppose that $r = \tilde{r}(i, s)$ for some $s$. Let $V^i_l$ be a vertical sub-rectangle of the Markov partition $(f, \cR)$ contained in $R_i$, define:
$$
\tilde{V}^{r}_l:= \overline{\overset{o}{\tilde{R}_r} \cap \overset{o}{V^i_l}},
$$
as the closure of the intersection between the interior of  $\tilde{R}_r$ and the interior of  $V^i_l$.

\end{defi}

\begin{lemm}\label{Lemm: Caracterization of vertica sub of R-S(w)}
The set $\tilde{V}^{r}_l$ is a vertical sub-rectangle of $\tilde{R}_r$
\end{lemm}

\begin{proof}
As the intersection of the interiors of a horizontal and a vertical sub-rectangle of $R_i$ has a unique connected component, in particular, $\overset{o}{V^i_l} \cap \overset{o}{\tilde{R}_r}$ has a unique connected component. Clearly, $V^i_l$ crosses $\tilde{R}_r$ from the bottom to the top, and therefore, $\tilde{V}^{r}_l$ is a vertical sub-rectangle of $\tilde{R}_r$.
\end{proof}

\begin{rema}\label{Rema: Coherence orientatios of R_i and R_r}
The horizontal and vertical orientations of $\tilde{R}_r$ coincide with those of $R_i$. Therefore, the labeling $\{\tilde{V}^{r}_l\}_{l=1}^{v_i}$ of the vertical sub-rectangles, as given in Definition \ref{Defi: Vertical sub R-S(w)}, is coherent with the horizontal direction of $\tilde{R}_r$.
\end{rema}

\begin{figure}[h]
	\centering
	\includegraphics[width=0.7\textwidth]{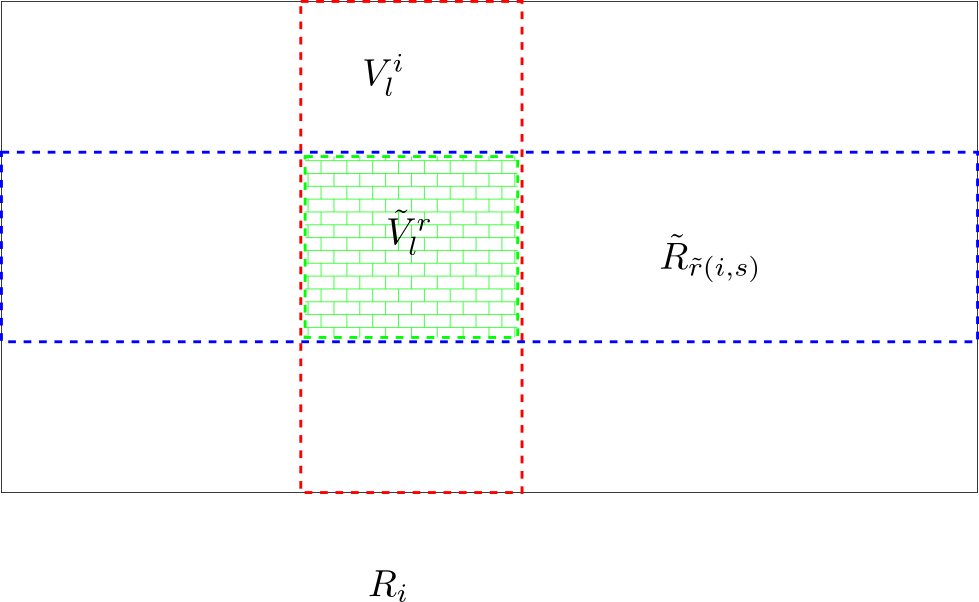}
	\caption{The vertical sub-rectangle $\tilde{V}^r_l$ of $\tilde{R}_{r(i,s)}$. }
	\label{Fig: Vertical sub}
\end{figure}

Before we continue, let us prove the next technical lemma.

\begin{lemm}\label{Lemm: unique sub (i,j) for a r}
Let $\tilde{R}_r\in \cR_{S(\underline{w})}$ with $r=\tilde{r}(i,s)$ for some $s$. Suppose $\tilde{H} \subset \tilde{R}_r$ is a horizontal sub-rectangle of $(f,\cR_{S(\underline{w})})$. Since $i$ was already determined before, there exists a unique pair $j\in\{1,\cdots,h_i\}$ such that $(i,j)\in \cH(T)$ and  $\overset{o}{H} \subset \overset{o}{H^i_{j}}$, i.e., $H$ is contained in the interior of the horizontal sub-rectangle $H^i_j$ of  $(f,\cR)$.
\end{lemm}

\begin{proof}
We claim that $\tilde{H}$ intersects a unique horizontal sub-rectangle $H^i_{j}$ in its interior. If this is not the case, there are two adjacent sub-rectangles of $R_i$, $H^i_{j_1}$ and $H^i_{j_2}$, sharing a stable boundary $I$ that lies inside the interior of $H$, i.e., $\overset{o}{I}\subset \overset{o}{H}$. Then, $f(\overset{o}{H})$ will intersect the stable boundary of $\cR$ in the stable interval $f(\overset{o}{I})$. However, this interval is contained in the stable boundary of $\cR_{S(\underline{w})}$, leading to a contradiction. Therefore, $\overset{o}{H}$ is contained in the interior of a unique horizontal sub-rectangle $H^i_{j}$.
\end{proof}

\begin{defi}\label{Defi: the unique pair (i,j-H)}
Let $\tilde{R}_r\in \cR_{S(\underline{w})}$ with $r=\tilde{r}(i,s)$. Let $\tilde{H} \subset \tilde{R}_r$ be a horizontal sub-rectangle of $\cR_{S_(\underline{w})}$, we denote $(i,j_{\tilde{H} })\in \cH(T)$ as the only pair that was determine in Lemma \ref{Lemm: unique sub (i,j) for a r}.
\end{defi}

\begin{lemm}\label{Lemma: Vertical sub R-ps}
The rectangles in $\{\tilde{V}^{r}_l\}_{l=1}^{v_i}$, as described in Definition \ref{Defi: Vertical sub R-S(w)}, are the vertical sub-rectangles of the Markov partition $(f,\cR_{S_(\underline{w})})$ that are contained in $\tilde{R}_r$.
\end{lemm}

\begin{proof}

Assume that $f^{-1}(V^i_l)=H^k_j$. Since the vertical boundaries of $\tilde{V}^r_l$ are contained within the vertical boundaries of $V^i_l$, then  $f^{-1}(\tilde{V}^r_l)$ is a horizontal sub-rectangle of $H^k_{j}$ whose unstable boundary is contained in $\partial^u R_k$. Furthermore, $f^{-1}(\tilde{V}^r_l)$ does not intersect the orbit of $p$ in its interior. Hence, there exists a unique $\tilde{R}_{r'}\in \cR_{S(\underline{w})}$ such that $f^{-1}(\tilde{V}^r_l)\subset \tilde{R}_{r'}$ and it is a horizontal sub-rectangle of $R_k$ whose stable boundary  ( under the action of $f$) is send to the stable boundary of the Markov partition $\cR_{S(\underline{w})}$. This implies that $f^{-1}(\tilde{V}^r_l)$ is a horizontal sub-rectangle of $(f,\cR_{S(\underline{w})})$, and then $\tilde{V}^r_l$ is a vertical sub-rectangle of $(f,\cR_{S(\underline{w})})$.

	Let $\tilde{V}\subset \tilde{R}_r $ be some vertical sub-rectangle of $(f,\cR_{S(\underline{w})})$. We know that $f^{-1}(\tilde{V})$ is a horizontal sub-rectangle of a certain rectangle $\tilde{R}_{r'}\subset R_k\in \cR$. Since $p:=\pi_f(\underline{w})$ is not $s$-boundary, the interior of $f^{-1}(\tilde{V})$ does not intersect the stable boundary of any horizontal sub-rectangles of $R_k$. Therefore, $f^{-1}(\tilde{V})$ is contained in a unique horizontal sub-rectangle $H^k_j$  of the Markov partition $(f,\cR)$, where $j:=j_{f^{-1}(V)}$ is given by  Lemma \ref{Lemm: unique sub (i,j) for a r}.
	
	 Consequently, the unstable boundary of $\tilde{V}$ is contained within the unstable boundary of $f(H^k_j)=V^k_l$, and the stable boundaries of $\tilde{V}$ are within the stable boundary of $\tilde{R}_r$. As a consequence:	
	$$
	\overset{o}{\tilde{V}}=\overset{o}{R_r}\cap \overset{o}{V^i_l}.
	$$
	and $\tilde{V}$ is one of the rectangles described in Definition \ref{Defi: Vertical sub R-S(w)} .
\end{proof}

We have a immediate corollaries.

\begin{coro}\label{Coro: number of vertical sub rec}
	If $r=\tilde{r}(i,s)$  then $V_r=v_i$.
\end{coro}

\begin{prop}\label{Prop: vertical possition is delimitate.}
Let $\tilde{H}\subset \tilde{R}_r$ be a horizontal sub-rectangle of $(f,\cR_{S(\underline{w})})$, assume that $\tilde{H}\subset H^i_j$ and $f(H^i_j)=V^k_l$. If $f(H)=\tilde{V}^{r'}_{l'}$, then $l=l'$.
\end{prop}

\begin{proof}
	As $\tilde{H}\subset H^i_j$, then $f(\tilde{H})\subset f( H^i_j)=V^k_l$. Simultaneously, $f(H)=\tilde{V}^{r'}_{l'}=\overline{ \overset{o}{\tilde{R}_{r'}}\cap \overset{o}{V^k_{l}} }$. Given that the intersection $\overset{o}{\tilde{R}_{r'}}\cap \overset{o}{V^k_{l}}$ has only one connected component, it is necessary that $\overset{o}{V^k_{l}}=\overset{o}{V^k_{l'}}$ and thus $l=l'$.
\end{proof}

\subsection{The stable boundary of a rectangle $\tilde{R}_r$}

The first task will be to determine the stable boundaries of a rectangle $\tilde{R}_{r} \in \cR_{S(\underline{w})}$. In view of Lemma \ref{lemm: determinating r=(i,s)}, we can add to our hypothesis the equality $r=\tilde{r}(i,s)$, as we can always determine the parameter $(i,s)$.

In the rest of this subsection we are going to assume that: 
\begin{itemize}
\item The geometric type $T=\{n,\{h_i,v_i\}_{i=1}^n, \Phi_T=(\rho_t,\epsilon_T)\}$  is in the pseudo-Anosov class
\item Its incidence matrix $A:=A(T)$ is binary.
\item The  periodic code $\underline{w}\in \Sigma_A$ have period $P\geq 1$ and its is not $s$-boundary.
\end{itemize}

The following lemma establishes that for any given terms $w_t$ and $w_{t+1}$ in $\underline{w}$, there exists a unique horizontal sub-rectangle $H^{w_t}_{j_t}$ whose image corresponds to a vertical sub-rectangle $V^{w_{t+1}}_{l_{t+1}}$ contained within $R_{w_{t+1}}$. We have expressed this in terms of $\rho_t$ to make explicit  that this condition is entirely determined by the geometric type $T$ and $\underline{w}$.

\begin{lemm}\label{Lemm: unique hriwontal for every prjection} 
Let $t\in \{1,\cdots,P\}$. There is a unique index $j_{t,\underline{w}}\in \{1,\cdots, w_{w_t}\}$ that satisfies the following condition: $(w_t,j_{t,\underline{w}})\in \cH(T)$ is the sole pair of indices for which $\rho_T(w_t,j_{t,\underline{w}})=(w_{t+1},l_{t+1})\in \cV(T)$ (for certain $l_{t+1}$). Furthermore, $l_{t+1}\in \{1,\cdots,v_{w_{t+1}}\}$ is also uniquely determined.
\end{lemm}

\begin{proof}
Let's suppose that $w_t=i$. As the incidence matrix of $T$ is binary, there exists at most one horizontal sub-rectangle $H^i_j$ of $R_i$ such that $f(H^i_j)\subset R_{w_{t+1}}$. However, since $\underline{w}$ is an admissible code, there exists at least one horizontal sub-rectangle $H^i_j$ of $R_i$ such that $f(H^i_j)\subset R_{w_{t+1}}$. Therefore, there exists only one rectangle of $R_i$ such that $f(H^i_j)\subset R_{w_{t+1}}$, and we declare $j_{t,\underline{w}}$ as this unique number. Clearly, if $\rho_T(i,j_{t,\underline{w}})=(w_{t+1},l_{t+1})$, then $l_{t+1}\in \{1,\cdots,h_{w_{t+1}}\}$ is uniquely determined.
\end{proof}

\begin{coro}\label{Coro: projetion stable segments}
Assume that $w_t=i$, then the stable interval $I_{t,\underline{w}}$ is contained in $H^i_{j_{t,\underline{w}}}\setminus \partial^s H^i_{j_{t,\underline{w}}}$.
\end{coro}

\begin{proof}

It is evident that $\pi_f(\underline{w})\in H^i_{j_{t,\underline{w}}}$. Now, consider another code $\underline{v}\in \underline{F}^s(\sigma^t(\underline{w}))$ that projects to $I_{t,\underline{w}}$. As in Lemma \ref{Lemm: unique hriwontal for every prjection}, there exists only one index $j_{t, \underline{v}}\in \{1,\ldots, h_{v_t}\}$ such that $\rho_T(i, j_{t, \underline{v}})=(v_{t+1},l')$, and therefore $\pi_f(\underline{v})\in H^i_{j_{t, \underline{v}}}$. However, due to the matching of non-negative codes between $\underline{w}$ and $\underline{v}$, we have $w_{t}=v_{t}$ and $w_{t+1}=v_{t+1}$, leading to $j_{t, \underline{v}}=j_{t, \underline{w}}$. Consequently, $\pi_f(\underline{v})\in H^i_{j_{t,\underline{w}}}$.

Since $\underline{w}$ is not an $s$-boundary point, the projection of its stable leaf codes cannot intersect the stable leaves of  $s$-boundary periodic points. In particular, it does not intersect the stable boundary of the horizontal sub-rectangles of $(f,\cR)$. Hence, we deduce that $I_{t,\underline{w}} \subset H^i_{j_{t,\underline{w}}}\setminus \partial^s H^i_{j_{t,\underline{w}}}$. This confirms the claim.

\end{proof}

\begin{defi}\label{Defi: vertical order of intervals}
Consider two horizontal segments, $I$ and $I'$, within the rectangle $R_i$. We say that $I$ is positioned below $I'$  in relation to the vertical orientation of $R_i$ if the intersection of the interior of segment $\overset{o}{I}$ with a vertical segment $J$ within $R_i$ is lower than the intersection of $\overset{o}{I'}$ with $J$. In this case we write $I<I'$.
\end{defi}

\begin{rema}\label{Rema: well defined order intervals}
The trivial bi-foliated structure of $\overset{o}{R_i}$ ensure that this ordering is independent of the specific vertical segment $J$.
\end{rema}

\begin{lemm}\label{Lemm: first order in O(i)}
Let $I_{t_1,\underline{w}}$ and $I_{t_2,\underline{w}}$ be two stable segments contained within $R_i$, i.e., $w_{t_1}=w_{t_2}=i$. Suppose that $j_{t_1,\underline{w}}<j_{t_2,\underline{w}}$. Then, with respect to the vertical orientation of $R_i$:
$$
I_{t_1,\underline{w}}<I_{t_2,\underline{w}}.
$$
\end{lemm}

\begin{proof}

As was observed in Corollary \ref{Coro: projetion stable segments}, $I_{t_1,\underline{w}}\subset H^i_{j_{t_1,\underline{w}}}$ and $I_{t_2,\underline{w}}\subset H^i_{j_{t_2,\underline{w}}}$. With respect to the vertical orientation of $R_i$ any stable interval contained in $H^i_{j_{t_1,\underline{w}}}$ lies below any stable interval contained in $H^i_{j_{t_2,\underline{w}}}$. In particular, we have $I_{t_1,\underline{w}}<I_{t_2,\underline{w}}$, as we claimed.
\end{proof}

\begin{lemm}\label{Lemm: diferentation moment of the order}
 If there exist two distinct numbers $t_1$ and $t_2$ in $\{0,\cdots, P-1\}$ such that $\sigma^{t_1}(\underline{w}) \neq \sigma^{t_2}(\underline{w})$ but both have $w_{t_1} = w_{t_2} = i$, then we consider the two elements in $\cO(i,\underline{w})$ determined by these iterations: $(t_1,\underline{w})$ and $(t_2,\underline{w})$. In this case:

\begin{itemize}
\item[i)]   There exists an natural number $M$ between $1$ and $P-1$ such that $w_{t_1+M} \neq w_{t_2+M}$, but for all $m$ from $0$ to $M-1$, $w_{t_1+m}=w_{t_2+m}$.

\item[ii)]  If $M=1$, then $j_{t_1,\underline{w}}\neq j_{t_2,\underline{w}}$. However, if $M\geq 2$, for all $m\in \{0,\cdots,M-2\}$, the indices $j_{t_1+m ,\underline{w}}$ and $j_{t_2+m,\underline{w}}$ (defined in \ref{Defi: the unique pair (i,j-H)}) are equal. Thus, for all $m$ in the range $0$ to $M-2$:
$$
\epsilon_T(w_{t_1+m},j_{t_1 +m, \underline{w}})=\epsilon_T(w_{t_2+m},j_{t_2 +m, \underline{w}}).
$$

\item[iii)]  The indices $j_{t_1+M-1,\underline{w}}, j_{t_2+M-1,\underline{w}}\in \{1,\cdots, h_{w_{t_1 +M-1}}=h_{w_{t_2 +M-1}}\}$ are different.
\end{itemize}

\end{lemm}

\begin{proof}
\textbf{Item} $i)$. Since $\sigma^{t_1}(\underline{w})\neq \sigma^{t_2}(\underline{w})$, and $\underline{w}$ has a period of $P$, there must exist an integer $M \in \{0, \cdots, P-1\}$ such that $w_{t_1+M} \neq w_{t_2+M}$, and $M$ is chosen as the smallest integer with this property. We claim that $M\neq 0$ because $w_{t_1+0}=w_{t_2+0}$. Since $M$ is the minimum integer where $w_{t_1+M} \neq w_{t_2+M}$, then for all $0\leq m<M$, it holds that $w_{t_1+m}= w_{t_2+m}$.

\textbf{Item} $ii)$.If $M=1$, then $w_{t_1+1} \neq w_{t_2+1}$. Since the incidence matrix $A$ is binary, this implies that $j_{t_1,\underline{w}} \neq j_{t_2,\underline{w}}$.

Now, let's consider the case when $M>1$. For each $m \in \{0, \cdots, M-2\}$, the indices $j_{t_1+m,\underline{w}}$ and $j_{t_2+m,\underline{w}}$ are the only ones that satisfy $\rho_T(w_{t_1+m},j_{t_1+m,\underline{w}})=(w_{t_1+m+1},l_m)$ and $\rho_T(w_{t_2+m},j_{t_2+m,\underline{w}})=(w_{t_2+m+1},l'_m)$. However, since $w_{t_1+m+1}=w_{t_2+m+1}$, it follows that $j_{t_1+m,\underline{w}}=j_{t_2+m,\underline{w}}$. Consequently, we have:
 $$
\epsilon_T(w_{t_1+m},j_{t_1 +m, \underline{w}})=\epsilon_T(w_{t_2+m},j_{t_2 +m, \underline{w}}).
$$

\textbf{Item} $iii)$. Remember that $\rho_T(w_{t_1 +M-1},j_{t_1+M-1,\underline{w}})=(w_{t_1+M},l)$ and $\rho_T(w_{t_2 +M-1},j_{t_2+M-1,\underline{w}})=(w_{t_2+M},l')$. Since $w_{t_1+M}\neq w_{t_2+M}$ and the incidence matrix of $T$ is binary, it is necessarily the case that $j_{t_1+M-1,\underline{w}}\neq j_{t_2+M-1,\underline{w}}$.
\end{proof}

\begin{defi}\label{Defi: Intetchage order funtion}
Let $T$ be a geometric type in the pseudo-Anosov class with a binary incidence matrix $A$. Suppose $\underline{w}\in \Sigma_A$ is a periodic code of period $P$ that is not an $s$-boundary code. Assume that $\sigma^{t_1}(\underline{w}) \neq \sigma^{t_2}(\underline{w})$ but $w_{t_1} = w_{t_2}=i$. Let $M$ be the number determined in Lemma \ref{Lemm: diferentation moment of the order}. We define the \emph{interchange order} between $(t_1, \underline{w})$ and $(t_2, \underline{w})$ as follows:

\begin{itemize}
\item If $M=1$ then $\delta((t_1,\underline{w}), (t_2,\underline{w})):=1$.
\item If $M>1$ then
$$
\delta((t_1,\underline{w}), (t_2,\underline{w}))=\prod_{m=0}^{M-2}\epsilon_T(w_{t_1+m},j_{t_1 +m, \underline{w}})=\prod_{m=0}^{M-2}\epsilon_T(w_{t_2+m},j_{t_2 +m, \underline{w}}).
$$
\end{itemize}
\end{defi}

\begin{lemm}\label{Lemm: Comparacion de ordenes}
Let $T$ be a geometric type in the pseudo-Anosov class with a binary incidence matrix $A$. Suppose $\underline{w}\in \Sigma_A$ is a periodic code of period $P\geq 1$ that is not an $s$-boundary code. Assume that $\sigma^{t_1}(\underline{w}) \neq \sigma^{t_2}(\underline{w})$ but $w_{t_1} = w_{t_2}=i$. Consider a realization $(f,\cR)$ of the geometric type $T$ by a pseudo-Anosov homeomorphism $f:S\to S$. Let $I_{t_1,\underline{w}}$ and $I_{t_2,\underline{w}}$ be the two stable segments determined by the numbers $t_1$ and $t_2$. Then, $I_{t_1,\underline{w}} < I_{t_2,\underline{w}}$ with respect to the vertical orientation of $R_i\in \cR$ if and only if one of the following situations occurs:
 \begin{itemize}
 	\item $j_{t_1+M-1}<j_{t_2+M-1}$ and $\delta((t_1,\underline{w}), (t_2,\underline{w}))=1$, or
	\item $j_{t_1+M-1}>j_{t_2+M-1}$ and $\delta((t_1,\underline{w}), (t_2,\underline{w}))=-1$, 
 \end{itemize}

\end{lemm}

\begin{proof}
The case when $M=1$ reduces to Lemma \ref{Lemm: first order in O(i)}, so let's assume that $M>1$. Assume that $I_{t_1,\underline{w}} < I_{t_2,\underline{w}}$ and call $H$ the horizontal sub-rectangle determined by such stable intervals. For all $0\leq m \leq M-2$, $j_{t_1+m}=j_{t_2+m}$ and $f^{m}(H)$ is contained in a unique horizontal sub-rectangle of $R_{w_{t_1+m}}$. In this case, after applying $f^{M-1}$ to $H$, we have two possibilities:

\begin{itemize}
\item[i)]  $j_{t_1+M-1}<j_{t_2+M-1}$ and $f^{M-1}$ preserves the vertical orientation, which happens if and only if $\delta((t_1,\underline{w}), (t_2,\underline{w}))=1$, or
\item[ii)]  $j_{t_1+M-1}>j_{t_2+M-1}$ and $f^{M-1}$ changes the vertical orientations, which happens if and only if $\delta((t_1,\underline{w}), (t_2,\underline{w}))=-1$.
\end{itemize}

Now, assume that $j_{t_1+M-1}<j_{t_2+M-1}$ and $\delta((t_1,\underline{w}), (t_2,\underline{w}))=1$. Let  $H$ be  the horizontal sub-rectangle of $R_{w_{t_1}}$ determined by the intervals $I_{t_1,\underline{w}}$ and $I_{t_2,\underline{w}}$, we need to determine what of such intervals is the upper boundary of $H$.

The number $\delta((t_1,\underline{w}), (t_2,\underline{w}))$ measures the change in the relative positions of $f^{M-1}(I_{t_1,\underline{w}})$ and $f^{M-1}(I_{t_2,\underline{w}})$ inside the rectangle $R_{w_{t_1+M-1}}$. In this case, the relative position of $f^{M-1}(I_{t_1,\underline{w}})$ and $f^{M-1}(I_{t_2,\underline{w}})$ inside $R_{w_{t_1+M-1}}$ is the same as the relative position of $I_{t_1+M-1,\underline{w}}$ and $I_{t_2+M-1,\underline{w}}$ inside $R_i$. Since $\delta((t_1,\underline{w}), (t_2,\underline{w}))=1$, the upper boundary of $f^{M-1}(H)$ coincides with the image of the upper boundary of $H$ and similarly, the inferior boundary of $H$ corresponds to the inferior boundary of $f^{M-1}(H)$.

Furthermore, since $j_{t_1+M-1}<j_{t_2+M-1}$ and according to Lemma \ref{Lemm: first order in O(i)}, $I_{t_1+M-1,\underline{w}} < I_{t_2+M-1,\underline{w}}$. So, we can conclude that the inferior boundary of $f^{M-1}(H)$ is contained within $I_{t_1+M-1,\underline{w}}$ and then the inferior boundary of $H$ is $I_{t_1,\underline{w}}$, and its upper boundary is $I_{t_2,\underline{w}}$. In this manner, we have $I_{t_1,\underline{w}}<I_{t_2,\underline{w}}$ with respect to the vertical order in $R_i$.

Let's consider a scenario where $j_{t_1+M-1}>j_{t_2+M-1}$ and $\delta((t_1,\underline{w}), (t_2,\underline{w}))=-1$. In this case, we take $H$ as the horizontal sub-rectangle of $R_{w_{t_1}}$ determined by the intervals $I_{t_1,\underline{w}}$ and $I_{t_2,\underline{w}}$.
The negative value of $\delta((t_1,\underline{w}), (t_2,\underline{w}))$ indicates that $f^{M-1}\vert_H$ alters the relative positions of the stable boundaries of $H$ with respect to the vertical order in $R_i$. Consequently, the upper boundary of $f^{M-1}(H)$ is the image of the inferior boundary of $H$, and simultaneously, the lower boundary of $f^{M-1}(H)$ is the image  of the upper boundary of $H$.

Given that $j_{t_1+M-1}>j_{t_2+M-1}$, it follows that the lower boundary of $f^{M-1}(H)$ is enclosed within $H^{w_{t_1+M-1}}_{j_{t_1+M-1,\underline{w}}}$, and the upper boundary is encompassed within $H^{w_{t_2+M-1}}_{j_{t_2+M-1,\underline{w}}}$. Therefore by  Corollary \ref{Coro: projetion stable segments}  we have that $I_{t_1+M-1,\underline{w}} > I_{t_2+M-1,\underline{w}}$, and the upper boundary of $f^{M-1}(H)$ is confined within $I_{t_1+M-1,\underline{w}}$.
 Using  Lemma \ref{Lemm: image of stable intervals}, we can deduce that   $f^{M-1}(I_{t_1,\underline{w}})\subset I_{t_1+M-1,\underline{w}}$  and then that $I_{t_1,\underline{w}}$ is the inferior boundary of $H$. Finally we get that: $I_{t_1,\underline{w}}<I_{t_2,\underline{w}}$ with respect to the vertical order in $R_i$.
 
\end{proof}

We will now establish a formal ordering for the set of stable segments within $R_i$, which constitute the stable boundaries of the rectangles contained in $\cR_{S(\underline{w})}$ that also lie within $R_i$.

\begin{defi}\label{Lemm: Boundaries code}
Let $R_i\in\cR$ be any rectangle, denote $I_{i,-1}:=\partial^s_{-1} R_i$ as the lower boundary of $R_i$, and $I_{i,+1}:=\partial^s_{+1} R_i$ as the upper boundary of $R_i$. Define:
$$
\underline{s}:\cO(i,\underline{w})\cup \{(i,+1), (i,-1) \} \rightarrow \{0,1,\cdots,O(i,\underline{w}),O(i,\underline{w})+1\},
$$
as the unique function such that:
\begin{itemize}
\item $\underline{s}(t_1,\underline{w})< \underline{s}(t_1,\underline{w})$ if and only if  $I_{t_1,\underline{w}}< I_{t_1,\underline{w}}$ respect the vertical  order in $R_i$.
\item It assigns $0$ to   $(i,-1)$ and  $O(i,\underline{w})+1$ to  $(i,+1)$.
\end{itemize}

We call this function $\underline{s}$ \emph{the order induced} in $\cO(i,\underline{w})$ by the geometric type $T$.

\end{defi}

The boundaries of any rectangle $\tilde{R}_r$ of $\cR_{S(\underline{w})}$ contained in $R_i$ are defined by two consecutive horizontal intervals in $\cO(i,\underline{w})\cup \{I_{i,-1},I_{i,+1}\}$. The previous order enables us to describe them accurately based on the relative position of $\tilde{R}_r$ inside $R_i$.

\begin{lemm}\label{Lemm: Determine boundaris of R-r}

Let $\tilde{R}_r$ be a rectangle in the geometric Markov partition $\cR_{S(\underline{w})}$ with $r=\tilde{r}(i,s)$. Then, the stable boundary of $\tilde{R}_r$ consists of two stable segments of $R_i$ determined as follows:

\begin{itemize}
\item[i)] If $s=1$, the lower boundary of $\tilde{R}_r$ is $I_{i,-1}$. If  $O(i,\underline{w})=0$, then its upper boundary is $I_{i,+1}$. However, if $O(i,\underline{w})>0$, there exists a unique $(t,\underline{w})\in \cO(i,\underline{w})$ such that $\underline{s}(t,\underline{w})=1$, and the upper boundary of $\tilde{R}_r$ is the stable segment $I_{t,\underline{w}}$.

\item[ii)] If $O(i,\underline{w})>1$ and $s \neq 1, O(i,\underline{w})+1$, there are unique $(t_1,\underline{w}), (t_2,\underline{w})\in \cO(i,\underline{w})$ such that: $\underline{s}(t_1,\underline{w})=s-1$, $\underline{s}(t_2,\underline{w})=s$. Furthermore, the lower boundary of $\tilde{R}_r$ is the stable segment $I_{t_1,\underline{w}}$, and the upper boundary of $\tilde{R}_r$ is the stable segment $I_{t_2,\underline{w}}$.

\item[iii)] If $O(i,\underline{w})>0$ and $s=O(i,\underline{w})+1$, then the upper boundary of $\tilde{R}_r$ is the stable segment $I_{i,+1}$. Additionally, there exists a unique $(t,\underline{w})\in \cO(i,\underline{w})$ such that $\underline{s}(t,\underline{w})=O(i,\underline{w})$, and the inferior boundary of $\tilde{R}_r$ is the stable segment $I_{t,\underline{w}}$.
\end{itemize}

\end{lemm}

\begin{proof}
	Recall that $s$ is the vertical position of $\tilde{R}_r$ inside $R_i$ but $\tilde{s}(t,\underline{w})$ is the position of the horizontal interval $I_{t,\underline{w}}$ inside $R_i$.
	
\textbf{Item } $i)$. Since $s=1$, within the vertical ordering of $R_i$, $\tilde{R}_r$ represents the first rectangle from $\cR_{S(\underline{w})}$ contained within $R_i$. Therefore, its lower boundary coincides with the lower boundary of $R_i$, which is $I_{i,-1}$.

If $O(i,\underline{w})=0$, then $\tilde{R}_r=R_i$, and it's the only rectangle from $\cR_{S(\underline{w})}$ within $R_i$. Consequently, the upper boundary of $\tilde{R}_r$ must be $I_{i,+1}$.

However, if $O(i,\underline{w})>0$, the upper boundary of $\tilde{R})_r$ must correspond to a stable interval $I_{t,\underline{w}}\in \cI(i,\underline{w})$, which occupies the first position according to the order $\underline{s}$; hence, $\underline{s}(t,\underline{w})=1$.

\textbf{Item } $ii)$. Since $\tilde{R}_r$ occupies the vertical position $1 < s < O(i,\underline{w})+1$ among all the rectangles in $\cR_{S(\underline{w})}$ within $R_i$, its stable boundaries cannot coincide with $I_{i,-1}$ or $I_{i,+1}$. Therefore, its stable boundary is defined by two consecutive stable segments in $\cI(i,\underline{w})$, denoted as $I_{t_1,\underline{w}} < I_{t_2,\underline{w}}$. Furthermore, the interval $I_{t_1,\underline{w}}$ holds the position $s-1$ relative to the elements in $\cI(i,\underline{w})\cup {I_{i,+1}, I_{i,-1}}$, and $I_{t_2,\underline{w}}$ is positioned at $s$. This leads to the conclusion that $\underline{s}(t_1,\underline{w})=s-1$ and $\underline{s}(t_2,\underline{w})=s$.

\textbf{Item }  $iii)$.Since $s=O(i,\underline{w})+1$, $\tilde{R}_r$ is the upper rectangle of $\cR_{S(\underline{w})}$ contained in $R_i$. Consequently, its upper boundary is $I_{i,+1}$, and its lower boundary is the segment $I_{t,\underline{w}}$ located in position $O(i,\underline{w})$ among the elements of $\cI(i,\underline{w})\cup \{I_{i,+1}, I_{i,-1}\}$. Therefore, $\underline{s}(t,\underline{w})=O(i,\underline{w})$.

\end{proof}

Let's determine the image of the stable boundaries of a rectangle $R_r$. We are going to use Lemma \ref{Lemm: image of stable intervals} as our main tool.

\begin{lemm}\label{Lemm: imge of boundaries}

Let $\tilde{R}_r$, where $r=\tilde{r}(i,s)$, be a rectangle in the geometric Markov partition $\cR_{S(\underline{w})}$. We can determine the image of the horizontal boundary of $\tilde{R}_r$ under $f$ in the following manner:
	
	\begin{itemize}
		\item[i)] If the lower boundary of $\tilde{R}_r$ is $I_{i,-1}$ and $\Phi_T(i,1)=(k,l,\epsilon)$, then $f(I_{i,-1})$ is contained within $I_{k,-1}$ if $\epsilon=1$, and it is contained within $I_{k,+1}$ otherwise.

		\item[ii)] If one boundary component of $\tilde{R}_r$ is the stable segment $I_{t,\underline{w}}$, then $f(I_{t,\underline{w}})\subset I_{t+1,\underline{w}}$.
		
		\item[iii)]  If $I_{i,+1}$ is the upper boundary of $\tilde{R}_r$ and $\Phi_T(i,h_i)=(k,l,\epsilon)$, then $f(I_{i,+1})$ is contained in $I_{k,+1}$ if $\epsilon=1$ and is contained in $I_{k,-1}$ otherwise.

	\end{itemize}

\end{lemm}

\begin{proof}
\textbf{Item} $i)$. It's evident that the image of $I_{i,-1}$ under $f$ will be contained either in the upper or lower boundary of $R_k$, depending on how $f$ alters the vertical orientation. Specifically, if $\epsilon=1$, then $f(I_{i,-1})$ is within the lower boundary of $R_k$, whereas if $\epsilon=-1$, it will be situated in the upper boundary of $R_k$.

\textbf{Item}  $ii)$. This is Lemma \ref{Lemm: image of stable intervals}.

\textbf{Item} $iii)$. The argument is the same as in Item $i)$.

\end{proof}

\subsection{The relative position of $H^i_j$ with respect to $\tilde{R}_r.$}

The number of horizontal sub-rectangles within $\tilde{R}_r$ can be determined by counting the number of sub-rectangles contained within each $H^i_j$ that intersect $\tilde{R}_r$   and then summing up all of these quantities. 

\begin{defi}\label{Defi cH(r) set}
	Let $\tilde{R}_r\in \cR_{S(\underline{w})}$ with $r=\tilde{r}(i,s)$. We define:
	$$
	\cH(r):=\{(i,j)\in \cH(T): \overset{o}{H^i_{j}}\cap \overset{o}{\tilde{R}_r}\neq \emptyset\}.
	$$
\end{defi}

If we know the stable segments that delineate $\tilde{R}_r$, we can determine the indices of the horizontal sub-rectangles within $\cR$ that contain the upper and lower boundaries of $\tilde{R}_r$. Let's assume  these horizontal sub-rectangles are in positions $j{-} < j{+}$. Thus, we can define $\cH(r)$ as the set of indices $(i, j)$ where $j$ ranges from $j_{-}$ to $j_{+}$. The following lemma provides a more detailed explanation of this observation.

\begin{lemm}\label{Lemm: determination cH(r)}
Let $\tilde{R}_r$ be a rectangle within the geometric Markov partition $\cR_{S(\underline{w})}$, where $r=\tilde{r}(i,s)$. There are three distinct scenarios to consider:
	
	\begin{itemize}
		\item[i)] If $O(i,\underline{w})=0$, then  $\cH(r)=\{(i,j)\in \cH(T)\}$.

		\item[ii)] If $s\neq 1$ and $s\neq O(i,\underline{w})+1$, we identify the lower boundary of $\tilde{R}_r$ as $I_{t_1,\underline{w}}$, and the upper boundary as $I_{t_2,\underline{w}}$ (given by Item $ii)$ of Lemma \ref{Lemm: Determine boundaris of R-r}), with corresponding indices $j_{t_1,\underline{w}}$ and $j_{t_2,\underline{w}}$ determined by Lemma \ref{Lemm: unique hriwontal for every prjection}. In this situation, $\cH(r)=\{(i,j)\in \cH(T): j_{t_1,\underline{w}} \leq j \leq j_{t_2,\underline{w}} \}$.

		\item[iii)]	If $O(i,\underline{w})>0$ and $s=O(i,\underline{w})+1$, we designate the upper boundary of $\tilde{R}_r$ as $I_{i,+1}$ and the lower boundary as $I_{t,\underline{w}}$, with $j_{t,\underline{w}}$ determined by Lemma \ref{Lemm: unique hriwontal for every prjection}. In this scenario, $\cH(r)=\{(i,j)\in \cH(T): j_{t,\underline{w}} \leq j \leq h_i \}$.
		
		\item[iv)]	If $O(i,\underline{w})>0$ and $s=1$, the inferior boundary of $\tilde{R}_r$ is $I_{i,-1}$ and its upper boundary is denoted by $I_{t,\underline{w}}$, where $j_{t,\underline{w}}$ was determined by Lemma \ref{Lemm: unique hriwontal for every prjection}. In this scenario, $\cH(r)=\{(i,j)\in \cH(T): 1 \leq j \leq j_{t,\underline{w}} \}$.
		
	\end{itemize}
	
\end{lemm}

\begin{proof}
	\textbf{Item} $i)$. If $O(i,\underline{w})=0$, then $\tilde{R}_r=R_i$, indicating that $\cH(r)$ must include the labels of all the horizontal sub-rectangles within the Markov partition $\cR$ that are contained in $R_i$. Thus, $\cH(r)=\{(i,j)\in \cH(T)\}$.
	
	\textbf{Item} $ii)$. The lower boundary of $\tilde{R}_r$ lies within $H^i_{j_{t_1,\underline{w}}}$, while the upper boundary is contained in $H^i_{j_{t_2,\underline{w}}}$. It's evident that all the horizontal sub-rectangles of the form  $H^i_j$ with $j_{t_1,\underline{w}}\leq j \leq j_{t_2,\underline{w}}$ are the only ones that satisfy that $\overset{o}{H^i_{j}}\cap \overset{o}{\tilde{R}_r}$. This implies that $\cH(r)=\{(i,j)\in \cH(T): j_{t_1,\underline{w}} \leq j \leq j_{t_2,\underline{w}} \}$.
		
	\textbf{Item} $iii)$. Every horizontal sub-rectangle $H^i_j$ with $j_{t,\underline{w}}\leq  j \leq \leq h_i$ satisfies that $\overset{o}{H^i_{j}}\cap \overset{o}{\tilde{R}_r}\neq \emptyset$, and these are the only horizontal sub-rectangles within $\cR$ with this property. Thus, $\cH(r)=\{(i,j)\in \cH(T): j_{t,\underline{w}} \leq j \leq h_i \}$.
	
	\textbf{Item} $iv)$. The argument is the same as in Item $iii)$.
			
		\end{proof}

\begin{figure}[h]
	\centering
	\includegraphics[width=0.7\textwidth]{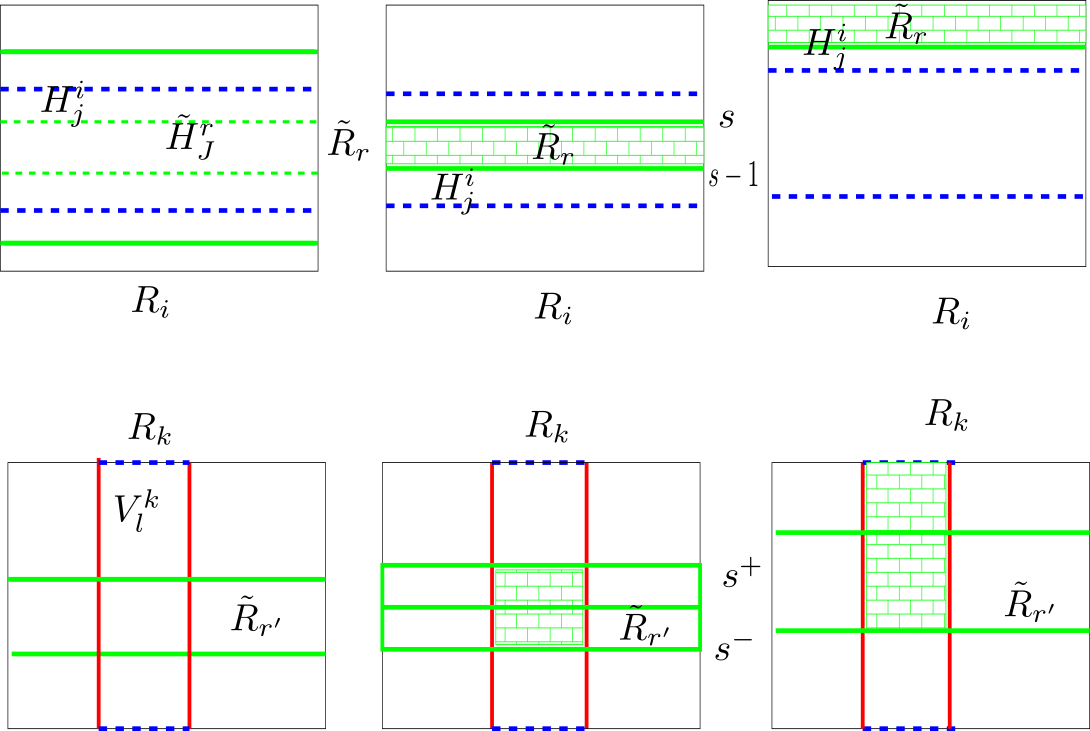}
	\caption{Three configurations for horizontal sub-rectangles of $\tilde{R}_{r}$ inside $\cH(r)$ }
	\label{Fig: 3 cases}
\end{figure}

Let $\tilde{R}_r\in \cR_{S(\underline{w})}$ be a rectangle with $r=\tilde{r}(i,s)$. For each $(i,j)\in \cH(r)$, there are three possible situations regarding the placement of $\tilde{R}_r$ with respect to the sub-rectangle $H^i_j$ of $R_i\in \cR$:
	
	\begin{itemize}
		\item[i)] $H^i_j\subset \tilde{R}_r$ 
		\item[ii)]$\tilde{R}_r \subset H^i_j\setminus\partial^s H^i_j$
		\item[iii)]  $H^i_j$ contains only one horizontal boundary component of $\tilde{R}_r$.
	\end{itemize}
Let's establish a criterion for every pair $(i,j)\in \cH(r)$ to know what kind of behavior they have.

\begin{prop}\label{Prop: determination case of ch(r) inside R-r}
Let $\tilde{R}_r\in \cR_{S(\underline{w})}$ be a rectangle with $r=\tilde{r}(i,s)$. 

Assume that $\tilde{R}_r\neq R_i$ and $ \cH(r)=\{(i,j_1),\cdots, (i,j_K)\}$ with $j_k<j_{k+1}$. Let $(i,j_k)\in \cH(r)$ in this manner:

\begin{itemize}
	\item[i)]  $\tilde{R}_r \subset H^i_{j_k}\setminus\partial^s H^i_{j_k}$ if and only if $\cH(r)=\{(i,j_k)\}$, i.e. $K=1$.
	
	\item[ii)]  $H^i_{j_k}\subset \tilde{R}_r$ if and only if $j_1<j_k<j_K$.
	
	\item[iii)]  The sub-rectangle $H^i_{j_k}$ contains only one horizontal boundary component of $\tilde{R}_r$ if and only if $K> 1$, i. e. the cardinality of $\cH(r)$ is bigger than $1$ ( $\sharp(\cH(r))>1$ ).  In this case, $H^i_{j_k}$ contains  the inferior boundary of $\tilde{R}_r$ if $k=1$ and the upper boundary of $\tilde{R}_r$ if $k=K$.
\end{itemize}

 In case that  $\tilde{R}_r= R_i$, for all $(i,j)\in \cH(r)$,   $H^i_j\subset \tilde{R}_r$.

\end{prop}

\begin{proof}

\textbf{Item} $i)$. If $\tilde{R}_r \subset H^i_{j_k}\setminus \partial^s H^i_{j_k}$, the only horizontal sub-rectangle of $R_i$ that intersects the interior of $\tilde{R}_r$ is $H^i_{j_k}$, so $\cH(r)=\{(i,j_k)\}$ and $K=1$.

Conversely, if $\cH(r)=\{(i,j_1)\}$, the only horizontal sub-rectangle of $R_i$ that intersects the interior of $\tilde{R}_r$ is $H^i_{j_1}$, then $\overset{o}{\tilde{R}_r} \subset H^i_{j_1}$. Since the stable boundaries of $\tilde{R}_r$ are disjoint from the stable boundaries of $H^i_{j_k}$, it is necessary that $\tilde{R}_r \subset H^i_{j_k}\setminus\partial^s H^i_{j_k}$.

\textbf{Item} $ii)$. The horizontal sub-rectangles $H^i_{j_{k-1}}$ and $H^{i}_{j_{k+1}}$ intersect the interior of $\tilde{R}_r$, and $H^i_{j_{k-1}}$ is below $H^i_{j_k}$, and $H^{i}_{j_{k+1}}$ is above $H^i_{j_k}$, then it is necessary that $H^i_{j_k}\subset \tilde{R}_r$. 

Conversely, if $H^i_{j_k}\subset \tilde{R}_r$ and $\tilde{R}_r \neq R_i$, the contention is proper, which implies the existence of horizontal sub-rectangles of $R_i$, $H^i_{j{k+1}}$ above $H^i_{j_k}$ and $H^i_{j_{k-1}}$ below $H^i_{j_k}$. They intersect the interior of $\tilde{R}_r$ and therefore, $(i,j_{k-1})$ and $(i,j_{k+1})$ are elements of $\cH(r)$. This implies that  $j_1\leq j_{k-1}<j_k<j_{k+1}\leq j_K$.

\textbf{Item} $iii)$.If $H^i_{j_k}$ contains only the inferior boundary of $\tilde{R}_r$, the index $j_k$ satisfies that for all $j' \in \cH(r)$, $j' \geq j_k$. Since the upper boundary of $\tilde{R}_r$ is not contained in the interior of $H^i_{j_k}$, there exists another $j' > j$ such that the upper boundary of $\tilde{R}_r$ is contained in the interior of $H^i_{j'}$. Consequently, the interior of $\tilde{R}_r$ intersects the interior of $H^i_{j'}$. Hence, $(i,j')\in \cH(r)$, $K>1$ and $j_k=j_1$. The case where $H^i_{j_k}$ contains only the upper boundary of $\tilde{R}_r$ is entirely symmetric.

Finally if $\tilde{R}_r= R_i$, then $\cH(r)=\{(i,j)\in \cH(T)\}$ and  then every horizontal sub-rectangle $H^i_j$ is contained in $R_i=\tilde{R}_r$, therefore $H^i_j\subset \tilde{R}_r$.

\end{proof}

\subsection{The number of horizontal sub-rectangles in $\tilde{R}_r$}

We will count the number of horizontal sub-rectangles within the Markov partition $\cR_{S(\underline{w})}$ in the intersection $\tilde{R}_r\cap H^i_j$, and we will also establish how these sub-rectangles are rearranged by  $f$. For that reason we consider the three different cases that we outlined in Proposition \ref{Prop: determination case of ch(r) inside R-r}.

The strategy will be:

\begin{enumerate}
\item For every $(i,j)\in \cH(r)$, compute  the number of horizontal sub-rectangles contained within $\tilde{R}_r\cap H^i_j$ and call it $\underline{h}(r,j)$.

\item Label the rectangles within $\tilde{R}_r\cap H^i_j$ as $\{\tilde{H}^r_{j,J}\}_{J=1}^{\underline{h}(r,j)}$.

\item Determine the function $\underline{\rho}_{(r,j)}$  that, when given $(r,J)$, returns a value $(r',l')$ such that $f(\tilde{H}^r_{j,J})=\tilde{V}^{r'}_{l'}$.

\item Calculate $H_r$ by summing $\underline{h}(r,j)$ over all $j$ such that $(i,j)\in \cH(r)$, which gives us the number of horizontal sub-rectangles within $\tilde{R}_r$.

\item Enumerate the horizontal sub-rectangles within $\tilde{R}_r$ by $\{\tilde{H}^r_{\underline{J}}\}_{\underline{J}=1}^{H_r}$ and specify how to determine the  \emph{relative position} of  a rectangle  $\tilde{H}^r_{\underline{J}}$.  This means finding a unique index $j$ and a unique $J$ such that  $\tilde{H}^r_{\underline{J}}=\tilde{H}^r_{j,J}$.
	
\item Finally, define $\rho_{S(\underline{w})}(r,\underline{J})$ as $\rho_{(r,j)}(r,J)$, where $j$ and $J$ were determined in the previous step.

\end{enumerate}

\subsubsection{First case: $\tilde{R}_r \subset H^i_{j}\setminus\partial^s H^i_{j}$ }

In some sense, this is the simplest scenario because all the horizontal sub-rectangles of $\tilde{R}_r$ are contained within $H^i_j$. This allows us to directly compute the value for $H_r$ and the function $\rho_{S(\underline{w})}$ when restricted to the horizontal sub-rectangles of $\tilde{R}_r$.

\begin{lemm}\label{Lemm: number horizontal sub case II}
	Let $\tilde{R}_r\in \cR_{S(\underline{w})}$, with $r=\tilde{r}(i,s)$. Suppose we have a pair of indices $(i,j)\in \cH(r)$ such that $\tilde{R}_r \subset H^i_j\setminus\partial^s H^i_j$. In this case $s\neq 1$ and $s\neq O(i,\underline{w})+1$ and we take $I_{t_1,\underline{w}}$ as the lower boundary of $\tilde{R}_r$ and $I_{t_2,\underline{w}}$ as the upper boundary of $\tilde{R}_r$.
	
	Assume that $\rho_T(i,j)=(k,l)$. Let $s^-=\underline{s}(t_1+1,\underline{w})$ and $s^+=\underline{s}(t_2+1,\underline{w})$. Then,  $w_{t_1+1}=w_{t_2+1}=k$ and the number of horizontal sub-rectangles of $\cR_{S(\underline{w})}$ contained in the intersection $\tilde{R}_r\cap H^i_j$ is given by the formula:
	
	\begin{equation}\label{Equa: h (i,j,r) caso II}
	\underline{h}(r,j)= \vert s^- - s^+\vert.
	\end{equation} 
	
\end{lemm}

\begin{proof}
	Look at Figure  \ref{Fig: Caso I}. Since $\tilde{R}_r \subset H^i_j\setminus\partial^s H^i_j$, $\tilde{R}_r$ doesn't share a stable boundary component with $R_i$. This means $s\neq 1$ and $s \neq O(i,\underline{w})+1$, and we can take the intervals $I_{t_1,\underline{w}}$ and $I_{t_2,\underline{w}}$ as described in the Lemma.	Since $f(\tilde{R}_r)\subset f(H^i_j)\subset R_k$, then $w_{t_1+1}=w_{t_2+1}=k$, as we claimed. Then the number of horizontal sub-rectangles of $\tilde{R}_r$ consists of all the rectangles in $\cR_{S(\underline{w})}$ contained in $R_k$ and located between the intervals $I_{t_1+1,\underline{w}}$ and $I_{t_2+2,\underline{w}}$. If their relative positions are given by the numbers $s^-$ and $s^+$, then the number of such rectangles is given by $\vert s^- - s^+\vert$ and we has proved our Lemma.
	
\end{proof}

\begin{defi}\label{Defi: order rectangles case II}
	The horizontal sub-rectangles of $\tilde{R}_r$ as described in Lemma \ref{Lemm: number horizontal sub case II} are labeled from the bottom to the top with respect to the orientation of $\tilde{R}_r$ as:
	$$
	\{\tilde{H}^r_{j,J}\}_{J=1}^{\underline{h}(r,j)}.
	$$
\end{defi}

\begin{rema}
	As established in Item $i)$ of Proposition \ref{Prop: determination case of ch(r) inside R-r}, the set $\cH(r)$ contains only one element. Therefore, $H_r = \vert s^- - s^+\vert$, 
\end{rema}

\begin{lemm} \label{Lemm: Permutation rho caso II}
	With the hypothesis of Lemma \ref{Lemm: number horizontal sub case II}, assume that $\Phi_T(i,j)=(k,l,\epsilon)$ and $f(\tilde{H}^r_{j,J})=\tilde{V}^{r'}_{l'}$. Then:
	\begin{itemize}
		\item[i)] 	If $\epsilon=\epsilon(i,j) = 1$, then $s^- < s^+$, $r'=\tilde{r}(k, s^- + J)$ and $l=l'$.
		\item[ii)] If $\epsilon=\epsilon(i,j) = -1$, then $s^- > s^+$, $r'=\tilde{r}(k, s^- - J+1)$ and $l=l'$.
	\end{itemize}
\end{lemm}

\begin{proof}
	\textbf{Item} $i)$. If $\epsilon(i,j) = 1$ then $f$ preserves the vertical orientation of $\tilde{R}_r$, and then $I_{t_1+1,\underline{w}} < I_{t_2+1,\underline{w}}$, so  $s^- < s^+$.
	
	The rectangle $\tilde{H}^r_{j,J}$ is in vertical position $J$ inside $\tilde{R}_r\cap H^i_j$ and it is sent to the rectangle whose position inside $R_k$ is at vertical position $J$  above the the rectangle in position $s^-$ withing $R_k$. Therefore, $f(\tilde{H}^r_{j,J})$ is inside a rectangle $\tilde{R}_{r'}$ inside to  $R_k$ whose position  is equal to $s^- + J$. In other words, $r'=\tilde{r}(k,s^-+J)$.
	
	\textbf{Item} $ii)$. If $\epsilon(i,j) = -1$ then $f$ inverts the vertical orientation of $\tilde{R}_r$. Then, $I_{t_1+1,\underline{w}} > I_{t_2+1,\underline{w}}$ and $s^- > s^+$.
	
	 The rectangle $\tilde{H}^r_{j,1}$ that is positioned at place $1$ inside $\tilde{R}_r\cap H^i_j$ is sent into to the rectangle $\tilde{R}_{r'}$, whose upper boundary is at position $s^-$, therefore $r'=s^-=s^--(1-1)$. In general the upper boundary of the rectangles $\tilde{R}\subset R_k$ determine its position inside $R_k$. That means that  $f(\tilde{H}^r_{j,J})$ is inside a rectangle $\tilde{R}_{r'}$ withing  $R_k$ whose position  is given by $s^- - J+1$. In other words, $r'=\tilde{r}(k,s^- -J+1)$.
	
	Finally in bot cases,  Proposition \ref{Prop: vertical possition is delimitate.} implies that $l'=l$.
\end{proof}

\begin{defi}\label{Defi: permutation caso II}
	With the hypothesis of Lemma \ref{Lemm: number horizontal sub case II}, assume that $\Phi_T(i,j)=(k,l,\epsilon)$. Then: 
	
	If $\epsilon(i,j) = 1$ we define:
	
	\begin{equation}\label{Equ: rho caso 2, +}
	\underline{\rho}_{(r,j)}(r,J)=(\tilde{r}(k, s^-+J), l  ).
	\end{equation}
	
	and if $\epsilon(i,j) = -1$:
	\begin{equation}\label{Equ: rho caso 2, -}
	\underline{\rho}_{(r,j)}(r,J)=(\tilde{r}(k, s^+  -J+1) , l    ).
	\end{equation}
\end{defi}

\begin{rema}\label{Rema: Justifcation notation case I}
In this particular case the function $\rho_{S(\underline{w})}(r,J)$ is determined by either Equation \ref{Equ: rho caso 2, -} or Equation \ref{Equ: rho caso 2, +}. Towards the end of this section, we will provide a unified approach of all the cases that we are going to consider, for this reason we are keeping the notation $\tilde{H}^r_{j,J}$ even if we can change it  by $\tilde{H}^r_{J}$ as the parameter $j$ is not important anymore.
\end{rema}

\begin{figure}[h]
	\centering
	\includegraphics[width=0.8\textwidth]{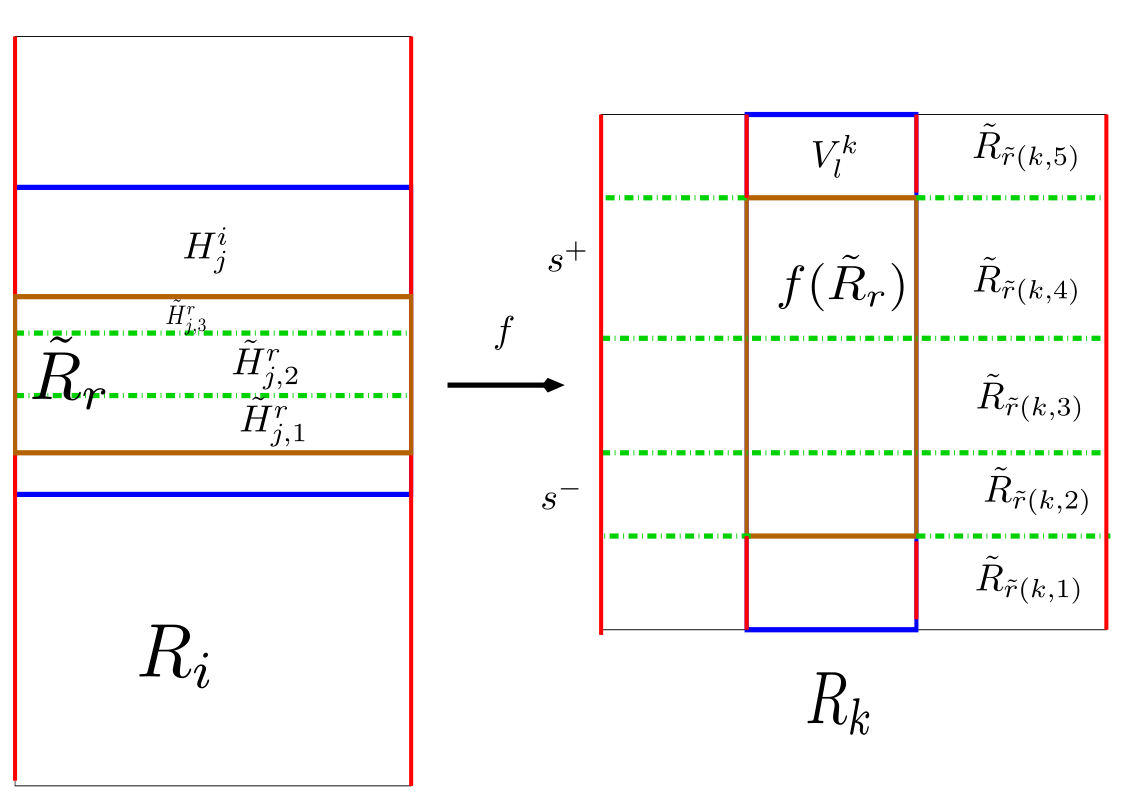}
	\caption{First case: $\tilde{R}_r \subset H^i_{j}\setminus\partial^s H^i_{j}$ }
	\label{Fig: Caso I}
\end{figure}

\subsubsection{Second case: $H^i_{j}\subset \tilde{R}_r$ }

\begin{lemm}\label{Lemm: number horizontal sub case I}

Let $\tilde{R}_r\in \cR_{S(\underline{w})}$ be a rectangle, with $r=\tilde{r}(i,s)$. Let $(i,j)\in \cH(r)$ be any pair of indices such that $H^i_j\subset \tilde{R}_r$ and assume that $\rho_T(i,j)=(k,l)$. Then, the number of horizontal sub-rectangles of $\cR_{S(\underline{w})}$ that are contained in the intersection $\tilde{R}_r\cap H^i_j$ is given by the formula:

\begin{equation}\label{Equa: h (i,j,r) caso I}
\underline{h}(r,j)=O(k,\underline{w})+1.
\end{equation}

		\end{lemm}

\begin{proof}
See  Figure  \ref{Fig: Caso II}. The vertical sub-rectangle $V^k_l=f(H^i_j)$ of $R_k\in \cR$  intersects to all the rectangles $\tilde{R}_{r'}\in \cR_{S(\underline{w})}$ that are contained in $R_k$, but by definition this quantity is equal to $O(k,\underline{w})+1$. Even more, this number coincides with the number of horizontal sub-rectangles of $\cR_{S(\underline{w})}$ that are contained inside $\tilde{R}_r\cap H^i_j$ and therefore: 

\begin{equation}\label{Equ: h(i,j,r) caso 1}
\underline{h}(r,j)= O(k,\underline{w})+1,
\end{equation}

is  the number of horizontal sub-rectangles of $\cR_{S(\underline{w})}$ inside $\tilde{R}_r\cap H^i_j$.

\end{proof}

We proceed to label such horizontal sub-rectangles.

\begin{defi}\label{Defi: order rectangles case I}
The horizontal sub-rectangles of $\tilde{R}_r$ described in Lemma \ref{Lemm: number horizontal sub case I} are labeled from the bottom to the top with respect to the vertical orientation of $\tilde{R}_r$ as:
$$
\{\tilde{H}^r_{j,J}\}_{J=1}^{\underline{h}(r,j)}.
$$
\end{defi}

		\begin{figure}[h]
	\centering
	\includegraphics[width=0.8\textwidth]{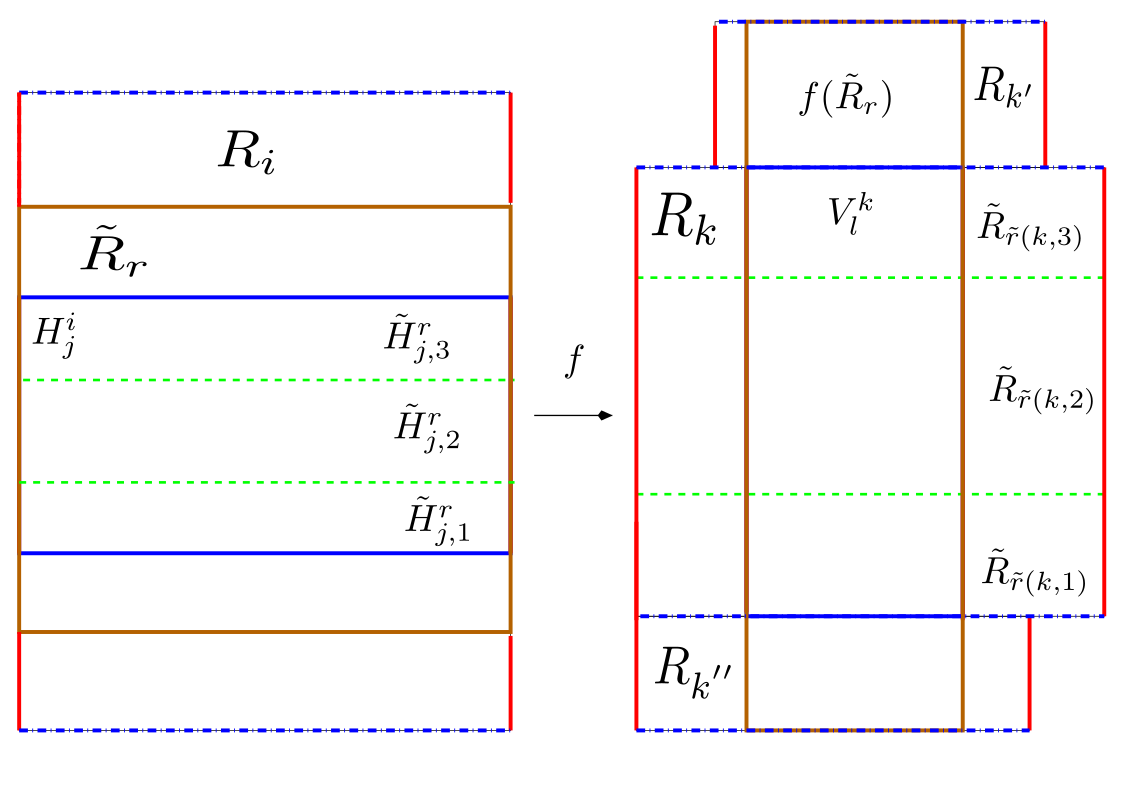}
	\caption{Second case: $H^i_{j}\subset \tilde{R}_r$  }
	\label{Fig: Caso II}
\end{figure}

Let's  determine the image of any such horizontal sub-rectangle of $\tilde{R}_r\cap H^i_j$.
		
		\begin{lemm}\label{Lemm: Permutation rho caso I}
	Let $\tilde{R}_r\in \cR_{S(\underline{w})}$, with $r=\tilde{r}(i,s)$. Let $(i,j)\in \cH(r)$ be any pair of indices such that $H^i_j\subset \tilde{R}_r$, and let $\tilde{H}^r_{j,J}$ be a horizontal sub-rectangle of $\tilde{R}_r$. Assume that $\rho_T(i,j)=(k,l)$  and that $f(\tilde{H}^r_{j,J})=\tilde{V}^{r'}_{l'}$. 
\begin{itemize}
\item[i)] If $\epsilon_T(i,j)=1$,  $r'=\tilde{r}(k,J)$ and $l=l'$.

\item[ii)] If  $\epsilon_T(i,j)=-1$, then $r'=\tilde{r}(k,O(k,\underline{w})+1 -J+1)$ and $l=l'$.
\end{itemize}
		
		\end{lemm}
		
		\begin{proof}
\textbf{Item} $i)$. If $\epsilon(i,j)=1$, the vertical orientation is preserved by $f$ restricted to $\tilde{R}_r$, and the horizontal sub-rectangle $\tilde{H}^r_{j,J}$ that is in position $J$ inside $\tilde{R}_r\cap H^i_j$ is sent to the rectangle $\tilde{R}_{r'}$ located in position $J$ inside $R_k$. This means that $r'=\tilde{r}(k,J)$ and in view of Proposition \ref{Prop: vertical possition is delimitate.}, $l'=l$.

\textbf{Item} $i)$. If $\epsilon(i,j)=-1$, the vertical orientation is invert by $f$ restricted to $\tilde{R}_r\cap H^i_j$, and the horizontal sub-rectangle $\tilde{H}^r_{j,J}$  that is in position $J$ inside $\tilde{R}_r\cap H^i_j$ is sent to the rectangle $\tilde{R}_{r'}$ located in position $J-1$ below the top rectangle of $\cR_{S(\underline{w})}$ (for $J=1$ the position is $O(k,\underline{w}))+1)$, and this position is equal to $O(k,\underline{w})+1-J+1$ inside $R_k$. This means that $r'=\tilde{r}(k,O(k,\underline{w})+1-J+1)$. Finally,  by  Proposition \ref{Prop: vertical possition is delimitate.},  $l'=l$.

		\end{proof}

\begin{defi}\label{Defi: permutation caso I}
Let $\tilde{R}_r\in \cR_{S(\underline{w})}$, with $r=\tilde{r}(i,s)$. Let $(i,j)\in \cH(r)$ be a pair of indices such that $H^i_j\subset \tilde{R}_r$, and let $\tilde{H}^r_{j,J}$ be  an horizontal sub-rectangle of $\tilde{R}_r\cap H^i_j$. Assume that $\rho_T(i,j)=(k,l)$. 

  If  $\epsilon(i,j)=1$ we set:
  
\begin{equation}\label{Equ: rho caso 1, +}
\underline{\rho}_{(r,j)}(r,J)=(\tilde{r}(k,J),l),
\end{equation}

If $\epsilon(i,j)=-1$ we set:

\begin{equation} \label{Equ: rho caso 1, -}
\underline{\rho}_{(r,j)}(r,J)=(\tilde{r}(k,O(k,\underline{w})+1 -J+1),l).
\end{equation}
\end{defi}

\subsubsection{Third case: $H^i_{j}$ contains only one horizontal boundary component of $\tilde{R}_r$ }

	We consider the case when $H^i_j$ contains the upper boundary of $\tilde{R}_r$.
 
 	\begin{lemm}\label{Lemm: number horizontal sub case III, upper}
 	
 	Consider $\tilde{R}_r\in \cR_{S(\underline{w})}$, where $r=\tilde{r}(i,s)$, and assume that $O(i,\underline{w})>0$.  Take a pair of indices $(i,j)\in \cH(r)$ such that $H^i_j$ exclusively contains the \emph{upper boundary} of $\tilde{R}_r$. Let $I_{t,\underline{w}}$ be the upper boundary of $\tilde{R}_r$. Let $s^+=\underline{s}(t+1,\underline{w})$ and assume that $\rho_T(i,j)=(k,l)$. Then, the number of horizontal sub-rectangles of  $\cR_{S(\underline{w})}$ that are contained inside  $\tilde{R}_r\cap H^i_j$ can be computed using one of the following formulas:

\begin{itemize}
	\item If $\epsilon_T(i,j)=1$:
	
	\begin{equation}\label{Equ: h(i,j,r), caso 3, +, up}
	\underline{h}(r,j)=s^+.
	\end{equation} 
	
	\item If $\epsilon_T(i,j)=-1$
	\begin{equation}\label{Equ: h(i,j,r), caso 3, -,up}
	\underline{h}(r,j)=O(k,\underline{w})+1-s^+ .
	\end{equation} 
\end{itemize}

	 \end{lemm}

	 		\begin{figure}[h]
	 	\centering
	 	\includegraphics[width=0.8\textwidth]{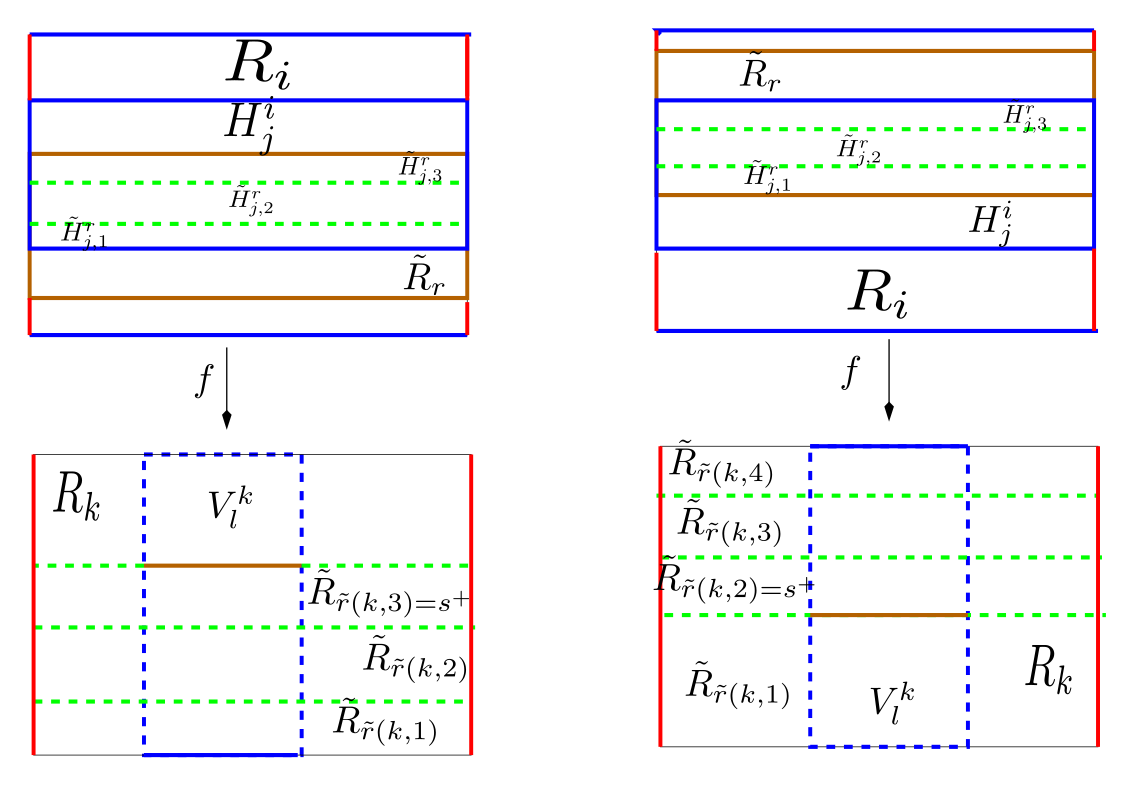}
	 	\caption{Third case: $H^i_{j}$ contains only one horizontal boundary component of $\tilde{R}_r$  }
	 	\label{Fig: Caso III}
	 \end{figure}

	 \begin{proof}
	 	
	Look the  Figure  \ref{Fig: Caso III} to get some pictorial intuition.
	
\textbf{Case} $\epsilon_T(i,j)=1$. The upper boundary of $\tilde{R}_r$ is mapped to the stable segment $I_{t+1,\underline{w}}$, positioned at place $s^+$ inside $R_k$ and the inferior  boundary of $\tilde{R}_r\cap H^i_j$ coincides with the inferior  boundary of $H^i_j$ and since $\epsilon_T(i,j)=1$, $f$ preserves the vertical orientation of $\tilde{R}_r\cap H^i_j$. Consequently, the image of the inferior  boundary of $\tilde{R}_r\cap H^i_j$ is contained in the inferior boundary of $R_k$. The number of rectangles in $\cR_{S(\underline{w})}$ between the segment $I_{t+1,\underline{w}}$ and the inferior boundary of $R_k$ is given by $s^+$ and this number is equal to the number of horizontal sub-rectangles of $\tilde{R}_r \cap H^i_j$.

\textbf{Case} $\epsilon_T(i,j)=-1$. The upper boundary of $\tilde{R}_r\cap H^i_j$ is equal to the upper boundary of $\tilde{R}_r$ and  is mapped by $f$ to the stable segment $I_{t+1,\underline{w}}$, positioned at place $s^+$ inside $R_k$. The inferior boundary of $\tilde{R}_r\cap H^i_j$ coincides with the inferior boundary of $H^i_j$ , but since $\epsilon_T(i,j)=-1$, $f$ inverts the vertical orientation of $\tilde{R}_r\cap H^i_j$. Consequently, the image of the inferior  boundary of $\tilde{R}_r\cap H^i_j$ is contained in the upper boundary of $R_k$. The number of rectangles in $\cR_{S(\underline{w})}$ between the segment $I_{t+1,\underline{w}}$ and the upper boundary of $R_k$ is given by $O(k,\underline{w})+1-s^+$ and this number is equal to the number of horizontal sub-rectangles of $\tilde{R}_r \cap H^i_j$.

	 \end{proof}
	 
	 The next lemma  addresses the situation in which the rectangle $H^i_j$ just contains the inferior boundary of $\tilde{R}_r$. Its proof follows a similar approach to that of Lemma \ref{Lemm: number horizontal sub case III, upper}, so we won't delve into further details.

\begin{lemm}\label{Lemm: number horizontal sub case III, lower}
	 	
Consider $\tilde{R}_r\in \cR_{S(\underline{w})}$, where $r=\tilde{r}(i,s)$, and assume that $O(i,\underline{w})>0$. Take a pair of indices $(i,j)\in \cH(r)$ such that  $H^i_j$ just contains the lower boundary of $\tilde{R}_r$. Let $I_{t,\underline{w}}$ be the lower boundary of $\tilde{R}_r$, and denote $s^+=\underline{s}(t+1,\underline{w})$. In this scenario, the number of horizontal sub-rectangles of  $\cR_{S(\underline{w})}$ that are contained in  $\tilde{R}_r\cap H^i_j$ can be determined using one of the following formulas:

\begin{itemize}
	\item If $\epsilon_T(i,j)=1$:
	
	\begin{equation}\label{Equ: h(i,j,r), caso 3, +}
	\underline{h}(r,j)=O(k,\underline{w})+1-s^+ .
	\end{equation} 
	
	\item If $\epsilon_T(i,j)=-1$
	\begin{equation}\label{Equ: h(i,j,r), caso 3, -}
	\underline{h}(r,j)= s^+ .
	\end{equation} 
\end{itemize}
\end{lemm}

	 
	 \begin{defi}\label{Defi: order rectangles case III}
The horizontal sub-rectangles of $\tilde{R}_r$ described in Lemma \ref{Lemm: number horizontal sub case III, upper}  and Lemma \ref{Lemm: number horizontal sub case III, lower}  are labeled from the bottom to the top with the vertical orientation of $\tilde{R}_r$ as:
$$
\{\tilde{H}^r_{j,J}\}_{J=1}^{\underline{h}(r,j)}.
$$
	 \end{defi}


	 \begin{lemm} \label{Lemm: Permutation rho caso, up III}
With the hypothesis of Lemma \ref{Lemm: number horizontal sub case III, upper}, assume that $\Phi_T(i,j)=(k,l,\epsilon_T(i,j))$ and $f(\tilde{H}^r_{j,J})=\tilde{V}^{r'}_{l'}$.  Suppose that  $H^i_j$  just contains the upper boundary of $\tilde{R}_r$, then we have the following situations: 
\begin{itemize}
\item[i)] 	If $\epsilon_T(i,j)=1$, then $r'=\tilde{r}(k,  J)$ and $l=l'$

\item[ii)] If $\epsilon_T(i,j)=-1$, then $r'=\tilde{r}(k, O(k,\underline{w})+1-J+1)$ and $l=l'$.
\end{itemize}

	 \end{lemm}
 
 \begin{proof}
\textbf{Item} $i)$. When $\epsilon_T(i,j) = 1$  $f$ preserves the vertical orientation  restricted to $\tilde{R}_r\cap H^i_j$. Consequently, the horizontal sub-rectangle $\tilde{H}^r_{j,J}$ that is located at position $J$ within $\tilde{R}_r\cap H^i_j$, is send to the rectangle $\tilde{R}_{r'}$ that is located at position $J$ over the rectangle that is in the position $s^+$ within $R_k$, therefore $r'=\tilde{r}(k,s^+ + J)$.
 
\textbf{Item}  $ii)$. When $\epsilon_T(i,j) = -1$,  $f$ inverts the vertical orientation  of $\tilde{R}_r\cap H^i_j$. Consequently, the horizontal sub-rectangle $\tilde{H}^r_{j,J}$ that is located at position $J$ within $\tilde{R}_r\cap H^i_j$, is send to the rectangle located at position $J-1$ below the upper rectangle in $R_k$, this is given by  $O(k,\underline{w})+1-(J-1)$ withing $R_k$. Therefore, $r'=\tilde{r}(k,O(k,\underline{w})+1-J+1)$. 

In bot cases  Proposition \ref{Prop: vertical possition is delimitate.} implies that $l'=l$. 

 \end{proof}

 The proof of the next lemma is the same, but in this case we consider the case when $H^i_j$ contains the inferior boundary of $\tilde{R}_r$.

 \begin{lemm} \label{Lemm: Permutation rho caso, low III}
 	With the hypothesis of Lemma \ref{Lemm: number horizontal sub case III, lower}, assume that $\Phi_T(i,j)=(k,l,\epsilon_T(i,j))$ and $f(\tilde{H}^r_{j,J})=\tilde{V}^{r'}_{l'}$.  	Suppose that  $H^i_j$  just contains the inferior boundary of $\tilde{R}_r$, then we have the following situations: 
 	
 	\begin{itemize}
 		\item[i)] 	If $\epsilon_T(i,j)=1$, then $r'=\tilde{r}(k, s^+ + J)$ and $l=l'$

 		\item[ii)] If $\epsilon_T(i,j)=-1$, then $r'=\tilde{r}(k, s^+-J+1)$ and $l=l'$.
 	\end{itemize} 	
 	
 \end{lemm}

 
 \begin{defi}\label{Defi: permutation caso, up III}
	With the hypothesis of Lemma \ref{Lemm: number horizontal sub case III, upper}, assume that $\Phi_T(i,j)=(k,l,\epsilon)$. Suppose that  $H^i_j$  just contains the upper boundary of $\tilde{R}_r$: 		
	
	\begin{itemize}
		\item[i)]  	If $\epsilon_T(i,j)=1$, we define:
		\begin{equation}\label{Equ: rho caso 3, +}
		\underline{\rho}_{(r,j)}(r,J)=(\tilde{r}(k,  J), l).
		\end{equation}
		
		\item[ii)] if $\epsilon_T(i,j)=-1$ we take:
		
		\begin{equation}\label{Equ: rho caso 3, -}
		\underline{\rho}_{(r,j)}(r,J)=(\tilde{r}(k, O(k,\underline{w})+1-J+1),l).
		\end{equation}
	\end{itemize} 	
\end{defi}

  \begin{defi}\label{Defi: permutation caso, low III}
 	With the hypothesis of Lemma \ref{Lemm: number horizontal sub case III, lower}, assume that $\Phi_T(i,j)=(k,l,\epsilon_T(i,j))$. Suppose that  $H^i_j$  just contains the lower boundary of $\tilde{R}_r$, then:	
 	
 	\begin{itemize}
 		\item[i)]  	If $\epsilon_T(i,j)=1$, we define:
 		\begin{equation}\label{Equ: rho caso 3, +,low}
 		\underline{\rho}_{(r,j)}(r,J)=(\tilde{r}(k, s^+ + J), l).
 		\end{equation}
 		
 		\item[ii)] if $\epsilon_T(i,j)=-1$ we take:
 		
 		\begin{equation}\label{Equ: rho caso 3, -,low}
 		\underline{\rho}_{(r,j)}(r,J)=(\tilde{r}(k, s^+-J+1),l).
 		\end{equation}
 	\end{itemize} 	
 	
 \end{defi}

\subsection{The number $H_r$}
 
 \begin{coro}\label{Coro: Number of horizontal sub}
 Let $\tilde{R}_r\in \cR_{S(\underline{w})}$ with $r=\tilde{r}(i,s)$, then the number $H_r$ of horizontal sub-rectangles of the Markov partition $(f,\cR_{S(\underline{w})})$ that are contained in $\tilde{R}_r$ is given by:
 	\begin{equation}
 	H_r=\sum_{\{j:(i,j)\in \cH(r)\}}\underline{h}(r,j).
 	\end{equation}
 \end{coro}

 \subsection{The function $\rho_{S(\underline{w})}$ in the geometric type $T_{S(\underline{w})}$}

 \begin{figure}[h]
 	\centering
 	\includegraphics[width=0.8\textwidth]{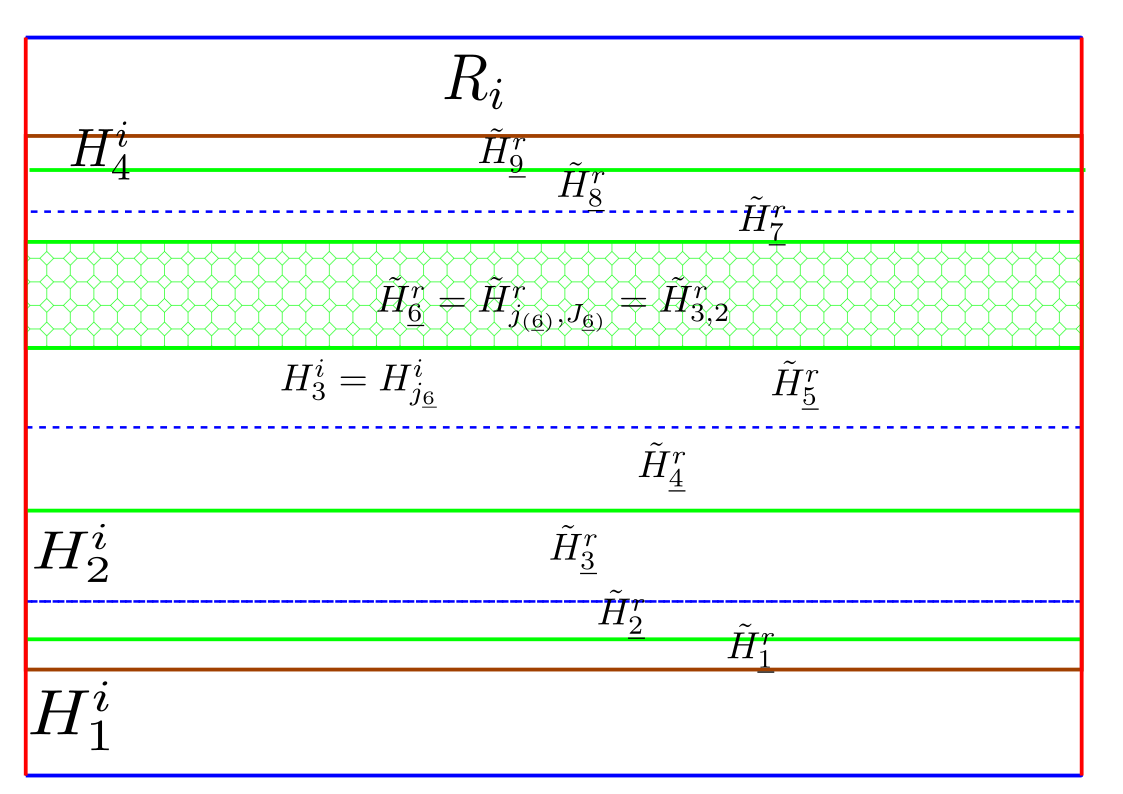}
 	\caption{The relative position of $\tilde{H}^r_{\underline{J}}$ for $\underline{J}=\underline{6}$ }
 	\label{Fig: Relative}
 \end{figure}

\begin{defi}\label{Defi: horizontal sub in R-r}
The sub-rectangles of the geometric Markov partition $\cR_{S(\underline{w})}$ contained in $\tilde{R}_r$ are labeled as:
$$
\{H^r_{\underline{J}}\}_{\underline{J}=1}^{H_r}
$$
from the bottom to the top with respect to the vertical orientation of  $\tilde{R}_r$.
\end{defi}

We've explained how to compute $\underline{\rho}_{(r,j)}(r,J)$ for the three distinct cases. Now, we need to integrate them into a unified function,  $\rho_{S(\underline{w})}$. The initial step in this process is to find a pair $(j,J)$ than, once given $r=\tilde{r}(i,s)$ and $\underline{J}\in\{1,\cdots,H_r\}$,  $\tilde{H}^r_{j,J}=\tilde{H}^r_{\underline{J}}$. We suggest to look the Figure 	\ref{Fig: Relative} to get some intuition about our procedure and definitions below.

\begin{lemm}\label{Lemm: parameter J}
For each $\underline{J}\in\{1,\cdots, H_r\}$, there exists a unique index $j_{(\underline{J})}$ such that $(i,j)\in\cH(r)$  and  $\tilde{H}^r_{\underline{J}}\subset H^i_j\cap \tilde{R}_r$. The index $j_{(\underline{J})}$  is determined by satisfy the next inequalities:
	
	\begin{equation}\label{Equa: Cut the index}
	\sum_{\{j': (i,j')\in \cH(r)  \text{ and } j'<j_{(\underline{J})}\} }\underline{h}(j',r)< \underline{J} \leq  \sum_{\{j': (i,j')\in \cH(r)  \text{ and } j'\leq j_{(\underline{J})}\} }\underline{h}(j',r).
	\end{equation} 
	
\end{lemm}

\begin{proof}
This is immediate.
\end{proof}

\begin{defi}\label{Defi: Horizontal parameter}
Given $\underline{J}\in \{1,\cdots, H_r\}$ we call the index $j_{(\underline{J})}$ that was determined in Lemma \ref{Lemm: parameter J} the \emph{horizontal parameter} of $\underline{J}$.
\end{defi}

\begin{defi}\label{Defi: relative index }
Let $r=\tilde{r}(i,s)$, $\underline{J}\in \{1,\cdots, H_r\}$, and let $j_{(\underline{J})}$ be the horizontal parameter of $\underline{J}$. The \emph{relative index} of $H^r_J$ inside of $\tilde{R}_r\cap H^i_{j_{\underline{J}}}$ is given by:
\begin{equation}\label{Equa: relative index}
J_{(\underline{J})}:=\underline{J}-\sum_{j'<j_{({\underline{J})}}}\underline{h}(r,j').
\end{equation}

\end{defi} 

\begin{defi}\label{Defi: relative position }
Let $\underline{J}\in \{1,\cdots, H_r\}$ the \emph{relative position} of $\underline{J}$ inside $\tilde{R}_r$ is the pair $(j_{(\underline{J})},J_{(\underline{J})})$ of the horizontal parameter of $\underline{J}$ and the relative index of $H^r_J$ inside of $\tilde{R}_r\cap H^i_{j_{(\underline{J}})}$
\end{defi}

\begin{lemm}\label{Lemm: relative position relation}
The rectangle $\tilde{H}^r_{j_{(\underline{J})},J_{(\underline{J})}}$ and the rectangle $\tilde{H}^r_{\underline{J}}$ are the same.
\end{lemm}

\begin{proof}
The rectangle $\tilde{H}^r_{\underline{J}}$ within $\tilde{R}_r$ holds position $\underline{J}$ among all its horizontal sub-rectangles.  If we know which horizontal sub-rectangle $H^i_j$ (or $\cR$ ) contains $\tilde{H}^r_{\underline{J}}$, we can determine the index $\underline{J}$ as the sum of all the total of horizontal sub-rectangle of $\tilde{R}_r$ that are contained in $H^i_{j'}\cap \tilde{R}_r$ with $j'<j$ and then determine the position $J$ of $\tilde{H}^r_{\underline{J}}$ inside $H^i_{j}\cap \tilde{R}_r$.
	
Equation \ref{Equa: Cut the index} specifies that $\tilde{H}^r_{\underline{J}}$ is not included in the horizontal sub-rectangles $H^i_{j'}$ for $j' < j_{ (\underline{J}) }$ but is contained in $H^i_{ j_{ (\underline{J}) } }$. Even more, equation \ref{Equa: relative index} said that $\tilde{H}^r_{ \underline{J} }$ holds  the position $J_{(\underline{J})}$ inside $H^i_{ j_{ (\underline{J}) } }$.   Thus, we can conclude that: 
	$$
	\tilde{H}^r_{j_{(\underline{J})},J_{(\underline{J})}}=\tilde{H}^r_{\underline{J}}
	$$ 
	as was claimed.
\end{proof}

The relative position of $\tilde{H}^r_{\underline{J}}$ inside the rectangle $\tilde{R}_r\cap H^i_j$ allows us to apply the corresponding formula to this $J_{(\underline{w})}$ in order to compute the label of the vertical sub-rectangle of $\cR_{S(\underline{w})}$ of its image.

\begin{coro}\label{Coro: The permutation}
Let $r=\tilde{r}(i,s)$ and $\underline{J}\in \{1,\cdots, H_r\}$. Let $(j_{(\underline{J})},J_{(\underline{J})})$ the relative position of $\underline{J}$ as was given in Definition \ref{Defi: relative position }. Then the function $\rho_{S(\underline{w})}$ is given by the formula:
		\begin{equation}\label{Equa: The permutation}
	\rho_{S(\underline{w})}(r,\underline{J}):=\underline{\rho}_{(r,j_{(\underline{J})})}(r,J_{(\underline{J})}).
	\end{equation}
\end{coro}

\begin{proof}
This is just a direct consequence of our previous constructions and the Lemma \ref{Lemm: relative position relation}.
\end{proof}

\subsection{The  orientation function $\epsilon_{S(\underline{w})}$}

\begin{lemm}\label{Lemm: The orientation}
	Let $r=\tilde{r}(i,s)$ and $\underline{J}\in \{1,\cdots, H_r\}$. Let $j_{(\underline{J})}$ the horizontal parameter of $\underline{J}$ (Definition  \ref{Defi: Horizontal parameter}). Then:
	\begin{equation}\label{Equa: permtation in the s refinament}
	\epsilon_{S(\underline{w})}(r,\underline{J}):=\epsilon_T(i,j_{(\underline{J})}).
	\end{equation}
\end{lemm}

\begin{proof}

Since $\tilde{H}^r_{\underline{J}}\subset H^i_{j_{(\underline{J})}}$, the reversal or preservation of the horizontal orientation only depends on $\epsilon_T(i,j)$. This gives the equality.
\end{proof}

\subsection{Conclusion}

\begin{coro}\label{Coro: Algoritm computation TS(w)}
Let $T$ be a geometric type in the pseudo-Anosov class with a binary incidence matrix denoted as $A:=A(T)$. Given a periodic code $\underline{w}\in \Sigma_A$, there is algorithm to compute the   geometric type:

\begin{equation*}
T_{S(\underline{w})}:=\{N,\{H_r,V_r\}_{r=1}^N,\Phi_{S(\underline{w})}:=(\rho_{S(\underline{w})},\epsilon_{S(\underline{w})})\}.
\end{equation*}

in therms of the given geometric type $T$ and the code $\underline{w}$. 
\end{coro}

\begin{proof}
Each parameter in $T_{S(\underline{w})}$ has been rigorously determined through the following steps:
 
 \begin{enumerate}
 	\item The number $N$ was calculated in Corollary  \ref{Coro: Number N in the type T-S(w)}.
 	\item The values of $V_r$ were established in Corollary \ref{Coro: number of vertical sub rec}. \ref{Coro: number of vertical sub rec}
 	\item The formula for $H_r$ was provided in Corollary \ref{Coro: Number of horizontal sub}
 	\item The permutation $\rho_{S(\underline{w})}$ was obtained in Corollary \ref{Coro: The permutation}
 	\item The function $\epsilon_{S(\underline{w})}$ was established in Lemma \ref{Lemm: The orientation}
 	\end{enumerate}
 
It's important to note that each of these results is derived from an explicit algorithm that solely relies on the combinatorial information of $T$ and $\underline{w}$. 
\end{proof}

\section{Cut along a family of periodic codes}

Let $T$ be a geometric type in the pseudo-Anosov class with a binary incidence matrix $A:=A(T)$. Consider a generalized pseudo-Anosov homeomorphism $f:S\rightarrow S$ with a geometric Markov partition $\cR$ of geometric type $T$. Suppose we have a finite family of non-$s$-boundary periodic codes:
 $$
\cW=\{\underline{w}^1,\cdots,\underline{w}^Q\}\in \Sigma_A.
$$

In this section, we aim to generalize the procedure for cutting along the stable intervals determined by a single periodic code to cutting along the family of stable intervals determined by the family $\cW$.

\begin{conv}
For the remainder of this section, we will maintain  fixed:
\begin{itemize}
\item The family of non $s$-boundary periodic codes $\cW=\{\underline{w}^1,\cdots,\underline{w}^Q\}$.
\item The   geometric type $T=\{n,\{h_i,v_i\}_{i=1}^n,\Phi_t:=(\rho_T,\epsilon_T)\}$ and its associated incidence matrix $A:=A(T)$, with the assumption that $A$  is binary.
\item A generalized pseudo-Anosov homeomorphism $f:S \rightarrow S$.
that have a geometric Markov partition  $\cR=\{R_i\}_{i=1}^n$  with  geometric type $T$.
\end{itemize}
\end{conv}

\subsection{The Markov partition $\cR_{S(\cW)}$.}

Lemma \ref{Lemm: who the intervals intersect } holds for every code $\underline{w}^q\in \cW$, and we can define the set of intervals that are contained into the same rectangle $R_i$ of $\cR$.

\begin{defi}\label{Defi: Stable intervals inside R-i, cW}

For every $i\in \{1,\cdots, n\}$ the set of stable intervals  determined by the family $\cW$ that are contained within $R_i$ is given by:

	\begin{equation}\label{Equa: cW horizontal sub}
	\cI(i,\cW):=\cup_{q=1}^{Q}\cI(i,\underline{w}^q).
	\end{equation}
	
\end{defi}

\begin{lemm}\label{Lemm: rectangles in R S(cW)}
	Let $\tilde{R}$ be the closure of a connected component of $\overset{o}{R_i}\setminus \cup \cI(i,\cW)$. Then $\tilde{R}$ is a horizontal sub-rectangle of $R_i$.
\end{lemm}

\begin{proof}
The argument is the same as in Lemma \ref{Lemm: Stable intervals of codes}. Cutting $\overset{o}{R_i}$ along the finite family of stable intervals in $\cI(i,\cW)$ produces connected components whose closures are horizontal sub-rectangles of $R_i$.

\end{proof}

\begin{defi}\label{Lemm: cR-S(cW) is markov part}

The family, which includes all the rectangles as described in Lemma \ref{Lemm: rectangles in R S(cW)}, corresponding to all the rectangles $R_i$ in $\cR$, is denoted as: $\cR_{S(\cW)}$.
\end{defi}

\begin{lemm}\label{Lemm: Markov partition cW }
	The family of rectangles $\cR_{S(\cW)}$  is a Markov partition for $f$.
\end{lemm}

\begin{proof}
The interiors of two distinct horizontal sub-rectangles within $R_i$ do not intersect, and a horizontal sub-rectangle from $R_i$ does not overlap with the interior of a horizontal sub-rectangle from $R_j$ if $j\neq i$. Consequently, the interiors of two different rectangles in $\cR_{\cW}$ are disjoint.

The union $\cup \cR_{S(\cW)}$ coincides with the union of the rectangles in $\cR$, but $\cR$ is  Markov partition for $f$, therefore  $\cup \cR_{S(\cW)}=\cup \cR =S$.

The stable boundary of $\cR_{S(\cW)}$ consists of two types of stable intervals: those that are part of the stable boundary of $\cR$ and whose image is contained in $\partial^s\cR \subset \partial^s \cR_{S(\cW)}$, and those of the form $I_{t,\underline{w}^q}$ for some $\underline{w}^q\in \cW$ that, as stated in Lemma \ref{Lemm: image of stable intervals}: $f(I_{t,\underline{w}^q}) \subset I_{t+1,\underline{w}^q}\subset \partial^s\cR_{S(\cW)}$. Therefore, $\partial^s\cR_{S(\cW)}$ is  $f$-invariant.

Since the unstable boundary of $\cR_{S(\cW)}$ coincides with the unstable boundary of $\cR$, $\partial^u\cR_{S(\cW)}$ is $f^{-1}$-invariant.

This proves that $\cR_{S(\cW)}$ is a Markov partition of $f$.
\end{proof}

\subsection{The geometrization of $\cR_{S(\cW)}$}

Let's choose an  orientation for every rectangle in $\cR_{S(\cW)}$ and label its elements.

\begin{defi}\label{Defi: Orientation in R-S(cW)}
	If the rectangle $\tilde{R} \in \cR_{S(\cW)}$ is contained within the rectangle $R_i \in \cR$, we assign to $\tilde{R}$ the vertical and horizontal orientations that are consistent with the respective vertical and horizontal orientations of $R_i$.
\end{defi}

\begin{defi}\label{Defi: Label of R-S(cW)}
	We label the rectangles in $\mathcal{R}_{S(\cW)}$ using the \emph{lexicographic order}. To do this, we begin by numbering the horizontal sub-rectangles of $R_i$ that belong to $\mathcal{R}_{S(\cW)}$ from the bottom to the top, denoting them as $\{\tilde{R}_{i,s}\}_{s=1}^{N_i}$, where $N_i \geq 1$.
	Each rectangle $\tilde{R}_{i,s} \in \cR_{S(\underline{w})}$, which is contained within $R_i$ at vertical position $s$, is labeled as $\tilde{R}:=\tilde{R}_{r}$, where the number $r$ is determined by the formula:
	$$
	r=\tilde{r}(i,s):=\sum_{i'<i}^{i} N_{i'} + s.
	$$
\end{defi}

\begin{defi}\label{Defi: Geometric s-boundary refinament cW}
	The Markov partition $\cR_{S(\cW)}$ with the geometrization given by Definitions \ref{Defi: Orientation in R-S(cW)} and \ref{Defi: Label of R-S(cW)} is called the $s$-\emph{boundary refinement} of $\cR$ with respect to the family $\cW$. Its geometric type is denoted by:
	
	\begin{equation}
	T_{S(\cW)}=\{N,\{H_r,V_r\}_{r=1}^N,\Phi_{S(\cW)}:=(\rho_{S(\cW)},\epsilon_{S(\cW)})\}.
	\end{equation}
	
	We give the following notation to the rectangles of the partition:
	\begin{equation}
	\cR_{S(\cW)}:=\{\tilde{R}_r\}_{r=1}^N
	\end{equation}
	
\end{defi}

\subsection{The number of rectangles in $\cR_{S(\cW)}$}
 
 Let $\underline{w}^q\in \cW$ be any code. We'll denote $\textbf{Per}(\underline{w}^q)$ as the period of $\underline{w}^q$, and let $w^q_t$ represent the term in position $t$ of the code $\underline{w}^q$.
 
\begin{defi}\label{Defi: cO(i,cW) and O(i,cW)}
For every $i\in \{1,\cdots, n\}$ let:
	\begin{equation}
\cO(i,\cW)=\cup_{q=1}^Q	\cO(i,\underline{w}^q)=\{(t,\underline{w}^q): t \in \{0, \ldots, \textbf{Per}(\underline{w}^q)-1\} \text{ and } w^q_t=i \}.
	\end{equation}
the indexes of the stable intervals determined by $\cW$ that are contained into $R_i$. Let   $O(i,\cW)$ be its cardinality.
\end{defi}

The next lemma is immediate.

\begin{lemm}\label{Lemm: Number  of rectangles in a rec of the s refinaent, cW}
	For all $i \in \{1, \ldots, n\}$, the number of rectangles in the Markov partition $\cR_{S(\cW)}$ that are contained in $R_i$ is equal to $O(i,\cW)  + 1$
\end{lemm}

\begin{coro}\label{Coro: Number N in the type T-S(cW)}
	The number $N$ in the geometric type $T_{S(\cW)}$ is equal to the number of rectangles in the $s$-boundary refinement respect to $\cW$ and is determined by the following formula:
	\begin{equation}\label{Equa: Number N,cW }
	N=\sum_{i=1}^{n} O(i,\cW)+1.
	\end{equation}
\end{coro}

\subsection{The vertical sub-rectangles in $T_{S(\cW)}$.}

\begin{defi}\label{Defi: Vertical sub R-S(cW)}
	Let $\tilde{R}_r$ be a rectangle in the geometric Markov partition $\cR_{S(\cW)}$ and suppose that $r = \tilde{r}(i, s)$. Let $V^i_l$ be a vertical sub-rectangle of the Markov partition $(f, \cR)$ contained in $R_i$. Define:
	$$
	\tilde{V}^{r}_l:= \overline{\overset{o}{\tilde{R}_r} \cap \overset{o}{V^i_l}}.
	$$
	as the closure of the intersection between the interior of  $\tilde{R}_r$ and the interior of  $V^i_l$.
	
\end{defi}

The proof of the following Lemma is  completely analogous to the such presented  in Lemma \ref{Lemm: Caracterization of vertica sub of R-S(w)}.

\begin{lemm}\label{Lemm: Caracterization of vertica sub of R-S(cW)}
	The set $\tilde{V}^{r}_l$ is a unique horizontal sub-rectangle of $\tilde{R}_r$
\end{lemm}

\begin{rema}\label{Rema: Order of vertical coherent with R_i}
The labeling of the vertical sub-rectangles of $\tilde{R}_r$: 
	$$
\{\tilde{V}^{r}_l\}_{l=1}^{v_i},
$$
is coherent with the horizontal orientation of $R_i$.
\end{rema}

Taking in account the previous Remark \ref{Rema: Order of vertical coherent with R_i}, the argument used to prove Lemma \ref{Lemm: unique sub (i,j) for a r} can be applied to demonstrate the following generalized result:

\begin{lemm}\label{Lemm: unique sub (i,j) for a r,cW}
	Let $\tilde{R}_r\in \cR_{S(\cW)}$ with $r=\tilde{r}(i,s)$. Let $\tilde{H} \subset \tilde{R}_r$ be a horizontal sub-rectangle of $\cR_{S(\cW)}$. Then there is a unique pair $(i,j)\in \cH(T)$ such that $\overset{o}{H} \subset \overset{o}{H^i_{j}}$.
\end{lemm}

\begin{defi}\label{Defi: the unique pair (i,j-H),cW}
	Let $\tilde{R}_r\in \cR_{S_(\cW)}$ with $r=\tilde{r}(i,s)$. Let $\tilde{H} \subset \tilde{R}_r$ be a horizontal sub-rectangle of $\cR_{S_(\cW)}$, we call $(i,j_{\tilde{H} })$ to the only pair that was determine in Lemma \ref{Lemm: unique sub (i,j) for a r,cW}.
\end{defi}

\begin{lemm}\label{Lemma: Vertical sub R-ps, cW}
	The rectangles $\{\tilde{V}^{r}_l\}_{l=1}^{v_i}$  described in Definition \ref{Defi: Vertical sub R-S(w)} are the vertical sub-rectangles of the Markov partition $(f,\cR_{S(\cW)})$ contained in $\tilde{R}_r$.
\end{lemm}

\begin{proof}

Consider a single rectangle $\tilde{V}^r_{l}$. Following the reasoning from Lemma \ref{Lemma: Vertical sub R-ps}, we can observe that $f^{-1}(\tilde{V}^r_{l})$ is a subset of $f^{-1}(V^k_l)$, and specifically, it is contained in some horizontal sub-rectangle of $\cR$, $H^i_{j}$. The preimage $f^{-1}(\tilde{V}^r_{l})$ is a horizontal sub-rectangle of $H^i_j$ whose interior don't intersect $\partial^s \cR_{S(\cW)}$, then  and it's exclusively contained within a unique rectangle $\tilde{R}_r$. Even more,  the stable boundary of $f^{-1}(\tilde{V}^r_{l})$ has  image contained within the stable boundary of the Markov partition $\cR_{S(\cW)}$. Therefore, $f^{-1}(\tilde{V}^r_{l})$ is a horizontal sub-rectangle of $\cR_{S(\cW)}$, and then $\tilde{V}^{r}_l$ is a vertical sub-rectangle of the Markov partition $\cR_{S(\cW)}$.

Conversely, if $\tilde{V}$ is a vertical sub-rectangle of $\cR_{S(\cW)}$, then $f^{-1}(\tilde{V})=\tilde{H}$ is a horizontal sub-rectangle of $\cR_{S(\cW)}$. Let $(i,j_{\tilde{H}})$ be the indexes give in Definition \ref{Defi: the unique pair (i,j-H),cW}. Thus, $\tilde{V}$ is a horizontal sub-rectangle of $V^k_l=f(H^i_{j_{\tilde{H}}})$, and its interior does not intersect the stable boundaries of $\cR_{S(\cW)}$. However, its stable boundaries are contained in $\partial^s \cR_{S(\cW)}$. This implies that $\tilde{V}\subset V^k_l\cap \tilde{R}_{r}$ for some $r\in \{1,\cdots, n\}$, and it turns out to be one of the rectangles $\tilde{V}^{r}_l$ given by $\tilde{V}^{r}_l:= \overline{\overset{o}{\tilde{R}_r} \cap \overset{o}{V^k_l}}$.

\end{proof}

This result  have the next immediate corollary.

\begin{coro}\label{Coro: number of vertical sub rec, cW}
	If $r=\tilde{r}(i,s)$  then $V_r=v_i$.
\end{coro}

The following result is a generalization of Proposition \ref{Prop: vertical possition is delimitate.}, and it is proven in a similar manner. The primary element in this proof is the consistency of the horizontal orientation of $R_i$ and the labeling of the vertical sub-rectangles $\{\tilde{V}^{r}_l\}_{l=1}^{v_i}$ within $\tilde{R}_r\subset R_i$. 

\begin{coro}\label{Coro: vertical possition is delimitate,cW}
	Let $\tilde{H}\subset \tilde{R}_r$ any horizontal sub-rectangle of $\cR_{S(\cW}$,  assume that $\tilde{H}\subset H^i_j$ and $f(H^i_j)=V^k_l$. If $f(H)=\tilde{V}^{r'}_{l'}$ then $l=l'$.
\end{coro}

\subsection{The stable boundary of a rectangle $\tilde{R}_r \in \cR_{S(\cW)}$}

Almost any result in this sub-section can be proved in the same manner as the respective result was proved for only one code. For example, the next lemma is proved with the same strategy developed in Lemma  \ref{Lemm: diferentation moment of the order}. The number $M$ is the minimum natural number such that $w^1_{t+1+M}\neq w^2_{t_2+M}$, and all the other properties follow from this number.

\begin{lemm}\label{Lemm: diferentation moment of the order,cW}

	Suppose that $\underline{w}^1$ and $\underline{w}^2$ are two codes in $\cW$ with periods $P_1$ and $P_2$, respectively. Let $t_1$ be a number in $\{0,\cdots, P_1-1\}$ and $t_2$ be a number in $\{0,\cdots, P_2-1\}$ such that $\sigma^{t_1}(\underline{w}^1) \neq \sigma^{t_2}(\underline{w}^2)$, but both $w^1_{t_1}$ and $w^2_{t_2}$ are equal to  $i\in\{1,\cdots,n\}$. Consider the two elements in $\cO(i,\cW)$ determined by these iterations: $(t_1,\underline{w}^1)$ and $(t_2,\underline{w}^2).$ In this way:

	\begin{itemize}
		\item[i)]  There exists an integer $M\in\{1,\cdots, P-1\}$ such that $w^1_{t_1+M} \neq w^2_{t_2+M}$, but for all $0 \leq m \leq M-1$, the numbers $w^1_{t_1+m}$ and $w^2_{t_2+m}$ are equal.
		
		\item[ii)] 	If $M=1$, the indexes $j_{t_1,\underline{w}^1}$ and $j_{t_2,\underline{w}^2}$ (Definition \ref{Defi: the unique pair (i,j-H)}) are distinct. However, when $M\geq 2$, for every $m$ in the set $\{0,\cdots,M-2\}$, the indices $j_{t_1+m,\underline{w}^1}$ and $j_{t_2+m,\underline{w}^1}$ (see Definition \ref{Defi: the unique pair (i,j-H)}) are equal. Thus, for all $m\in\{0,\cdots,M-2\}$:
		$$
		\epsilon_T(w^1_{t_1+m},j_{t_1 +m, \underline{w}^1})=\epsilon_T(w^2_{t_2+m},j_{t_2 +m, \underline{w}^2}).
		$$
		
		\item[iii)]  The  indices  $j_{t_1+M-1,\underline{w}^1}$ and $j_{t_2+M-1,\underline{w}^2}$ in $\{1,\cdots, h_{w^1_{t_1 +M-1}}=h_{w^2_{t_2 +M-1}}\}$  are different.
	\end{itemize}
	
\end{lemm}

This permits us to define the interchange function as  in Definition \ref{Defi: Intetchage order funtion} but this time respect two non necessarily equal  codes in $\cW$.

\begin{defi}\label{Defi: Intetchage order funtion,cW}
	
	Suppose that $\underline{w}^1$ and $\underline{w}^2$ are two codes in $\cW$ with periods $P_1$ and $P_2$, respectively. Let $t_1$ be a number in $\{0,\cdots, P_1-1\}$ and $t_2$ be a number in $\{0,\cdots, P_2-1\}$ such that $\sigma^{t_1}(\underline{w}^1) \neq \sigma^{t_2}(\underline{w}^2)$ but $w^1_{t_1} = w^2_{t_2}=i$. Let $M$ be the number determined in Lemma \ref{Lemm: diferentation moment of the order,cW}. We define the \emph{interchange order} between $(t_1, \underline{w}^1)$ and $(t_2, \underline{w}^2)$ as follows:
	
	\begin{itemize}
		\item If $M=1$ then $\delta((t_1,\underline{w}^1),(t_2,\underline{w}^2))=1$.
				\item If $M>1$ then
		$$
	\delta((t_1,\underline{w}^1),(t_2,\underline{w}^2))=\prod_{m=0}^{M-2}\epsilon_T(w^1_{t_1+m},j_{t_1 +m, \underline{w}^1})=\prod_{m=0}^{M-2}\epsilon_T(w^2_{t_2+m},j_{t_2 +m, \underline{w}^2}).
		$$
	\end{itemize}
\end{defi}

Lemma \ref{Lemm: Comparacion de ordenes} admits the following generalization whose proof is essentially the same.

\begin{lemm}\label{Lemm: Comparacion de ordenes, cW}
	Let $I_{t_1,\underline{w}^1}$ and $I_{t_2,\underline{w}^2}$ be two stable segments such that $w^1_{t_1}=w^2_{t_2}=i$. Then, $I_{t_1,\underline{w}^1} < I_{t_2,\underline{w}^2}$ with respect to the vertical orientation of $R_i$ if and only if one of the following situations occurs:
	
	\begin{itemize}
		\item $j_{t_1+M-1,\underline{w}^1}<j_{t_2+M-1,\underline{w}^2}$ and $\delta((t_1,\underline{w}^1)(t_2,\underline{w}^2))=1$, or
		\item $j_{t_1+M-1,\underline{w}^1}>j_{t_2+M-1,\underline{w}^2}$ and $\delta((t_1,\underline{w}^1)(t_2,\underline{w}^2))=-1$
	\end{itemize}
\end{lemm}

Now, we can establish a formal order for the set of stable segments in $R_i$ that are part of the stable boundary of the rectangles in $\cR_{S(\cW)}$ contained within $R_i$.

\begin{defi}\label{Lemm: Boundaries code,CW}
	Let $R_i\in\cR$ be any rectangle, denote $I_{i,-1}:=\partial^s_{-1} R_i$ as the lower boundary of $R_i$, and $I_{i,+1}:=\partial^s_{+1} R_i$ as the upper boundary. Define
	$$
	\underline{s}:\cO(i,\cW)\cup \{(i,+1), (i,-1)\} \rightarrow \{0,1,\cdots,O(i,\cW),O(i,\cW)+1\}.
	$$
	as the unique function such that:
	\begin{itemize}
		\item $\underline{s}(t_1,\underline{w}^1)< \underline{s}(t_2,\underline{w}^2)$ if and only if  $I_{t_1,\underline{w}^1}< I_{t_1,\underline{w}^2}$. 
		\item It assigns $0$ to $(i,-1)$ and  $O(i,\cW)+1$ to  $(i,+1)$.
	\end{itemize}

\end{defi}

Lemma \ref{Lemm: Determine boundaris of R-r, cW} is a generalization of \ref{Lemm: Determine boundaris of R-r} , and Lemma \ref{Lemm: imge of boundaries,cW} is a generalization of  \ref{Lemm: imge of boundaries}. Their proofs are exactly the same as the results they generalize.

\begin{lemm}\label{Lemm: Determine boundaris of R-r, cW}
	Let $\tilde{R}_r$ be a rectangle in the geometric Markov partition $\cR_{S(\cW)}$ with $r=\tilde{r}(i,s)$. Then, the stable boundary of $\tilde{R}_r$ consists of two stable segments of $R_i$ determined as follows:
	
	\begin{itemize}
		\item[i)] If $s=1$, the lower boundary of $\tilde{R}_r$ is $I_{i,-1}$ and if  $O(i,\cW)=0$,  its upper boundary is $I_{i,+1}$. However, if $O(i,\cW)>0$, there exists a unique $(t,\underline{w}^{q})\in \cO(i,\cW)$ such that $\underline{s}(t,\underline{w}^q)=1$, and the upper boundary of $\tilde{R}_r$ is the stable segment $I_{t,\underline{w}^q}$.
		
		\item[ii)] If $O(i,\cW)>1$, $s \neq 1$ and $s\neq O(i,\cW)+1$, there are unique $(t_1,\underline{w}^1), (t_2,\underline{w}^2)\in \cO(i,\cW)$ such that: $\underline{s}(t_1,\underline{w}^1)=s-1$, $\underline{s}(t_2,\underline{w}^2)=s$. Furthermore, the lower boundary of $\tilde{R}_r$ is the stable segment $I_{t_1,\underline{w}^1}$, and the upper boundary of $\tilde{R}_r$ is the stable segment $I_{t_2,\underline{w}^2}$.
		
		\item[iii)] If $O(i,\cW)>0$ and $s=O(i,\cW)+1$, then the upper boundary of $\tilde{R}_r$ is the stable segment $I_{i,+1}$. Additionally, there exists a unique $(t,\underline{w}^q)\in \cO(i,\cW)$ such that $\underline{s}(t,\underline{w}^q)= O(i,\cW)$, and the inferior boundary of $\tilde{R}_r$ is the stable segment $I_{t,\underline{w}^q}$.
	\end{itemize}
	
\end{lemm}

\begin{lemm}\label{Lemm: imge of boundaries,cW}
	
	Let $\tilde{R}_r$ with $r=\tilde{r}(i,s)$ be a rectangle in the geometric Markov partition $\cR_{S(\cW)}$. Then, the stable boundary components of $\tilde{R}_r$ have  images under $f$ determined in the following manner:
	
	\begin{itemize}
		\item[i)] If the lower boundary of $\tilde{R}_r$ is $I_{i,-1}$ and $\Phi_T(i,1)=(k,l,\epsilon)$, then: if $\epsilon=1$,   $f(I_{i,-1}) \subset I_{k,-1}$, but if $\epsilon=-1$, $f(I_{i,-1}) \subset I_{k,+1}$.

		\item[ii)] If one boundary component of $\tilde{R}_r$ is the stable segment $I_{t,\underline{w}}$, then $f(I_{t,\underline{w}})\subset I_{t+1,\underline{w}}$.
		
		\item[iii)]  If $I_{i,+1}$ is the upper boundary of $\tilde{R}_r$ and $\Phi_T(i,h_i)=(k,l,\epsilon)$, then:  if $\epsilon=1$,   $f(I_{i,+1}) \subset I_{k,+1}$, but if $\epsilon=-1$, $f(I_{i,+1}) \subset I_{k,-1}$.
	\end{itemize}

\end{lemm}

\subsection{The relative position of $H^i_j$ with respect to $\tilde{R}_r$}

Like we did before, we are going to determine the set of horizontal sub-rectangles of $\cR$ that intersect a certain rectangle $R_r$ from $\cR_{S(\cW)}$.

\begin{defi}\label{Defi cH(r) set, cW}
	Let $\tilde{R}_r\in \cR_{S(\cW)}$ with $r=\tilde{r}(i,s)$. We define:
	$$
	\cH(r):=\{(i,j)\in \cH(T): \overset{o}{H^i_{j}}\cap \overset{o}{\tilde{R}_r}\neq \emptyset\}.
	$$
\end{defi}

The fundamental tool to prove the next result is Lemma \ref{Lemm: unique hriwontal for every prjection} , which we presented earlier and is valid for any element in $\cO(i,\cW)$. We'll leave the proof out as the philosophy is the same as our previous Lemma \ref{Lemm: determination cH(r)}.

\begin{lemm}\label{Lemm: determination cH(r), cW}
	Let $\tilde{R}_r$ be a rectangle in the geometric Markov partition $\cR_{S(\underline{w})}$ with $r=\tilde{r}(i,s)$. Then the set $\cH(r)$ is determined by one of the three following formulas:
	
	\begin{itemize}
		\item[i)] If $O(i,\cW)=0$, then:  $\cH(r)=\{(i,j)\in \cH(T)\}$.

		\item[ii)]  If $O(i,\cW)>0$, $s \neq 1$ and $s \neq O(i,\cW) + 1$, let $I_{t_1,\underline{w}^1}$ be the inferior boundary of $\tilde{R}_r$, and let $I_{t_2,\underline{w}^2}$ be the upper boundary of $\tilde{R}_r$. Let $j_{t_1,\underline{w}^1}$ and $j_{t_1,\underline{w}^2}$ be the indices determined in Lemma \ref{Lemm: unique hriwontal for every prjection}. In this situation:
		 $$
		\cH(r)=\{(i,j)\in \cH(T): j_{t_1,\underline{w}^1} \leq j \leq j_{t_2,\underline{w}^2} \}.
		$$
		
		\item[iii)] If $O(i,\cW)>0$ and $s=O(i,\cW)+1$, let $I_{t,\underline{w}^1}$ be the inferior boundary of $\tilde{R}_r$ and  $I_{i,+1}$ the upper boundary of $\tilde{R}_r$. Let $j_{t,\underline{w}^1}$ be the index determined in \ref{Lemm: unique hriwontal for every prjection}. In this manner, 
		$$
		\cH(r)=\{(i,j)\in \cH(T): j_{t,\underline{w}^1} \leq j \leq h_i \}.
		$$
		\item[iv)] If $O(i,\cW)>0$ and $s=1$, let $I_{t,\underline{w}^1}$ be the upper boundary of $\tilde{R}_r$, let $I_{i,-1}$ be its inferior boundary and let $j_{t,\underline{w}^1}$ be the index determined in \ref{Lemm: unique hriwontal for every prjection}. In this manner:
		$$
		\cH(r)=\{(i,j)\in \cH(T): 1 \leq j \leq j_{t,\underline{w}^1} \}.
		$$
		
	\end{itemize}
	
\end{lemm}

Let $\tilde{R}_r\in \cR_{S(\cW)}$ be a rectangle with $r=\tilde{r}(i,s)$. For each $(i,j)\in \cH(r)$, there are three possible situations regarding the placement of $\tilde{R}_r$ with respect to the sub-rectangle $H^i_j$ of $R_i\in \cR$:

\begin{itemize}
	\item[i)] $H^i_j\subset \tilde{R}_r$ 
	\item[ii)]$\tilde{R}_r \subset H^i_j\setminus\partial^s H^i_j$
	\item[iii)]  $H^i_j$ contains only one horizontal boundary component of $\tilde{R}_r$.
\end{itemize}

As done in Proposition \ref{Prop: determination case of ch(r) inside R-r}, we need to provide a criterion to determine which of the three situations previously mentioned we are dealing with, in fact the criterion is the same, but we reproduce here just for completeness.

\begin{prop}\label{Prop: determination case of ch(r) inside R-r, cW}
	Let $\tilde{R}_r\in \cR_{S(\cW)}$ be a rectangle with $r=\tilde{r}(i,s)$. 
	
	Assume that $\tilde{R}_r\neq R_i$ and $ \cH(r)=\{(i,j_1),\cdots, (i,j_K)\}$ with $j_k<j_{k+1}$. Let $(i,j_k)\in \cH(r)$ in this manner:
	
	\begin{itemize}
		\item[i)]  $\tilde{R}_r \subset H^i_{j_k}\setminus\partial^s H^i_{j_k}$ if and only if $\cH(r)=\{(i,j_k)\}$, i.e. $K=1$.
		
		\item[ii)]  $H^i_{j_k}\subset \tilde{R}_r$ if and only if $j_1<j_k<j_K$.
		
		\item[iii)]  The sub-rectangle $H^i_{j_k}$ contains only one horizontal boundary component of $\tilde{R}_r$ if and only if $K> 1$, i. e. the cardinality of $\cH(r)$ is bigger than $1$ ( $\sharp(\cH(r))>1$ ).  In this case, $H^i_{j_k}$ contains  the inferior boundary of $\tilde{R}_r$ if $k=1$ and the upper boundary of $\tilde{R}_r$ if $k=K$.
	\end{itemize}
	
	In case that  $\tilde{R}_r= R_i$, for all $(i,j)\in \cH(r)$,   $H^i_j\subset \tilde{R}_r$.
	
\end{prop}

\subsection{The number of horizontal sub-rectangles of $\cR_{S(\cW)}$}
We are going to count how many rectangles in the Markov partition $\cR_{S(\cW)}$ are contained in the intersection $\tilde{R}_r\cap H^i_j$ and describe the way that $f$ send such horizontal sub-rectangles into vertical sub-rectangles. Demonstrations of all the following Lemmas is a compilation of the arguments developed in the previous section, and we believe they don't require further explanation. However, we enunciate them and give some hits about their proof for completeness.

\subsubsection{Fist case: $\tilde{R}_r \subset H^i_{j}\setminus\partial^s H^i_{j}$ }

\begin{lemm}\label{Lemm: number horizontal sub case I, cW}
	Let $\tilde{R}_r\in \cR_{S(cW)}$, with $r=\tilde{r}(i,s)$. Suppose we have a pair of indices $(i,j)\in \cH(r)$ such that $\tilde{R}_r \subset H^i_j\setminus\partial^s H^i_j$. In this case $s\neq 1$ and $s\neq O(i,\cW)+1$ and we take $I_{t_1,\underline{w}^1}$ as the lower boundary of $\tilde{R}_r$ and $I_{t_2,\underline{w}^2}$ as the upper boundary of $\tilde{R}_r$.
	
	Assume that $\rho_T(i,j)=(k,l)$. Let $s^-=\underline{s}(t_1+1,\underline{w}^1)$ and $s^+=\underline{s}(t_2+1,\underline{w}^1)$. Then,  $w^2_{t_1+1}=w^2_{t_2+1}=k$ and the number of horizontal sub-rectangles of $\cR_{S(\cW)}$ contained in the intersection $\tilde{R}_r\cap H^i_j$ is given by the formula:
	
	\begin{equation}\label{Equa: h (i,j,r) caso I, cW}
	\underline{h}(r,j)= \vert s^- - s^+\vert.
	\end{equation} 
	
\end{lemm}

\begin{defi}\label{Defi: order rectangles case I, cW}
	The horizontal sub-rectangles of $\tilde{R}_r$ as described in Lemma \ref{Lemm: number horizontal sub case I, cW} are labeled from the bottom to the top with respect to the orientation of $\tilde{R}_r$ as:
	$$
	\{\tilde{H}^r_{j,J}\}_{J=1}^{\underline{h}(r,j)}.
	$$
\end{defi}

\begin{lemm} \label{Lemm: Permutation rho caso I, cW}
	With the hypothesis of Lemma \ref{Lemm: number horizontal sub case I, cW}, assume that $\Phi_T(i,j)=(k,l,\epsilon)$ and $f(\tilde{H}^r_{j,J})=\tilde{V}^{r'}_{l'}$. Then:
	\begin{itemize}
		\item[i)] 	If $\epsilon=\epsilon(i,j) = 1$, then $s^- < s^+$, $r'=\tilde{r}(k, s^- + J)$ and $l=l'$.
	\item[ii)] If $\epsilon=\epsilon(i,j) = -1$, then $s^- > s^+$, $r'=\tilde{r}(k, s^- - J+1)$ and $l=l'$.
	\end{itemize}
\end{lemm}

\begin{rema}
	The only modification in the statement of Lemmas \ref{Lemm: number horizontal sub case I, cW} and \ref{Lemm: Permutation rho caso I, cW}  respect to such in Lemma \ref{Lemm: number horizontal sub case II} and \ref{Lemm: Permutation rho caso II} is that the numbers $s^-$ and $s^+$ in our previous section depend on the same code $\underline{w}$ in this case $s^-$ depends on $\underline{w}^1$ and $s^+$ depends on $\underline{w}^2$. But once we have fixed the values $s^-$ and $ s^+$ all the arguments in the proofs are the same.
\end{rema}

\begin{defi}\label{Defi: permutation caso I, cW}
	With the hypothesis of Lemma \ref{Lemm: number horizontal sub case I, cW}, assume that $\Phi_T(i,j)=(k,l,\epsilon)$. Then: 
	
	If $\epsilon(i,j) = 1$ we define:
	
	\begin{equation}\label{Equ: rho caso 1, +,cW}
\underline{\rho}_{(r,j)}(r,J)=(\tilde{r}(k, s^-+J), l  ).
	\end{equation}
	
	and if $\epsilon(i,j) = -1$:
	\begin{equation}\label{Equ: rho caso 1, -, cW}
	\underline{\rho}_{(r,j)}(r,J)=(\tilde{r}(k, s^+  -J+1) , l    ).
	\end{equation}
\end{defi}

\subsubsection{Second case: $H^i_{j}\subset \tilde{R}_r$ }

\begin{lemm}\label{Lemm: number horizontal sub case II, cW}
	
	Let $\tilde{R}_r\in \cR_{S(\cW)}$, with $r=\tilde{r}(i,s)$. Let $(i,j)\in \cH(r)$ be a pair of indices such that $H^i_j\subset \tilde{R}_r$, assume that $\rho_T(i,j)=(k,l)$, then the number of horizontal sub-rectangles of $\cR_{S(\cW)}$ that are contained in the intersection $\tilde{R}_r\cap H^i_j$ is given by the formula:
	
	\begin{equation}\label{Equa: h (i,j,r) caso I,cW}
	\underline{h}(r,j)=O(k,\cW)+1,
	\end{equation} 
		
\end{lemm}

\begin{proof}
The image of $f^(H^i_j\cap \tilde{R}_r)=f(H^i_j)=V^k_l$ and $V^k_l$ intersect all the rectangles of $\cR_{S(\cW)}$ that are contained in $R_k$, the total amount of them is $O(k,\cW)+1,$ and this is the number of horizontal sub-rectangles of $\cR_{S(\cW)}$ that are contained in $\tilde{R}_r\cap H^i_j$.
\end{proof}

\begin{defi}\label{Defi: order rectangles case II,cW}
	The horizontal sub-rectangles of $\tilde{R}_r$ described in Lemma \ref{Lemm: number horizontal sub case II, cW}  are labeled from the bottom to the top with respect to the orientation of $\tilde{R}_r$ as:
	$$
	\{\tilde{H}^r_{j,J}\}_{J=1}^{\underline{h}(r,j)}.
	$$
\end{defi}

Now we can determine the image of any horizontal sub-rectangle.

\begin{lemm}\label{Lemm: Permutation rho caso I,cW}
	Let $\tilde{R}_r\in \cR_{S(\cW)}$, with $r=\tilde{r}(i,s)$. Let $(i,j)\in \cH(r)$ be a pair of indices such that $H^i_j\subset \tilde{R}_r$, and let $\tilde{H}^r_(j,J)$ be  an horizontal sub-rectangle of $\tilde{R}_r\cap H^i_j$. Assume that $\rho_T(i,j)=(k,l)$,  and that $f(\tilde{H}^r_{j,J})=\tilde{V}^{r'}_{l'}$, then:	
	\begin{itemize}
		\item[i)] If $\epsilon(i,j)=1$, then $r'=\tilde{r}(k,J)$ and $l=l'$.
		
		\item[ii)]  If  $\epsilon_T(i,j)=-1$, then $r'=\tilde{r}(k,O(k,\cW)+1 -J+1)$ and $l=l'$.
	\end{itemize}
	
\end{lemm}

\begin{defi}\label{Defi: permutation caso II, cW}
	Let $\tilde{R}_r\in \cR_{S(\cW)}$, with $r=\tilde{r}(i,s)$. Let $(i,j)\in \cH(r)$ be a pair of indices such that $H^i_j\subset \tilde{R}_r$, and let $\tilde{H}^r_{j,J}$ be  an horizontal sub-rectangle of $\tilde{R}_r\cap H^i_j$. Assume that $\rho_T(i,j)=(k,l)$, then:  
	
	If  $\epsilon(i,j)=1$ we define:
	\begin{equation}\label{Equ: rho caso 2, +,cW}
\underline{\rho}_{(r,j)}(r,J)=(\tilde{r}(k,J),l),
	\end{equation}
	
	and  if $\epsilon(i,j)=-1$ 
	
	\begin{equation} \label{Equ: rho caso 2, -, cW}
\underline{\rho}_{(r,j)}(r,J)=(\tilde{r}(k,O(k,\cW)+1 -J+1),l).
	\end{equation}
\end{defi}

\subsubsection{Third case: $H^i_{j}$ contains only one horizontal boundary component of $\tilde{R}_r$ }

\begin{lemm}\label{Lemm: number horizontal sub case III, upper, cW}
	
	Consider $\tilde{R}_r\in \cR_{S(\cW)}$, where $r=\tilde{r}(i,s)$, and assume that $O(i,\cW)>0$. 	Take a pair of indexes $(i,j)\in \cH(r)$ such that $H^i_j$ exclusively contains the upper boundary of $\tilde{R}_r$. Let $I_{t,\underline{w}^1}$ (with $\underline{w}^1\in \cW$) be the upper boundary of $\tilde{R}_r$. Let $s^+=\underline{s}(t+1,\underline{w}^1)$ and assume that $\rho_T(i,j)=(k,l)$. Then, the number of horizontal sub-rectangles of  $\cR_{S(\cW)}$ that are contained in  $\tilde{R}_r\cap H^i_j$ can be computed using one of the following formulas:
	
	\begin{itemize}
		\item If $\epsilon_T(i,j)=1$:
		
		\begin{equation}\label{Equ: h(i,j,r), caso 3, +, up,cW}
		\underline{h}(r,j)=s^+
		\end{equation} 
		
		\item If $\epsilon_T(i,j)=-1$
		\begin{equation}\label{Equ: h(i,j,r), caso 3, -,up,cW}
			\underline{h}(r,j)=O(k,\cW)+1-s^+ .
		\end{equation} 
	\end{itemize}

\end{lemm}

\begin{proof}
	In fact the proof is exactly the same as  Lemma \ref{Lemm: number horizontal sub case III, upper} the only difference is that function $\underline{s}$ now take in count all the rectangles determined by the family $\cW$, when previously only depended on a single code $\underline{w}$.
	
\end{proof}

\begin{lemm}\label{Lemm: number horizontal sub case III, lower, cW}
	
	Consider $\tilde{R}_r\in \cR_{S(\cW)}$, where $r=\tilde{r}(i,s)$, and assume that $O(i,\cW)>0$.	Take a pair of indexes $(i,j)\in \cH(r)$ such that  $H^i_j$ just contains the lower boundary of $\tilde{R}_r$. Let $I_{t,\underline{w}^1}$ (with $\underline{w}^1\in \cW$ ) be the lower boundary of $\tilde{R}_r$, and denote $s^+=\underline{s}(t+1,\underline{w}^1)$. In this scenario, the number of horizontal sub-rectangles of  $\cR_{S(\cW)}$ that are contained in  $\tilde{R}_r\cap H^i_j$ can be determined using one of the following formulas:
	
	\begin{itemize}
		\item If $\epsilon_T(i,j)=1$:
		
		\begin{equation}\label{Equ: h(i,j,r), caso 3, 1, low cW}
	\underline{h}(r,j)=O(k,\cW)+1-s^+
		\end{equation} 
		
		\item If $\epsilon_T(i,j)=-1$
		\begin{equation}\label{Equ: h(i,j,r), caso 3, -1, low cW}
		\underline{h}(r,j)= s^+ .
		\end{equation} 
	\end{itemize}
\end{lemm}

\begin{defi}\label{Defi: order rectangles case III,cW}
	The horizontal sub-rectangles of $\tilde{R}_r$ described in Lemma \ref{Lemm: number horizontal sub case III, upper, cW}  and Lemma \ref{Lemm: number horizontal sub case III, lower, cW}  are labeled from the bottom to the top with the vertical orientation of $\tilde{R}_r$ as:
	$$
	\{\tilde{H}^r_{j,J}\}_{J=1}^{\underline{h}(r,j)}.
	$$
\end{defi}


\begin{lemm} \label{Lemm: Permutation rho caso III, upper cW}
	With the hypothesis of Lemma \ref{Lemm: number horizontal sub case III, upper, cW}, suppose that  $H^i_j$  just contains the upper boundary of $\tilde{R}_r$ and assume that $\Phi_T(i,j)=(k,l,\epsilon_T(i,j))$ and $f(\tilde{H}^r_{j,J})=\tilde{V}^{r'}_{l'}$.	Then we have the following situations: 
	\begin{itemize}
		\item[i)] 	If $\epsilon_T(i,j)=1$, then $r'=\tilde{r}(k,  J)$ and $l=l'$

		\item[ii)] If $\epsilon_T(i,j)=-1$, then $r'=\tilde{r}(k, O(k,\cW)+1-J+1)$ and $l=l'$.
	\end{itemize}	
\end{lemm}

\begin{lemm}\label{Lemm: Permutation rho caso III, lower cW}
With the hypothesis of Lemma \ref{Lemm: number horizontal sub case III, lower, cW}, suppose that  $H^i_j$  just contains the lower boundary of $\tilde{R}_r$ and assume that $\Phi_T(i,j)=(k,l,\epsilon_T(i,j))$ and $f(\tilde{H}^r_{j,J})=\tilde{V}^{r'}_{l'}$.	Then we have the following situations: 

\begin{itemize}
	\item[i)] 	If $\epsilon_T(i,j)=1$, then $r'=\tilde{r}(k, s^+ + J)$ and $l=l'$

	\item[ii)] If $\epsilon_T(i,j)=-1$, then $r'=\tilde{r}(k, s^+-J+1)$ and $l=l'$.
\end{itemize} 	
\end{lemm}


\begin{defi}\label{Defi: permutation caso III, up cW}
	With the hypothesis of Lemma  \ref{Lemm: number horizontal sub case III, upper, cW}, suppose that  $H^i_j$  just contains the upper boundary of $\tilde{R}_r$ and  assume that $\Phi_T(i,j)=(k,l,\epsilon)$. Then: 	
	
	\begin{itemize}
		\item[i)]  	If $\epsilon_T(i,j)=1$, we define:
		\begin{equation}\label{Equ: rho caso 3, +,cW, upper}
			\underline{\rho}_{(r,j)}(r,J)=(\tilde{r}(k,  J), l).
		\end{equation}
		
		\item[ii)] if $\epsilon_T(i,j)=-1$ we take:
		
		\begin{equation}\label{Equ: rho caso 3, -, cW, upper }
		\underline{\rho}_{(r,j)}(r,J)=(\tilde{r}(k, O(k,\cW)+1-J+1),l).
		\end{equation}
	\end{itemize}

\end{defi}

\begin{defi}\label{Defi: permutation caso III, low cW}
With the hypothesis of Lemma \ref{Lemm: number horizontal sub case III, lower, cW}, suppose that  $H^i_j$  just contains the lower boundary of $\tilde{R}_r$ and  assume that $\Phi_T(i,j)=(k,l,\epsilon)$. Then: 	

\begin{itemize}
	\item[i)]  	If $\epsilon_T(i,j)=1$, we define:
	\begin{equation}\label{Equ: rho caso 3, +,low,cW}
 		\underline{\rho}_{(r,j)}(r,J)=(\tilde{r}(k, s^+ + J), l).
	\end{equation}
	
	\item[ii)] if $\epsilon_T(i,j)=-1$ we take:
	
	\begin{equation}\label{Equ: rho caso 3, -,low cW}
	\underline{\rho}_{(r,j)}(r,J)=(\tilde{r}(k, s^+-J+1),l).
	\end{equation}
\end{itemize} 	
\end{defi}

\subsubsection{Collecting the information}

Now, we can collect the information from the three cases to compute the number of horizontal sub-rectangles in $\tilde{R}_r$.

\begin{coro}\label{Coro: Number of horizontal sub, cW}
	Let $\tilde{R}_r\in \cR_{S(\cW)}$ with $r=\tilde{r}(i,s)$, then the number $H_r$ of horizontal sub-rectangles of the Markov partition $(f,\cR_{S(\cW)})$ that are contained in $\tilde{R}_r$ is given by:
	\begin{equation}
	H_r=\sum_{\{j:(i,j)\in \cH(r)\}}\underline{h}(r,j).
	\end{equation}
\end{coro}

 \subsection{The function $\rho_{S(\cW)}$ in the geometric type $T_{S(\cW)}$}

\begin{defi}\label{Defi: horizontal sub in R-r, cW}
	The sub-rectangles of the geometric Markov partition $\cR_{S(\cW)}$ contained in $\tilde{R}_r$ are labeled as:
	$$
	\{H^r_{\underline{J}}\}_{\underline{J}=1}^{H_r}
	$$
	from the bottom to the top with respect to the vertical orientation of  $\tilde{R}_r$.
\end{defi}

\begin{lemm}\label{Lemm: parameter J, cW}
	For each $\underline{J}\in\{1,\cdots, H_r\}$, there exists a unique index $j_{(\underline{J})}$ such that $(i,j)\in\cH(r)$  and  $\tilde{H}^r_{\underline{J}}\subset H^i_j\cap \tilde{R}_r$. The index $j_{(\underline{J})}$  is determined by satisfy the next inequalities:
	
	\begin{equation}\label{Equa: Cut the index, cW}
	\sum_{\{j': (i,j')\in \cH(r)  \text{ and } j'<j_{(\underline{J})}\} }\underline{h}(j',r)< \underline{J} \leq  \sum_{\{j': (i,j')\in \cH(r)  \text{ and } j'\leq j_{(\underline{J})}\} }\underline{h}(j',r).
	\end{equation} 
	
\end{lemm}

\begin{defi}\label{Defi: Horizontal parameter, cW}
	Given $\underline{J}\in \{1,\cdots, H_r\}$ we call the index $j_{(\underline{J})}$ that was determined in Lemma \ref{Lemm: parameter J, cW} the \emph{horizontal parameter} of $\underline{J}$.
\end{defi}

\begin{defi}\label{Defi: relative index, cW }
	Let $r=\tilde{r}(i,s)$, $\underline{J}\in \{1,\cdots, H_r\}$, and let $j_{(\underline{J})}$ be the horizontal parameter of $\underline{J}$. The \emph{relative index} of $H^r_J$ inside of $\tilde{R}_r\cap H^i_{j_{(\underline{J}})}$ is given by:
	\begin{equation}\label{Equa: relative index, cW}
	J_{(\underline{J})}:=\underline{J}-\sum_{j'<j_{({\underline{J})}}}\underline{h}(r,j').
	\end{equation}
	
\end{defi}

\begin{defi}\label{Defi: relative position, cW }
	Let $\underline{J}\in \{1,\cdots, H_r\}$ the \emph{relative position} of $\underline{J}$ inside $\tilde{R}_r$ is the pair $(j_{\underline{J}},J_{(\underline{J})})$ of the horizontal parameter of $\underline{J}$ and the relative index of $H^r_J$ inside of $\tilde{R}_r\cap H^i_{j_{\underline{J}}}$
\end{defi}

\begin{lemm}\label{Lemm: relative position relation, cW}
	The rectangle $\tilde{H}^r_{j_{(\underline{J})},J_{(\underline{J})}}$ and the rectangle $\tilde{H}^r_{\underline{J}}$ are the same.
\end{lemm}


\begin{coro}\label{Lemm: The permutation, cW}
	Let $r=\tilde{r}(i,s)$ and $\underline{J}\in \{1,\cdots, H_r\}$. Let $(j_{(\underline{J})},J_{(\underline{J})})$ the relative position of $\underline{J}$ as was given in Definition \ref{Defi: relative position, cW }. Then the function $\rho_{S(\cW)}$ is given by the formula:
	\begin{equation}\label{Equa: The permutation, cW}
	\rho_{S(\cW)}(r,\underline{J}):=\underline{\rho}_{(r,j_{(\underline{J})})}(r,J_{(\underline{J})}).
	\end{equation}

\end{coro}

\subsection{The orientation $\epsilon_{S(\cW)}$}

\begin{lemm}\label{Lemm: The orientation, cW}
	Let $r=\tilde{r}(i,s)$ and $\underline{J}\in \{1,\cdots, H_r\}$. Let $j_{(\underline{J})}$ the horizontal parameter of $\underline{J}$ (Definition  \ref{Defi: Horizontal parameter, cW}). Then:
	\begin{equation}\label{Equa: permtation in the s refinament, cW}
\epsilon_{S(\underline{w})}(r,\underline{J}):=\epsilon_T(i,j_{(\underline{J})}).
	\end{equation}
\end{lemm}

\subsection{Conclusion}

\begin{coro}\label{Coro: Algoritm computation TS(w), cW}
	Let $T$ be a geometric type in the pseudo-Anosov class with a binary incidence matrix denoted as $A:=A(T)$. Let $\cW$ a family of periodic codes contained in $\Sigma_A$, there is algorithm to compute the   geometric type:
	
	\begin{equation*}
	T_{S(\cW)}:=\{N,\{H_r,V_r\}_{r=1}^N,\Phi_{S(\cW)}:=(\rho_{S(\cW)},\epsilon_{S(\cW)})\}.
	\end{equation*}
	
using  the given geometric type $T$ and the codes in  $\cW$. 
\end{coro}

\section{The $u$-boundary refinement}.

In this part, we are going to define the $u$-boundary refinement of a Markov partition along a family of periodic codes. This procedure involves cutting $\cR$ along the unstable segments that correspond to the orbit of any of the codes in $\cW$. In fact, we can define it using our previous $s$-boundary refinement applied to $T^{-1}$ and $f^{-1}$.

\begin{defi}\label{Defi: U boundary refinament}

Let $T$ be a geometric type in the pseudo-Anosov class with an incidence matrix $A := A(T)$ that is binary. Consider a generalized pseudo-Anosov homeomorphism $f: S \rightarrow S$ with a geometric Markov partition $\cR$ of geometric type $T$.
 Let 
 $$
 \cW = \{\underline{w}^1, \cdots, \underline{w}^Q\},
 $$
 be a family of periodic codes that are non-$u$-boundary. 
\begin{itemize}
\item Let $\cR_{U(\cW)}$  the $s$-boundary refinement of $\cR$ when $\cR$ is viewed as a Markov partition of $f^{-1}$. i.e $(\cR,f^{-1})$. Denote by $T^{-1}_{S(\cW)}$ the  geometric type of $(\cR_{U(\cW)},f^{-1})$.

\item Let $T_{U(\cW)}$ the geometric type of $(\cR_{U(\cW)},f)$.

\end{itemize}

In this case the geometry Markov partition $\cR_{U(\cW)}$ for $f$ called the $u$-boundary refinement of $\cR$ respect the family $\cW$.
\end{defi}

\begin{rema}\label{Rema: some properties of U-boundary refinament}

We can make the following direct observations:
	
\begin{enumerate}
\item If $\underline{w}$ is not a $u$-boundary point for $(\cR,f)$, the $\underline{w}$ it is not a $s$-boundary point for $(\cR,f^{-1})$, which justifies our assumption.

\item The geometric Markov partition $(\cR_{U(\cW)},f)$ has a geometric type $T_{U(\cW)}$, and this geometric type is the inverse of the geometric type of $(\cR_{U(\cW)},f^{-1})$, therefore:
$$
T_{U(\cW)} := (T^{-1}_{S(\cW)})^{-1}.
$$
and we can determine $T_{U(\cW)}$ in an algorithmic manner.

\item For every $\underline{w}\in \cW$, the point $\pi_f(\underline{w})$ (where $\pi_f$ is the projection with respect to $(\cR,f)$) is a $u$-boundary point of $\cR_{U(\cW)}$.

\item If $\underline{w}$ is a $u$-boundary point, clearly $\cR_{U(\underline{w})}=\cR$ and $T_{U(\underline{w})}=T$. This justifies taking any code in $\cW$ as a non $u$-boundary code.
\end{enumerate}
\end{rema}
\section{The corner refinement}

Let $T$ be a geometric type within the pseudo-Anosov class with binary incidence matrix $A(T)$. Consider a generalized pseudo-Anosov homeomorphism $f: S \to S$ with a geometric Markov partition $\cR$ of geometric type $T$. The projection $\pi_f: \Sigma_{A(T)} \to S$, as defined in \ref{Defi: projection pi}, depends on the specific geometric type $T$ of the Markov partition. We will relabel it as $\pi_T: \Sigma_{A(T)} \to S$ to enable easier comparison with another projection $\pi_{T'}$ associated with a different geometric type $T'$

For a given geometric type $T$, we have defined its $s$ and $u$ generating functions in \ref{Defi: s-boundary generating funtion}  and \ref{Defi: u-boundary generating funtion}, respectively. Iterating these functions generates codes in $\Sigma_{A(T)}$, which project to the stable and unstable boundary components of $\cR$. Furthermore, as described in Section \ref{Sub: symbolic dynamics of a Type}, using $T$,  it is feasible to determine the sets of periodic $s$-boundary codes of $T$:
$$
\textit{Per}(\underline{S(T)}):=\{\underline{w}\in \Sigma_{A(T)} : \underline{w} \text{ is periodic for $\sigma$ and } \pi_T(\underline{w})\in \partial^s \cR \}.
$$
and the sets of periodic $u$-boundary codes:
  $$
  \textit{Per}(\underline{(T)}):=\{\underline{w}\in \Sigma_{A(T)} : \underline{w} \text{ is periodic for $\sigma$ and } \pi_T(\underline{w})\in \partial^u \cR \}.
  $$
  Their union is the set of boundary periodic codes:
    $$
  \textit{Per}(\underline{B(T)}):=  \textit{Per}(S(T)) \cup  \textit{Per}(U(T)).
  $$
 
 In order to avoid to heavy notations we  relabel these sets:
  $$
\underline{S(T)}:=\textit{Per}(\underline{S(T)}), \, \, \underline{U(T)}:=\textit{Per}(\underline{U(T)}) \text{ and } \underline{B(T)}:=\textit{Per}(\underline{B(T)})  
  $$
  
 Finally a corner periodic code  is any code in the  intersection:
 $$
 \underline{C(T)}:=\underline{S(T)} \cap \underline{U(T)}.
 $$

Definition \ref{Defi: s,u-boundary periodic codes} reveals that $ \underline{B(T)}$ can be  described in a combinatorial manner using $T$, avoiding  the use of the Markov partition or the homeomorphism $f$. Based on the theory developed in Section \ref{Sub: symbolic dynamics of a Type}, we can draw the following conclusion, which we present in the form of a corollary:

\begin{coro}\label{Coro: algoritmic per codes}
Given a specific geometric type $T$, we can compute the sets: $\underline{(S(T)})$, $\underline{U(T)}$, $\underline{B(T)}$, and $\underline{C(T)}$, in terms of the geometric type $T$.
\end{coro}

\begin{defi}\label{Defi: Corner type}
A geometric type $T$ have the \emph{corner property} if every periodic boundary code is a  corner peridic code, i.e $\underline{B(T)}=\underline{S(T)}\cap \underline{U(T)}$.
\end{defi}

We fix the following notation:
\begin{itemize}
\item $P_s(f)$ is the set of $s$-boundary periodic points of $(f,\cR)$.
\item $P_u(f)$ is the set of $u$-boundary periodic points of $(f,\cR)$
\item  $P_u(f)=P_s(f)\cup P_u(f)$  is the set of periodic boundary  codes of $(f,\cR)$, and
\item $P_c(f)=P_s(f)\cap P_u(f)$ is the set of corner points of $\cR$.
\end{itemize}

\begin{defi}\label{Defi: Corner type partition }
	A geometric Markov partition $\cR$ of $f$ have the \emph{corner property} if every rectangle $R\in \cR$ that contains a boundary point $p\in P_b(f,\cR)$ have $p$ as a corner point.
\end{defi}

A corner point $p$ of a rectangle $R \in \cR$ is on the boundary of another rectangle $R' \in \cR$ but it is not necessarily a corner point of $R'$; it can be on the interior of the unstable boundary of $R'$, for example. In this situation, $\pi_T^{-1}(p) \subset \Sigma_{A(T)}$ contains a corner code and a $u$-boundary code that is not an $s$-boundary code. Therefore, $\underline{C(T)} \neq \underline{B(T)}$ and $T$  don't have the corner property. This discrepancy makes it necessary to introduce the following lemma

\begin{lemm}\label{Lemm: Corner prop equivalence}
The geometric Markov partition $\cR$ has the corner property if and only if its geometric type $T$ has the corner property.
\end{lemm}

\begin{proof}

If $\cR$ satisfies the corner property, and we have an $s$-boundary code $\underline{w} \in \underline{S(T)}$, then by  Lemma  \ref{Lemm: Projection Sigma S,U,I}, $\pi_T(\underline{w})$ is a stable boundary periodic point contained in the rectangle $R_{w_0}$, where $w_0=\underline{w}_0$, is $0$ therm of the code $\underline{w}$. However, by hypothesis, if $p$ is  a $s$-boundary point  of $R_{w_0}$, it is a corner point of $R_{w_0}$,   then, $p$  is a $u$-boundary point. We can use Lemma  \ref{Lemm: Projection Sigma S,U,I} to deduce that the code $\underline{w}$ is also a $u$-boundary code. So, $\underline{w} \in \underline{C(T)}$. The case when the code $\underline{w}$ is a $u$-boundary code, i.e., $\underline{w} \in \underline{U(T)}$, is similarly proven.

On the other hand, assume  $T$ satisfies the corner property. Let $p\in R_{w_0}\in \cR$ a $s$-boundary point of $R$. There exist a $s$-boundary code $\underline{w}\in \pi^{-1}_T(p)$ such that $\underline{w}_0=w_0$. Like $T$ have the corner property, any $s$-boundary periodic code is $u$-boundary, which in turn means $\underline{w}$ is a $u$-boundary code. As a consequence of Lemma  \ref{Lemm: Projection Sigma S,U,I}, $p=\pi_T(\underline{w})\in R_{w_0}$ is a $u$-boundary point contained in $R_{w_0}$, therefore $p$ is a corner point of such rectangle, $R_{w_0}$. The situation when $p$ is $s$-boundary point of $R_{w_0}$ is proved in the same manner.

On the other hand, assume that $T$ satisfies the corner property. Let $p \in  \in \partial^s R_{w_0}$ be an $s$-boundary point of $R_{w_0}$. There exists an $s$-boundary code $\underline{w} \in \pi^{-1}_T(p)$ such that $\underline{w}0=w_0$. Since $T$ has the corner property, any $s$-boundary periodic code is also a $u$-boundary code, which, in turn, means that $\underline{w}$ is a $u$-boundary code. As a consequence of \ref{Lemm: Projection Sigma S,U,I}, $p=\pi_T(\underline{w})\in \partial^s R{w_0}$ is a $u$-boundary point contained in $R{w_0}$; therefore, $p$ is a corner point of such a rectangle, $R_{w_0}$. The situation when $p$ is an $s$-boundary point of $R_{w_0}$ is proved in the same manner.

\end{proof}

\subsection{Construction of the corner refinement}

The first goal of this section is to provide a method for creating a geometric Markov partition $\cR_{C(T)}$ for the homeomorphism $f$, along with its corresponding geometric type $C(T)$, in such a way that both $\cR_{C(T)}$ and $T_C$ satisfy the corner property.

Let $T$ be a geometric type in the pseudo-Anosov class with a binary incidence matrix $A(T)$. Consider a generalized pseudo-Anosov homeomorphism $f: S \to S$ with a geometric Markov partition $\cR$ of geometric type $T$. The procedure is as describe in the following paragraph:

\begin{enumerate}
\item Get the $s$-boundary refinement of $\cR$ along the family of boundary periodic  of codes $\underline{B(T)}$ to obtain:
$$
\cR_{S(\underline{B(T)})}.
$$

and its geometric type: $T_{S(\underline{B(T)})}$.

\item Compute the set of boundary periodic codes of the geometric type $T_{S(\underline{B(T)})}$ (Corollary \ref{Coro: algoritmic per codes}):
$$
\underline{S(T_{S(\underline{B(T)})})} \subset \Sigma_{A(T_{S(\underline{B(T)})})}.
$$

\item Obtain the $u$-boundary refinement of $\cR_{S(\underline{B(T)})}$ along the family of boundary periodic codes $\underline{S(T_{S(\underline{B(T)})})}$ to obtain the corner refinement of $\cR$:

\begin{equation}\label{Equa: Corner refinament}
\cR_{C(T)}:=[\cR_{S(\underline{B(T)})}]_{U(\underline{S(T_{S(\underline{B(T)})})})}.
\end{equation}

whose geometric type is $T_C$.
\end{enumerate}

\begin{defi}\label{Defi: Corner refi}
The geometric Markov partition for $f$, $\cR_{C(T)}$, described by equation (\ref{Equa: Corner refinament}), is the \emph{Corner refinement} of the geometric Markov partition $(f,\cR)$, and its geometric type $T_C$ is the \emph{Corner refinement} of $T$.
\end{defi}

\begin{prop}\label{Prop: Corner ref periodic boundary points}
Let $T$ be a geometric type in the pseudo-Anosov class with a binary incidence matrix $A(T)$. Consider a generalized pseudo-Anosov homeomorphism $f: S \to S$ with a geometric Markov partition $\cR$ of geometric type $T$. Let $(f,\cR_{C(T)})$ and $T_C$ be their corner refinement. Then:
\begin{itemize}
\item[i)] The boundary periodic points of $(f, \cR_{C(T)})$ and the boundary periodic points of $(f, \cR)$ are exactly the same. In simpler terms: $P_b(f, \cR_{C(T)})$ equals $P_b(f, \cR)$.

\item[ii)] The Markov partition $(f, \cR_{C(T)})$ exhibits the corner property

\item[iii)]  The new geometric type $T_C$ also exhibits the corner property:
$$
\underline{B(T_C)} = \underline{U(T_C)} \cap \underline{S(T_C)}.
$$
In other terms every boundary periodic code of $T_C$ is a corner code.

\end{itemize}

\end{prop}

\begin{proof}

\textbf{Item} $i)$. For the $s$-boundary refinement, the following equality holds:
$$
P_b(f,\cR_{S(\underline{B(T)}))}) : = P_b(f,\cR)\cup \pi_T(\underline{B(T)})=P_b(f,\cR).
$$
As a consequence, the boundary points of $(f, \cR_{C(T)})$ coincide with the boundary points of $(f, \cR).$

\textbf{Item} $ii)$. Let $p$ be a boundary point of $\cR_{S(\underline{B(T)})}$, and let $R \in \cR_{S(\underline{B(T)})}$ be a rectangle that contains $p$. Let $R_{0} \in \cR$ be a rectangle of $\cR$ that also contains $R$. After the $s$-boundary refinement:

\begin{enumerate}
\item If $p$ belonged to both $\partial^s R_{0}$ and $\partial^u R_0$, then $p$ becomes a corner point of the rectangle $R$.
\item If $p$ were  in $\partial^s R_{0}$ but not in $\partial^u R_0$, it remains in $\partial^s R$ but not in $\partial^u R$. There is no change.
\item If $p$ was in $\partial^u R_{0}$ but not in $\partial^s R_0$, then it becomes a corner point of $R$. 
\end{enumerate}

Let $p$ be a periodic boundary point in $\cR_{C(T)}$. Consider a rectangle $R$ in $$
\cR_{C(T)}=[\cR_{S(\underline{B(T)}}]_{U(\underline{B(T_{S(\underline{B(T)})})})},
$$
 that contains $p$, along with a rectangle $R_0$ in $R_{0} \in \cR_{S(\underline{B(T)})}$  that contains $R$. After the $u$-boundary refinement, we can have two possible outcomes
 
\begin{enumerate}
	\item 	If $p\in \partial^s R_0\cap \partial^u R_0$  was a corner point of $R_{w_0}$, then the rectangle $R$ has $p$ as one of its corner points.
	
	\item 	If $p\in \partial^u R_0 \setminus \partial^s R_0$, after the $u$-boundary refinement, $p$ is in $\partial^s R \cap \partial^u R$ and is a corner point of $R$.
\end{enumerate}

In conclusion, the geometric Markov partition $(f, \cR_{C(T)})$ exhibits the corner property, as asserted in Item $ii$.

\textbf{Item} $iii)$. In light of Lemma \ref{Lemm: Corner prop equivalence}, $T_{C(T)}$ exhibits the corner property, and thus, we can proceed to prove item $iii$.

\end{proof}

\subsection{The corner refinement in a family of periodic codes} 
 
Let $\cR$ be a geometric Markov partition of $f: S \to S$ with geometric type $T$ that has the corner property. Take a family of periodic codes $\cW\subset \Sigma_{A(T)}$ with its corresponding projected family $P_{\cW} = \pi_T(\cW)$ of periodic points for $f$. The $s$-boundary refinement of $\cR$ concerning this family yields $\cR_{S(\cW)}$ and its geometric type $T_{S(\cW)}$. The boundary points of this Markov partition are given by the union of the boundary points of $\cR$ and the periodic points determined by the projections of the codes in $\cW$:
 
$$
 P_b(f,\cR_{S(\cW)})=P_b(f,\cR)\cup P_{\cW}
 $$ 

Let  $\cR_{C(T_{S(\cW)})}$ be the corner refinement of  $\cR_{S(\cW)}$ with geometric type $[T_{S(\cW)}]_C$. A corollary of Proposition \ref{Prop: Corner ref periodic boundary points}, when applied to the geometric type $T_{S(\cW)}$, is as follows:
 
 \begin{coro}\label{Coro: Boundary refinament}
 Let $\cR_{C(T_{S(\cW)})}$ be the corner refinement of $\cR_{S(\cW)}$ with geometric type $T_{C(T_{S(\cW)})}$. Then:
 
\begin{itemize}
	\item[i)] The boundary periodic points of $(f, \cR_{C(T_{S(\cW)})})$ and the boundary periodic points of $(f, \cR_{S(\cW)})$ are exactly the same. In simpler terms: $P_b(f, \cR_{C(T_{S(\cW)})}) =P_b(f, \cR_{S(\cW)})$.
	
	\item[ii)] The Markov partition $(f, \cR_{C(T_{S(\cW)})})$ exhibits the corner property.
	
	\item[iii)]  The new geometric type $[T_{S(\cW)})]_C$ also exhibits the corner property:
	$$
	\underline{B([T_{S(\cW)})]_C)} = \underline{U([T_{S(\cW)})]_C)} \cap \underline{S([T_{S(\cW)})]_C)}.
	$$
	In other terms, every boundary periodic code of $[T_{S(\cW)})]_C$ is a corner code.
	
\end{itemize}

\end{coro}

\subsection{The corner refinement in  codes of bounded period} 

Finally, let us make a construction that will help us during the discussion of the Béguin algorithm.

\begin{defi}\label{Defi: boudary of period P }

Let $T$ be a geometric type in the pseudo-Anosov class with a binary incidence matrix $A(T)$, and let $f: S \to S$ be a generalized pseudo-Anosov homeomorphism with a geometric Markov partition $\cR$ of geometric type $T$. Assume that $(f, \cR)$ (and $T$) has the corner property. Now, consider the following definitions:

\begin{itemize}
\item Let  $P_{\underline{B(T)}}$ be  the maximum period among the boundary periodic codes of $T$:
$$
P_{\underline{B(T)}}:=\max\{\textbf{Per}(\underline{w},\sigma_{T}): \underline{w}\in \underline{B(T)}\}
$$
\item For any integer $P \geq P_{\underline{B(T)}}$, define $\cW_P$ as the set of admissible codes with periods less than or equal to $P$:
$$
\cW_{P}:=\{\underline{w}\in \Sigma_{A(T)}: \textbf{Per}(\underline{w},\sigma_{T})\leq P\}.
$$
\item The corner refinement of $\cR_{S(\cW_P)}$ along the family $\cW_P$  is  given by  $\cR_{C(T_{S(\cW_P)})}$, with its corresponding geometric type referred to as $T_{C(P)}$.
\end{itemize}

\end{defi}

\begin{lemm}\label{Lemm: relation periods projection refinament C}
Let $P\geq P_{\underline{B(T)}}$. Suppose $p\in S$ is a periodic point of $f$ with a period less than or equal to $P$; i.e., $\textbf{Per}(p,f)\leq P$. Consider $\underline{w}\in \Sigma_{A(T)}$ as a code that projects to $p$: $\pi_T(\underline{w})=p$. Then, the period of $\underline{w}$ is less than or equal to $P$, and $\underline{w}\in \cW_{P}$.

Conversely, if $\underline{w}\in \cW_{P}$, then the point $p:=\pi_T(\underline{w})\in S$ is a periodic point of $f$ with a period $\textbf{Per}(p,f)$ less than or equal to $P$. In this manner:

\begin{equation}\label{Equa: Projection codes an point period P}
\pi_T(\cW_{P})=\{p\in S : \textit{Per}(p,f)\leq P\} \text{ and } \pi_T^{-1}(\{p\in S : \textit{Per}(p,f)\leq P\})=\cW_P
\end{equation}

\end{lemm}

\begin{proof}
Let $p\in S$ be a periodic point of $f$ with period less or equal than $P$. We have two possible situations: 
\begin{itemize}
\item The periodic point $p$ is a boundary point. In this case, $\underline{w}\in \underline{B(T)}\subset \cW_{P}$ just by definition of $\cW_P$. Let's recall that any code in $\underline{B(T)}$ has a period less than or equal to $P_{\underline{B(T)}}\leq P$.

\item The periodic point $p$ is interior. In this case, $\pi^{-1}(p)$ contains a unique code $\underline{w}$ whose period coincides with the period of $p$, i.e., $\textbf{Per}(\underline{w},\sigma_{T})=\textbf{Per}(p,f)\leq P$. Therefore, $\underline{w}\in \cW_{P}$.

\end{itemize}

This implies that:
\begin{equation}\label{Equa: period codes and points, I}
\pi_T^{-1}(\{p\in S : \textit{Per}(p,f)\leq P\}) \subset \cW_P
\end{equation}
and 

\begin{equation}\label{Equa: period codes and points, II}
 \{p\in S : \textit{Per}(p,f)\leq P\}= \pi_T(\pi_T^{-1}(\{p\in S : \textit{Per}(p,f)\leq P\}))  \subset \pi_T(\cW_P).
\end{equation}
 
 For any code $\underline{w}\in \cW_{P}$ with a period $Q$ no greater than $P$, we can observe the following relationship:
 
  $$
  p=\pi_T(\underline{w})=\pi_T(\sigma^Q(\underline{w}))=f^Q(\pi_T(\underline{w}))=f^Q(p)
  $$

This implies that $P\geq Q \geq\textbf{Per}(p,f)$, i.e., the period of $p$ with respect to $f$ is less than or equal to $P$. Then:

\begin{equation}\label{Equa: period codes and points, II, b}
\pi_{T}(\cW_P)  \subset \{p\in S : \textit{Per}(p,f)\leq P\}.
\end{equation}

This inclusion, along with the information provided in \ref{Equa: period codes and points, II}, implies that:
$$
\pi_{T}(\cW_P)  = \{p\in S : \textit{Per}(p,f)\leq P\}.
$$

Furthermore, if $\underline{w}\in \cW_P$, trivially $\underline{w}\in \pi_T^{-1}(\pi_T(\underline{w}))$ and $p:=\pi_T(\underline{w})$ has a period less than or equal to $P$. In this manner:
 
\begin{equation}\label{Equa: period codes and points, I, b}
\cW_P \subset \pi_T^{-1}(\{p\in S : \textit{Per}(p,f)\leq P\}).
\end{equation}
Using this information and the contention provided in \ref{Equa: period codes and points, I}, we establish the second equality of our lemma:
$$
\pi_T^{-1}(\{p\in S : \textit{Per}(p,f)\leq P\})=\cW_P.
$$

This ends our proof.
\end{proof}

\begin{coro}\label{Coro: P boundary refinament}
When $P\geq P_{B(T)}$, the periodic boundary points of the corner refinement $(f, \cR_{C(T_{S(\cW_P)})})$ coincides with the set of all periodic points of $f$ having periods less than or equal to $P$ i.e.,
$$
P_b(f,\cR_{C(T_{S(\cW_P)})} )=\{p\in S : \textit{Per}(p,f)\leq P\}. 
$$
Furthermore, $(f,\cR_{C(T_{S(\cW_P)})} )$ also possesses the corner property.
\end{coro}

\begin{proof}
From Corollary \ref{Coro: Boundary refinament}: 
$$
 P_b(f, \cR_{C(T_{S(\cW_P)})}) =P_b(f, \cR_{S(\cW_P)}).
$$
But, we recall that:
$$
P_b(f, \cR_{S(\cW_P)})= P_b(f,\cR) \cup \pi_T(\cW_{P}).
$$
In view of Lemma \ref{Lemm: relation periods projection refinament C} 
$$
\pi_T(\cW_P) = \{p\in S : \textit{Per}(p,f)\leq P\}
$$ 
 and clearly $P_b(f,\cR)\subset \{p\in S : \textit{Per}(p,f)\leq P\}$. Therefore: 
$$
 P_b(f, \cR_{C(T_{S(\cW_P)})}) =P_b(f, \cR_{S(\cW_P)})=P_b(f,\cR) \cup \pi_T(\cW_{P})= \{p\in S : \textit{Per}(p,f)\leq P\}.
$$
\end{proof}

\section{The boundary refinement respect to $\cW$}

In this section we deal with next objects:

\begin{enumerate}
\item  A geometric type in the pseudo-Anosov class, $T$  whose incidence matrix $A(T)$ is binary.
\item  A generalized pseudo-Anosov homeomorphism $f: S\rightarrow S$. With a geometric Markov partition $(f,\cR)$ of geometric type $T$.
\item Finite family of periodic points of $f$, $\cP:={p_\alpha}_{\alpha=1}^A$
\end{enumerate}

Given a geometric type $\cG$ we must fix the following notation, whenever it have seance:

\begin{enumerate}
	\item $A(\cG)$ its it incidence matrix, $(\Sigma_{A(\cG)},\sigma_{A(\cG)})$ the sub-shift of finite associate to $\cG$ (??????)and $\pi_{\cG}: \Sigma_{A(\cG)} \to S$ the natural projection defined in Def ????.
	\item $\cW_{\cG}(\cP)=\pi^{-1}_{\pi_{\cG}}(\cP)$. 
\end{enumerate}

\begin{defi}
The boundary refinement of $\cR$ and $T$ with respect to the family of periodic points $\cP$ is:
$$
\cR_{B(\cP)}:=(\cR_{S(\cW_T)})_{U(\cW_{T_{S(\cW_T(\cP))}}(\cP))}.
$$

\end{defi}

To clarify the ideas:

\begin{itemize}
\item We obtain the $s$-boundary refinement of $\cR$ along the periodic codes $\cW_T$, to get another Markov partition $\cR_{S(\cW_T)}$ with geometric type $T_{S(\cW_T(\cP))}$.

\item $\cW_{T_{S(\cW_T(\cP))}}(\cP))$ is a family of periodic codes in $\Sigma_{A( T_{S(\cW_{T}(\cP))})}$. We obtain the $u$-boundary refinament of $\cR_{S(\cW_T)}$ along this family of peridoic codes. To obtain:

$$
(\cR_{S(\cW_T)})_{U(\cW_{T_{S(\cW_T(\cP))}}(\cP))}.
$$
as was defined.
\end{itemize}

This definition have total seance from a formal point but we is necessarily that $(\cR_{S(\cW_T)})$ be  commutable in therms of the geometric type $T$ and the family of periodic codes $\cW_T$ to be interesting.  We have develop a effective approach to compute the $s$ and $u$ boundary refinements of a Markov partition and their respective geometric type, by applied such process to the respective Markov partitions we should obtain the boundary refinement an its geometric type. The only difficulty is loved a effective computation of the family codes of $\cW_{T_{S(\cW_T(\cP))}}(\cP)$  in therms of $T$. 

Lets to recall that $T\neq T_{S(\cW_{T})}$ at leas the family just contains $s$-boundary codes. Therefore the shift space $\Sigma_{A(T)}$ is different that $\Sigma_{A( T_{S(\cW_{T})})}$. The objective of this section is solve  the following problem:

 \begin{prob}
Let $T_{S(\cW_T(\cP))}$ be the geometric type of the $s$-boundary refinement of $T$ a long the family $\cW_T(\cP)$ of periodic codes in $\Sigma_{A(T)}$. Compute every code in the family $\cW_{T_{S(\cW_T(\cP))}}(\cP)$, in therms of the geometric type $T$ and $\cW_T(\cP)$.
 \end{prob}

 \subsection{From a code in $\Sigma_{A(T)}$ to a pair of codes in $\Sigma_{A(T_{S(\cW)})}$}

We are going to consider the next simplified problem:

\begin{prob}
	Let $\{p_m\}_{m=1}^m$ a family of periodic points of $f$ that are no $s$ or $u$-boundary codes. Let $\cW:=\cup_{m=1}^m\pi_f^{-1}(p_i)$ and let $\cW_s:=\cup_{m=1}^m\pi_{f_s}^{-1}(p_i)$. Construct any element of   $\cW_s$ in therms of elements in $\cW$.
\end{prob}

Let $\underline{w}\in \cW$ lets see how the boundary refinement split it. Lets to take:
\begin{enumerate}
\item The period of $\underline{w}$ is $P$.
\item We can determine for all $t\in \{0,\cdots, P-1\}$ the following parameters: $s_t:=\underline{s}(t,\underline{w})$,  $\tilde{r}(w_t,\underline{s}(t,\underline{w}))$ and $\epsilon_t:=\epsilon_t(w_t,j_{t,\underline{w}} )$.
\end{enumerate}

Now we are going to construct the two codes in $\Sigma_{A(T_{S(\cW)})}$ that project to $p:=\pi_f(\underline{w})$ using a recursive process:

\begin{itemize}
\item The point $p$ is in the stable boundary of two adjacent rectangles of $\cR_{S(\cW)}$ $\tilde{R}_{\tilde{r}(w_0,\underline{s}(0,\underline{w}))}$ and $\tilde{R}_{\tilde{r}(w_0,\underline{s}(0,\underline{w})+1)}$. Clearly $\tilde{R}_{\tilde{r}(w_0,\underline{s}(0,\underline{w}))}$  is below $\tilde{R}_{\tilde{r}(w_0,\underline{s}(0,\underline{w})+1)}$. Depending on this choose we are going to construct the different codes. In this construction we are going to take the inferior rectangle
$$
v_0:= \tilde{R}_{\tilde{r}(w_0,\underline{s}(0,\underline{w}))}
$$
 which means that we need to look for the orbit of the any sector $e$ of $p$ that point toward the bottom of $R_{w_0}$.
\item Now the stable segment $f(I_{0,\underline{w}})\subset I_{1,\underline{w}}$ and $I_{1,\underline{w}}$ is the stable boundary of two adjacent rectangles of $\cR_{S(\cW)}$ they are $\tilde{R}_{\tilde{r}(w_1,\underline{s}(1,\underline{w}))}$ and $\tilde{R}_{\tilde{r}(w_1,\underline{s}(1,\underline{w})+1)}$. To determine the index $v_1$ we need to know if $f$ change or not the vertical direction of our sector $e$, i.e. we need to look at $\epsilon_0$, in this manner:

\begin{itemize}
\item If $\epsilon_0=1$ we take $v_1=\tilde{r}(w_1,\underline{s}))$ because $f(e)$ point towards the bottom of $R_{w_1}$. 
\item If $\epsilon_0=-1$ we take $v_1=\tilde{r}(w_1,\underline{s}+1))$ because $f(e)$ point towards the top of $R_{w_1}$.
\end{itemize}
In any case we know the therm $v_1$ and the direction towards $f(e)$ points, and this position is given by  the sing og $-\epsilon_0$.

\item Now it comes our recursive law. Imagine we have determined the therm $v_T$ and we know the direction towards $f^T(e)$ points towards by use  the formula:
  $$
  e_T:= - \prod_{0}^{T-1}\epsilon_t
  $$ 
  
  then:

\begin{itemize}
\item[i)]  If $ e_T=- \prod_{t=0}^{T-1}\epsilon_t=-1$, $f^T(e)$ points toward the bottom of $R_{w_T}$ and :

\begin{itemize}
\item[a)] If $\epsilon_T=1$,  $v_{t+1}=\tilde{r}(w_{T+1},\underline{s}(T+1,\underline{w}))$ and $- \prod_{t=0}^{T-1}\epsilon_t=-1$.
\item[b)] If $\epsilon_T=1$,  $v_{t+1}=\tilde{r}(w_{T+1},\underline{s}(T+1,\underline{w})+1)$ and $- \prod_{t=0}^{T-1}\epsilon_t=1$.
\end{itemize}

\item[i)]  If $ e_T=- \prod_{t=0}^{T-1}\epsilon_t=1$, $f^T(e)$ points toward the top of $R_{w_T}$ and :
  
  \begin{itemize}
  	\item[a)] If $\epsilon_T=1$, $v_{t+1}=\tilde{r}(w_{T+1},\underline{s}(T+1,\underline{w})+1)$ and $- \prod_{t=0}^{T-1}\epsilon_t=1$.
  	\item[b)] If $\epsilon_T=1$,  $v_{t+1}=\tilde{r}(w_{T+1},\underline{s}(T+1,\underline{w}))$ and $- \prod_{t=0}^{T-1}\epsilon_t=-1$.
  \end{itemize}
  
\end{itemize}

\item This processes end after $P$ or $2P$ iterations depending on the change of the vertical orientation of $f^P$ in the unstable manifold of $p$. If end after $2P$ iterations we obtain the two different codes of $\Sigma_{A(T_{S(\cW)})}$ that project to $p$ at the same time. If not we need to incite the process assuming that the sector $e$ point towards the top of $R_{w_0}$ and take  the upper rectangle
$$
v_0:= \tilde{R}_{\tilde{r}(w_0,\underline{s}(0,\underline{w})+1)}.
$$

\end{itemize}

\section{The genus and the geometric type}\label{Sec: euler characteristic}

Along this section, let $T=\{n,\{h_i,v_i\}_{i=1}^n, \Phi_T:=\{\rho_T,\epsilon_T\}\}$  be a geometric type in the pseudo-Anosov class, assume that $f:S\to S$ is a generalized pseudo-Anosov homeomorphism with a geometric Markov partition $\cR$ of geometric type $T$. We must determine the genus of the surface $S$, gen$(S)$, using the \emph{Euler-Poincaré formula} (\cite[Proposition 11.4]{farb2011primer}):

\begin{prop}\label{Prop; Euler Poincare formula}
	Let $S$ be a closed and oriented surface with a singular foliation $\cF$. Let $P_s$ be the number of prongs at a singular point $p_s$. Then
	$$
2-2\text{gen}(S)=\frac{1}{2}\sum_{p_s\in \text{Sing}(\cF)}(2-P_s).
	$$
\end{prop}

If we are able to compute the number of singularities of $\mathcal{F}^{s,u}$ and the number of prongs in each of them, we have essentially computed the genus, as we can take the sum over all of the singularities and apply the Euler-Poincaré formula.

 The set of $s$-\emph{boundary periodic codes} of $T$, Per$(\underline{\cS(T)}$, was introduced in \ref{Defi: s,u-boundary periodic codes} and Proposition  \ref{Prop: positive codes are boundary}  establishes that each such periodic code, $\underline{w}(i,\epsilon)$, projects into $\partial^s_{\epsilon}R_i$ and the positive part is given by $I^+(i,\epsilon)$ (Definition \ref{Defi: positive negative boundary codes }). Similarly, the $u$-generating funtion $\Upsilon(T)$ (\ref{{Defi: u-boundary generating funtion}}) lets us determine the negative codes $J^{-}(i,\epsilon)$ for each unstable boundary component of $\cR$.
 
After a refinement in the geometric type,  we can assume that the incidence matrix of $T$, $A(T)$, is binary. In Chapter\ref{Chapter: TypeConjugacy} , we use the $s$-\emph{generating function} $\Gamma(T)$ (Definition  \ref{Defi: s-boundary generating funtion}) to produce the $s$-\emph{boundary codes} in the sub-shift of finite type, $(\Sigma_{A(T)},\sigma_{A(T)})$, associated with $T$ (Definition \ref{Defi: Subshift of type for a Markov partition}). These codes are  projected by $\pi_{(f,\cR)}:\Sigma_{A(T)} \to S$ (Definition \ref{Defi: projection pi}) into the stable boundary components of the Markov partition $\mathcal{R}$.

If we think for a while, any $k$-prong of the stable foliation  $\cF^s$ is always in the boundary of $\cR$ and if $T$ have the corner property,   the number $k$ of prong in each singularity is equal to the number of the different periodic boundary codes in $\Sigma_{T}$ that project into $p$. This discussion suggest the following proposition.

\begin{prop}\label{Prop: Peridic boundary codes}
Let $T$ be a geometric type in the pseudo-Anosov class whose incidence matrix $A(T)$ is binary and has the corner property. Let $f:S\to S$ be a pseudo-Anosov homeomorphism with a geometric Markov partition $\mathcal{R}$ of geometric type $T$. Let $(\Sigma_{A(T)},\sigma_{A(T)})$ be the sub-shift associated with $T$, and let $\pi=\pi_{(f,\mathcal{R})}: \Sigma_{A(T)} \to S$ be the natural projection, then the following statements holds:
	
\begin{itemize}
\item[i)] For each boundary label (\ref{Defi; s,u boundary labels of T}) of $T$  it is possible to determine by iterating $\Gamma(T)$ at most $2n$ times if the corresponding stable boundary $\partial_{\epsilon}^s R_i$ of $\mathcal{R}$ is either periodic or not. In the periodic case, the positive periodic expression is given by the positive code, $I^{+}(i,\epsilon)$, and has a period less than or equal to $2n$.   Therefore the set of \emph{periodic and positive} $s$-boundary codes of $T$ is given by:
$$
\underline{\cS}^{+}(T)=\{I^{+}(i_{\iota},\epsilon_{\iota})\}_{\iota=1}^{\textbf{P}(T)},
$$

Where  $1\leq\textbf{P}(T)\leq 2n$  is equal to the number of  periodic stable boundary positive codes of $T$, the index  $i_{\iota}\in \{1,\cdots,n\}$ may appear in at most $2$ pairs but if $i_{\iota}=i_{\iota'}$ then $\epsilon_{\iota}\neq \epsilon_{\iota'}$.

\item[ii)] A similar statement holds  for the set of \emph{periodic and negative} $u$-boundary codes of $T$ by applying the $u$-generating funtion $\Upsilon(T)$  at most $2n$ times. Moreover, the set of \emph{periodic and negative} $u$-boundary codes of $T$ is given by:
 $$
 \underline{\cU}^{-}(T)=\{J^{-}(i_{\iota},\epsilon_{\iota}'\}_{\iota=1}^{\textbf{P}(T)},
 $$
where the number $\textbf{P}(T)$ and the indexes $i_{\iota}\in\{1,\cdots,n\}$ are the  same  than those given  in item $i)$ for the set $\underline{\cS}^{+}(T)$,  but in this case  $\epsilon_{\iota}'$ may be different from $\epsilon_{\iota}$.

\item[iii)] The set of \emph{ periodic boundary codes} of $T$, denoted by Per$(\underline{\cB(T)})$,  is given by such periodic codes $\underline{p}\in \text{Per}(\Sigma_{A(T)})$ for whose there exist  $ \iota \in \{1, \cdots, \textbf{P}(T)\}$ and such that: 
\begin{equation}
\underline{p}_+= I^{+}(i_{\iota},\epsilon_{\iota}) \text{ and  }  \underline{p}_-= J^{-}(i_{\iota},\epsilon_{\iota}')
\end{equation}

In this case we label such periodic periodic code as $\underline{p}_(\iota,\epsilon8_{\iota})$ keeping the label of its positive part. In this manner:
$$
\text{ Per }(\underline{\cB(T)})=\{\underline{p}(\iota,\epsilon_{\iota})\}_{\iota}^{\textbf{P}(T)}
$$
\end{itemize}
\end{prop}
  
 \begin{proof}
 \textbf{Item $i)$ :} For each \emph{boundary label} $(i,\epsilon)$,  the procedure to construct its positive boundary code  $\underline{I}_+ (i,\epsilon)$  as given in Definition \ref{Defi: positive negative boundary codes },  determine the image by $f$ of the horizontal boundary, $\partial^s_{\epsilon}R_{i}$ and the image of $\partial^s_{\epsilon}R_i$ can be in at most $2n$ different boundary components of $\cR$ before come back to itself, in this manner  after at most $2n$ iteration of $\Gamma(T)$ we can determine if $\partial^s_{\epsilon}R_{i}$  is or not periodic, while we can recover the periodic expression of $\underline{I}^+(i,\epsilon)$ by look the first $2n$ iteration of $\Gamma(T)$. 
 
 The number of periodic boundary codes of $T$, \textbf{P}$(T)$, is well determined and is less or equal than $2n$ that is the number of different of boundary codes of $T$. 
 
 If  the code $I^+(i,\epsilon)$ is periodic the rectangle $R_i$ have a periodic point $p$ in the  stable boundary component $\partial^s_{\epsilon}R_{i}$, but  for the corner property of $\cR$, $p$  is a corner point of $R_i$.  The rectangle $R_i$ cant have two different periodic corner points, but $R_i$ can have $p$ as corner point whit two different quadrants contained in $R_i$, is this situation which force that the periodic codes generated by the $s$-boundary labels, $I^+(i,1)$ and $I^+(i,-1)$ be different. 

\textbf{Items $ii)$ :} The arguments are the same than in item $ii)$, but we recall that, by the corner property every periodic boundary point $p$  is in the corner of each rectangle that contain it. In this manner if  $I^+(i_{\iota},\epsilon_{\iota})$ is periodic we can deduce that  the periodic point $p$ in $\partial^s_{\epsilon_{\iota}}R_{\iota}$ it is in $\partial^s_{\epsilon_{\iota}'}R_{\iota}$ where $\epsilon_{\iota}'$ is  such that $J^-(i_{\iota},\epsilon_{\iota}')$ is periodic with the same periodic expression than $I^+(i_{\iota},\epsilon_{\iota})$.
   
  \textbf{Items $iii)$ :} As  the code $\underline{p}$ is periodic the positive part and the negative part $\underline{w}$ must have the same periodic expression, and this determine the $(i_{\iota},\epsilon_{\iota})$ and $(i_{\iota},\epsilon_{\iota}')$ in the proposition. The rest of the notation is coherent and natural.
 
 \end{proof}
 
The equivalence relation  $\sim_{T}$ was introduced in \ref{Defi: Sim-T equivalent relation}, when is applied to the set of periodic boundary codes, each equivalence class, $[\underline{p}(i_{\iota},\epsilon_{\iota}]_{T}$,  is in correspondence with a unique periodic boundary point by the natural projection $\pi_{(f,\cR)}$, there fore we can determine the number of periodic boundary points by cont the number of such equivalence classes:
 $$
  \text{ Per }(\underline{\cB(T)})/\sim_{T}=\{[\underline{p}_s]\}_{s=1}^{\textbf(S)(T)} 
 $$
where $\textbf(S)(T)\in \NN$ is the number of $\sim_{T}$-equivalence classes, the construction of $\sim_{T}$ permits to determine in finite time such number and the elements of each conjugacy class, in this manner we pose the following notation: 
$$
\underline{P}_s:=\#[\underline{p}_s]_{T}
$$
be the number of periodic codes in the $\sim_{T}$ equivalence class of  $\underline{p}_s$.

\begin{lemm}
	Assume that $T$ has the corner property and the incidence matrix of $T$ is binary. Let $\underline{w}$ be a code in $[\underline{p}_s]_{T}$ and let  $p_s:=\pi_{f,\cR}(\underline{w})\in S$,   be the corresponding  periodic boundary point of $(f,\cR)$. The point $p_s$ is a $P_s$-prongs  where,  $ P_s= \frac{1}{2} \underline{P}_s $.
\end{lemm}

\begin{proof}
Like $T$ have the corner property, then the number $\underline{P}_s$ is equal to the number of sectors of the point $p_s$. Now,  consider that the number of prongs of $p_s$ is the half of the number of its sectors to obtain the equality.
\end{proof}

 \begin{theo}\label{Theo: Genus in therms of T}
 	Let $T$ a geometric type in the pseudo-Anosov class and let $f: S \to S$ a generalized pseudo-Anosov homeomorphism that realize $T$. If the $T$ have the corner property and its incidence matrix  is binary, the Euler-Characteristic of $S$ is given by the formula: 	
 	$$
 \chi(S)=\frac{1}{2} \sum_{s\in \text{Sing}(\cF)}(2-P_s) = \sum_{s=1}^{\textbf{S}} (2- \frac{1}{2} \underline{P}_s.
 	$$
 \end{theo}

\begin{proof}
The only thing to observe is that, ff $p_s$ is a regular (boundary) point $2-P_s=0$ and $2- \frac{1}{2} \underline{P}_s=0$. Therefore we don't add anything to the sum in the left side, by consider regular points.
\end{proof}
 
 \begin{coro}\label{Coro: algoritmo genero}
 	Let $T$ be a geometric type of the pseudo-Anosov class, and let $f:S\rightarrow S$ be a generalized pseudo-Anosov homeomorphism with a geometric Markov partition of type $T$. There exists a finite algorithm that, given $T$, determines the number of singularities of $f$ and the number of stable and unstable separatrices at each singularity and subsequently the genus of $S$.
 \end{coro}

\chapter{The formal derived from Anosov of a geometric type }

\section{The formal DA of a geometric type.}\label{Sec: Formal DA}

Let $T$ be a geometric type in the pseudo-Anosov class. Our previous results showed the existence of a closed surface $S$ and a Smale diffeomorphism $\phi: S \rightarrow S$, that have a  mixing saddle-type basic piece $K(T)$ which has a Markov partition of geometric type $T$. Therefore, we can define $\Delta(T) := \Delta(K(T))\subset S$ as the domain of $K(T)$. Moreover, the diffeomorphism $\phi$ uniquely determines (up to conjugation) a diffeomorphism in the domain $\phi_T:\Delta(T) \rightarrow \Delta(T)$. This justify the following definition.

\begin{defi}\label{Defi: Formal DA}
		Let $T$ be a geometric type of the pseudo-Anosov class. Let $S$ be a closed surface and $\phi: S\rightarrow S$ be a Smale surface diffeomorphism having a mixing saddle-type basic piece $K(T)$ with a geometric Markov partition of geometric type $T$, let $\Delta(T)\subset S$ be the domain of the basic piece $K(T)$.  We define the \emph{formal derived from Anosov} of $T$ to be the triplet
	$$
	\textbf{DA}(T):=(\Delta(T),K(T),\phi_T).
	$$
where $\phi_T: \Delta(T) \to \Delta(T)$ is the restriction of the diffeomorphism $\phi$ restricted to the domain of $K(T)$.
	
\end{defi}

The following proposition is consequence of  the Bonatti-Langevin results that were presented in the Chapter  \ref{Sub-sec: Domain basic piece} of Preliminaries.

\begin{prop}\label{Defi: Unique DA}
	Let $T$ be a geometric type in the pseudo-Anosov class. Up to a topological conjugacy in their respective domains, that is the identity in the non-boundary periodic points of $K(T)$, there exists a unique \textbf{DA}$(T)$.
\end{prop}

Let $T$ be a geometric type in the pseudo-Anosov class. Consider a generalized pseudo-Anosov homeomorphism $f$ with a geometric Markov  partition $\cR$ whose geometric type is $T$. The concept of formal derived from Anosov is rooted in the classical idea of considering all the periodic points in the boundary $\cR$ and open  the stable and unstable manifolds of these periodic points. Then a saddle-type basic piece without singularities obtained. In this process, each periodic point on the boundary of the Markov $\cR$ partition becomes a corner point of the basic piece.  However, with our definition of \textbf{DA}$(T)$, what happens is that, the periodic points on the stable boundary of $\cR$ become $s$-boundary points of the basic piece $K(T)$ while the periodic points on the unstable boundary of $\cR$ become $u$-boundary points of $K(T)$. The need to open all boundary points of $\cR$ in the stable and unstable direction is what has motivated us to introduce the corner refinement of $T$.

When comparing two saddle-type basic pieces obtained as the formal \textbf{DA} of two different geometric types in the pseudo-Anosov class, the first condition for these basic pieces to be topologically conjugate is that they must possess an equal number of periodic $s$ and $u$-boundary points. To achieve this, we introduce the concept of a \emph{joint refinement} of two geometric Markov partition and its associated geometric type. This  process allows us to transform two basic pieces originating from different geometric types into ones with the same number of corner points. We will delve into the details of this construction in the upcoming section.

\section{The joint refinement of two geometric types.} \label{Sec: Boundary refinament}

Let to recall that $\cW_{T,P}:=\{\underline{w}\in \Sigma_{A(T)}: \textbf{Per}(\sigma_{T},\underline{w})\leq P\}$.

\begin{defi}\label{Defin: Comun refinament}
	Let $f$ and $g$ be pseudo-Anosov homeomorphisms with geometric Markov partitions $\cR_f$ and $\cR_g$ whose geometric types $T_f$ and $T_g$ have binary incidence matrices.  Let  $P_{\underline{B(T_f)}}$ be  the maximum period among the boundary periodic codes of $T_f$:
	$$
	P_{\underline{B(T_f)}}:=\max\{\textbf{Per}(\underline{w},\sigma_{T_f}): \underline{w}\in \underline{B(T_f)}\}
	$$
	and similarly $P_{\underline{B(T_g)}}$ be  the maximum period among the boundary periodic codes of $T_g$. Lets take define $N(f,g):=\max\{P_{\underline{B(T_f)}}, P_{\underline{B(T_g)}}\}$. Consider:
	
	\begin{enumerate}
\item 	The \emph{joint refinement} (Definition \ref{Defi: boudary of period P }) of $\cR_f$ with respect to $\cR_g$ is the corner refinement of $\cR_f$ with respect to the family of periodic codes $\cW_{T_f, N(f,g)}\subset \Sigma_{A(T_f)}$:
$$
\cR_{(R_f,R_g)}:=[\cR_f]_{C(\cW_{T_f,N(f,g)})}.
$$
Its geometric type is denoted by $T_{(T_f,T_g)}$.

\item 	The \emph{joint refinement} of $\cR_g$ with respect to $\cR_f$ is the corner refinement of $\cR_g$ with respect to the family of boundary periodic codes $\cW(g,N(f,g))\subset \Sigma_{A(T_g)}$:
$$
\cR_{(\cR_g,\cR_f)}:=[\cR_g]_{C(\cW_{g,N(f,g)})}.
$$
Its geometric type is denoted  $T_{(T_g,T_f)}$.

	\end{enumerate}

\end{defi}

\begin{rema}\label{Rema: corner pint are th eperiod eauql N}
By construction, the geometric Markov partition $\cR_{(R_f,R_g)}$ has the corner property and all its boundary points are corners of any rectangle that contains them. The number $N(f,g)$ is equal to $N(g,f)$, and the corner points of $\cR_{(R_f,R_g)}$ coincide with the set of periodic points of $f$ with periods less than or equal to $N(f,g)$. Similarly, for $\cR_{(\cR_g,\cR_f)}$,
\end{rema}

\section{The lifting of a geometric  Markov partition of $f$ to the formal \textbf{DA}$(T)$.}

The important property of the joint refinement it that every periodic boundary point is a corner point of any rectangle that contain such point. We must use this property to prove the following Proposition. 

\textbf{Notation:} Once we have fixed $\cR_f$ and $\cR_g$ to avoid an overwhelming notation, we make the following conventions:
\begin{eqnarray}
\cR_{f,g}:=\cR_{(\cR_f,\cR_g)} \text{ and } \cR_{g,f}:=\cR_{(\cR_g,\cR_f)}, \\
T_{f,g}:=T_{(T_f,T_g)} \text{ and } T_{g,f}:=T_{(T_g,T_f)}.
\end{eqnarray}

The formal derived from pseudo-Anosov of such geometric types are:
\begin{eqnarray}
\textbf{DA}(T_{f,g}):=\{\Delta(T_{f,g}), K(T_{f,g}),\Phi_{T_{f,g}}) \} \text{ and } \\
\textbf{DA}( T_{g,f}):=\{\Delta( T_{g,f}), K( T_{g,f}),\Phi_{ T_{g,f}})\} 
\end{eqnarray}

\begin{prop}\label{Prop: preimage is Markov Tfg type}
	Let $f:S_f \rightarrow S_f$ and $g:S_g \rightarrow S_g$ be two generalized pseudo-Anosov homeomorphisms with geometric Markov partitions, $\cR_f$ and $\cR_g$, whose geometric types, $T_{f}$ and $T_g$, have binary incidence matrix. Suppose that $f$ and $g$  are  topologically conjugate through an orientation preserving homeomorphism $h:S_f\rightarrow S_g$ and let:
	$$
	\pi:\Delta(T_{g,f})\rightarrow S_g,
	$$
	be  the projection given by Proposition \ref{Prop: type of basic piece is type of pseudo-Anosov}. In this situation:
	\begin{itemize}
		\item[(1)] The induced geometric partition, $h(\cR_{f,g})$ is a geometric Markov partition of $g$ whose geometric type $ T_{f,g}$ have the corner property.
		
		\item[2)] A  point $p\in S_f$ is a periodic boundary point of the Markov partition $(f,\cR_{f,g})$ if and only if $h(p)$ is a periodic boundary point of $(g,\cR_{g,f})$. In particular, the corner periodic points of $h(\cR_{f,g})$ are the corner periodic points of $\cR_{\cR_{g,f}}$.
		
		\item[(3)] There exit a unique Markov partition $\underline{h}(\cR_{f,g})=\{\underline{h}(R_i)\}_{i=1}^n$ of $K(T_{g,f})\subset \Delta(T_{(g,f)}))$ such that:
		$$
		\pi\left(\overset{o}{\underline{h}}(R_i)\right)=\overset{o}{(h(R_i))}.
		$$
		\item[(4)] Keep the previous label of the rectangles in  $\underline{h}(\cR_{f,g})$ and give to the rectangle $\underline{h}(\cR_{f,g})$  the unique vertical and horizontal orientation  such that $\pi$ is increasing along both the vertical and horizontal leaves within $\underline{h}(R_i)$. With this geometrization, the geometric type of $\underline{h}(\cR)$ is $T_{f,g}$.
	\end{itemize}
\end{prop}

\begin{proof}

	\textbf{Item} $(1)$. 	In Definition  \ref{Defi: induced geometric Markov partition}, we introduced the induced geometric Markov partition $h(\cR_{f,g})$. Theorem\ref{Theo: Conjugated partitions same types} establishes that $h(\cR_{f,g})$ is a geometric Markov partition with the same geometric type as $\cR_{f,g}$, specifically $T_{f,g}$. The joint refinement was obtained as the corner refinement of another geometric type and Proposition \ref{Prop: Corner ref periodic boundary points} implies that $T_{f,g}$ has the corner property.

\textbf{Item} $(2)$.	By definition of the joint refinement, a periodic point $p$ of $f$ is a corner point  of $\cR_{f,g}$ if and only if it is a period is less than or equal to $N(f,g)$. But $h(p)$ has the same period that $p$, therefore $h(p)$ is a periodic point of $g$ with period less than or equal to $N(g,f)=N(f,g)$, then $h(p)$ is corner point of $\cR_{g,f}$, by definition it is a boundary point. This proves the second item of the proposition.

	\textbf{Item} $(3)$. We are going to prove the second statement through a series of lemmas. 	The reader can refer to Figure \ref{Fig: rectangle split} for a visual representation of our notation.
	
	\begin{figure}[h]
		\centering
		\includegraphics[width=0.7\textwidth]{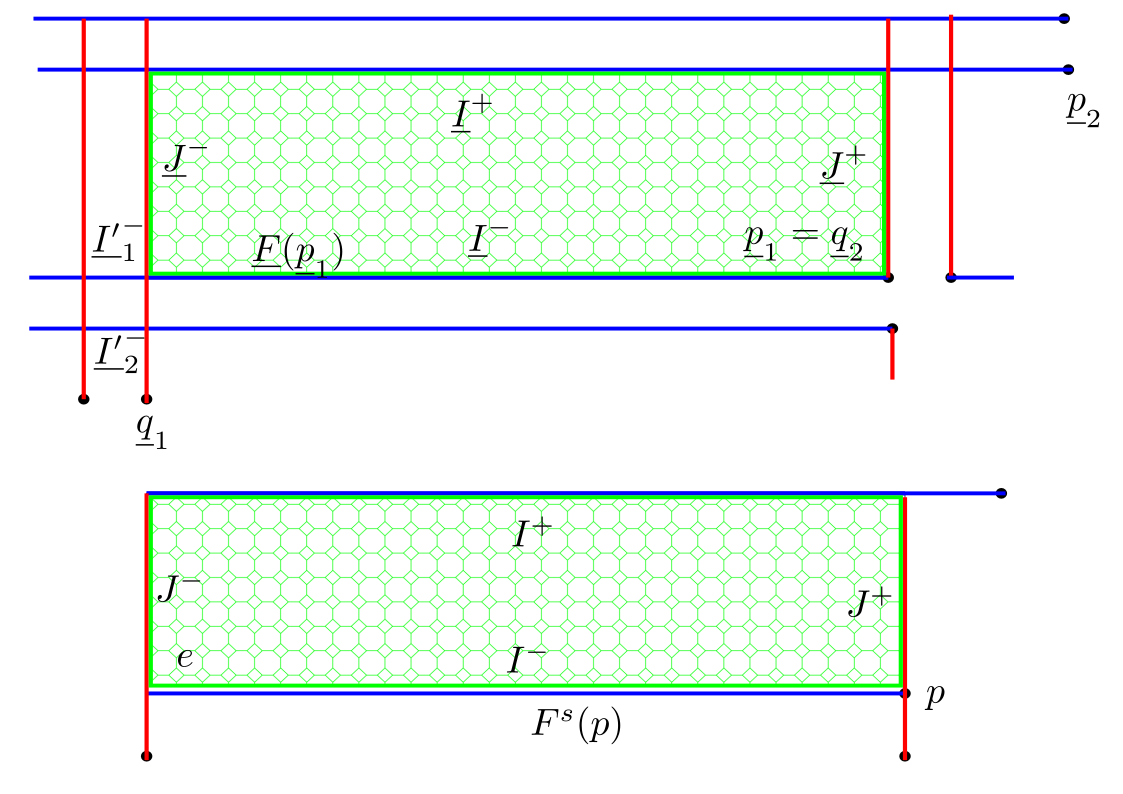}
		\caption{The lift of a rectangle $R$. }
		\label{Fig: rectangle split}
	\end{figure}

	\begin{lemm}\label{Lemm: preimage pi rectangle}
		Let $R$ be a rectangle in the Markov partition $h(\cR_{f,g})$. Then $\pi^{-1}(R)$ contains a unique rectangle $\underline{R}\subset \Delta(T_{g,f})$ with the following properties:
		\begin{itemize}
			\item The interior of $\underline{R}$ is mapped under $\pi$ to the interior of $R$, i.e., $\pi(\overset{o}{\underline{R}})=\overset{o}{R}$, and the image of $\underline{R}$ is the rectangle $R$, i.e., $\pi(\underline{R})=R$.
			
			\item The horizontal  boundary of $\underline{R}$ is contained in a stable separatrice of an $s$-boundary periodic point of $K(T_{g,f})$, and its vertical boundary is in an unstable separatrice of an $u$-boundary point.

			\item The rectangle $\underline{R}$ is isolated from the basic piece $K(T_{g,f})$ from its exterior, i.e., there is an open neighborhood $U$ of $\underline{R}$ such that $K(T_{g,f})\cap U =K(T_{g,f}) \cap \underline{R}$.
		\end{itemize}
	\end{lemm}
	
	\begin{proof}
		
		Denote the left side of $R$ as $J^-$, its right side as $J^+$, its lower boundary as $I^-$, and its upper boundary as $I^+$. All these intervals are contained within the invariant manifold of a corner periodic point of $h(\cR_{f,g})$. Moreover, the boundary points of $h(\cR_{f,g})$ are all corner points and as proven in Item $(2)$, they coincide with the corner points of $\cR_{g,f}$. This implies that neither $J^{\pm}$ nor $I^{\pm}$ contain a periodic point in their interior; therefore, they are contained within a unique separatrice of a corner point. Now, let's analyze the set $\pi^{-1}(I^-)$ to get some conclusions that can be applied to the remaining intervals $J^+, J^-, I^+$ in the boundary of $R$.
		
		Assume that $I^{-}$ is contained in the  stable separatrice $F^s(p)$ of the corner point $p$.  Inside the domain  $\textbf{DA}(T_{g,f})$ there are two $s$-boundary periodic points of $K(T_{g,f})$, $\underline{p}_1$ and $\underline{p}_2$ with different stable separatrices $\underline{F}^s(\underline{p}_1)$ and $\underline{F}^s(\underline{p}_2)$, such that:  $\pi^{-1}(\overset{o}{I^-})$ has two connected components that are are contained in such separatrices:
		$$
		\underline{\overset{o}{I'}}_1\subset F^s(\underline{p}_1) \text{ and } \underline{\overset{o}{I'}}_2\subset F^s(\underline{p}_2).
		$$
		and $\pi(\underline{\overset{o}{I'}}_1)=\pi(\underline{\overset{o}{I'}}_2)=\overset{o}{I^-}$. 
		Such stable intervals are contained in the stable boundary of  $s$-boundary periodic point, therefore only one of these intervals have points that can be approximated by singletons (Item $i)$ in Definition \ref{Defi: equivalen clases sim-r} ) of the form $\pi^{-1}(x_n)\in K(T_{g,f})$, where $x_n\in \overset{o}{R}$ and $x_n$ converge to a point in $I^-$. Assume that $\overset{o}{\underline{I'}_1}$  is such an interval. 
		
		Let's understated the closure of $\overset{o}{\underline{I'}_1}$ inside the stable separatrice $\underline{F}^s(p_1)$.  We need to consider two different of extreme points in $I^-$:
		\begin{itemize}
			\item[i)] The interval $I^{-}$ has a periodic point $p$ as one of its endpoints. In this case, $\pi^{-1}(p)$ is contained in a cycle. The endpoint of $\underline{I'}^-=\overline{\overset{o}{\underline{I'}_1}}$ which projects to $p$ needs to be contained in such a cycle. However, at the same time, such a point can be approximated by points in $K(T_{g,f})\cap \overset{o}{\underline{I'}_1}$ and is, therefore, an element of $K(T_{g,f})$. The unique hyperbolic points in the cycle are periodic points, so the endpoint of $\underline{I'}^-$ that projects to $p$ is the periodic point $\underline{p}_1$, which is approximated by points in  $\underline{F}^s(\underline{p}_1)$.
			
			\item[ii)]   The endpoint $O\in I^-$ is not periodic. In this case, $\pi^{-1}(O)$ is a minimal rectangle. Similar to before, the endpoint of $\underline{I'}^-=\overline{\overset{o}{\underline{I'}_1}}$ which projects to $O$ needs to be contained in such a minimal rectangle and be an element of the basic piece. Therefore, it is the unique corner point $\underline{O}$, in the minimal rectangle  that is saturated by points in $K(T_{g,f})\cap \pi^{-1}(\overset{o}{\underline{I'}_1})$.
		\end{itemize}
		
		In this setting: If $I^-$ have a periodic end point:
		$$
		\underline{I}^{-}:=[\underline{O},\underline{p}_1]^s\subset F^s(\underline{p}_1).
		$$
		If $I^-=[O_1,O_2]^s$ have no periodic end points, we can applied the analysis in Item $i)$ to construct point $\underline{O_1}$ and $\underline{O_2}$ such that:
		
		$$
		\underline{I}^{-}:=[\underline{O_1},\underline{O_1}]^s\subset F^s(\underline{p}_1).
		$$ 
		
		\begin{rema}\label{Rema: boundary points s-u}
			The extreme points of $\underline{I}^-$ are, in any case, $u$-boundary points. Therefore,  $\underline{I}^-$ is contained within an open interval $\underline{I}_1^-\subset \underline{F}^s(\underline{p}_1)$ such that $K(T_{g,f})\cap \underline{I}_1^- = K(T_{g,f})\cap \underline{I}^-$ and $\pi(\underline{I'}_1^-)=I^-$.  Furthermore, $\underline{I}^{-}$  is the smallest interval contained in $\underline{F}^s(\underline{p}_1)$ that projects to $I'$. These two properties determine it uniquely.
		\end{rema}

		The same construction applies to the other boundary components. That is, there are unique periodic corner points $\underline{p}_2$, $\underline{q}_1$, and $\underline{q}_2$ of the basic piece $K(T_{g,f})$ and unique separatrices of such points that contain minimal intervals:
		
		\begin{eqnarray*}
			\underline{I}^{+}\subset \underline{F}^s(\underline{p}_2)\\
			\underline{J}^{-}\subset \underline{F}^s(\underline{q}_1)\\
			\underline{J}^{+}\subset \underline{F}^s(\underline{q}_2)
		\end{eqnarray*}

		These intervals project into the respective boundary components of $R$. The stable separatrice $\underline{F}^s(\underline{p}_2)$ is the only stable separatrice that is saturated by stable leaves of non $s$ boundary periodic points that project to $R$. Clearly, these intervals are unique, as they are minimal with these properties.
		
		We claim that $\vert \underline{I}^-\cap\underline{J}^-\vert =1$. Let $O$ be the corner of $R$ determined by $I^-$ and $J^-$, and let $e$ be the sector of $O$ that is determined by the left-inferior corner of $R$. Then $\pi^{-1}(O)$ could be:
		
		\begin{itemize}
			\item Part of a cycle if $O$ is a periodic point. In this cycle, all the boundary points are corner points and has only one sector saturated by points in $K(T_{g,f})$. Let $\{x_n\}\subset \overset{o}{R}$ be a sequence on non-boundary periodic codes that converges to $0$ in the sector $e$, then $\pi^{-1}(x_n)$ is a singleton and then: $\lim_{n\to \infty} \pi^{-1}(x_n)$ is simultaneously the periodic endpoint of  $\underline{I}^{-}$ and the periodic endpoint of $\underline{J}^{-}$, i.e   $\underline{p}_1=\underline{q}_1$.

			\item  	A minimal rectangle if $O$ is not periodic. In this case, there is only one corner of the minimal rectangle where $\underline{I}_{-}$ and $\underline{J}^{-}$ intersect. Such  is the only one accumulated by singletons in $\pi^{-1}(\overset{o}{R})$ that converge to $\underline{O}$.
		\end{itemize}

		The same property follows for the other vertices of $R$. Furthermore, since the basic piece $K(T_{g,s})$ doesn't contain impasses, we deduce that $\underline{I}^-$ is in a different stable separatrice than $\underline{I}^+$, and then $\underline{I}^-\cap \underline{I}^+=\emptyset$. In conclusion, $\underline{I}^- \cup \underline{I}^{+} \cup \underline{J}^- \cup \underline{J}^+$ forms a closed curve. We shall see what happens with the interior of $\tilde{R}$.

		Take any point $x\in \overset{o}{R}$. The set $\pi^{-1}(x)$ is not part of a cycle because the only $s$ and $u$ boundary points of $K(T_{g,f})$ are projected to corner points of the original Markov partition $h(\cR_{f,g})$. Therefore, $\pi^{-1}(x)$ is a rectangle and is trivially bi-foliated. We can deduce that if $\overset{o}{J}\subset \overset{o}{R}$ is an unstable segment, its preimage $\pi^{-1}(\overset{o}{J})$ is a rectangle without its horizontal boundary. Clearly, the horizontal boundary of $\pi^{-1}(\overset{o}{J})$ has one connected component intersecting $\underline{I}^-$ and the other intersecting $\underline{I}^+$. The same holds for any horizontal segment $\overset{o}{I}\subset \overset{o}{R}$.

		In conclusion, $\underline{I}^- \cup \underline{I}^{+} \cup \underline{J}^- \cup \underline{J}^+$ bounds a region that is trivially bi-foliated, and therefore, it is a rectangle denoted as $\underline{R}$. It's evident that its projection is $R$, and its interior projects to the interior of $R$.
		
		The stable boundary of $\cR$ consists of two disjoint intervals contained in the $s$ and $u$ boundary leaves, as we observed in Remark \ref{Rema: boundary points s-u}. By taking slightly larger intervals that contain $\underline{I}^-,\underline{I}^{+},\underline{J}^-$, and $\underline{J}^+$, we can construct a small neighborhood $U$ of $\underline{R}$ such that $K(T_{g,f})\cap U =K(T_{g,f}) \cap \underline{R}$.
	\end{proof}

	As a consequence, for all rectangles $h(R_i)\in h(\cR_{f,g})$, there is a unique rectangle $\underline{h(R_i)} \subset \textbf{DA}(T_{g,f})$ as described in Lemma \ref{Lemm: preimage pi rectangle} that projects to $h(R_i)$. Furthermore, these rectangles have their boundaries isolated of the basic piece from its exterior. In particular, we have the following corollary:
	
	\begin{coro}\label{Coro: Disjoint rectangles}
		If $i\neq j$, then $\underline{h(R_i)} \cap \underline{h(R_j)}=\emptyset$.
	\end{coro}

	Let
	$$
	\underline{h}(\cR_{f,g}):=\{\underline{h(R_i)}\}_{i=1}^n
	$$
	be the family of rectangles constructed in Lemma \ref{Lemm: preimage pi rectangle} for each rectangle in the Markov partition $h(\cR_{f,g})$.

	\begin{lemm}\label{Lemm: Union rectangle is K}
		$K(T_{g,f})\subset \cup_{i=1}^n \underline{h(R_i)}$
	\end{lemm}
	
	\begin{proof}

		The set of periodic points of $f$ that are not $(s,u)$-boundary is dense inside every rectangle $h(R_i)$ and is contained in its interior. Denote the set of these periodic points by $P(i)$. Then $\pi^{-1}(P(i))$ consists of all the periodic points of $\phi_{T_{g,f}}$ that are contained in the interior of $\pi^{-1}(h(R_i))$ and there fore  that are not $(s,u)$-boundary. Such points  are dense in $K(T_{g,f})\cap R_i$, and then:
		
		$$
		K(T_{g,f})\subset \cup_{i=1}^n \overline{\pi^{-1}(P(i))} \subset \cup_{i=1}^n \overline{\overset{o}{\underline{h(R_i)}}} = \cup_{i=1}^n \underline{h(R_i)}
		$$
		
	\end{proof}

	\begin{lemm}\label{Lemm; h(R -gf) is invariant}
		The stable boundary of the partition $\underline{h}(\cR_{f,g})$ is $\Phi_{g,f}$-invariant.
	\end{lemm}

	\begin{figure}[h]
		\centering
		\includegraphics[width=0.7\textwidth]{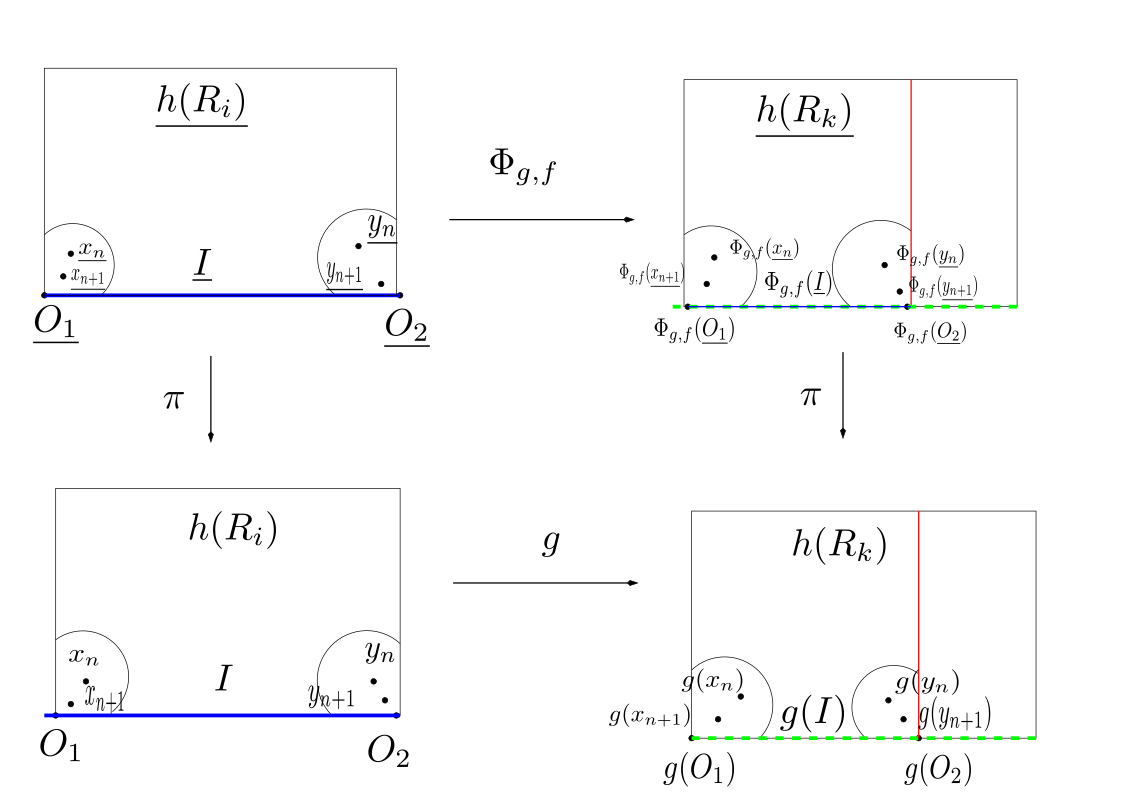}
		\caption{The stable boundary of the partition $\underline{h}(\cR_{f,g})$}
		\label{Fig: dete boundary}
	\end{figure}

	\begin{proof}
		
		Let $\underline{h(R_i)}\in \underline{h}(\cR_{f,g})$ be a rectangle. Take $\underline{I}$ as the lower boundary of such a rectangle, and let $\underline{O_1}$ and $\underline{O_2}$ its end points, the analysis for the upper boundary is the same. 
		
		There is a sequence ${\underline{x}_n}\subset \overset{o}{\underline{h(R_i)}}$ consisting of non $s,u$-boundary periodic points that converges to $\underline{O_1}$ as shown in Figure \ref{Fig: dete boundary}. Similarly, there exists a sequence $\{\underline{y}_n\} \subset \overset{o}{\underline{h(R_i)}}$ of non $s,u$-boundary periodic points that converges to $\underline{O_2}$. In these points, $\pi$ is a one-to-one map and serves as a conjugation restricted to their image. This implies that:
		
		$$
		\Phi_{g,f}(\underline{O_1})=\lim_{n\rightarrow \infty} \Phi_{g,f}(\underline{x}_n)= \lim_{n\rightarrow \infty}  \pi^{-1}\circ g \circ \pi (\underline{x_n}).
		$$
		and 
		
		$$
		\Phi_{g,f}(\underline{O_2})=\lim_{n\rightarrow \infty} \Phi_{g,f}(\underline{y}_n)= \lim_{n\rightarrow \infty}  \pi^{-1}\circ g \circ \pi (\underline{y_n}).
		$$

		By taking sub sequences of $\{\underline{x}_n\}$ and $\{\underline{y}_n\}$, we can assume that $x_n := \pi(\underline{x_n})$ and $y_n := \pi(\underline{y_n})$ are always contained in the interior of the rectangle $h(R_i).$ Furthermore, since $h(\cR_{f,g})$ is a Markov partition of $g$, we can assume that the lower horizontal sub-rectangle $H := H^i_1$ contained in $h(R_i)$ of the Markov partition $(g, \cR_{f,g})$ contains to $\{x_n\}$ and $\{y_n\}$ within its interior. Therefore, if $g(H) = V^k_l$, then $\{g(x_n)\}$ and $\{g(y_n)\}$ are contained in $\overset{o}{V^k_l} \subset \overset{o}{h(R_k)}$ and they converge to $g(O_1)$ and $g(O_2)$, which are in the stable boundary  of $\underline{h(R_k)}$.

		We claim that $\Phi_{g,f}(\underline{I})$ is a stable interval contained in a single stable boundary component of $\underline{h(R_k)}$.	 Suppose this is not the case, then $\Phi_{g,f}(\underline{I})$ contains an $s$-arc, $\underline{\alpha}'=[\underline{a}',\underline{b}']^s$  joining two corners of the rectangle $\underline{h(R_k)}$, where $\underline{a}'$ is on the lower boundary of  $\underline{h(R_i)}$ and $\underline{b}'$ is on the upper boundary of $\underline{h(R_i)}$ , as shown in Figure \ref{Fig: Same s bound}. This implies that $\underline{I}$ contains in its interior the $s$-arc
		 $\Phi_{g,f}^{-1}(\underline{\alpha}'):=\underline{\alpha}=[\underline{a},\underline{b}]^s$, that we recall it is contained in the inferior boundary of $\underline{h(R_i)}$
		 
		 	\begin{figure}[h]
		 	\centering
		 	\includegraphics[width=0.7\textwidth]{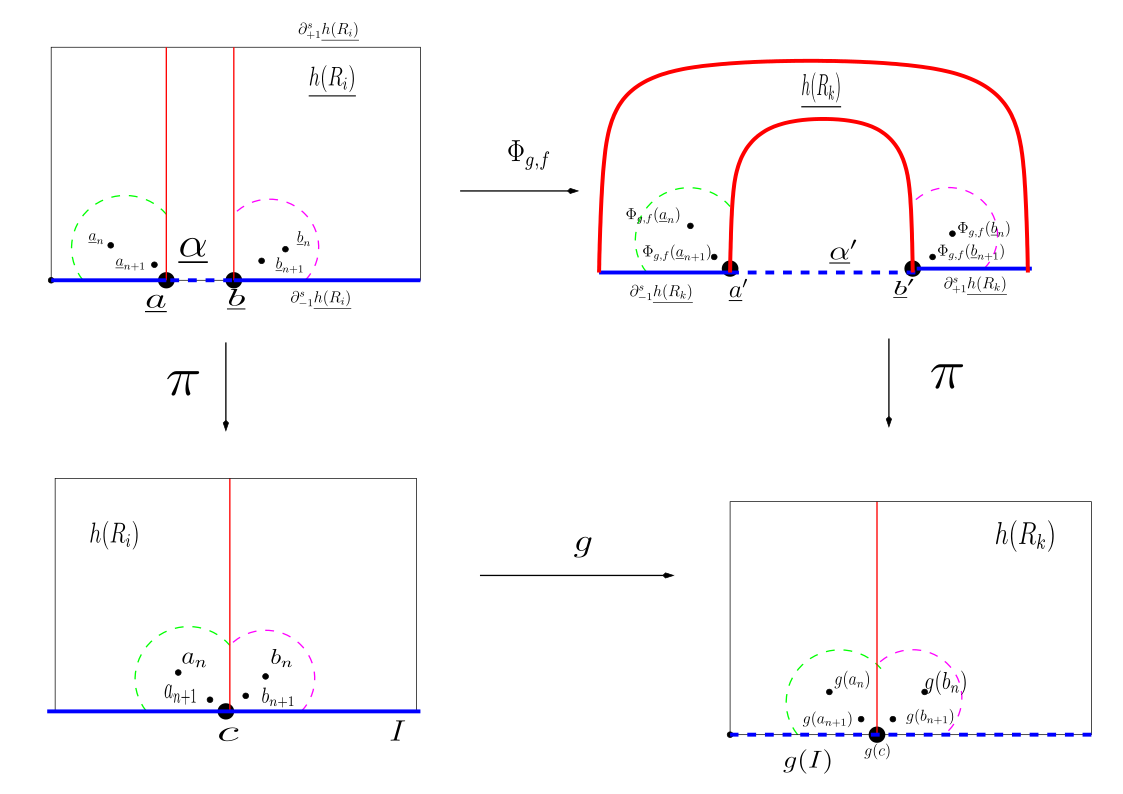}
		 	\caption{$\Phi_{g,f}(\underline{I})$  is in single stable boundary component}
		 	\label{Fig: Same s bound}
		 \end{figure}

		   Just as we did before for the endpoints of  $\underline{I}$ we can take decreasing successions $\{\underline{a}_n\},\{\underline{b}_n\}\subset  \overset{o}{\underline{h(R_i)}}$  of periodic non-corner points converging to $\underline{a}$ and $\underline{b}$ in two distinct sectors (the green and purple semicircles in Figure \ref{Fig: Same s bound}). Let us define,  $a_n:=\pi(\underline{a}_n)$ and $b_n:=\pi(\underline{b}_n)$. They are decreasing sequences (for the vertical order of $h(R_i)$) that converge to the same point  $c=\pi(\underline{a})0\pi(\underline{b})\in I$.  Clearly $\{g(a_n)\}$ is a decreasing (or increasing) sequence for the vertical order of $h(R_k)$ if and only if $\{g(b_n)\}$ is a decreasing (resp. increasing) sequence, and both converge to $g(c)$, without loss of generality suppose they are decreasing. Thus $\pi^{-1}(g(a_n))$ and $\pi^{-1}(g(b_n))$ are decreasing for the vertical order of $\underline{h(R_i)}$. Moreover 
		 $$
		\lim \pi^{-1}(g(a_n))= \lim \Phi_{g,f}(\underline{a_n})=\phi_{g,f}(\underline{a})= \underline{a}'
		$$
and 
			 $$
		\lim \pi^{-1}(g(b_n))= \lim \phi_{g,f}(\underline{b_n})=\phi_{g,f}(\underline{b})=\underline{b}'
		$$

		This implies that $\underline{a}'$ and $\underline{b}'$ are on the lower boundary of $\underline{h(R_i)}$ which is a contradiction.	 Then $\Phi_{g,f}(\underline{I})$ is contained in a single stable boundary component of $\underline{h(R_k)}$ and we conclude that the stable boundary of $\pi^{-1}(h(\cR_{f,g}))$ is $f$-invariant.

	\end{proof}
	
	The $\Phi^{-1}{g,f}$-invariance of $\partial^u \underline{h}(\cR_{f,g})$ is similarly proved. Consequently, $\underline{h}(\cR_{f,g})$ is a Markov partition of the basic piece $K(T_{g,f})$.

	\textbf{Item} $(4)$.  Considering the geometrization of $\underline{h}(\cR_{f,g})$  that was indicated in our last item, let $T'$ be its geometric type. It was proven in Proposition \ref{Prop: type of basic piece is type of pseudo-Anosov} that:
	
	$$		
	\pi(\underline{h}(\cR_{f,g}))=h(\cR_{f,g})
	$$		 
	
	is a geometric Markov partition of $g$ with geometric type $T'$. Since the geometric type of $h(\cR_{f,g})$ is $T_{f,g}$, we can conclude that $T = T_{f,g}$, which confirms the last statement of our proposition.
	
\end{proof}

\chapter{Equivalent geometric types: The Béguin's algorithm}\label{Chapter: Algorithm}

\section*{The equivalence problem.}

Let $T_f$ and $T_g$ be two geometric types in the pseudo-Anosov class, and let $(f, \cR)$ and $(g, \cR_g)$ be realizations of such geometric types by generalized pseudo-Anosov homeomorphisms. Their joint refinement was defined in \ref{Defin: Comun refinament} and denoted as $T_{f, g}$ and $T_{g,f)}$. Their  formal derived form Anosov, denoted as \textbf{DA}$(T_{f, g})$ and \textbf{DA}$(T_{g,f})$, were introduced in \ref{Defi: Formal DA}. Two geometric types, $T_1$ and $T_2$ are  \emph{strongly equivalents} if \textbf{DA}$(T_1)$ has a Markov partition of geometric type $T_2$ and \textbf{DA}$(T_2)$ has a Markov partition of geometric type $T_1$. Proposition \ref{Prop: preimage is Markov Tfg type} yields to the following corollary.

\begin{coro}\label{Coro: equivalence pA and DA}
	Let $f$ and $g$ be generalized pseudo-Anosov homeomorphisms with geometric Markov partitions $\cR_f$ and $\cR_g$ of geometric types $T_f$ and $T_g$, respectively. Let $\cR_{f,g}$ be the joint refinement of $\cR_f$ with respect to $\cR_g$, and let $\cR_{g,f}$ be the joint refinement of $\cR_g$ with respect to $\cR_f$, whose geometric types are $T_{f,g}$ and $T_{g,f}$, respectively.  Under these hypotheses: $f$ and $g$ are topologically conjugated through an orientation preserving homeomorphism if and only if $T_{f,g}$ and $T_{g,f}$ are strongly equivalent.
\end{coro}

\begin{proof}

	Let's assume that $f$ and $g$ are topologically conjugated through an orientation-preserving homeomorphism. Proposition \ref{Prop: preimage is Markov Tfg type} establishes that, in this situation, the saddle-type basic piece $K_{f,g}$ in \textbf{DA}$(T_{f,g})$ has a geometric Markov partition with geometric type $T_{g,f}$; therefore, $T_{f,g}$ and $T_{g,f}$ are strongly equivalent.

	If $T_{f,g}$ and $T_{g,f}$ are strongly equivalent, then the saddle-type basic piece $K_{f,g}$ in \textbf{DA}$(T_{f,g})$ has a geometric Markov partition $\cR$ with geometric type $T_{g,f}$.  Let 
	
	$$\pi:\textbf{DA}((T_{f,g})\to S_f$$
	
	 be the projection introduced in Theorem \ref{Theo: Basic piece projects to pseudo-Anosov}.  Proposition \ref{Prop: type of basic piece is type of pseudo-Anosov} states that $\pi(\cR)$ is a geometric Markov partition of $f$ with type $T_{g,f}$. Theorem \ref{Theo: conjugated iff  markov partition of same type} implies that $f$ and $g$ are topologically conjugated through an orientation-preserving homeomorphism.
	
\end{proof}

Part of the classification  of saddle-type basic pieces for surface Smale diffeomorphisms was carried out in 2004 by François Béguin in \cite{beguin2004smale}. In that work, he developed an algorithm that determines in finite time when two geometric types associated with non-trivial saddle-type basic pieces of surface Smale diffeomorphisms are strongly equivalent.

\begin{theo}[ Béguin \cite{beguin2004smale}]\label{Theo: Beguin algorithm}
	Let $T_1$ and $T_2$ be geometric types realizable as basic pieces $K_1$ and $K_2$ of Smale surface diffeomorphisms. There exists a finite algorithm that determines if  $T_1$ and $T_2$ are strongly equivalent.
\end{theo}

Now we can state a similar result when is applied to geometric types in the pseudo-Anosov class.

\begin{theo}\label{Theo: algorithm conjugacy class}
Let $T_f$ and $T_g$ be two geometric types within the pseudo-Anosov class. Assume that $f: S \rightarrow S_f$ and $g: S_g \rightarrow S_g$ are two generalized pseudo-Anosov homeomorphisms with geometric Markov partitions $\cR_f$ and $\cR_g$, having geometric types $T_f$ and $T_g$ respectively.  We can compute the geometric types $T_{f,g}$ and $T_{g,f}$ of their joint refinements through the algorithmic process described in Chapter \ref{Chap: Computations}, and the homeomorphisms $f$ and $g$ are topologically conjugated by an orientation-preserving homeomorphism if and only if the algorithm developed by Béguin determines that $T_{f,g}$ and $T_{g,f}$ are strongly equivalent.
\end{theo}

\begin{proof}
	
In view of Corollary \ref{Coro: equivalence pA and DA},  the algorithm determines that $T_{f,g}$ and $T_{g,f}$ are strongly equivalent if and only if   $f$ and $g$ are conjugated. 
\end{proof}

\bibliographystyle{alpha}
\bibliography{BibPseudoA}

\vspace{.5cm}

\end{document}